\makeatletter

\input{ifpdf.sty}
\pdffalse

\newdimen\paperwidth
\newdimen\paperheight

\def\papersize#1#2{\let\p@persize\relax\paperwidth#1\paperheight#2}

\def\Afour{\papersize{210truemm}{297truemm}}

\let\p@persize\Afour

\let\onesidestyle\@twosidefalse
\let\twosidestyle\@twosidetrue

\def\margins{\@ifnextchar[{\@margins}{\@margins[\z@]}}

\def\@margins[#1]#2#3{
  \p@persize\dimen0 #3\dimen0 .5\dimen0\normalsize%
  \oddsidemargin-1truein\advance\oddsidemargin#2%
  \evensidemargin-1truein\advance\evensidemargin#2%
  \topmargin-1truein\advance\topmargin\dimen0\headsep\dimen0\footskip\dimen0%
  \textwidth\paperwidth\advance\textwidth-#2\advance\textwidth-#2%
  \textheight\paperheight\advance\textheight-#3\advance\textheight-#3%
  \headheight\baselineskip\advance\topmargin-.5\baselineskip%
  \advance\headsep-.5\baselineskip%
  \footheight\baselineskip
  \advance\textwidth-#1\advance\oddsidemargin#1
  \if@twoside\def\@themargin%
    {\ifodd\count\z@\oddsidemargin\else\evensidemargin\fi}\fi}

\def\headlinesep#1{\advance\topmargin\headsep\advance\topmargin -#1
  \advance\topmargin.5\baselineskip\headsep #1\advance\headsep-.5\baselineskip}

\def\headline{\if@twoside\let\n@xt\h@dlin@\else\let\n@xt\h@@dlin@\fi\n@xt}
  
\def\h@dlin@#1#2{%
  \def\@oddhead{%
    {{\leftskip\z@\rightskip\z@\noindent\normalsize#1}}}
  \def\@evenhead{%
    {{\leftskip\z@\rightskip\z@\noindent\normalsize#2}}}}

\def\h@@dlin@#1{%
  \def\@oddhead{{{\leftskip\z@\rightskip\z@\noindent\normalsize#1}}}}

\def\footline{\if@twoside\let\n@xt\f@tlin@\else\let\n@xt\f@@tlin@\fi\n@xt}

\def\f@tlin@#1#2{%
  \def\@oddfoot{%
    {{\leftskip\z@\rightskip\z@\noindent\normalsize#1}}}
  \def\@evenfoot{%
    {{\leftskip\z@\rightskip\z@\noindent\normalsize#2}}}}

\def\f@@tlin@#1{%
  \def\@oddfoot{{{\leftskip\z@\rightskip\z@\noindent\normalsize#1}}}}

\def\normalpage{\global\@specialpagefalse}

\def\ft{\@ifnextchar[{\ft@s}{\ft@}}
\def\ft@{\ft@@@s[\f@size]}
\def\ft@s[{\@ifnextchar{a}{\ft@sz[}{\ft@@s[}}
\def\ft@@s[{\@ifnextchar{s}{\ft@sz[}{\ft@@@s[}}
\def\ft@@@s[#1]{\ft@sz[at #1pt]}
\def\ft@sz[#1]#2{\font\fonttemp=#2 #1\fonttemp\ignorespaces}

\def\pbf{\series\bfdefault\selectfont\boldmath}

\def\@@bold{bold}

\def\widebar{\ifx\math@version\@@bold
  \let\@widebar\@@@widebar\else\let\@widebar\@@widebar\fi\@widebar}

\def\@@widebar#1{\text{\setbox15\hbox{$#1$}%
  \dimen15 0.45\wd15\advance\dimen15 0.15\ht15%
  \dimen16\ht15\advance\dimen16 0.00em\advance\dimen16 0.3ex%
  \dimen17 0.65\wd15\advance\dimen17 0.05\ht15\advance\dimen17 0.1ex%
  \dimen18 0.035em\advance\dimen18 0.00ex
  \put[\dimen15,\dimen16][c]{\vrule depth 0pt height \dimen18 width \dimen17}}#1}

\def\@@@widebar#1{\text{\setbox15\hbox{$#1$}%
  \dimen15 0.45\wd15\advance\dimen15 0.15\ht15%
  \dimen16\ht15\advance\dimen16 0.00em\advance\dimen16 0.26ex%
  \dimen17 0.65\wd15\advance\dimen17 0.05\ht15\advance\dimen17 0.1ex%
  \dimen18 0.05em\advance\dimen18 0.00ex
  \put[\dimen15,\dimen16][c]{\vrule depth 0pt height \dimen18 width \dimen17}}#1}

\def\smallsquare{\raise-.065em\hbox{$\Box$}}
\def\smallblacksquare{%
  \kern.3ex\vrule depth-.03ex height1.27ex width1.15ex \kern-1.45ex \smallsquare}

\def\smallcircc{\mathop{\mkern3.5mu\text{\raise.58ex\hbox{\ft{lcircle10}a}}}}
\def\varemptyset{{\text{\raise.21ex\hbox{$\not$}}\mkern.15mu\mathrm{O}\mkern.15mu}}

\def\put{\@ifnextchar[{\@put}{\@@rput[\z@,\z@][r]}}
\def\@put[#1]{\@ifnextchar[{\@@put[#1]}{\@@@@@put[#1]}}
\def\@@put[#1][{\@ifnextchar{l}{\@@lput[#1][}{\@@@put[#1][}}
\def\@@@put[#1][{\@ifnextchar{c}{\@@cput[#1][}{\@@@@put[#1][}}
\def\@@@@put[#1][{\@ifnextchar{r}{\@@rput[#1][}{\relax}}
\def\@@@@@put[{\@ifnextchar{l}{\@@lput[\z@,\z@][}{\@@@@@@put[}}
\def\@@@@@@put[{\@ifnextchar{c}{\@@cput[\z@,\z@][}{\@@@@@@@put[}}
\def\@@@@@@@put[{\@ifnextchar{r}{\@@rput[\z@,\z@][}{\@@@@@@@@put[}}
\def\@@@@@@@@put[#1]{\@@rput[#1][r]}

\let\hm@d@\leavevmode

\long\def\@@lput[#1,#2][l]#3{\setbox0\hbox{#3}\hm@d@\raise#2\hbox to\z@{\dimen0 #1%
  \advance\dimen0-\wd0\kern\dimen0\dp0\z@\ht0\z@\wd0\z@\box0\hss}\ignorespaces}
\long\def\@@cput[#1,#2][c]#3{\setbox0\hbox{#3}\hm@d@\raise#2\hbox to\z@{\dimen0 #1%
  \advance\dimen0-.5\wd0\kern\dimen0\dp0\z@\ht0\z@\wd0\z@\box0\hss}\ignorespaces}
\long\def\@@rput[#1,#2][r]#3{\setbox0\hbox{\kern#1\raise#2\hbox{#3}}%
  \dp0\z@\ht0\z@\wd0\z@\hm@d@\box0\ignorespaces}

\def\flbox{\@ifnextchar[{\@flbox}{\@@rflbox[\z@,\z@][r]}}
\def\@flbox[#1]{\@ifnextchar[{\@@flbox[#1]}{\@@@@@flbox[#1]}}
\def\@@flbox[#1][{\@ifnextchar{l}{\@@lflbox[#1][}{\@@@flbox[#1][}}
\def\@@@flbox[#1][{\@ifnextchar{c}{\@@cflbox[#1][}{\@@@@flbox[#1][}}
\def\@@@@flbox[#1][{\@ifnextchar{r}{\@@rflbox[#1][}{\relax}}
\def\@@@@@flbox[{\@ifnextchar{l}{\@@lflbox[\z@,\z@][}{\@@@@@@flbox[}}
\def\@@@@@@flbox[{\@ifnextchar{c}{\@@cflbox[\z@,\z@][}{\@@@@@@@flbox[}}
\def\@@@@@@@flbox[{\@ifnextchar{r}{\@@rflbox[\z@,\z@][}{\@@@@@@@@flbox[}}
\def\@@@@@@@@flbox[#1]{\@@rflbox[#1][r]}
\long\def\@@lflbox[#1,#2][l]#3{\@@lput[#1,#2][l]{%
  \vtop{\leftskip\z@\parindent\z@\raggedleft\hm@d@#3}}}
\long\def\@@cflbox[#1,#2][c]#3{\@@cput[#1,#2][c]{%
  \vtop{\leftskip\z@\parindent\z@\raggedcenter\hm@d@#3}}}
\long\def\@@rflbox[#1,#2][r]#3{\@@rput[#1,#2][r]{%
  \vtop{\leftskip\z@\parindent\z@\raggedright\hm@d@#3}}}

\input{epsf.sty}

\def\epsfdirectory#1{\xdef\epsfdir@{#1}}

\def\epsfgetlitbb#1#2 #3 #4 #5]#6{\epsfgrab #2 #3 #4 #5 .\\%
   \epsfsetgraph{\epsfdir@#6}}%
\def\epsfnormal#1{\epsfgetbb{\epsfdir@#1}\epsfsetgraph{\epsfdir@#1}}%

\def\epsfsize#1#2{%
  \ifnum\epsfscale=1000
    \ifnum\epsfxsize=0
      \ifnum\epsfysize=0
      \else \computescale{\epsfysize}{#2}
      \fi
    \else \computescale{\epsfxsize}{#1}
    \fi
  \else
    \epsfxsize=#1
    \divide\epsfxsize by 1000
    \multiply\epsfxsize by \epsfscale
  \fi}

\epsfdirectory{}

\def\showfig#1#2{\epsfbox{#2}}
\def\fig@#1#2{\leavevmode{\framebox{\figstyl@\strut{ #1 }}}}

\def\figstyle#1{\def\figstyl@{#1}}

\figstyle{\family{cmr}\shape{n}\series{m}\selectfont\normalsize}

\def\showfigurestrue{\let\fig\showfig}
\def\showfiguresfalse{\let\fig\fig@}

\showfigurestrue

\makeatother 

  \let\epsilon\varepsilon
\let\textheta\theta      \let\theta\vartheta
          \let\phi\varphi
   \let\emptyset\varemptyset

\documentstyle[12pt]{article}

\makeatletter

\let\Larg@\Large
\let\hug@\huge

\def\usepackage#1{\input{#1.sty}}

\input{geom.sty}
\input{amstext.sty}

\def\rm{\ifmmode\else\protect\prm\fi}
\def\it{\ifmmode\else\protect\pit\fi}
\let\mathbf\bold

\def\r@adlabel#1#2{\global\@namedef{#1@\the\@key}{#2}}

\let\Large\Larg@
\let\huge\hug@

\def\smallskip{\vskip\smallskipamount}
\def\medskip{\vskip\medskipamount}
\def\bigskip{\vskip\bigskipamount}

\def\mytrivlist{\parsep\parskip\@nmbrlistfalse
  \my@trivlist \labelwidth\z@ \leftmargin\z@
  \itemindent\z@ \def\makelabel##1{##1}}

\def\my@trivlist{\global\@newlisttrue \@outerparskip\parskip}

\def\end#1{\csname end#1\endcsname\@checkend{#1}%
  \expandafter\endgroup\if@endpe\@doendpe\fi
  \if@ignore \global\@ignorefalse \ignorespaces\fi}

\def\maketitle{\par
 \begingroup
 \def\thefootnote{\fnsymbol{footnote}}
 \def\@makefnmark{\hbox 
 to 0pt{$^{\@thefnmark}$\hss}} 
 \if@twocolumn 
 \twocolumn[\@maketitle] 
 \else 
 \global\@topnum\z@ \@maketitle \fi\thispagestyle{plain}\@thanks
 \endgroup
 \setcounter{footnote}{0}
 \let\maketitle\relax
 \let\@maketitle\relax
 \gdef\@thanks{}\gdef\@author{}\gdef\@title{}\let\thanks\relax}

\def\@maketitle{ 
 \null
 \vskip 2em \begin{center}
 {\LARGE \@title \par} \vskip 1.5em {\large \lineskip .5em
\begin{tabular}[t]{c}\@author 
 \end{tabular}\par} 
 \vskip 1em {\large \@date} \end{center}
 \par
 \vskip 1.5em}
 
\def\partbeforeskip#1{\def\p@rtbeforeskip{#1}}
\def\partstyle#1{\def\p@rtstyl@{#1}}
\def\partdot#1{\def\partd@t{#1}}
\def\partafterskip#1{\def\p@rtafterskip{#1}}
\def\partintrostyle#1{\def\partintr@styl@{#1}}
\def\partintrodot#1{\def\partintr@dot{#1}}
\long\def\partintrosep#1{\long\def\partintr@sep{#1}}
\def\partnewpagetrue{\def\p@rtnewp@ge{\newpage}}
\def\partnewpagefalse{\long\def\p@rtnewp@ge{\par}}

\partbeforeskip{4ex}
\partstyle{\centering\Large\bf}
\partdot{}
\partafterskip{3ex}
\partintrostyle{\large}
\partintrodot{}
\partintrosep{\par}
\partnewpagefalse

\def\partname{Part}
\def\part{\p@rtnewp@ge\addvspace\p@rtbeforeskip\@afterindentfalse\secdef\@part\@spart}

\def\@part[#1]#2{\ifnum \c@secnumdepth >-1\relax  
        \refstepcounter{part}                     
        \def\@tempa{\addcontentsline{toc}{part}}  %
        \expandafter\@tempa\expandafter{\thepart  
          \hspace{1em}#1}\else                    
        \addcontentsline{toc}{part}{#1}\fi        
   {\p@rtstyl@                       
    \ifnum \c@secnumdepth >-1\relax        
      {\partintr@styl@\partname\ \thepart  
       \partintr@dot}\partintr@sep\nobreak 
    \fi                                    
    #2\partd@t\markboth{}{}\par}
    \nobreak                       
    \vskip\p@rtafterskip           
   \@afterheading                  
    }                              

\def\@spart#1{{\p@rtcentering\p@rtstyl@                      
    #1\partd@t\par}                 
    \nobreak                        
    \vskip\p@rtafterskip            
    \@afterheading                  
  }                                 

\newif\ifsection@ftind
\newif\ifsection@ftpar

\def\sectionbeforeskip#1{\def\s@ctbeforeskip{#1}}
\def\sectionstyle#1{\def\s@ctstyl@{#1}}
\def\sectiondot#1{\def\sectiond@t{#1}}
\def\sectionafterskip#1{\def\s@ctafterskip{#1}}
\def\sectionintrostyle#1{\def\sectionintr@styl@{#1}}
\def\sectionintro#1{\def\sectionintr@{#1}}
\def\sectionintrodot#1{\def\sectionintr@dot{#1}}
\def\sectionintrosep#1{\def\sectionintr@sep{#1}}
\def\sectionindenttrue{\def\s@ctind{\parindent}}
\def\sectionindentfalse{\def\s@ctind{\z@}}
\def\sectionafterindenttrue{\section@ftindtrue}
\def\sectionafterindentfalse{\section@ftindfalse}
\def\sectionafternewlinetrue{\section@ftpartrue}
\def\sectionafternewlinefalse{\section@ftparfalse}

\newif\ifsubsection@ftind
\newif\ifsubsection@ftpar

\def\subsectionbeforeskip#1{\def\ss@ctbeforeskip{#1}}
\def\subsectionstyle#1{\def\ss@ctstyl@{#1}}
\def\subsectiondot#1{\def\subsectiond@t{#1}}
\def\subsectionafterskip#1{\def\ss@ctafterskip{#1}}
\def\subsectionintrostyle#1{\def\subsectionintr@styl@{#1}}
\def\subsectionintro#1{\def\subsectionintr@{#1}}
\def\subsectionintrodot#1{\def\subsectionintr@dot{#1}}
\def\subsectionintrosep#1{\def\subsectionintr@sep{#1}}
\def\subsectionindenttrue{\def\ss@ctind{\parindent}}
\def\subsectionindentfalse{\def\ss@ctind{\z@}}
\def\subsectionafterindenttrue{\subsection@ftindtrue}
\def\subsectionafterindentfalse{\subsection@ftindfalse}
\def\subsectionafternewlinetrue{\subsection@ftpartrue}
\def\subsectionafternewlinefalse{\subsection@ftparfalse}

\newif\ifsubsubsection@ftind
\newif\ifsubsubsection@ftpar

\def\subsubsectionbeforeskip#1{\def\sss@ctbeforeskip{#1}}
\def\subsubsectionstyle#1{\def\sss@ctstyl@{#1}}
\def\subsubsectiondot#1{\def\subsubsectiond@t{#1}}
\def\subsubsectionafterskip#1{\def\sss@ctafterskip{#1}}
\def\subsubsectionintrostyle#1{\def\subsubsectionintr@styl@{#1}}
\def\subsubsectionintro#1{\def\subsubsectionintr@{#1}}
\def\subsubsectionintrodot#1{\def\subsubsectionintr@dot{#1}}
\def\subsubsectionintrosep#1{\def\subsubsectionintr@sep{#1}}
\def\subsubsectionindenttrue{\def\sss@ctind{\parindent}}
\def\subsubsectionindentfalse{\def\sss@ctind{\z@}}
\def\subsubsectionafterindenttrue{\subsubsection@ftindtrue}
\def\subsubsectionafterindentfalse{\subsubsection@ftindfalse}
\def\subsubsectionafternewlinetrue{\subsubsection@ftpartrue}
\def\subsubsectionafternewlinefalse{\subsubsection@ftparfalse}

\newif\ifparagraph@ftind
\newif\ifparagraph@ftpar

\def\paragraphbeforeskip#1{\def\p@rbeforeskip{#1}}
\def\paragraphstyle#1{\def\p@rstyl@{#1}}
\def\paragraphdot#1{\def\paragraphd@t{#1}}
\def\paragraphafterskip#1{\def\p@rafterskip{#1}}
\def\paragraphintrostyle#1{\def\paragraphintr@styl@{#1}}
\def\paragraphintro#1{\def\paragraphintr@{#1}}
\def\paragraphintrodot#1{\def\paragraphintr@dot{#1}}
\def\paragraphintrosep#1{\def\paragraphintr@sep{#1}}
\def\paragraphindenttrue{\def\p@rind{\parindent}}
\def\paragraphindentfalse{\def\p@rind{\z@}}
\def\paragraphafterindenttrue{\paragraph@ftindtrue}
\def\paragraphafterindentfalse{\paragraph@ftindfalse}
\def\paragraphafternewlinetrue{\paragraph@ftpartrue}
\def\paragraphafternewlinefalse{\paragraph@ftparfalse}

\newif\ifsubparagraph@ftind
\newif\ifsubparagraph@ftpar

\def\subparagraphbeforeskip#1{\def\sp@rbeforeskip{#1}}
\def\subparagraphstyle#1{\def\sp@rstyl@{#1}}
\def\subparagraphdot#1{\def\subparagraphd@t{#1}}
\def\subparagraphafterskip#1{\def\sp@rafterskip{#1}}
\def\subparagraphintrostyle#1{\def\subparagraphintr@styl@{#1}}
\def\subparagraphintro#1{\def\subparagraphintr@{#1}}
\def\subparagraphintrodot#1{\def\subparagraphintr@dot{#1}}
\def\subparagraphintrosep#1{\def\subparagraphintr@sep{#1}}
\def\subparagraphindenttrue{\def\sp@rind{\parindent}}
\def\subparagraphindentfalse{\def\sp@rind{\z@}}
\def\subparagraphafterindenttrue{\subparagraph@ftindtrue}
\def\subparagraphafterindentfalse{\subparagraph@ftindfalse}
\def\subparagraphafternewlinetrue{\subparagraph@ftpartrue}
\def\subparagraphafternewlinefalse{\subparagraph@ftparfalse}

\sectionbeforeskip{\bigskipamount}
\sectionstyle{\large\bf}
\sectiondot{}
\sectionafterskip{.5\bigskipamount}
\sectionintrostyle{}
\sectionintro{}
\sectionintrodot{.}
\sectionintrosep{1.25ex}
\sectionindentfalse
\sectionafterindenttrue
\sectionafternewlinetrue

\subsectionbeforeskip{.8\bigskipamount}
\subsectionstyle{\normalsize\bf}
\subsectiondot{}
\subsectionafterskip{.4\bigskipamount}
\subsectionintrostyle{}
\subsectionintro{}
\subsectionintrodot{.}
\subsectionintrosep{1.25ex}
\subsectionindentfalse
\subsectionafterindenttrue
\subsectionafternewlinetrue

\subsubsectionbeforeskip{.6\bigskipamount}
\subsubsectionstyle{\normalsize\bf}
\subsubsectiondot{}
\subsubsectionafterskip{.3\bigskipamount}
\subsubsectionintrostyle{}
\subsubsectionintro{}
\subsubsectionintrodot{.}
\subsubsectionintrosep{1.25ex}
\subsubsectionindentfalse
\subsubsectionafterindenttrue
\subsubsectionafternewlinetrue

\paragraphbeforeskip{.5\bigskipamount}
\paragraphstyle{\normalsize\bf}
\paragraphdot{.}
\paragraphafterskip{1.25ex}
\paragraphintrostyle{}
\paragraphintro{}
\paragraphintrodot{.}
\paragraphintrosep{1.25ex}
\paragraphindentfalse
\paragraphafterindenttrue
\paragraphafternewlinefalse

\subparagraphbeforeskip{.5\bigskipamount}
\subparagraphstyle{\normalsize\bf}
\subparagraphdot{.}
\subparagraphafterskip{1.25ex}
\subparagraphintrostyle{}
\subparagraphintro{}
\subparagraphintrodot{.}
\subparagraphintrosep{1.25ex}
\subparagraphindenttrue
\subparagraphafterindenttrue
\subparagraphafternewlinefalse

\let\@partoken\par
\long\def\@@gobble#1{}
\def\ignorepar{\@ifnextchar\@partoken{\expandafter\ignorepar\@@gobble}{\ignorespaces}}

\def\@startsection#1#2#3#4#5#6{
   \@tempskipa #4\relax
   \csname if#1@ftind\endcsname\@afterindenttrue\else\@afterindentfalse\fi
   \advance\@tempskipa by\presection
   \if@nobreak \everypar{}\else
     \addpenalty{\@secpenalty}\addvspace{\@tempskipa}%
     \allowbreak\vskip -\presection \fi \@ifstar
     {\@ssect{#1}{#2}{#3}{#4}{#5}{#6}}{\@dblarg{\@sect{#1}{#2}{#3}{#4}{#5}{#6}}}}

\def\@sect#1#2#3#4#5#6[#7]#8{\def\object@type{#1}%
   \ifnum #2>\c@secnumdepth\def\@svsec{}\def\@tempb{}%
      \else\refstepcounter{#1}\def\@svsec{{\csname #1intr@styl@\endcsname%
        {\csname #1intr@\endcsname}\csname the#1\endcsname%
        \csname #1intr@dot\endcsname\kern\csname #1intr@sep\endcsname}}%
        \edef\@tempb{\noexpand\numberline{\csname the#1\endcsname}}\fi%
   \def\@tempa{\addcontentsline{toc}{#1}}%
   \csname if#1@ftpar\endcsname%
      \begingroup #6\relax%
        \@hangfrom{\hskip #3\relax\@svsec}{\interlinepenalty \@M{#8}%
        \csname #1d@t\endcsname\par}%
      \endgroup%
      \csname #1mark\endcsname{#7}%
      \expandafter\@tempa\expandafter{\@tempb #7}%
      \ifautolabel\label*{#8}\fi%
   \else%
      \def\@svsechd{#6\hskip #3\relax%
         \@svsec{#8}%
         \csname #1d@t\endcsname%
         \csname #1mark\endcsname{#7}%
         \expandafter\@tempa\expandafter{\@tempb #7}%
         \ifautolabel\label*{#8}\fi}\fi%
   \@xsect{#1}{#5}\ignorepar}

\def\@ssect#1#2#3#4#5#6#7{%
   \ifnum #2>\c@secnumdepth\def\@tempb{}\else \def\@tempb{\numberline{}}\fi%
     \def\@tempa{\addcontentsline{toc}{s#1}}%
     \csname if#1@ftpar\endcsname
        \begingroup #6\relax
           \@hangfrom{\hskip #3}{\interlinepenalty \@M{#7}%
           \csname #1d@t\endcsname\par}%
        \endgroup
        \csname s#1mark\endcsname{#7}%
        \ifstarredcontents\expandafter\@tempa\expandafter{\@tempb #7}\fi%
        \ifautolabel\label*{#7}\fi%
     \else%
        \def\@svsechd{#6\hskip #3\relax{#7}%
        \csname #1d@t\endcsname%
        \csname s#1mark\endcsname{#7}%
        \ifautolabel\label*{#7}\fi}\fi
   \@xsect{#1}{#5}\ignorepar}

\def\@xsect#1#2{
   \csname if#1@ftpar\endcsname 
       \par \nobreak \vskip #2\relax \@afterheading
    \else \global\@nobreakfalse \global\@noskipsectrue
       \everypar{\if@noskipsec \global\@noskipsecfalse
                   \clubpenalty\@M \hskip -\parindent
                   \begingroup \@svsechd \endgroup \unskip
                   \hskip #2\relax  
                  \else \clubpenalty \@clubpenalty
                    \everypar{}\fi}\fi\ignorespaces}

\def\section{\@startsection{section}{1}{\s@ctind}
  {\s@ctbeforeskip}{\s@ctafterskip}{\s@ctstyl@}}
\def\subsection{\@startsection{subsection}{2}{\ss@ctind}
  {\ss@ctbeforeskip}{\ss@ctafterskip}{\ss@ctstyl@}}
\def\subsubsection{\@startsection{subsubsection}{3}{\sss@ctind}
  {\sss@ctbeforeskip}{\sss@ctafterskip}{\sss@ctstyl@}}
\def\paragraph{\@startsection{paragraph}{4}{\p@rind}
  {\p@rbeforeskip}{\p@rafterskip}{\p@rstyl@}}
\def\subparagraph{\@startsection{subparagraph}{4}{\sp@rind}
  {\sp@rbeforeskip}{\sp@rafterskip}{\sp@rstyl@}}

\def\statementabove#1{\def\th@bove{#1}}
\def\statementstyle#1{\def\thstyl@{#1}}
\def\statementbelow#1{\def\thb@low{#1}}
\def\statementindentfalse{\let\thind@nt\relax}
\def\statementindenttrue{\let\thind@nt\indent}

\def\statementintrostyle#1{\def\thintr@style{#1}}
\def\statementintrodot#1{\def\thintr@dot{#1}}
\def\statementintrosep#1{\def\thintr@sep{#1}}
\def\statementintrobrackets#1#2{\def\thintr@left{#1}\def\thintr@right{#2}}

\statementabove{\medskip}
\statementstyle{\sl}
\statementbelow{\medskip}
\statementindenttrue

\statementintrostyle{\normalshape\bf}
\statementintrodot{.}
\statementintrosep{\kern1.25ex}
\statementintrobrackets{(}{)}

\def\@thskip{\dimen100\lastskip\vskip-\dimen100%
  \th@bove\dimen101\lastskip\vskip-\dimen101%
  \ifdim\dimen100>\dimen101\else\dimen100\dimen101\fi\vskip\dimen100\vskip0pt}

\long\def\@@newtheorem#1#2#3{%
  \newenvironment{#3}%
    {\def\object@type{#3}\par\@thskip%
     \@ifnextchar[{\@enva{#3}{\thstyl@#1{#2}}}{\@envb{#3}{\thstyl@#1{#2}}}}%
    {\end{#3@}}%
  \@ifnextchar[{\@nothm{#3}}{\@nnthm{#3}}}

\def\@nothm#1[#2]#3{%
  \@ifundefined{c@#2}{\@latexerr{No theorem environment `#2' defined}\@eha}%
  {\expandafter\@ifdefinable\csname #1@\endcsname
  {\global\@namedef{the#1}{\@nameuse{the#2}}%
   \global\@namedef{c@#1}{\@nameuse{c@#2}}
   \global\@namedef{p@#1}{\@nameuse{p@#2}}
   \global\@namedef{#1@}{\@nnnthm{#2}{#3}}%
   \global\@namedef{end#1@}{\@endtheorem}}}}

\def\@nnnthm#1#2{\refstepcounter
    {#1}\@ifnextchar[{\@ynnnthm{#1}{#2}}{\@xnnnthm{#1}{#2}}}
 
\def\@xnnnthm#1#2{\@begintheorem{#2}{\csname the#1\endcsname}\ignorespaces}
\def\@ynnnthm#1#2[#3]{\@opargbegintheorem{#2}{\csname the#1\endcsname}{#3}\ignorespaces}

\def\renewtheorem{\@ifnextchar[{\@renewtheorem}{\@renewtheorem[{}{}]}}

\long\def\@renewtheorem[#1]{\@@renewtheorem#1}

\long\def\@@renewtheorem#1#2#3{%
  \expandafter\let\csname#3@\endcsname\undefined
  \renewenvironment{#3}%
    {\def\object@type{#3}\par\@thskip%
     \@ifnextchar[{\@enva{#3}{\thstyl@#1{#2}}}{\@envb{#3}{\thstyl@#1{#2}}}}%
    {\end{#3@}}%
  \@ifnextchar[{\@nothm{#3}}{\@nnthm{#3}}}

\def\@begintheorem#1#2{\@opargbegintheorem{#1}{#2}{}}

\def\@opargbegintheorem#1#2#3{%
        \edef\@tempx{#1}%
        \expandafter\let\expandafter\@tempy#2
        \def\@tempz{#3}%
        \mytrivlist\item[\thind@nt\hskip\labelsep%
        {\thintr@style%
          #1\ifx\@tempx\@empty\else\ifx\@tempy\relax\else\kern1ex\fi\fi#2%
          \ifx\@tempz\@empty%
            \ifx\@tempx\@empty\ifx\@tempy\relax%
            \else\thintr@dot\thintr@sep\fi\else\thintr@dot\thintr@sep\fi%
            \else%
            \ifx\@tempx\@empty\ifx\@tempy\relax\else\kern1ex\fi\else\kern1ex\fi%
           \thintr@left{#3}\thintr@right\thintr@dot\thintr@sep\fi}%
            \hskip-\labelsep]%
        \ifautolabel\label*{#3}\fi}

\def\@endtheorem{\strut\endtrivlist\thb@low}

\def\proofname{Proof}

\def\proofabove#1{\def\pf@bove{#1}}
\def\proofstyle#1{\def\pfstyl@{#1}}
\def\proofbelow#1{\def\pfb@low{#1}}
\def\proofindentfalse{\let\pfind@nt\relax}
\def\proofindenttrue{\let\pfind@nt\indent}

\def\proofintrostyle#1{\def\pfintr@style{#1}}
\def\proofintrodot#1{\def\pfintr@dot{#1}}
\def\proofintrosep#1{\def\pfintr@sep{#1}}
\def\proofintrobrackets#1#2{\def\pfintr@left{#1}\def\pfintr@right{#2}}

\proofabove{\medskip}
\proofstyle{}
\proofbelow{\medskip}
\proofindenttrue

\proofintrostyle{\sl}
\proofintrodot{.}
\proofintrosep{\kern1.25ex}
\proofintrobrackets{of\kern1ex}{}

\def\@pfskip{\dimen100\lastskip\vskip-\dimen100%
  \pf@bove\dimen101\lastskip\vskip-\dimen101%
  \ifdim\dimen100>\dimen101\else\dimen100\dimen101\fi\vskip\dimen100\vskip0pt}

\renewenvironment{proof}%
  {\@pfskip\mytrivlist\item[\pfind@nt]\@ifnextchar[{\pro@f}{\pro@f[\prooftag]}}
  {\ifvoid\provedbox\else\hproved\fi\endtrivlist\pfb@low}

\def\pro@f[#1]{\setbox\provedbox\hbox{\provedboxcontents{#1}}\proofintro{#1}}

\def\proofintro#1{\expandafter\def\expandafter\@tempa\expandafter{#1}%
  {\pfintr@style{\proofname\ifx\@tempa\empty\else\kern1ex\pfintr@left{#1}%
  \pfintr@right\fi}\pfintr@dot\pfintr@sep}\pfstyl@\ignorespaces}

\def\provedmark#1{\def\prm@rk{#1}}
\def\provedsep#1{\def\prs@p{#1}}

\provedmark{$\square$}
\provedsep{\kern1.25ex}

\def\provedtexttrue{\def\prb@x##1{\fbox{\small##1}}}
\def\provedtextfalse{\def\prb@x##1{\prm@rk}}
\def\provedmarkrighttrue{\let\prhf@l\hfill}
\def\provedmarkrightfalse{\let\prhf@l\relax}

\provedtextfalse
\provedmarkrightfalse

\def\provedboxcontents#1{\expandafter\def\expandafter\@tempa\expandafter{#1}%
  \ifx\@tempa\empty\prm@rk\else\prb@x{#1}\fi}

\def\proved{\ifmmode\eqno{\box\provedbox}\else\hproved\fi}

\def\hproved{\unskip\nobreak\prhf@l\penalty50\prs@p\hbox{}\nobreak\prhf@l
  \box\provedbox{\finalhyphendemerits=0\par}}

\def\captionstyle#1{\def\c@ptstyl@{#1}}
\def\captionintrostyle#1{\def\c@pintr@style{#1}}
\def\captionintrodot#1{\def\c@pintr@dot{#1}}
\def\captionintrosep#1{\def\c@pintr@sep{#1}}

\captionstyle{\small\sf}
\captionintrostyle{\bf}
\captionintrodot{.}
\captionintrosep{\hskip1.25ex}

\long\def\@makecaption#1#2{%
    \vskip\captionskip
    \setbox\@tempboxa\hbox{%
      \ifproofing\@ifundefined{the@label}{}
        {\hbox to 0pt{\vbox to 0pt{\vss\hbox{\tiny\the@label}\bigskip}\hss}}\fi
      \c@ptstyl@{\c@pintr@style #1\c@pintr@dot}\ignorespaces #2}%
    \@captionwidth=\hsize \advance\@captionwidth-2\@captionmargin
    \ifdim \wd\@tempboxa >\@captionwidth {%
        \rightskip=\@captionmargin\leftskip=\@captionmargin
        \unhbox\@tempboxa\par}%
      \else
        \hbox to\hsize{\hfil\box\@tempboxa\hfil}%
    \fi}

\def\end@Float#1{%
  \expandafter\caption\expandafter[\the@title]{%
   {\c@pintr@style%
   \ifx\the@caption\empty\ifx\the@title\empty
   \else\c@pintr@sep\fi\else\c@pintr@sep\fi
    \the@title\ifx\the@caption\empty%
     \expandafter\label\expandafter*\expandafter{\the@label}%
    \else\ifx\the@title\empty%
     \expandafter\label\expandafter*\expandafter{\the@label}%
    \else\c@pintr@dot\c@pintr@sep%
     \expandafter\label\expandafter*\expandafter{\the@label}\fi\fi}%
   \ignorespaces\the@caption}%
  \end{#1}}

\renewenvironment{Figure}{\@ifnextchar[%
  {\@myFloat{figure}}{\@myFloat{figure}[htbp]}}{\end@Float{figure}}

\def\@myFloat#1[#2]#3{%
  \begin{#1}[#2]\def\the@label{#3}}

\def\showfig{\showfigurestrue\fig}
\def\fig#1{\@ifnextchar[{\@fig{#1}}{\@fig{#1}[0pt]}}

\def\@fig#1[#2]#3{\@ifnextchar[{\@@fig{#1}[#2]{#3}}{\@@fig{#1}[#2]{#3}[0pt]}}
\def\@@fig#1[#2]#3[#4]#5#6{%
  \def\the@title{#5}\def\the@caption{#6}\centerline{\fig@{#1}{#2}{#3}}\vskip#4}

\def\fig@@#1#2#3{\leavevmode{\figstyl@\vrule width 0pt height 1.8ex%
 \smash{\framebox{\strut\def\@temp{#1}\ifx\@temp\@empty{ #3 }\else{ #1 }\fi}}}}
\def\fig@@@#1#2#3{\leavevmode\kern#2\epsfbox{#3}}

\def\figstyle#1{\def\figstyl@{#1}}

\figstyle{\family{cmr}\shape{n}\series{m}\selectfont\normalsize}

\def\oldFigure{%
  \renewenvironment{Figure}{\@ifnextchar[%
    {\@Float{figure}}{\@Float{figure}[htbp]}}{\end@Float{figure}}
    \let\showfig\@ldshowfig \let\fig\@ldfig
    \let\showfigurestrue\@ldshowfigurestrue
    \let\showfiguresfalse\@ldshowfiguresfalse
    \showfiguresfalse}

\def\@ldshowfig#1#2{\epsfbox{#2}}
\def\@ldfig@#1#2{\leavevmode{\framebox{\figstyl@\strut{ #1 }}}}
\def\@ldshowfigurestrue{\let\fig\@ldshowfig}
\def\@ldshowfiguresfalse{\let\fig\@ldfig@}

\newcounter{diagram}

\let\thediagram\theequation
\def\ftype@diagram{2}
\def\ext@diagram{lod}
\def\diagram{\@float{diagram}}
\let\enddiagram\end@float
\newif\if@diagnum

\def\diag#1{\@ifnextchar[{\@diag{#1}}{\@diag{#1}[0pt]}}

\def\@diag#1[#2]#3{\@ifnextchar[{\@@diag{#1}[#2]{#3}}{\@@diag{#1}[#2]{#3}[0pt]}}

\def\@@diag#1[#2]#3[#4]#5{
  \def\the@tag{#5}\@eqnswtrue%
  \centerline{\setbox0\hbox{\diag@{#1}{#2}{#3}}
  \dimen0 -0.5\wd0\dimen1 0.5\ht0\box0%
  \advance\dimen0 0.5\hsize\advance\dimen0 -\rightskip\advance\dimen1 #4%
  \let\@currentlabel\the@tag%
  \setbox0\hbox to 0pt{\hss\family{cmr}\shape{n}\series{m}\selectfont(\the@tag)}%
  \ifx\the@tag\@empty\refstepcounter{equation}\let\@currentlabel\theequation%
    \setbox0\hbox to 0pt{\hss\family{cmr}\shape{n}\series{m}\selectfont(\thediagram)}\fi%
  \if@eqnsw\else\let\@currentlabel\relax\setbox0\hbox to 0pt{}\fi%
  \advance\dimen1 -0.5\ht0%
  \put[\dimen0,\dimen1][l]{%
    \box0\expandafter\label\expandafter*\expandafter{\the@label}\kern0.15em}}}

\def\diag@@#1#2#3{\leavevmode{\diagstyl@\vrule width 0pt height 1.8ex%
 \smash{\framebox{\strut\def\@temp{#1}\ifx\@temp\@empty{ #3 }\else{ #1 }\fi}}}}
\def\diag@@@#1#2#3{\leavevmode\kern#2\epsfbox{#3}}

\def\diagstyle#1{\def\diagstyl@{#1}}

\diagstyle{\family{cmr}\shape{n}\series{m}\selectfont\normalsize}

\def\showfiguresfalse{\let\fig@\fig@@}
\def\showfigurestrue{\let\fig@\fig@@@}

\def\showdiagramsfalse{\let\diag@\diag@@}
\def\showdiagramstrue{\let\diag@\diag@@@}

\def\n@number{\@eqnswfalse\let\@currentlabel\relax\let\the@tag\relax}

\def\equation{$$
  \@eqnswtrue\def\object@type{equation}\let\nonumber\n@number%
  \advance\c@equation1\edef\@currentlabel{\theequation}\advance\c@equation-1%
  \def\the@tag{\refstepcounter{equation}\eqno\hbox{\@eqnnum}}}

\def\tag#1{\edef\@currentlabel{#1}\def\the@tag{\eqno\hbox{\reset@font\rm(#1)}}}

\def\endequation{\the@tag$$
  \global\@ignoretrue}

\let\it@m\item
\def\item{\@ifnextchar[{\item@}{\item@@}}
\def\item@[#1]{\it@m[#1]\vskip-\lastskip\vskip\itemsep}
\def\item@@{\it@m\vskip-\lastskip\vskip\itemsep}
\def\s@titemsep{\@ifnextchar[{\s@@titemsep}{\relax}}
\def\s@@titemsep[#1]{\itemsep#1}

\let\@itemize\itemize
\let\@enditemize\enditemize
\renewenvironment{itemize}
{\@itemize\itemsep3pt\parsep0pt\topsep0pt\partopsep0pt\s@titemsep}
{\@enditemize\vskip-\lastskip\vskip\itemsep}

\let\@enumerate\enumerate
\let\@endenumerate\endenumerate
\renewenvironment{enumerate}
{\@enumerate\itemsep3pt\parsep0pt\topsep0pt\partopsep0pt\s@titemsep}
{\@endenumerate\vskip-\lastskip\vskip\itemsep}

\let\@description\description
\let\@enddescription\enddescription
\renewenvironment{description}
{\@description\itemsep3pt\parsep0pt\topsep0pt\partopsep0pt\s@titemsep}
{\@enddescription\vskip-\lastskip\vskip\itemsep}

\topsep\dimen200
\partopsep\dimen201

\@definecounter{bibenumi}

\def\thebibliography#1{%
 \section*{\refname}\vskip-\lastskip%
 \list{[\arabic{bibenumi}]}{\topsep0pt\settowidth\labelwidth{[#1]}%
 \leftmargin\labelwidth\advance\leftmargin\labelsep\usecounter{bibenumi}}%
 \def\newblock{\hskip .11em plus .33em minus .07em}%
 \sloppy\clubpenalty4000\widowpenalty4000\sfcode`\.=1000\relax}

\makeatother

\makeatletter
\textfloatsep\floatsep
\makeatother

\abovedisplayskip\smallskipamount
\belowdisplayskip\smallskipamount
\abovedisplayshortskip\smallskipamount
\belowdisplayshortskip\smallskipamount

\proofingfalse
\autolabelfalse

\showfigurestrue
\showdiagramstrue

\newtheorem{stat}{\statname}  \unnumbered{stat}
\newenvironment{statement}[1]{\def\statname{#1}\begin{stat}}{\end{stat}}

\newtheorem{nstat}{\nstatname}[section]

\newtheorem{definition}[nstat]{Definition}
\newtheorem{lemma}[nstat]{Lemma}
\newtheorem{proposition}[nstat]{Proposition}
\newtheorem{theorem}[nstat]{Theorem}
\newtheorem{corollary}[nstat]{Corollary}

\newtheorem{remark}[nstat]{Remark}

\let\ns\normalshape

\papersize{216truemm}{279truemm}
\margins{33truemm}{20.5truemm}
\advance\voffset-5mm
\advance\hoffset7.5truemm
\advance\footskip2.5truemm

\headline{\hfill}
\footline{\small\hfill--\kern1ex\thepage\kern1ex--\hfill}

\textfloatsep\floatsep

\abovedisplayskip\smallskipamount
\belowdisplayskip\smallskipamount
\abovedisplayshortskip\smallskipamount
\belowdisplayshortskip\smallskipamount

\sectionbeforeskip{1.5\bigskipamount}
\sectionstyle{\centering\normalsize\bf}
\sectionafterskip{\bigskipamount}
\sectionafterindenttrue

\subsectionbeforeskip{\bigskipamount}
\subsectionstyle{\normalsize\bf}
\subsectionafterskip{.5\bigskipamount}
\subsectionafterindenttrue

\paragraphbeforeskip{\bigskipamount}
\paragraphstyle{\centering\normalsize\sl}
\paragraphafterskip{.5\bigskipamount}
\paragraphafternewlinetrue
\paragraphafterindenttrue
\paragraphdot{}

\statementintrostyle{\sc}
\renewtheorem[{\ns}{}]{definition}[nstat]{Definition}
\renewtheorem[{\ns}{}]{remark}[nstat]{Remark}
\newtheorem[{\ns}{}]{block}[nstat]{}

\renewtheorem{nstat}{\nstatname}[subsection]

\makeatletter

\@addtoreset{figure}{subsection}
\@addtoreset{table}{subsection}
\let\c@table\c@figure
\makeatother

\captionstyle{\small}
\captionintrostyle{\sc}

\oldFigure
\showfigurestrue

\def~{\kern0.33em}

\def\appendix#1{\refstepcounter{section}
  \section*{Appendix \thesection. #1}
  \setcounter{theorem}{0}}

\def\tocfill{{\ms\ \leaders\hbox{. }\hfill}}
\def\tocpageref#1{\hbox to 12pt{\hss\pageref{#1}}}

\let\ms\normalshape
\let\boldsymbol\bold
\let\nine\footnotesize

\let\mycal\cal
\def\cal#1{{\mycal #1}}

\let\mymathrm\mathrm
\def\mathrm#1{{\mymathrm #1}}

\let\texbf\bf
\def\bf{\texbf\boldmath}

\makeatletter

\def\rightarrow{\ifx\math@version\@@bold
  \let\@rightarrow\@@@rightarrow\else\let\@rightarrow\@@rightarrow\fi\@rightarrow}

\def\@@rightarrow%
  {\mathrel{\mkern-0.5mu\relbar\mkern-9.625mu%
            \text{\ft{msam10}9\kern-0.325em\ft{lasy10})}}}
       
\mathchardef\@@@rightarrow="3221

\def\leftarrow{\ifx\math@version\@@bold
  \let\@leftarrow\@@@leftarrow\else\let\@leftarrow\@@leftarrow\fi\@leftarrow}

\def\@@leftarrow%
  {\mathrel{\text{\ft{lasy10}(\kern-0.325em\ft{msam10}9}%
            \mkern-9.625mu\relbar\mkern-0.5mu}}
            
\mathchardef\@@@leftarrow="3220

\def\mapsto%
  {\mathrel{\mapstochar\mkern-0.5mu\relbar\mkern-9.625mu%
            \text{\ft{msam10}9\kern-0.325em\ft{lasy10})}}}
            
\def\downarrow{\ifx\math@version\@@bold
  \let\@downarrow\@@@downarrow\else\let\@downarrow\@@downarrow\fi\@downarrow}

\def\@@downarrow%
  {\mathrel{\text{%
            \put[.1425em,0.61ex]{\ft{msbm10}{\char112}}%
            \put[.1425em,-0.3ex]{\ft{msbm10}{\char112}}%
            \put[0em,-.455ex]{\ft{lasy10}+}\kern1.16ex}}}
            
\def\@@@downarrow{\delimiter"3223379 }            
                
\def\uparrow{\ifx\math@version\@@bold
  \let\@uparrow\@@@uparrow\else\let\@uparrow\@@uparrow\fi\@uparrow}

\def\@@uparrow%
  {\mathrel{\text{%
            \put[.1435em,0.51ex]{\ft{msbm10}{\char112}}%
            \put[.1435em,-0.4ex]{\ft{msbm10}{\char112}}%
            \put[0em,1.615ex]{\ft{lasy10}*}\kern1.16ex}}}

\def\@@@uparrow{\delimiter"3222378 }
            
\let\to\rightarrow

\makeatother

\def\({\gdef\nextsp{}\mbox\bgroup{\ns(}\sl\aux}
\def\aux#1{\def\tempx{#1}\let\next\aux%
   \if\tempx)\let\next\egroup{\/\nextsp\ns)}%
   \else\if\tempx f{\nextsp f\gdef\nextsp{\kern0.2ex}}%
   \else\if\tempx i{\nextsp\kern0.1ex i\gdef\nextsp{\kern0.1ex}}%
   \else\if\tempx j{\nextsp\kern0.2ex j\gdef\nextsp{\kern0.1ex}}%
   \else\if\tempx l{\nextsp\kern0.1ex l\gdef\nextsp{\kern0.2ex}}%
   \else\if\tempx I{\nextsp I\gdef\nextsp{\kern0.2ex}}%
   \else\if\tempx'{\/$\mkern0.5mu'$\gdef\nextsp{}}%
   \else\if\tempx-{\/\ns-\gdef\nextsp{}}%
   \else\nextsp\tempx\gdef\nextsp{}%
   \fi\fi\fi\fi\fi\fi\fi\fi\next}

\makeatletter
\def\up{\@ifnextchar[{\@up}{\mathop{\uparrow}\nolimits}}
\def\@up[#1]{{\uparrow}\text{\raise .6ex\hbox{$_#1$}}}
\def\down{\@ifnextchar[{\@down}{\mathop{\downarrow}\nolimits}}
\def\@down[#1]{{\downarrow}\text{\raise .6ex\hbox{$_#1$}}}
\makeatother

\newcommand{\mapright}[1]
  {\smash{\mathop{\longrightarrow}\limits^{\text{\hbox to 0pt{\hss$#1$\hss}}}}}

\renewcommand{\_}{{\hbox to 1.2ex{\hss\vrule width1ex height0pt depth.4pt\hss}}}
\newcommand{\bs}{\text{\raise.4ex\hbox{\bf$\scriptscriptstyle\backslash$}}}

\newcommand\tp[2]{(#1\phantom{,}#2)}

\newcommand\red%
  {\mkern0.2mu\text{\ft{msam10}9\kern-0.2em9\kern-0.325em\ft{lasy10})}\mkern0.4mu}
\newcommand\seq{\Pi\mkern1mu}

\newcommand{\diam}{\diamond}
\newcommand{\rdiam}{\mathbin{\diamond\mkern-6.5mu\diamond}}

\let\emptyset\varemptyset
\let\theta\textheta

\let\bar\widebar
\let\hat\widehat
\let\tilde\widetilde

\newcommand{\Alg}{{\cal A\mkern0.6mu l\mkern-1.2mu g}}

\newcommand{\Chb}{{\cal C\mkern-1.8mu h\mkern-0.3mu b}}

\newcommand{\Cob}{{\cal C\mkern-1.1mu o\mkern-0.2mu b}}

\newcommand{\Cobt}{\tilde{\vrule height1.4ex width0pt \smash{\Cob}}{}}
\newcommand{\bfCobt}{\text{\put[-0.07pt,0.2pt]{%
  $\tilde{\phantom{\vrule height1.4ex width0pt \smash{\Cob}}}$}}%
  \tilde{\vrule height1.4ex width0pt \smash{\Cob}}{}}

\newcommand{\Vec}{\mathrm{Vec}}

\newcommand{\Kb}{\bar\K}
\newcommand{\Kbb}{\bar\Kb}
\newcommand{\Sb}{\bar\S}
\newcommand{\Sbb}{\bar\Sb}
\newcommand{\Hb}{{\mkern1mu\bar{\mkern-1mu{\H}}}{}}
\newcommand{\Hbb}{{\mkern1mu\bar{\bar{\mkern-1mu{\H}}}}{}}

\newcommand{\Sbbh}{{\mkern3.6mu\hat{\mkern-3.6mu\Sbb}}{}}
\newcommand{\Kbbh}{{\mkern3.8mu\hat{\mkern-3.8mu\Kbb}}{}}

\newcommand{\Obj}{\mathop{\mathrm{Obj}}\nolimits}
\newcommand{\Mor}{\mathop{\mathrm{Mor}}\nolimits}

\newcommand{\id}{\mathrm{id}}
\newcommand{\one}{{\boldsymbol 1}}

\newcommand{\Int}{\mathop{\mathrm{Int}}\nolimits}
\newcommand{\Bd}{\mathop{\mathrm{Bd}}\nolimits}
\newcommand{\Cl}{\mathop{\mathrm{Cl}}\nolimits}
\newcommand{\fr}{\mathop{\mathrm{fr}}\nolimits}
\newcommand{\lk}{\mathop{\mathrm{lk}}\nolimits}

\newcommand{\pr}{\mathrm{pr}}

\newcommand{\ad}{\mathop{\mathrm{ad}}\nolimits}

\newcommand{\rot}{\mathop{\mathrm{rot}}\nolimits}

\newcommand{\sym}{\mathop{\mathrm{sym}}\nolimits}
\newcommand{\inp}{\mathrm{in}}
\newcommand{\out}{\mathrm{out}}

\newcommand{\SO}{\mathrm{SO}}
\newcommand{\Spin}{\mathrm{Spin}}

\newcommand{\B}{{\cal B}}
\newcommand{\C}{{\cal C}}
\newcommand{\D}{{\cal D}}
\newcommand{\G}{{\cal G}}
\renewcommand{\H}{{\cal H}}
\newcommand{\K}{{\cal K}}

\renewcommand{\S}{{\cal S}}
\newcommand{\T}{{\cal T}}

\newcommand{\bH}{\mkern2mu\bar{\mkern-2mu H}}
\newcommand{\bTheta}{\mkern-1mu\bar{\mkern1mu \Theta\mkern-1.4mu}\mkern1.4mu}
\newcommand{\bPhi}{\mkern-1mu\bar{\mkern1mu \Phi\mkern-1.6mu}\mkern1.6mu}
\newcommand{\bPsi}{\mkern-1mu\bar{\mkern1mu \Psi\mkern-1.6mu}\mkern1.6mu}
\newcommand{\bXi}{\mkern-1mu\bar{\mkern1mu \Xi\mkern-1.6mu}\mkern1.6mu}

\newcommand{\bbPhi}{\mkern-1.27mu\bar{\mkern1.27mu{\bPhi}\mkern-2.03mu}\mkern2.03mu}
\newcommand{\bbPsi}{\mkern-1.25mu\bar{\mkern1.25mu{\bPsi}\mkern-2.06mu}\mkern2.06mu}
\newcommand{\bbTheta}{\mkern-1.27mu\bar{\mkern1.27mu{\bTheta}\mkern-1.84mu}\mkern1.84mu}
\newcommand{\bbXi}{\mkern-1.29mu\bar{\mkern1.29mu{\bXi}\mkern-2.04mu}\mkern2.04mu}

\begin{document}

\title{\large\bf
ON FOUR-DIMENSIONAL 2-HANDLEBODIES\\ AND THREE-MANIFOLDS
\label{Version 1.0 / \today}}
\author{\normalsize\sc I. Bobtcheva\\
\normalsize\sl Dipartimento di Scienze Matematiche\\[-3pt]
\normalsize\sl Universit\`a Politecnica delle Marche -- Italia\\
\small\tt bobtchev@dipmat.univpm.it
\and
\normalsize\sc R. Piergallini\\
\normalsize\sl Scuola di Scienze e Tecnologie\\[-3pt]
\normalsize\sl Universit\`a di Camerino -- Italia\\
\small\tt riccardo.piergallini@unicam.it}
\date{}

\maketitle

\bigskip

\begin{abstract}
\baselineskip13.5pt
\medskip

We show that for any $n \geq 4$ there exists an equivalence functor $\S_n^c \to
\Chb^{3+1}$ from the category $\S_n^c$ of $n$-fold connected simple coverings of $B^3
\times [0, 1]$ branched over ribbon surface tangles up to certain local ribbon moves, and
the cobordism category $\Chb^{3+1}$ of orientable relative 4-dimensional 2-handlebody
cobordisms up to 2-deformations.

As a consequence, we obtain an equivalence theorem for simple coverings of $S^3$ branched
over links, which provides a complete solution to the long-standing Fox-Montesinos
covering moves problem. This last result generalizes to coverings of any degree results by
the second author and Apostolakis, concerning respectively the case of degree 3 and 4. We
also provide an extension of the equivalence theorem to possibly non-simple coverings of
$S^3$ branched over embedded graphs.

Then, we factor the functor above as $\S_n^c \to \H^r \to \Chb^{3+1}$, where
$\S_n^c \to \H^r$ is an equivalence functor to a universal braided category $\H^r$ freely
generated by a Hopf algebra object $H$. In this way, we get a complete algebraic
description of the category $\Chb^{3+1}$. From this we derive an analogous description of 
the category $\Cobt^{2+1}$ of 2-framed relative 3-dimensional cobordisms, which
resolves a problem posed by Kerler.



\end{abstract}
\newpage

\section*{Contents}

\vskip-\lastskip
\begin{itemize}\itemsep6pt

\item[]{\bf Introduction
  \tocfill\tocpageref{introduction/sec}}

\item[\bf\ref{preliminaries/sec}.]
  {\bf Preliminaries
  \tocfill\tocpageref{preliminaries/sec}}
\begin{itemize}\itemsep0pt
\item[\ref{links/sec}.]
  {Links and diagrams
  \tocfill\tocpageref{links/sec}}
\item[\ref{handles/sec}.]
  {Handlebodies
  \tocfill\tocpageref{handles/sec}}
\item[\ref{ribbons/sec}.]
  {Ribbon surfaces
  \tocfill\tocpageref{ribbons/sec}}
\item[\ref{coverings/sec}.]
  {Branched coverings
  \tocfill\tocpageref{coverings/sec}}
\item[\ref{categories/sec}.]
  {Categories
  \tocfill\tocpageref{categories/sec}}
\end{itemize}\vskip-\lastskip

\item[\bf\ref{cobordisms/sec}.]
  {\bf 4-dimensional cobordisms and Kirby tangles
  \tocfill\tocpageref{cobordisms/sec}}
\begin{itemize}\itemsep0pt
\item[\ref{Chbn/sec}.]
  {The categories of relative handlebody cobordisms $\Chb_n^{3+1}$
  \tocfill\tocpageref{Chbn/sec}}
\item[\ref{Kirby/sec}.]
  {Bridged tangles and labeled Kirby tangles
  \tocfill\tocpageref{Kirby/sec}}
\item[\ref{K/sec}.]
  {The categories $\K_n$ and the functors $\up_k^n$ and $\down_k^n$
  \tocfill\tocpageref{K/sec}}
\end{itemize}\vskip-\lastskip

\item[\bf\ref{surfaces/sec}.]
  {\bf Labeled ribbon surface tangles
  \tocfill\tocpageref{surfaces/sec}}
\begin{itemize}\itemsep0pt
\item[\ref{rs-tangles/sec}.]
  {The category $\S$ of ribbon surface tangles
  \tocfill\tocpageref{rs-tangles/sec}}
\item[\ref{S/sec}.]
  {The categories $\S_n$ and the functors $\up_k^n$
  \tocfill\tocpageref{S/sec}}
\item[\ref{Theta/sec}.]
  {The functors $\Theta_n: \S_n \to \K_n$
  \tocfill\tocpageref{Theta/sec}}
\item[\ref{fullness/sec}.]
  {Fullness of $\Theta_n: \S^c_n \to \K^c_n$ for $n \geq 3$
  \tocfill\tocpageref{fullness/sec}}
\item[\ref{Xi/sec}.]
  {The functor $\Xi_n: \K_1 \to \S^c_n$ for $n \geq 4$
  \tocfill\tocpageref{Xi/sec}}
\item[\ref{K=S/sec}.]
  {Equivalence between $\K^c_n$ and $\S^c_n$ for $n \geq 4$
  \tocfill\tocpageref{K=S/sec}}
\end{itemize}\vskip-\lastskip

\item[\bf\ref{algebra/sec}.]
  {\bf Universal groupoid ribbon Hopf algebra
  \tocfill\tocpageref{algebra/sec}}
\begin{itemize}\itemsep0pt
\item[\ref{HG/sec}.]
  {The universal groupoid Hopf algebras $\H(\G)$ and $\H^u(\G)$
  \tocfill\tocpageref{HG/sec}}
\item[\ref{HrG/sec}.]
  {The universal groupoid ribbon Hopf algebra $\H^r(\G)$
  \tocfill\tocpageref{HrG/sec}}
\item[\ref{Phi/sec}.]
  {The functors $\Phi_n: \H^r_n \to \K_n$
  \tocfill\tocpageref{Phi/sec}}
\item[\ref{adjoint/sec}.]
  {The adjoint morphisms
  \tocfill\tocpageref{adjoint/sec}}
\item[\ref{reduction/sec}.]
  {The stabilization and reduction functors $\up_X$ and $\down_X$
  \tocfill\tocpageref{reduction/sec}}
\item[\ref{Psi/sec}.]
  {The functors $\Psi_n: \S_n \to \H_n^r$
  \tocfill\tocpageref{Psi/sec}}
\item[\ref{K=H/sec}.]
  {Equivalence between $\K_n^c$ and $\H_n^{r,c}$
  \tocfill\tocpageref{K=H/sec}}
\end{itemize}\vskip-\lastskip

\item[\bf\ref{boundaries/sec}.]
  {\bf 3-dimensional cobordisms as boundaries
  \tocfill\tocpageref{boundaries/sec}}
\begin{itemize}\itemsep0pt
\item[\ref{Cob/sec}.]
  {The categories of relative cobordisms $\bfCobt^{2+1}_n$ and $\Cob^{2+1}_n$
  \tocfill\tocpageref{Cob/sec}}
\item[\ref{quotients/sec}.]
  {The quotient categories $\Kb_n$, $\Kbb_n$, $\Sb_n$ and $\Sbb_n$
  \tocfill\tocpageref{quotients/sec}}
\item[\ref{Kb=Sb/sec}.]
  {Equivalences $\Kb_n^c \cong \Sb_n^c$ and $\Kbb_n^c \cong \Sbb_n^c$ for $n \geq 4$
  \tocfill\tocpageref{Kb=Sb/sec}}
\item[\ref{Hb/sec}.]
  {The quotient categories $\Hb_n$ and $\Hbb_n$
  \tocfill\tocpageref{Hb/sec}}
\item[\ref{Kb=Hb/sec}.]
  {Equivalences $\Kb_n^c \cong \Hb_n^{r,c}$ and $\Kbb_n^c \cong \Hbb_n^{r,c}$
  \tocfill\tocpageref{Kb=Hb/sec}}
\end{itemize}\vskip-\lastskip

\item[\bf\ref{covering-moves/sec}.]
  {\bf Branched coverings of $B^4$ and $S^3$
  \tocfill\tocpageref{covering-moves/sec}}
\begin{itemize}\itemsep0pt
\item[\ref{B4/sec}.]
  {Covers of $B^4$ simply branched over ribbon surfaces
  \tocfill\tocpageref{B4/sec}}
\item[\ref{S3/sec}.]
  {Equivalence of branched covers of $S^3$
  \tocfill\tocpageref{S3/sec}}
\end{itemize}\vskip-\lastskip

\item[]{\bf References
  \tocfill\tocpageref{references/sec}}

\end{itemize}

\newpage

\section*{Introduction%
\label{introduction/sec}}

The present work is based on the preprints \cite{BP04} and \cite{BP05} and extends
their main results from closed manifolds to the category of cobordisms. In particular, it
is the synthesis of years-long search for the answers to the following distinct problems
in the topology of 3-manifolds:

\begin{statement}{Problem A}
Find a finite set of local moves that relate any two labeled links representing the
same 3-manifold as a simple branched covering of $S^3$.
\end{statement}

\begin{statement}{Problem B}
Find a universal monoidal braided category freely generated by a Hopf algebra object,
which is equivalent to the category $\Cobt^{2+1}\!$ of 2-framed relative 3-cobordisms.
\end{statement}

Problem A is an old one, having been formulated by Fox and Montesinos in the seventies.
It naturally arose in relation to the Hilden-Hirsch-Montesinos theorem
\cite{Hi74,Hs74,Mo74}, which states that any closed connected oriented 3-manifold can be
realized as a 3-fold covering of $S^3$ branched over a link (actually a knot). Such a
covering can be described in terms of its branching link, with each arc labeled by the
monodromy of the corresponding meridian (a transposition in the symmetric group
$\Sigma_n$). Hence, Problem A just expresses the equivalence problem for branched covers
in terms of moves for labeled links, usually called covering moves.

In \cite{Mo85}, Montesinos conjectured that the two covering moves \(M1) and \(M2) in
Figure \ref{covering-moves05/fig} suffice to relate 3-fold simple coverings of $S^3$
representing the same 3-manifold, up to 4-fold stabilization. This was proved to be true
by the second author in \cite{Pi95}, where the question was also posed, whether these
local moves together with stabilization also suffice for simple coverings of arbitrary
degree. In \cite{Ap03} Apostolakis answered this question in the positive for 4-fold
coverings, up to 5-fold stabilization.

The complete solution of Problem A is given by Theorem \ref{equiv3/thm}, asserting
that labeled isotopy and Montesinos moves suffice to relate any two $n$-fold branched
covering representations of the same 3-manifold, provided that $n \geq 4$. The proof is
independent on all the above-mentioned partial results, being based on the branched
covering interpretation of the generalized Kirby calculus developed in Chapter
\ref{surfaces/sec} and Section \ref{Kb=Sb/sec}.

We note that this result and its proof were already contained in the preprint \cite{BP04}
(see below for some already published applications). However, the exposition in the
present work is adapted to the category of cobordisms, as a preliminary result for
answering Problem B.

\medskip

Problem B was risen by Kerler in \cite{Ke02} (cf. \cite[Problem 8-16 (1)]{Oh02}). In that
paper, the author considers a universal monoidal braided category $\Alg$, freely generated
by a Hopf algebra object and a full (surjective) functor $\Alg \to \Cobt^{2+1}$. Moreover,
he poses the challenge to find a set of additional relations for $\Alg$, such that the
above functor induces a category equivalence on the quotient of $\Alg$ by the new
relations. 

This problem is related to the origin of the $(2+1)$-dimensional TQFT's (Topological
Quantum Field Theories). We remind that any such theory is a functor from $\Cobt^{2+1}$ to
the category of vector spaces $\Vec_k$ over a field $k$. The first constructions of TQFT's
were based on semisimple quotients of the representation spaces of quasi\-triangular
ribbon Hopf algebras (see \cite{RT91}). Such approach was generalized in \cite{KL01},
where a TQFT was associated to any modular category with a special braided Hopf algebra in
it. The solution to Problem B implies that any TQFT's lives on the representation
space of a braided Hopf algebra. Of course, the braided Hopf algebra axioms alone are too
weak. Indeed, in the definition of $\Alg$, Kerler adds the requirement that there exist a
ribbon and a non-degenerate Hopf copairing morphisms, satisfying some extra axioms. 

Here, we prove that in order to have that the quotient algebraic category is equivalent to
$\Cobt^{2+1}$, two more axioms suffice in addition to the ones presented by Kerler in
\cite{Ke02} (cf. Theorem \ref{alg-kirby-eq-bd/thm}). The first one describes the
propagation of the ribbon element through the comultiplication morphism, while the second
one relates the copairing and the braiding morphisms. We call the resulting quotient
algebra {\sl the universal self-dual ribbon Hopf algebra} and denoted it by $\Hb^r$. The
complete list of its relations, including the ones of Kerler, is presented in Tables
\ref{table-Hr1axioms/fig} and \ref{table-Hrbb1/fig}. We note that the algebra $\Hb^r$
(with all its relations) was already defined in the preprint \cite{BP05}. Moreover, it was
shown there that the functor $\Hb^r\to \Cobt^{2+1}$ induces a bijective map on the closed
morphisms. So, the only new point in the present work is the generalization of this result
to all morphisms, completing in this way the answer to Problem B.

\medskip

As it should be clear from the discussion above, Problems A and B originate from different
contexts and at first glance they seem quite distant. Yet, an indication that they may be
related comes from the very reason for which they are posed. Their solution leads in both
cases to a complete diagrammatic language for describing the concerned topological
objects. Such languages are respectively based on planar diagrams of links labeled by
transpositions in $\Sigma_n$, and on planar diagrams of certain decorated uni- and
tri-valent graphs describing morphisms in the corresponding Hopf algebra. Both kinds of
diagrams are taken modulo a finite set of local moves, i.e. moves which only change a
given portion of the diagram inside a disk in a way independent on its outside. The fact
that the equivalence moves are finite in number and local in nature, is an important
feature in view of the definition of invariants.

\medskip

Of course, the reason for putting the two results together in the present work is not the
philosophical similarity in the statements, but the fact that the proofs are intrinsically
related. In both cases, we go up one dimension and discuss the following problems on the
cobordism category $\Chb^{3+1}$ of relative 4-dimensional 2-handlebodies up to
2-deformations (handle slidings and creation/cancelation of 1/2-handle pairs).

\begin{statement}{Problem A$'$}
Find a finite set of local moves which relate any two labeled ribbon surface tangles
representing the same cobordism in $\Chb^{3+1}$ as a simple branched covering of $B^3
\times [0,1]$.
\end{statement}

\begin{statement}{Problem B$'$}
Find a universal monoidal braided category freely generated by a Hopf algebra object,
which is equivalent to the cobordisms category $\Chb^{3+1}\!$.
\end{statement}
 
Any cobordism in $\Cobt^{2+1}$ represents the framed boundary of one in $\Chb^{3+1}$.
Moreover, $\Cobt^{2+1}$ is equivalent to the quotient of $\Chb^{3+1}$ modulo 1/2-handle
tradings. This allows us to derive the solutions of Problems A and B from those of
Problems A$'$ and B$'$ respectively.

\medskip

If we think of the handlebodies in $\Chb^{3+1}$ as build on a single 0-handle, a standard
way of describing them is through Kirby tangles, representing the attaching maps of the
corresponding 1- and 2-handles in the boundary of the 0-handle. Such tangles, modulo
isotopy and 1/2-handle moves, form a category $\K$ equivalent to $\Chb^{3+1}$.

Even if our goal is the description of the morphisms in $\Chb^{3+1}$, as an intermediate
step we will need to work with the category of relative 4-dimensional 2-handlebodies build
on $n$ 0-handles. In Section \ref{K/sec}, we introduce the corresponding category $\K_n$
of generalized admissible Kirby tangles, which describes such handlebodies through the
attaching maps of their handles in the boundaries of the $n$ 0-handles. Clearly, $\K_n$
with $n > 1$ is much ``bigger'' then $\K = \K_1$, but for any $n > k \geq 1$ there exists
an injective stabilization functor $\up_k^n: \K_k \to \K_n$. The restriction of this
functor to the subcategories $K_k^c \subset \K_k$ and $K_n^c \subset \K_n$ of
``connected'' cobordisms is invertible, i.e. there exists a reduction functor $\down_k^n:
\K_n^c \to \K_k^c$ such that ${\down_k^n} \circ {\up_k^n} = \id$, while ${\up_k^n} \circ
{\down_k^n} \simeq \id$ up to natural transformation. In particular, for any $n \geq 2$,
the category $\K_n^c$ is equivalent to $\K^c_1 = \K_1$.

The reason for considering 4-dimensional 2-handlebodies build on $n$ 0-handles is that
they naturally occur when representing handlebodies in $\Chb^{3+1}$ as simple $n$-fold ($n
\geq 2$) coverings of $B^3 \times [0,1]$ branched over ribbon surfaces. In fact, we will
use such branched covering representation of handlebodies in $\Chb^{3+1}$ to solve 
problems A$'$ and B$'$, according to the following scheme.

\medskip

In Chapter \ref{surfaces/sec}, we construct the category $\S_n$ of ribbon surface tangles,
labeled in the symmetric group $\Sigma_n$, up to certain labeled isotopy moves, called
1-isotopy moves, and two covering moves. Analogously to $\K_n$ such category describes
possibly disconnected cobordisms. Indeed, there is a naturally defined functor $\Theta_n:
\S_n \to \K_n$. Then, we consider the subcategory $\S_n^c \subset \S_n$ which describes
connected cobordisms through the restriction $\Theta_n: \S_n^c \to \K_n^c$. Finally, we
define a functor $\Xi_n: \K_1 \to \S_n^c$ for $n \geq 4$, and show that ${\down_1^n} \circ
\Theta_n$ and $\Xi_n$ are inverse to each other up to natural equivalences. Therefore, we
have the following diagram of category equivalences, which solves Problem A$'$.

\vskip9pt
\centerline{\epsfbox{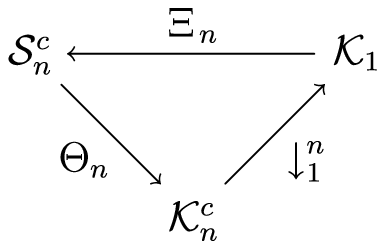}}
\vskip12pt

As a consequence, Theorem \ref{equiv4/thm} provides local moves (see Figures
\ref{covering-moves01/fig} and \ref{covering-moves02/fig}) to relate any two labeled
ribbon surfaces representing the same connected oriented 4-dimensional 2-handlebody (up to
2-equivalence) as $n$-fold simple covering of $B^4$, for $n \geq 4$. Based on this result
and passing to suitable quotient categories representing the boundary, we obtain the
above-mentioned answer to Problem A.

\medskip

In order to answer Question B, in Chapter \ref{algebra/sec} we introduce the algebraic
analog of the category $\K_n$, which is the universal ribbon $\G_n$-Hopf algebra $\H_n^r$.
Here $\G_n$ is the groupoid with the numbers $1, \dots, n$ as the objects and with a
unique morphism for each ordered pair of objects. In particular, we define a functor
$\Phi_n: \H_n^r \to \K_n$, which is the analog of Kerler's functor $\Alg \to \Cobt^{2+1}$
in the context of relative 4-dimensional 2-handlebodies build on multiple 0-handles.
Analogously to $\K_n$, for any $n > k \geq 1$ there exist a stabilization functor
${\up_k^n}: \H_k^r \to \H_n^r$, which is invertible (up to natural transformations) on a
subcategory $\H_n^{r,c}$, through the reduction functor ${\down_k^n}: \H_n^{r,c} \to
\H_k^{r,c}$. In particular, for any $n \geq 2$, the category $\H_n^{r,c}$ is equivalent to
$\H_1^{r,c} = \H_1^r$. Moreover, for any $n > k \geq 1$ we have that ${\down_k^n} \circ
\Phi_n = \Phi_k \circ {\down_k^n}: \H_n^r \to \K_k$.

Our goal is to show that $\Phi_1: \H^r \to \K$ is a category equivalence. We prove this
by factoring the functor $\Theta_n: \S_n \to \K_n$ as $\Phi_n\circ \Psi_n$ with $\Psi_n:
\S_n \to \H_n^r$, and showing that for any $n \geq 4$ the restriction $\Psi_n: \S_n^c \to
\H_n^{r,c}$ is a category equivalence. This fact and the commutative diagram below imply 
that the same is true for $\Phi_k: \H_k^{r,c} \to \K_k^c$ for any $k \geq 1$.

\vskip9pt
\centerline{\epsfbox{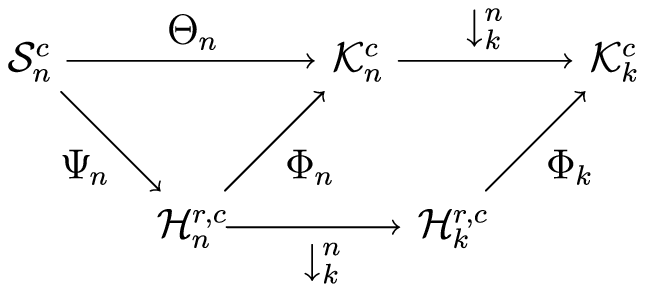}}
\vskip12pt

In this way, we get a completely algebraic description of the category $\Chb^{3+1}$ of
connected 4-dimensional 2-handlebodies up to 2-equivalence, as the universal ribbon
category $\H^r = \H_1^r$ (cf. Theorem \ref{ribbon-kirby/thm}). This solves Problem B$'$. 
Then, the answer to Problem B is obtained in Section \ref{Kb=Hb/sec}, by passing
once again to suitable quotient categories $\Cobt^{2+1}$ and $\Hb^r$.

\smallskip

We would like to make few comments about applications of the present work. 
\begin{itemize}
\item[{\sl a}\/)] 
The solution of Problem A, i.e. the description of 3-dimensional manifolds in terms of
labeled links modulo isotopy and covering moves, has already been ap\-plied in \cite{Ht10}
and \cite{No11} to the construction of invariants of 3-manifolds, and in \cite{DMA10} to
the construction of convolution algebras of spin networks and spin foams.
\item[{\sl b}\/)] 
The branched covering representation of 4-dimensional 2-handlebodies provided by Theorem
\ref{equiv4/thm} is used in \cite{APZ11}, to relate the monodromy descriptions of any two
topological Lefschetz fibrations over the disk having 2-equivalent total spaces.
\item[{\sl c}\/)]
The notion of 2-deformation of 4-dimensional 2-handlebodies is conjectured to be different
from the one of diffeomorpism, which in this context is equivalent to 3-deformation (cf.
Section \ref{handles/sec} and Remark \ref{1-isotopy/rem}). The solution of Problem~B$'$
associates to any braided ribbon Hopf algebra an invariant of 4-dimensional 2-handlebodies
under 2-deformations and implies that, if the conjecture is true, there should exist a
braided ribbon Hopf algebra whose invariant distinguishes diffeomorphic but not
2-equivalent handlebodies. The search for such Hopf algebras is a non-trivial challenge,
since they have to combine the properties of being unimodular, not self-dual and not
semisimple (see the discussion in \cite{BM03} about the the HKR-type invariants associated
to an ordinary ribbon Hopf algebra).
\item[{\sl d}\/)] 
By restricting the map ${\down_1^2} \circ \Psi_2$ to double branched covers of $B^4$, i.e.
to ribbon surfaces labeled with the single permutation $(1\;2)$, one obtains an invariant
of ribbon surfaces under 1-isotopy moves, taking values in $\H^r$. We remind that the
description of all the moves relating isotopic ribbon surfaces is still an open question.
We conjecture here (see Remark \ref{1-isotopy/rem}) that 1-isotopy is weaker than isotopy
and that braided ribbon Hopf algebras should detect such a difference.
\item[{\sl e}\/)] 
In chapter \ref{algebra/sec} we introduce and study the general concept of a groupoid
ribbon Hopf algebra, even if it is being used later only in the case of the specific and
very simple groupoid $\G_n$. The reasons for doing this are two. The first one is that
working with the general case does not make heavier the algebraic part, actually it makes
it easier to follow. The second one is that, in our believe, the group ribbon Hopf algebra
(which is another particular case of the construction) should be useful in finding an
algebraic description of other types of topological objects, for example the group
manifolds studied in \cite{Vi01}.
\end{itemize}

\medskip

As we already said, the present work is based on the preprints \cite{BP04} and
\cite{BP05}. The editor of the journal to which we submitted those preprints for
publication, suggested to put them together in a single monograph. Moreover, the referee
of \cite{BP05} observed that it should not be too difficult to extend the result there to
cobordisms. Indeed, we have been able to achieve such extension by following the main line
of the proof in \cite{BP05} and overcoming some technical obstacles in the definition of
the natural transformations. Following the indications of the referee, we have also
expanded Section \ref{cobordisms/sec} with the definition of the topological category and
tried to improve the exposition of the algebraic part. In particular, we have added
Section \ref{adjoint/sec}, where we describe the properties of the adjoint morphisms and
express the two new axioms for $\H^r$ in terms of such action. This allows to simplify the
proofs in Section \ref{reduction/sec}, even if they still remain highly technical.

\newpage

\section{Preliminaries%
\label{preliminaries/sec}}

In this chapter we collect some preliminary definitions and results. Some are known and we
just recall them in order to establish terminology and notations. Some others are new and
we include them here rather than in a specific chapter, since they are widely used in the 
following or they seem to have interest in their own right, independently from their
application in the present context.

In particular: in Section \ref{links/sec} we define the notion of vertically trivial
state for a link diagram and prove an elementary property of such states, that makes them
much more manageable than trivial states; in Section \ref{handles/sec} we briefly
discuss relative handlebodies build on a manifold with boundary; in Section
\ref{ribbons/sec} we define the 1-isoto\-py relation for ribbon surfaces and express it 
in terms of certain moves of planar diagrams; in Section \ref{coverings/sec} we introduce 
ribbon moves for labeled ribbon surfaces representing simple branched coverings of $B^4$.

\subsection{Links and diagrams%
\label{links/sec}}

As usual, we represent a link $L \subset R^3 \subset R^3 \cup \infty \cong S^3$ by a
planar {\sl diagram} $D \subset R^2$, consisting of the orthogonal projection of $L$ into
$R^2$, that can be assumed self-transversal after a suitable horizontal (height
preserving) isotopy of $L$, with a {\sl crossing state} for each double point, telling
which arc passes over the other one. Such a diagram $D$ uniquely determines $L$ up to
vertical isotopy. On the other hand, link isotopy can be represented in terms of diagrams
by crossing preserving isotopy in $R^2$ and Reidemeister moves.

A link $L$ is called {\sl trivial} if it bounds a disjoint union of disks in $R^3$. It is
well-known that any link diagram $D$ can be transformed into a diagram $D'$ of a trivial
link by suitable crossing changes, that is by inverting the state of some of its
crossings. We say that $D'$ is a {\sl trivial state} of $D$. Actually, any non-trivial
link diagram $D$ has many trivial states, but it is not clear at all how they are related
to each other. For this reason, we are lead to introduce the more restrictive notions of
vertically trivial link and vertically trivial state of a link diagram.

\begin{definition}\label{vert-trivial/def}
We say that a link $L \subset R^3$ is {\sl vertically trivial} if it meets any horizontal
plane (parallel to $R^2$) in at most two points belonging to the same component.
\end{definition}

If $L$ is a vertically trivial link, then the height function separates the components of
$L$ (that is the height intervals of different components are disjoint), so that we can
vertically order the components of $L$ according to their height. Moreover, each component
can be split into two arcs on which the height function is monotone, assuming the only
unique minimum and maximum values at the common endpoints. Then, all the (possibly
degenerate) horizontal segments spanned by $L$ in $R^3$ form a disjoint union of disks
bounded by $L$. This proves that $L$ is a trivial link.

\begin{definition}\label{vert-state/def}
By a {\sl vertically trivial state} of a link diagram $D \subset R^2$ we mean any trivial
state of $D$ which is the diagram of a vertically trivial link.
\end{definition}

A vertically trivial state $D'$ of $D$ can be constructed by the usual naive unlinking
procedure: 1) number the components of the link $L$ represented by $D$ and fix on each
component an orientation and a starting point away from crossings; 2) order the points of
$L$ lexicographically according to the numbering of the components and then to the
starting point and the orientation of each component; 3) resolve each double point of $D$
into a crossings of $D'$ by letting the arc which comes first in the order pass under the
other one. The link $L'$ represented by $D'$ can be clearly assumed to be vertically
trivial, considering on it a height function which preserves the order induced by the
vertical bijection with $L$ except for a small arc at the end of each component. Figure
\ref{links01/fig} \(a) shows how the height function of a component looks like with
respect to a parametrization having the starting point and the orientation fixed above.
Keeping the parametrization fixed but changing the starting point or the orientation we
get different height functions as in Figures \ref{links01/fig} \(b) and \(c)
respectively.

\begin{Figure}[htb]{links01/fig}
{}{Height functions for vertically trivial knots}
\centerline{\fig{}{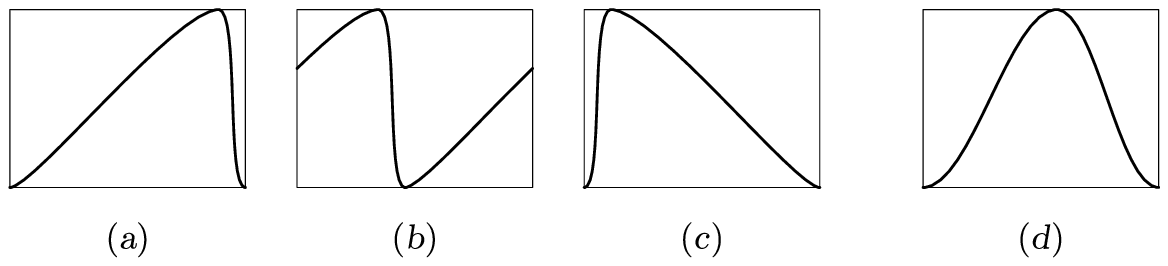}}
\vskip-6pt
\end{Figure}

Notice that the above unlinking procedure gives us only very special vertically trivial
states. While it is clear how to pass from \(a) to \(b), by moving the starting point
along the component, going from \(a) to \(c) turns out to be quite mysteri\-ous without
considering generic vertically trivial states. The height function of a component for such
a state, with respect to a parametrization starting from the unique minimum point, looks
like in Figure \ref{links01/fig} \(d), that is apparently an intermediate state between
\(a) and \(c). The following proposition settles the problem of relating different
vertically trivial states of the same link diagram.

\begin{proposition}\label{vert-state/thm}
Any two vertically trivial states $D'$ and $D''$ of a link diagram $D$ are related by a
sequence $D_0, D_1, \dots, D_n$ of vertically trivial states of $D$, such that $D_0 = D'$,
$D_n = D''$ and, for each $i = 1, \dots, n$, $D_i$ is obtained from $D_{i-1}$ by changing 
all the crossings between two vertically adjacent components or by changing a single 
self-crossing of one component. Moreover, in the latter case the singular link between 
$D_{i-1}$ and $D_i$ is trivial, meaning that the unique singular component of it spans a 
1-point union of two disks disjoint from all the other components.
\end{proposition}

\begin{proof} 
Since the effect of changing all the crossings between two vertically adjacent components
is the transposition of these components in the vertical order, by iterating this kind of
modification we can permute as we want the vertical order of all the components. Hence, we
only need to address the case of a knot diagram.

Given a knot diagram $D \subset R^2$ with double points $x_1, \dots, x_n \in R^2$, we
consider a parametrization $f:S^1 \to D$ and denote by $t'_i,t''_i \in S^1$ the two values
of the parameter such that $f(t'_i) = f(t''_i) = x_i$, for any $i = 1, \dots, n$.

For any smooth knot $K \subset R^3$ which projects to a vertically trivial state of $D$,
let $f_K: S^1 \to R^3$ be the parametrization of $K$ obtained by lifting $f$ and $h_K:S^1
\to R$ be the composition of $f_K$ with the height function. Then, $h_K$ is a smooth
function with the following properties: 1)~$h_K$ has only one minimum and one maximum;
2)~$h_K(t'_i) \neq h_K(t''_i)$, for any $i = 1, \dots, n$. In this way, the space of all
smooth knots which project to vertically trivial states of $D$ can be identified with the
space of all smooth functions $h:S^1 \to R$ satisfying properties 1 and 2.

Now, the space $\cal S$ of all smooth functions $h:S^1 \to R$ satisfying property 1 is
clearly pathwise connected, while the complement $\cal C \subset \cal S$ of property 2 is
a closed codimension 1 stratified subspace. Therefore, if $K'$ and $K''$ are knots
projecting to the vertically trivial states $D'$ and $D''$, then we can join $h_{K'}$ and
$h_{K''}$ by a path in $\cal S$ transversal with respect to $\cal C$. This, path gives
rise to a finite sequence of self-crossing changes as in the statement, one for each
transversal intersection with $\cal C$.

The second part of the proposition, follows from \cite[Theorem 1.4]{ST89} but can also be 
immediately realized by considering the union of all the (possibly degenerate) horizontal 
segments spanned by the singular component.
\end{proof}

We emphasize that the property of vertically trivial states given by the proposition,
which will play a crucial role in the proof of Lemma \ref{SK-welldef/thm}, is far from
being shared by all trivial states. For example, Figure \ref{links02/fig} shows a
(non-vertically) trivial knot diagram, that is made knotted by any crossing change
performed on it.

\begin{Figure}[htb]{links02/fig}
{}{}
\vskip3pt
\centerline{\fig{}{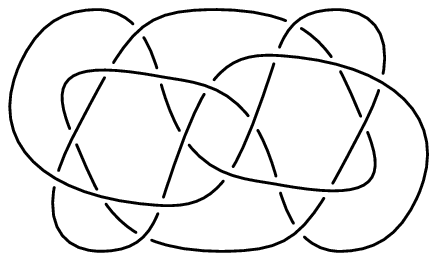}}
\end{Figure}

\subsection{Handlebodies%
\label{handles/sec}}

In this section we review some basic definitions and facts about handlebodies, referring
to \cite{GS99} for a detailed discussion of the subject. Our aim here is only to extend
the standard general set up given in \cite{GS99} to the notion of relative handlebody
build on a bounded manifold.

\medskip

Given $0 \leq i \leq d$, a {\sl $d$-dimensional $i$-handle} (or handle of {\sl index} $i$)
is a copy $H^i$ of $B^i \times B^{d-i}$ attached to a smooth $d$-manifold $W$ by a smooth
embedding $\phi: S^{i-1} \times B^{d-i} \to \Bd W$. The map $\phi$ and its image
$\phi(S^{i-1} \times B^{d-i}) \subset \Bd W$ are called respectively the {\sl attaching
map} and the {\sl attaching region} of $H^i$.

The two balls $B^i \times \{0\}$ and $\{0\} \times B^{d-i}$ are called respectively the
{\sl core} and the {\sl cocore} of $H^i$, while their boundaries $S^{i-1} \times \{0\}$
and $\{0\} \times S^{{d-i-1}}$ are called the {\sl attaching sphere} and the {\sl belt
sphere} of $H^i$ (this terminology refers to the corresponding subspaces of $W \cup_\phi
H^i$ as well). Inside $H^i$, {\sl longitudinal} means parallel to the core and {\sl
transversal} means parallel to the cocore.

By smoothing the corners in a canonical way, $W \cup_\phi H^i$ becomes a smooth
$d$-manifold, which only depends on the isotopy class of the attaching map $\phi$.

\medskip

We work only with 0- and 1-handles in dimension $d \leq 4$ and with 2-handles in dimension
$d = 4$. So, let us have a closer look at these kinds of handles.

The 0-handles are topologically trivial. In fact, attaching a 0-handle $H^0$ to $W$ is
the same as taking the disjoint union $W \sqcup H^0 \cong W \sqcup B^d$.

For a 1-handle $H^1$, the attaching map $\phi: S^0 \times B^d \cong B^d \sqcup B^d \to \Bd
W$, is uniquely determined up to isotopy by the (possibly coinciding) components of $\Bd
W$ where the two $(d-1)$-balls forming the attaching region are located and by the
orientations induced by $\phi$ on such attaching balls. By reversing the handle $H^1$
those balls can be interchanged, hence only an unordered pair of components of $\Bd W$
needs to be specified when describing $H^1$. Concerning the orientations of the attaching
balls, the only relevant information is whether they coincide or not, when the balls are
in the same component of $\Bd W$ and this is orientable. If $W$ is orientable, the
difference between the two possibilities results in the orientability or non-orientability
of $W \cup_\phi H^1$. In dimension 3 and 4 we always assume that $W$ and $W \cup_\phi H^1$
are orientable, which implies that there is essentially only one way to attach $H^1$ to
$W$ for a given pair of components of $\Bd W$.

The case of a 2-handle attached to a 4-manifold $W$ requires some more work. In this case
the attaching sphere is a knot $K = \phi(S^1 \times \{0\})$ in $\Bd W$. Up to reversing
the handle $H^2$ the orientation of $K$ is not relevant. Then $H^2$ is uniquely determined
by the isotopy class of the unoriented knot $K \subset \Bd W$ and by the choice of a
framing along $K$, that means an isotopy class of trivializations of the normal disk
bundle $\nu K$ of $K$ in $\Bd W$. Since we always assume $W$ to be orientable, actually
$K$ can be any knot in $\Bd W$. Moreover, the set of all possible framings along $K$
bijectively corresponds to $\pi_1(\SO(2)) \cong \Bbb Z$. Unfortunately, unless $K$ is
null-homologous in $\Bd W$, there is no canonical way of fixing the zero framing, so
in general an explicit description for the framing is needed. This can be given in terms
of framed knots in the sense of the following definition.

\begin{definition}\label{framed-curve/def}
Let $M$ be a possibly bounded 3-manifold and $C \subset M$ be regularly embedded smooth
curve. Then, a {\sl framed curve} based on $C$ is an embedded smooth band $B \subset M$
with an identification $B \cong C \times [0,1]$, such that $C \times \{0\}$ coincides with
the {\sl base curve} $C$, $\Bd C \times [0,1]$ corresponds to $B \cap \Bd M$, while the
regularly embedded smooth curve $C_{\fr} \subset M$ corresponding to $C \times \{1\}$ is
the {\sl framing curve} which represents the framing along $C$. In particular, if $C$ is a
knot (resp. a link) in $\Int M$ then we call the framed curve a {\sl framed knot} (resp. a
{\sl framed link}).
\end{definition}

Therefore, to specify a 2-handle attached to a 4-manifold $W$ along a knot $C \subset \Bd
W$ with attaching map $\phi: S^1 \times B^2 \to \Bd W$, we can use a framed knot based on
$C$ by letting the band $B$ in the definition above to be $\phi(S^1 \times [0,p])$ for an
arbitrary point $p \in B^2 - \{0\}$. If $C$ is null-homologous in $\Bd W$, then the
framing is determined by the {\sl framing number} $\fr(C) = \lk(C,C_{\fr}) \in \Bbb
Z$ for $C$ and $C_{\fr}$ coherently oriented.

\begin{definition}\label{handlebody/def}
Let $M$ be a compact smooth $(d-1)$-manifold with (possibly empty) boundary and let $k
\leq d$. A $d$-dimensional {\sl relative $k$-handlebody} build on $M$ is a smooth
$d$-manifold $W$ with a given filtration $M \times [0,1] = W^{-1} \subset W^0 \subset W^1
\subset \dots \subset W^k = W$ of smooth $d$-submanifolds, such that $W^i = W^{i-1}
\cup_{j=1}^{n_i} H^i_j$ is obtained by attaching $n_i$ disjoint $i$-handles to $W^{i-1}$,
with attaching regions contained in $\delta W^{i-1}$, where $\delta W^i = \Bd W^i - (M
\times \{0\} \cup \Bd M \times [0,1])$ for any $0 \leq i \leq k$.

By identifying $M$ with $M \times \{0\} \subset W$, we think of it as a smooth submanifold
of $\Bd W$ of dimension $d - 1$. Then, the given family of handles forming $W$ starting
from $M \times [0,1]$ is called a {\sl relative $k$-handlebody decomposition} of the pair
$(W,M)$.

When $M$ is empty, we simply say that $W$ is a {\sl $k$-handlebody} and call the given
family of handles a {\sl $k$-handlebody decomposition} of $W$.
\end{definition}

We remark that in the usual definition of relative handlebody (cf. Definition 4.2.1 in
\cite{GS99}) the manifold $M$ is taken to be closed. Yet, the cobordisms which we study
here are relative handlebodies build on 3-manifolds with boundary, so we need this
generalization.

\medskip

By a well-known result of Cerf \cite{Ce70} (cf. Theorem 4.2.12 in \cite{GS99}, for the
case when $\Bd M = \emptyset$), two handlebody decompositions of the same pair $(W,M)$ can
be related, up to ambient isotopy of $W$ fixing $M$, by a finite sequence of the following
{\sl handle moves}, all considered for any $i = 1, \dots, d$ (the viceversa trivially
holds as well):
\begin{itemize}
\item[1)]
{\sl isotoping the attaching maps} of $i$-handles $\sqcup_{j=1}^{n_i}H^i_j$ in the
submanifold $\delta W^{i-1}$ of the boundary of the $(i-1)$-handlebody $W^{i-1}$;
\item[2)]
{\sl adding/deleting a pair of canceling handles}, that is an $i$-handle $H^i$ and an
$(i-1)$-handle $H^{i-1}$, such that the attaching sphere of $H^i$ intersects the belt
sphere of $H^{i-1}$ transversally in a single point;
\item[3)]
{\sl handle sliding} of one $i$-handle $H^i_{j_1}$ over another one $H^i_{j_2}$, that
means changing the attaching map of $H^i_{j_1}$ by an isotopy in the submanifold $\delta
W^i_{j_1}$ of the boundary of the $i$-handlebody $W^i_{j_1} = W^{i-1} \cup_{j \neq j_1}
H^i_j \subset W^i$, which pushes attaching\break sphere of $H^i_{j_1}$ through the belt
sphere of $H^i_{j_2}$.
\end{itemize}

\begin{definition}\label{k-equivalent/def}
Let $W$ be a $d$-dimensional relative $k$-handlebody $W$ for some $k \leq d$. Then a {\sl
$k$-deformation} of $W$ consists in a finite sequence of the handle moves described above,
such that at each stage we still have a relative $k$-handlebody, i.e. we never add any
canceling $i$-handle with $i > k$. Two relative $k$-handlebodies related by a
$k$-deformation will be called {\sl $k$-equivalent}.
\end{definition}

In the light of this definition, the result of Cerf can be restated by saying that two
$d$-dimensional relative handlebodies $(W,M)$ and $(W',M)$ are diffeomorphic if and
only if they are $d$-equivalent. 

\medskip

For any compact smooth $d$-manifold $W$ and any (possibly empty) smooth
$(d-1)$-submanifold $M \subset \Bd W$ with (possibly empty) boundary, a relative
$d$-handlebody decomposition of the pair $(W,M)$ can be derived from a suitable Morse
function.

Such a decomposition can be assumed to have only one 0-handle in each component disjoint
from $M$ and no 0-handles in the other components, and to have only one $d$-handle in each
closed component and no $d$-handle in the other components. Concerning 0-handles, this can
be proved by using handle moves to reduce them to at most one, while the case of
$d$-handle is dual (see \cite[p. 103]{GS99} for handlebody duality).

Actually, the reduction of 0-handles and dually of $d$-handles also applies to
deformations, hence something more can be said about them. Namely, if $(W,M)$ and $(W',M)$
are diffeomorphic and have the same number of 0-handles in the corresponding components,
then no addition/deletion of canceling 0-handles is needed to deform $(W,M)$ into
$(W',M)$. An analogous dual fact also holds for addition/deletion of $d$-handles. In
particular, if $(W,M)$ and $(W',M)$ are diffeomorphic relative $(d-1)$-handlebodies, then
they are $(d-1)$-equivalent.

\medskip

The main focus of this paper is on 2-equivalence of 4-dimensional 2-handlebodies. By the
discussion above, two such handlebodies are diffeomorphic if and only if they are
3-equivalent, but whether they are 2-equivalent is an open question, which is expected to
have negative answer (cf. Section I.6 of \cite{Ki89} and Section 5.1 of
\cite{GS99}).\break A list of 4-dimensional 2-handlebodies which are diffeomorphic, but
conjecturally not 2-equivalent can be found in \cite{Go91}.

Lemma \ref{ribbon-to-kirby/thm} will provide some connection between the notion of
2-deformation of 4-dimensional 2-handlebodies and that of embedded 1-deformation of
2-dimen\-sional 1-handlebodies in $B^4$. In a speculative sense, this relates the above
mentioned diffeomorphism problem for 4-dimensional 2-handle\-bodies to the question of
whether isotopic ribbon surfaces in $B^4$ are 1-isotopic, according to the definition we
will give in Section \ref{ribbons/sec} (cf. Remark \ref{1-isotopy/rem}).

\medskip

Since the reduction of 0-handles appears repeatedly and in different contexts in the
paper, we state the result in the next proposition and give a sketch of the reduction
procedure in the proof. We also consider here the case of $k$-deformations of relative
handlebodies with the minimum number of 0-handles as above, which is essentially the only
one occurring in this paper.

\begin{proposition}\label{0-handles/thm}
Any $k$-handlebody decomposition of a pair $(W,M)$ can be changed by handle isotopy,
1-handle sliding and deletion of canceling 0/1-handle pairs, in order to leave only one
0-handle in each component disjoint from $M$ and no 0-handles in the other components.
Moreover, if $(W,M)$ and $(W',M)$ are $k$-equiva\-lent relative handlebodies build on the
same manifold $M$, both having only one\break 0-handle in each component disjoint from $M$
and no 0-handles in the other components, then there is a $k$-deformation relating them
that does not involve any extra 0-handle.
\end{proposition}

\begin{proof}
We can limit ourselves to the case when $W$ and $W'$ are connected.

Given any $k$-handlebody decomposition $(W,M)$, let $G$ be the graph, whose vertices
represent the 0-handles and $M$ if this is non-empty and whose edges represent the
1-handles (with an edge connecting two vertices for each 1-handle between the
corresponding subspaces of $W^0$). If $W$ is connected then $G$ is connected as well.
Consider a maximal tree $T \subset G$ having as the root $M$ if this is non-empty, or any
0-handle if $M$ is empty. Then, $G$ contains all the 0-handles of $W$ and we can use
deletion of canceling 0/1-handle pairs to reduce all the handles in $T$ to the root.

Now, given any $k$-deformation realizing the $k$-equivalence between $(W,M)$ and $(W',M)$,
we can perform the 0-handle reduction described above on all the intermediate handlebodies
in the order they appear, in such a way that any addition/deletion of a canceling
0/1-handle pair corresponds to an edge expansion/contraction of the tree $T$ used for the
reduction. The resulting sequence of handlebodies represents a new $k$-deformation from
$(W,M)$ to $(W',M)$ without any addition/deletion of canceling 0/1-handle pairs and with
some of the 1-handle slidings of the original $k$-deformation turned into attaching map
isotopy.
\end{proof}

\medskip

We conclude this section by briefly discussing the notion of cobordism between  
oriented manifolds with boundary and its relations with that of relative handlebody.

By an oriented $d$-manifold with {\sl marked $B$-boundary}, we mean any pair $(M,\beta)$
where $M$ is a compact oriented smooth $d$-manifold and $\beta: \Bd M \to B$ is an
orientation preserving diffeomorphism, with $B$ a given (possibly empty) closed oriented
smooth $(d-1)$-manifold. Due to the existence of a collar of the boundary, we can always
consider $\beta$ up to isotopy.

If $(M_0,\beta_0)$ and $(M_1,\beta_1)$ are two oriented $d$-manifolds with marked
$B$-boundary, then a {\sl relative cobordism} (with corners) between them is a compact
oriented smooth $(d+1)$-manifold with marked boundary $(W,\eta)$, where $\eta: \Bd W \to
-M_0 \cup_{\beta_0 \times \{0\}} (B \times [0,1]) \cup_{\beta_1 \times \{1\}} M_1$ is an
orientation preserving diffeomorphism. We can think of $-M_0$ and $M_1$ as oriented smooth
$d$-submanifolds of $\Bd W$, by identifying them with their images under $\eta^{-1}$.

We say that $(M_0, \beta_0)$ and $(M_1, \beta_1)$ are {\sl cobordant} when a relative
cobordism as above exists. This gives an equivalence relation on the set of all the
compact smooth $d$-manifolds with marked $B$-boundary, for a fixed closed smooth
$(d-1)$-manifold $B$. In particular, the transitivity is guaranteed by the possibility of
composing two relative cobordisms $(W_1,\eta_1)$ and $(W_2,\eta_2)$ respectively from
$(M_0,\beta_0)$ to $(M_1, \beta_1)$ and from $(M_1, \beta_1)$ to $(M_2,\beta_2)$ to form a
cobordism $(W,\eta)$ from $(M_0, \beta_0)$ to $(M_2, \beta_2)$, with $W = W_1 \cup_{M_1}\!
W_2$ and $\eta$ canonically induced in the obvious way by $\eta_{1|M_0} \!\cup
\eta_{2|M_2}$.

Any relative cobordism $(W,\eta)$ between $(M_0, \beta_0)$ and $(M_1, \beta_1)$ can be
endowed with a structure of relative handlebody build on $M_0$ with $\delta W = M_1$
(cf. notation introduced in Definition \ref{handlebody/def}), such that the product $M_0
\times [0,1] \subset W$ satisfies the condition $\eta_{|\Bd M_0 \times [0,1]} = \beta_0
\times \id_{[0,1]}$. Viceversa, any relative handlebody $W$ build on $M_0$ with $\delta
W = M_1$ gives raise to a relative cobordism $(W,\eta)$ from $(M_0,\beta_0)$ to
$(M_1,\beta_1)$, such that $\eta_{|\Bd M_0 \times [0,1]} = \beta_0 \times \id_{[0,1]}$. Of
course, here the roles of $M_0$ and $M_1$ can be exchanged by handlebody duality.

Handlebody decomposition allows us to express any relative cobordism as the composition of
elementary cobordisms, meaning cobordisms admitting a relative handlebody structure with
only one handle. Namely, for an {\sl elementary cobordism} $(W,\eta)$ relating
$(M_0,\beta_0)$ to $(M_1,\beta_1)$ we have $W = M_0 \times [0,1] \cup_\phi H^i$ and $M_1 =
\delta W$. This implies that $M_1$ can be obtained from $M_0$, once this is canonically
identified with $M_0 \times \{1\}$, by replacing the attaching region $\phi(S^{i-1} \times
B^{d-i}) \subset \Int M_0$ of $H^i$ with $B^i \times S^{d-i-1} \subset \Bd H^i$ attached
to $\Cl(M_0 - \phi(S^{i-1} \times B^{d-i}))$ through the restriction of $\phi$ to $S^{i-1}
\times S^{d-i-1}$.

The replacement of a diffeomorphic copy of $S^{i-1} \times B^{d-i}$ in $\Int M_0$ by a 
diffeomorphic copy of $B^i \times S^{d-i-1}$ attached to the rest of $M_0$ through a given 
diffeomorphism between $S^{i-1} \times S^{d-i-1} = \Bd(B^i \times S^{d-i-1})$ and $\Bd 
(S^{i-1} \times B^{d-i}) \subset \Int M_0$ is usually referred to as a {\sl surgery} of 
index $i$ ($i$-surgery in short) on $M_0$. 

Since surgery takes place in the interior of the manifold and leaves the boundary
unchanged, it can be performed on manifolds with marked $B$-boundary as well. Moreover,
the argument above can be reversed to see that if $(M_1,\beta_1)$ is obtained by
$i$-surgery from $(M_0,\beta_0)$, then there is an elementary cobordism from
$(M_0,\beta_0)$ to $(M_1,\beta_1)$ with only one handle of index $i$. This allows us to 
conclude that two oriented manifolds with marked $B$-boundary are cobordant if and only if 
they are surgery equivalent, that is they are related by a finite sequence of surgeries.

Now, let $W = M_0 \times [0,1] \cup_\phi H^i$ be a $d$-dimensional elementary cobordism
from $M_0$ to $M_1 = \delta W$, with $H^i$ a {\sl trivially attached} $i$-handle,
meaning that the attaching map $\phi$ is isotopic to an embedding of $S^{i-1} \times
B^{d-i}$ in some smooth $(d-1)$-cell $C \subset M_0 \times \{1\}$, which is diffeomorphic
to the standard embedding $S^{i-1} \times B^{d-i} \to B^{d-1}$. In this case, the entire
handle $H^i$ can be thought as standardly embedded in a copy of $B^d_+$ attached to $M_0
\times [0,1]$ through a diffeomorphism $\rho: B^{d-1} \to C$. This gives raise, to a
cobordism $W' = M_0 \times [0,1] \cup_\rho B^d_+ = M_0 \times [0,1] \cup_\phi H^i
\cup_\psi H^{i+1}$, obtained by adding to $W$ an $(i+1)$-handle $H^{i+1}$ that forms a
canceling pair with the $i$-handle $H^i$. Both $W'$ and its dual handlebody $\bar W' = M_0
\times [0,1] \cup_{\bar\psi} \bH^{d-i-1} \cup_{\bar\phi} \bH^{d-i}$ are diffeomorphic to
the trivial cobordism $M_0 \times [0,1]$. Moreover, $\bar W = M_0 \times [0,1]
\cup_{\bar\psi} \bH^{d-i-1}$ is a new elementary cobordism from $M_0$ to $M_1 = \delta
\bar W = \delta W$.

The replacement of the trivially attached $i$-handle $H^i$ by the trivially attached
$(d-i-1)$-handle $\bH^{d-i-1}$, converting the cobordism $W$ into the new cobordism $\bar
W$ between the same manifolds, is called an {\sl $i/(d-i-1)$-handle trading}.

Performing such an $i/(d-i-1)$-handle trading produces the same effect on $W$ as an 
$(i+1)$-surgery along the $i$-sphere given by the union of the core of $H^i$ and a smooth
$i$-cell properly embedded in $M_0 \times [0,1]$ spanned by the attaching sphere of $H^i$.
Hence, the cobordisms $W$ and $\bar W$ turns out to be cobordant.

We observe that under the condition of isotopic triviality of the attaching map of the
original handle $H^i$ (or equivalently of the new handle $\bH^{d-i-1}$), an analogous
$i/(d-i-1)$-handle trading can also be performed on any (possibly non-elementary)
cobordism $(W,\eta)$ between two $d$-manifolds with $B$-marked boundary $(M_0,\beta_0)$ 
and $(M_1,\beta_1)$, with a given handlebody decomposition. The result is a new cobordism 
$(W',\eta')$ between $(M_0,\beta_0)$ and $(M_1,\beta_1)$, with a handlebody decomposition
obtained from the original one by replacing the trivially attached $i$-handle $H^i$ with 
the trivially attached $(d-i-1)$-handle $\bH^{d-i-1}$ and then reordering the handles.

Actually, handle trading together with handlebody deformation can be proved to generate
the cobordism equivalence on the set of all the cobordisms between given manifolds
$(M_0,\beta_0)$ and $(M_1,\beta_1)$ with marked $B$-boundary. In other words, two such
cobordisms are cobordant if and only if they admit handlebody decompositions that are
related by handlebody deformation and handle trading.

The case of interest in this paper is that of 1/2-handle trading in dimension 4.\break In
a 4-dimensional cobordism $W$, for a 1-handle being trivially attached means that both its
attaching balls are in the same component of $\delta W^0$ (orientability is assumed for
cobordisms), while a 2-handle is trivially attached if the attaching map is determined by
a trivial knot with trivial framing in $\delta W^1$. Then, given a 4-dimensional
relative 2-handlebody decomposition of a relative cobordism, once the reduction described
in Proposition \ref{0-handles/thm} has been performed, we can always trade all the
1-handles for trivially attached 2-handles. This way, we obtain a relative handlebody
without 1-handles representing a new relative cobordism between the same manifolds, which
is cobordant to the original one.

\subsection{Ribbon surfaces%
\label{ribbons/sec}}

A regularly embedded smooth compact surface $F \subset B^4$ with is called a {\sl ribbon
surface} if the Euclidean norm restricts to a Morse function on $F$ with no local maxima
in $\Int F$. Assuming $F \subset R^4_+ \subset R^4_+ \cup \{\infty\} \cong B^4$, this
property is topologi\-cally equivalent to the fact that the fourth Cartesian coordinate
restricts to a Morse height function on $F$ with no local minima in $\Int F$. 
In particular, $F$ has non-empty boundary $\Bd F \subset R^3$ (actually $F$ has no closed 
components).

\begin{definition}\label{ribbontangle/def}
By a {\sl ribbon surface tangle} in $E \times [0,1] \times \left[0,1\right[\,$, where $E =
[0,1]^2$ denotes the standard square, we mean a slice $S = F \cap (E \times [0,1] \times
\left[0,1\right[\,)$ of a ribbon surface $F \subset \Int E \times R \times
\left[0,1\right[ \subset R^4_+$, such that the intersections $\partial_0 S = F \cap (E
\times \{0\} \times \left[0,1\right[\,)$ and $\partial_1 S = F \cap (E \times \{1\} \times
\left[0,1\right[\,)$ are transversal and project to trivial families of regularly embedded
arcs in $E \times \left[0,1\right[$ (trivial means that the arcs are unknotted and
unlinked), by the projection forgetting the third coordinate. We call $\partial_0 S$ and
$\partial_1 S$ respectively the {\sl lower end} and the {\sl upper end} of $S$. We also
call $\partial S = S \cap (E \times [0,1] \times \{0\}) = S \cap \Bd F$ the {\sl tangle
boundary} of $S$, while the boundary of $S$ as a surface is $\Bd S = \partial S \cup
\partial_0 S \cup \partial_1 S$.
\end{definition}

Of course, ribbon surface tangles with empty ends reduce to ribbon surfaces contained in
$\Int E \times \left]0,1\right[ \times \left[0,1\right[ \subset R^4_+$. Then, without loss
of generality up to isotopy, we can think of any ribbon surface as a tangle with empty
ends. In this way, all the definitions and results given below for ribbon surface tangles
specialize to ribbon surfaces.

A ribbon surface tangle $S \subset E \times [0,1] \times \left[0,1\right[ \subset R^4_+$
can be isotoped, through an isotopy preserving the fourth coordinate, to make its
projection into $E \times [0,1] \subset R^3$ a regularly immersed surface, whose
self-intersections consist only of disjoint double arcs as in Figure
\ref{ribbon-surf01/fig} \(a).

\begin{Figure}[htb]{ribbon-surf01/fig}
{}{A ribbon intersection}
\centerline{\fig{}{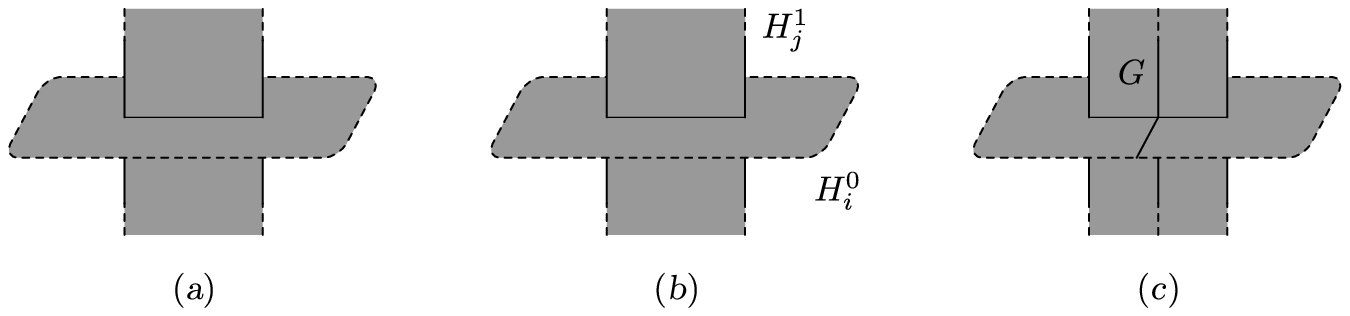}}
\vskip-3pt
\end{Figure}

We will refer to such a projection as a 3-dimensional {\sl diagram} of $S$ and use the
same notation $S$ for it (the meaning will be clear from the context). Actually, any
regularly immersed smooth compact surface $S \subset \Int E \times [0,1]$ with no closed
components and such that all its self-intersections are as above and $S \cap (E \times
\{0\})$ and $S \cap (E \times \{1\})$ both consist of trivial regularly embedded arcs, is
the diagram of a ribbon surface tangle uniquely determined up to vertical isotopy. This
can be obtained by pushing $\Int S$ up inside $\Int R^4_+$ in such a way that all
self-intersections disappear.

\medskip

Since a ribbon surface tangle $S$ is a surface with no closed components, it admits a
relative handlebody decomposition $S = C \cup H^0_1 \cup \dots \cup H^0_r \cup H^1_1 \cup
\dots \cup H^1_s$ build on $\partial_0 S \cup \partial_1 S$ with only 0- and 1-handles,
where $C \cong (\partial_0 S \cup \partial_1 S) \times [0,1]$ is a collar of $\partial_0 S
\cup \partial_1 S$ in $S$.

\begin{definition}\label{adapted-handlebody/def}
A relative 1-handlebody decomposition of a ribbon surface tangle $S$ as above is called
{\sl adapted} when the following conditions are satisfied:
\begin{itemize}
\item[\(a)]
each ribbon self-intersection involves an arc contained in the interior of a 0-han\-dle
and a proper transversal arc inside a 1-handle as in Figure \ref{ribbon-surf01/fig} \(b)
(cf. \cite{Ru85} for the case of ribbon surfaces);
\item[\(b)] 
to each component of $C$ is attached a single 1-handle connecting it to a 0-handle or 
to another component of $C$ itself.
\end{itemize}
The ribbon surface tangle $S$ endowed with such an adapted 1-handlebody decomposition will
be called an {\sl embedded 2-dimensional relative 1-handlebody} (build on $\partial_0 S \cup \partial_1 S$).
\end{definition}

According to this definition, the $H^0_i$'s are disjoint non-singular disks, while the
$H^1_j$'s are non-singular bands attached to $C$ and to the $H^0_i$'s, which possibly pass
through the $H^0_i$'s as shown in Figure \ref{ribbon-surf01/fig} \(b). The handlebody
decomposition is induced by the height function, if $S$ is realized as a suitable smooth
perturbation of the boundary of $((C \cup H^0_1 \cup \dots \cup H^0_r) \times [0,2/3])
\cup ((H^1_1 \cup \dots \cup H^1_s) \times [0,1/3])$ in $E \times [0,1] \times
\left[0,1\right[\,$.

\begin{definition}\label{1-equivalence/def}
We say that two embedded 2-dimensional relative 1-handle\-bodies build on the same sets of
intervals are equivalent up to {\sl embedded 1-defor\-mation}, or briefly that they are
{\sl 1-equivalent}, if they are related by a finite sequence of the following
modifications, all keeping those intervals fixed:
\begin{itemize}
\item[\(a)]
{\sl adapted isotopy}, that is isotopy of embedded relative 1-handlebodies build on a
fixed set of intervals, all adapted except for a finite number of intermediate critical
stages, at which one of the modifications described in Figure \ref{ribbon-surf02/fig} 
takes place (between any two such critical stages, we have isotopy of diagrams in $R^3$,
preserving ribbon intersections);
\begin{Figure}[htb]{ribbon-surf02/fig}
{}{Adapted isotopy moves}
\centerline{\fig{}{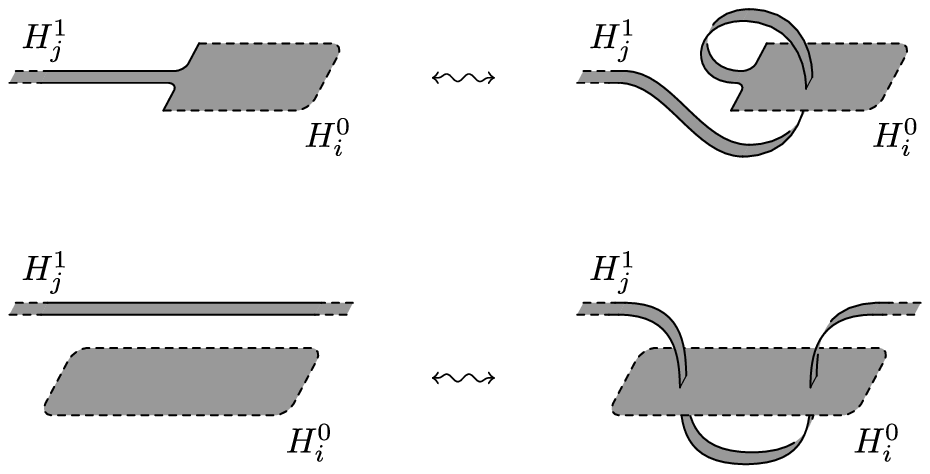}}
\end{Figure}
\item[\(b)]
{\sl ribbon intersection sliding}, allowing a ribbon intersection to run along a 1-handle
from one 0-handle to another one, as shown in Figure \ref{ribbon-surf03/fig};%
\begin{Figure}[htb]{ribbon-surf03/fig}
{}{Sliding a ribbon intersection along a 1-handle}
\centerline{\fig{}{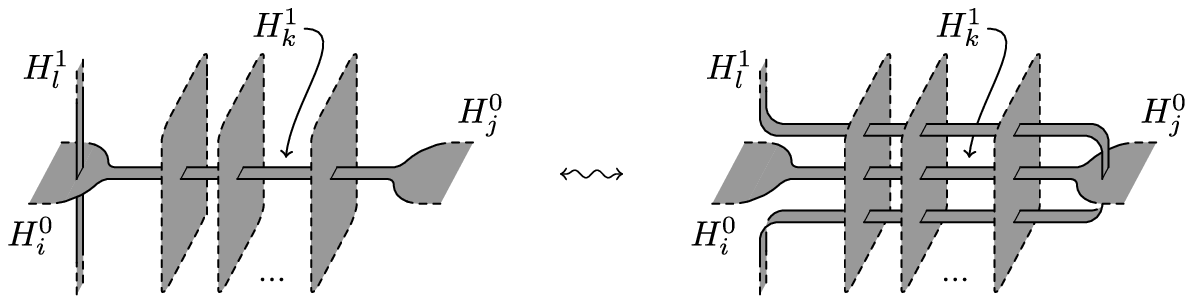}}
\end{Figure}
\item[\(c)]
{\sl embedded handles operations}, that is addition/deletion of canceling pairs of
0/1-handles and embedded 1-handle slidings (see Figure \ref{ribbon-surf04/fig}).
\begin{Figure}[htb]{ribbon-surf04/fig}
{}{Embedded handle operations}
\centerline{\fig{}{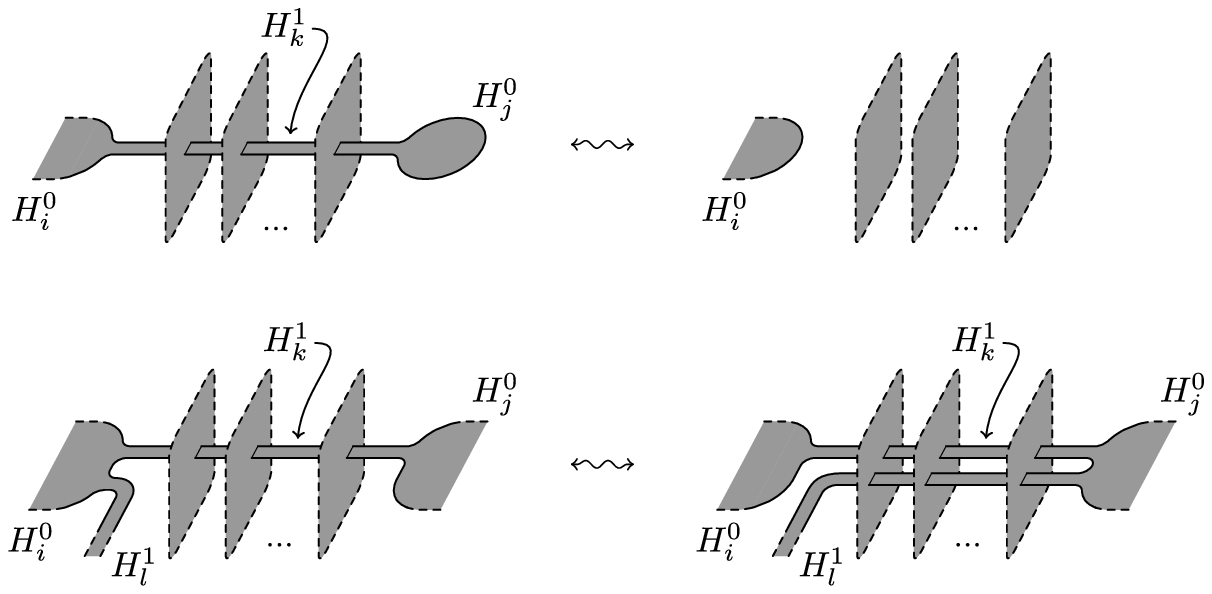}}
\vskip-3pt
\end{Figure}
\end{itemize}
\end{definition}

\pagebreak

We observe that the second modification of Figure \ref{ribbon-surf02/fig} is actually
redundant in presence of the handle operations of Figure \ref{ribbon-surf04/fig} (cf.
proof of Proposition \ref{1-isotopy/thm}).\break It is also worth noting that the
0-handles $H^0_i$ and $H^0_j$ in Figures \ref{ribbon-surf03/fig} and
\ref{ribbon-surf04/fig} can be assumed to be distinct in all the cases, up to
addition/deletion of canceling pairs of 0/1-handles (where they are always distinct).

\begin{proposition} \label{1-handles/thm}
All the adapted 1-handlebody decompositions of a ribbon surface tangle $S$ are
1-equivalent as embedded 2-dimensional relative 1-handle\-bodies. More precisely, up to
isotopy inside the 3-dimensional diagram of $S$, they are related to each other by the
special cases without vertical disks of the moves of Figures \ref{ribbon-surf03/fig} and
\ref{ribbon-surf04/fig}, realized inside $S$ in such a way that $S$ itself is kept
invariant.
\end{proposition}

\begin{proof}
First of all, we observe that the moves specified in the statement allow us to realize the
following two modifications: 1) split a 0-handle along any regular arc that avoids ribbon
intersections in the diagram, into two 0-handles joined by a new 1-handle, which coincides
with a regular neighborhood of the splitting arc; 2) split a 1-handle at any transversal
arc that avoids ribbon intersections in the diagram, into two 1-handles, by inserting a
new 0-handle given by a regular neighborhood of the splitting arc. We leave the
straightforward verification of this to the reader.

Let $S = C \cup H^0_1 \cup \dots \cup H^0_r \cup H^1_1 \cup \dots \cup H^1_s = \bar C \cup
\bH^0_1 \cup \dots \cup \bH^0_{\bar r} \cup \bH^1_1 \cup \dots \cup \bH^1_{\bar s}$\break
be any two 1-handlebody decompositions of a ribbon surface tangle $S$, which we denote
respectively by $H$ and $\bH$. Up to isotopy, we can assume $C = \bar C$. Moreover, we can
suitably split the 1-handles of $H$ and $\bH$, in such a way that any 1-handle contains at
most one ribbon self-intersection of $S$. Up to isotopy, we can also assume that the
1-handles of $H$ and $\bH$ attached to the same component of $C$ and those forming the
same ribbon self-intersection coincide. Let $H_1 = \bH_1, \dots, H_k = \bH_k$ be all these
1-handles. Then, it suffices to see how to transform the remaining 1-handles $H^1_{k+1},
\dots, H^1_s$ into $\bH^1_{k+1}, \dots, \bH^1_{\bar s}$, without changing $H^1_1, \dots,
H^1_k$.

Calling $\eta_i$ (resp. $\bar \eta_j$) the cocore of $H^1_i$ (resp. $\bH^1_j$), we have
$\eta_1 = \bar\eta_1, \dots, \eta_k = \bar\eta_k$, while the arcs $\eta_{k+1}, \dots,
\eta_s$ can be assumed to be transversal with respect to the arcs $\bar\eta_{k+1}, \dots,
\bar\eta_{\bar s}$. Up to isotopy, we can think of each 1-handle as a tiny regular
neighborhood of its cocore, so that the intersection between $H^1_{k+1} \cup \dots \cup
H^1_s$ and $\bH^1_{k+1} \cup \dots \cup \bH^1_{\bar s}$ consists only of a certain number
$h$ of small four-sided regions.

We eliminate all these intersection regions in turn, by pushing them outside $S$ along the
$\bH^1_j$'s. This is done by performing on $H$ moves of the types specified in the
statement of the proposition, as suggested by the Figure \ref{ribbon-surf05/fig}, which
concerns the $l$-th elimination. Namely, in \(a) we assume that the intersection is the
first one along $\bar\eta_j$ starting from the shown end-point in $\partial S$, then we
generate a new 1-handle $H^1_{s+l}$ by 0-handle splitting to get \(b), while \(c) is
obtained by handle sliding.

\begin{Figure}[htb]{ribbon-surf05/fig}
{}{Eliminating intersections of 1-handles of different decompositions}
\centerline{\fig{}{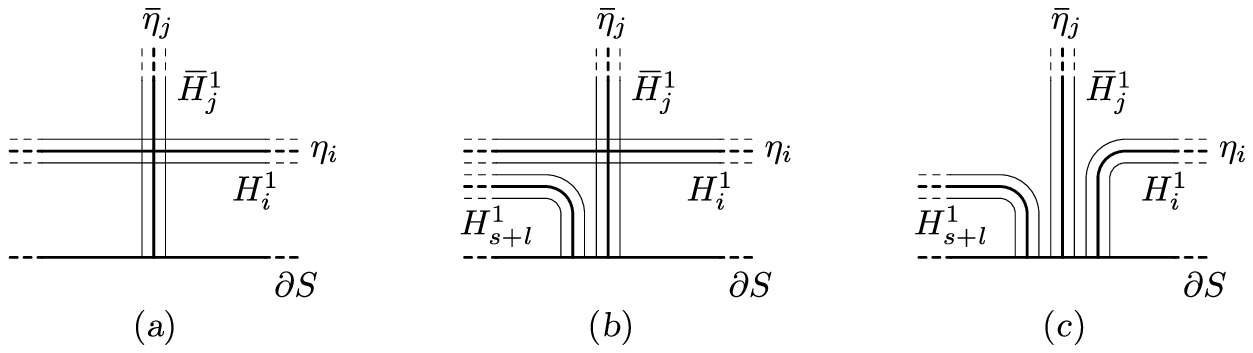}}
\vskip-3pt
\end{Figure}

After that, $H$ has been changed into a new handlebody decomposition $H'$ with 1-handles
$H^1_1, \dots, H^1_{s + h}$, such that $H^1_i$ is the same as above for $i \leq k$, while
it is disjoint from the $\bH^1_j$'s for $i > k$. Hence, $H^1_1, \dots, H^1_k,H^1_{k+1},
\dots, H^1_{s+h}, \bH^1_{k+1}, \dots, \bH^1_{\bar s}$ can be considered as the 1-handles
of a relative handlebody decomposition of $S$ which can be obtained from both $H'$ and
$\bH$ by 0-handle splitting.
\end{proof}

Forgetting the handlebody structure, 1-equivalence of embedded 2-dimensional relative
1-handlebodies induces an equivalence relation between ribbon surface tangles, which we
call 1-isotopy.

\begin{definition}\label{1-isotopy/def}
Two ribbon surface tangles are called {\sl 1-isotopic} when they admit 1-equivalent
adapted relative 1-handlebody decompositions (by the above proposition, this implies that
any two such 1-handlebody decompositions of them are 1-equivalent).
\end{definition}

Of course 1-isotopy implies isotopy, but the converse is not known and seems to be a
delicate question. Actually, also the problem of finding a complete set of moves
representing isotopy of ribbon surface tangles is still open, even in the special case of
ribbon surfaces.

As we anticipated in Section \ref{handles/sec}, the problem of whether isotopy of ribbon
surfaces implies 1-isotopy looks like an embedded lower dimensional analog of the
problem of whether diffeomorphism of 4-dimensional 2-handlebodies implies 2-equivalence.
In Section \ref{Theta/sec}, we will see how this analogy is supported by the connection
between the two concepts given in terms of branched coverings.

\medskip

Now, we want to provide a description of ribbon surface tangles and 1-isotopy in terms of
certain planar diagrams and moves between them. A planar diagram of ribbon surface
tangle will be based on the projection of a suitable 3-dimensional diagram $S \subset
\Int E \times [0,1]$ into the square $\left]0,1\right[ \times [0,1]$, given by forgetting
the second coordinate of $E = [0,1]^2$.

Since 1-isotopy fixes the ends $\partial_0 S$ and $\partial_1 S$, we will consider only
the case when these project regularly into disjoint unions of intervals in
$\left]0,1\right[ \times \{0\}$ and $\left]0,1\right[ \times \{1\}$ respectively.
In this case we say that the ribbon surface tangle $S$ has {\sl flat ends}.

We start with the observation that any 3-dimensional diagram of a ribbon surface tangle
$S$ with flat ends, considered as a 2-dimensional complex in $\Int E \times [0,1]$,
collapses to a graph $G$. We can choose $G$ to be the projection of a smooth simple spine
$P$ of $S$ (simple means that all the vertices have valence one or three) contained in $S
- \partial S$. Moreover, we can assume that $G$ meets at exactly one 1-valent vertex each
arc of $\partial_0 S \cup \partial_1 S$ and at exactly one 3- or 4-valent vertex each
ribbon intersection arc of $S$. In this way, $G$ turns out to have vertices of valence 1,
3 and 4.

We call {\sl flat vertices} the 1- or 3-valent vertices of $G$ whose inverse image in $P$
is one vertex with the same valence and {\sl singular vertices} the 3- or 4-valent
vertices of $G$ located at the ribbon intersections. The inverse image in $P$ of a 
singular vertex of $G$ consists of two points along edges of $P$ in the case of valence 4, 
while it consists of one point along an edge of $P$ and one end point of $P$ in the case 
of valence 3.

Finally, we assume $G$ to have three distinct tangent lines at each flat 3-valent vertex
and two distinct tangent lines at each 3- or 4-valent singular vertex.

Up to horizontal isotopy of $S$ mod $\partial_0 S \cup \partial_1 S$, we can contract its
diagram to a narrow regular neighborhood of the graph $G$. Then, by considering a planar
diagram of $G$, we can easily get a planar diagram of $S$ in the sense of the following
definition.

\begin{definition}\label{planar-diagram/def}
A {\sl planar diagram} of a ribbon surface tangle with flat ends is the planar projection
of a 3-dimensional diagram $S \subset \Int E \times [0,1]$ in the square $\left]0,1\right[
\times [0,1]$ given by forgetting the second coordinate of $E = [0,1]^2$, decorated
consistently with the height function in correspondence of crossings and ribbon
intersections, in such a way that it consists of a certain number of copies of the spots
\(a) to \(h) in Figure \ref{ribbon-surf06/fig} and some flat bands connecting pairwise the
end arcs of them and those in the projection of $\partial_0 S \cup \partial_1 S$. A planar
diagram whose ribbon intersections are all modeled on spot \(h) will be called a {\sl
special planar diagram}.
\end{definition}

\begin{Figure}[htb]{ribbon-surf06/fig}
{}{Local models for planar diagrams}
\centerline{\fig{}{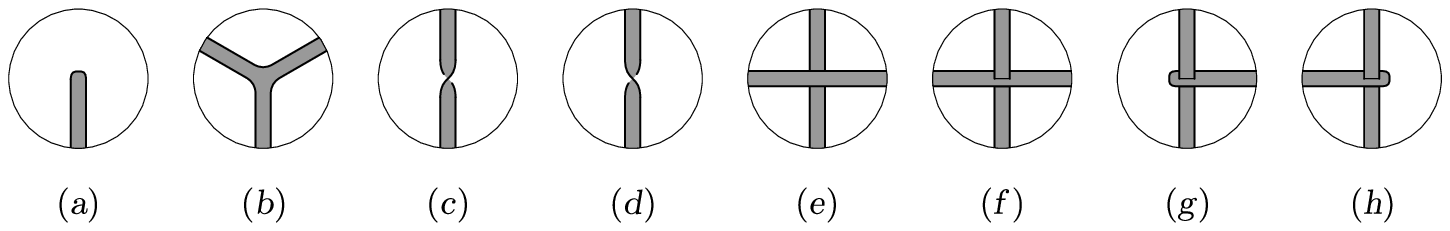}}
\vskip-3pt
\end{Figure}

As said above, a planar diagram of a ribbon surface tangle with flat ends arises as a
diagram of a pair $(S,G)$ where $G$ is a graph in a 3-dimensional diagram $S$, and this is
the right way to think of it.

However, we omit to draw the diagram of the graph $G$ in the pictures of a planar diagram,
since it can be trivially recovered, up to diagram isotopy, as the core of the diagram
itself. In particular, its singular vertices and diagram crossings are located at the
centers of the spots modeled on the three rightmost ones in Figure
\ref{ribbon-surf06/fig}, while the flat vertices of $G$ are located at the centers of the
spots modeled on the two leftmost ones in the same figure.

To be precise, there are two choices in recovering the graph $G$ at a singular vertex, as
shown in Figure \ref{ribbon-surf07/fig} for a ribbon intersection of type \(h). They give
the same graph diagram of each other, but differ for the way the graph is embedded in $S$.
We consider the move depicted in Figure \ref{ribbon-surf07/fig} as an equivalence move for
the pair $(S,G)$. Up to this move, which does not change the tangle, $G$ is uniquely
determined also as a graph in $S$.

\begin{Figure}[htb]{ribbon-surf07/fig}
{}{The graph $G$ at a ribbon intersection of type \(h)}
\centerline{\fig{}{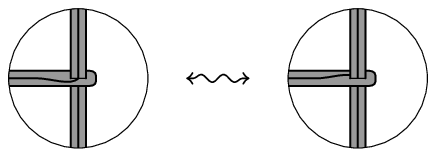}}
\end{Figure}

Like 3-dimensional diagrams, also planar diagrams uniquely determine the ribbon surface
tangle $S$ up to vertical isotopy. Here, by vertical isotopy we mean an isotopy that
preserves the first and third coordinates. In other words, the 3-dimensional height
function (as well as the 4-dimensional one) is left undetermined when presenting $S$ by
a planar diagram. Of course, this height function is required to be consistent with the
restrictions deriving from the local configurations in Figure \ref{ribbon-surf06/fig}.

We remark that any planar diagram can be made special by using the moves \(S1) and \(S2)
depicted in Figure \ref{ribbon-surf08/fig} to remove spots of type \(f) and \(g)
respectively.

\begin{Figure}[htb]{ribbon-surf08/fig}
{}{Making ribbon intersections special}
\centerline{\fig{}{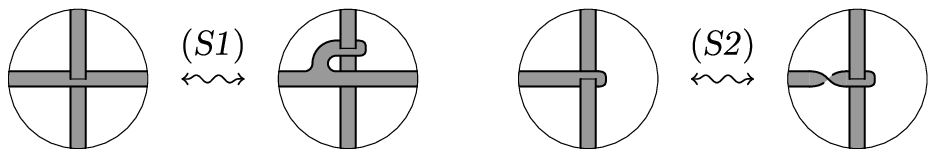}}
\end{Figure}

\begin{proposition}\label{planar-diagram/thm}
Any 3-dimensional diagram of a ribbon surface tangle $S$ with flat ends $\partial_0 S$
and $\partial_1 S$, up to 3-dimensional isotopy mod $\partial_0 S \cup \partial_1 S$, has
a (special) planar diagram.
\end{proposition}

\begin{proof}
Starting from the 3-dimensional diagram $S$, we can get a special planar diagram by the
following steps: 1) contract $S$ to a narrow regular neighborhood $N$ of the graph $G$; 2)
perturb $N$ to make the planar projection of $G$ into a graph diagram and the projection
of $N$ itself regular expect for a finite number of half-twists as in Figure
\ref{ribbon-surf06/fig} \(c) and \(d); 3) perform moves \(S1) and \(S2) at each ribbon
intersection of type \(f) and \(g) respectively. Then, the proposition immediately follows
from the observation that all this steps can be realized by 3-dimensional isotopies
keeping $\partial_0 S \cup \partial_1 S$ fixed.
\end{proof}

Figures \ref{ribbon-surf09/fig}, \ref{ribbon-surf10/fig} and \ref{ribbon-surf11/fig} show
how to interpret the 3-dimensional diagram isotopy in terms of planar diagrams, up to
planar isotopy. Actually, moves \(S3) and \(S4) could be realized by planar isotopy if we
think of the planar diagram just as representing the surface $S$, but this not true if we
take into account the graph $G$.

\begin{Figure}[htb]{ribbon-surf09/fig}
{}{Graph changing moves for special planar diagrams}
\centerline{\fig{}{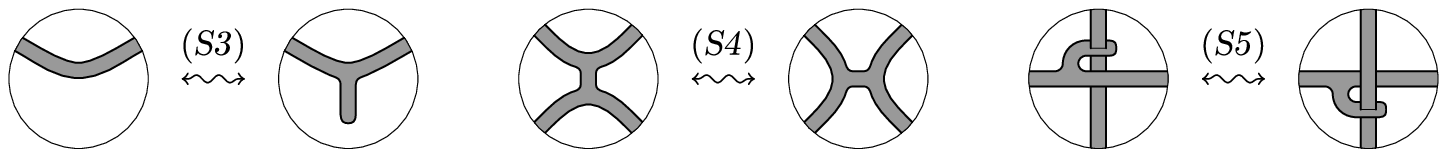}}
\end{Figure}

\begin{Figure}[htb]{ribbon-surf10/fig}
{}{Regular isotopy moves for special planar diagrams}
\centerline{\fig{}{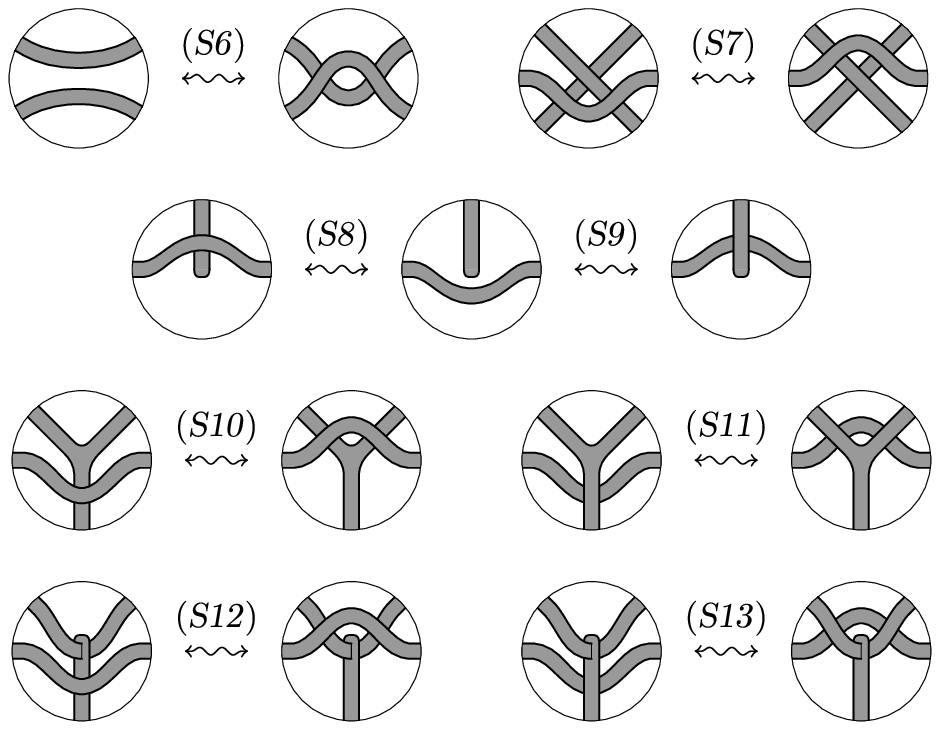}}
\end{Figure}

\begin{Figure}[htb]{ribbon-surf11/fig}
{}{Other 3-dimensional isotopy moves for special planar diagrams}
\vskip3pt
\centerline{\fig{}{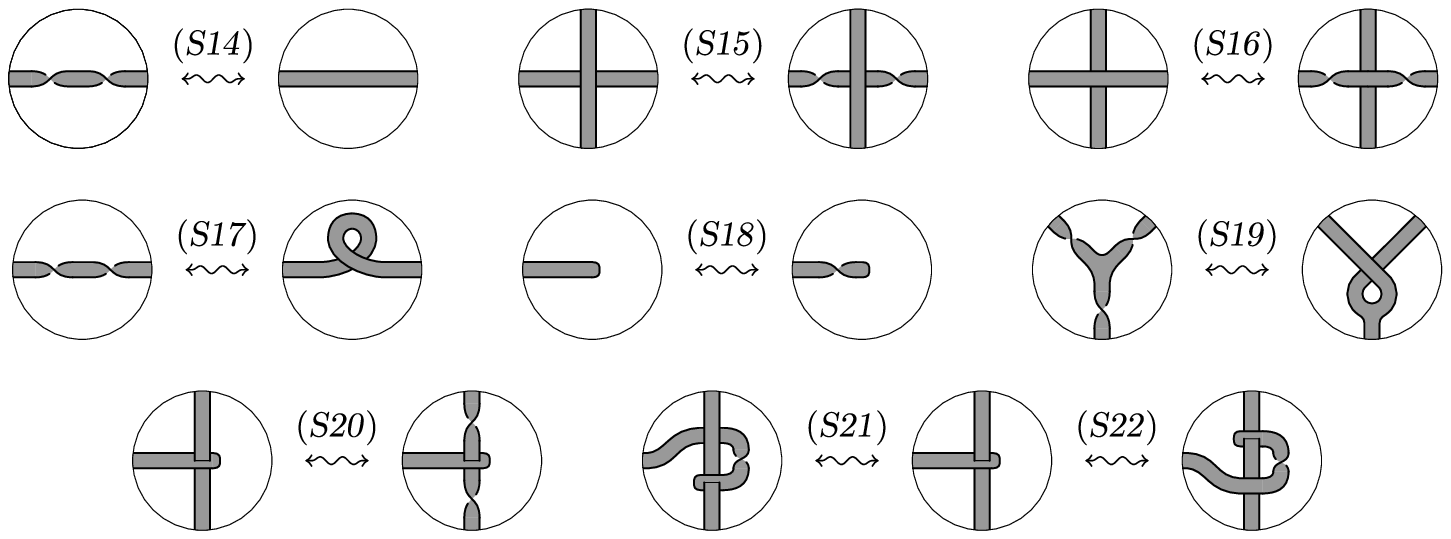}}
\end{Figure}

\begin{proposition}\label{isotopy/thm}
Two planar diagrams represent the same 3-dimensional diagram of a ribbon surface tangle
$S$, up to 3-dimensional isotopy mod $\partial_0 S \cup \partial_1 S$, if and only if they
are related by finite sequence of planar isotopies (induced by smooth ambient isotopies of
$\left]0,1\right[ \times [0,1]$ mod $\left]0,1\right[ \times \{0,1\}$ in the projection
plane) and moves \(S1) to \(S22) as in Figures \ref{ribbon-surf08/fig},
\ref{ribbon-surf09/fig}, \ref{ribbon-surf10/fig} and \ref{ribbon-surf11/fig}. Moreover,
moves \(S1) and \(S2) are not needed if both the planar diagrams are special.
\end{proposition}

\begin{proof}
The ``if'' part is trivial, since all the moves \(S1) to \(S22) represent special
3-dimensional isotopies of $S$ mod $\partial_0 S \cup \partial_1 S$. To prove the ``only
if'' part, we need to show that these moves do generate any such 3-dimensional isotopy of
the 3-di\-mensional diagrams represented by planar diagrams.

Moves \(S1) and \(S2) allow us to restrict our attention to special planar diagrams. All
the remaining moves \(S3) to \(S22) only contain terminal ribbon intersections of type
\(h), hence they can be performed in the context of special planar diagrams. Actually, we
will use only them in this special case, proving in this way also the last part of the
statement.

Now, consider two special planar diagrams representing ribbon surface tangles $S_0$ and
$S_1$ as regular neighborhoods of their core graphs $G_0$ and $G_1$, such that there is a
3-dimensional isotopy $H :(S_0,G_0) \times [0,1] \to \Int E \times [0,1]$ taking
$(S_0,G_0)$ to $(S_1,G_1)$, as singular surfaces with ribbon self-intersections, and
keeping the tangle ends fixed. Notice that the intermediate pairs $(S_t,G_t) =
H((S_0,G_0),t)$ with $0 < t < 1$ do not necessarily project suitably into
$\left]0,1\right[ \times [0,1]$ to give planar diagrams.

Of course, we can assume that $H$ is smooth, as a map defined on a pair of smooth
stratified spaces, and that the graph $G_t$ regularly projects to a diagram in
$\left]0,1\right[ \times [0,1]$ for every $t \in [0,1]$, except a finite number of $t$'s
corresponding to extended Reidemeister moves for graphs. For such exceptional $t$'s, the
lines tangent to $G_t$ at its vertices are assumed not to be vertical.

We define $\Gamma \subset G_0 \times [0,1]$ as the subspace of pairs $(x,t)$ for which
$S_t$ has a vertical tangency at $x_t = H(x,t)$ (if $x \in G$ is a singular vertex, there
are two such tangent planes and we require that one of them is vertical).

We can assume that $\Gamma$ does not meet $(\partial_0 S_0 \cup \partial_1 S_0) \times
[0,1]$. Moreover, by a standard transversality argument, we can perturb $H$ in such a way
that:
\begin{itemize}\itemsep0pt
\item[\(a)]
$\Gamma$ is a graph embedded in $G_0 \times [0,1]$ as a smooth stratified subspace of
constant codimension 1 and the restriction $\eta: \Gamma \to [0,1]$ of the height function
$(x,t) \mapsto t$ is a Morse function on each edge of $\Gamma$;
\item[\(b)] 
the edges of $\Gamma$ locally separate regions consisting of points $(x,t)$ for which the
projection of $S_t$ into $\left]0,1\right[ \times [0,1]$ has opposite local orientations
at $x_t$;
\item[\(c)] 
the two planes tangent to any $S_t$ at a singular vertex of $G_t$ are not both vertical,
and if one of them is vertical then it does not contain both the lines tangent to $G_t$ at
that vertex.
\end{itemize}

As a consequence of \(b), for each flat vertex $x \in G_0$ of valence one (resp. three)
there are finitely many points $(x,t) \in \Gamma$, all of which have the same valence one
(resp. three) as vertices of $\Gamma$. Similarly, as a consequence of \(c), for each
singular vertex $x \in G_0$ there are finitely many points $(x,t) \in \Gamma$, all of
which have valence one or two as vertices of $\Gamma$. Moreover, the above-mentioned
vertices of $\Gamma$ of valence one or three are the only vertices of $\Gamma$ of valence
$\neq 2$.

Let $0 < t_1 < \dots < t_k < 1$ be the critical levels $t_i$ at which one of the
following facts happens:
\begin{itemize}\itemsep0pt
\item[1)]\vskip-\lastskip
$G_{t_i}$ does not project regularly in $R^2$, since there is one point $x_i$ along an
edge of $G_0$ such that the line tangent to $G_{t_i}$ at $H(x_i,t_i)$ is vertical;
\item[2)] 
$G_{t_i}$ projects regularly in $R^2$, but its projection is not a graph diagram, due to a
multiple tangency or crossing;
\item[3)]
there is one point $(x_i,t_i) \in \Gamma$ with $x_i$ a uni-valent or a singular vertex of
$G_0$;
\item[4)]
there is one critical point $(x_i,t_i)$ for the function $\eta$ along an edge of $\Gamma$.
\end{itemize}

Without loss of generality, we assume that only one of the four cases above occurs for
each critical level $t_i$. Notice that the points $(x,t)$ of $\Gamma$ such that $x \in
G_0$ is a flat tri-valent vertex represent a subcase of 2 and for this reason
they are not included in case 3.

For $t \in [0,1] - \{t_1, \dots, t_k\}$, there exists a sufficiently small regular
neighborhood $N_t$ of $G_t$ in $S_t$, such that the pair $(N_t,G_t)$ projects to a special
planar diagram, except for the possible presence of some ribbon intersection projecting
in the wrong way, as in Figure \ref{ribbon-surf06/fig} \(g). We fix this problem by
inserting an auxiliary positive half-twists along the tongues containing those ribbon
intersections, as in move \(S2). The resulting singular surfaces, still denoted by $N_t$,
projects to a special planar diagram.

Actually, we modify the $N_t$'s all together to get a new isotopy where no wrong
projection of ribbon intersection occurs, so that $N_t$ projects to a special planar
diagram for each $t \in [0,1] - \{t_1, \dots, t_k\}$. Namely, at each critical level
when a wrong\break projection of a ribbon intersection is going to appear in the original
isotopy, we insert an auxiliary half-twist, to prevent the projection from becoming wrong.
Such half-twist remains close to the ribbon intersection until the first critical level
when the projection becomes good again in the original isotopy (remember that
3-dimensional diagram isotopy preserves ribbon intersections). At that critical level we
remove the auxiliary half-twist. We remark that the second part of condition \(c) is
violated when inserting/removing an auxiliary half-twist at critical points of type
2, as it can be seen by looking at moves \(S21) and \(S22) where this happens.

We observe that the planar diagram of $N_t$ is uniquely determined up to diagram isotopy
by that of the graph $G_t$ and by the tangent planes of $S_t$ at $G_t$. In fact, the
half-twists of $N_t$ along the edges of $G_t$ correspond to the transversal intersections
of $\Gamma$ with $G \times \{t\}$ and their signs, depend only on the local behavior of
the tangent planes of $S_t$. In particular, the planar diagrams of $(N_0,G_0)$ and
$(N_1,G_1)$ coincide, up to diagram isotopy, with the original ones of $(S_0,G_0)$ and
$(S_1,G_1)$.

If the interval $[t',t'']$ does not contain any critical level $t_i$, then each single
half-twist persists between the levels $t'$ and $t''$, and hence the planar isotopy
relating the diagrams of $G_{t'}$ and $G_{t''}$ also relate the diagrams of $N_{t'}$ and
$N_{t''}$, except for possible slidings of half-twists along ribbons over/under crossings.
Therefore the planar diagrams of $(N_{t'}, G_{t'})$ and $(N_{t''}, G_{t''})$, up to
diagram isotopy and moves \(S14), \(S15) and \(S16).

On the other hand, if the interval $[t',t'']$ is a sufficiently small neighborhood of a
critical level $t_i$, then the planar diagrams of $N_{t'}$ and $N_{t''}$ are related by
the moves in Figures \ref{ribbon-surf10/fig} and \ref{ribbon-surf11/fig}, depending on the
type of $t_i$ as follows.

If $t_i$ is of type 1, then a positive or negative kink is appearing (resp. disappearing)
along an edge of the core graph. When the kink is positive and $(x_i, t_i)$ is a local
maximum (resp. minimum) point for $\eta$, i.e. two positive half-twists along the ribbon
corresponding to the edge are being converted into a kink (resp. viceversa), the diagrams
of $N_{t'}$ and $N_{t''}$ are directly related by move \(S17). The cases when $(x_i,t_i)$
is not an extremal point for $\eta$, that is one or two negative half-twists appear (resp.
disappear) together with the positive kink, can be reduced to the previous case by means
of move \(S14). On the other hand, by using the regular isotopy moves \(S6) and \(S7) in
order to create or delete in the usual way a pair of canceling kinks (without introducing
any half-twist) along the ribbon, we can reduce the case of a negative kink to that of a
positive one.

If $t_i$ is of type 2, then either a regular isotopy move is occurring between $G_{t'}$
and $G_{t''}$ or two tangent lines at a tri-valent vertex $x_i$ of the graph project to
the same line in the plane. In the first case, the regular isotopy move occurring between
$G_{t'}$ and $G_{t''}$, trivially extends to one of the moves \(S6) to \(S13). In the
second case, $x_i$ may be either a flat or a singular vertex. If $x_i$ is a flat vertex,
then the tangent plane to $S_t$ at $H(x_i,t)$ is vertical for $t = t_i$ and its projection
reverses the orientation when $t$ passes from $t'$ to $t''$. Move \(S19) (modulo moves
\(S6) and \(S14)) describes the effect on the diagram of such a reversion of the tangent
plane. If $x_i$ is a singular vertex, then $N_{t'}$ changes into $N_{t''}$ by one of the
moves \(S21) or \(S21), where auxiliary half-twists are inserted according to what we have
said above.

If $t_i$ is of type 3, then either a half-twist is appearing/disappearing at the tip of
the tongue of surface corresponding to a uni-valent vertex or one of the two bands at the
ribbon intersection corresponding to a singular vertex is being reversed in the planar
projection. The first case corresponds to move \(S18) (here we have a positive half-twist,
for dealing with a negative one we combine this move with \(S14)). The second case may
happen in two different ways, depending on which band is being reversed. If such band is
the one passing through the other in the ribbon intersection, then, we can transform
$N_{t'}$ into $N_{t''}$ by applying move \(S20), possibly modulo \(S14). Otherwise, the
projection of the ribbon intersection is changing from good to wrong, and the appearing
half-twist is compensated by the auxiliary one up to move \(S14).

Finally, if $t_i$ is of type 4, a pair of canceling half-twists is appearing or
disappearing along an edge of the graph. This is just move \(S14).

At this point, to conclude that moves \(S3) to \(S22) suffice to realize
3-dimension\-al isotopy between any two special planar diagrams of a given ribbon surface
tangle $S$, it is left to prove that, given two different core graphs $G'$ and $G''$ of
$S$ as above, the planar diagrams $S'$ and $S''$ determined respectively by $G'$ and
$G''$, are related by those moves. This is quite straightforward. In fact, by cutting the
3-dimensional diagram $S$ along the ribbon intersection arcs, we get a new surface $\bar
S$ with some marked arcs. This operation also makes the graphs $G'$ and $G''$ into two
simple spines $P'$ and $P''$ of $\bar S$ relative to those marked arcs (Figure
\ref{ribbon-surf12/fig} shows the effect of the cut at the ribbon intersection in Figure
\ref{ribbon-surf07/fig}) and to the arcs in $\partial_0 S \cup \partial_1 S$.

\begin{Figure}[htb]{ribbon-surf12/fig}
{}{Cutting a ribbon surface tangle at a ribbon intersection}
\centerline{\fig{}{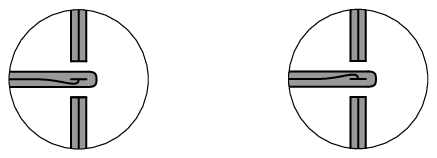}}
\end{Figure}

From intrinsic point of view, that is considering $\bar S$ as an abstract surface and
forgetting its inclusion in $R^3$, the theory of simple spines tells us that the moves in
Figure \ref{ribbon-surf09/fig} suffice to transform $P'$ into $P''$. In particular, moves
\(S3) and \(S4) correspond to the well-known moves for simple spines of surfaces, while
\(S5) relates the different positions of the spine with respect to the marked arcs in the
interior of $\bar S$. It remains only to observe that, up to a 3-dimensional diagram
isotopy preserving the core graph, hence up to the moves in Figures
\ref{ribbon-surf10/fig} and \ref{ribbon-surf11/fig}, the portion of the surface involved
in each single spine modification can be isolated in the planar diagram.
\end{proof}

The next proposition says that, up to 3-dimensional diagram isotopy, 1-isotopy of ribbon
tangles is generated by the local isotopy moves of Figure \ref{ribbon-surf13/fig}.

\begin{proposition} \label{1-isotopy/thm}
Two ribbon surface tangles with flat ends are 1-iso\-topic if and only if their planar
diagrams can be related by a finite sequence of 3-dimensional diagram isotopies fixing the
tangle ends (cf. Proposition \ref{isotopy/thm}) and moves \(S23) to \(S26) in Figure
\ref{ribbon-surf13/fig}.
\end{proposition}

\begin{Figure}[htb]{ribbon-surf13/fig}
{}{1-isotopy moves for planar diagrams}
\centerline{\fig{}{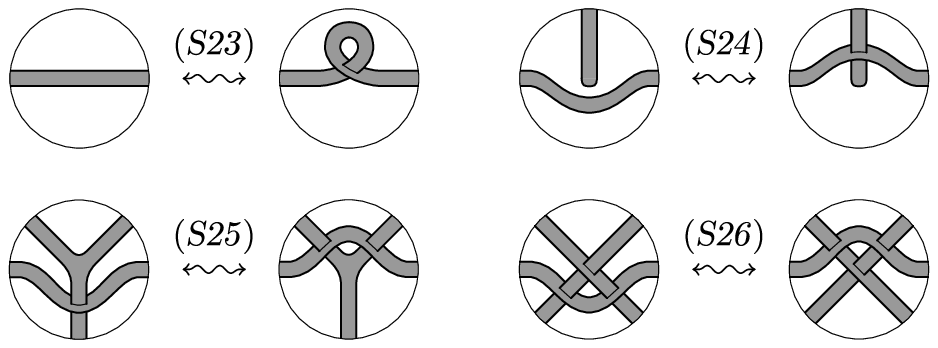}}
\end{Figure}

\begin{proof}
On one hand, we have to realize the modifications of Figures \ref{ribbon-surf02/fig},
\ref{ribbon-surf03/fig} and \ref{ribbon-surf04/fig}, disregarding the handlebody
structure, by moves \(S23) to \(S26). Proceeding in the order: one move \(S23) suffices
for the upper part of Figure \ref{ribbon-surf02/fig}, while the lower part can be obtained
by combining one move \(S24) with one move \(S25); Figure \ref{ribbon-surf03/fig} requires
three moves for each vertical disk, one \(S24), one \(S25) and one \(S26); the upper
(resp. lower) part of Figure \ref{ribbon-surf04/fig} can be achieved by one move \(S24)
(resp. \(S25)) for each vertical disk.

On the other hand, the ribbon surfaces of Figure \ref{ribbon-surf13/fig} can be easily
provided with adapted handlebody decompositions, so that the relations just described
between moves \(S23) to \(S26) and the above modifications can be reversed. In fact, only
the special cases of those modifications with one vertical disk are needed.
\end{proof}

\subsection{Branched coverings%
\label{coverings/sec}}

In this section we provide a short account of the theory of branched coverings, with a 
special emphasis on covering moves in dimension 3 and 4. 

Here, we adopt the piecewise-linear approach, which is the most suitable for a unified
exposition of the topic in a general context. However, this approach is equivalent to the
differentiable one for the special cases we will consider in the next chapters: branched
coverings in dimensions 2 and 3 or branched coverings of $B^4$ whose branching set is a
ribbon surface.

\begin{definition}\label{covering/def}
A non-degenerate PL map $p:P \to Q$ between compact PL manifolds of the same dimension $d$
is called a {\sl branched covering} if there exists an $(d-2)$-dimensional subcomplex $B_p
\subset Q$, the {\sl branching set} of $p$, such that the restriction $p_|: P-p^{-1}(B_p)
\to Q - B_p$ is an ordinary covering of finite degree $n$. By the {\sl monodromy} of
$p$ we mean the monodromy $\omega_p: \pi_1(Q-B_p, \ast) \to \Sigma_n$ of this restriction,
which is defined up to conjugation in the permutation group $\Sigma_n$, depending on the
choice of the base point $\ast$ and on the numbering of $p^{-1}(\ast)$. We call
$p$ a {\sl simple} branched covering if the monodromy of any meridian loop around a
$(d-2)$-simplex of $B_p$ is a transposition.
\end{definition}

If the subcomplex $B_p \subset Q$ in the definition is minimal with respect to the
required property, then we have $B_p = p(S_p)$, where $S_p$ is the {\sl singular set} of
$p$, that is the set of points at which $p$ is not locally injective. In this case, both
$B_p$ and $S_p$, as well as the {\sl pseudo-singular set} $S'_p = \Cl(p^{-1}(B_p) - S_p)$,
are (possibly empty) homogeneously $(d-2)$-dimensional complexes.

Since $p$ is completely determined, up to PL homeomorphisms, by the ordinary covering
$p_|$ (cf. \cite{Fo57}), we can describe it in terms of the branching set $B_p$ and the
monodromy $\omega_p$. In particular, the monodromies of the meridians around the
$(d-2)$-simplices of $B_p$ determine the structure of the singularities of $p$. If $p$ is
simple. then every point in the interior of a $(d-2)$-simplex of $B_p$ is the image of one
singular point, at which $p$ is topologically equivalent to the complex map $z \mapsto
z^2$, and $n-2$ pseudo-singular points.

Starting from $B_p \subset Q$ and $\omega_p$, we can explicitly reconstruct $P$ and $p$ by
following steps: 1) choose a $(d-1)$-dimensional {\sl splitting complex}, that means a
subcomplex $C \subset Q - \{\ast\}$ such that $B_p \subset C$ and the restriction
$\omega_{p|}:\pi_1(Q - C,\ast) \to \Sigma_d$ vanishes; 2) cut $Q$ along $C$ in such a way
that each $(d-1)$-simplex $\sigma$ of $C$ gives raise to 2 simplices $\sigma^-$ and
$\sigma^+$; 3) take $n$ copies of the obtained complex (called the {\sl sheets} of the
covering) and denote by $\sigma^\pm_1, \dots, \sigma^\pm_n$ the corresponding copies of
$\sigma^\pm$; 4) identify in pairs the $\sigma^\pm_i$'s according to the monodromy $\rho =
\omega_p(\alpha)$ of a loop $\alpha$ meeting $C$ transversally at one point of $\sigma$,
that is $\sigma^-_i$ with $\sigma^+_{\rho(i)}$. Up to PL homeomorphisms, $P$ is the result
of such identification and $p$ is the map induced by the natural projection of the sheets
onto $Q$.

\medskip

A convenient representation of $p$ can be given by labeling each $(d-2)$-simplex of $B_p$
by the monodromy of a preferred meridian around it and each generator (in a finite
generating set) of $\pi_1(Q,\ast)$ by its monodromy, since those loops together generate
$\pi_1(Q-B_p,\ast)$. Of course, only the labels on $B_p$ are needed when $Q$ is simply
connected. In any case, with a slight abuse of language if $Q$ is not simply connected, we
refer to such a representation as a {\sl labeled branching set}.

\begin{definition}\label{covering-equiv/def}
Two branched coverings $p:P \to Q$ and $p':P' \to Q$ are called {\sl equivalent} if and
only if there exists PL homeomorphism $h:Q \to Q$ isotopic to the identity which lifts to
a PL homeomorphism $k:P \to P'$.
\end{definition}

By the classical theory of ordinary coverings and \cite{Fo57}, such a lifting $k$ of $h$
exists if and only if $h(B_p) = B_{p'}$ and $\omega_{p'}h_* = \omega_p$ up to conjugation
in $\Sigma_n$, where $h_\ast: \pi_1(Q - B_p,\ast) \to \pi_1(Q - B_{p'},h(\ast))$ is the
homomophism induced by $h$. Therefore, in terms of labeled branching set, the equivalence
of branched coverings can be rep\-resented by {\sl labeled isotopy}.

\medskip

Before going on, let us say some further words about the representation through labeled
diagrams of the branched coverings in the cases of interest for our purposes. We remark
once again that in all those cases PL and smooth are interchangeable.

We represent an $n$-fold covering $p: P \to S^3$ (resp. $B^3$) branched over a link $L
\subset S^3$ (resp. a tangle $T \subset B^3$) by a $\Sigma_n$-labeled oriented diagram $D$
of $L$ (resp. $T$), which describes the monodromy of $p$ in terms of the Wirtinger
presentation of $\pi_1(S^3 - L)$ (resp. $\pi_1(B^3 - T)$) associated to $D$. Namely, we
label each arc of $D$ by the monodromy of the standard positive meridian around it. Of
course, the Wirtinger relations impose constraints on the labeling at crossings, and each
$\Sigma_n$-labeling of $D$ satisfying such constraints do actually represent an $n$-fold
covering branched over $L$ (resp. $T$). In this context, labeled isotopy can be realized
by means of labeled Reidemeister moves.

For simple coverings, the orientation of $D$ is clearly unnecessary and there are only
three possible ways of labeling the arcs at each crossing. These are depicted in Figure
\ref{coverings01/fig}. The extension from branching links/tangles to branching embedded
graphs is straightforward. In fact, we only need to take into account extra labeling
constraints and labeled moves at the vertices of the graph.

\begin{Figure}[htb]{coverings01/fig}
{}{Simple labelings of a link/tangle ($i,j,k$ and $l$ all different)}
\centerline{\fig{}{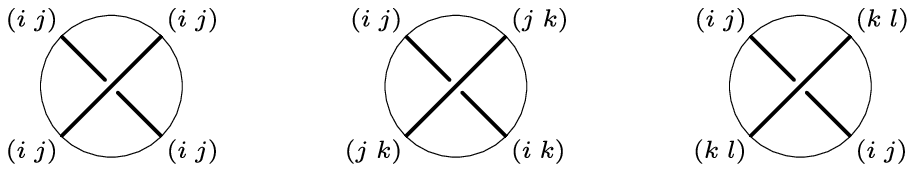}}
\end{Figure}

Now consider an $n$-fold covering $p:P \to B^4$ branched over a ribbon surface $S \subset
B^4$. We represent it by labeled locally oriented planar diagram of $S$, where the labels
give the monodromy of the meridians in the Wirtinger presentation of $\pi_1(B^4 - S)$
associated to the diagram. Actually, since we will only consider simple coverings, we will
never need local orientations.

The same labeling rules as above apply to ribbon intersections as well as to ribbon
crossings, these are depicted in Figure \ref{coverings02/fig}. In particular, at a ribbon
intersection the label of the band passing through the other gets conjugated by the label
of this. Notice that, contrary to what happens for ribbon intersections (in the case of
distinct but not disjoint labels), when a ribbon crosses under another one, its label
changes only locally (on the undercrossing region).

\begin{Figure}[htb]{coverings02/fig}
{}{Simple labelings of a ribbon surface ($i,j,k$ and $l$ all different)}
\centerline{\fig{}{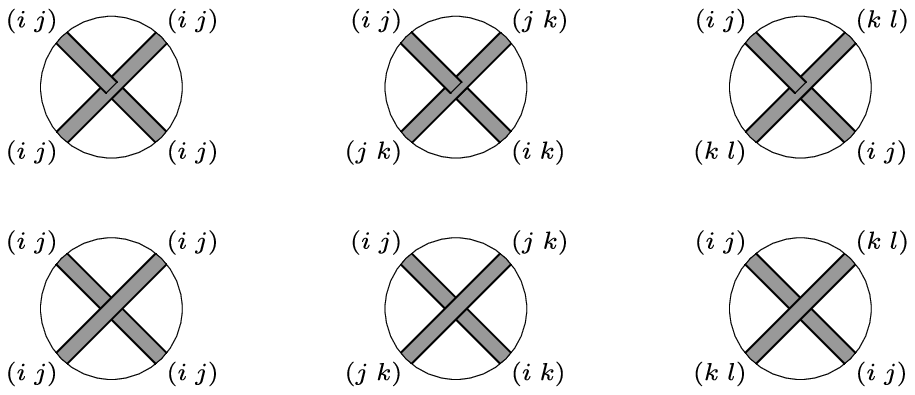}}
\end{Figure}

We remark that, if $S \subset B^4$ is a labeled ribbon surface representing an $n$-fold
(simple) covering of $p: P \to B^4$, then $L = S \cap S^3$ is a labeled link representing
the restriction $p_{|\Bd}: \Bd P \to S^3$. This is still a $n$-fold (simple) covering,
having the diagram of $S$ as a splitting complex.

Labeled ribbon surfaces in $B^4$ (that is coverings of $B^4$ simply branched over ribbon
surfaces) represent all the 4-dimensional 2-handlebodies. In fact, by Montesinos
\cite{Mo78} (cf. Chapter \ref{surfaces/sec}), for the connected case it suffices to take
labels from the three transpositions of $\Sigma_3$ (that is to consider 3-fold simple
coverings).

Though labeled isotopy of branching ribbon surfaces preserves the covering manifold $P$ up
to diffeomorphisms, we are interested in the (perhaps more restrictive) notion of {\sl
labeled 1-isotopy}, which preserves $P$ up to 2-deformations (cf. Lemma
\ref{ribbon-to-kirby/thm}). This can be realized by means of labeled diagram isotopy and
labeled 1-iso\-topy moves, that is diagram isotopy and 1-isotopy moves of Figure
\ref{ribbon-surf13/fig} suitably labeled according to the rules discussed above.

\medskip

By a {\sl covering move}, we mean any non-isotopic modification making a labeled branching
set representing a branched covering $p: P \to Q$ into one representing a different
branched covering $p': P \to Q$ between the same manifolds (up to PL homeomorphisms). We
call such a move {\sl local}, if the modification takes place inside a cell and can be
performed whatever is the rest of labeled branching set outside. In the figures depicting
local moves, we will draw only the portion of the labeled branching set inside the
relevant cell, assuming everything else to be fixed.

As a primary source of covering moves, we consider the following two very general
equivalence principles (cf. \cite{PZ03}). Several special cases of these principles have
already appeared in the literature and we can think of them as belonging to the
``folklore'' of branched coverings.

\begin{statement}{Disjoint monodromies crossing}
Subcomplexes of the branching set of a covering that are labeled with disjoint
permutations can be isotoped independently from each other without changing the covering
manifold.
\end{statement}

The reason why this principle holds is quite simple. Namely, being the labeling of the
subcomplexes disjoint, the sheets non-trivially involved by them do not interact, at least
over the region where the isotopy takes place. Hence, the relative position of such
subcomplexes is not relevant in determining the covering manifold. Typical applications of
this principle are given by the local covering moves \(M2), \(R2) and \(R4) in Figures 
\ref{coverings04/fig}, \ref{coverings05/fig} and \ref{coverings07/fig}).

It is worth observing that, abandoning transversality, the disjoint monodromies crossing
principle also gives the special case of the next principle when the $\sigma_i$'s are
disjoint and $L$ is empty.

\begin{statement}{Coherent monodromies merging}
Let $p:P \to Q$ be any branched covering with branching set $B_p$ and let $\pi: E \to K$
be a connected disk bundle embedded in $Q$, in such a way that: 1)~there exists a
(possibly empty) subcomplex $L \subset K$ for which $B_p \cap \pi^{-1}(L) = L$ and the
restriction of $\pi$ to $B_p \cap \pi^{-1}(K-L)$ is an unbranched covering of $K-L$;
2)~the monodromies $\sigma_1, \dots, \sigma_k$ relative to a fundamental system $\omega_1,
\dots, \omega_k$ for the restriction of $p$ over a given disk $D = \pi^{-1}(x)$, with $x
\in K - L$, are coherent in the sense that $p^{-1}(D)$ is a disjoint union of disks. Then,
by contracting the bundle $E$ fiberwise to $K$, we get a new branched covering $p': P \to
Q$, whose branching set $B_{p'}$ is equivalent to $B_p$, except for the replacement of
$B_p \cap \pi^{-1}(K-L)$ by $K-L$, with the labeling uniquely defined by letting the
monodromy of the meridian $\omega = \omega_1 \dots \omega_k$ be $\sigma = \sigma_1 \dots
\sigma_k$.
\end{statement}

We remark that, by connectedness and property 1, the coherence condition required in 2
actually holds for any $x \in K$. Then, we can prove that $p$ and $p'$ have the same
covering manifold, by a straightforward fiberwise application of the Alexander's trick to
the components of the bundle $\pi \circ p: p^{-1}(E) \to K$. A coherence criterion can be
immediately derived from Section 1 of \cite{MP01}.

The coherent monodromy merging principle originated from a classical perturbation argument
in algebraic geometry and appeared in the literature as a way to deform non-simple
coverings between surfaces into simple ones, by going in the opposite direction from $p'$
to $p$ (cf. \cite{BE79}). In the same way, it can be used in dimension 3, both for
achieving simplicity (cf. \cite{Hr03}) and removing singularities from the branching set.
We will do that in the proof of Theorem \ref{equiv3g/thm} by means of the moves \(G1) and 
\(G2) of Figure \ref{coverings03/fig}, which are straightforward applications of this 
principle.\break Actually, similar results could be proved in dimension 4 for labeled 
singular surfaces representing possibly non-simple branched coverings, but we will not 
consider them here. 

\begin{Figure}[htb]{coverings03/fig}
{}{Covering moves for labeled graphs ($\sigma = \sigma_1 \!\cdot \sigma_2$)}
\centerline{\fig{}{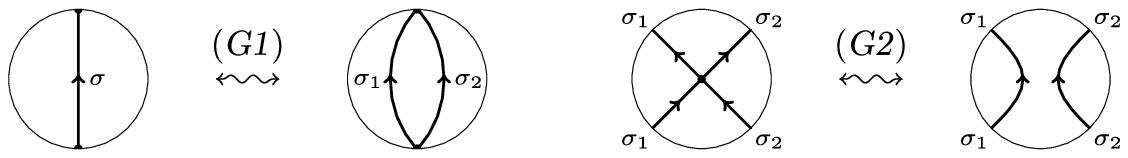}}
\vskip-3pt
\end{Figure}

Finally, we consider the notion of stabilization. This is a particular local covering
move, which makes sense only for branched coverings of $S^m$ or $B^m$ and, differently
from all the previous moves, changes the degree of the covering, increasing it by one.

\begin{statement}{Stabilization}
A branched covering $p:P \to S^d$ (resp. $p:P \to B^d$) of degree $n$, can be stabilized
to degree $n + 1$ by adding to the labeled branching set a trivial separate
$(d-2)$-sphere (resp. regularly embedded $(d-2)$-disk) labeled with the transposition
$\tp{i}{n{+}1}$, for some $i = 1, \dots, n$.
\end{statement}

The covering manifold of such a stabilization is still $P$, up to PL homeomorphisms. In
fact, it turns out to be the connected sum (resp. boundary connected sum) of $P$ itself,
consisting of the sheets $1, \dots, n$, with the copy of $S^d$ (resp. $B^d$) given by the
extra trivial sheet $n+1$.

By {\sl stabilization to degree $m$} (or {\sl $m$-stabilization}) of a branched covering
$p: P \to S^d$ (resp. $p:P \to B^d$) of degree $n \leq m$ we mean the branched covering of
degree $m$ obtained from it by performing $m - n$ stabilizations as above. In particular,
this leaves $p$ unchanged if $m = n$.

\medskip

Concerning the cases of interest for this paper, Figure \ref{coverings04/fig} shows the
covering moves introduced by Montesinos (cf. \cite{Mo85}) for labeled links representing
simple coverings of $S^3$, while in Figure \ref{coverings05/fig} we introduce new local
covering moves for labeled ribbon surfaces representing simple coverings of $B^4$, which
we call {\sl ribbon moves}.

\begin{Figure}[htb]{coverings04/fig}
{}{Montesinos moves ($i,j,k$ and $l$ all different)}
\centerline{\fig{}{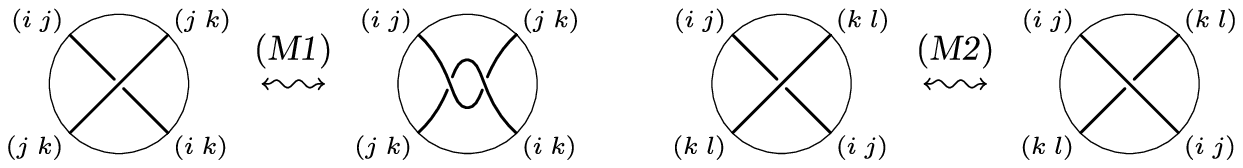}}
\vskip-3pt
\end{Figure}

\begin{Figure}[htb]{coverings05/fig}
{}{Ribbon moves ($i,j,k$ and $l$ all different)}
\centerline{\fig{}{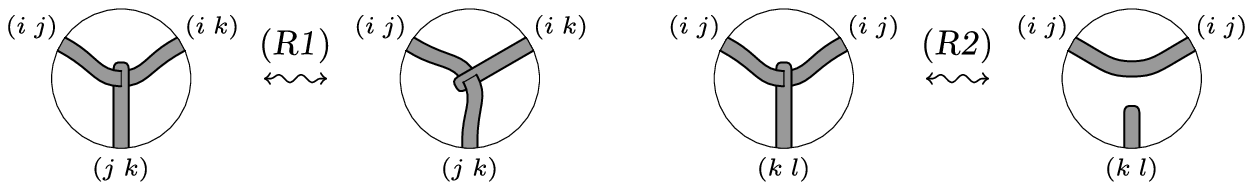}}
\vskip-6pt
\end{Figure}

The coherent monodromy merging principle provides an easy way to verify that \(M1) and
\(R1) are covering moves. This is shown in Figure \ref{coverings06/fig}, where the
principle is applied in the first and in the last step for both the moves, the middle step
being just labeled isotopy. Moves \(R2) and \(M2) are nothing else than simple
applications of the disjoint monodromies principle, as we observed above.

\begin{Figure}[htb]{coverings06/fig}
{}{Moves \(M1) and \(R1) are local covering moves}
\centerline{\fig{}{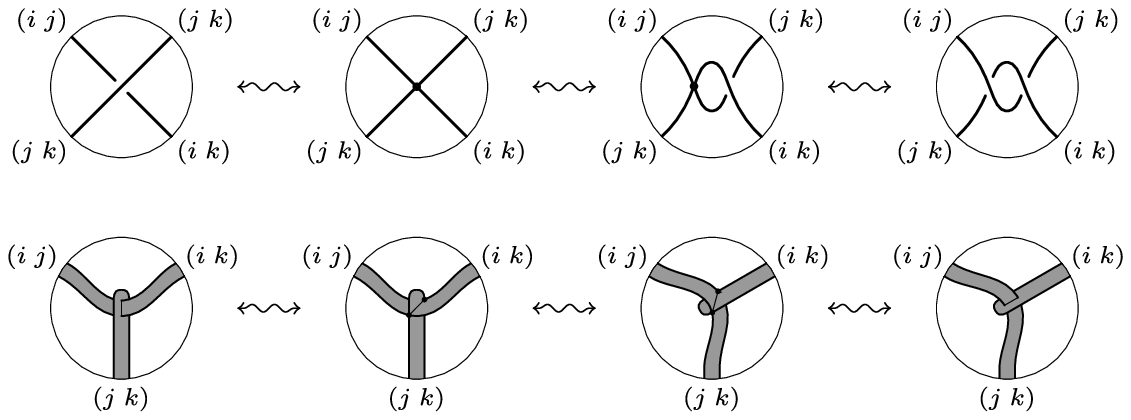}}
\vskip-6pt
\end{Figure}

In the next proposition, we derive from \(R1) and \(R2) the auxiliary covering moves 
\(R3) to \(R6) depicted in Figure \ref{coverings07/fig}, which will be very useful in 
the following chapters. 

\begin{Figure}[htb]{coverings07/fig}
{}{Other moves for labeled ribbon surfaces ($i,j,k$ and $l$ all different)}
\vskip-3pt
\centerline{\fig{}{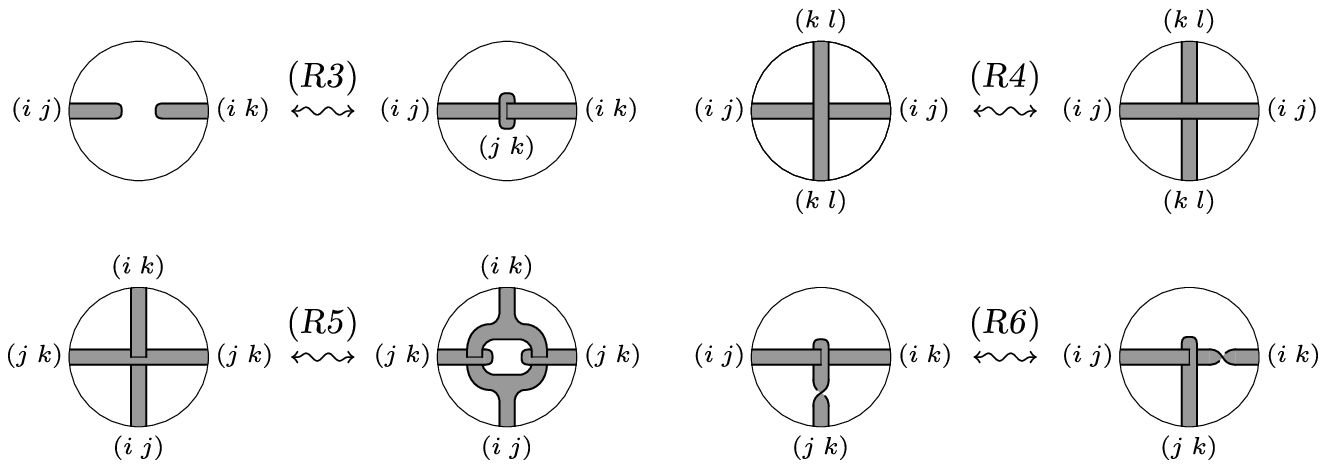}}
\vskip-6pt
\end{Figure}

\begin{proposition} \label{moves-aux/thm}
Up to labeled 1-isotopy, the local moves \(R3) to \(R6) depicted in Figure
\ref{coverings07/fig} can be generated the ribbon moves \(R1) and \(R2), hence they are
covering moves.
\end{proposition}

\begin{proof}
Move \(R4) can be easily obtained as the composition of two moves \(R2). Figures
\ref{coverings08/fig}, \ref{coverings09/fig} and \ref{coverings10/fig}
respectively shows how to get moves \(R3), \(R5) and \(R6) in terms of labeled 1-isotopy
and moves \(R1).

\begin{Figure}[b]{coverings08/fig}
{}{Deriving the covering move \(R3)}
\centerline{\fig{}{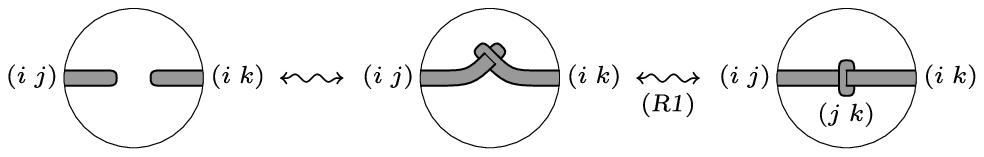}}
\vskip-6pt
\end{Figure}

\begin{Figure}[htb]{coverings09/fig}
{}{Deriving the covering move \(R5)}
\centerline{\fig{}{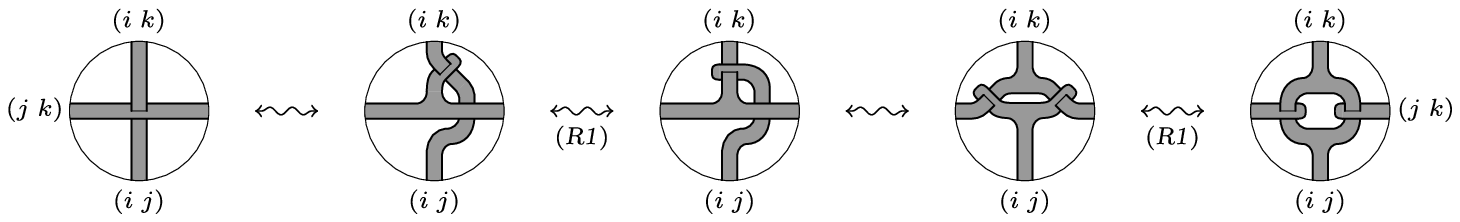}}
\vskip-6pt
\end{Figure}

\begin{Figure}[htb]{coverings10/fig}
{}{Deriving the covering move \(R6)}
\centerline{\fig{}{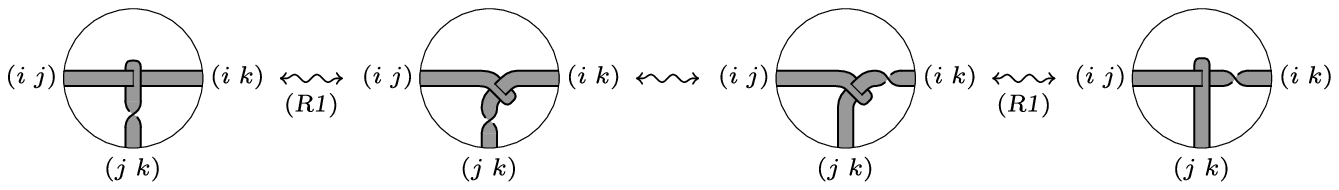}}
\vskip-6pt
\end{Figure}

Here, all the arrows which are not marked by \(R1) represent labeled 1-isotopy. Namely, we
use the following labeled 1-isotopy moves: \(S24) in Figure \ref{coverings08/fig};
\(S1), \(S2), \(S5), \(S21) and \(S25) in Figure \ref{coverings09/fig}; \(S14) and
\(S20) in Figure \ref{coverings10/fig}.
\end{proof}

\begin{remark} \label{orient/rem}
By labeled 1-isotopy and moves above, any $n$-labeled ribbon surface tangle $S$ can be
made orientable under the mild hypothesis that there are enough different labels to
generate $\Sigma_n$, or equivalently that the covering space represented by $S$ is
connected. In fact, twist transfer \(R6) allows us to eliminate non-orientable bands as
shown in Figure \ref{coverings11/fig}.
\begin{Figure}[htb]{coverings11/fig}
{}{Making labeled ribbon surfaces orientable ($i,j$ and $k$ all different)}
\centerline{\fig{}{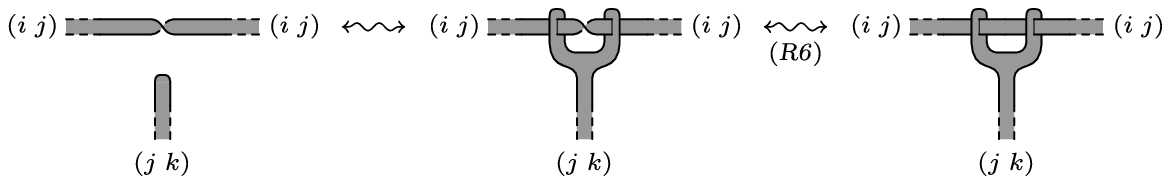}}
\vskip-6pt
\end{Figure}
\end{remark}

\subsection{Categories%
\label{categories/sec}}

We list here some basic definitions and statements from the general theory of
categories, which are used repeatedly in the paper. A complete reference on the subject 
is \cite{McL71}.

\medskip

Given a category $\C$, we denote by $\Obj \C$ the set of its objects and by $\Mor \C$ the
set of morphisms of $\C$, always assuming that $\C$ is a small category. Moreover, for any
$C,C' \in \Obj \C$, the set of morphisms of $\C$ with source $C$ and target $C'$ are
denoted by $\Mor_\C(C,C')$.

\begin{definition}\label{nat-equiv/def} 
Given two functors $S,T: \C \to \D$, a {\sl natural transformation} $\tau: S \to T$ is a
map that assigns to each object $C \in \Obj\C$ a morphism $\tau_C: S(C) \to T(C)$ of $\D$,
in such a way that for every morphism $f: C \to C'$ of $\C$, the following diagram
commutes.

\vskip3pt
\centerline{\epsfbox{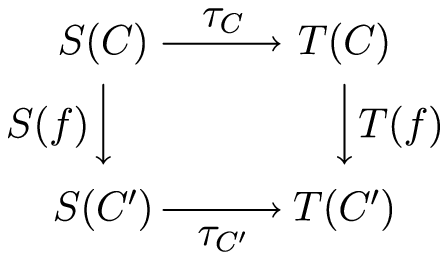}}
\vskip6pt

A natural transformation $\tau$ is called {\sl natural equivalence}, if $\tau_C$ is an 
isomorphisms for every $C \in \Obj \C$. In this case, we write $\tau: S \simeq T$.
\end{definition}

\begin{definition}\label{equiv-cat/def}
Two categories $\C$ and $\D$ are said to be {\sl equivalent} if there exist functors
$T:\C \to \D$ and $T':\D \to \C$, such that $T' \circ T \simeq \id_\C$ and $T \circ T' 
\simeq \id_\D$. In this case we call $T$ (and $T'$ as well) an {\sl equivalence of 
categories}.
\end{definition}

We remind that a functor $T:\C \to \D$ is {\sl faithful} (resp. {\sl full}\/) if for
every $C,C' \in \Obj \C$ the induced map $\Mor_\C(C,C') \to \Mor_\D(T(C),T(C'))$ is
injective (resp. surjective).

\begin{proposition}\label{cat-equiv/thm}
A functor $\,T:\C \to \D$ is an equivalence of categories if and\break only if it is full
and faithful and for any object $D \in \Obj\D$ there exists an object $C \in \Obj\C$ and
an isomorphism $\sigma_D: D \to T(C)$.
\end{proposition}

\begin{proof}
Suppose that $T$ is equivalence of categories. Then there exist a functor $T': \D \to \C$
and natural equivalences $\tau: T'\circ T \simeq \id_\C$ and $\sigma: T \circ T' \simeq
\id_\D$. If $T(f_1) = T(f_2)$ for some $f_1,f_2 \in \Mor_\C(C,C')$, we have $T'(T(f_1)) =
T'(T(f_2))$. Since $T'(T(f_i)) = \tau_{C'}^{-1} \circ f_i \circ \tau_C$ we can conclude
that $f_1 = f_2$. Therefore $T$ is faithful, and symmetrically $T'$ is faithful as well.
To see that $T$ is full, take $g \in \Mor_\D(T(C),T(C'))$ and let $f = \tau_{C'} \circ
T'(g) \circ \tau_C^{-1} \in \Mor_\C(C,C')$. Then $T'(T(f)) = \tau_{C'}^{-1} \circ f \circ
\tau_C = T'(g)$. Since $T'$ is faithful, this implies that $T(f) = g$, which proves that
$T$ is full. Moreover, any $D \in \Obj\D$ is isomorphic to $T(T'(D))$ through $\sigma_D$.
This completes the proof that the condition is necessary.

To see that the condition is sufficient, suppose that $T: \C \to \D$ is full and
faithful and that for any object $D \in \Obj\D$ there exists an object $C \in \Obj\C$ and
an isomorphism $\sigma_D: D \to T(C)$. Then define the functor $T': \D \to \C$ as 
$T'(g) = T^{-1}(\sigma_{D'} \circ g \circ \sigma_D^{-1})$ for any $g \in \Mor_\D(D,D')$, 
and put $\tau_C: T^{-1}(\sigma_{T(C)})$ for any $C \in \Obj \C$. We leave to the reader 
the rest of the details.
\end{proof}

We remind that a subcategory $\B\subset \C$ is said to be {\sl full} if $\Mor_\B(B,B') =
\Mor_\C(B,B')$ for any $B,B' \in \B$, in other word if the inclusion functor $\iota: \B
\hookrightarrow \C$ is a full functor.

\begin{corollary}\label{subcat-equiv/thm} 
Given a subcategory $\,\B \subset \C$, the inclusion functor $\iota: \B \hookrightarrow 
\C$ is an equivalence of categories if and only if $\,\B$ is full and any object of $\,\C$ 
is isomorphic to an object of $\,\B$.
\end{corollary}

\begin{definition}\label{strict-mon-cat/def} 
We say that $\C$ is a {\sl strict monoidal category} if there exists an associative
bifunctor $\diam: \C \times \C \to \C$ and an object $\one$ which is right and
left unit for $\diam\,$. Then $\diam$ is called the {\sl product} on $\C$ and $\id_\one$
is called the {\sl unit} of $\diam\,$.
\end{definition}

\noindent Being more explicit, the bifunctor $\diam$ consists in two associative binary 
operations
$$
\begin{array}{c}
  \diam: \Obj(\C) \times \Obj(\C) \to \Obj(\C)\,,\\[4pt]
  \diam: \Mor_\C(A,B) \times \Mor_\C(A',B') \to \Mor_\C(A \,\diam B, A' \diam B')\,,
\end{array}
$$
such that
$$\one \diam A = A = A \diam \one \quad \text{and} \quad 
\id_A \diam \id_B = \id_{A \diam B}$$
for any $A,B \in \Obj\C$, while
$$\id_\one \diam f = f = f \diam \id_\one \quad \text{and} \quad 
(f' \diam g') \circ (f \diam g) = (f' \circ f) \diam (g' \circ g)$$
for any $f \in \Mor_\C(A,B), f' \in \Mor_\C(B,C), g \in \Mor_\C(A',B'), g' \in 
\Mor_\C(B',C')$.

\begin{definition}\label{braided-cat/def} 
A strict monoidal category $\C = (\C,\diam\,,\one)$ is called {\sl braided}\break if for
any $A,B \in \Obj(\C)$ there exists natural isomorphism $\gamma_{A,B}: A \diam B \to B
\diam A$ such that
$$\gamma_{A, B \diam C} = (\id_B \diam \gamma_{A,C}) \circ (\gamma_{A,B} \diam \id_C)
\quad \text{and} \quad
\gamma_{A \diam B, C} = (\gamma_{A,C} \diam \id_C) \circ (\id_A \diam \gamma_{B,C})\,.$$
We recall that a family of morphisms $\nu_A: A \to B$ is called {\sl natural} if for any
morphism $f: A \to A'$ we have $\nu_{A'} \circ f = f \circ \nu_A$.
\end{definition}

\begin{definition}\label{autonomous/def}
A braided strict monoidal category $\C = (\C,\diam\,,\one,\gamma)$ is called 
{\sl autonomous} (or rigid) (see \cite{Sh94}) if for every $A \in \Obj \C$ it is given a 
{\sl right dual} object $A^\ast \in \Obj \C$ and two morphisms
$$
\begin{array}{ll} 
\Lambda_A: \one \to A^\ast \diam A \ \ & \text{({\sl coform}\/)\,,}\\[2pt]
  \lambda_A: A \diam A^\ast \to \one \ \ &\text{({\sl form}\/)\,,}
\end{array}
$$
such that the compositions
$$
\begin{array}{ccccccccccc}
  A &\mapright{} & A \diam \one & \mapright{\id {\diam} \Lambda_A} 
  & A \diam (A^\ast \diam A) & \mapright{} & (A \diam A^\ast) \diam A   
  &\mapright{\lambda_A {\diam} \id} & \one \diam A & \mapright{} & A\,,\\[4pt]
  A^\ast & \mapright{} & \one \diam A^\ast & \mapright{\Lambda_A {\diam} \id}
  & (A^\ast \diam A) \diam A^\ast & \mapright{} & A^\ast \diam (A \diam A^\ast)
  & \mapright{\id {\diam} \lambda_A} &A^\ast \diam \one & \mapright{} & A^\ast,
\end{array}
$$
are the identities. Then, given any morphism $F: A \to B$ in $\C$, we define its 
{\sl dual} $F^\ast: B^\ast \to A^\ast$ as follows:
$$F^\ast = (\id_{A^\ast} \diam \lambda_B) \circ (\id_{A^\ast} \diam F \diam \id_{B^\ast})
\circ (\Lambda_A \diam \id_{B^\ast})\,.$$
\end{definition}

\begin{definition}\label{tortile/def}
A {\sl twist} in a braided strict monoidal category $\C = (\C, \diam\,,\break \one,
\gamma)$ is a family of natural isomorphisms $\theta_A: A \to A$ with $A \in \Obj\C$, such
that $\theta_{\one} = \id_{\one}$ and $$\theta_{A \diam B} = \gamma_{B,A} \circ (\theta_B
\diam \theta_A) \circ \gamma_{A,B}$$
for any $A,B \in \Obj\C$.

An autonomous braided strict monoidal category $\C = (\C, \diam\,, \one, \gamma, \_^*,
\Lambda, \lambda)$\break equipped with a distinguished twist such that $\theta_{A^\ast} =
(\theta_A)^\ast$ for any object $A \in \Obj\C$ is called {\sl tortile} (the terminology is
from \cite{Sh94}).
\end{definition}

Many of the categories that we use in the present work are strict monoidal categories
generated by certain set of elementary morphisms and relations between those morphisms. 
We outline here the general construction of such categories (see also section 2 in
\cite{Sh94} and the references thereby).

Let $S$ be a (finite) set of {\sl elementary objects}. We denote by $\seq S = \cup_{m=
0}^\infty S^m$ the free monoid generated by $S$, concretely realized as the set of
(possibly empty) finite sequences of objects in S, with monoidal product $\diam: \seq S
\times \seq S \to \seq S$ given by juxtaposition of sequences and unit $\one$ given by the
empty sequence (cf. Corollary II.7.2 in \cite{McL71}). Let also $G(S,E)$ be a directed
graph, having $\seq S$ as set of vertices and a finite set $E$ of arrows (oriented
edges), which we call {\sl elementary morphisms}. We will always assume that $E$
contains a distinguished loop arrow $\langle \id_s \rangle$ starting and ending at $s$,
for every $s \in S$.

Let now $\bar G(S,E)$ be the graph with the same set of vertices $\seq S$ and arrows
$$\sigma \diam e \diam \sigma': \sigma \diam \sigma_0 \diam \sigma' \to \sigma \diam
\sigma_1 \diam \sigma'$$ for every $e: \sigma_0 \to \sigma_1 \in E$ and every $\sigma,
\sigma' \in \seq S$. We call any arrow of this form an {\sl expansion} of the elementary
morphism $e$ and use the notations $\sigma \diam e = \sigma \diam e \diam \one$ and $e
\diam \sigma' = \one \diam e \diam \sigma'$. Moreover, we identify $e$ with $\one \diam e
\diam \one$ for every $e \in E$, in such a way that $G(S,E) \subset \bar G(S,E)$.

Finally, we denote by $F(S,E)$ the free category generated by the graph $\bar G(S,E)$. We
recall from \cite{McL71} that set of objects of $F(S,E)$ is $\seq S$, the set of vertices
of the\break graph, while the set of morphisms $\Mor_{F(S,E)}(\sigma_0, \sigma_1)$
consists of all paths of consecutive arrows from $\sigma_0$ to $\sigma_1$ in $\bar
G(S,E)$. Here, the composition is given by path concatenation and the identity morphisms
represented by the paths of length 0.

The notion of expansion can be extended in a unique way to a compositive biaction of $\seq 
S$ on the morphisms of $F(S,E)$. Explicitly, this is defined by the formulas
\vskip-4pt
$$
\begin{array}{c} 
\sigma \diam (\tau \diam e \diam \tau') \diam \sigma' =
(\sigma \diam \tau) \diam e \diam (\tau' \diam \sigma')\,,\\[4pt]
\sigma \diam (f_1 \dots f_n) \diam \sigma' = (\sigma \diam f_1 \diam \sigma') 
\dots (\sigma \diam f_n \diam \sigma')\,,
\end{array}
$$
for every $e \in E$, every $\sigma,\sigma',\tau, \tau' \in \seq S$ and every $f_1, \dots, 
f_n \in \bar G(E,S)$ with $n \geq 0$. We still call $\sigma \diam f \diam \sigma'$ an {\sl 
expansion} of $f$ for every $f = f_1 \dots f_n \in \Mor F(S,E)$. In particular, for $n = 
0$ we have that identities expand to identities.

\begin{proposition}\label{presentation/thm}
Given $S$ and $E$ as above and $R = \{R(\sigma_0,\sigma_1)\,|\,\sigma_0,\sigma_1 \in \seq
S\}$, with each $R(\sigma_0,\sigma_1)$ being a (possibly empty) relation on
$\Mor_{F(S,E)}(\sigma_0, \sigma_1)$, let $C(S,E,R)$ be the quotient category of $F(S,E)$
modulo the equivalence relations generated by:
\begin{itemize}
\item[\(a)] 
$\langle \id_s \rangle \simeq \id_s$ for every $s \in S$;
\item[\(b)]
$\sigma \diam \id_s \diam \sigma' \simeq \id_{\sigma \diam s \diam \sigma'}$ for every 
$s \in S$ and $\sigma,\sigma' \in \seq S$; 
\item[\(c)]
$(f \diam \sigma '_1) \circ (\sigma_0 \diam f') \simeq (\sigma_1 \diam f') \circ (f \diam
\sigma_0')$ for any $f: \sigma_0 \to \sigma_1,f'\!: \sigma'_0 \to \sigma'_1 \in F(S,E)$;
\item[\(d)] 
all the expansions of the relations in $R$.
\end{itemize}
Then $C(S,E,R)$ admits a strict monoidal structure defined by $$f \diam f' = (f \diam
\sigma '_1) \circ (\sigma_0 \diam f') = (\sigma_1 \diam f') \circ (f \diam \sigma_0')$$
for any $f: \sigma_0 \to \sigma_1,f'\!: \sigma'_0 \to \sigma'_1 \in C(S,E,R)$, with the
unit object $\one$ being the empty set.
\end{proposition}

\begin{proof}
In the light of the definitions, the proof is straightforward.
\end{proof}

\begin{definition}\label{presentation/def}
We call the strict monoidal category $C(S,E,R)$ defined in the previous proposition the
{\sl strict monoidal category generated by the set $S$ of elementary objects and the set
$E$ of elementary morphisms modulo the set $R$ of elementary relations}. In the case when
$R = B \cup A$, where $B$ represents defining axioms of a braided structure in a monoidal
category, while $E$ and $A$ represent respectively the basic morphisms and the defining
axioms which endow $S$ with a certain algebraic structure, for example braided Hopf
algebra, we will refer to the category $C(S,E,R)$ as the {\sl braided strict monoidal
category freely generated by the algebraic structure $S$} (cf. \cite{Ke02}, Section 4,
Definition 1).
\end{definition}

\newpage

\section{4-dimensional cobordisms and Kirby tangles%
\label{cobordisms/sec}}

In this chapter we introduce 4-dimensional relative 2-handlebody cobordisms between 
3-dimensional 1-handlebodies and their representations in terms of bridged tangles and 
Kirby tangles. 

We allow handlebodies to have more than one 0-handles and even to be disconnected, then we
need to generalize the usual notions of bridged tangle and Kirby tangle to include the
case of multiple 0-handles. This is done in Section \ref{Kirby/sec}.

{\sl All the handlebodies in this chapter are assumed to be orientable.} According to the
discussion in Section \ref{handles/sec}, this means that 1-handles can be specified
just by the unordered pair of (possibly coinciding) 0-handles where they are attached.

\subsection{The categories of relative handlebody cobordisms $\Chb_n^{3+1}$%
\label{Chbn/sec}}

Here, we define the category $\Chb_n^{3+1}$ of 4-dimensional relative 2-handlebody
co\-bordisms between 3-dimensional 1-handlebodies with $n$ 0-handles, for any $n \geq 1$.

An object in $\Chb_n^{3+1}$ is a 3-dimensional relative 1-handlebody $M$ build on the
disjoint union $\sqcup_{i=1}^n \Int B_i^2$ of $n$ disks, having no 0-handles. With some
abuse of notation we write $M = \cup_{i=1}^n H^0_i\cup_{j=1}^m H^1_j$ and call $H^0_i =
B_i^2 \times [0,1]$ the $i$-th 0-handle of $M$.\break While the $H^1_j$'s are the
1-handles whose attaching regions are contained in $\delta M^0 = \sqcup_{i=1}^n \Int B_i^2
\times \{1\}$ (cf. Definition \ref{handlebody/def}). We define the {\sl front boundary} of
$M$\label{front-bdM}\ as
$$\partial M = \Cl \delta M = \Cl(\Bd M - \cup_{i=1}^n (B^2_i 
\times \{0\}) \cup (\Bd B^2_i \times [0,1]))\,.$$

\begin{Figure}[b]{cobordism01/fig}
{}{The base space of a relative handlebody cobordism in $\Chb_n^{3+1}$}
\vskip3pt
\centerline{\fig{}{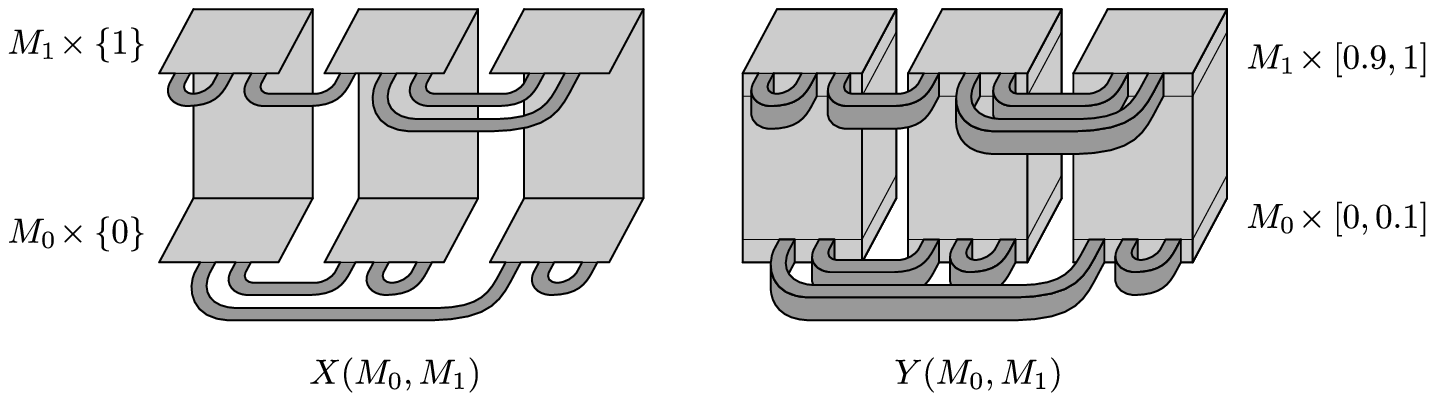}}
\vskip-6pt
\end{Figure}

Given two objects $M_0 = \cup_{i=1}^n H^0_{i,0} \cup_{j=1}^{m_0} H^1_{j,0}$ and $M_1 =
\cup_{i=1}^n H^0_{i,1} \cup_{k=1}^{m_1} H^1_{k,1}$,\label{X/def}\ let $X(M_0, M_1)$ be the
3-dimensional 1-handlebody obtained from $M_0 \sqcup M_1$ by attaching for any $i \leq n$
a single 1-handle $\bH_i^1 \cong B_i^2 \times [0,1]$ between the $i$-th 0-handles of $M_0$
and $M_1$, through the identification of $B^2_i\times\{c\}\subset \bH_i^1$ with
$B^2_i\times\{0\} \subset H^0_{i,c}$ for $c = 0,1$. Actually, one could cancel the new
1-handles against some of the 0-handles, thinking of $X(M_0, M_1)$ as a 3-dimen\-sional
1-handlebody with $n$ 0-handles of the form $\bH^0_i = H^0_{i,0} \cup H^0_{i,1} \cup
\bH^1_i$.

Now, let $Y(M_0,M_1) = \sqcup_{i=1}^n B^2_i \times [0,1] \times [0,1] \cup (M_0 \times
[0,0.1]) \cup (M_1 \times [0.9,1])$.\break We identify $X(M_0,M_1) \times [0,1]$ with
$Y(M_0,M_1)$, in such a way that $X(M_0,M_1) \times \{0\}$ is canonically identified with
$\sqcup_{i=1}^n B^2_i\times \{0\} \times [0,1] \cup (M_0 \times \{0\}) \cup (M_1 \times
\{1\}) \subset Y(M_0,M_1)$, while $X(M_0,M_1) \times \{1\}$ corresponds to $\sqcup_{i=1}^n
(B^2_i \times \{1\} \times [0.1,0.9]) \cup_{j=1}^{m_0} (H_{j,0}^1 \times \{0.1\})
\cup_{k=1}^{m_1} (H_{k,1}^1 \times \{0.9\}) \subset Y(M_0,M_1)$. An example of
$X(M_0,M_1)$ and $Y(M_0,M_1)$ is presented in Figure \ref{cobordism01/fig}, where all the
horizontal segments represent copies of $B^2$.

A morphism $W: M_0 \to M_1$ in $\Chb_n^{3+1}$ is a 4-dimensional relative 2-handlebody
build on $X(M_0, M_1)$ without 0-handles, considered up to 2-deformation of relative
handlebodies (cf. Section \ref{handles/sec}). We remind that $W$ is obtained by attaching
1- and 2- handles to $W^0 = W^{-1} = X(M_0,M_1) \times [0,1]$. According to the setting
above, we identify $W^0$ with $Y(M_0,M_1)$, in such a way that $M_0$ and $M_1$ correspond
to $M_0 \times \{0\} \subset Y(M_0,M_1)$ and $M_1 \times \{1\} \subset Y(M_0,M_1)$
respectively. Moreover, we define the {\sl front boundary} of $W$\label{front-bdW}\ as the
following bounded 3-manifold in $\Bd W$: $$\partial W = \Cl(\Bd W - (X(M_0,M_1) \times
\{0\}) - \cup_{i=1}^n (\Bd B^2_i \times [0,1] \times [0,1]))\,.$$ In particular, if $W$
has no 1- and 2-handles we have: $$\partial Y(M_0,M_1) \cong (\partial M_0 \times [0,1])
\cup (X(M_0,M_1) \times \{1\}) \cup (\partial M_1 \times [0,1])\,.$$

The composition of two morphisms $W_1: M_0\to M_1$ and $W_2: M_1 \to M_2$ in\break
$\Chb_n^{3+1}$ is the morphism $W: M_0 \to M_2$ defined in the following way. Let $W = W_1
\cup_{M_1}\! W_2$ be the space obtained from $W_1 \sqcup W_2$ by identifying the target
$M_1 \subset W_1$ of $W_1$ with the source $M_1 \subset W_2$ of $W_2$, through the
identity. Up to rescaling the last coordinate, we regard the subspace $Y(M_0,M_1)
\cup_{M_1}\! Y(M_1,M_2) \subset W$ as $Y(M_0,M_2)$ with certain 1-handles attached to it
on the part of the boundary corresponding to the interior of $X(M_0,M_2) \times \{1\}$.
Namely, we have one (4-dimensional)\break 1-handle $\bH_k^1 = (H_{k,1}^1 \times [0.9,1])
\cup_{H_{k,1}^1}\! (H_{k,1}^1 \times [0,0.1])$ attached to $Y(M_0,M_2)$ for each
(3-dimensional) 1-handle $H_{k,1}^1$ of $M_1$ (see Figure \ref{cobordism02/fig} for an
example). Now, all the handles of the relative handlebodies $W_1$ and $W_2$ are attached
to $Y(M_0,M_1) \cup_{M_1}\! Y(M_1,M_2) = Y(M_0,M_1) \cup_{k=1}^{m_1}\! \bH_k^1$, to endow
$W$ of a structure of 4-dimensional relative 2-handlebody build on $X(M_0,M_2)$. We notice
that $\partial W = \partial W_1 \cup_{\partial M_1}\! \partial W_2$.

\begin{Figure}[htb]{cobordism02/fig}
{}{Composing two relative handlebody cobordisms in $\Chb_n^{3+1}$}
\vskip3pt
\centerline{\fig{}{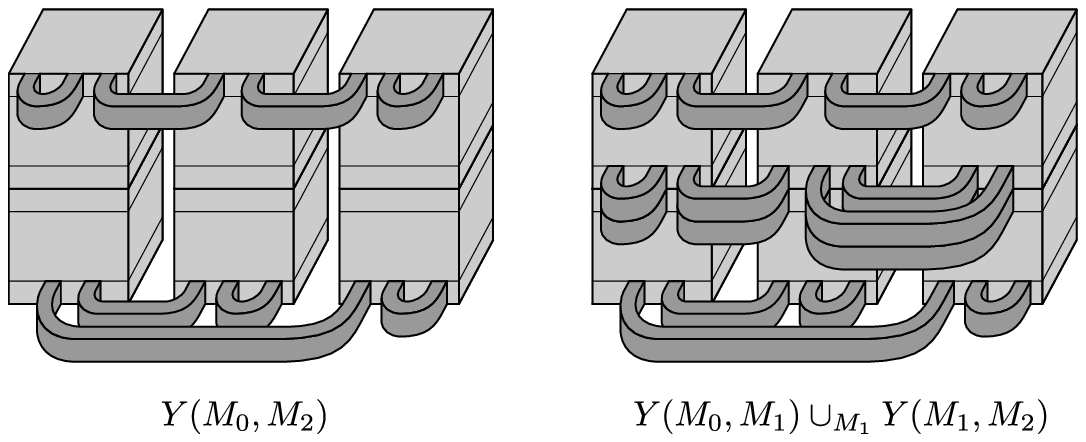}}
\vskip-6pt
\end{Figure}

The identity $\id_M$ of an object $M = \cup_{i=1}^n H^0_i \cup_{j=1}^m H^1_j$ is
represented by the product cobordism $M \times [0,1]$, whose handlebody structure is given
by one 2-handle $H_j^2 = H_j^1 \times [0.1,0.9] \subset M \times [0,1]$ for each 1-handle
$H_j^1$ of $M$ (cf. Figure \ref{cobordism03/fig}).

\begin{Figure}[htb]{cobordism03/fig}
{}{Identity and isotopy cobordisms}
\centerline{\fig{}{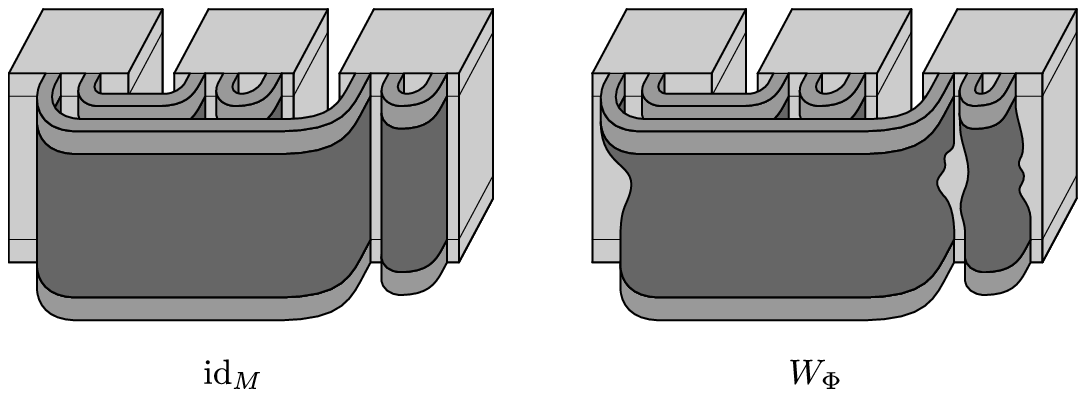}}
\vskip-6pt
\end{Figure}

Moreover, given any ambient isotopy $\Phi = (\phi_t)_{t \in [0,1]}$ of $\cup_{i=1}^n \Bd
H_i^0$ fixed outside $\cup_{i=1}^n B^2_i \times \{1\}$ and constant in the time
intervals $[0,0.1]$ and $[0.9,1]$, we consider the new object $M'$ which differs from $M$
only for the attaching maps of the 1-handles being isotoped through $\Phi$. In other
words, if $h_j^1$ is the attaching map of the 1-handle $H_j^1$ in $M$, then $\phi_1
\circ h_j^1$ is the attaching maps of the same 1-handle in $M'$. Now, for any $j \leq
m$ we attach $H_j^1 \times [0,1]$ to $\cup_{i=1}^n H^0_i \times[0,1]$ by the map $(x,t)
\mapsto (\phi_t(h_j^1(x)),t))$. In this way, we obtain a cobordism $W_\Phi$ between $M$
and $M'$ which has a natural structure of 4-dimensional relative 2-handlebody build on
$X(M,M')$, with one\break 2-handle $H_j^2$ given by $H_j^1 \times [0.1,0.9] \subset
W_\Phi$ for each 1-handle $H_j^1$ of $M$ (cf. Figure \ref{cobordism03/fig}). Then, the
morphism $W_\Phi: M \to M'$ represented by such cobordism is an isomorphism, whose inverse
is $W_{\Phi^{-1}}:M' \to M$ with $\Phi^{-1} = (\phi_t^{-1})_{t \in [0,1]}$.

According to Corollary \ref{subcat-equiv/thm}, the fact that there is an isomorphism
between any two 3-dimensional handlebodies which are obtained by changing the attaching
maps through isotopy, allows us to replace the category $\Chb_n^{3+1}$ with the full
subcategory equivalent to it, whose objects are standard handlebodies.

\medskip

For $n \geq 1$, let $\G_n = \{1,\dots,n\}^2$ be the set of ordered pairs of integers
between 1 and $n$ (the reason for this notation is that in Chapter \ref{algebra/sec} we
will consider $\G_n$ as a groupoid) and let $\seq\G_n = \cup_{m=0}^\infty \G_n^m$ be
the set of (possibly empty) finite sequences of pairs in $\G_n$.

\begin{definition}\label{standard-hb/def}
Given any $\pi = ((i_1,j_1), \dots, (i_m,j_m)) \in \seq\G_n$ we define the
{\sl standard 3-dimensional 1-handlebody} $M^n_\pi = \cup_{i=1}^n H^0_i \cup_{j=1}^m
H^1_j$ as follows: 
\begin{itemize}
\item[1)]
for $1 \leq i \leq n$, let $H_i^0 = E_i \times [0,1]$ be a copy of $E \times [0,1]$, with 
$E = [0,1]^2$ the standard square;%
\item[2)]
if $m \geq 1$, we consider the $2m$ boxes
$$\begin{array}{c}
b'_{m,k} = [(k - 0.8)/m, (k - 0.7)/m] \times [0.4,0.6] \subset E\\[6pt]
b''_{m,k} = [(k - 0.3)/m, (k - 0.2)/m] \times [0.4,0.6] \subset E
\end{array}$$
with $1 \leq k \leq m$, and let $b'_{m,k,i\!}$ (resp. $b''_{m,k,i}$) be the copy of
$b'_{m,k\!}$ (resp. $b''_{m,k}$) in $E_i$;
\item[3)]
for $1 \leq k \leq m$, let $H_k^1$ be the 1-handle between the 0-handles $H_{i_k}^0$ and
$H_{j_k}^0$ given by the identification of $b'_{m,k,i_k} \!\times \{1\}$ with
$b''_{m,k,j_k} \!\times \{1\}$ through the map $(x,y,1) \mapsto (x + 0.5/m,1-y,1)$.
\end{itemize}
In particular, $M^n_\emptyset = \cup_{i=1}^n H^0_i = \cup_{i=1}^n E_i \times [0,1]$ is the 
disjoint union of $n$ 3-balls. 
\end{definition}

If $M^n_{\pi_0}$ and $M^n_{\pi_1}$ are two standard 3-dimensional 1-handlebodies with
$\pi_0,\pi_1 \in \seq\G_n$, then $X(M^n_{\pi_0}, M^n_{\pi_1})$ can be thought as a 
quotient of $\sqcup_{i=1}^n E_i \times [0,1]$, up to canonical identifications of 
$H_{i,0}^0 = E_i \times [0,1] \cong E_i \times [0,1] \times \{0\}$ and $H_{i,1}^0 = E_i 
\times [0,1] \cong E_i \times [0,1] \times \{1\}$ respectively with $E_i \times [0,0.1]$ 
and $E_i \times [0.9,1]$.

The standard 3-dimensional 1-handlebodies $M^3_{\pi_0}$ and $M^3_{\pi_1}$ for $\pi_0 =
((1,3), (1,2),\break (2,2), (3,3))$ and $\pi_1 = ((1,1), (1,2), (2,3), (2,3))$, which are
respectively isomorphic to $M_0$ and $M_1$ of Figure \ref{cobordism01/fig}, are shown on
the left side of Figure \ref{cobordism04/fig}. Here, the arrows indicate the prescribed
identifications. On the right side of the same Figure \ref{cobordism04/fig} is shown a
copy of $X(M^3_{\pi_0}, M^3_{\pi_1})$ described as a quotient of $\sqcup_{i=1}^3 E_i
\times [0,1]$.

\begin{Figure}[htb]{cobordism04/fig}
{}{Standard 3-dimensional 1-handlebodies}
\centerline{\fig{}{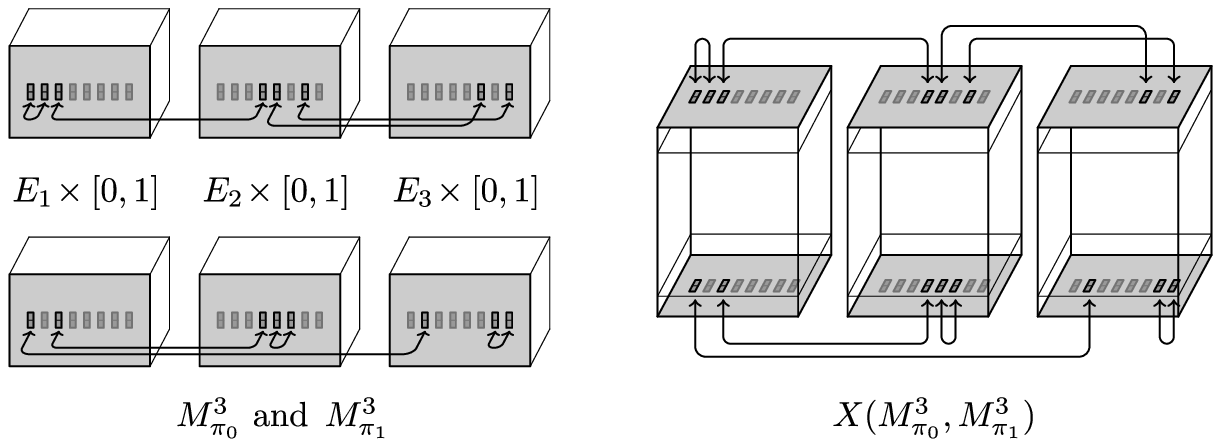}}
\vskip-6pt
\end{Figure}

\begin{proposition}\label{standard-hb/thm}
For any $n \geq 1$, the category $\Chb_n^{3+1}$ is equivalent to its full subcategory,
whose objects are the standard 3-dimensional 1-handlebodies $M^n_\pi$ with $\pi \in
\seq\G_n$, through the inclusion functor.
\end{proposition}

\begin{proof}
This is an immediate consequence of Corollary \ref{subcat-equiv/thm}.
\end{proof}

{\sl From now on, $\Chb_n^{3+1}$ will denote the full subcategory given in Proposition
\ref{standard-hb/thm}.}

\medskip

The set $\seq\G_n$, can be endowed with a monoidal structure, where the product $\pi \diam
\pi'$ is the juxtaposition of $\pi$ and $\pi'$ and the unit element is the empty sequence. 
This induces a monoidal structure on $\Obj \Chb_n^{3+1} = \{M^3_\pi \;|\; \pi \in
\seq\G_n\}$, which will be denoted in the same way: $$M^n_\pi \diam M^n_{\pi'} = M^n_{\pi
\diam \pi'}.$$

We observe that the unit of this product is $M^n_\emptyset$, the disjoint union of $n$
3-balls, and that given $\pi = ((i_1,j_1), \dots , (i_m,j_m)) \in \seq\G_n$ we have:
$$M^n_\pi = M^n_{(i_1,j_1)} \diam \dots \diam M^n_{(i_m,j_m)}.$$

Moreover, $M^n_\pi \diam M^n_{\pi'}$ is the 3-dimensional handlebody obtained from
$M^n_\pi \sqcup M^n_{\pi'}$\break by glueing the $i$-th 0-handle of $M^n_\pi$ to that of
$M^n_{\pi'}$ through the map $(1,y,z) \mapsto (0,y,z)$ and then suitably reparametrizing
the first coordinate.

Now, we extend the monoidal structure on $\Obj \Chb_n^{3+1}$ to the whole category
$\Chb_n^{3+1}$ in the following way. Let $W: M^n_{\pi_0} \to M^n_{\pi_1}$ and $W':
M^n_{\pi'_0} \to M^n_{\pi'_1}$ be two morphisms in $\Chb_n^{3+1}$. Starting from
$X(M^n_{\pi_0}, M^n_{\pi_1}) \sqcup X(M^n_{\pi_0'}, M^n_{\pi_1'})$, we glue the $i$-th
0-handle of $X(M^n_{\pi_0}, M^n_{\pi_1})$ (consisting of the $i$-th 0-handles of
$M^n_{\pi_0}$ and $M^n_{\pi_1}$ joined by the 1-handle $\bH_i^1$) to that of
$X(M^n_{\pi'_0}, M^n_{\pi'_1})$ through the map $(1,y,z) \mapsto (0,y,z)$. Then,
we identify the result of the gluing with $X(M^n_{\pi_0} \diam M^n_{\pi_0'}, M^n_{\pi_1}
\diam M^n_{\pi_1'})$, by thinking all the three $X$ spaces involved as quotients of
$\sqcup_{i=1}^n E_i \times [0,1]$ and applying a reparametrization of the first coordinate
depending on the last one, that is a diffeomorphism $(x,y,z) \mapsto (h^z(x),y,z)$ with
$h^z$ an increasing function of $x$ for every $z$.\break The same construction crossed 
by $[0,1]$, gives $Y(M^n_{\pi_0} \diam M^n_{\pi_0'}, M^n_{\pi_1} \diam M^n_{\pi_1'})$ 
starting from $Y(M^n_{\pi_0}, M^n_{\pi_1}) \sqcup Y(M^n_{\pi_0'}, M^n_{\pi_1'})$.

We define the product $W \diam W'$ to be the 4-dimensional relative 2-handlebody build on
$X(M^n_{\pi_0} \diam M^n_{\pi_0'}, M^n_{\pi_1} \diam M^n_{\pi_1'})$ that one obtains by
starting from $W \sqcup W'$ and performing the construction above to the subspace
$Y(M^n_{\pi_0}, M^n_{\pi_1}) \sqcup Y(M^n_{\pi_0'}, M^n_{\pi_1'})$. The handlebody
structure on $W \diam W'$ consists of all the handles of $W$ and $W'$.

The unit of the product between morphisms is the ``empty'' relative handlebody
$\id_{M^n_\emptyset} = Y(M^n_\emptyset, M^n_\emptyset): M^n_\emptyset \to M^n_\emptyset$
without any handle, that is the dis\-joint union of $n$ 4-balls.

\begin{proposition}\label{monoidal-str/thm}
The product $\,\diam: \Chb_n^{3+1} \times\, \Chb_n^{3+1}\to \Chb_n^{3+1}$ makes 
$\Chb_n^{3+1}$ into a strict monoidal category, for any $n \geq 1$.
\end{proposition}

\begin{proof}
The proof consists into observing that $\diam$ is associative and verifying the identities
which follow the Definition \ref{strict-mon-cat/def} of a strict monoidal category. This
is straightforward and we leave it to the reader.
\end{proof}

For $n > 1$ the category $\Chb_n^{3+1}$ may have quite complex structure, but what we
really want is to understand the category $\Chb_1^{3+1}$ ($n = 1$) of connected cobordisms
between connected 3-dimensional handlebodies. In this perspective, the categories with $n
> 1$ represent an intermediate step in order to establish in a natural way the connection
between the morphisms in $\Chb_1^{3+1}$ and the morphisms of the category of labeled
ribbon surfaces coming from the description of the 4-dimensional handlebodies as $n$-fold
branched coverings. In particular, we need a faithful functor (or embedding) from
$\Chb_1^{3+1} \to \Chb_n^{3+1}$ whose image defines a subcategory of $\Chb_n^{3+1}$
equivalent to $\Chb_1^{3+1}$.

First of all, we have that the inclusion $\seq\G_k \subset \seq\G_n$ with $1 \leq k < n$
induces inclusions $M^k_\pi \subset M^n_\pi$, $X(M^k_{\pi_0}, M^k_{\pi_1}) \subset
X(M^n_{\pi_0}, M^n_{\pi_1})$ and $Y(M^k_{\pi_0}, M^k_{\pi_1}) \subset Y(M^n_{\pi_0},
M^n_{\pi_1})$, for any $\pi,\pi_0,\pi_1 \in \seq\G_k$.

Then, we define the faithful functor $\iota_k^n: \Chb_k^{3+1} \subset \Chb_n^{3+1}$, by
putting $\iota_k^n(M^k_\pi) = M^n_\pi$ for any $M^k_\pi \in \Obj \Chb_k^{3+1}$ and
$\iota_k^n(W) = W \cup_{Y(M^k_{\pi_0}, M^k_{\pi_1})} Y(M^n_{\pi_0}, M^n_{\pi_1})$ for
any $W: M^k_{\pi_0} \to M^k_{\pi_1} \in \Mor \Chb_k^{3+1}$.

\begin{definition}\label{stabilization/def}
Given $n > k \geq 1$, let $\pi_{n \red k} = ((n,n-1), \dots, (k+1,k))$ and let
$\id_{\pi_{n \red k}}$ be the identity morphism of $M^n_{\pi_{n \red k}}$. Then, the
{\sl stabilization functor} $\up_k^n: \Chb_k^{3+1} \to \Chb_n^{3+1}$ is defined by:
$$
\begin{array}{c}
\up_k^n M^k_\pi = M^n_{\pi_{n \red k}} \!\diam \iota_k^n(M^k_\pi) 
\,\text{ for any } M^k_\pi \in \Obj \Chb_k^{3+1},\\[6pt]
\up_k^n W = \id_{\pi_{n \red k}} \!\diam \iota_k^n(W) 
\,\text{ for any } W \in \Mor \Chb_k^{3+1}.
\end{array}$$
Moreover, we put $\Chb_n^{3+1,c} = \up_1^n(\Chb_1^{3+1}) \subset \Chb_n^{3+1}.$
\end{definition}

We will show in Section \ref{K/sec} that the restriction of $\up_k^n$ to $\Chb_k^{3+1,c}$
gives a category equivalence between the $\Chb_k^{3+1,c}$ and $\Chb_n^{3+1,c}$. In
particular, $\Chb_1^{3+1} \cong \Chb_n^{3+1,c}$ for any $n \geq 1$. Even if the proof of
this will be somewhat technical, the reader\break should find the statement quite obvious.
Indeed $M^n_{\pi_{n \red 1}}\!$ is a 3-ball and therefore the product cobordism
$\id_{\sigma_{1 \to n}}\!$ is a 4-ball. Hence, we may think of the cobordism $\up_1^n W =
\id_{\pi_{n \red 1}} \!\diam \iota_1^n(W)$ as the handlebody obtained by attaching the 1-
and 2-handles of $W$ to a single 4-ball.

\subsection{Bridged tangles and labeled Kirby tangles%
\label{Kirby/sec}}

We recall that in the literature can be found two different diagrammatic descriptions of
4-dimensional 2-handlebodies with a single 0-handle: bridged links and Kirby diagrams.
Both are based on the representation of the attaching maps of the 1- and 2-handles in the
boundary of the 0-handle. In the case of a bridged link (cf. \cite{Ke99}), the 1-handles
are represented directly by drawing their attaching regions (disjoint pairs of 3-balls),
while in the case of a Kirby diagram (cf. \cite{Ki89, GS99}) the 1-handles are represented
instead by a ``dotted'' 0-framed trivial knot.

In \cite{Ke99} bridged tangles were used for describing cobordisms in $\Chb_1^{3+1}$ and
then in \cite{Ke02, KL01} ribbon tangles, equivalent under Kirby moves, are used to
describe the morphisms of the category $\Cob^{2+1}_1$ of 2-framed relative 3-cobordisms
which, as explained in Chapter \ref{boundaries/sec}, is the ``framed boundary'' of
$\Chb_1^{3+1}$.

Here we will generalize the notion of bridged tangles and Kirby tangles in order to be
able to describe the morphisms of $\Chb_n^{3+1}$, i.e. 4-dimensional relative
2-handle\-bodies build on $X(M^n_{\pi_0},M^n_{\pi_1})$, where $M^n_{\pi_0}$ and
$M^n_{\pi_1}$ are standard 3-dimensional handlebodies with $n$ 0-handles.

\begin{definition}\label{bridged-tangle/def}
For any $n \geq 1$ and any two finite sequences of ordered pairs $\pi_0 = ((i^0_1,j^0_1),
\dots, (i^0_{m_0},j^0_{m_0}))$ and $\pi_1 = ((i^1_1,j^1_1), \dots, (i^1_{m_1},j^1_{m_1}))$
in $\seq\G_n$, a {\sl bridged tangle} from $\pi_0$ to $\pi_1$ consists of the following
data:
\begin{itemize}
\item[1)]
the space $Z_n = \sqcup_{i=1}^n E_i \times [0,1]$ disjoint union of $n$ numbered copies of
$E \times [0,1]$, where $E = [0,1]^2$ denotes the standard square, with the subspaces and 
maps (cf. Definition \ref{standard-hb/def})
$$\mkern-6mu
\begin{array}{l}
\begin{array}{l}
B'_{\inp,k} = b'_{m_0,k,i^0_k} \times [0,0.1] \subset E_{i^0_k} \times [0,1]\\[4pt]
B''_{\inp,k} = b''_{m_0,k,j^0_k} \times [0,0.1] \subset E_{j^0_k} \times [0,1]\\
\rho_{\inp,k} : B'_{\inp,k} \!\to B''_{\inp,k} 
\text{ given by } (x,y,z) \mapsto (x + 0.5/m_0, 1-y, z)
\end{array}
\!\raise1.5pt\hbox{\nine$\left.\vphantom{\vrule height 32pt}\right\}$}
\;\text{for } 1 \leq k \leq m_0\,,\mkern-5mu\\[28pt]
\begin{array}{l}
B'_{\out,k} = b'_{m_1,k,i^1_k} \times [0.9,1] \subset E_{i^1_k} \times [0,1]\\[4pt]
B''_{\out,k} = b''_{m_1,k,j^1_k} \times [0.9,1] \subset E_{j^1_k} \times [0,1]\\
\rho_{\out,k}\!:\!B'_{\out,k} \!\to\! B''_{\out,k} 
\text{ given by } 
(x,y,z) \mapsto (x \!+\! 0.5/m_1, 1 \!-\! y, z)
\end{array}
\!\raise1.5pt\hbox{\nine$\left.\vphantom{\vrule height 32pt}\right\}$}
\;\text{for } 1 \leq k \leq m_1;\mkern-5mu
\end{array}
$$
\smallskip
\item[2)]
an orientation preserving embedding $\Phi: P \to \sqcup_{i=1}^n \Int E_i \times
\left]0.1,0.9\right[ \subset Z_n$, where $P = \sqcup_{k=1}^r P_k$ with $P_k = P'_k \sqcup
P''_k$ a copy of $B^3(0,0,2) \sqcup B^3(0,0,-2) \subset R^3$, the union of the pair of
unitary 3-balls with centers $(0,0,\pm2)$; we put
$$
\begin{array}{l}
B'_k = \Phi(P'_k)\\[6pt]
B''_k = \Phi(P''_k)\\[4pt]
\rho_k: B'_k \to B''_k \text{ given by } \Phi(x,y,z) \mapsto \Phi(x,y,-z)
\end{array}
\!\raise1.5pt\hbox{\nine$\left.\vphantom{\vrule height 30pt}\right\}$}
\;\text{for } 1 \leq k \leq r\,;
$$
\smallskip
\item[3)]
an embedding $\Psi: Q \to \sqcup_{i=1}^n \Int E_i \times [0.1,0.9] - \cup_{k=1}^r (\Int
B'_k \sqcup \Int B''_k) \subset Z_n$, with $Q$ the space obtained by cutting the disjoint
union $\sqcup_{h=1}^s A_h$ of $s$ copies of the annulus $A = S^1 \times [0,1]$ along a set
of meridian arcs $\alpha_j = \{p_j\} \times [0,1] \subset A_{h_j}$ for $j = 1, \dots, c$;
denoting by $\alpha'_j$ and $\alpha''_j$ the two copies of $\alpha_j$ in $\Bd Q$, we
require that: $\Psi(Q)$ meets the boundary of the ambient space transversally in
$\cup_{j=1}^c \Psi(\alpha'_j \sqcup \alpha''_j) \subset \cup_{k=1}^r (\Bd B'_k \cup \Bd
B''_k) \cup_{k=1}^{m_0} (\Bd B'_{\inp,k} \cup \Bd B''_{\inp,k}) \cup_{k=1}^{m_1} (\Bd
B'_{\out,k} \cup \Bd B''_{\out,k})$ and $(\cup_{k=1}^r \rho_k \cup_{k=1}^{m_0}
\rho_{\inp,k} \cup_{k=1}^{m_1} \rho_{\out,k}) (\Psi(\alpha'_j)) = \Psi(\alpha''_j)$ for
any $j = 1, \dots, c$.
\end{itemize}

\medskip

We will call $B'_{\inp,k}$ and $B''_{\inp,k}$ with $k = 1, \dots, m_0$ (resp.
$B'_{\out,k}$ and $B''_{\out,k}$ with $k = 1, \dots, m_1$) the {\sl in-boxes} (resp. {\sl
out-boxes}) of the bridged tangle, while the 3-balls $B'_k$ and $B''_k$ with $k = 1,
\dots, r$ will be called {\sl internal balls}.

\medskip

By a {\sl bridged tangle diagram} we mean the set of the in- and out-boxes together with
the images of $\Phi$ and $\Psi$ in $Z_n$. We will use the notation $T(\Phi,\Psi)$ for both 
the bridged tangle determined by $\Phi$ and $\Psi$ and its diagram.
\end{definition}

In a bridged tangle diagram we think of $\Psi(Q)$ as a set of framed curves. In the
figures, we always represent such a framed curve as a narrow band and draw the base curve
$C$ as a thick curve and the framing curve $C'$ as a ``parallel'' thin curve (cf. Section
\ref{handles/sec}). Of course, the choices of the base curve and the framing curve for
different components of $\Phi(Q)$ have to be compatible with the map $\cup_{k=1}^r \rho_k
\cup_{k=1}^{m_0} \rho_{\inp,k} \cup_{k=1}^{m_1} \rho_{\out,k}$, according to property 3 in
Definition \ref{bridged-tangle/def}.

An example of a bridged tangle diagram is presented in Figure \ref{bridged-tang01/fig}.
Here, $\pi_0$ and $\pi_1$ are as in Figure \ref{cobordism04/fig} and the arrows
specify the pairing of boxes and balls.

\begin{Figure}[htb]{bridged-tang01/fig}
{}{Bridged tangle representation of a cobordism}
\centerline{\fig{}{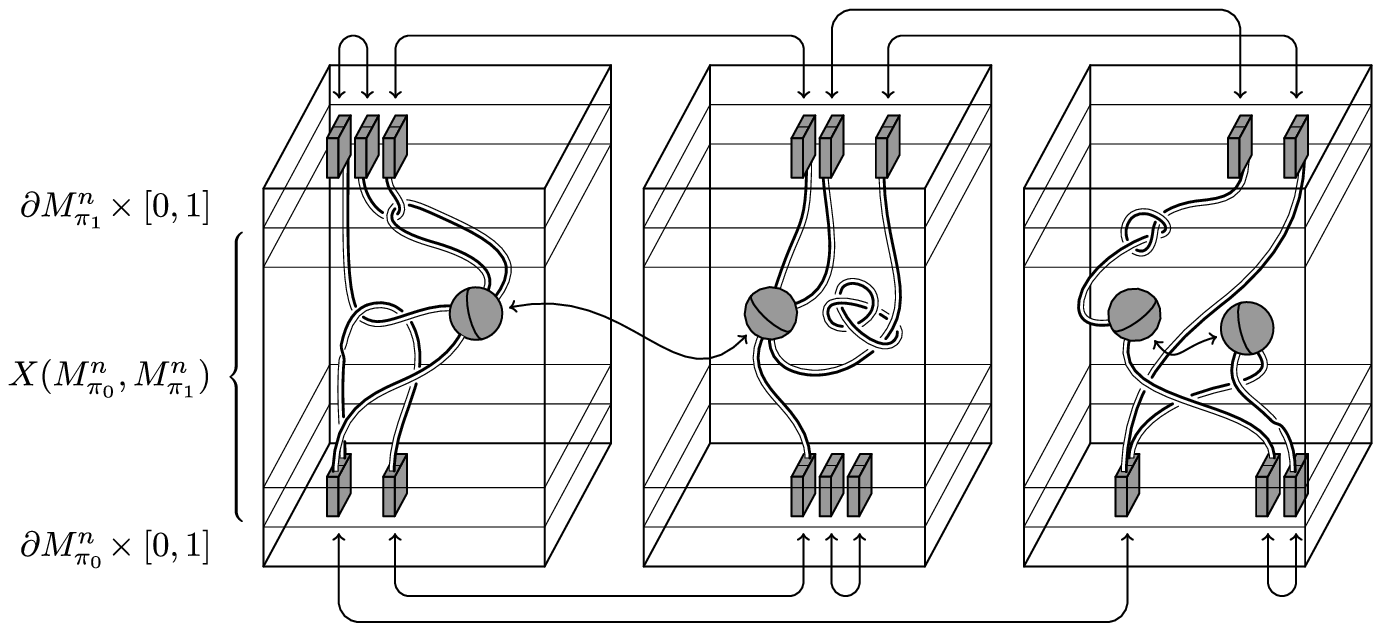}}
\end{Figure}

\begin{proposition} \label{bridged-tangle/thm}
Given $\pi_0,\pi_1 \in \seq\G_n$, a bridged tangle $T = T(\Phi,\Psi)$ from $\pi_0$ to
$\pi_1$ specifies a unique 4-dimensional relative 2-handlebody $W_T = W(\Phi,\Psi)$ build
on $X(M^n_{\pi_0},M^n_{\pi_1})$ without 0-handles. Viceversa, any 4-dimensional relative
2-handle\-body build on $X(M^n_{\pi_0}, M^n_{\pi_1})$ without 0-handles can be specified
in this way by a bridged tangle, up to isotopy of the attaching maps of the 2-handles.
\end{proposition}

\begin{proof}
Given the bridged tangle $T(\Phi,\Psi)$, let $Z_n$ be as in Definition
\ref{bridged-tangle/def} and let $S$ be the space obtained from $Z_n$ by removing the
interiors of all the in-boxes $B'_{\inp,k}$ and $B''_{\inp,k}$ and all the out-boxes
$B'_{\out,k}$ and $B''_{\out,k}$ and identifying their boundaries through the maps
$\rho_{\inp,k}$ and $\rho_{\out,k}$ respectively. Then $S$ can be identified in a
canonical way with $\partial Y(M^n_{\pi_0},M^n_{\pi_1})$ (cf. Figure
\ref{bridged-tang01/fig}). The subspace $S' \subset S$ coming from $\sqcup_{i=1}^n E_i
\times [0.1,0.9]$ can be canonically identified with $X(M^n_{\pi_0},M^n_{\pi_1}) \times
\{1\}$. Then $\Phi$ specifies the attaching maps of $r$ 1-handles $H^1_1, \dots, H^1_r$ to
$X(M^n_{\pi_0}, M^n_{\pi_1}) \times [0,1]$.\break The relative 1-handlebody $W^1$ deriving
from attaching these 1-handles, can be realized as the quotient of $S'$ given by the
identification of the internal balls $B'_k$ and $B''_k$ through the map $\rho_k$ for $k =
1, \dots, r$. Under this identification $\Psi$ represents the attaching maps of $s$
2-handles $H^2_1, \dots, H^2_s$ in the boundary $W^1$ in terms of framed knots, assuming
the attachment longitudinal along each 1-handle (cf. Section \ref{handles/sec}).
Observe that the requirements in point 3 of Definition \ref{bridged-tangle/def} insure
that the framing of the 2-handles is well-defined. Then, $W(\Phi,\Psi)$ is the relative
handlebody $(X(M^n_{\pi_0}, M^n_{\pi_1}) \times [0,1]) \cup_{k=1}^r H^1_k \cup_{h=1}^s
H^2_h$.

At this point, the second part of the proposition just follows from the fact that the
construction above can be reversed starting from any 4-dimensional relative 2-handlebody
build on $X(M^n_{\pi_0},M^n_{\pi_1})$ without 0-handles, once the attaching maps of the
2-handles have been isotoped to be parallel to the core along each 1-handle.
\end{proof}

Now we want to interpret 2-equivalence of 4-dimensional relative 2-handlebodies in terms 
of bridged tangles. To this aim, let us consider the following operations on a bridged 
tangle $T(\Phi,\Psi)$.
\begin{itemize}
\item[\(a)]\label{bt-moves/def}
{\sl Isotopy} of $\Phi$ and $\Psi$, which preserves the intersections of $\Psi(Q)$ with
the in- and out-boxes and with the internal balls, as well as the conditions in point 3 of
Definition \ref{bridged-tangle/def}.
\item[\(b)]
{\sl Pushing through an internal pair of 3-balls} (cf. Figure \ref{bridged-tang02/fig}).
Let $B'_k$ and $B''_k$ be a pair of internal balls and put $\Phi'_k = \Phi_{|P'_k}$ and
$\Phi''_k = \Phi_{|P''_k}$. Assume that we are given: 1)~a 3-ball $B \subset
\sqcup_{i=1}^n \Int E_i \times \left]0.1,0.9\right[$, such that $B'_k \subset \Int B$,
while $\Bd B$ is disjoint from all the internal balls and meets transversally $\Psi(Q)$
along the image of some meridian arcs like those in point 3 of Definition
\ref{bridged-tangle/def}; 2)~a 3-ball $C \subset \Int B''_k$ and an orientation preserving
diffeomorphism $\eta: \Cl(B - B'_k) \to \Cl(B''_k - C)$ such that $\eta_{|\Bd B'_k} =
\rho_{k|\Bd B'_k}$. Then, we modify $T(\Phi,\Psi)$ as follows: replace $\Phi'_k$ and
$\Phi''_k$ with $\bPhi'_k: P'_k \to \bar B'_k = B$ and $\bPhi''_k: P''_k \to \bar B''_k =
C$ respectively, such that $\eta ( \bPhi'_k(x,y,z)) = \bPhi''_k(x,y,-z)$ for any $(x,y,z)
\in \Bd P'_k$; delete the part of the diagram in $\Cl(B - B'_k)$ and insert its image
through $\eta$ in $\Cl(B''_k - C)$; perform all the consequent modifications on the space
$Q$ and the map $\Psi$. An example is depicted in Figure \ref{bridged-tang02/fig} (the two
balls $B'_k$ and $B''_k$ can lie in two different connected components of $Z_n$).
\begin{Figure}[htb]{bridged-tang02/fig}
{}{Pushing through an internal pair of 3-balls}
\centerline{\fig{}{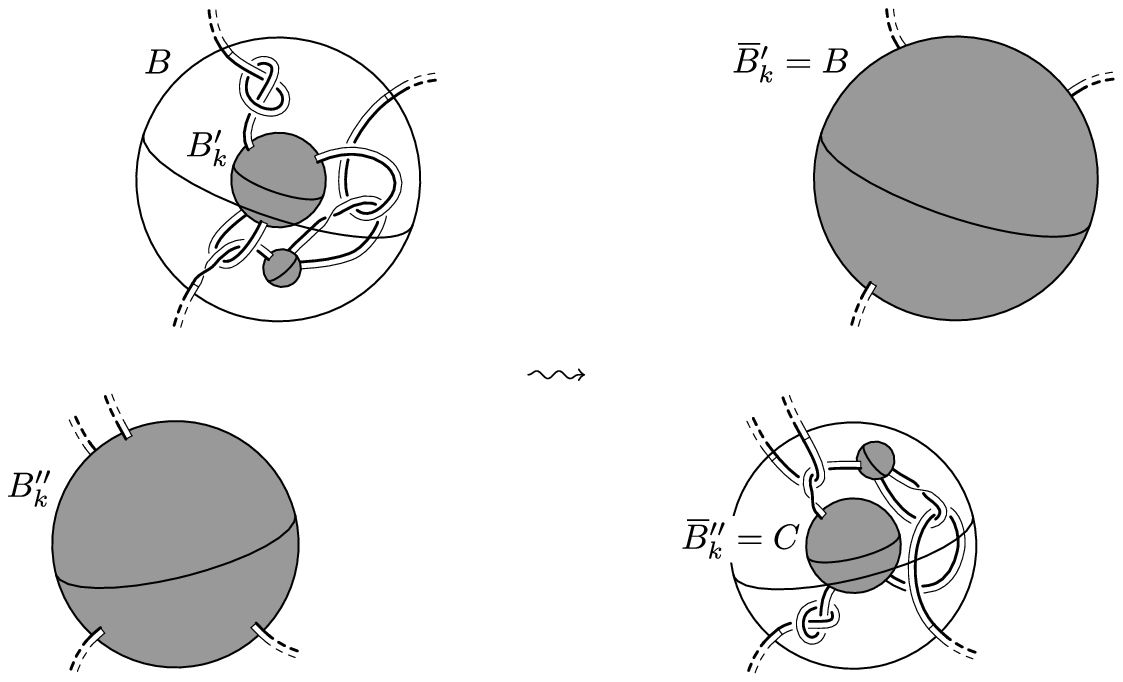}}
\end{Figure}
\item[\(c)]
{\sl Pushing through an in/out pair of boxes}. This move is analogous to \(b), but instead
of 3-balls one works with boxes and, as an additional final step, one has to rescale the
new boxes to make them as in point 1 of Definition \ref{bridged-tangle/def}.
\item[\(d)] 
{\sl 2-handle sliding}. Given two different annuli $A_i$ and $A_j$ with $i,j = 1, \dots,
s$ as in point 3 of Definition \ref{bridged-tangle/def}, let $Q_i$ and $Q_j$ the
corresponding subspaces of $Q$. Then the move consists in taking a parallel copy
$\Psi(Q_j)^\parallel$ of $\Psi(Q_j)$ in the diagram and replacing $\Psi(Q_i)$ by the band
connected sum of it with $\Psi(Q_j)^\parallel$, made through band $\beta$ connecting any
two components of $\Psi(Q_i)$ and $\Psi(Q_j)^\parallel$ contained in the same component of
$Z_n$ (cf. Figure \ref{bridged-tang03/fig}).
\begin{Figure}[htb]{bridged-tang03/fig}
{}{A 2-handle sliding}
\vskip3pt
\centerline{\fig{}{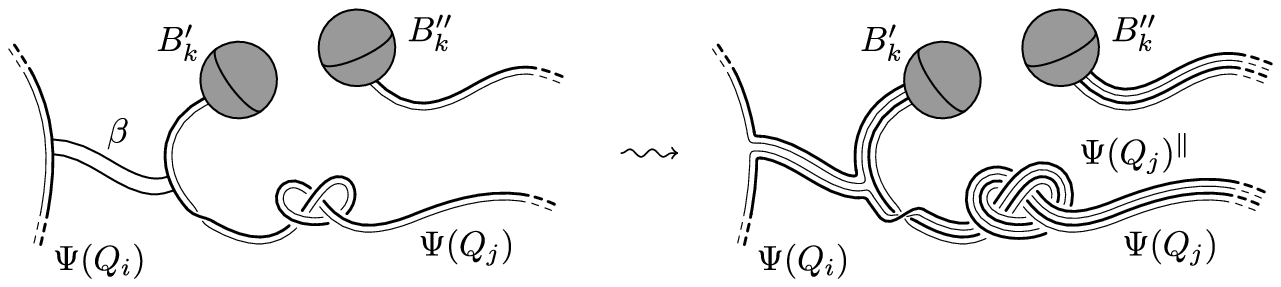}}
\end{Figure}
\item[\(e)]
{\sl Adding/deleting a canceling 1/2-pair}, that is two internal balls $B'_k$ and $B''_k$
which are joined by a single band and do not intersect elsewhere $\Psi(Q)$ (cf. Figure
\ref{bridged-tang04/fig}).
\begin{Figure}[htb]{bridged-tang04/fig}
{}{A canceling 1/2-pair}
\centerline{\fig{}{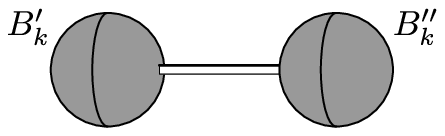}}
\end{Figure}
\end{itemize}

\begin{definition}\label{bt-equivalence/def}
Two bridged tangles are said to be {\sl 2-equivalent} if they are related by a finite 
sequence of moves of types \(a) to \(e) and their inverses.
\end{definition}

\begin{proposition}\label{bt-equivalence/thm}
Given $\pi_0,\pi_1 \in \seq\G_n$, two bridged tangles $T(\Phi_1,\Psi_1)$ and
$T(\Phi_2,\Psi_2)$ from $\pi_0$ to $\pi_1$ are 2-equivalent in the sense of the previous
definition if and only if the corresponding 4-dimensional relative 2-handlebodies
$W(\Phi_1,\Psi_1)$ and $W(\Phi_2,\Psi_2)$ are 2-equivalent as relative handlebodies.
\end{proposition}

\begin{proof}
First of all, we observe that the moves \(a) to \(e) on a bridged tangle $T(\Phi,\Psi)$
represent 2-deformations of the handlebody $W(\Phi,\Psi)$. In particular, \(b) represents
isotopy of the attaching maps of the 2-handles and 1-handles slidings over the $k$-th
1-handle of $W(\Phi,\Psi)$.

Viceversa, assume that $W(\Phi_1,\Psi_1)$ and $W(\Phi_2,\Psi_2)$ are 2-equivalent. Since
they are relative handlebodies without 0-handles, there is a 2-deformation relating them
that does not involve addition/deletion of canceling pairs of 0- and 1-handles (see
Proposition \ref{0-handles/thm}). Such a 2-deformation consists of a finite sequence of
the following modifications: isotopies of the attaching maps of the 1-handles and of the
2-handles; addition/deletion of pairs of canceling 1- and 2-handles; sliding of a 1- or
2-handle over another 1- or 2-handle. Isotopy and 1-handle sliding on a handlebody
$W(\Phi,\Psi)$ can be represented by moves \(a), \(b) and \(c) on the bridged tangle
$T(\Phi,\Psi)$, while the other two modifications just correspond to moves \(d) and \(e).
\end{proof}

In the light of Propositions \ref{bridged-tangle/thm} and \ref{bt-equivalence/thm}, for
any $n \geq 1$ we can form a strict monoidal category of bridged tangles $\T_n$ equivalent
to $\Hb_n^{3+1}$ through the functor induced by the map $T(\Phi,\Psi) \mapsto 
W(\Phi,\Psi)$.

\pagebreak

The objects of $\T_n$ are the sequences in $\seq\G_n$. Given two sequences $\pi_0,\pi_1
\in \seq\G_n$, a morphism in $\T_n$ with source $\pi_0$ and target $\pi_1$ is a bridged
tangle from $\pi_0$ to $\pi_1$, considered up to 2-equivalence of bridged tangles.

If $T_1 = T(\Phi_1,\Psi_1): \pi_0 \to \pi_1$ and $T_2 = T(\Phi_2,\Psi_2): \pi_1 \to \pi_2$
are two morphisms in $\T_n$, then their composition is the morphism $T = T(\Phi,\Psi):
\pi_0 \to \pi_2$ defined as follows. Translate $T_2$ in the space $\bar Z_n =
\sqcup_{i=1}^n E_i \times [1,2]$, glue it to $T_1$ by identifying the two copies of
$\sqcup_{i=1}^n E_i \times \{1\}$ in $Z_n$ and $\bar Z_n$, and rescale the third
coordinate by the factor $1/2$. In the identification the out-boxes of $T_1$ are glued to
the corresponding in-boxes of $T_2$ to give a new boxes that, up to smoothing the corners,
can be considered as extra internal balls. Then, under the above identification, $T$ is
determined by $\Phi = \Phi_1 \sqcup \Phi_2$ extended to include those extra internal balls
and by $\Psi = \Psi_1 \sqcup \Psi_2$.

The identity morphism $\id_\pi$ of a sequence $\pi \in \seq\G_n$ is represented by the
bridged tangle from $\pi$ to $\pi$ without any internal ball and with a single band
connecting any in-box with the corresponding out-box. Notice that the two bands connecting
a pair $B'_k$ and $B''_k$ of in-boxes with the corresponding pair of out-boxes, give the
attaching map of the 2-handle of $\id_{M^n_\pi}\!$ between the two copies (in the source a
in the target) of the $k$-th 1-handle of $M^n_\pi$.

\label{T-monoidal/def}
Finally, $\T_n$ has a strict monoidal structure whose product (which we will denote again
by $\diam\,$) is given by juxtaposition on the objects, while on the morphisms $T \diam
T'$ is given by translating $T'$ in the space $Z_n' = \sqcup_{i=1}^n [1,2] \times [0,1]
\times [0,1]$, glueing the two tangles by identifying the corresponding copies of $\{1\}
\times [0,1] \times [0,1]$ in $Z_n$ and $Z_n'$ and then applying a reparametrization of
the first coordinate depending on the last one, that is a diffeomorphism $(x,y,z) \mapsto
(h^z(x),y,z)$ with $h^z$ increasing function of $x$ for every $z$. The unit of the product
between morphisms is the empty tangle.

\medskip

Describing the cobordisms in $\Chb_n^{3+1}$ through bridged tangles presents two main
difficulties. The first one is that the space $Z_n$ in which lives the tangle diagram is
not connected when $n > 1$. The second and most important one is that from a diagrammatic
point of view the ``pushing through 1-handle'' move is highly non-local. Indeed, part of
the diagram which is in the neighborhood of one 3-ball appears in the a neighborhood of
another 3-ball; moreover, this last ball can even be in a different connected component of
the diagram.

Introducing $n$-labeled Kirby tangles resolves both these problems. Basically the idea
is to allow only the use of bridged tangle diagrams in special form. Obviously, this
requires a specification of the 2-equivalence moves relating two such special diagrams.

\begin{definition}\label{special-bt/def}
A bridged tangle $T(\Phi,\Psi)$ will be called a {\sl special bridged tangle} if, using 
the same notation as in Definition \ref{bridged-tangle/def}, the following properties are 
satisfied:
\begin{itemize}
\item[1)]
the map $\pr: Z_n = \sqcup_{i=1}^n E_i \times [0,1] \to E \times [0,1]$ projecting each 
copy $E_i \times [0,1]$ onto $E \times [0,1]$ by the identity is injective on $\Psi(P) 
\cup \Phi(Q)$;
\item[2)]
$\Psi(Q)$ meets any internal ball $B'_k$ (resp. $B''_k$) only in the disks $D'_k =
\Phi(S^2_+(0,0,2))\break \subset \Bd B'_k$ (resp. $D''_k = \Phi(S^2_-(0,0,-2)) \subset \Bd
B''_k$) image of the copy in $P_k$ of the upper half-sphere $S^2_+(0,0,2) \subset \Bd
B^3(0,0,2)$ (resp. the lower half-sphere $S^2_-(0,0,-2) \subset \Bd B^3(0,0,-2)$);
moreover, $\Psi(Q)$ meets each in- or out-box $B'_k$ (resp. $B''_k$) only in points of the
half-space $y > 0$ (resp. $y < 0$);
\item[3)] 
an embedding $\bPhi: C \to E \times \left]0.1,0.9\right[$ is given, with $C =
\sqcup_{k=1}^r C_k$ and $C_k$ the convex hull of $P_k$, such that $\bPhi_{|P} = \pr \circ
\Phi$ and $\bPhi(C) \cap \pr(\Psi(Q)) = \pr(\Phi(P) \cap \Psi(Q))$; we denote by $D_k =
\bPhi(B^2)$ the image of the copy in $C_k$ of the standard disk $B^2 \subset R^2 \subset
R^3$ (cf. Figure \ref{bridged-tang05/fig}).
\end{itemize}
\end{definition}

Representing a special bridged tangle diagram is much simpler then a general one. Indeed,
thanks to property 1, instead of using $n$ copies of $E \times [0,1]$, one for each
connected component of the diagram, we can draw the diagram directly in $E \times [0,1]$.
Namely, we consider $\pr \circ (\Phi \sqcup \Psi) \subset E \times [0,1]$ and label
each part of the diagram with a number from 1 to $n$, to keep track of the original
component where it lives.

Moreover, by properties 2 (first part) and 3, we can draw only the disk $D_k$ in place of
each pair of internal 3-balls $\pr(B'_k)$ and $\pr(B''_k)$, extending the ribbons inside
$\bPhi(C_k)$ by the image of vertical bands under the identification $\bPhi: C_k \to
\bPhi(C_k)$, until they intersect it as it is shown on the right side of Figure
\ref{bridged-tang05/fig}.\break As usual, in the diagrams we mark the unknot $\Bd D_k$ by
a dot to indicate that it stands for the attaching map of a 1-handle.

\begin{Figure}[htb]{bridged-tang05/fig}
{}{Dot notation for a pair of internal balls in a special bridged tangle}
\vskip-3pt
\centerline{\hskip-20pt\fig{}{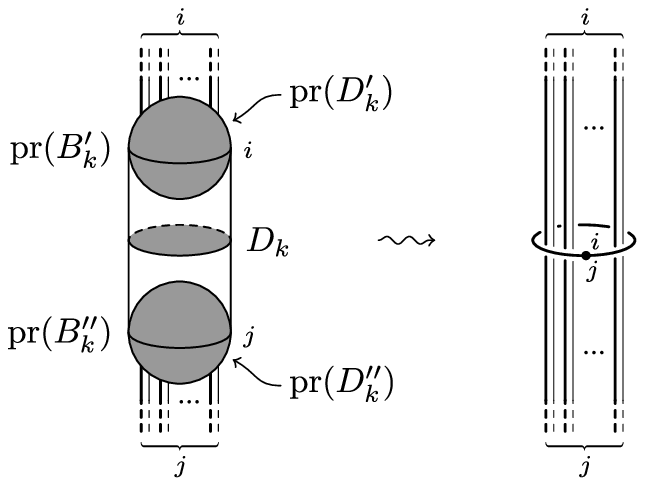}}
\vskip-6pt
\end{Figure}

\begin{Figure}[b]{bridged-tang06/fig}
{}{Dot notation for pairs of in/out-boxes in a special bridged tangle}
\vskip-3pt
\centerline{\fig{}{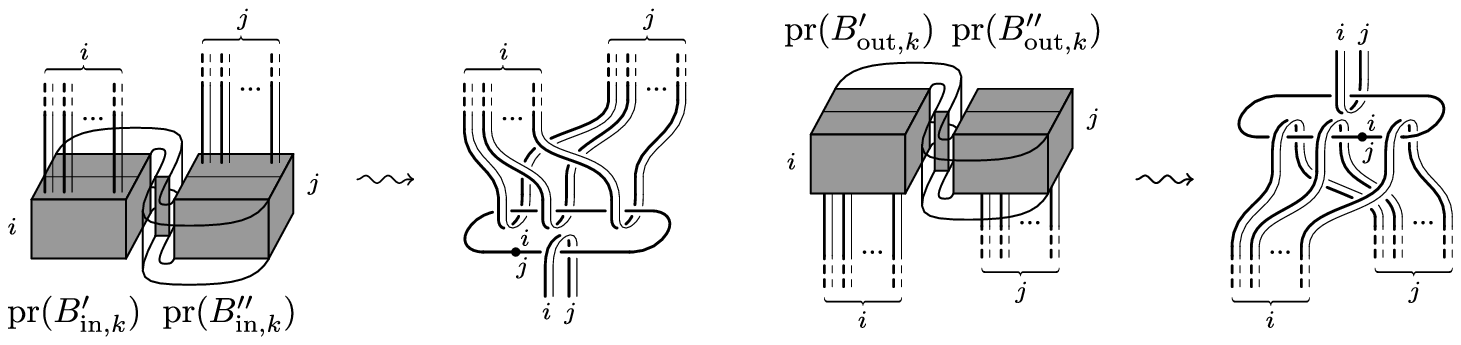}}
\vskip-6pt
\end{Figure}

Taking into account the last part of property 2, we can do a similar thing for the $k$-th
pair of in-boxes (resp. out-boxes), but adding an extra open framed component. Such extra
component passes once through the resulting dotted component and ends in a canonical pair
of intervals in $E \times \{0\}$ (resp. $E \times \{1\}$), which represent the $k$-th
element of the source (resp. target) of the tangle (cf. definition below), as shown in
Figure \ref{bridged-tang06/fig}. The 3-cell $\bPhi(C_k)$ and the disk $D_k$ in the above
construction, are respectively replaced by the pair of boxes joined by the bent
rectangular tube drawn in the figure and by the meridian rectangle of it. The
corresponding dotted unknot and the involved framed components has been isotoped in a
different standard position to make the diagram simpler. After such isotopy, the whole
configuration is assumed to have the standard form depicted in the figure.

In order to introduce the notion of $n$-labeled Kirby tangle, we need a preliminary
definition. Given any finite sequence $\pi = ((i_1,j_1), \dots, (i_m,j_m)) \in \seq\G_n$,
let $I_\pi = ((a'_{m,1},a''_{m,1}), \dots, (a'_{m,m},a''_{m,m}))$ the sequence of pairs of
intervals in $E = [0,1]^2$, each labeled with the corresponding element of $\pi$, defined
for $1 \leq k \leq m$ by
$$\begin{array}{c}
a'_{m,k} = [(2k - 1.6)/2m, (2k - 1.4)/2m] \times \{0.5\}\,,\\[6pt]
a''_{m,k} = [(2k - 0.6)/2m, (2k - 0.4)/2m] \times \{0.5\}\,.
\end{array}$$

\smallskip

\begin{definition}\label{kirby-tangle/def}
Let $n \geq 1$ and $\pi_0 = ((i^0_1,j^0_1), \dots, (i^0_{m_0},j^0_{m_0}))$ and $\pi_1 =
((i^1_1,j^1_1), \dots, (i^1_{m_1},j^1_{m_1}))$ be two sequences of in $\seq\G_n$. Then an
{\sl $n$-labeled (admissible) Kirby tangle} in $E \times [0,1]$ from $I_{\pi_0}$ to
$I_{\pi_1}$ consists of the following data:
\begin{itemize}
\item[1)]
$r$ dotted unknots spanning disjoint flat disks $D_1, \dots, D_r \subset \Int E \times 
\left]0,1\right[$;
\item[2)]
a tangle consisting of $s$ framed curves $C_1, \dots, C_s$ (cf. Section
\ref{handles/sec}) regularly embedded in $\Int E \times \left[0,1\right]$ and transversal
with respect to those disks, such that each open framed curve $C_h$ joins a pair of
intervals $(a'_{m_0,k} \times \{0\},a''_{m_0,k} \times \{0\})$ for some
$(a'_{m_0,k},a''_{m_0,k}) \in I_{\pi_0}$ or $(a'_{m_1,k} \times \{1\},a''_{m_1,k} \times
\{1\})$ for some $(a'_{m_1,k},a''_{m_1,k}) \in I_{\pi_1}$, with the base curve always
ending in the left end-points of the intervals;
\item[3)]
a labeling from $\{1, \dots, n\}$ for each side of the disks $D_1, \dots, D_r$ and each
component of the tangle once it has been cut at the intersection with these disks; the
labeling must be consistent in the sense that all the framed arcs coming out from one side
of a disks (or ending at an interval from $I_{\pi_0}$ or $I_{\pi_1}$) have the same label
of that side (or that interval) (cf. Figures \ref{bridged-tang05/fig} and
\ref{bridged-tang06/fig}).
\end{itemize}
\end{definition}

The term ``admissible'' in the denomination of Kirby tangles refers to the condi\-tion in
point 2 of the definition, that no framed curve $C_k$ joins an interval of $I_{\pi_0}\!$
at level 0 with one of $I_{\pi_1}\!$ at level 1 (cf. \cite{MaP92, KL01}). However, since
we will always work with\break admissible tangles, we will simply write Kirby tangle to
mean an admissible one.

Moreover, the consistency rule in point 3 of the definition makes the labeling redundant
and sometimes we will omit the superfluous labels. Observe also that in the case $n = 1$,
all labels in the diagram have value 1, so they will be omitted, and we are reduced to an
ordinary Kirby tangle.

\medskip

What we have said after Definition \ref{special-bt/def} can be restated by saying that to
any special bridged tangle $T$ we can associate a uniquely determined Kirby tangle $K_T$
representing it. This is obtained from $T$ by replacing internal balls and in/out-boxes as
described in Figures \ref{bridged-tang05/fig} and \ref{bridged-tang06/fig}. Notice that
the replacement of internal balls depends on the extra structure given by the embedding
$\bPhi$ in point 3 of Definition \ref{special-bt/def}, hence the construction of $K_T$
cannot be immediately applied to a (non-special) bridged tangle $T$.

Viceversa, given an $n$-labeled Kirby tangle $K$, we can construct a corresponding special
bridged tangle, which we will denote by $T_K$, in the following way: we first convert the
dot notation for 1-handles into the ball notation, by reversing the step in Figure
\ref{bridged-tang05/fig}; after that we take the disjoint union $Z_n = \sqcup_{i=1}^n E_i
\times [0,1]$ of $n$ copies of $E \times [0,1]$ and put in the $i$-th component the
portion of the diagram labeled with $i$; eventually, we transform the intervals of
$I_{\pi_0}$ at level 0 into in-boxes and those of $I_{\pi_1}$ at level 1 into out-boxes.

We observe that the maps $T \mapsto K_T$ and $K \mapsto T_K$ are not exactly the inverse
of each other. In particular, $T_{K_T}\!$ does not coincide with $T$, but we will see in
Proposition \ref{kirby-tangle/thm} that they are 2-equivalent.

\medskip

In order to extend the definition of $K_T$, modulo certain moves, to any bridged tangle
$T$ (see proof Proposition \ref{kirby-tangle/thm}) and to interpret the 2-equivalence of
bridged tangles in terms of Kirby tangles, we need the moves depicted in Figure
\ref{kirby-tang01/fig}.

\begin{Figure}[htb]{kirby-tang01/fig}
{}{Equivalence moves for $n$-labeled Kirby tangles}
\vskip-3pt
\centerline{\fig{}{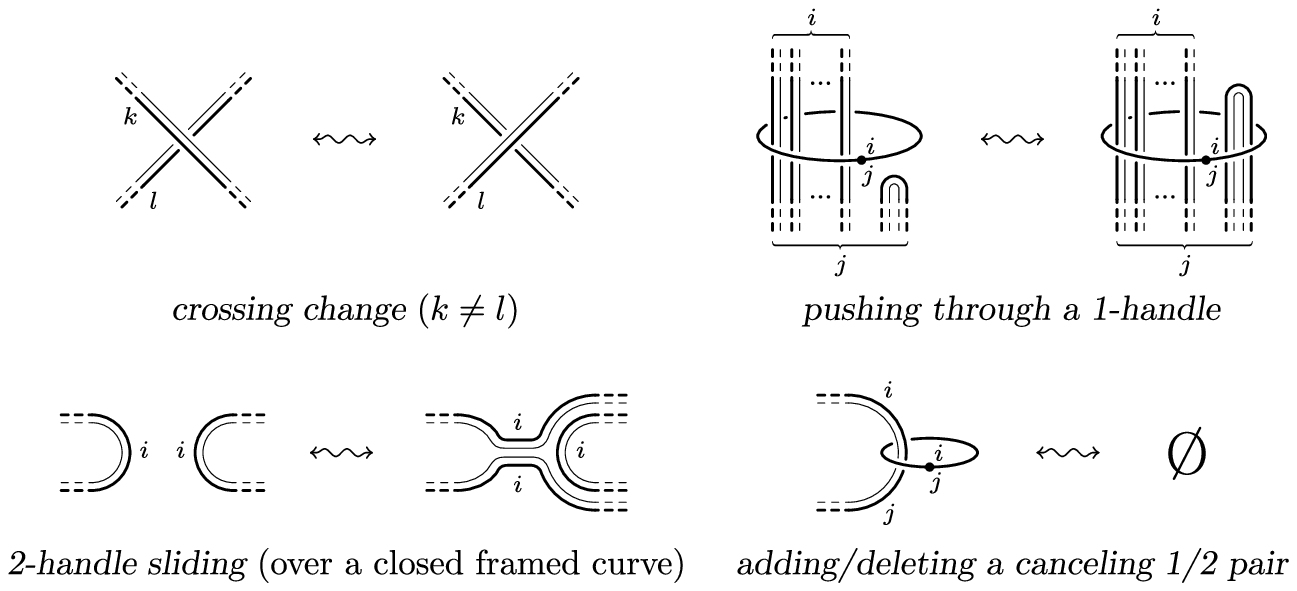}}
\vskip-3pt
\end{Figure}

We want to emphasize some facts: 1)~crossing change and pushing through a 1-handle are
local moves, while 2-handle sliding and adding/deleting a canceling 1/2-pair are global
moves; 2)~in the 2-handle sliding, we can slide any (possibly open) framed curve, but the
framed curve over which the sliding takes place has to be a closed one; 3)~for ordinary
Kirby tangles, that is for $n=1$, the crossing change cannot be realized and the other
moves are considered only in the usual ordinary case, that is for $i = j = 1$.

\begin{definition}\label{kt-equivalence/def}
Two $n$-labeled Kirby tangles are said to be {\sl 2-equivalent} if they are related by
labeled isotopy (preserving the intersections between framed curves and disks) and the
moves of Figure \ref{kirby-tang01/fig}.
\end{definition}

\begin{Figure}[b]{kirby-tang02/fig}
{}{1-handle moves for $n$-labeled Kirby tangles}
\vskip-3pt
\centerline{\kern6pt\fig{}{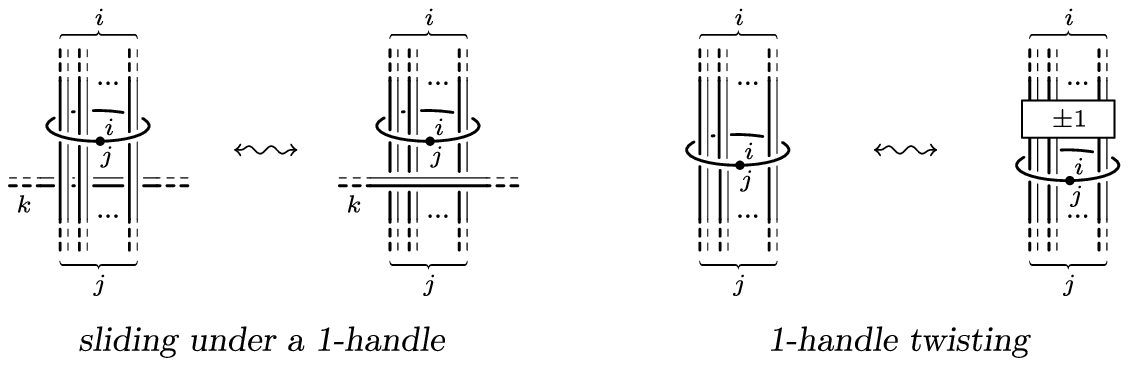}}
\vskip-3pt
\end{Figure}

Before going on, we introduce two auxiliary 2-equivalence moves on Kirby tangles, which
will be needed in the proof of the next proposition. These are the 1-handle moves in
Figure \ref{kirby-tang02/fig}. They can be derived from the 2-equivalence moves in
Figure \ref{kirby-tang01/fig}, as shown in Figure \ref{kirby-tang03/fig}.

\begin{Figure}[htb]{kirby-tang03/fig}
{}{Deriving the 1-handle moves}
\centerline{\fig{}{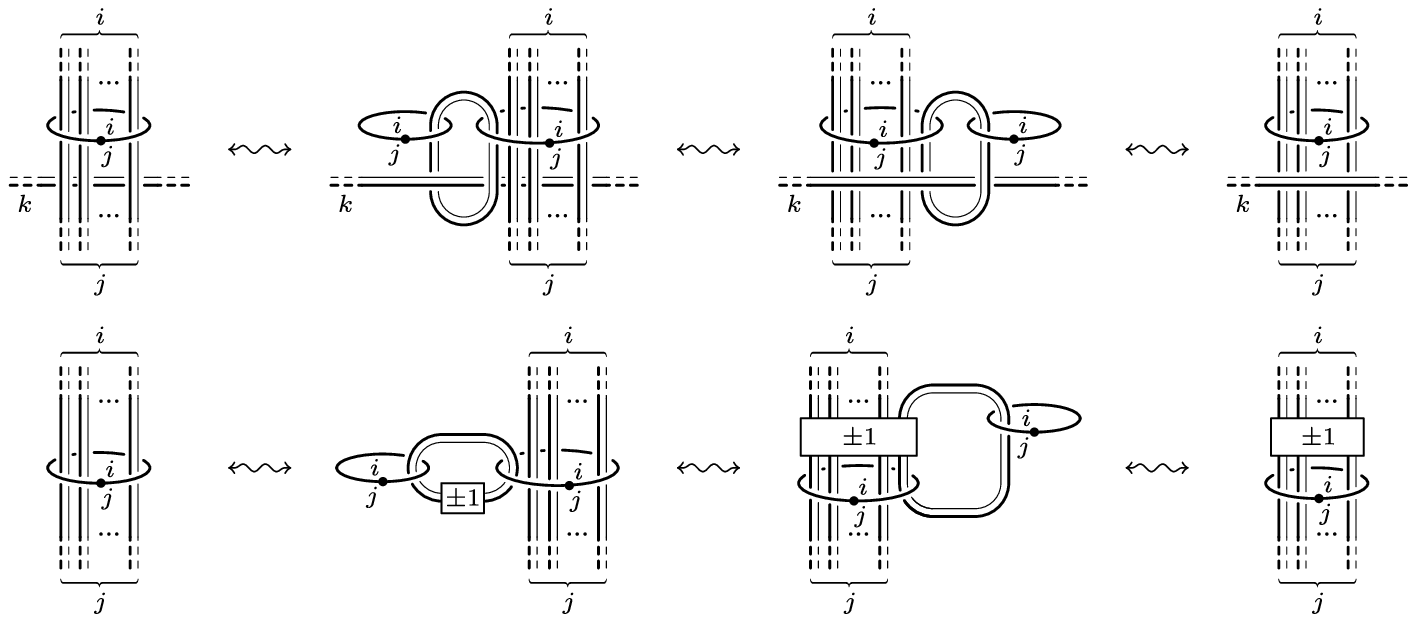}}
\vskip-3pt
\end{Figure}

\begin{proposition}\label{kirby-tangle/thm}
Any orientable 4-dimensional relative 2-handlebody $W$ build on $X(M^n_{\pi_0},
M^n_{\pi_1})$, with $\pi_0,\pi_1 \in \seq\G_n$, can be represented by an $n$-labeled Kirby
tangle $K$ such that $W = W_{T_K}$. Moreover, given two $n$-labeled Kirby tangles $K_1$
and $K_2$, the 4-dimensional relative 2-handlebodies $W_1$ and $W_2$ that they represent
are 2-equivalent if and only if $K_1$ and $K_2$ are 2-equivalent in the sense of the
previous definition.
\end{proposition}

\begin{proof}
By Proposition \ref{bt-equivalence/thm} it suffices to show that the map $K \mapsto T_{K}$
from labeled Kirby tangles to (special) bridged tangles induces a bijection between the
equivalence classes of Kirby tangles under isotopy and the moves in Figures
\ref{kirby-tang01/fig} and the 2-equiva\-lence classes of bridged tangles.

First, we observe that the map $K \mapsto T_K$ induces a well-defined map at the level of
such equivalence classes, that is changing $K$ through isotopy and the moves in Figures
\ref{kirby-tang01/fig} produces a 2-equivalent bridged tangle. Indeed, any isotopy of $K$,
as well as any crossing change, clearly induces a suitable isotopy of the maps $\Phi$ and
$\Psi$ determining $T_K$, while the other three moves in Figure \ref{kirby-tang01/fig}
just induces homonymous moves of bridged tangles (cf. definition starting at page
\pageref{bt-moves/def}).

Now, we want to define the inverse map which associates to the 2-equivalence class of
bridged tangle $T$ the 2-equivalence class of a labeled Kirby tangle $K_T$, based on
the construction $T \mapsto K_T$ considered above for $T$ a special bridged tangle.

Any bridged tangle $T = T(\Phi,\Psi)$ can be made into a special one $T'$ by an isotopy of
bridged tangles (that is an isotopy of the maps $\Phi$ and $\Psi$, as specified in point
\(a) at page \pageref{bt-moves/def}) to achieve properties 1 and 2 in Definition
\ref{special-bt/def}, followed by the extension of $\Phi$ to a map $\bPhi$ as in point 3
of the same definition.

The resulting special bridged tangle $T'$ is not unique, since it depends on the choice of
the isotopy and of the extension $\bPhi$. Nevertheless, we are going to show that the
2-equivalence class of $K_{T'}$ is uniquely determined, depending only on the original
bridged tangle $T$.

Once the isotopy is fixed, different choices of $\bPhi$ leads to labeled Kirby tangles
which can be related by labeled tangle isotopy (preserving the intersections between
framed curves and disks) and the two moves in Figure \ref{kirby-tang02/fig}.
In fact, it is clear that isotopy of $\bPhi$ which preserves property 3 of Definition
\pagebreak
\ref{special-bt/def} induces an isotopy of the corresponding Kirby tangle. Up to such an
isotopy, $\bPhi$ is determined by the set of arcs $\bPhi(\sqcup_{k=1}^r \gamma_k)
\subset E \times \left]0.1, 0.9\right[ - \pr(\cup_{k=1}^r \Int(B'_k \cup B''_k))$ with
$\gamma_k = \{(0,0)\} \times [0,1] \subset C_k$ and by the framings along them. Of course,
different choices for these arcs are always isotopic keeping their end-points fixed, but
during the isotopy they could cross the framed curves of $T$, and each time this happens
$K_T$ changes by a sliding under a 1-handle. While adding a full twist to the framing
along any arc induces on $T_K$ a twisting on the corresponding 1-handle.

Concerning the choice of the isotopy, we have that any other special
bridged tangle $T''$ isotopic to $T$ is also isotopic to $T'$. Moreover, we can assume the 
isotopy relating $T'$ and $T''$ to be realized by bridged tangles which satisfy
properties 1 and 2 of Definition \ref{special-bt/def} at every time, except for a finite
number of crossing changes between two framed curves. It follows that the labeled Kirby
diagrams $K_{T'}$ and $K_{T''}$ are related by labeled isotopy and crossing changes as in
Figure \ref{kirby-tang01/fig} (remember that the conditions in point 3 of Definition
\ref{bridged-tangle/def} has to be preserved during the isotopy and this allows us to
trivially extend it inside the balls $\bPhi(C_k)$ when passing to Kirby tangles).

Then, we can define $K_T$ up to 2-equivalence of Kirby tangles for any (possibly
non-special) bridged tangle $T$, just by putting $K_T = K_{T'}$ for some special bridged
tangle $T'$ isotopic to $T$.

At this point, we have to show that if $T_1$ and $T_2$ are 2-equivalent bridged tangles,
then $K_{T_1}$ and $K_{T_2}$ are equivalent through labeled isotopy and the moves in
Figures \ref{kirby-tang01/fig}). If $T_1$ and $T_2$ are isotopic bridged tangles, then
$K_{T_1}$ and $K_{T_2}$ are equivalent by argument above.

Concerning the other operations on bridged tangles, we have only to address the pushing
through 1-handle operations \(b) and \(c), since 2-handle sliding and adding/deleting a
canceling 1/2-pair are explicitly represented in terms of Kirby diagrams in Figure
\ref{kirby-tang01/fig}. Moreover, if a pushing through a 1-handle move involves only 
pieces of framed tangle and no 3-ball, then it can be realized through labeled isotopy and 
the top-right move in Figure \ref{kirby-tang01/fig}.

\begin{Figure}[b]{kirby-tang04/fig}
{}{Pushing a 3-ball through a 1-handle}
\centerline{\fig{}{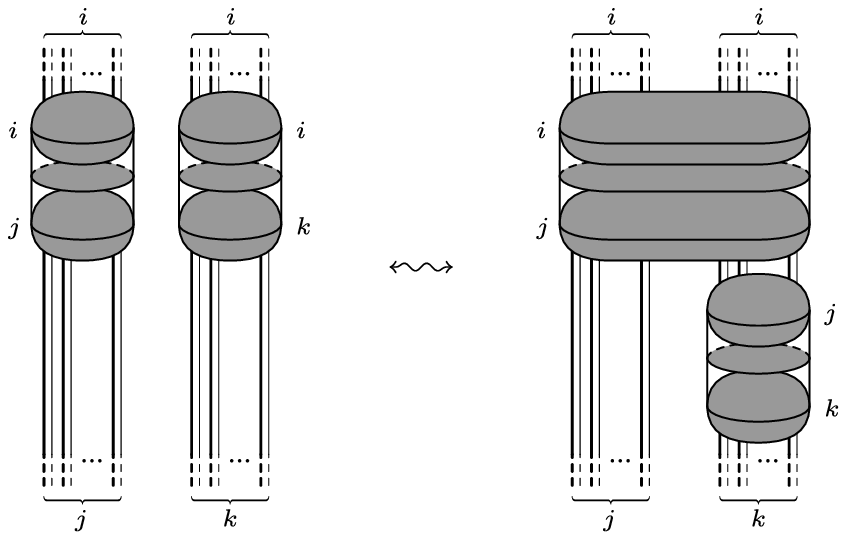}}
\vskip-3pt
\end{Figure}

So, we only need to discuss the case when a 3-ball is pushed through a 1-handle (cf.
Figure \ref{kirby-tang04/fig} for an internal pair of balls). The proof that in this case
the corresponding Kirby tangle changes through adding/deleting canceling 1/2-pairs and
2-handle slides is presented in Figure \ref{kirby-tang05/fig}.

\begin{Figure}[htb]{kirby-tang05/fig}
{}{Deriving the 1-handle sliding}
\centerline{\fig{}{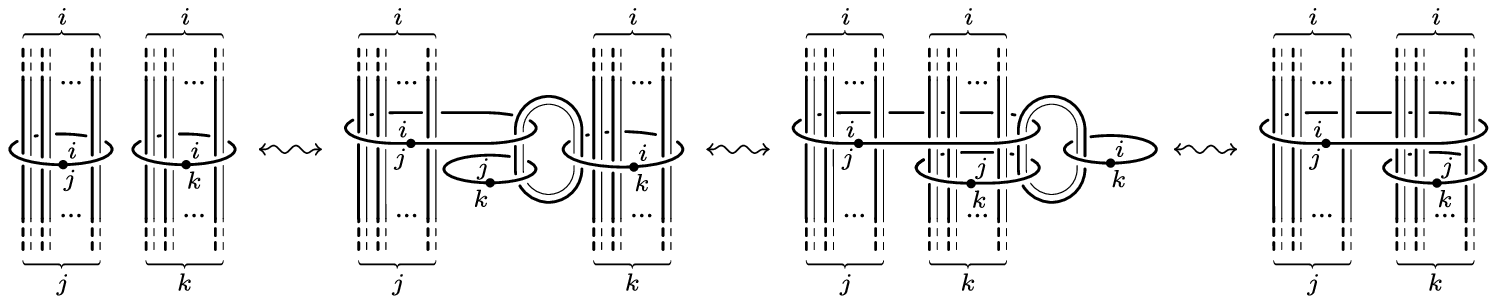}}
\vskip-3pt
\end{Figure}

To conclude the proof, we observe that the two maps $K \mapsto T_K$ and $T \mapsto K_T$
are inverses to each other on 2-equivalence classes. In other words, $K_{T_K}$ is
2-equivalent to $K$ as a Kirby tangle and $T_{K_T}$ is 2-equivalent to $T$ as a bridged
tangle. In both the cases, the 2-equivalence is realized by performing the obvious
sequence of 2-handle slidings, pushing through a 1-handle and then deletion of a canceling
1/2-pair, in order to remove the extra pair of 1- and 2-handles that arises from each pair
of in- and out-boxes as shown in Figure \ref{bridged-tang06/fig}.
\end{proof}

In the next chapters, labeled Kirby tangles will be always represented through their 
planar diagrams, so we conclude this section by discussing such representation in some 
more details.

A labeled Kirby tangle lives in $\Int E \times [0,1]$ and a {\sl planar diagram} of it is
always realized by the projection into the square $\left]0,1\right[ \times [0,1]$
forgetting the second coordinate (in such a way that $E$ projects into
$\left]0,1\right[\,$). As usual, we require that the restriction of the projection to the
tangle, including both framed and dotted curves, is regular and that it is injective
except for a finite number of transversal double points, which give rise to the crossings.

\begin{Figure}[b]{kirby-tang06/fig}
{}{The 1-handle moves for representing isotopy by planar diagrams}
\vskip-3pt
\centerline{\fig{}{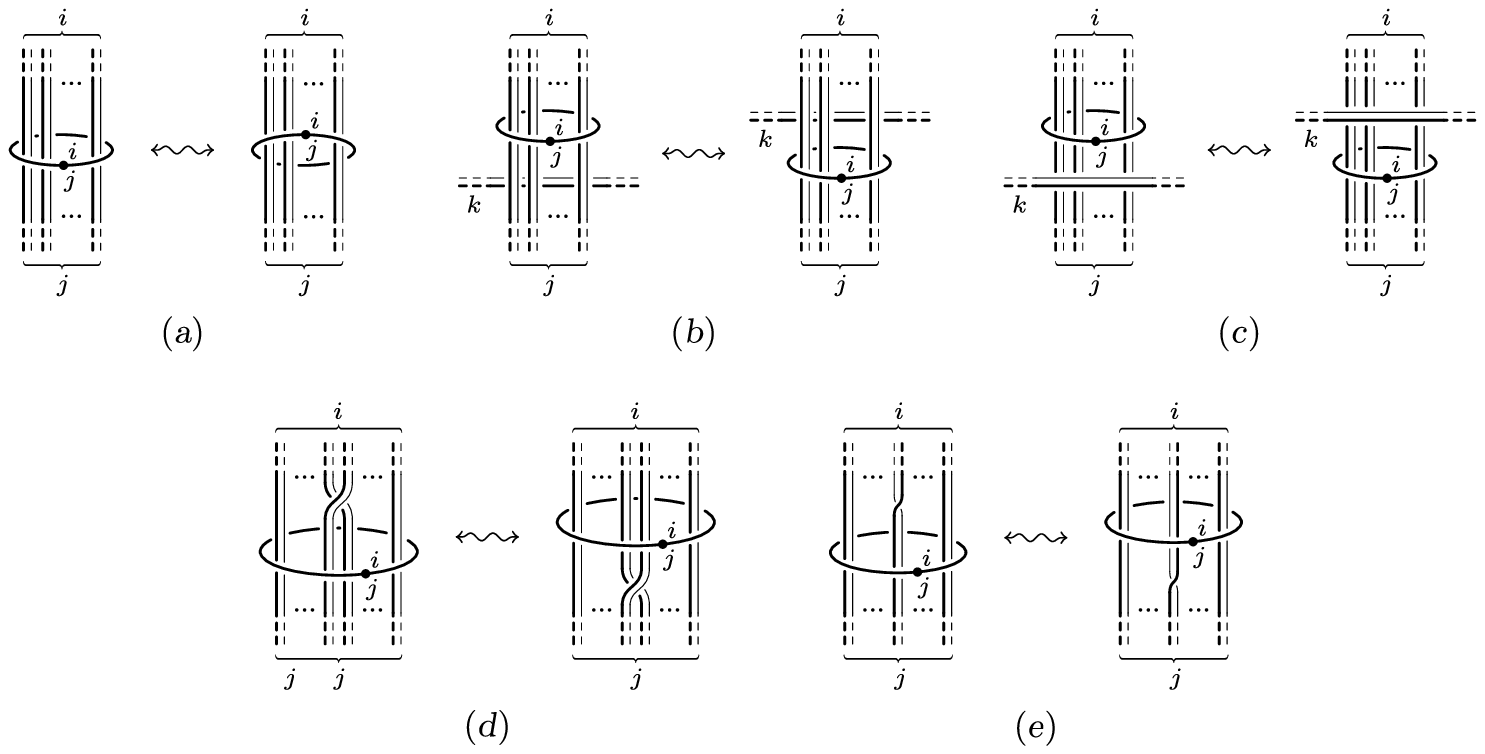}}
\vskip-3pt
\end{Figure}

\begin{definition}\label{strictly-regular/def}
We call {\sl strictly regular} a planar diagram of a labeled Kirby tangle $K$, if the 
disks $D_1, \dots, D_r$ spanned by the dotted unknots projects bijectively onto disjoint 
planar disks and the projection of the framed tangle intersects each of such disks as
presented on the right side of Figure \ref{bridged-tang05/fig} (up to planar isotopy).
\end{definition}

Of course, any labeled Kirby tangle admits a strictly regular planar diagram. The proof of 
the following proposition is quite a standard exercise left to the reader.

\begin{proposition}\label{strictly-regular-isotopy/thm}
Two strictly regular planar diagrams represent isotopic labeled Kirby tangles (where the
isotopy is assumed to preserve the intersections between framed curves and disks) if and
only if they are related by planar isotopy, labeled framed Reidemeister moves and the
moves presented in Figure \ref{kirby-tang06/fig}.
\end{proposition}

All the planar diagrams we have drawn until now are strictly regular, but using strictly
regular diagrams would make pictures quite heavy for Kirby tangles which are not so
simple. In this case, when this does not cause confusion we will draw planar diagrams that
are not strictly regular. However, we will always keep the condition that the disks $D_1,
\dots, D_r$ project bijectively onto disjoint planar disks.

Sometimes it could be convenient to derogate from the labeling consistency rule
for Kirby tangles (cf. point 3 of Definition \ref{kirby-tangle/def}), by allowing a framed
component with label $k$ to cross a disk spanned by a dotted component with labels $i$ and
$j$, provided that $k \notin \{i,j\}$. Clearly, such a crossing does not mean that the
framed loop goes over the 1-handle represented by the dotted one, since it originates
from the identification of different 0-handles. Figure \ref{kirby-tang07/fig} shows the
way to eliminate it.

\begin{Figure}[htb]{kirby-tang07/fig}
{}{Derogating from the labeling consistency rule for $k \notin \{i,j\}$}
\vskip3pt
\centerline{\fig{}{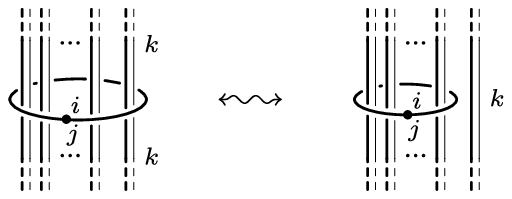}}
\vskip-3pt
\end{Figure}

\subsection{The categories $\K_n$ and the functors $\up_k^n$ and $\down_k^n$%
\label{K/sec}}

For any $n \geq 1$, we define the category $\K_n$ of $n$-labeled Kirby tangles as follows.
The objects of $\K_n$ are the sequences of pairs of labeled intervals $I_\pi$ with $\pi =
((i_1,j_1), \dots, (i_m,j_m)) \in \seq\G_n$ (cf. the notation introduced just before
Definition \ref{kirby-tangle/def}), while the morphisms of $\K_n$ with source $I_{\pi_0}$
and target $I_{\pi_1}$ are the $n$-labeled Kirby tangles from $I_{\pi_0}$ to $I_{\pi_1}$,
considered up to 2-equivalence of Kirby tangles. The composition of two morphisms $K_1:
I_{\pi_0} \to I_{\pi_1}$ and $K_2: I_{\pi_1} \to I_{\pi_2}$ in $\K_n$ is given by
translating $K_2$ on the top of $K_1$, glueing the two tangles along $I_{\pi_1}$ and then
rescaling the third coordinate by the factor $1/2$.

On $\K_n$ we also define a strict monoidal structure, whose product, once again denoted by
$\diam\,$, is given by $I_\pi \diam I_{\pi'} = I_{\pi \diam \pi'}$ on the objects, while
on the morphisms $K \diam K'$ is obtained by translating $K'$ in the space on the right of
$K$ and then applying a reparametrization of the first coordinate depending on the last
one (cf. definition of the monoidal structure on $\T_n$ at page \pageref{T-monoidal/def}).
The unit of the product is the empty tangle $\id_\emptyset: I_\emptyset \to I_\emptyset$.

For $\pi = ((i_1,j_1), \dots, (i_m,j_m)) \in \seq\G_n$, we denote by $\id_\pi: I_\pi \to
I_\pi$ the identity morphism of $I_\pi$. It is easy to see that $\id_{(i,j)}$ is given by
the tangle presented in Figure \ref{kirby-morph01/fig}, and that $\id_\pi =
\id_{(i_1,j_1)} \diam \dots \diam \id_{(i_m,j_m)}$. Indeed, if $K: I_{\pi_0} \to
I_{\pi_1}$ is any Kirby\break tangle, in $K \circ \id_{\pi_0}$ the upper framed components
of $\id_{\pi_0}$ get closed and we can slide the lower open components over the closed
ones and then cancel them with the dotted components; a symmetric argument works on the
top of ${\id_{\pi_1}} \!\circ K$.

\begin{Figure}[htb]{kirby-morph01/fig}
{}{Some morphisms in $\K_n$}
\vskip-3pt
\centerline{\fig{}{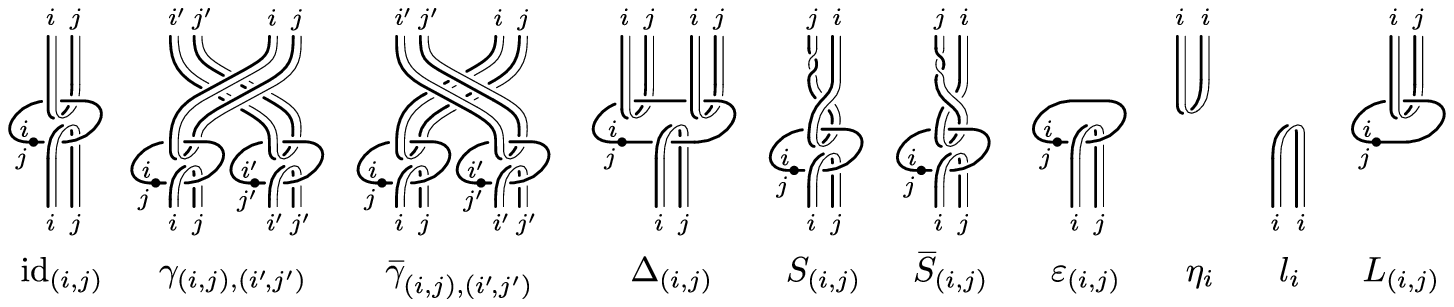}}
\vskip-6pt
\end{Figure}

Finally, we endow $\K_n$ with a family of braiding isomorphisms $\gamma_{\pi,\pi'}: I_{\pi
\diam \pi'} \to I_{\pi' \diam \pi}$. Namely, $\gamma_{(i,j), (i',j')}$ is presented in
Figure \ref{kirby-morph01/fig} and its inverse is $\gamma_{(i,j), (i',j')}^{-1} =
\bar\gamma_{(i',j'),(i,j)}$, while the braiding isomorphisms on the other objects are
obtained inductively by the relations in Definition \ref{braided-cat/def} (see Figure
\ref{kirby-morph02/fig} and note that $\gamma_{\pi, \pi'}^{-1} = \bar\gamma_{\pi',\pi}$).

\begin{Figure}[htb]{kirby-morph02/fig}
{}{The braiding isomorphisms in $\K_n$}
\vskip-3pt
\centerline{\fig{}{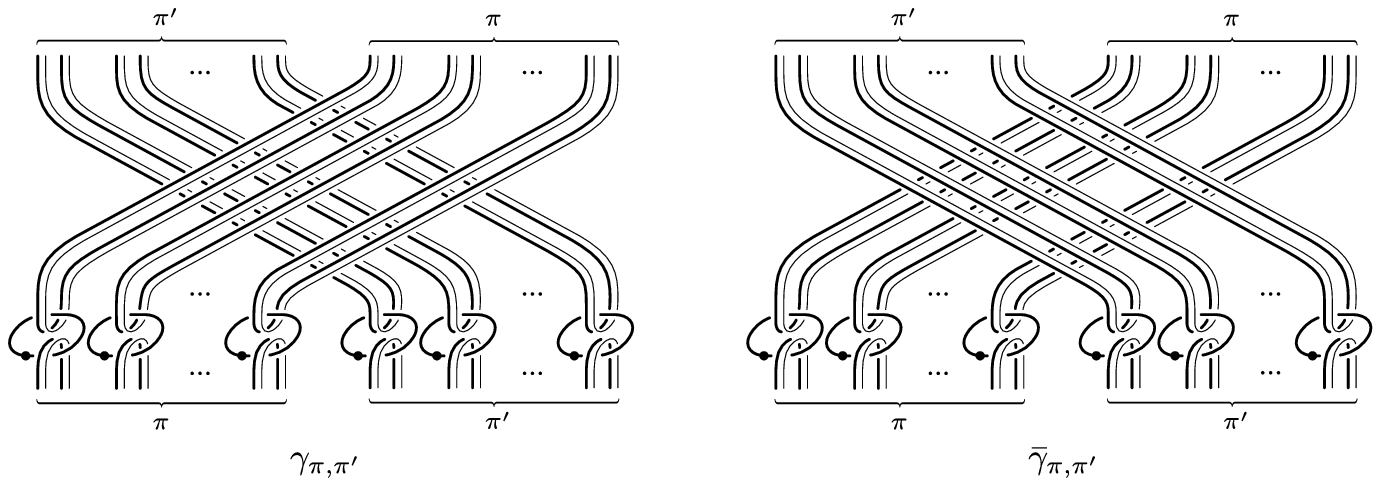}}
\vskip-6pt
\end{Figure}

\begin{proposition}\label{K-category/thm}
For any $n \geq 1$, the category $\K_n$ is equivalent as a strict monoidal category to
$\Chb_n^{3+1}$, through the functor induced by the map $K \mapsto W_{T_K}$. Moreover, the
family of braiding isomorphisms defined above makes $\K_n$ into a braided strict monoidal
category.
\end{proposition}

\begin{proof}
The first part of the statement is nothing else than the categorical version of
Proposition \ref{kirby-tangle/thm}. The proof of the second part consists in the
straightforward verification of the naturality of the braiding isomorphisms.
\end{proof}

Before going on, we observe that the Kirby tangles in Figure \ref{kirby-morph01/fig}
represent the elementary morphisms of a braided Hopf algebra structure on $\K_n$ in the
sense of Definition \ref{hopf-algebra/def}. This will be shown in Section \ref{Phi/sec}, 
where we will relate $\K_n$ to the algebraic category $\H^r_n$. Here, we only need the 
comultiplication of that structure, in order to introduce reducible Kirby tangles, as 
discussed below.

\medskip

For $n > k \geq 1$, we denote by $\iota_k^n: \K_k \subset \K_n$ the faithful functor
induced by the inclusion $\G_k \subset \G_n$, which considers any $k$-labeled Kirby
tangle as an $n$-labeled one. This functor corresponds to the homonymous functor
$\iota_k^n: \Chb_k^{3+1} \subset \Chb_n^{3+1}$ through the equivalence of Proposition
\ref{K-category/thm}.

At the end of Section \ref{Chbn/sec} we defined the stabilization functor $\up_k^n:
\Chb_k^{3+1} \to \Chb_n^{3+1}$ and the subcategory $\Chb_n^{3+1,c}$ of $\Chb_n^{3+1}$ for
any $n > k \geq 1$, and we claimed that $\Chb_n^{3+1,c}$ is equivalent to $\Chb_1^{3+1}$
through the this functor. Here, we will translate those definitions in terms of Kirby
tangles and prove the statement.

\begin{definition}\label{K-stabilization/def}
Given $n > k \geq 1$, let $\pi_{n \red k} = ((n,n-1), \dots, (k+1,k))$ and let
$\id_{\pi_{n \red k}}\!$ be the identity morphism of $I_{\pi_{n \red k}}\!$ in $\K_n$.
Then, the {\sl stabilization functor}\break $\up_k^n: \K_k \to \K_n$ is defined by:
$$
\begin{array}{c}
\up_k^n I_\pi = I_{\pi_{n \red k}} \!\diam \iota_k^n(I_\pi) 
\,\text{ for any } I_\pi \in \Obj \K_k\,,\\[6pt]
\up_k^n K = \id_{\pi_{n \red k}} \!\diam \iota_k^n(K) 
\,\text{ for any } K \in \Mor \K_k\,.
\end{array}$$\vskip-\lastskip
\end{definition}\vskip-\lastskip

From the definition, we immediately see that $\up_k^n = {\up_{n-1}^n} \circ \dots \circ
{\up_k^{k+1}}$. Figure \ref{kirby-stab01/fig} shows the stabilization $\up_k^n K$ of
an $k$-labeled Kirby tangle $K \in \K_k$ from $\pi_0$ to $\pi_1$.

\begin{Figure}[htb]{kirby-stab01/fig}
{}{The stabilization $\up_k^n K$ of $K \in \K_k$}
\vskip-3pt
\centerline{\fig{}{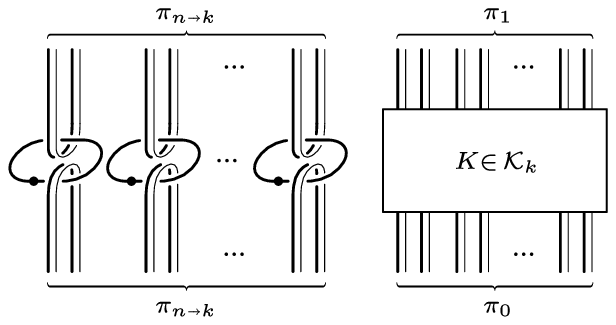}}
\vskip-6pt
\end{Figure}

Clearly, this stabilization functor corresponds, through the category equivalences given
by Proposition \ref{K-category/thm}, to the stabilization functor $\up_k^n: \Chb_k^{3+1}
\to \Chb_n^{3+1}$ defined at the end of Section \ref{Chbn/sec}.

Now, let $\Delta_{(i,j)}:I_{(i,j)} \to I_{(i,j)} \diam I_{(i,j)}$ be the tangle presented
in Figure \ref{kirby-morph01/fig}, for any $(i,j) \in \G_n$. We extend this definition to
$\Delta_\pi: I_\pi \to I_\pi \diam I_\pi$ for any $\pi \in \seq\G_n$ (see Figure
\ref{kirby-morph03/fig}) by the recursive formula:
\label{Delta-pi/def}
$$
\Delta_\pi = \Delta_{\pi' \diam \pi''} = 
(\id_{\pi'} \diam \gamma_{\pi',\pi''} \diam
\id_{\pi''}) \circ (\Delta_{\pi'} \diam \Delta_{\pi''})\,,
$$
which can be easily seen to give always the same result for $\Delta_\pi$, whatever
the decomposition $\pi = \pi' \diam \pi''$ with $\pi',\pi'' \in \seq \G_n$.

\begin{Figure}[htb]{kirby-morph03/fig}
{}{The comultiplication morphism $\Delta_\pi$ in $\K_n$}
\centerline{\fig{}{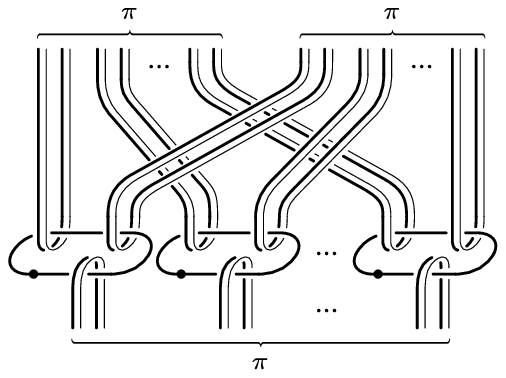}}
\vskip-6pt
\end{Figure}

According to observation above, the family $\Delta_\pi$ represents the comultiplication in 
$\K_n$ in sense of Definition \ref{hopf-algebra/def}. In particular, next proposition 
states that $\Delta$ satisfies the coassociativity property.

\begin{proposition}\label{K-coassociativity/thm}
For any $\pi \in \seq\G_n$, we have $$ (\Delta_\pi \diam \id_\pi) \circ \Delta_\pi = 
(\id_\pi \diam \Delta_\pi) \circ \Delta_\pi\,.$$
\end{proposition}

\begin{proof}
The elementary case of $\pi = (i,j)$ is shown in Figure \ref{kirby-morph04/fig}, while the 
general case follows from this by induction and tangle isotopy.
\end{proof}

\begin{Figure}[t]{kirby-morph04/fig}
{}{The coassociativity property for $\Delta_{(i,j)}$}
\centerline{\fig{}{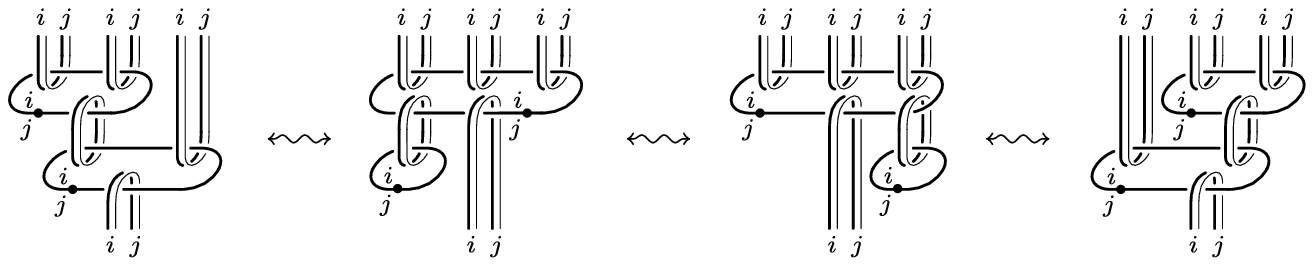}}
\vskip-3pt
\end{Figure}

\begin{definition}\label{K-reducible/def}
Given $n > k \geq 1$ and $\pi_0,\pi_1 \in \seq\G_n$, we say that an $n$-labeled Kirby
tangle $K \in \K_n$ from $I_{\pi_{n \red k}} \!\diam I_{\pi_0}$ to $I_{\pi_{n \red k}}
\!\diam I_{\pi_1}$ is {\sl $k$-reducible} if it has the form $$K = (\id_{\pi_{n \red k}}
\!\diam L) \circ (\Delta_{\pi_{n \red k}} \!\diam \id_{\pi_0})\,,$$ 
\vskip3pt\noindent
for some $n$-labeled Kirby tangle $L \in \K_n$ from $I_{\pi_{n \red k}} \!\diam I_{\pi_0}$
to $I_{\pi_1}$ (see Figure \ref{kirby-stab02/fig}).

The composition of two $k$-reducible Kirby tangles is still $k$-reducible (by
coassociativity) and we denote by $\K_{n \red k}$ the subcategory of $\K_n$, whose objects
are $I_{\pi_{n \red k}} \!\diam I_\pi$ with $\pi \in \seq\G_n$ and whose morphisms are
$k$-reducible $n$-labeled Kirby tangles. In particular, we denote by $\K_n^c$ the
subcategory $\K_{n \red 1}$ of $1$-reducible tangles in $\K_n$.
\end{definition}

\begin{Figure}[htb]{kirby-stab02/fig}
{}{The generic $k$-reducible morphism $K \in \K_{n \red k}$}
\vskip-6pt
\centerline{\fig{}{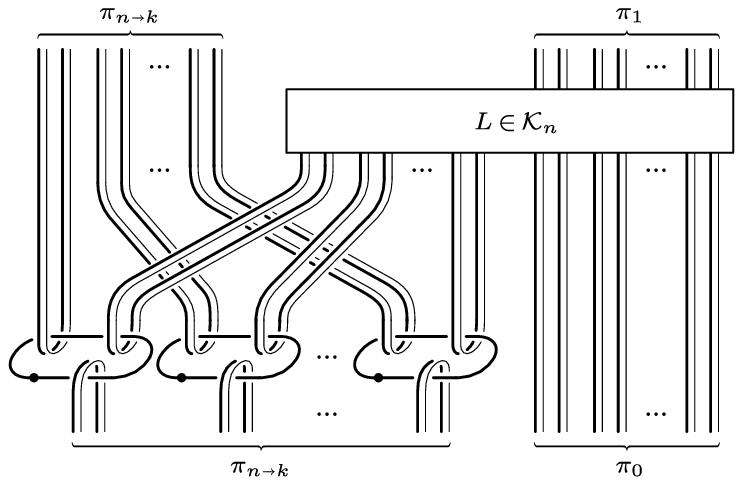}}
\vskip-6pt
\end{Figure}

We observe that the subcategory $\K_{n \red k}$ of $k$-reducible Kirby tangles is not
closed with respect to the monoidal product ${\diam}: \K_n \times \K_n \to \K_n$.
Nevertheless, we can define a product structure ${\rdiam}: \Mor_{\K_{n \red k}} \times
\Mor_{\K_{n \red k}} \to \Mor_{\K_{n \red k}}$ in the following way.

Given two morphisms $K = (\id_{\pi_{n \red k}} \! \diam L) \circ (\Delta_{\pi_{n \red k}}
\!\diam \id_{\pi_0}):I_{\pi_{n \red k}}\diam I_{\pi_0}\to I_{\pi_{n \red k}}\diam
I_{\pi_1}$ and $K' = (\id_{\pi_{n \red k}} \! \diam L') \circ (\Delta_{\pi_{n \red k}}
\!\diam \id_{\pi'_0}):I_{\pi_{n \red k}}\diam I_{\pi_0'}\to I_{\pi_{n \red k}}\diam
I_{\pi_1'}$ in $\K_{n \red k}$, their product $K \rdiam K': I_{\pi_{n \red k}}\diam
I_{\pi_0 \diam \pi_0'} \to I_{\pi_{n \red k}} \diam I_{\pi_1\diam\pi_1'}$ is defined by
\begin{eqnarray*}
K \rdiam K' 
&=& K \circ (\id_{\pi_{n \red k}} \!\diam \gamma_{\pi'_1,\pi_0}) \circ (K' \diam
\id_{\pi_0}) \circ (\id_{\pi_{n \red k}} \!\diam \gamma^{-1}_{\pi'_0,\pi_0})\\
&=& (\id_{\pi_{n \red k}} \!\diam L \diam L') \circ (\Delta_{\pi_{n \red k}} \!\diam
\gamma_{\pi_{n \red k},\pi_0} \!\diam \id_{\pi'_0}) \circ (\Delta_{\pi_{n \red k}}
\!\diam \id_{\pi_0 \diam \pi'_0})\,.
\end{eqnarray*}
These two expressions for $K \rdiam K'$ are related by diagram isotopy, as the reader can
easily realize by looking at Figure \ref{kirby-stab03/fig} that represents the second one.
The associativity of $\,\rdiam\,$ is a consequence of the coassociativity property of
$\Delta$ and its unit is given by $\id_{\pi_{n \red k}}\!$. Observe that $\rdiam$ does not
define a monoidal structure on $\K_{n \red k}$\break since in general the product of
compositions $(K_2 \circ K_1) \rdiam (K'_2 \circ K'_1)$ does not coincide with the
composition of products $(K_2 \rdiam K_2') \circ (K_1 \rdiam K_1')$. Yet, we will find the
notation a useful tool in describing some identities.

\begin{Figure}[htb]{kirby-stab03/fig}
{}{The product $K \rdiam K'$ of two morphisms in $\K_{n \red k}$}
\vskip-3pt
\centerline{\fig{}{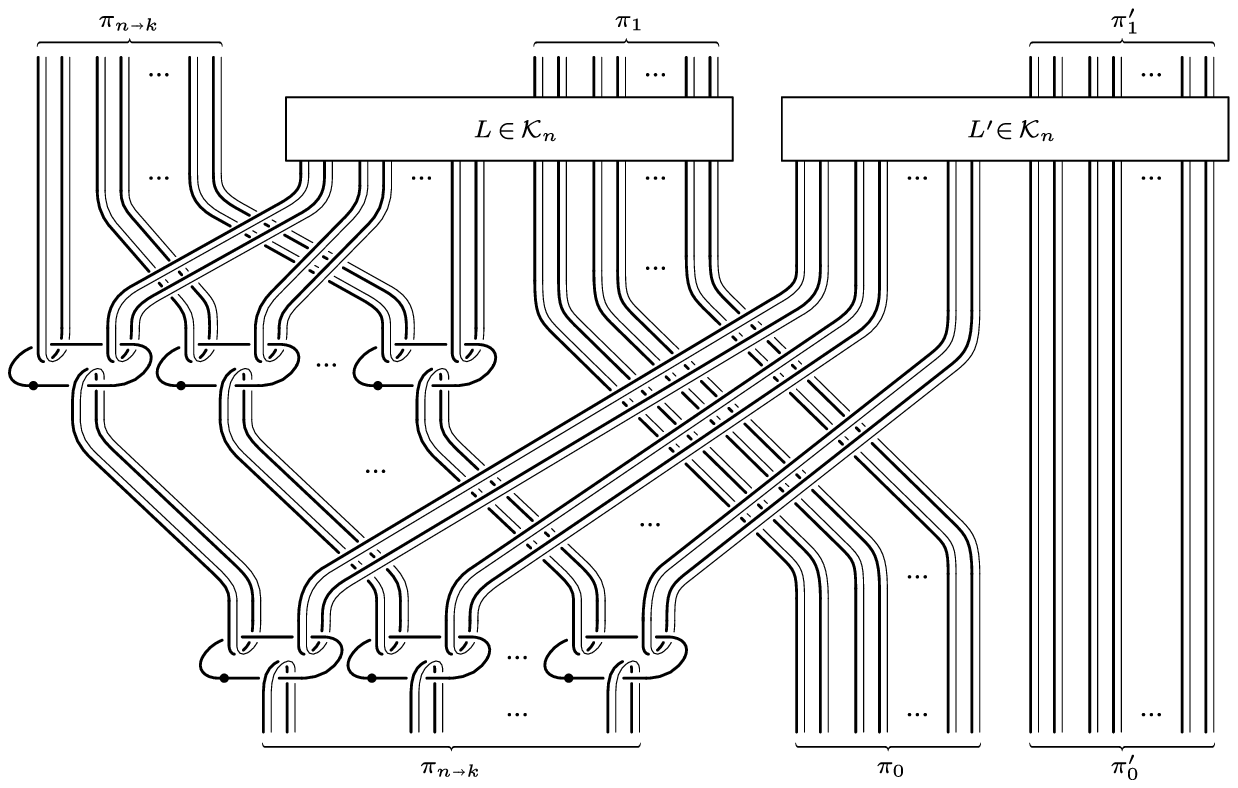}}
\vskip-6pt
\end{Figure}

\begin{proposition}\label{K-monoidal-stab/thm}
For any $n > k \geq 1$, the image $\up_k^n \K_k$ of the stabilization functor is a
subcategory of $\K_{n \red k}$. Moreover, for any two morphisms $K$ and $K'$ in $\K_k$, we
have $\up_k^n(K \diam K') = (\up_k^n K) \rdiam (\up_k^n K')$. Hence, the product 
$\,\rdiam\,$ defines a monoidal structure on the subcategory $\up_k^n \K_k$.
\end{proposition}

\begin{proof}
Figure \ref{kirby-stab04/fig} shows that the $n$-stabilization of an $(n-1)$-labeled Kirby
tangle is $(n-1)$-reducible, in other words $\up_{n-1}^n\! \K_{n-1} \subset \K_{n \red
(n-1)}$, for any $n \geq 2$. This fact easily implies by induction that $\up_k^n \K_k$ is
a subcategory of $\K_{n \red k}$ for any $n > k \geq 1$.

\begin{Figure}[htb]{kirby-stab04/fig}
{}{Stabilizations are reducible}
\vskip-3pt
\centerline{\fig{}{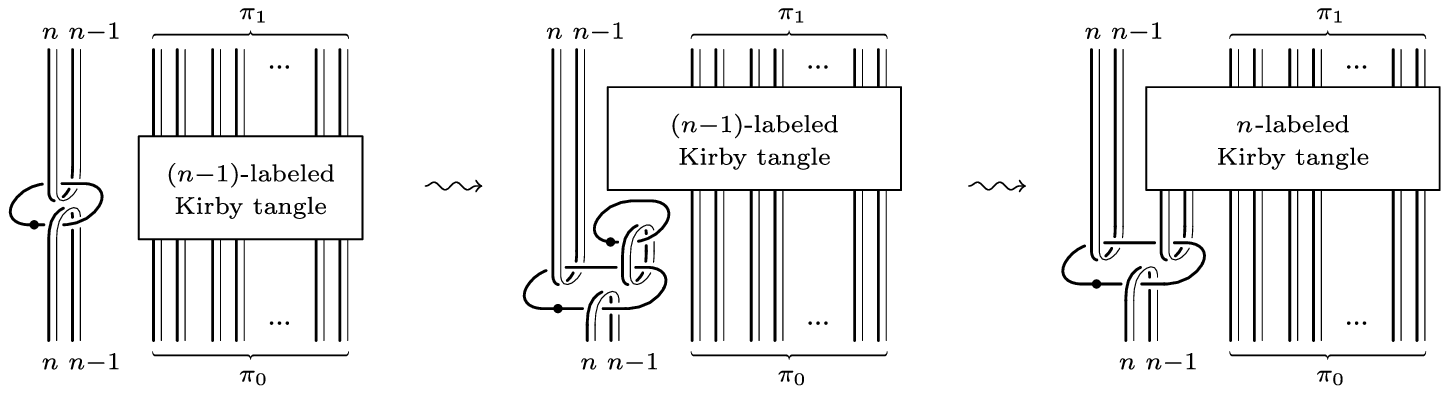}}
\vskip-6pt
\end{Figure}

The identity $\up_k^n (K \diam K') = (\up_k^n K) \rdiam (\up_k^n K')$ for any $K$ and $K'$
in $\K_k$, immediately follows from the definition of the product $\rdiam$ in $\K_{n \red
k}$.
\end{proof}

Our goal is to prove that $\up_k^n: \K_k \to \K_{n \red k}$ is actually an equivalence of
monoidal categories. We will show this by defining a reduction functor $\down_k^n: \K_{n
\red k} \to \K_k$, which is the inverse of the stabilization functor up to natural
equivalence. Actually, it is enough to define $\down_{n-1}^n$ and then proceed
inductively. The main idea behind the definition of such functor is that, in the presence
of a 1-handle with label $(n,n-1)$, we can push the part of diagram contained in the
$n$-th 0-handle through such 1-handle, obtaining in this way a diagram which lives
entirely in the first $n-1$ 0-handles. Therefore, as first step we will prove Lemma
\ref{K-push/thm} below, which formalizes the move of pushing through a 1-handle with
generic label $(i_0,j_0)$. In the present context such formalization may seem excessive,
since the statement of the lemma and its proof are quite straightforward. Nevertheless,
the proof of its algebraic analog (Proposition \ref{H-push/thm}) will require a
significant amount of work and we hope that seeing how things work out for Kirby tangles
will be helpful.

\begin{lemma}\label{K-push/thm}
Given $x = (i_0,j_0) \in \G_n$ and $\pi \in \seq\G_n$, let $\pi^x$  be the sequence
obtained from $\pi$ by changing all elements  $i_0$ to $j_0$. Then, there exists a 
monoidal functor 
$\_^x: \K_n\to\K_n$ such that $(I_\pi)^x = I_{\pi^x}$ on the set of objects and the 
following properties hold:
\begin{itemize}
\item[\(a)] 
if $i_0 = j_0$, then $K^x = K^{(i_0,i_0)} = K$ for every $K \in \K_n$, that is
$\_^{(i_0,i_0)} = \id_{\K_n}$;
\item[\(b)] 
if $i_0 \neq j_0$, then $K^x = K^{(i_0,j_0)} \in \K_n^{\bs i_0}$ for every $K \in \K_n$,
where $\K_n^{\bs i_0}$ is the sub\-category of $\K_n$ generated by objects and morphisms
which do not contain the label $i_0$, hence we have a functor $\_^{(i_0,j_0)}: \K_n \to
\K_n^{\bs i_0}$;
\item[\(c)]
given any other $y = (i_0,k_0) \in \seq\G_n$, there exists a natural equivalence 
$$\xi^{x,y}: \id_{(k_0,j_0)} \diam \_^x \to \id_{(k_0,j_0)} \diam \_^y\,.$$
\smallskip
\end{itemize}
\vskip-\lastskip
In particular, we put $\xi^x = \xi^{x,(i_0,i_0)}$ and denote by $\xi^x_\pi: I_{x}\diam
I_{\pi^x}\to I_{x} \diam I_\pi$ the relative isomorphism for $\pi \in \seq\G_n$, in such a
way that the following identity holds for any diagram $K \in \K_n$ from $I_{\pi_0}$ to
$I_{\pi_1}$:
$$(\id_x \diam K) \circ \xi_{\pi_0}^x = \xi_{\pi_1}^x \circ (\id_x \diam K^x)\,.$$
\end{lemma}

Before proving the lemma, we make a few observations. The natural equivalence $\xi^x$
will be given by a 1-handle of label $x = (i_0,j_0)$ (cf. Figure \ref{kirby-stab06/fig}).
Then the last identity implies that the map $K \mapsto K^x$ represents how a Kirby tangle
changes when the part of the diagram which lives in the $i_0$-th 0-handle is pushed to the
$j_0$-th 0-handle through such 1-handle. In this perspective, points \(a) and \(b) of the
statement indicate that if $i_0 = j_0$ the tangle can be pushed through without any
change, while if $i_0 \neq j_0$ the resulting tangle lives in the other 0-handles
different from the $i_0$-th. More generally, for $y = (i_0,k_0)$, the natural equivalence 
$\xi^{x,y}$ in \(c) will be given by a 1-handle of label $(k_0,j_0)$. Then \(c) implies 
that $K^y$ can be obtained from $K^x$ by pushing it through such 1-handle (cf. Figures
\ref{kirby-stab06/fig} and \ref{kirby-stab08/fig}).

\begin{proof}[Lemma \ref{K-push/thm}]
Let $K \in \K_n$ be a labeled Kirby tangle from $I_{\pi_0}$ to $I_{\pi_1}$. We define
$K^x$ for $x = (i_0,j_0)$ to be the labeled Kirby tangle obtained from any strictly
regular plane diagram of $K$ in the following way (see Figure \ref{kirby-stab05/fig}): we
first pull all the parts of the diagram with label $i_0$ on the top of the ones with label
$i \neq i_0$, by performing a crossing change (as in Figure \ref{kirby-tang01/fig}) at the
crossings where a framed arc labeled $i_0$ passes under one labeled $i \neq i_0$, and
flipping over (as in Figure \ref{kirby-tang06/fig} \(a)) the spanning disks of the dotted
unknots with the bottom side labeled $i_0$ and the top one $i \neq i_0$; then we replace
all labels $i_0$ by $j_0$.

\begin{Figure}[htb]{kirby-stab05/fig}
{}{The functorial map $K \mapsto K^x$ ($x = (i_0,j_0)$ and $i \neq i_0$)}
\centerline{\fig{}{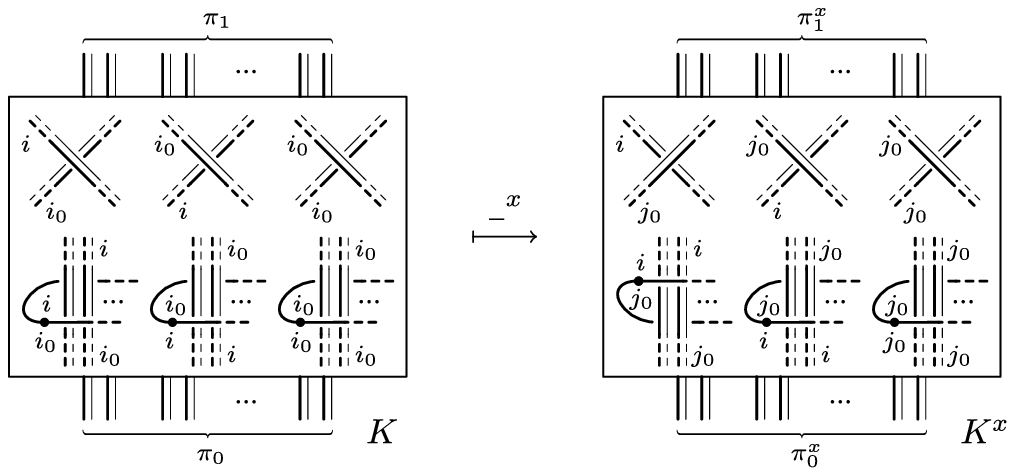}}
\vskip-6pt
\end{Figure}

\pagebreak

As the reader might have understood, the only essential modifications before the label
replacement are the crossing changes with $i = j_0 \neq i_0$, being the other crossing
changes and the disk flippings reversible after that replacement. Nevertheless,
including also these inessential modifications in the definition of $K^x$, better
interprets the geometric idea of pulling all the parts labeled by $i_0$ on the top and
makes more transparent the proof of \(c) below.

To see that $K^x$ is well-defined, we first check that it does not depend on the strictly
regular diagram of $K$ we started from. Indeed, when changing $K$ by planar isotopy,
labeled framed Reidemeister moves and the moves in Figure \ref{kirby-tang06/fig}, $K^x$
changes in the same way, except for some extra crossing changes and an obvious extra
sliding under a 1-handle for particular labelings of the moves \(b) and \(c) in Figure
\ref{kirby-tang06/fig}. Namely, the exceptions occur in \(b) (resp. \(c)) when $k = i_0$
(resp. $k \neq i_0$) and exactly one of $i$ and $j$ coincides with $i_0$.
Then, we observe that any of the moves in Figure \ref{kirby-tang01/fig} applied to $K$, 
induces an analogous move on $K^x$. Therefore, the 2-equivalence class of $K^x$ depends 
only on the 2-equivalence class of $K$.

\begin{Figure}[b]{kirby-stab06/fig}
{}{The morphisms $\xi^{x,y}_\pi$ and $(\xi^{x,y}_\pi)^{-1}$
   ($x = (i_0,j_0)$ and $y = (i_0,k_0)$)}
\vskip-3pt
\centerline{\fig{}{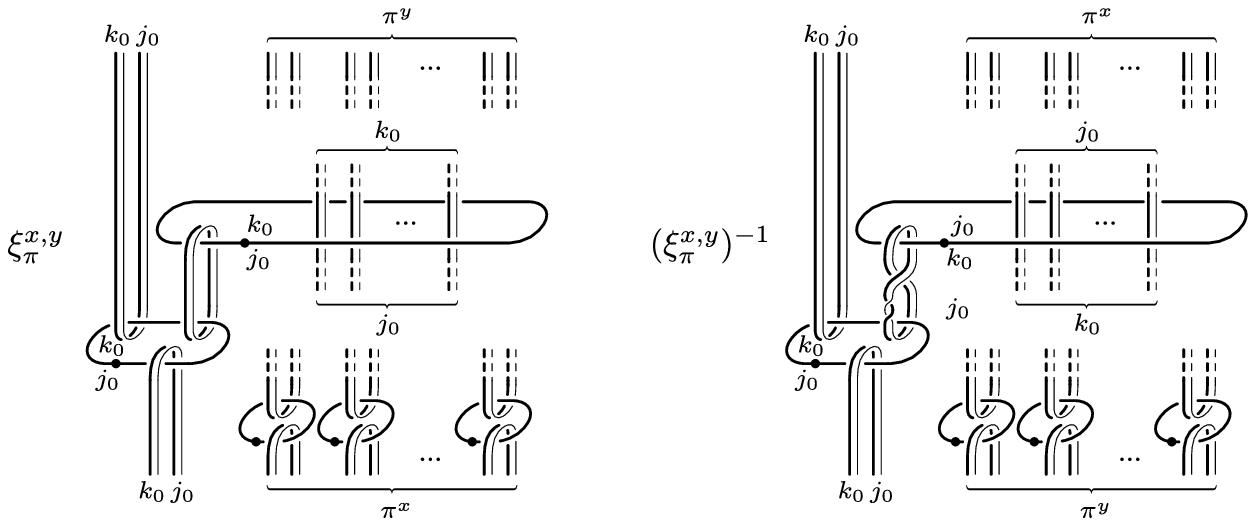}}
\vskip-6pt
\end{Figure}

At this point, the functoriality and the monoidality of the map $\_^x: K \mapsto K^x$, as
well as property \(b), are immediate. Moreover, if $i_0 = j_0$ no label change occurs and
we can undo the crossing changes obtaining $K^x = K$, which gives \(a).

It remains to prove \(c). Given $x = (i_0,j_0)$ and $y = (i_0,k_0) $ in $\G_n$, we define
$\xi^{x,y}_\pi: I_{(k_0,j_0)} \diam I_{\pi^x} \to I_{(k_0,j_0)} \diam I_{\pi^y}$ to be the
tangle presented in Figure \ref{kirby-stab06/fig}, where the largest dotted unknot
embraces only the framed strings originally labeled $i_0$ in $I_\pi$.\break
In the same Figure \ref{kirby-stab06/fig} it is also presented the inverse tangle
$(\xi^{x,y}_\pi)^{-1}$. The equivalence $\xi^{x,y}_\pi \circ (\xi^{x,y}_\pi)^{-1} =
\id_{(k_0,j_0) \diam \pi^y}$ is shown in Figure \ref{kirby-stab07/fig}, where the last
step is not drawn and consists in canceling the two framed unknots with the dotted ones.
The reader can check in a similar way that $(\xi^{x,y}_\pi)^{-1} \!\circ \xi^{x,y}_\pi =
\id_{(k_0,j_0) \diam \pi^x}$ as well.

\begin{Figure}[t]{kirby-stab07/fig}
{}{$\xi^{x,y}_\pi \circ (\xi^{x,y}_\pi)^{-1} = \id_{(k_0,j_0) \diam \pi^y}$ 
   ($x = (i_0,j_0)$ and $y = (i_0,k_0)$)}
\centerline{\fig{}{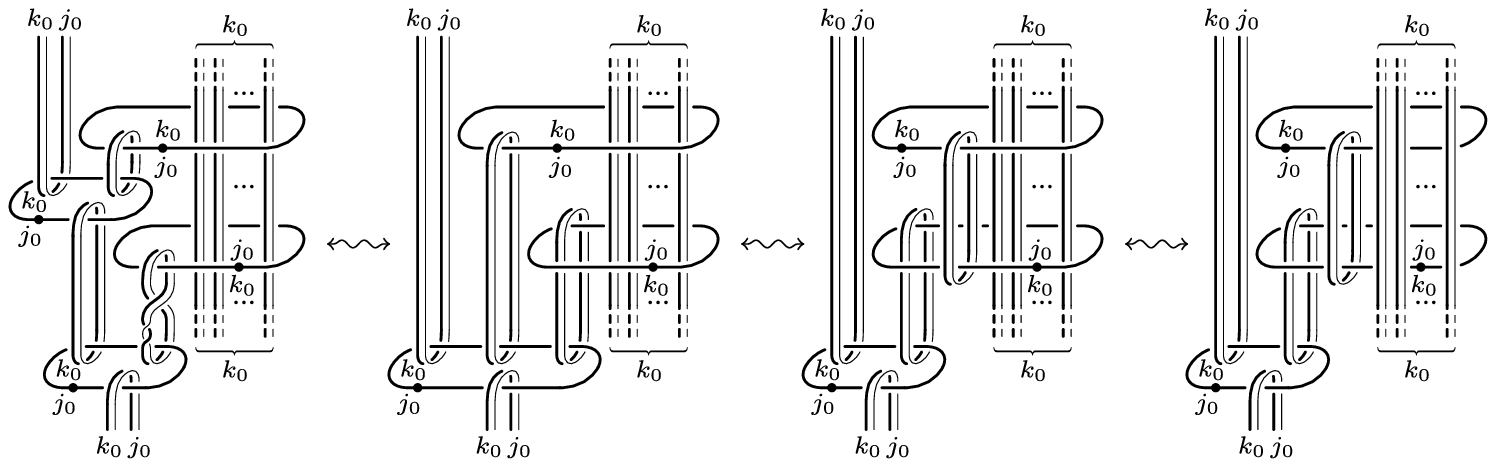}}
\vskip-3pt
\end{Figure}

\begin{Figure}[b]{kirby-stab08/fig}
{}{The natural equivalence $\xi^{x,y}$ ($x = (i_0,j_0)$ and $y = (i_0,k_0)$)}
\vskip3pt
\centerline{\fig{}{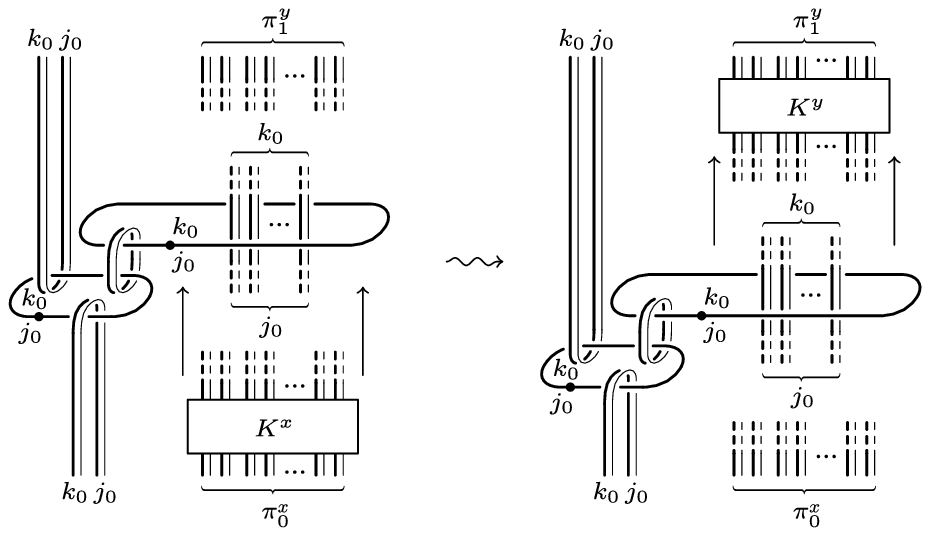}}
\vskip-6pt
\end{Figure}

\begin{Figure}[htb]{kirby-stab09/fig}
{}{Pushing a dotted unknot through another one ($i \neq j_0$)}
\centerline{\fig{}{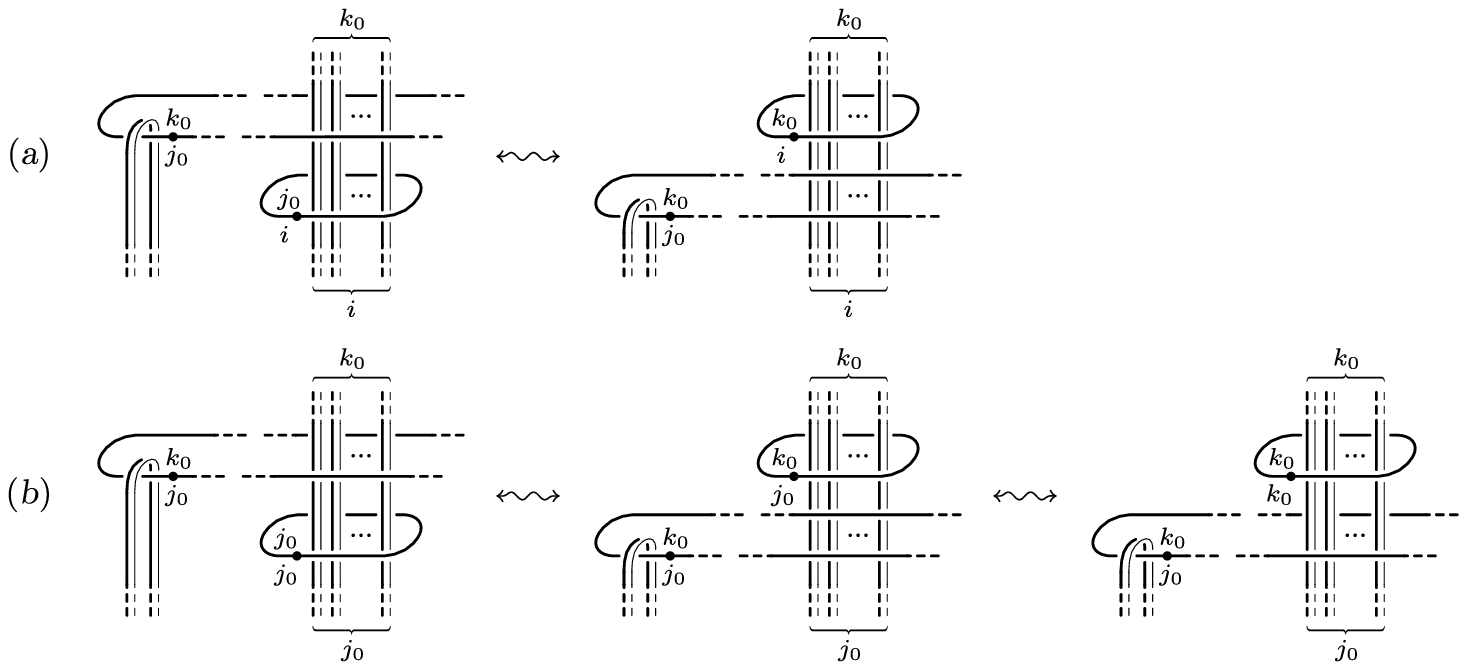}}
\vskip-6pt
\end{Figure}

With the above definition of $\xi^{x,y}_\pi$, the identity 
\vskip-12pt
$$(\id_{(k_0,j_0)} \diam K^y)
\circ \xi_{\pi_0}^{x,y} = \xi_{\pi_1}^{x,y} \circ (\id_{(k_0,j_0)} \diam K^x)$$
\vskip4pt\noindent
corresponding to \(c) is obtained by pushing all (and only) the tangle components labeled
by $j_0$ in $K^x$ and originally labeled $i_0$ in $K$, through the spanning disk of the
dotted unknot of $\xi^{x,y}_{\pi_1^x}$, as indicated by the arrows in Figure
\ref{kirby-stab08/fig}. Since those components lie above the rest of the tangle, the
pushing through can be realized by Reidemeister moves and the moves \(d) and \(e) of
Figure \ref{kirby-tang06/fig}, until a dotted unknot is encountered, whose spanning disk
has one or both sides labeled by $j_0$ while originally labeled $i_0$ in $K$. When this
happens we proceed as shown in Figure \ref{kirby-stab09/fig} \(a) or \(b) respectively,
once the disk has been put in the right position by planar isotopy. In this figure, each
step consists in a 1-handle sliding (cf. Figure \ref{kirby-tang05/fig}).
\end{proof}

Looking at Figure \ref{kirby-morph01/fig}, let us discuss how the functor $\_^x$ with $x =
(i_0,j_0)$ acts on the diagrams drawn there. This will be useful in Section
\ref{adjoint/sec}, when we will relate $\_^x$ to its algebraic analog (cf. Proposition
\ref{KtoH-push/thm}).

First of all, we observe that all of them but $\Delta_{(i,j)}$ are represented by strictly
regular diagrams, therefore the definition directly applies on those diagrams. According
to the definition of $\_^x$ and the subsequent observation about the essential
modifications occurring in it, we immediately see that $\_^x$ only leads to labeling
changes for $\id_{(i,j)}$, $\epsilon_{(i,j)}$, $\eta_i$, $l_i$ and $L_{(i,j)}$, while it
also involves some crossing changes for $S_{(i,j)}$, $\bar S_{(i,j)}$,
$\gamma_{(i,j),(i',j')}$ and $\bar\gamma_{(i,j),(i',j')}$, in the cases when some of the
labels are equal to $i_0$. We emphasize once again that not all the involved crossing
changes are essential. The only essential ones occur on $S_{(i,j)}$ when $i = j_0$ and $j
= i_0$, on $\bar S_{(i,j)}$ when $i = i_0$ and $j = j_0$, on $\gamma_{(i,j),(i',j')}$ when
$j_0 \in \{i,j\}$ and $i_0 \in \{i',j'\}$, and on $\bar\gamma_{(i,j),(i',j')}$ when $i_0
\in \{i,j\}$ and $j_0 \in \{i',j'\}$ (cf. Figures \ref{kirby-stab10/fig} and
\ref{kirby-stab11/fig}).

\begin{Figure}[htb]{kirby-stab10/fig}
{}{The tangles $(S_{(j_0,i_0)})^x$ and $(\bar S_{(i_0,j_0)})^x$}
\vskip3pt
\centerline{\fig{}{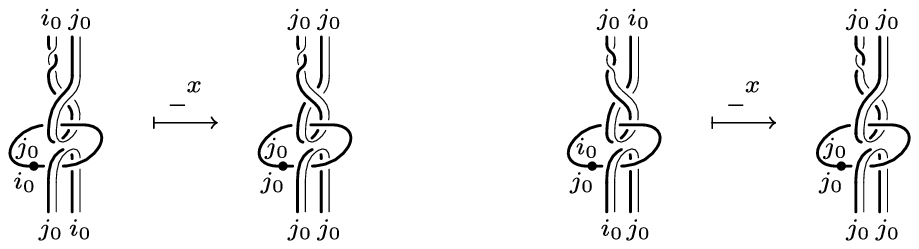}}
\vskip-3pt
\end{Figure}

\begin{Figure}[htb]{kirby-stab11/fig}
{}{The tangle $(\gamma_{(i,j),(i',j')})^x$ for some particular $(i,j),(i',j') \in \G_n$}
\centerline{\fig{}{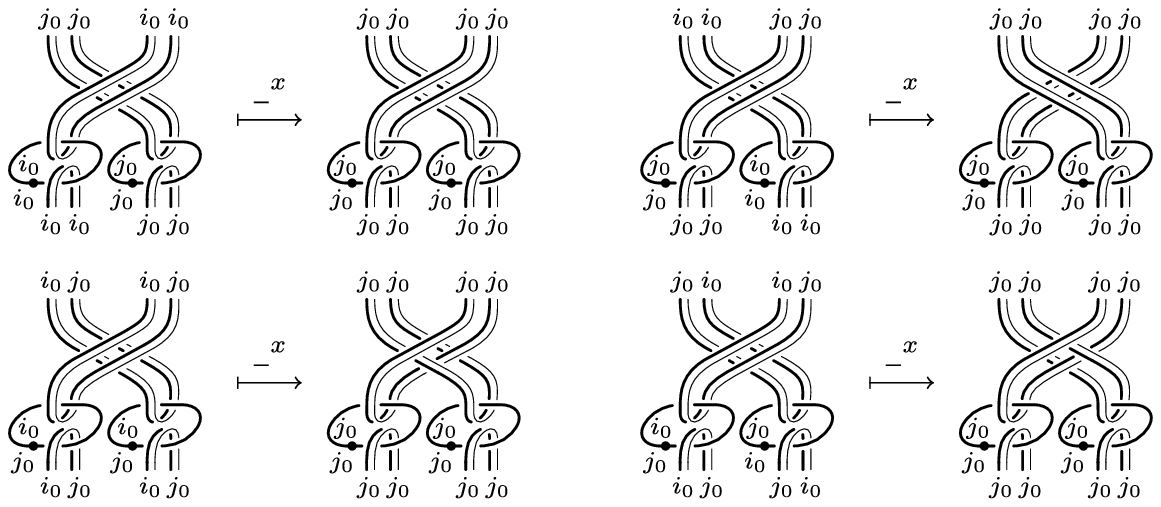}}
\vskip-3pt
\end{Figure}

Concerning the remaining morphism $\Delta_{(i,j)}$, we have first to put its diagram in
strictly regular form. However, also for this morphism the functor $\_^x$ does not imply
anything more that a labeling change, except in the case when $i \neq i_0$ and $j = i_0$,
which is illustrated in Figure \ref{kirby-stab12/fig}. Here, \(b) and \(c) are the
required strictly regular diagrams of the morphism and of its image under $\_^x$, \(d) is
just the same as \(c) up diagram isotopy, while a 1-handle twist is needed to get \(e). It
is worth noticing that the final result is the same as if we had reversed the 1-handle in
the original diagram, by using the general reversing procedure described in Figure
\ref{ribbon-kirby08/fig} of the next chapter, and then performed the prescribed crossing
changes and labeling replacement, even without converting the diagram in strictly regular
form. Moreover, we observe that the crossing change between \(b) and \(c) is essential
only when $i = j_0$ as above, hence for $i \neq j_0$ we still have $(\Delta_{(i,i_0)})^x =
\Delta _{(i,j_0)}$.

\begin{Figure}[htb]{kirby-stab12/fig}
{}{The tangle $(\Delta_{(i,j)})^x$ for $i \neq i_0$ and $j = i_0$}
\centerline{\fig{}{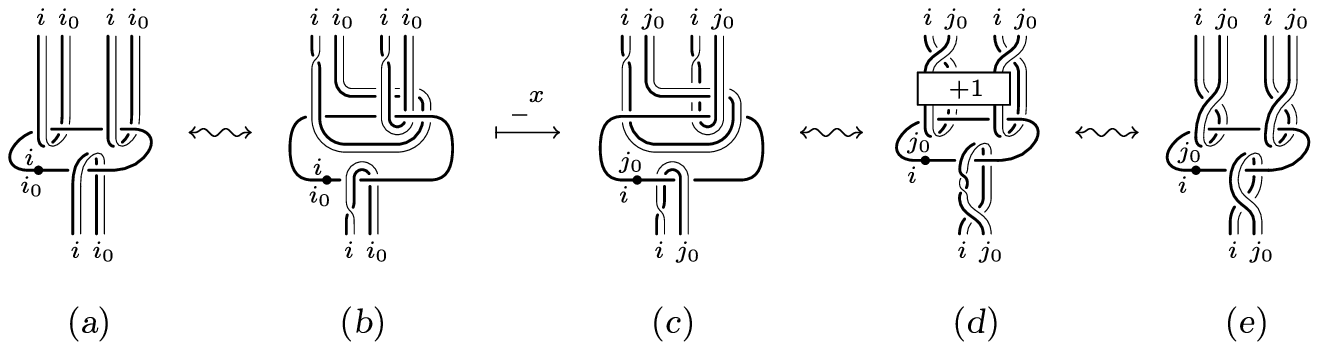}}
\vskip-3pt
\end{Figure}


Now we can proceed with the definition of the reduction functor.

\begin{Figure}[b]{kirby-stab13/fig}
{}{The reduction functor $\down_{n-1}^n$}
\centerline{\fig{}{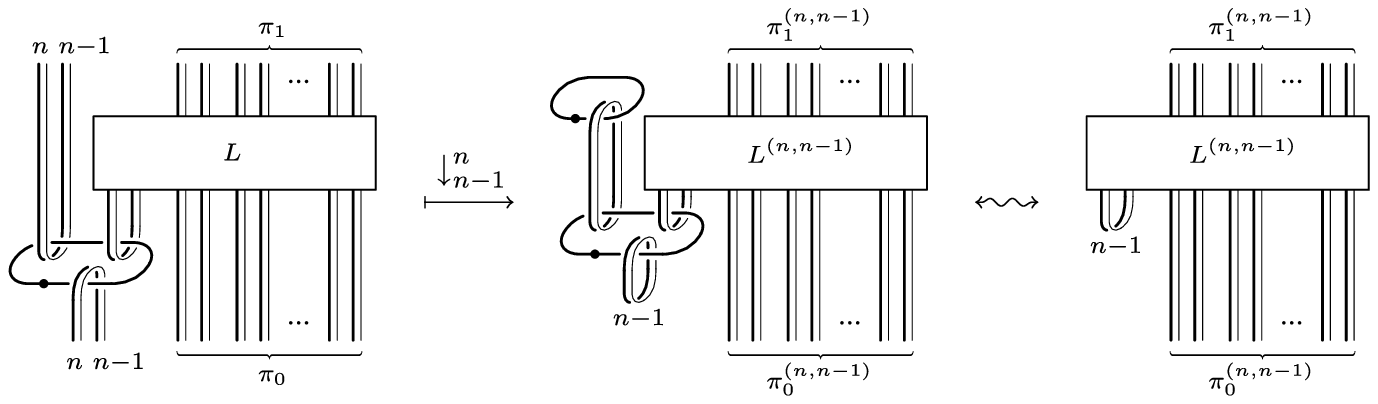}}
\vskip-6pt
\end{Figure}

\begin{definition}\label{K-reduction/def}
Given $n \geq 2$, we define the {\sl elementary reduction functor} $\down_{n-1}^n: \K_{n 
\red (n-1)} \!\to \K_{n-1}$ as follows. For any object of $I_{(n,n-1)} \diam I_\pi$ of 
$\K_{n \red (n-1)}$ we put
$$\down_{n-1}^n (I_{(n,n-1)} \diam I_\pi) = I_{\pi^{(n,n-1)}}\,,$$
while for any morphism $K = (\id_{(n,n-1)} \! \diam\, L) \circ (\Delta_{(n,n-1)} \!\diam
\id_{\pi_0})$ of $\K_{n \red (n-1)}$ from $I_{(n,n-1)} \diam I_{\pi_0}$ to $I_{(n,n-1)}
\diam I_{\pi_1}$ we define (see Figure \ref{kirby-stab13/fig})
\begin{eqnarray*}
\down_{n-1}^n K
&=& (\epsilon_{(n-1,n-1)} \diam \id_{\pi_1^{(n,n-1)}}) \circ K^{(n,n-1)} 
\circ (\eta_{n-1} \diam \id_{\pi_0^{(n,n-1)}})\\
&=& L^{(n,n-1)} \circ (\eta_{n-1} \diam \id_{\pi_0^{(n,n-1)}})\,,
\end{eqnarray*}
where $\epsilon_{(n-1,n-1)}$ and $\eta_{n-1}$ are as in Figure \ref{kirby-morph01/fig}
(for $i = j = n-1)$, $K^{(n,n-1)}$ is defined in Lemma \ref{K-push/thm}, and the last
equality is obtained by 1/2-handle cancelation.

Given $n \geq k \geq 1$, the {\sl reduction functor} $\down_k^n: \K_{n \red k} \!\to \K_n$
is defined as the composition $\down_k^n = {\down_k^{k+1}}\circ \dots \circ
{\down_{n-1}^n}$ of elementary reduction functors.
\end{definition}

\begin{lemma}\label{K-reduction1/thm}
For any $n \geq 2$, the reduction ${\down_{n-1}^n}: \K_{n \red (n-1)} \!\to \K_{n-1}$ is a
functor such that ${\down_{n-1}^n} \circ {\up_{n-1}^n} = \id_{\K_{n-1}}$, while
${\up_{n-1}^n} \circ {\down_{n-1}^n} \simeq \id_{\K_{n \red (n-1)}}$ up to the natural
equivalence $\xi^{(n,n-1)} = \xi^{(n,n-1),(n,n)}$. Therefore, $\down_{n-1}^n$ and
$\up_{n-1}^n$ are category equiva\-lences between $\K_{n \red (n-1)}$ and $\K_{n-1}$.
\end{lemma}

\begin{proof}
The functoriality of $\down_{n-1}^n$ directly follows from that of the map $K \mapsto
K^{(n,n-1)}$ (cf. Lemma \ref{K-push/thm}). Looking at Figure \ref{kirby-stab13/fig}, we
see that $\down_{n-1}^n \circ \up_{n-1}^n = \id_{\K_{n-1}}$. In fact, if the leftmost
diagram in the figure comes from the stabilization of an $(n-1)$-labeled Kirby tangle as
on the left in Figure \ref{kirby-stab04/fig}, then $\pi_0^{(n,n-1)} = \pi_0$,
$\pi_1^{(n,n-1)} = \pi_1$ and essen\-tially nothing is changed inside the box by the
reduction. Hence, we end up with the rightmost diagram that represents the $(n-1)$-labeled
Kirby tangle itself (with an extra canceling 1/2-pair labeled by $n-1$).

\begin{Figure}[htb]{kirby-stab14/fig}
{}{The natural equivalence $\xi^{(n,n-1)}: \up_{n-1}^n \circ \down_{n-1}^n 
   \simeq \id_{\K_{n \red (n-1)}}$}
\centerline{\fig{}{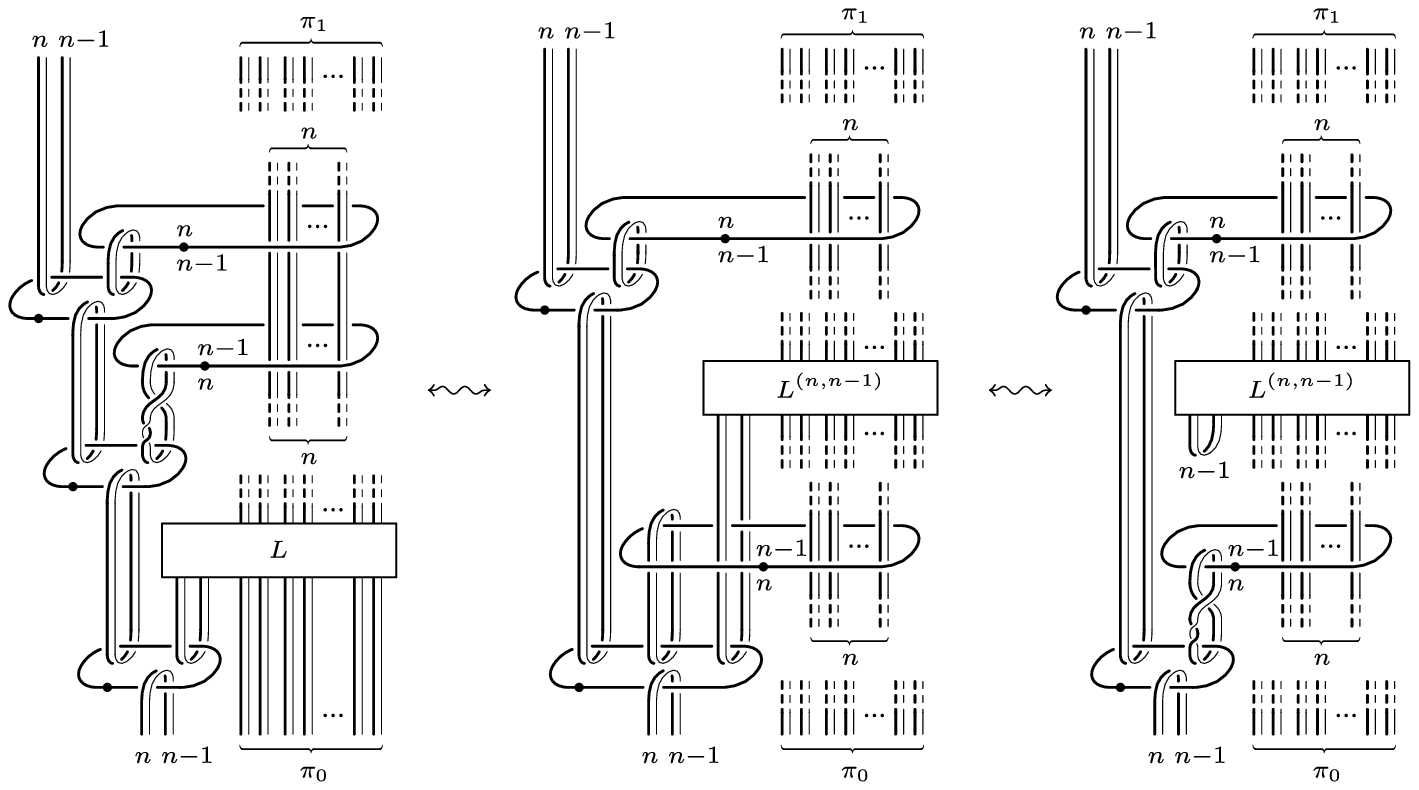}}
\vskip-6pt
\end{Figure}

The proof that $\xi^{(n,n-1)}$ gives a natural equivalence $\id_{\K_{n \red (n-1)}} \to
{\up_{n-1}^n} \circ {\down_{n-1}^n}$ is provided by Figure \ref{kirby-stab14/fig}. Here,
the first diagram is equivalent to the $(n-1)$-re\-ducible tangle $K = (\id_{(n,n-1)}
\!\diam L) \circ (\Delta_{(n,n-1)} \!\diam \id_{\pi_0})$ since $\xi^{(n,n-1)}$ cancels
with its inverse, the second tangle is obtained from the first by using Lemma
\ref{K-push/thm} \(c) (cf. Figure \ref{kirby-stab08/fig} with $x = (n,n-1)$ and $y =
(n,n)$), while the third tangle is obtained through 2-handle slidings and isotopy. This
completes the proof that $\down_{n-1}^n$ and $\up_{n-1}^n$ are equivalences of monoidal
categories.
\end{proof}

\begin{proposition} \label{K-reduction/thm}
For any $n > k \geq 1$, the reduction $\down_k^n: \K_{n \red k}\! \to \K_k$ is a functor
such that ${\down_k^n} \circ {\up_k^n} = \id_{\K_k}$, while there is a natural equivalence
$\xi^{n \red k}: {\up_k^n} \circ {\down_k^n} \simeq \id_{\K_{n \red k}}$, inductively
defined by $\xi^{n \red (n-1)} = \xi^{(n,n-1)}$ and $$\xi^{n \red k} = \xi^{(n,n-1)} \circ
(\id_{(n,n-1)} \!\diam \xi^{(n-1) \red k})\,.$$ Therefore, $\down_k^n$ and $\up_k^n$ are
category equivalences between $\K_{n \red k}$ and $\K_k$.
\end{proposition}

\begin{proof}
We prove the statement by induction on the difference $n - k$. For $k = n - 1$ it follows
from the previous lemma. For $k < n - 1$, taking into account that ${\up_k^n} =
{\up_{n-1}^n} \circ {\up_k^{n-1}}$ and $\up_k^n \K_k \subset \K_{n \red k}$, by the
induction hypothesis we have $${\down_k^n} \circ {\up_k^n} = {\down_k^{n-1}} \circ
{\down_{n-1}^n} \circ {\up_{n-1}^n} \circ {\up_k^{n-1}} = {\down_k^{n-1}} \circ
{\up_k^{n-1}} = \id_{\K_k}\,.$$ Moreover, for any $K \in \K_{n \red k}$ we can write
$\up_k^n \down_k^n K = \id_{n,n-1} \diam (\up_k^{n-1} \down_k^{n-1} \down_{n-1}^n K)$,
which induces a natural equivalence $\id_{(n,n-1)} \diam \xi^{(n-1) \red k}: {\up_k^n}
\circ {\down_k^n} \simeq {\up_{n-1}^n} \circ {\down_{n-1}^n}$. Then, by composing with
$\xi^{(n,n-1)}: {\up_{n-1}^n} \circ {\down_{n-1}^n} \simeq \id_{\K_{n \red (n-1)}}$, we
get the natural equivalence $\xi^{n \red k}: {\up_k^n} \circ {\down_k^n} \simeq \id_{\K_{n
\red k}}$.
\end{proof}

\newpage

\section{Labeled ribbon surface tangles%
\label{surfaces/sec}}

In Chapter \ref{cobordisms/sec} we have seen how to represent 4-dimensional relative
2-handlebody cobordisms between 3-dimensional 1-handlebodies with $n$ 0-handles up to
2-equiva\-lence by means of $n$-labeled Kirby tangles in $E \times [0,1]$. These tangles,
considered up to isotopy and 2-deformation moves, was encoded as the morphisms of the
categories $\K_n$ with $n \geq 1$, equivalent to the categories of cobordisms 
$\Chb^{3+1}_n$.

The goal of this chapter is to give a different representation of such cobordisms in terms
of $n$-labeled ribbon surface tangles, by describing them as $n$-fold simple coverings of
$E \times [0,1] \times [0,1]$ (cf. Sections \ref{ribbons/sec} and \ref{coverings/sec}).
Such ribbon surface tangles, up to labeled 1-isotopy and ribbon moves (cf. Figure
\ref{coverings05/fig}), will represent the morphisms of the categories $\S_n$ with $n \geq
2$.

The relation between the two representations of 4-dimensional relative 2-han\-dlebody
cobordisms, as Kirby tangles and as labeled ribbon surface tangles, will be established by
the functor $\Theta_n: \S_n \to \K_n$ defined in Section \ref{Theta/sec}. The restriction
$\Theta_n: \S^c_n \to \K^c_n$ to a suitable subcategory $\S^c_n \subset \S_n$,
representing connected hanldebodies, will be proved to be a category equivalence for $n
\geq 4$ in Section \ref{K=S/sec}.

We emphasize that the same does not hold for $n < 4$. Indeed, it is shown in \cite{Mo78}
that all connected 4-dimensional 2-handlebodies are 3-fold (irregular) branched covers of
$B^4$, but only the symmetric ones are 2-fold (regular) branched covers of $B^4$. This
implies that $\Theta_2$ cannot be full. Moreover, the result in \cite{BE78} (cf. also
\cite{Mo85}) implies two 3-fold branched covering of $B^4$ representing the same
4-dimensional 2-handlebody are not necessarily related through 1-isotopy and the ribbon
move \(R1) (note that move \(R2) does not appear when $n \leq 3$). Therefore $\Theta_3$,
as well as $\Theta_2$, cannot be faithful.

\subsection{The category $\S$ of ribbon surface tangles%
\label{rs-tangles/sec}}

We define the category $\S$ of ribbon surface tangles as follows. An object of $\S$ is any
finite (possibly empty) trivial family $A$ of regularly embedded arcs in $\Int E \times
\left[0,1\right[\,$, with $E = [0,1]^2$. Given two objects $A_0,A_1 \in \Obj \S$, a
morphism of $\S$ with source $A_0$ and target $A_1$ is a ribbon surface tangle $S \subset
E \times [0,1] \times \left[0,1\right[\,$, considered up to 1-isotopy, such that
$\partial_0 S = i_0(A_0)$ and $\partial_1 S = i_1(A_1)$, where $i_0,i_1: E \times
\left[0,1\right[ \to E \times [0,1]\times \left[0,1\right[$ are the inclusions defined
$i_0(x,y,t) = (x,y,0,t)$ and $i_1(x,y,t) = (x,y,1,t)$ respectively.

The composition of two morphisms $S_1:A_0 \to A_1$ and $S_2: A_1 \to A_2$ in $\S$ is
represented by the ribbon surface tangle obtained by stacking $S_2$ over $S_1$, so that
$\partial_0 S_2$ coincides with $\partial_1 S_1$, and then smoothing the union surface and
rescaling the third coordinate by $1/2$. Then, the identity morphism $\id_A: A \to A$ of
an object $A \in \S$ is represented by the product ribbon surface tangle $\{(x,y,z,t)
\;|\; (x,y,t) \in A \text{ and } z \in [0,1]\} \subset E \times [0,1] \times
\left[0,1\right[\,$.

\medskip

\label{S-standard}
For any $m \geq 0$ we consider the {\sl standard object} $J_m$ of $\S$ to be the (possibly
empty) sequence $J_m = (a_{m,1}, \dots, a_{m,m})$ of regularly embedded arcs in $\Int E
\times \left[0,1\right[\,$ defined as follows. For any $1 \leq k \leq m$, we start with
the interval in $E$
$$
a_{m,k} = [(k - 0.8)/m,(k - 0.2)/m] \times \{0.5\}\,,
$$
then we push its interior in the interior of $\Int E \times \left[0,1\right[\,$ to get
the regularly embedded arc $a_{m,k}$ (we use the same notation $a_{m,k}$ for both the 
interval and the arc). All the arcs forming the $J_m$'s are assumed to be equivalent up to
translations and rescaling of the first coordinate. 

We say that a ribbon surface tangle $S$ has {\sl standard ends} if it represents a
morphism between standard objects, that is $\partial_0 S = i_0(J_{m_0})$ and
$\partial_1 S = i_1(J_{m_1})$ for some $m_0,m_1 \geq 0$. In particular, we denote by 
$\id_m: J_m \to J_m$ the identity morphism of $J_m$ for any $m \geq 0$.

\medskip

We can introduce a strict monoidal structure on the full subcategory of $\S$ with standard
objects, in the following way. On the set of standard objects, we consider the product
$J_m \diam J_{m'} = J_{m + m'}$. Then, given two ribbon surface tangles $S$ and $S'$
with standard ends, we define their product $S \diam S'$, by horizontal juxtaposition of
$S$ and $S'$, followed by a reparame\-trization of the first coordinate depending only on
the third one (that is a diffeomorphism $(x,y,z,t) \mapsto (h^z(x),y,z,t)$ with $h^z$
increasing function of $x$ for every $z$).

\begin{proposition}\label{subcategoryS/thm}
$\S$ is equivalent to its full subcategory whose objects are the standard objects $J_m$
for $m \geq 0$, through the inclusion functor. The product defined above, makes such
subcategory into a strict monoidal category, whose unit is represented by the empty 
tangle $\id_0: J_0 \to J_0$.
\end{proposition}

\begin{proof}
For the first part of the statement, according to Corollary \ref{subcat-equiv/thm}, it
suffices to prove that any object of $A \in \S$ consisting of $m$ regularly embedded arcs
in $\Int E \times \left[0,1\right[$ is isomorphic to the standard object $J_m$. In fact,
an isomorphism from $J_m$ to $A$ is represented by the ribbon surface tangle (with no
ribbon self-intersections) $S = \{(x(p,t), y(p,t), t, z(p,t))\ |\ p \in J_m\,, t \in
[0,1]\}$, with $h_t(p) = (x(p,t), y(p,t), z(p,t))$ any ambient isotopy of $\Int E \times
\left[0,1\right[$ such that $h_1(J_m) = A$.

Concerning the strict monoidal structure, the verification of the required properties is
straightforward, being the unit of the product the empty morphism, the one represented
by the empty ribbon surface tangle.
\end{proof}

{\sl From now on, we will use the notation $\S$ for the strict monoidal category of ribbon
surface tangles with standard ends given by Proposition \ref{subcategoryS/thm}.}

\medskip

The rest of this section will be dedicated to the study of the structure of $\S$. We will
prove that $\S$ is equivalent to a strict monoidal braided category generated by a single
object and a set of elementary morphisms and relations in sense of Definition
\ref{presentation/def}. Moreover, we will show that such category carries also a tortile
structure (cf. Definition \ref{tortile/def}).

\medskip

We start with the elementary diagrams in Figure \ref{ribbon-morph01/fig}, considered up to
isotopy preserving horizontal lines, that is isotopy of the form $((x,y) \mapsto
(h_t^y(x),h_t(y)))_{t \in [0,1]}$ with $h_t^y$ an increasing function of $x$ for every $y$
and $t$, and $h_t$ increasing function of $y$ for every $t$. Here, the elementary diagram 
\(a) represents the identity of $J_1$.

\begin{Figure}[htb]{ribbon-morph01/fig}
{}{Elementary diagrams in the category $\S$}
\centerline{\fig{}{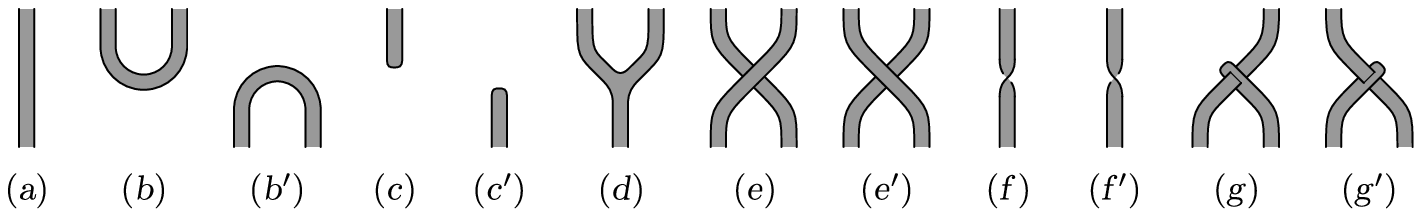}}\vskip-4pt
\end{Figure}

Then, also iterated products/compositions of such elementary diagrams, defined by
horizontal/vertical juxtaposition and rescaling, turn out to be well-defined up to
isotopy preserving horizontal lines.

In this context, the planar isotopy moves in Figure \ref{ribbon-tang01/fig}, where the
boxes $D$ and $D'$ in \(I1) contain any of the elementary diagrams in Figure
\ref{ribbon-morph01/fig}, can be interpreted as relations between iterated
products/compositions of elementary diagrams.

\begin{Figure}[htb]{ribbon-tang01/fig}
{}{Planar isotopy relations in the category $\S$}
\centerline{\fig{}{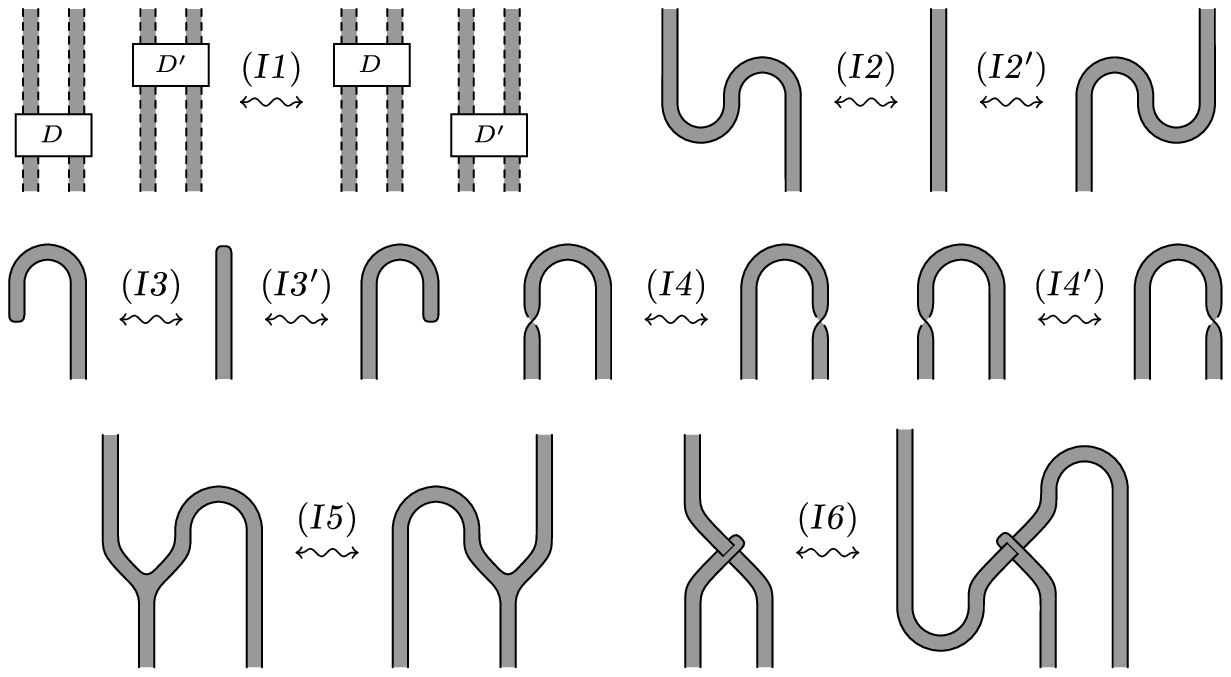}}
\end{Figure}

An analogous interpretation of the moves introduced in Figures \ref{ribbon-surf09/fig},
\ref{ribbon-surf10/fig} and \ref{ribbon-surf11/fig} to realize 3-dimensional diagram
isotopy of ribbon surface tangles, and of the 1-isotopy moves in Figure
\ref{ribbon-surf13/fig}, leads to the relations depicted in Figures
\ref{ribbon-tang02/fig}, \ref{ribbon-tang03/fig} and \ref{ribbon-tang04/fig}, where the
box $D$ in \(I8) and \(I9) contains any of the elementary diagrams in Figure
\ref{ribbon-morph01/fig}. In particular, relations \(I7-7') and \(I12-12') say that \(e')
and \(f') are the inverses of \(e) and \(f) with respect to the composition.

\begin{Figure}[htb]{ribbon-tang02/fig}
{}{3-dimensional isotopy relations in the category $\S$}
\centerline{\fig{}{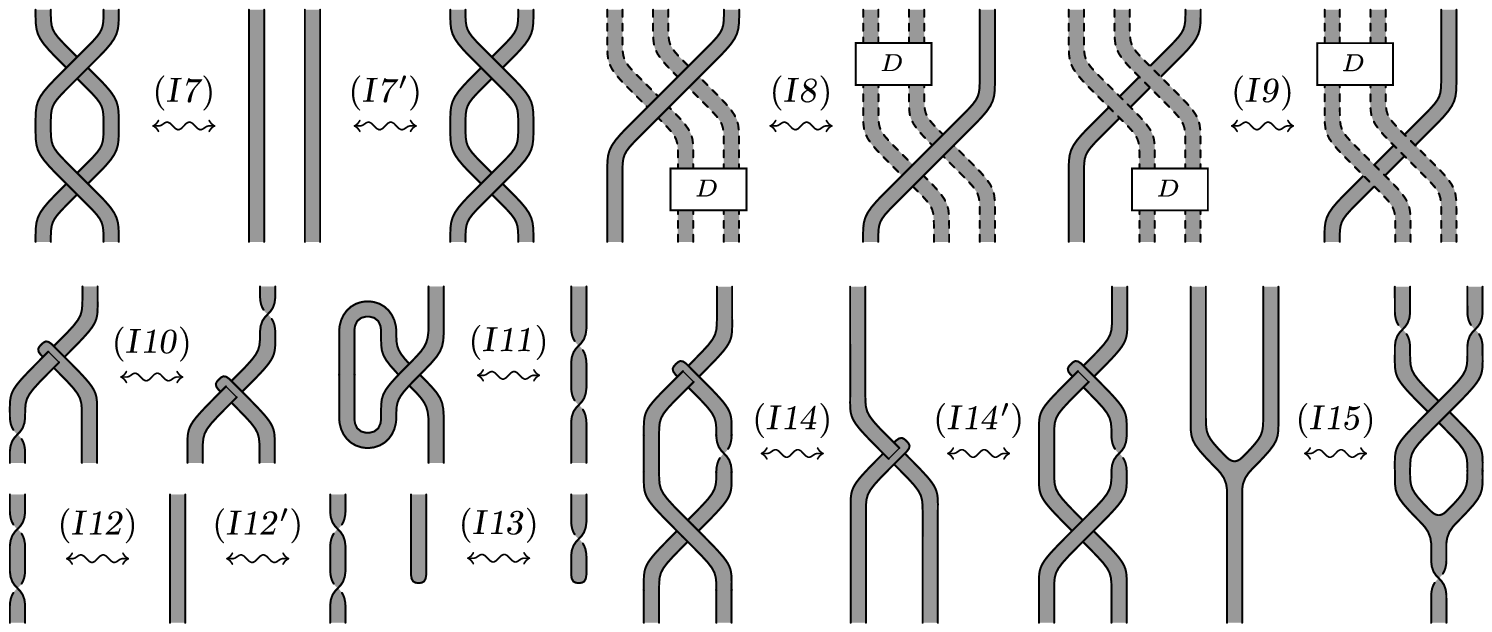}}
\end{Figure}

\begin{Figure}[htb]{ribbon-tang03/fig}
{}{Graph changing relations in the category $\S$}
\centerline{\fig{}{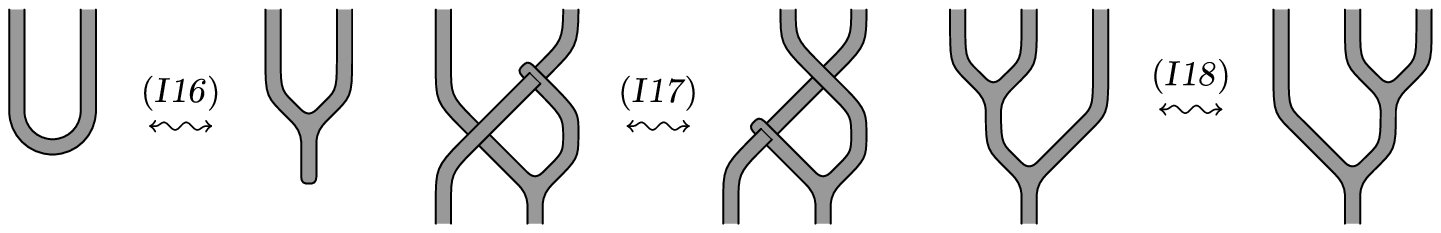}}
\end{Figure}

\begin{Figure}[htb]{ribbon-tang04/fig}
{}{1-isotopy relations in the category $\S$}
\centerline{\fig{}{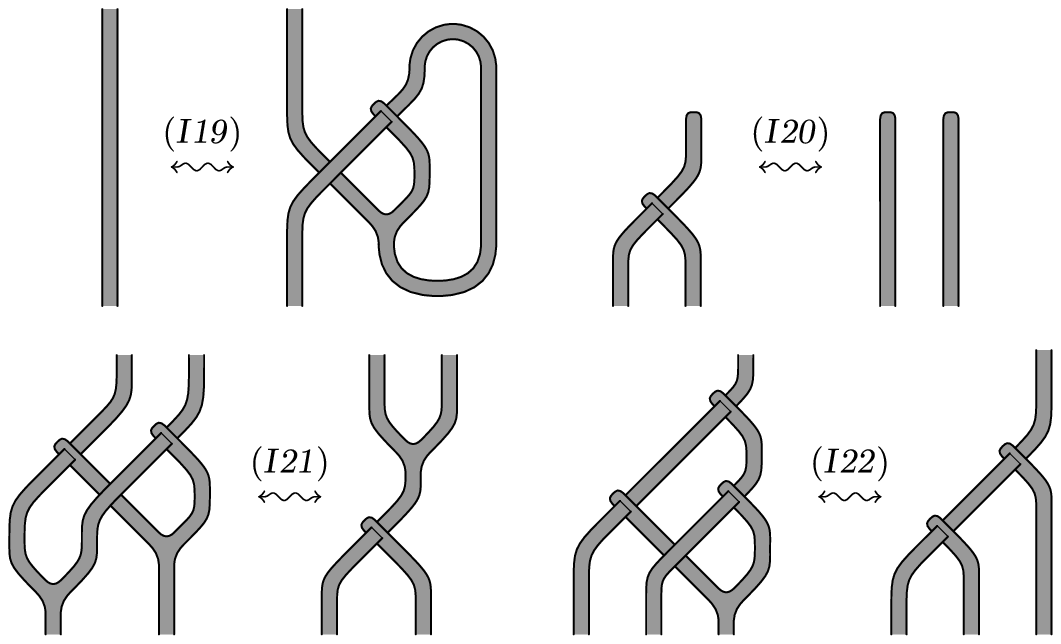}}
\end{Figure}

\pagebreak

Now, consider the strict monoidal category $C(J_1,E,R)$ generated by the single object
$J_1$ and the set $E$ of elementary morphisms in Figure \ref{ribbon-morph01/fig} modulo
the set $R$ of elementary relations \(I1) to \(I22). Here, the relations \(I1) could be
omitted, since it is universally valid as requested in point \(c) of Proposition
\ref{presentation/thm}. 

Remember from Section \ref{categories/sec} that $\Obj C(J_1,E,R)$ is the free monoid
$\seq J_1$ generated by $J_1$, which can be identified with the set of standard objects
$J_m = J_1 \diam \dots \diam J_1$ for $m \geq 0$. On the other hand, each iterated
product/composition of elements of $E$ can be expressed as a composition of expansions
(that is products of with identities) of elements of $E$, that is a morphism of the free
category $F(J_1,E)$, which turns out to be uniquely determined up to \(I1) (and its
expansions). Then, all the relations \(I2) to \(I22) can be seen as relations in the free
category $F(J_1,E)$, hence as defining relations for $C(J_1,E,R)$ according to Definition
\ref{presentation/def}.

\begin{proposition}\label{S-category/thm}
$\S$ is equivalent to $C(J_1,E,R)$ as a strict monoidal category, through a monoidal 
equivalence functor $C(J_1,E,R) \to \S$, which is the identity on the objects and send 
each morphism of $C(J_1,E,R)$ represented by a given diagram to the morphism of $\S$ 
represented by the corresponding ribbon surface tangle.
\end{proposition}

\begin{proof}
We consider the map $C(J_1,E,R) \to \S$ defined by formal propagation over products and
compositions of the requirement that elementary diagrams are sent to the corresponding
ribbon surface tangle. In other words, any planar diagram given as iterated
product/composition of elementary ones is sent to the corresponding ribbon surface tangle
as well. Since any morphism in $C(J_1,E,R)$ is represented by a composition of expansions
of elements of $E$, in order to have a well-defined map on the level of morphisms, it
suffices to observe the images of the relations in $R$ are all 1-isotopy moves between the
corresponding ribbon surface tangles. Indeed, the relations in Figure
\ref{ribbon-tang01/fig} represent planar diagram isotopies, the ones in Figures
\ref{ribbon-tang02/fig} correspond to 3-dimensional diagram isotopies, while those in
Figures \ref{ribbon-tang03/fig} and \ref{ribbon-tang04/fig} can be realized by the moves
in Figures \ref{ribbon-surf09/fig} and \ref{ribbon-surf13/fig} respectively.

Therefore, the map $C(J_1,E,R) \to \S$ defined as the identity on the objects and as
described above on the morphisms is a strict monoidal functor. It remains to see that this 
functor is full and faithful, hence an equivalence of categories.

The fullness simply means that any morphism of $\S$ can be presented as a composition of
expansions of elementary diagrams. First of all, we observe that such a morphism is
assumed to have standard ends (after Proposition \ref{subcategoryS/thm}), hence it can be
represented by a special planar diagram thanks to Proposition \ref{planar-diagram/thm},
since standard ends are flat.

We say that a special planar diagram of a ribbon surface tangle $S$ is in {\sl normal
position} with respect to the $y$-axis if is satisfies the following
properties:
\begin{itemize}\itemsep0pt
\item[\(a)]
each edge of the core graph $G$ projects to a regular smooth arc immersed in
$\left]0,1\right[ \times [0,1]$, such that the $y$-coordinate restricts to a Morse
function on it;
\item[\(b)] 
vertices, half-twists, crossings and local minimum/maximum points for the $y$-coordinate
along the edges of the core graph $G$ have all different $y$-coordinates (in particular,
there are no horizontal tangencies at vertices, half-twists and crossings).
\end{itemize}

We observe that all the elementary diagrams in Figure \ref{ribbon-morph01/fig} are in
normal position with respect to the $y$-axis, hence this is also true for any 
composition of expansions of them.

Figure \ref{ribbon-tang05/fig} shows the different ways, up to planar isotopy preserving
horizontal lines, to put the spots \(a) to \(e) and \(h) of Figure \ref{ribbon-surf06/fig}
in normal position with respect to the $y$-axis, by planar diagram isotopies which do not
introduce any local minimum/maximum for the $y$-coordinate along the edges of the core
graph.

\begin{Figure}[htb]{ribbon-tang05/fig}
{}{Local models for planar diagrams in normal position}
\vskip3pt
\centerline{\fig{}{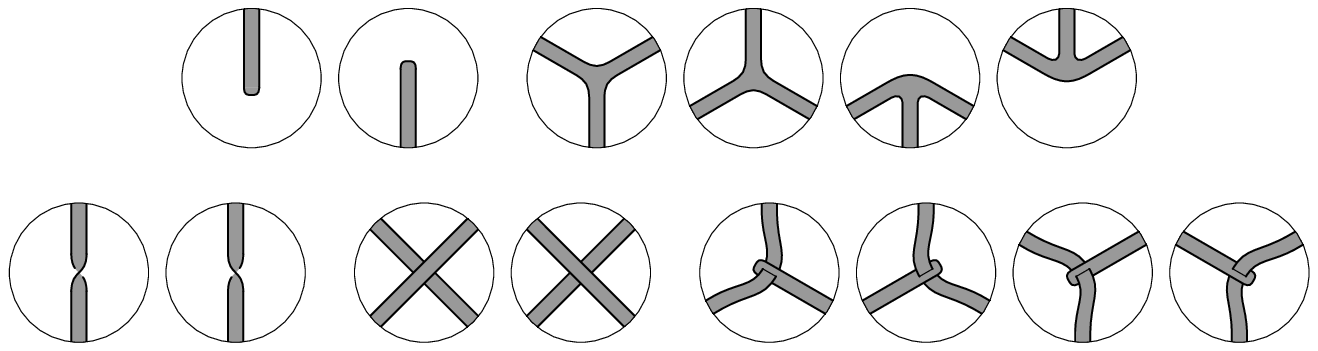}}
\end{Figure}

All such local configurations appear among the elementary diagrams in Figure
\ref{ribbon-morph01/fig}, except for some of those at 3-valent vertices of the core graph.
Namely, only the first of those at a flat 3-valent vertex and the first two of those at a
singular vertex are considered as elementary diagrams. Up to planar isotopy, the others 
can be expressed in terms of them as in Figure \ref{ribbon-tang06/fig}.

\begin{Figure}[htb]{ribbon-tang06/fig}
{}{Expressing local models in terms of elementary diagrams}
\centerline{\fig{}{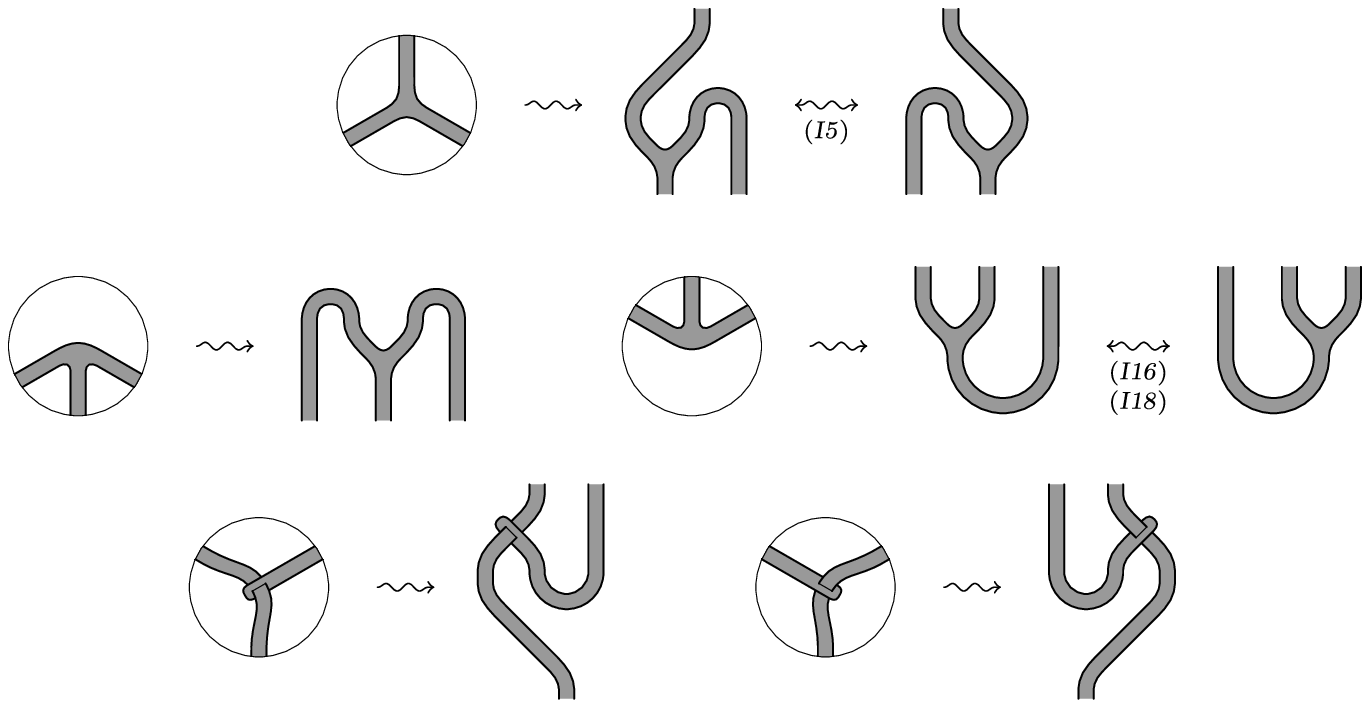}}
\end{Figure}

Now, let $S$ be any ribbon surface tangle with standard ends, represented by a special
planar diagram, as said above. Then, we can perturb such diagram to get normal position
with respect to the $y$-axis. Finally, the local isotopies described in Figure
\ref{ribbon-tang06/fig} can be applied to obtain a presentation of $S$ as composition of 
expansions of elementary diagrams. This completes the proof of the fullness.

In order to see that the functor is faithful, we have to prove that any two of $S$ as
composition of expansions of elementary diagrams are equivalent up to plane isotopy
preserving horizontal lines and the relations \(I1) to \(I22). In other words, that these
relations allow us to interpret any 1-isotopy of ribbon surface tangles with standard
ends.

We first focus on planar isotopy, with an argument analogous to that of Proposition
\ref{isotopy/thm}. Assume we are given an arbitrary smooth planar isotopy relating any two
given presentations of $S$ as above. Denote by $\bar G$ the planar graph associated to the
planar diagram of the core graph $G$ of $S$, whose vertices, other than the projections of
the vertices of $G$ and the 4-valent vertices at the crossings, also include 2-valent
vertices at the half-twists.

We use transversality to perturb the given planar isotopy in such a way there is only a
finite number of critical levels where $\bar G$ presents exactly one of the following:
\begin{itemize}\itemsep0pt
\item[1)]
the $y$-coordinate on one edge is not a Morse function;
\item[2)]
there is a horizontal tangent line at one of the vertices (including half-twists and 
crossings) ;
\item[3)] 
two points among the extremal ones along edges and the vertices (including half-twists and
crossings) have the same $y$-coordinate.
\end{itemize}

Away from these critical levels the diagram is in normal position with respect to the 
$y$-axis and the isotopy can be assumed to preserve horizontal lines. 

We will show that the relations in Figures 3.1.2 e 3.1.3 suffice to realize all changes 
occurring in the diagram when passing through one critical level.

The cases of critical levels of types 1 and 3 are respectively covered by relations
\(I2-2') and \(I1). At critical levels of type 2 the vertex with horizontal tangency is
switching from one to another of its normal positions depicted in Figure
\ref{ribbon-tang05/fig}. At the same time, one extremal point (resp. one pair of canceling
extremal points) for the $y$-coordinate is appearing/disappearing along the edge (resp.
the opposite edges) presenting the horizontal tangency. The cases when the vertex we are
considering is a uni-valent flat vertex or a half-twists correspond respectively to
relations \(I3-3') or \(I4-4'), modulo \(I1) and \(I2-2'). If the vertex is a crossing,
there are four symmetric possibilities and Figure \ref{ribbon-tang07/fig} shows how to
realize one them. The other three can be realized in a similar way, by using relations
\(I7-7'), \(I8) and \(I9). In order to deal with tri-valent vertices, we need to replace
the normal positions that are missing in the elementary diagrams as indicated in Figure
\ref{ribbon-tang06/fig}. After that, modulo \(I1) and \(I2-2'), all the cases reduce to
relation \(I6) and to the modification described in Figure \ref{ribbon-tang08/fig} for
singular vertices, and to relation \(I5) for flat vertices.

\begin{Figure}[htb]{ribbon-tang07/fig}
{}{Switching a crossing}
\centerline{\fig{}{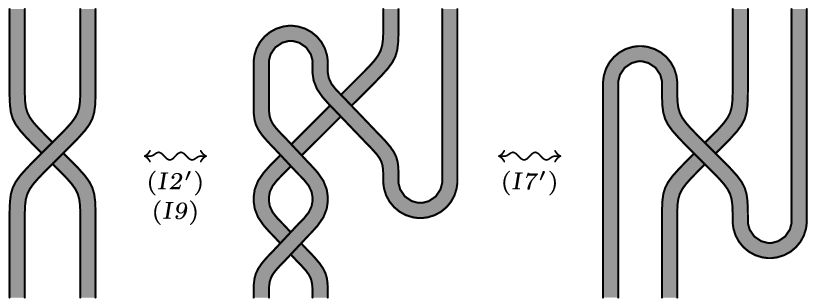}}
\end{Figure}

\begin{Figure}[htb]{ribbon-tang08/fig}
{}{Switching a ribbon intersection}
\centerline{\fig{}{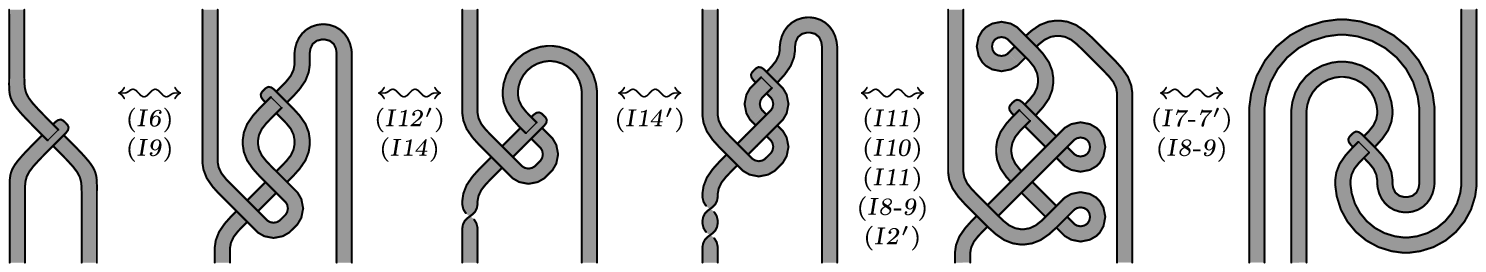}}
\end{Figure}

Now we pass to 3-dimensional diagram isotopy. By Proposition \ref{isotopy/thm}, we have to
interpret in terms of relations the moves \(S3) to \(S22) in Figures
\ref{ribbon-surf09/fig}, \ref{ribbon-surf10/fig} and \ref{ribbon-surf11/fig}. Since those
moves are defined up to planar isotopy, such interpretation is quite immediate. In
particular, we have that: the relations in Figure \ref{ribbon-tang03/fig} give the moves
in Figure \ref{ribbon-surf09/fig}; the relations in the first line of Figure
\ref{ribbon-tang02/fig} give the moves in Figure \ref{ribbon-surf10/fig} and \(S15-16);
the remaining relations in Figure \ref{ribbon-tang02/fig} give the moves \(S14) and \(S17)
to \(S22) in Figure \ref{ribbon-surf11/fig}, modulo the previous moves in the case of
\(S19).

Finally, concerning 1-isotopy moves in \ref{ribbon-surf13/fig} (cf. Proposition
\ref{1-isotopy/thm}), it is enough to observe that, after after expressing them in terms
of special planar diagrams by using move \(S1), they essentially correspond to the
relations in Figure \ref{ribbon-tang04/fig}.
\end{proof}

We can use the elementary diagrams in Figure \ref{ribbon-morph01/fig} to define a tortile
structure on the category $\S$ (cf. Proposition \ref{S-tortile/thm} below). In particular,
we make the following natural choices: \(f) for the braiding isomorphism
$\gamma_{J_1,J_1}$; \(b) and \(b') for the form and coform morphisms $\Lambda_{J_1}$ and
$\lambda_{J_1}$ respectively; the composition $(\lambda_{J_1}\diam \id_{J_1})\circ
(\id_{J_1}\diam \gamma_{J_1,J_1})\circ (\Lambda_{J_1}\diam \id_{J_1})$ for the twist
isomorphism $\theta_{J_1}$ (see Figure \ref{ribbon-tang09/fig}).

\begin{Figure}[htb]{ribbon-tang09/fig}
{}{The twist isomorphism $\theta_1$}
\centerline{\fig{}{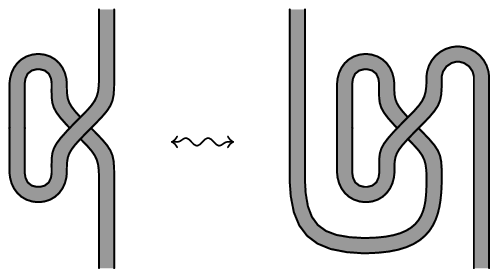}}
\end{Figure}

\medskip

{\sl From now on we will write only $m$ instead of $J_m$ in the subscripts of the
notation for the morphisms of $\S$, to keep that notation as simple as possible.} For
example $\gamma_{1,1}$, $\lambda_1$, $\Lambda_1$ and $\theta_1$ will denote the morphisms
just mentioned above.

\begin{proposition}\label{S-tortile/thm}
$\S$ is a tortile category, with the tortile structure uniquely determined by $J_m^\ast =
J_m$ for any $m \geq 0$, and by the choices above for $\gamma_{1,1}$, $\Lambda_1$,
$\lambda_1$ and $\theta_1$ and by the following recursive formulas for $m > 1$:
$$\Lambda_m = (\id_1 \diam \Lambda_{m-1} \diam \id_1) \circ \Lambda_1 \quad \text{and}
\quad \lambda_m = \lambda_1 \circ (\id_1 \diam \lambda_{m-1} \diam \id_1)\,.$$
\vskip-18pt
\end{proposition}

\begin{proof}
The general relations for the braiding isomorphisms (cf. Definition \ref{braided-cat/def})
imply that $\gamma_{0,m} = \gamma_{m,0} = \id_m$ for any $m \geq 0$ and that the recursive
formula $\gamma_{m,m'} \!= (\id_{m'-1} \!\diam \gamma_{m,1}) \circ (\gamma_{m, m'-1}
\!\diam \id_1) = (\gamma_{m-1, m'} \!\diam \id_1) \circ (\id_{m-1} \!\diam \gamma_{1,
m'})$ holds for any $m, m' \geq 1$. Analogously, for the twist isomorphisms (cf. Section
\ref{categories/sec}) we have $\theta_0 \!= \id_0$ and $\theta_m = \gamma_{1,m-1} \!\circ
(\theta_{m-1} \!\diam \theta_1) \circ \gamma_{m-1,1}$ for any $m \geq 1$. Then, taking
also into account the recursive formulas in the statement, the definitions of
$\gamma_{m,m'}$, $\Lambda_m$, $\lambda_m$ and $\theta_m$ for any $m, m' \geq 0$, are
uniquely determined by the ones of $\gamma_{1,1}$, $\Lambda_1$, $\lambda_1$ and
$\theta_1$.

The verification that the families of morphisms we get in this way from our start\-ing
choices give a true tortile structure is straightforward. For $m = 1$, the defining
properties of forms and coforms reduce to relations \(I1) and \(I2-2'), while the
self-duality of the twist $\theta_m$ follows from relations \(I2), \(I4) and \(I11)
(cf. Figure \ref{ribbon-tang09/fig}). Then, we can proceed by induction on $m$, using
moves \(I7-7'), \(I8) and \(I9).
\end{proof}

\subsection{The categories $\S_n$ and the functors $\up_k^n$%
\label{S/sec}}

As we said at the beginning of the chapter, we want to consider $n$-fold simple coverings 
of $E \times [0,1] \times [0,1]$, branched over ribbon surface tangles. According to 
Section \ref{coverings/sec}, these can be described in terms of ribbon surface tangles 
labeled by transpositions in the permutation group $\Sigma_n$.

Here, we construct the categories $\S_n$ of such labeled ribbon surface tangles up to
certain moves preserving the diffeomorphism type of the covering space for $n \geq 2$.
Moreover, we define stabilization functors $\up_k^n: \S_k \to \S_n$ relating them for any
$n > k \geq 2$.

\medskip

Let $\Gamma_n \subset \Sigma_n$ be the set of all transpositions in the permutation group
$\Sigma_n$ and $\seq\Gamma_n = \cup_{m=0}^\infty \Gamma_n^m$ be the set of (possibly
empty) finite sequences of elements of $\Gamma_n$. For any sequence $\sigma =
(\tp{i_1}{j_1}, \dots, \tp{i_m}{j_m}) \in \seq\Gamma_n$, we denote by $J_\sigma$ the
sequence of intervals $J_m = (a_{m,1}, \dots,a_{m,m})$ (the standard object of $\S$
defined at page \pageref{S-standard}) with each $a_{m,k}$ labeled by the corresponding
transposition $\tp{i_k}{j_k}$.

\begin{definition}\label{lrs-tangle/def}
Let $n \geq 2$ and $\sigma_0 = (\tp{i_1^0}{j_1^0}, \dots, \tp{i_{m_0}^0}{j_{m_0}^0})$ and
$\sigma_1 = (\tp{i_1^1}{j_1^1}, \dots, \tp{i_{m_1}^1}{j_{m_1}^1})$ be two sequences 
in $\seq\Gamma_n$. By an {\sl $n$-labeled ribbon surface tangle} from $J_{\sigma_0}$ 
to $J_{\sigma_1}$ we mean a ribbon surface tangle $S$ from $J_{m_0}$ to $J_{m_1}$ with a
$\Gamma_n$-label\-ing satisfying the following properties:
\begin{itemize}
\item[1)]
each region in the diagram of $S$ is labeled by a transposition in $\Gamma_n$ associated
to the corresponding meridian, in such a way that the Wirtinger relations in $\pi_1(E 
\times [0,1] \times [0,1] - S)$ are respected (cf. Section \ref{coverings/sec});
\item[2)]
the labels on $\partial_0 S$ and $\partial_1 S$ coincide with those in $J_{\sigma_0}$ and
$J_{\sigma_1}$ respectively.
\end{itemize}\vskip-\lastskip
\vskip-12pt
\end{definition}

We notice that property 1 is the same as requiring that the labeling represents the
monodromy homomorphism $\omega_p: \pi_1(E \times [0,1] \times [0,1] - S) \to \Sigma_n$ of
an $n$-fold simple covering of $p: W \to E \times [0,1] \times [0,1]$ branched over $S$
(cf. Section \ref{coverings/sec}). By \cite{Mo78} we know that $W$ a is relative
4-dimensional 2-handlebody cobordism between the 3-dimensional 1-handlebodies $M_0 =
p^{-1}(E \times \{0\} \times [0,1])$ and $M_1 = p^{-1}(E \times \{1\} \times [0,1])$.
Moreover, as discussed in Section \ref{coverings/sec}, labeled isotopy and the moves \(R1)
and \(R2) in Figure \ref{ribbon-moves01/fig} (cf. Figure \ref{coverings05/fig}) do not
change $W$ up to diffeomorphism.

\begin{Figure}[htb]{ribbon-moves01/fig}
{}{Ribbon moves ($i,j,k$ and $l$ all different)}
\vskip3pt
\centerline{\fig{}{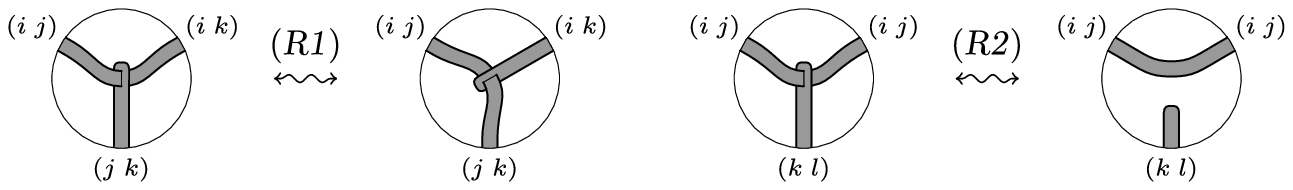}}
\vskip-3pt
\end{Figure}

Actually, in Section \ref{Theta/sec} we will also see that the handlebody structures
of $M_0$ and $M_1$ are uniquely determined, while that of $W$ is determined only up to
2-equivalence, depending on the choice of an adapted 1-handlebody decomposition of $S$. We
do not know whether labeled isotopy preserves the 2-equivalence class of $W$, but we will
show in Proposition \ref{theta/thm} that labeled 1-isotopy does and that the same holds 
for moves \(R1) and \(R2). This is the motivation for the following definition.

\begin{definition}\label{lrs-equivalence/def}
Two $n$-labeled ribbon surface tangles are said to be {\sl equivalent} if they are related
by the labeled version of the 1-isotopy moves \(S1) to \(S26) in Figures
\ref{ribbon-surf08/fig}, \ref{ribbon-surf09/fig}, \ref{ribbon-surf10/fig},
\ref{ribbon-surf11/fig} and \ref{ribbon-surf13/fig}, and by the two covering moves \(R1)
and \(R2) in Figure \ref{ribbon-moves01/fig}.
\end{definition}

\begin{Figure}[b]{ribbon-moves02/fig}
{}{Other moves for labeled ribbon surfaces ($i,j,k$ and $l$ all different)}
\centerline{\fig{}{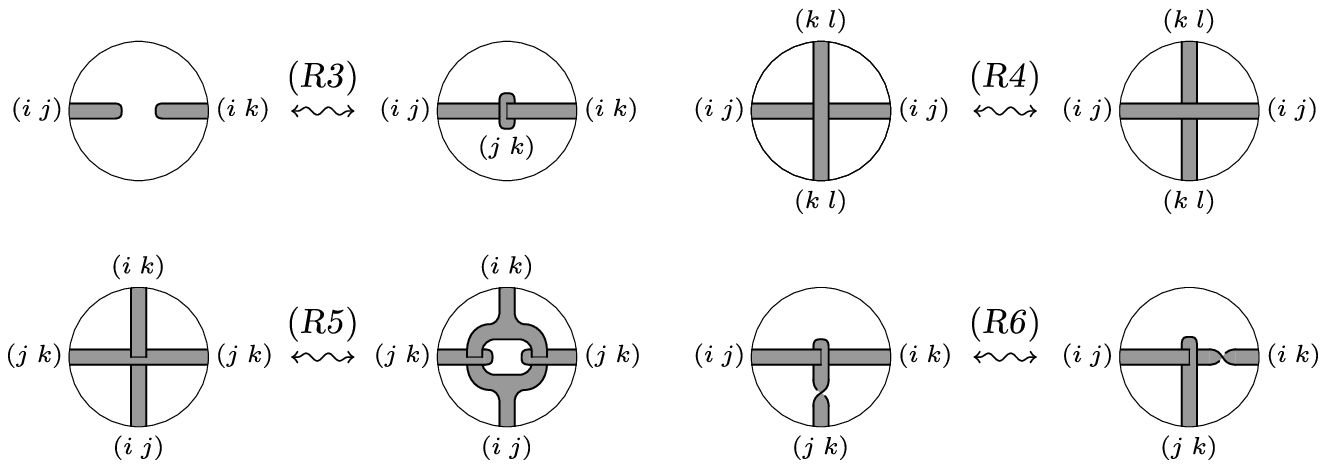}}
\vskip-3pt
\end{Figure}

Before going on, we recall the auxiliary moves \(R3) to \(R6) in Figure
\ref{ribbon-moves02/fig}, which were introduced in Section \ref{coverings/sec} (cf. Figure
\ref{coverings07/fig}). These have been shown to derive from the equivalence moves \(R1)
and \(R2) up to 1-isotopy (cf. Proposition \ref{moves-aux/thm}), hence they are
equivalence moves for labeled ribbon surface tangles as well. Once we will prove that
\(R1) and \(R2) preserve the 2-equivalence class of the covering handlebody $W$, we will
know that the same holds for the moves \(R3) to \(R6).

At this point, we can define the category $\S_n$ of $n$-labeled ribbon surface tangles for
any $n \geq 2$, in the following way. The objects of $\S_n$ are the sequences of labeled
intervals $J_\sigma$ with $\sigma = (\tp{i_1}{j_1}, \dots, \tp{i_m}{j_m}) \in
\seq\Gamma_n$ (cf. the notation introduced before Definition \ref{lrs-tangle/def}), while
the morphisms in $\S_n$ with source $J_{\sigma_0}$ and target $J_{\sigma_1}$ are the
$n$-labeled ribbon surface tangles from $J_{\sigma_0}$ to $J_{\sigma_1}$, considered up
equivalence in the sense of Definition \ref{lrs-equivalence/def}. The composition of
morphisms in $\S_n$ is just the labeled version of that in $\S$. Similarly, we define a
strict monoidal structure on $\S_n$ having as the product $\diam$ the labeled version
of that of $\S$. In particular, $J_\sigma \diam J_{\sigma'} = J_{\sigma \diam \sigma'}$ 
for every $\sigma,\sigma' \in \seq\Gamma_n$.

The next Propositions \ref{Sn-category/thm} and \ref{Sn-tortile/thm} generalize
Propositions \ref{S-category/thm} and \ref{S-tortile/thm} to the categories of labeled
ribbon surface tangles $\S_n$ with $n \geq 2$. 

In Propositions \ref{Sn-category/thm}, $E_n$ will denote the set of the $n$-labeled
versions of the elementary diagrams in Figure \ref{ribbon-morph01/fig}, while $R_n$ will
denote the set consisting of the $n$-labeled versions of relations \(I1) to \(I22) in
Figures \ref{ribbon-tang01/fig}, \ref{ribbon-tang02/fig}, \ref{ribbon-tang03/fig} and
\ref{ribbon-tang04/fig} and of the relations \(R1) and \(R2) in Figure
\ref{ribbon-moves03/fig}.

\begin{Figure}[htb]{ribbon-moves03/fig}
{}{Covering relations in $\S_n$ ($i,j,k$ and $l$ all different)}
\centerline{\fig{}{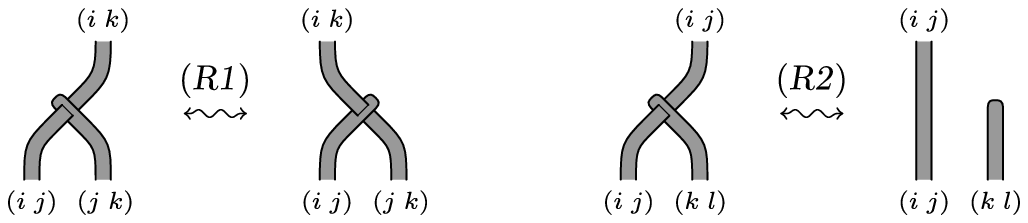}}
\vskip-6pt
\end{Figure}

\begin{proposition}\label{Sn-category/thm}
For any $n \geq 2$, the strict monoidal category $\S_n$ is equivalent to the strict
monoidal category $C(\{J_\tau\}_{\tau \in \Gamma_n}, E_n, R_n)$ generated by the objects
$J_\tau$ with $\tau \in \Gamma_n$ and the set of elementary morphisms $E_n$ modulo the
relations $R_n$, where $E_n$ and $R_n$ are the sets defined above. The equivalence is
given by a monoidal functor $C(\{J_\tau\}_{\tau \in \Gamma_n}, E_n, R_n) \to
\S_n$, which is the identity on the objects and send each morphism of $C(\{J_\tau\}_{\tau
\in \Gamma_n}, E_n, R_n) \to \S_n$ represented by a given $n$-labeled diagram to the
morphism of $\S$ represented by the corresponding $n$-labeled ribbon surface tangle.
\end{proposition}

\begin{proof}
We observe that the relations \(R1) and \(R2) in Figure \ref{ribbon-moves03/fig} express
the homonymous covering moves of Figure \ref{ribbon-moves01/fig} in terms of $n$-labeled
elementary diagrams (and their expansions). Then, the statement follows immediately from
Proposition \ref{S-category/thm} and Definition \ref{lrs-equivalence/def}.
\end{proof}

\begin{Figure}[b]{ribbon-morph02/fig}
{}{Some elementary morphisms in $\S_n$}
\centerline{\fig{}{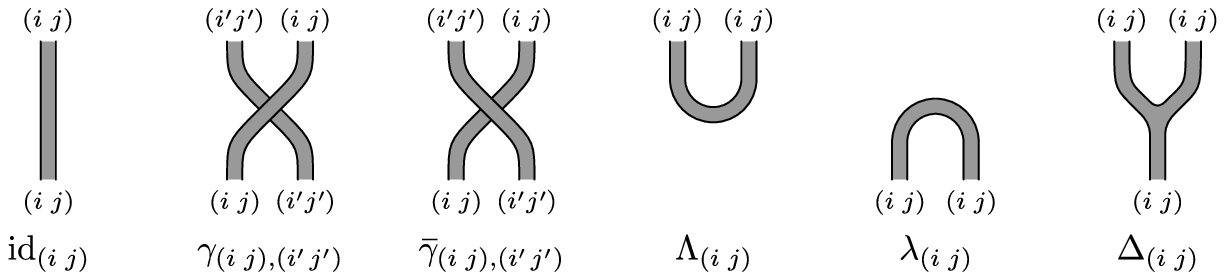}}
\vskip-6pt
\end{Figure}

{\sl Analogously to what we did for $\S$, we will write $\sigma$ instead of $J_\sigma$ in
the subscripts of the notation for the morphisms of $\S_n$.} Adopting this convention,
Figure \ref{ribbon-morph02/fig} shows some elementary morphisms of $\S_n$.

\begin{proposition}\label{Sn-tortile/thm}
For any $n \geq 2$, $\S_n$ is a braided category with tortile structure, where for
$\sigma$ and $\sigma'$ sequences in $\seq\Gamma_n$ of length $m$ and $m'$ respectively:
the braiding morphism $\gamma_{\sigma,\sigma'}$ is the braiding morphism $\gamma_{m,m'}$
of $\S$ labeled according to $\sigma$ and $\sigma'$ (see Figure \ref{ribbon-morph02/fig}
and Figure \ref{ribbon-morph03/fig}, and note that $\gamma_{\sigma,\sigma'}^{-1} =
\bar\gamma_{\sigma', \sigma}$); the dual object of $J_\sigma$ is $J_\sigma^{\,\ast} =
J_{\sigma^\ast}$, with $\sigma^\ast$ the sequence obtained by reversing the order of
$\sigma$; the coform, form and twist morphisms $\Lambda_\sigma$, $\lambda_\sigma$ and
$\theta_\sigma$ are the homologous morphisms $\Lambda_m$, $\lambda_m$ and $\theta_m$ of
$\S$ labeled according to $\sigma$ (cf. Figure \ref{ribbon-morph02/fig}).
\end{proposition}

\begin{proof}
This is an immediate consequence of Proposition \ref{S-tortile/thm}, once we take into
account the labeling.
\end{proof}

\begin{Figure}[htb]{ribbon-morph03/fig}
{}{The braiding isomorphisms in $\S_n$}
\vskip-6pt
\centerline{\fig{}{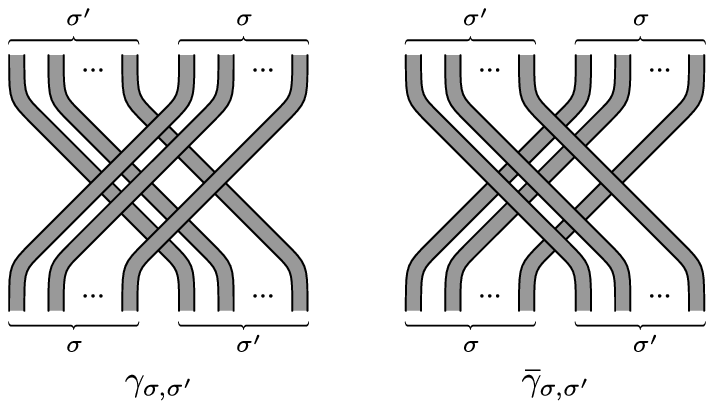}}
\vskip-6pt
\end{Figure}

Now, for any $n > k \geq 2$ we define the stabilization functor $\up_k^n: \S_k \to \S_n$, 
which is analogous to the functor $\up_k^n: \K_k \to \K_n$ introduced in Definition
\ref{K-stabilization/def}. Similarly to the case of Kirby tangles, we denote by
$\iota_k^n: \S_k \to \S_n$ the faithful functor induced by the inclusion $\Gamma_k
\subset \Gamma_n$, which considers any $k$-labeled ribbon surface tangle as an
$n$-labeled one.

\begin{definition}\label{S-stabilization/def}
Given $n > k \geq 1$, let $\sigma_{n \red k} = (\tp{n}{n{-}1}, \dots, \tp{k{+}1}{k})$
and let $\id_{\sigma_{n \red k}}\!$ be the identity morphism of $J_{\sigma_{n \red k}}\!$ 
in $\S_n$. Then, for any $n > k \geq 2$ the {\sl stabilization functor} $\up_k^n: 
\S_k \to \S_n$ is defined by:
$$
\begin{array}{c}
\up_k^n J_\sigma = J_{\sigma_{n \red k}} \!\diam \iota_k^n(J_\sigma) 
\,\text{ for any } J_\sigma \in \Obj \S_k\,,\\[6pt]
\up_k^n S = \id_{\sigma_{n \red k}} \!\diam \iota_k^n(S) 
\,\text{ for any } S \in \Mor \S_k\,.
\end{array}$$\vskip-\lastskip
\end{definition}\vskip-\lastskip

\begin{Figure}[b]{ribbon-morph04/fig}
{}{The comultiplication morphism $\Delta_\sigma\,$ in $\S_n$}
\vskip-3pt
\centerline{\fig{}{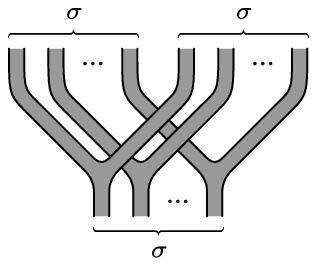}}
\vskip-6pt
\end{Figure}

For any transposition $\tp{i}{j} \in \Gamma_n$, let $\Delta_{\tp{i}{j}}: J_{\tp{i}{j}} \to
J_{\tp{i}{j}} \diam J_{\tp{i}{j}}$ be the labeled ribbon surface tangle presented in
Figure \ref{ribbon-morph02/fig}. As shown in Figure \ref{ribbon-morph04/fig}, this
definition extends inductively to $\Delta_\sigma : J_\sigma \to J_\sigma \diam J_\sigma$
for any sequence $\sigma \in \seq\Gamma_n$, by putting
$$
\Delta_\sigma = \Delta_{\sigma' \diam \sigma''} = (\id_{\sigma'} \diam 
\gamma_{\sigma',\sigma''} \diam \id_{\sigma''}) \circ (\Delta_{\sigma'} \diam 
\Delta_{\sigma''})\,,
$$
which turns out to not depend on the decomposition $\sigma = \sigma' \diam \sigma''$ with 
$\sigma',\sigma'' \in \seq\Gamma_n$.

Up to labeled 1-isotopy and move \(I18) in Figure \ref{ribbon-tang03/fig}, $\Delta_\sigma$
satisfies the coassociativity property, that is:
$$
(\Delta_\sigma \diam \id_\sigma) \circ \Delta_\sigma =
(\id_\sigma \diam \Delta_\sigma) \circ \Delta_\sigma\,.
$$

\medskip

\begin{definition}\label{S-reducible/def}
Given $n > k \geq 1$, we say that a labeled ribbon surface tangle $S \in \S_n$ from
$J_{\sigma_{n \red k}} \!\diam J_{\sigma_0}$ to $J_{\sigma_{n \red k}} \!\diam 
J_{\sigma_1}$ is {\sl $k$-reducible} if it is of the form
$$S = (\id_{\sigma_{n \red k}} \!\diam T) \circ 
(\Delta_{\sigma_{n \red k}} \!\diam \id_{\sigma_0})\,,$$
for some $T: J_{\sigma_{n \red k}} \!\diam J_{\sigma_0} \to J_{\sigma_1} \in \S_n$ (see
Figure \ref{ribbon-stab01/fig}). We will refer to the vertical ribbons forming
$\id_{\sigma_{n \red k}}$ as the {\sl reduction ribbons} of $S$.

The composition of two $k$-reducible labeled tangles is still $k$-reducible (by
coassociativity) and we denote by $\S_{n \red k}$ the subcategory of $\S_n$, whose
objects are $J_{\sigma_{n \red k}} \!\diam J_\sigma$ with $\sigma \in \seq\Gamma_n$ and
whose morphisms are $k$-reducible labeled tangles.\break In particular, for any $n \geq
2$, we denote by $\S_n^c$ the subcategory $\S_{n \red 1}$ of $1$-reducible labeled tangles
in $\S_n$.
\end{definition}

\begin{Figure}[htb]{ribbon-stab01/fig}
{}{The generic $k$-reducible morphism $S \in \S_{n \red k}$}
\vskip-15pt
\centerline{\fig{}{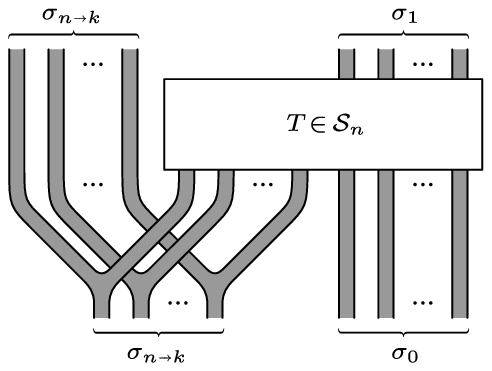}}
\vskip-6pt
\end{Figure}

\begin{Figure}[b]{ribbon-stab02/fig}
{}{The product $S \rdiam S'$ of two morphisms in $\S_{n \red k}$}
\vskip-3pt
\centerline{\fig{}{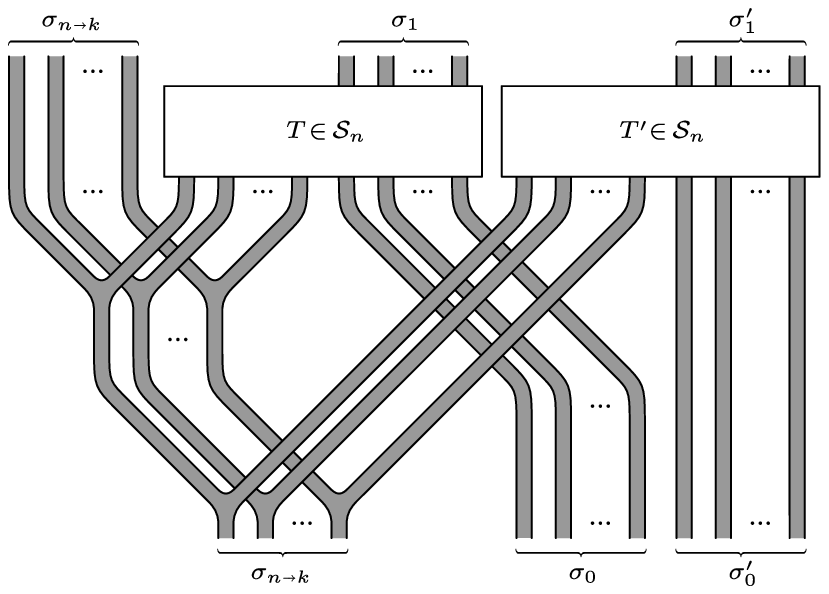}}
\vskip-6pt
\end{Figure} 

Since the subcategory $\S_{n \red k}$ of $k$-reducible ribbon surface tangles is not
closed with respect to the product $\diam: \S_n \times \S_n \to \S_n$, we endow it with a
product structure ${\rdiam}: \Mor_{\S_{n \red k}} \times \Mor_{\S_{n\red k}} \to
\Mor_{\S_{n \red k}}$, similarly to what we have done for $k$-reducible Kirby tangles.
Namely, given two morphisms $S = (\id_{\sigma_{n \red k}} \! \diam T) \circ
(\Delta_{\sigma_{n \red k}} \!\diam \id_{\sigma_0}): J_{\sigma_{n \red k}} \!\diam
J_{\sigma_0}\to J_{\sigma_{n \red k}} \!\diam J_{\sigma_1}$ and $S' = (\id_{\sigma_{n \red
k}} \! \diam T') \circ (\Delta_{\sigma_{n \red k}} \!\diam \id_{\sigma'_0}): 
\pagebreak
J_{\sigma_{n
\red k}} \!\diam J_{\sigma_0'}\to J_{\sigma_{n \red k}} \!\diam J_{\sigma_1'},$ in $\S_{n
\red k}$, their product $S \rdiam S': J_{\sigma_{n \red k}} \!\diam J_{\sigma_0 \diam
\sigma'_0}\to J_{\sigma_{n \red k}} \!\diam J_{\sigma_1 \diam \sigma'_1}$ is defined by:
\begin{eqnarray*}
S \rdiam S' 
&=& S \circ (\id_{\sigma_{n \red k}} \!\diam \gamma_{\sigma'_1,\sigma_0}) \circ (S'
\diam \id_{\sigma_0}) \circ (\id_{\sigma_{n \red k}} \!\diam 
\gamma^{-1}_{\sigma'_0,\sigma_0})\\
&=& (\id_{\sigma_{n \red k}} \!\diam T \diam T') \circ (\Delta_{\sigma_{n \red k}}
\!\diam \gamma_{\sigma_{n \red k},\sigma_0} \diam \id_{\sigma'_0}) \circ
(\Delta_{\sigma_{n \red k}} \!\diam \id_{\sigma_0 \diam \sigma'_0})\,.
\end{eqnarray*}
These two expressions for $S \rdiam S'$ are related by diagram isotopy, as the reader can
easily realize by looking at Figure \ref{ribbon-stab02/fig} representing the second one.
The associativity of $\,\rdiam\,$ is a consequence of the coassociativity property of
$\Delta$ and its unit is given by $\id_{\sigma_{n \red k}}$. Observe that, as in the case
of the category of $k$-reducible Kirby tangles,\break $\rdiam\,$ is a useful tool, but it
does not define a monoidal structure on $\S_{n \red k}$, since the product of the
compositions of two morphisms, does not coincide with the composition of the corresponding
products.

\begin{proposition}\label{S-monoidal-stab/thm}
For any $n > k \geq 2$, the image $\up_k^n \S_k$ of the stabilization functor is a
subcategory $\S_{n \red k}$. Moreover, for any two morphisms $S$ and $S'$ in $\S_k$, we
have $\up_k^n(S \diam S') = (\up_k^n S) \rdiam (\up_k^n S')$. Hence, the product
$\,\rdiam\,$ defines a monoidal structure on the subcategory $\up_k^n \S_k$.
\end{proposition}

\begin{proof}
Given a ribbon surface tangle $S$ in $\S_k$, the stabilization $\up_k^n S$ can be put in
the form shown in Figure \ref{ribbon-stab01/fig}, by expanding a tongue from each
stabilization ribbon to the box containing $S$ itself. The identity $\up_k^n (S \diam S')
= (\up_k^n S) \rdiam (\up_k^n S')$ for any $S$ and $S'$ in $\S_k$, immediately follows
from the definition of the product $\rdiam$ in $\S_{n \red k}$.
\end{proof}

We are going to prove that $\S_{n \red k}$ is equivalent to $\up_{k+2}^n \S_{(k+2) \red
k}$ when $n \geq k+3 \geq 4$. As a consequence, $\S_n^c$ is equivalent to $\up_3^n \S_3^c$
for $n \geq 4$. Actually, the much stronger statement that $\up_4^n: \S_4^c \to \S_n^c$ is
a category equivalence for $n \geq 5$ is also true, as it will follow from Proposition
\ref{K-reduction/thm} once Proposition \ref{stab-theta/thm} and Theorem
\ref{ribbon-kirby/thm} will be established.

\begin{proposition}\label{S-reduction/thm}
For any $n \geq k + 3 \geq 4$, the inclusion of $\up_{k+2}^n \S_{(k+2) \red k}$ in $\S_{n 
\red k}$ is an equivalence of categories. 
\end{proposition}

\begin{proof} 
We first prove the statement for $k = n - 3$. More precisely, to any se\-quence $\sigma
\in \seq\Gamma_n$ we associate a sequence $\sigma' \in \seq\Gamma_{n-1}$ and a natural
isomorphism
$$\zeta^{n,n-3}_\sigma : J_{\sigma_{n \red (n-3)}} \!\diam J_\sigma \to 
J_{\sigma_{n \red (n-3)}} \!\diam J_{\sigma'}$$
such that: any labeled ribbon surface tangle 
$$\zeta^{n,n-3}_{\sigma_1} \circ (\id_{\sigma_{n \red (n-3)}} \!\diam S) \circ 
(\Delta_{\sigma_{n \red (n-3)}} \!\diam \id_{\sigma_0}) \circ 
(\zeta^{n,n-3}_{\sigma_0})^{-1}$$
with $S: J_{\sigma_{n \red (n-3)}} \!\diam J_{\sigma_0} \to J_{\sigma_1}$ in $\S_n$, 
is equivalent to the product of $\id_{\tp{n}{n{-}1}}$ and an $(n-3)$-reducible tangle in
$\S_{n-1}$, that is to
$$\id_{\tp{n}{n{-}1}} \diam ((\id_{\sigma_{(n-1) \red (n-3)}} \!\diam T) \circ 
(\Delta_{\sigma_{(n-1) \red (n-3)}} \!\diam \id_{\sigma'_0}))$$
for some $T: J_{\sigma_{(n-1) \red (n-3)}} \!\diam J_{\sigma'_0} \to J_{\sigma'_1}$ in 
$\S_{n-1}$ (see Figure \ref{ribbon-stab03/fig}).
Then Corollary \ref{subcat-equiv/thm}, allows us to conclude that the inclusion of
$\up_{n-1}^n \S_{(n-1) \red (n-3)}$ in $\S_{n \red (n-3)}$ is an equivalence of monoidal
categories.

\begin{Figure}[htb]{ribbon-stab03/fig}
{}{The equivalence between $\S_{n \red (n-3)}$ and $\up_{n-1}^n \S_{(n-1) \red (n-3)}$}
\centerline{\fig{}{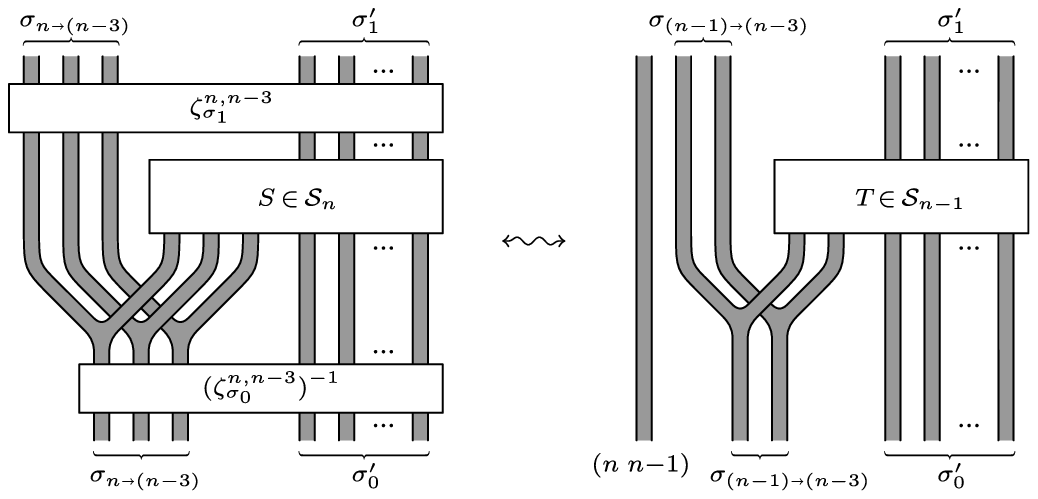}}
\vskip-6pt
\end{Figure}

Given $\sigma = (\tp{i_1}{j_1}, \dots, \tp{i_m}{j_m})$, we let $\sigma'$ be obtained from
$\sigma$ by replacing any transposition $\tp{i}{n}$ with $\tp{i}{n{-}1}$ if $i < n-1$ and
with $\tp{n{-}2}{n{-}1}$ if $i = n-1$. Moreover, we define $\zeta^{n,n-3}_\sigma =
\zeta^{n,n-3}_{\sigma,m} \circ \dots \circ \zeta^{n,n-3}_{\sigma,1}$, where
$\zeta^{n,n-3}_{\sigma,h}$ is the identity if $\tp{i_h}{j_h} \neq \tp{i}{n}$ while it is
illustrated in Figure \ref{ribbon-stab04/fig} for $\tp{i_h}{j_h} = \tp{i}{n}$. Here, the
tongue starting from the reduction ribbon of $\id_{\tp{n}{n{-}1}}$ passes through the
reduction ribbon of $\id_{\tp{n{-}1}{n{-}2}}$ if $i = n-1$ and in front of all the other
ribbons in any case, then it forms a ribbon intersection with the $h$-th ribbon of
$\id_\sigma$. The $\zeta^{n,n-3}_{\sigma,h}$'s are isomorphisms, their inverses being
obtained just by vertical reflection. Therefore $\zeta^{n,n-3}_\sigma$ is an
isomorphism as well.

\begin{Figure}[htb]{ribbon-stab04/fig}
{}{The isomorphism $\zeta^{n,n-3}_{\sigma,h}$ ($i < n-1$)}
\vskip-3pt
\centerline{\fig{}{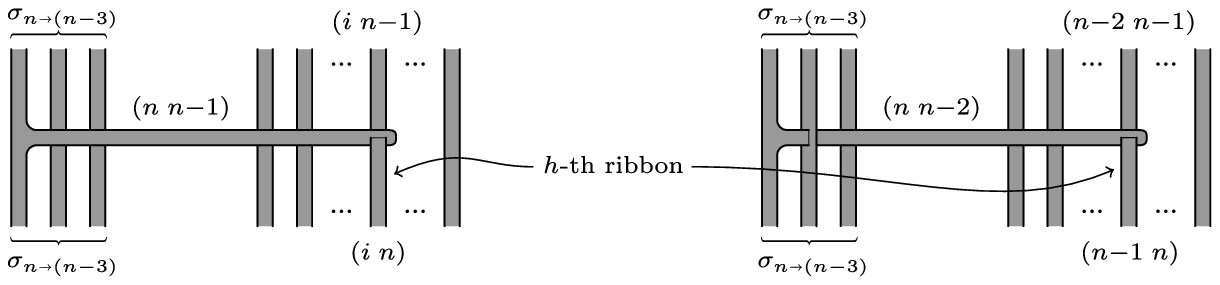}}
\vskip-6pt
\end{Figure}

We observe that $\zeta^{n,n-3}_{\sigma_1} \circ (\id_{\sigma_{n \red (n-3)}} \!\diam S)
\circ (\Delta_{\sigma_{n \red (n-3)}} \!\diam \id_{\sigma_0}) \circ
(\zeta^{n,n-3}_{\sigma_0})^{-1}$ factorizes as the composition $P_1 \circ P_2$, where
$P_1 = \zeta^{n,n-3}_{\sigma_1} \circ (\id_{\sigma_{n \red (n-3)}} \!\diam S) \circ
(\zeta^{n,n-3}_{\sigma_{n \red (n-3)} \diam \sigma_0})^{-1}$ and
$P_2 = \zeta^{n,n-3}_{\sigma_{n\red (n-3)} \diam \sigma_0} \circ (\Delta_{\sigma_{n \red
(n-3)}} \!\diam \id_{\sigma_0})\circ (\zeta^{n,n-3}_{\sigma_0})^{-1}$. Then, it will
suffice to show that both these factors are equivalent to the product of
$\id_{\tp{n}{n{-}1}}$ and an $(n-3)$-reducible tangle in $\S_{n-1}$.

According to Proposition \ref{planar-diagram/thm}, we can assume $S$ to be presented by a
labeled special planar diagram. Moreover, Figure \ref{ribbon-stab05/fig} shows how $S$ can
be transformed through equivalence moves in such a way that:
\begin{itemize}
\item[\(a)]
at any crossing between two ribbons respectively labeled by $\tp{i}{n}$ and 
$\tp{j_1}{j_2}$ with $j_1,j_2 < n$, the first ribbon passes in front of the second one;
\item[\(b)] 
there are no crossings between ribbons labeled by $\tp{i}{n}$ and $\tp{j}{n}$ with 
$i \neq j$ (actually, we will need this only for $i$ or $j$ equal to $n-1$);
\item[\(c)]
there are no ribbon intersections involving the three transpositions $\tp{n{-}2}{n{-}1}$, 
$\tp{n{-}2}{n}$ and $\tp{n{-}1}{n}$ as labels.
\end{itemize}

\begin{Figure}[htb]{ribbon-stab05/fig}
{}{}
\centerline{\kern9pt\fig{}{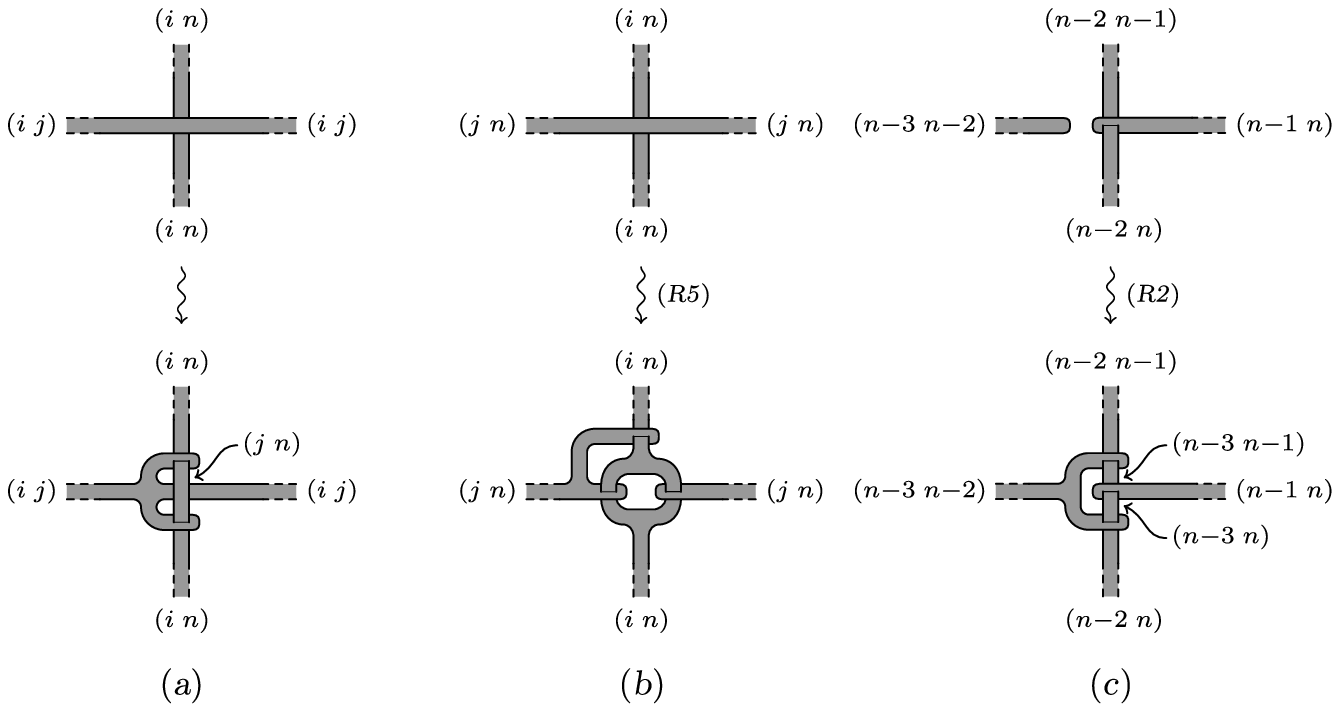}}
\vskip-3pt
\end{Figure}

In particular, to get property \(a) we invert any wrong crossing involving the labels
$\tp{i}{n}$ and $\tp{j_1}{j_2}$ by applying a move \(R4) if $j_1,j_2 \neq i$, and by
1-isotopy (inserting two extra ribbon intersections) as shown in Figure \ref{ribbon-stab05/fig} 
\pagebreak
\(a) otherwise. Then, we proceed as in Figure
\ref{ribbon-stab05/fig} \(b) to transform all the crossings forbidden by property \(b)
into ribbon intersections. Finally, in Figure \ref{ribbon-stab05/fig} \(c) we see how to
get rid of the ribbon intersections forbidden by property \(c), possibly after performing
move \(R1) and/or move \(S2) in Figure \ref{ribbon-surf08/fig} to obtain the starting
configuration in the figure.

After all those modifications have been performed, we use move \(S2) to put the diagram of
$S$ again into special form, without losing properties \(a), \(b) and \(c).

\begin{Figure}[b]{ribbon-stab06/fig}
{}{($i < n-2$)}
\centerline{\fig{}{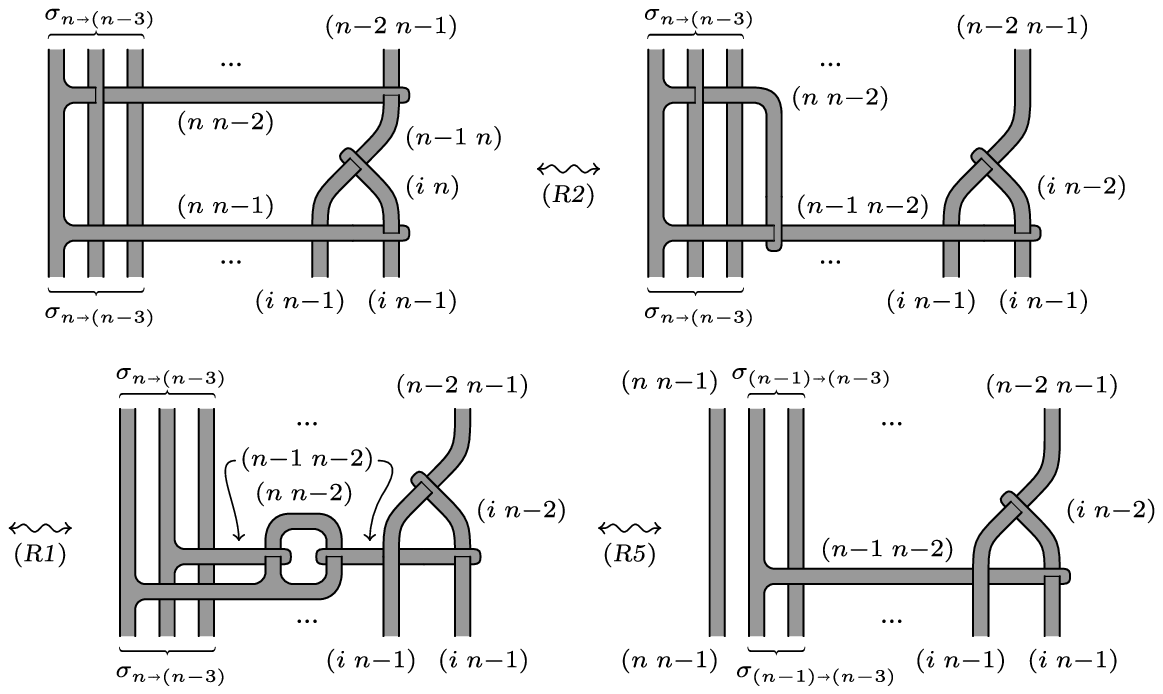}}
\vskip-6pt
\end{Figure}

Then, we express $S$ as a composition $S_l \circ \dots \circ S_1$, where each $S_k$ is a
product of a single labeled elementary tangle (cf. Figure \ref{ribbon-morph01/fig}) and
some identity ribbons on the left and/or on the right of it. We do that by means of planar
isotopy (cf. proof of Proposition \ref{S-category/thm}). Up to move \(R1), we can also
assume that each ribbon intersection in the $S_k$'s is like \(g) in Figure
\ref{ribbon-morph01/fig}. Moreover, if the labels involved in \(g) are $\tp{i}{n{-}1}$,
$\tp{i}{n}$ and $\tp{n{-}1}{n}$ thanks to (c) we have necessarily that $i < n-2$, and up
to plane isotopy and move \(I14'), we can assume that the target is $J_{\tp{n{-}1}{n}}$
and the source is $J_{\tp{i}{n-1}}\diam J_{\tp{i}{n}}$ (cf. the first diagram of Figure
\ref{ribbon-stab06/fig}). Observe that the resulting plane diagram still satisfies
properties \(a), \(b) and \(c).

At this point, in order to prove that $P_1$ is the product of $\id_{\tp{n}{n{-}1}}$ with
an $(n-3)$-reducible tangle in $\S_{n-1}$, we can limit ourselves to consider the case of
$\zeta^{n,n-3}_{\sigma_1} \circ\,(\id_{\sigma_{n \red (n-3)}} \!\diam S)
\,\circ\,(\zeta^{n,n-3}_{\sigma_0})^{-1}$ with $S$ from $\sigma_0$ to $\sigma_1$ given by
the product of a single labeled elementary tangle and some identity ribbons as above.

If such elementary tangle is not of type \(g) with ribbons labeled $\tp{i}{n{-}1}$,
$\tp{i}{n}$ and $\tp{n{-}1}{n}$, then a straightforward case by case verification shows
that $\zeta^{n,n-3}_{\sigma_1} \circ (\id_{\sigma_{n \red (n-3)}} \!\diam S) \circ
(\zeta^{n,n-3}_{\sigma_0})^{-1}$ is equivalent up to labeled 1-isotopy to $\id_{\sigma_{n
\red (n-3)}} \!\diam S'$, where $S'$ is obtained from $S$ by replacing any transposition
$\tp{i}{n}$ with $\tp{i}{n{-}1}$ if $i < n-1$ and with $\tp{n{-}2}{n{-}1}$ if $i = n-1$.
In fact, starting from $h = 1$, we can progressively cancel $\zeta^{n,n-3}_{\sigma_1,h}$'s
with $(\zeta^{n,n-3}_{\sigma_0,h})^{-1}$'s on all identity ribbons on the left of the
elementary tangle involved. Then, if the resulting tangle is itself the identity, \(f-f')
or \(e-e'), we can continue the cancelation until the end. Observe that in the case of the
crossings, we can do that only because of conditions \(a) and \(b). On the other hand, if
we reach an elementary tangle of the type \(b-b'), \(c-c') or \(d), before the cancelation
some moves \(I20) and \(I21) must be performed. Finally, if we reach a ribbon intersection
with ribbons labeled $\tp{i}{n}$, $\tp{j}{n}$ and $\tp{i}{j}$, where $i\neq n-1\neq j$, we
achieve the cancelation of the ribbons involved in the natural transformation, after
performing moves \(I22) and \(R2).

\begin{Figure}[b]{ribbon-stab07/fig}
{}{}
\vskip6pt
\centerline{\fig{}{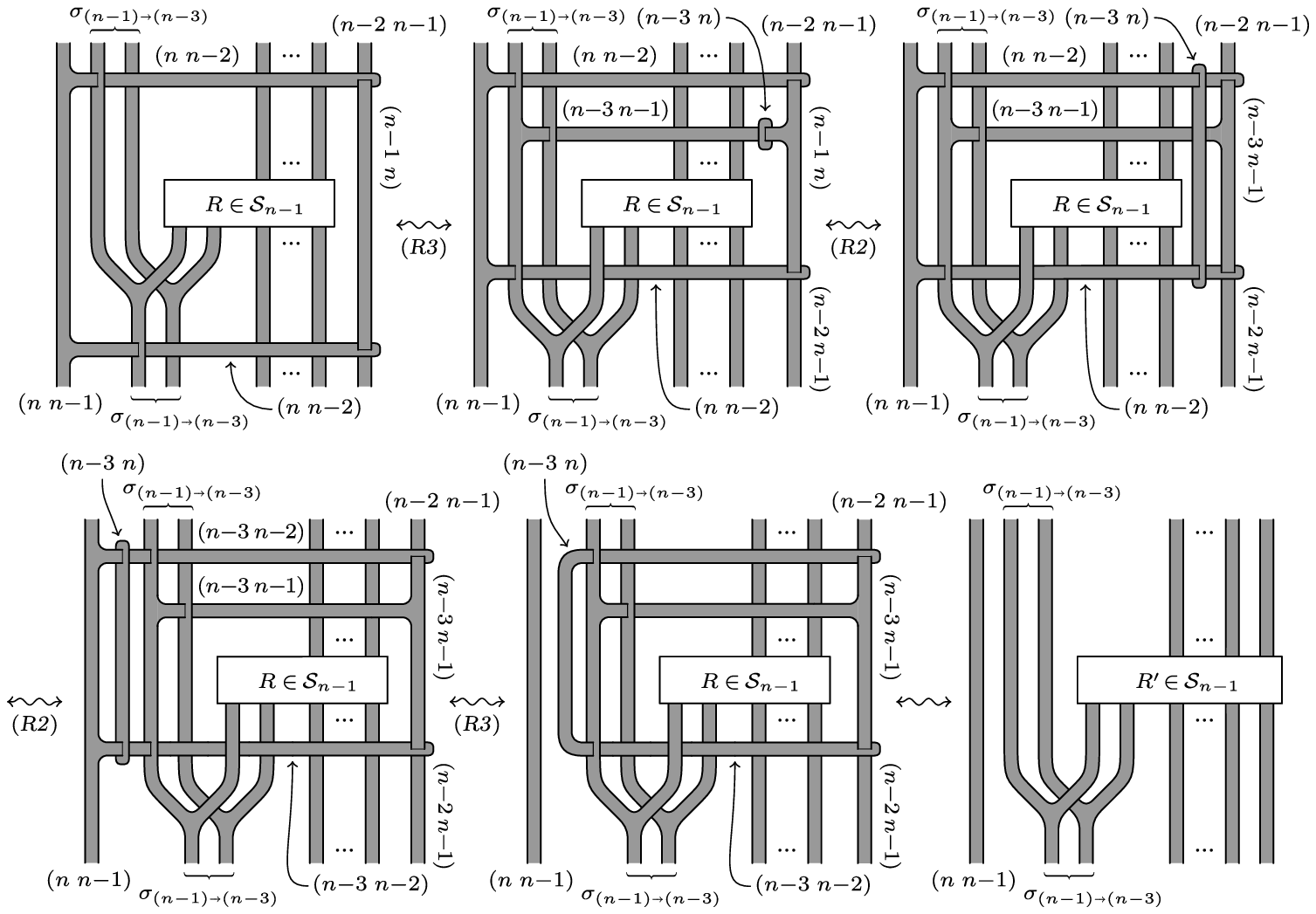}}
\vskip-6pt
\end{Figure}

It remains to consider the case in which the elementary tangle is of type \(g) with
ribbons labeled $\tp{i}{n{-}1}$, $\tp{i}{n}$ and $\tp{n{-}1}{n}$. We start as above by
cancelling $\zeta^{n,n-3}_{\sigma_1,h}$'s and $(\zeta^{n,n-3}_{\sigma_0,h})^{-1}$'s,
corresponding to the identity ribbons on the left of the elementary diagram. Then, we
perform the modifications described in Figure \ref{ribbon-stab06/fig} (here the 1-isotopy
moves are not indicated). Notice that the final tangle in Figure \ref{ribbon-stab06/fig}
is the product of $\id_{\tp{n}{n{-}1}}$ with an $(n-3)$-reducible tangle in $\S_{n-1}$.
After that, we proceed by induction on the number of identity ribbons on the right of the
elementary diagram. The inductive step consists in including into a new $(n-3)$-reducible
tangle in $\S_{n-1}$ a single identity ribbon on the right together with the corresponding
$\zeta^{n,n-3}_{\sigma_1,h}$ and $(\zeta^{n,n-3}_{\sigma_0,h})^{-1}$. If the label of the
identity ribbon is $\tp{i}{j}$ with $i,j < n$ we are already done, while a trivial
1-isotopy suffices if the label is $\tp{i}{n}$ with $i < n-1$. Figure
\ref{ribbon-stab07/fig} shows how to deal the remaining case of label $\tp{n{-}1}{n}$.
This completes the proof that $P_1$ is equivalent to the product of $\id_{\tp{n}{n{-}1}}$
and an $(n-3)$-reducible tangle in $\S_{n-1}$.

Now, we show that the same holds for the tangle $P_2$. We proceed by induction on the 
difference $n - k$. The starting case is when $n - k = 3$, that is $k = n-3$.
In this case, modify the horizontal ribbon of $\zeta^{n,n-3}_{\sigma_{n \red (n-3)}}$ as
described in Figure \ref{ribbon-stab08/fig}. Once again, we get the product of
$\id_{\tp{n}{n{-}1}}$ with an $(n-3)$-reducible tangle in $\S_{n-1}$. Then, we can
conclude by applying the inductive argument in Figure \ref{ribbon-stab07/fig}.

\begin{Figure}[htb]{ribbon-stab08/fig}
{}{}
\vskip-3pt
\centerline{\fig{}{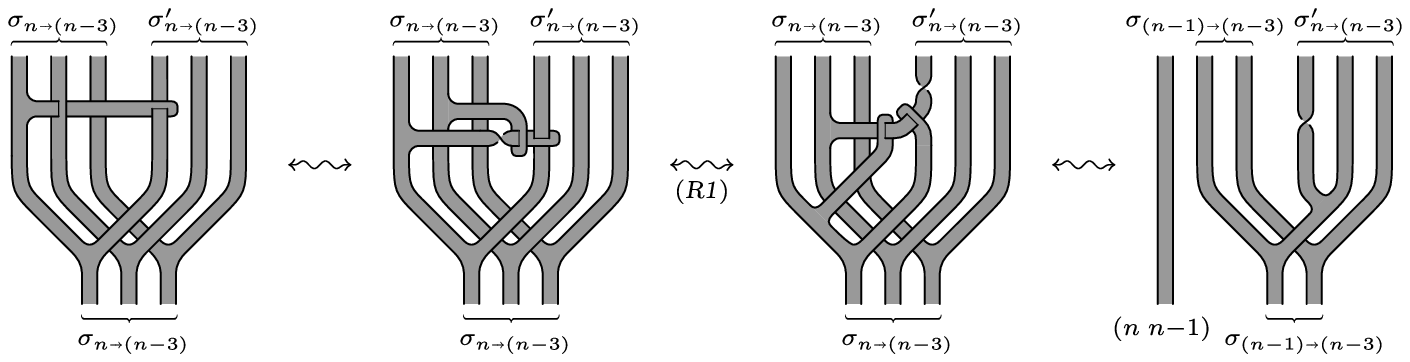}}
\vskip-6pt
\end{Figure}

For $n - k > 3$, the inductive step goes as follows. We define inductively the natural
equivalences $\zeta^{n,k}_\sigma = (\id_{\tp{n}{n{-}1}} \diam \zeta^{n-1,k}_\sigma) \circ
\zeta^{n,n-3}_\sigma$. Then, by applying the result for $k = n-3$ and the inductive
hypothesis, we obtain that for any $n$-labeled ribbon surface tangle $S$ in $\S_{n \red
k}$, the composition $\zeta^{n,k}_{\sigma_1} \circ S \circ (\zeta^{n,k}_{\sigma_0})^{-1}$
is equivalent to a tangle in $\up_{n-1}^n \up_{k+2}^{n-1} \S_{(k+2) \red k} = \up_{k+2}^n
\S_{(k+2) \red k}$.
\end{proof}

\subsection{The functors $\Theta_n: \S_n \to \K_n$%
\label{Theta/sec}}

The object of this section is to define the family of functors $\Theta_n: \S_n \to \K_n$
for $n \geq 2$, which provide the branched covering representation of relative
4-dimensional 2-handlebody cobordisms.

This will be done by exploiting the ideas introduced by Montesinos in \cite{Mo78}, to give
an effective explicit construction of a generalized Kirby tangle $K_S \in \K_n$ for the
branched covering space of $E \times [0,1] \times [0,1]$ determined by an $n$-labeled
ribbon surface tangle $S$ with $n \geq 2$. 

Before of going into details, let us briefly sketch how such construction derives from
\cite{Mo78}. Assuming $n \geq 2$, let $n$-labeled ribbon surface tangle $S$ from
$J_{\sigma_0}$ to $J_{\sigma_1}$, with $\sigma_0 = (\tp{i^0_1}{j^0_1}, \dots,
\tp{i^0_{m_0}}{j^0_{m_0}})$ and $\sigma_1 = (\tp{i^1_1}{j^1_1}, \dots,
\tp{i^1_{m_1}}{j^1_{m_1}})$ sequences in $\seq\Gamma_n$, and let $p: W \to E \times [0,1]
\times [0,1]$ be the simple $n$-fold branched covering represented by $S$, according to
Section \ref{coverings/sec}.

\pagebreak

We start with the simple $n$-fold branched covering of $E \times [0,1] \times [0,1]$
represented by the labeled ribbon surface tangle $T$ depicted on the left side of Figure
\ref{ribbon-kirby01/fig}. The corresponding covering space can be easily seen to be
$Y(M_{\pi_0},M_{\pi_1})$, where $\pi_0 = ((i^0_1,j^0_1), \dots, (i^0_{m_0},j^0_{m_0}))$
and $\pi_1 = ((i^1_1,j^1_1), \dots, (i^1_{m_1},j^1_{m_1}))$ are determined by assuming
$i^c_h > j^c_h$ for any $h = 1, \dots, m_c$ and $c = 0,1$. In particular, the 3-cell $E
\times \{c\} \times [0,1]$ with the labeled arcs $J_{\sigma_c}$ represents $M_{\pi_c}
\cong M_{\pi_c} \times \{c\} \subset Y(M_{\pi_0},M_{\pi_1})$ as a $n$-fold simple branched
cover. An $n$-labeled Kirby tangle for $Y(M_{\pi_0},M_{\pi_1})$ is given on the right side
of the same Figure \ref{ribbon-kirby01/fig}.

\begin{Figure}[htb]{ribbon-kirby01/fig}
{}{$Y(M_{\pi_0},M_{\pi_1})$ as a branched cover of $E \times [0,1] \times [0,1]$}
\centerline{\fig{}{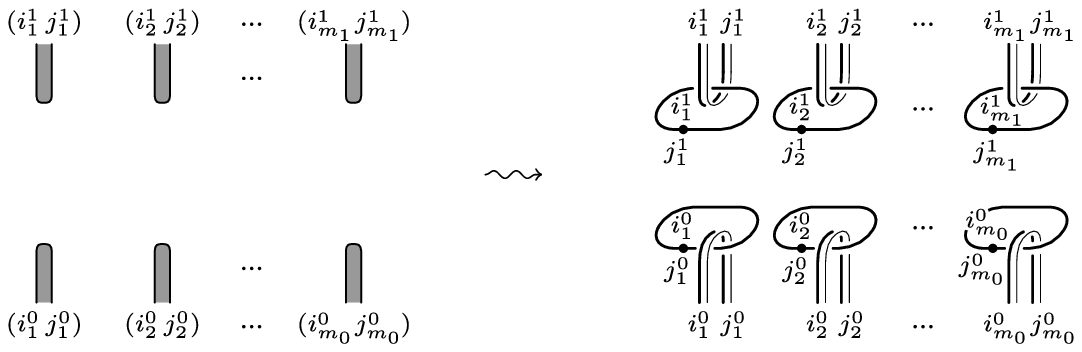}}
\end{Figure}

The ribbon surface tangle $S$ can be obtained from $T$ as follows. First add some
regularly embedded disjoint labeled disks $D_1, \dots, D_r$, and then attach to $S^0 = T
\cup D_1 \cup \dots \cup D_r$ some disjoint regularly embedded labeled bands $B_1, \dots,
B_s$, which possibly pass through the disks to form ribbon intersections. In the end,
those disks and bands will give an adapted 1-handlebody decomposition of $S$, considered
as a labeled embedded 1-handlebody build on $J_{\sigma_0} \cup J_{\sigma_1}$.

According to \cite{Mo78} (cf. also \cite{IP02}), each disks $D_h$ with label $\tp{i}{j}$
gives raise to a 1-handle $H^1_h$ attached to $Y(M_{\pi_0},M_{\pi_1})$ between the sheets
$i$ and $j$ of its branched covering representation, while each band $B_h$ gives raise to
a 2-handle attached to the covering space represented by $S^0$, whose attaching framed
knot coincides with the unique annular component of the counterimage of $B_h$ in such
covering space. The final result is a relative 2-handlebody decomposition of the space $W$
build on $X(M_{\pi_0},M_{\pi_1})$.

\medskip

Now, we want to make the above sketchy recipe into a formal definition of the functor
$\Theta_n: \S_n \to \K_n$ for $n \geq 2$. The non-trivial points here are: 1) the
description of the attaching framed knots of the 2-handles; 2) the proof that the
2-deformation class of the 2-handlebody structure of $W$ only depends on the equivalence
class of $S$ in the sense of Definition \ref{lrs-equivalence/def}, and in particular it
does not depend on the 1-handlebody decomposition of $S$ used in the construction. As we
will see, generalized Kirby tangles provide a quite natural way to face the first point,
while the second point requires some work.

To define $\Theta_n$ on the objects, we consider the map $\Gamma_n \to \G_n$ given by
$\tp{i}{j} \mapsto (i,j)$ with $i > j$ and the induced map $\pi: \seq\Gamma_n \to
\seq\G_n$ on the sequences. Then, we put $\Theta_n (J_{\sigma}) = I_{\pi(\sigma)}$ for any
$\sigma \in \seq\Gamma_n$.

The definition of $\Theta_n$ on the morphisms is much more involved. Let $S \in \S_n$ be
an $n$-labeled ribbon surface tangle from $J_{\sigma_0}$ to $J_{\sigma_1}$, with $\sigma_0
= (\tp{i^0_1}{j^0_1}, \dots, \tp{i^0_{m_0}}{j^0_{m_0}})$ and $\sigma_1 =
\tp{(i^1_1}{j^1_1}, \dots, \tp{i^1_{m_1}}{j^1_{m_1}})$ sequences in $\seq\Gamma_n$,
considered up to 3-dimensional diagram isotopy. Then we define the generalized Kirby
tangle $\Theta_n(S) = K_S$ (well-defined up to 2-equivalence, cf. Lemma
\ref{ribbon-to-kirby/thm} below) by the following steps.
\begin{itemize}
\item[1)]
Start with the $n$-labeled Kirby tangle on the right side of Figure
\ref{ribbon-kirby01/fig}, where $\pi_0 = \pi(\sigma_0) = ((i^0_1,j^0_1), \dots,
(i^0_{m_0},j^0_{m_0}))$ and $\pi_1 = \pi(\sigma_1) = ((i^1_1,j^1_1), \dots,
(i^1_{m_1},j^1_{m_1}))$.
\item[2)]
Choose an adapted relative 1-handlebody decomposition $S = T \cup D_1 \cup \dots \cup D_r
\cup B_1 \cup \dots \cup B_s$ build on $J_{\sigma_0} \!\cup J_{\sigma_1}$, where: $T =
(\cup_{h=1}^{m_0} T^0_h) \cup (\cup_{h=1}^{m_1} T^1_h)$ is a collar of $J_{\sigma_0}
\!\cup J_{\sigma_1}$ in $S$, with $T^c_h \cong J_{\tp{i^c_h}{j^c_h}} \times [0,1]$ a
collar of $J_{\tp{i^c_h}{j^c_h}}$ for $h = 1,\dots,m_c$ and $c = 0,1$; $D_1, \dots, D_r$
are disjoint disks (the 0-handles of the decomposition); $B_1, \dots, B_s$ are disjoint
bands attached to $S^0 = T \cup D_1 \cup \dots \cup D_r$ (the 1-handles of the
decomposition).
\item[3)]
Add to the starting Kirby tangle a dotted unknot spanning the disk $D_h$ for each $h = 1,
\dots, r$; if $\tp{i}{j} \in \Gamma_n$ is the label of $D_h$ in $S$, then choose one of
the two possible ways to assign the labels $i$ and $j$ to the faces of $D_h$ (cf. Figure
\ref{ribbon-kirby02/fig}). We call such a choice a {\sl polarization} of the disk $D_h$.
\begin{Figure}[htb]{ribbon-kirby02/fig}
{}{From $0$-handles of $S$ to $1$-handles of $K_S$}
\vskip-3pt
\centerline{\fig{}{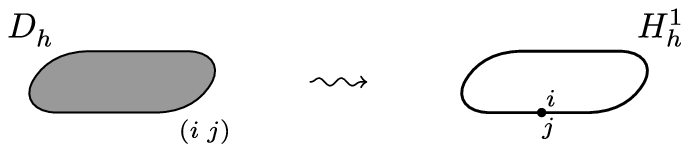}}
\vskip-3pt
\end{Figure}
\item[4)] 
Replace the terminal part of each band attached to any disk $D_h$ by a labeled framed arc
consisting of two opposite parallel displacements of it joined together to form a ribbon
intersection with $D_h$ as shown in Figure \ref{ribbon-kirby03/fig} \(a). Do the same for
the terminal parts of the bands attached to $T$, as shown in Figure
\ref{ribbon-kirby03/fig} \(b) and \(c). Notice that in all the cases the labeling of the
framed arc is uniquely determined by that of the disk spanned by the dotted component.
\begin{Figure}[htb]{ribbon-kirby03/fig}
{}{From $1$-handles of $S$ to $2$-handles of $K_S$: the ends}
\vskip-3pt
\centerline{\fig{}{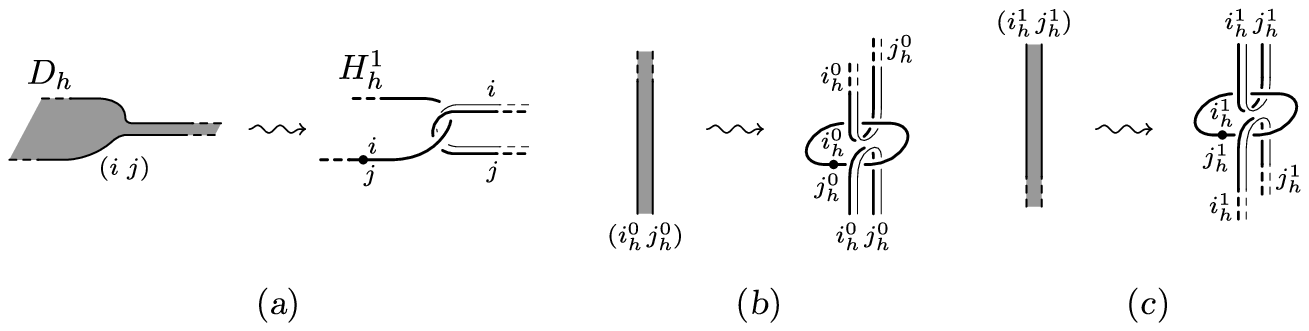}}
\vskip-3pt
\end{Figure}
\item[5)]
For each ribbon intersection arc in $D_h$ choose a regularly embedded arc $\alpha \subset
D_h$ transversally starting from it and ending in $\Bd D_h \cap \Bd S$. All these arcs are
chosen to be disjoint from each other and not to meet elsewhere the ribbon intersections
in $D_h$. Moreover, up to 3-dimensional diagram isotopy moving the relative ribbon 
intersection inside $D_h$, we can contract each arc $\alpha$ in a small neighborhood of
its end point in $\Bd D_h$ (cf. Figure \ref{ribbon-kirby04/fig}).
\begin{Figure}[htb]{ribbon-kirby04/fig}
{}{From $1$-handles of $S$ to $2$-handles of $K_S$: the ribbon intersections}
\centerline{\fig{}{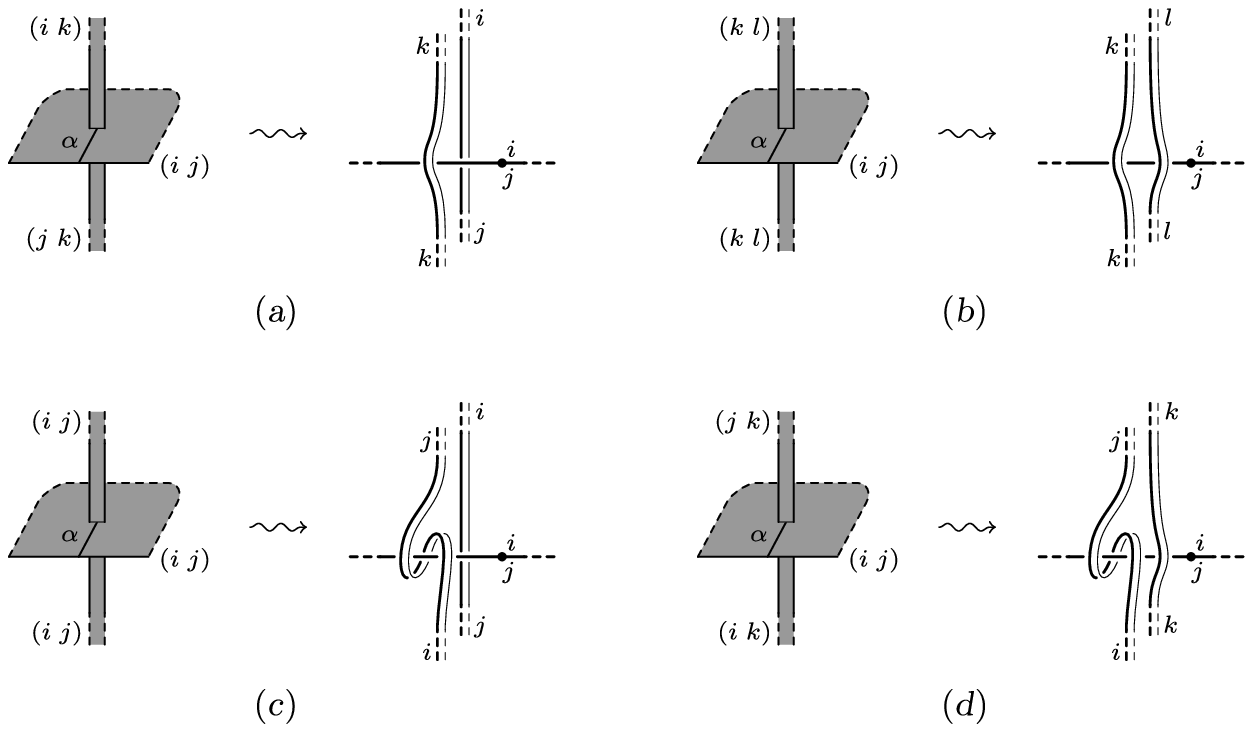}}
\vskip-3pt
\end{Figure}
\item[6)]
Replace a small portion of the band involved in each ribbon intersection in a neighborhood
of the corresponding arc $\alpha$ by two labeled framed arcs as shown in Figure
\ref{ribbon-kirby04/fig} (the four cases depend on how the local labeling of the band is
related to the polarization of $D_h$). Here, the framed arcs are two opposite parallel
displacements of the band suitably modified in order to allow labeling compatibility (the
opposite choice for the kinks in \(c) and \(d) would be equivalent up to labeled isotopy)
Like in the previous point 4, the labeling of the framed arcs is uniquely determined by
the polarization of the disk, with the only exception of case \(b) where the labels $k$
and $l$ could be interchanged.
\item[7)]
Finally, replace the remaining part of each band by opposite parallel displacements of it,
joining those already inserted in the previous points 4 and 6, inserting one crossover as
shown in Figure \ref{ribbon-kirby05/fig} (the two possible choices for the crossover are
equivalent up to labeled isotopy) where needed to match the labeling.
\begin{Figure}[htb]{ribbon-kirby05/fig}
{}{From $1$-handles of $S$ to $2$-handles of $K_S$: the crossovers}
\vskip-3pt
\centerline{\fig{}{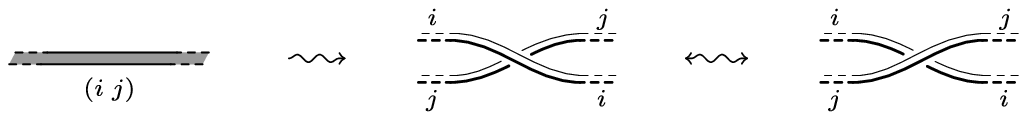}}
\end{Figure}
\end{itemize}

We remark that at the end of the construction each band is replaced by a framed knot in
the resulting generalized Kirby tangle. By the very definition of generalized Kirby
tangle, one could easily check that such framed knot does really represent the unique
annular component in the counterimage of the band through the branched covering determined
by the labeled ribbon surface tangle $S^0 = T \cup D_1 \cup \dots \cup D_r$, no matter
what choices are made at points 3, 5 and 6 \(b). Then, it would immediately follows from
\cite{Mo78}, that $K_S$ represents a relative 2-handlebody structure of $W$, which only
depends on the 1-handlebody decomposition of $S$ chosen at point 2.

Nevertheless, in the next lemma we prove directly that $K_S$ is well-defined up to
2-equivalence (cf. Definition \ref{kt-equivalence/def}) for a given $n$-labeled ribbon
surface tangle $S$. Actually, in Proposition \ref{theta/thm} we will see that the
2-equivalence class of $K_S$ is also invariant under 1-isotopy and covering moves of $S$.
In other words, the 2-deformation class of the relative 4-dimensional 2-handlebody
represented by $K_S$ does not depend on the choices involved in the above construction of
$K_S$ (included the 1-handlebody decomposition of $S$), in fact it only depends on the
equivalence class of $S$ (cf. Definition \ref{lrs-equivalence/def}). Then, from the
results of next Sections (Propositions \ref{full-theta3/thm}, \ref{xi4/thm} and
\ref{full-xi/thm}), it will follow that such 2-deformation class is the same of the
relative 2-handlebody structure of $W$ deriving from \cite{Mo78}.

\begin{lemma} \label{ribbon-to-kirby/thm}
The generalized Kirby tangle $K_S$ constructed above from a given $n$-labeled ribbon 
surface tangle $S$ with $n \geq 2$ is well-defined up to 2-equivalence.
\end{lemma}

\begin{proof}
First of all, we note that the construction of $S_K$ is clearly invariant under labeled
3-dimensional diagram isotopy (preserving ribbon intersections). Then, the relevant
choices occurring in it are in the order: the adapted 1-handlebody decomposition of $S$;
the polarizations of the $D_h$'s; the arcs $\alpha$ associated to the ribbon
intersections; the labeling of the framed arcs in Figure \ref{ribbon-kirby04/fig} \(b). We
prove that $K_S$ does not depend on them up to 2-equivalence, by proceeding in the
reversed orther and assuming each time that all previous choices are kept fixed.

For the labeling of the framed arcs in Figure \ref{ribbon-kirby04/fig} \(b), it suffices
to observe that switching the labels $k$ and $l$ is compensated up to labeled isotopy by 
the crossovers inserted in point 7.

Concerning the arcs $\alpha$, up to labeled 3-dimensional diagram isotopy different
choices can be related by a finite sequence of the elementary moves of Figure
\ref{ribbon-kirby06/fig}, where we replace a single arc $\alpha$ by $\alpha'$. By the
invariance under labeled 3-dimensional diagram isotopy, we only need to deal with these
elementary moves.

\begin{Figure}[htb]{ribbon-kirby06/fig}
{}{The elementary moves for the arcs $\alpha$}
\centerline{\fig{}{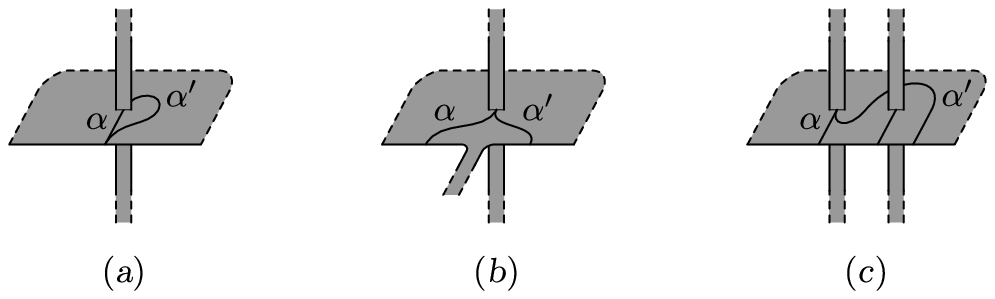}}
\vskip-3pt
\end{Figure}

Move \(a) is the same as inserting two opposite half-twists in the vertical band just
above and below the ribbon intersection (while leaving $\alpha$ unchanged). But this does
not produce any difference on the resulting Kirby tangle, due to the extra crossovers
(possibly canceling preexisting ones) needed to keep the labeling consistency. Figure
\ref{ribbon-kirby07/fig} shows how the insertion of one half-twist along a band only 
changes the framing of the corresponding framed knot by one full twist of the same sign.

\begin{Figure}[htb]{ribbon-kirby07/fig}
{}{Inserting a half-twist along a band}
\centerline{\fig{}{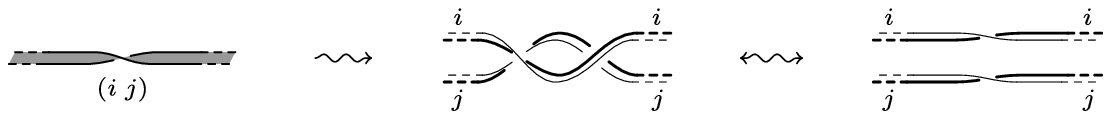}}
\vskip-3pt
\end{Figure}

A simple case by case comparison of the portion of Kirby tangles obtained using $\alpha$
and $\alpha'$ for all the possible labelings in moves \(b) and \(c), confirms that also
these moves change the resulting Kirby tangle by labeled isotopy.

What happens when we invert the polarization of a disk $D_h$ is described in Figure
\ref{ribbon-kirby08/fig}. We start with a given polarization in \(a), where we assume that
the\break framed arcs passing through $D_h$, coming either from bands attached to it or
from ribbon intersection inside it, have been isotoped all together into a canonical
position. Then, we isotope upside down the dotted unknot to obtain \(b) and use labeled
isotopy to make the arcs labeled by $i$ and those labeled by $j$ form separate negative
half twists. These two half twists add up to give a unique negative full twist in \(c).
In the end, we get \(d) by performing a positive twist on the 1-handle represented by the
dotted unknot (cf. Figures \ref{kirby-tang02/fig} and \ref{kirby-tang03/fig}). This last
diagram, possibly after canceling some of the crossovers appearing in it with preexisting
ones or with kinks coming from ribbon intersections (as in Figure \ref{ribbon-kirby04/fig}
\(c) and \(d)), is exactly what one gets by choosing the reversed polarization for $D_h$.

\begin{Figure}[htb]{ribbon-kirby08/fig}
{}{Reversing the polarization of a disk}
\centerline{\fig{}{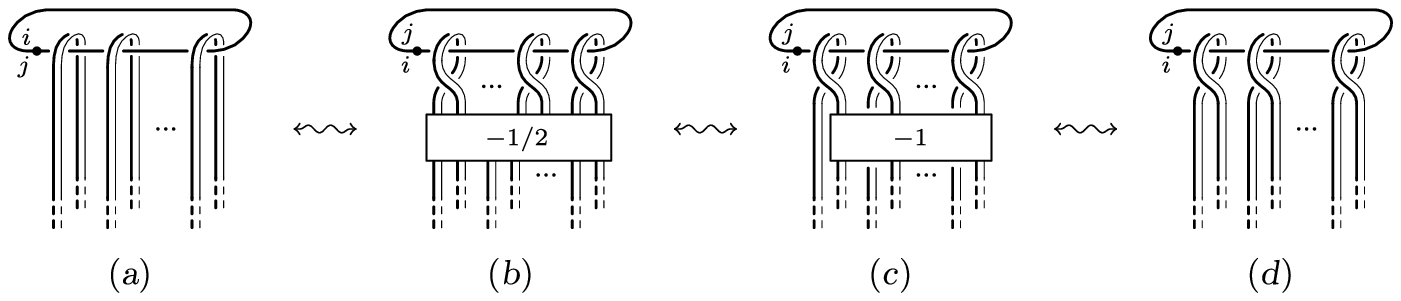}}
\vskip-3pt
\end{Figure}

Finally, the independence of $K_S$ of the adapted 1-handlebody decomposition of $S$ will
follow from Proposition \ref{1-handles/thm}, once we prove that performing on $S$ labeled
versions of the moves of Figures \ref{ribbon-surf03/fig} and \ref{ribbon-surf04/fig},
without vertical disks, corresponds to modifying $K_S$ by certain 2-deformation moves.

In all cases, since $H^0_i$ and $H^0_j$ can be assumed to be distinct 0-handles (cf. note
after Figure \ref{ribbon-surf04/fig}), we can choose the polarizations of them in such a
way that no crossover appears along $H^1_k$. Then, apparently the two moves of Figure
\ref{ribbon-surf04/fig} correspond respectively to addition/deletion of a canceling pair
of 1/2-handles and to sliding the 2-handle deriving from $H^1_l$ over the one deriving
from $H^1_k$. Similarly, in the case of move of Figure \ref{ribbon-surf03/fig} we have two
slidings involving the same 2-handles, one sliding for each of the two parallel copies of
$H^1_l$ forming the framed loop originated from it. We leave to the reader the
straightforward verification of this fact for all the four cases of Figure
\ref{ribbon-kirby04/fig}.
\end{proof}

\begin{proposition}\label{theta/thm}
$\Theta_n: \S_n \to \K_n$, defined by $\Theta_n(J_\sigma) = I_{\pi(\sigma)}$ and
$\Theta_n(S) = K_S$ as above, is a braided monoidal functor from the category of
$n$-labeled ribbon surface tangles to the category of $n$-labeled Kirby tangles, for any
$n \geq 2$.
\end{proposition}

\begin{proof}
When thinking of an $n$-labeled ribbon surface tangle $S$ as a morphism of $\S_n$, we
consider it up to the equivalence relation introduced in Definition
\ref{lrs-equivalence/def}. Hence, we have to prove that the 2-equivalence class of $K_S$
in invariant under such equivalence relation. In the light of Proposition
\ref{1-isotopy/thm} and Lemma \ref{ribbon-to-kirby/thm}, we only need to check the
2-equivalence invariance of $K_S$ when $S$ is changed by the labeled versions of the
1-isotopy moves \(S23), \(S24), \(S24) and \(S26) of Figure \ref{ribbon-surf13/fig} and by
the covering moves \(R1) and \(R2) of Figure \ref{ribbon-moves01/fig}.

Move \(S23) admits a unique labeling up to conjugation in $\Sigma_n$. The generalized
Kirby diagrams arising from the labeled ribbon surfaces involved in the resulting labeled
move are depicted in Figure \ref{ribbon-kirby09/fig} (we assume that the surfaces are
endowed with the handlebody structures of the corresponding move of Figure
\ref{ribbon-surf02/fig}). As the reader can easily realize, the two diagrams are related
by labeled isotopy.

\begin{Figure}[htb]{ribbon-kirby09/fig}
{}{Realizing move \(S23)}
\centerline{\fig{}{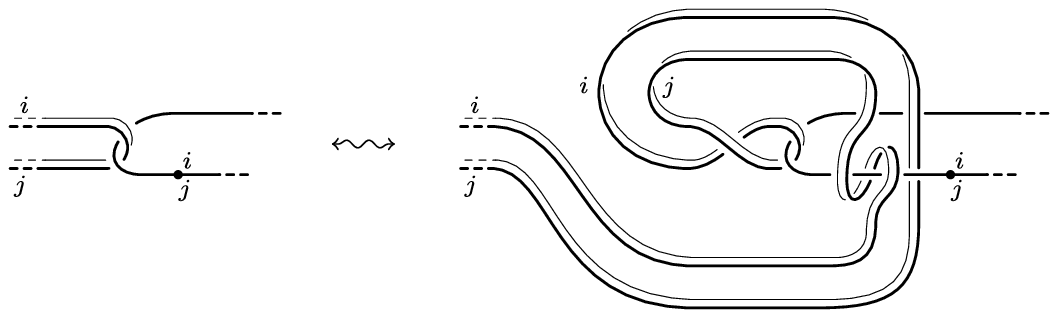}}
\end{Figure}

Moves \(S24) and \(S25) admit three distinct labelings. Namely, if $\tp{i}{j}$ is the
label of the horizontal component, then the top end of the vertical one can be labeled by
$\tp{i}{j}$, $\tp{j}{k}$ or $\tp{k}{l}$.

The first case is considered in Figure \ref{ribbon-kirby10/fig} for \(S24) and in Figure
\ref{ribbon-kirby11/fig} for \(S25). Looking at these figures, we have that the leftmost
and rightmost diagrams correspond respectively to the surfaces on the left and right side
of the move with the simplest adapted handlebody structures. The first step in both
figures is given by 1/2-handle addition, followed by a 2-handle sliding only in Figure
\ref{ribbon-kirby11/fig}. The next two steps are obtained in turn by a 2-handle sliding
and 1/2-handle cancelation. The same figures also apply to the second case, after
replacing by $k$'s all the $i$'s in the upper half and the $j$'s in the lower half (except
for the labels of the dotted line in the middle). The third case is trivial and we leave
it to the reader.

\begin{Figure}[htb]{ribbon-kirby10/fig}
{}{Realizing move \(S24)}
\centerline{\fig{}{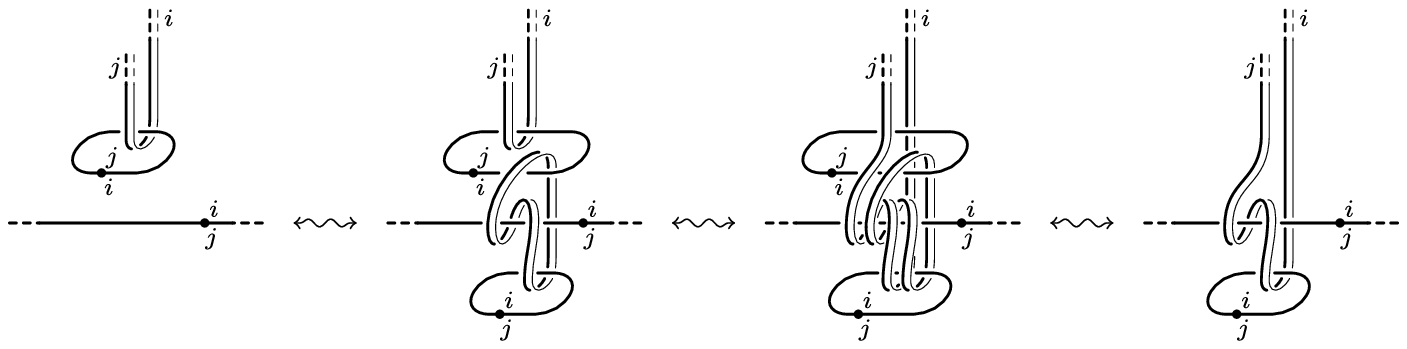}}
\vskip-3pt
\end{Figure}

\begin{Figure}[hbt]{ribbon-kirby11/fig}
{}{Realizing move \(S25)}
\centerline{\fig{}{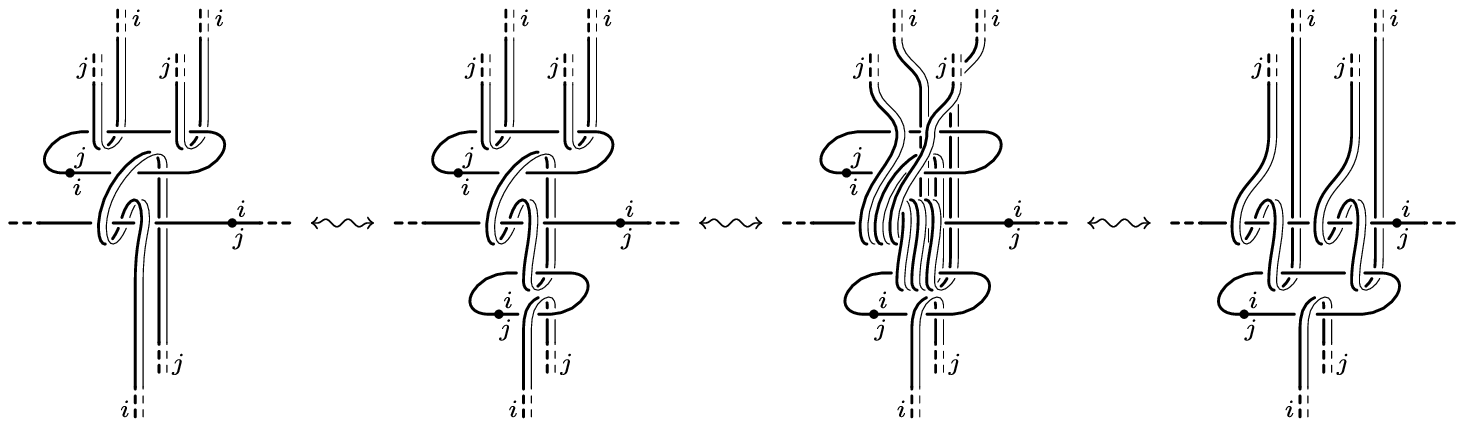}}
\end{Figure}

Finally, let us come to move \(S26), which requires a bit more work. As above, let
$\tp{i}{j}$ be the label of the horizontal band. Then, there are eighteen possible
ways\break to label the move, each one determined by the transpositions $\lambda$ and
$\rho$ labeling respectively the left and right bottom ends of the diagonal bands.

By direct inspection we see that, excluding the trivial cases when at least two of
the three ribbon intersections involve bands with disjoint monodromies, which are left to
the reader, and taking into account the symmetry of the move with\break respect to its
inverse, there are only seven relevant cases:
1)~$\lambda = \tp{i}{j}$ and $\rho = \tp{i}{j}$; 
2)~$\lambda = \tp{i}{j}$ and $\rho = \tp{i}{k}$;
3)~$\lambda = \tp{i}{k}$ and $\rho = \tp{i}{j}$; 
4)~$\lambda = \tp{i}{k}$ and $\rho = \tp{i}{k}$;
5)~$\lambda = \tp{i}{k}$ and $\rho = \tp{i}{l}$; 
6)~$\lambda = \tp{i}{k}$ and $\rho = \tp{j}{l}$;
7)~$\lambda = \tp{i}{k}$ and $\rho = \tp{k}{l}$.

Figure \ref{ribbon-kirby12/fig} regards case 1. Here, the first and last diagrams
correspond respectively to the surfaces on the left side and right side of the move with
suitable adapted handlebody structures, while the second one is related to the first by
two 2-handle slidings and to the third by labeled isotopy. This figure also applies to
case 4, after replacing by $k$'s all the $i$'s in the upper half and the $j$'s in the 
lower half (except for the labels of the dotted line in the middle), as above.

\begin{Figure}[b]{ribbon-kirby12/fig}
{}{Realizing move \(S26) -- I}
\vskip3pt
\centerline{\fig{}{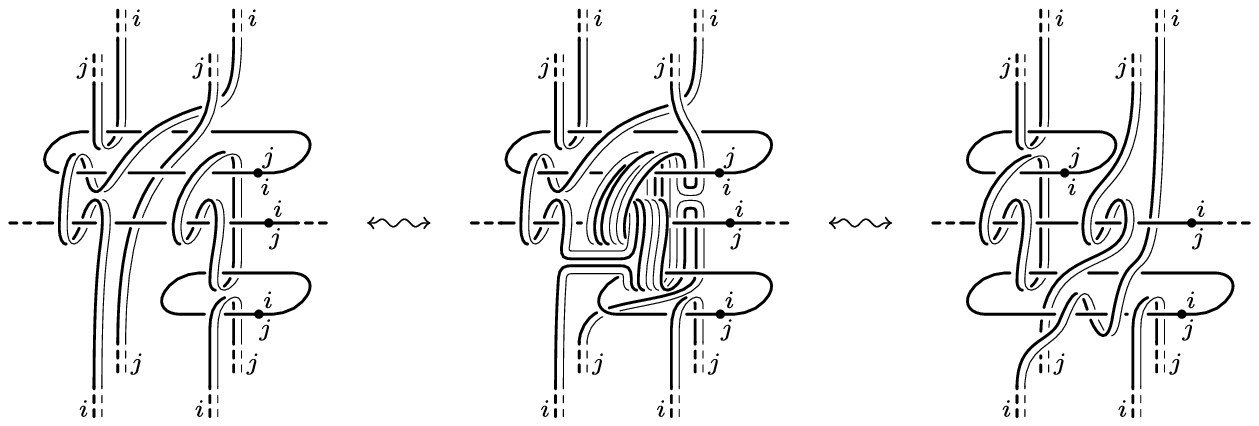}}
\end{Figure}

\begin{Figure}[htb]{ribbon-kirby13/fig}
{}{Realizing move \(S26) -- II}
\vskip6pt
\centerline{\fig{}{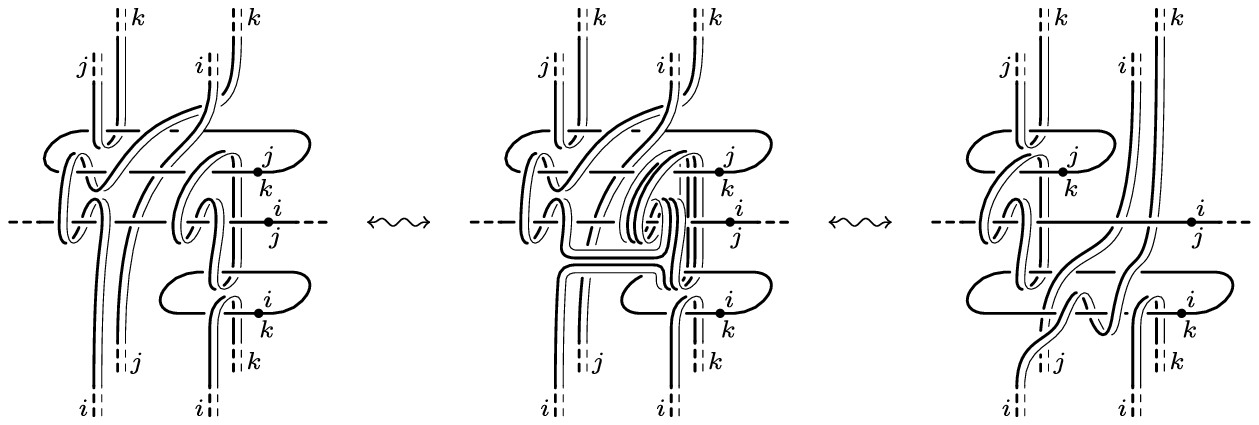}}
\end{Figure}

\begin{Figure}[htb]{ribbon-kirby14/fig}
{}{Realizing move \(S26) -- III}
\vskip6pt
\centerline{\fig{}{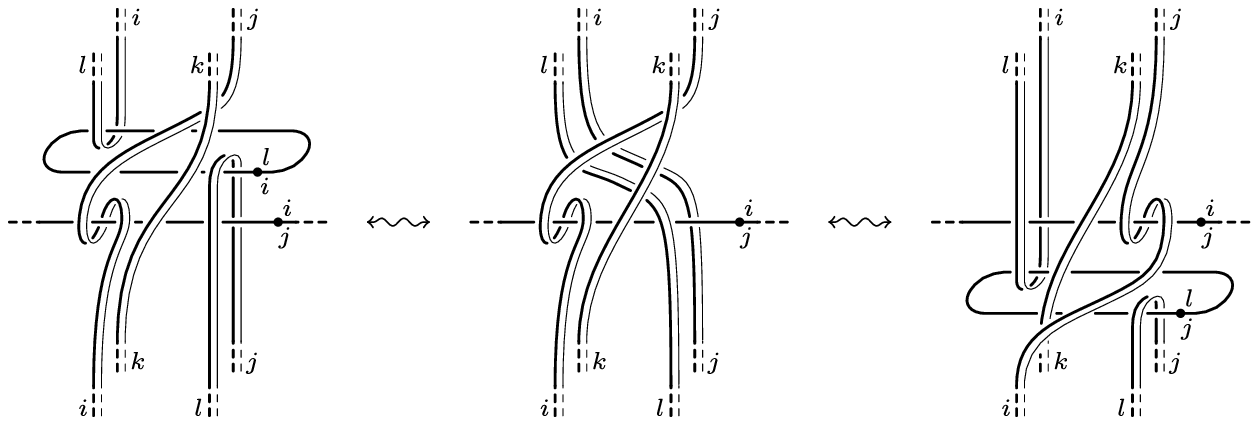}}
\end{Figure}

\begin{Figure}[htb]{ribbon-kirby15/fig}
{}{Realizing move \(S26) -- IV}
\centerline{\fig{}{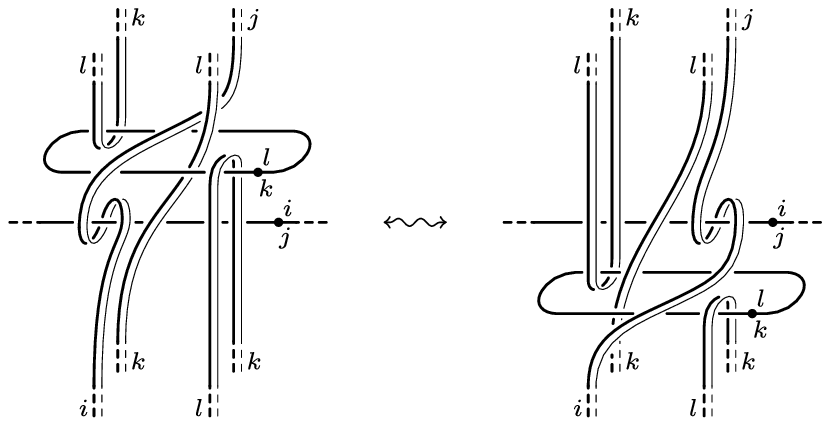}}
\end{Figure}

Similarly, Figure \ref{ribbon-kirby13/fig} concerns case 2 and, after the appropriate
label replacements, also cases 3 and 5. This time only one 2-handle sliding is needed
to relate the first two diagrams. Figures \ref{ribbon-kirby14/fig} and
\ref{ribbon-kirby15/fig} deal with the remaining cases 5 and 7. The three diagrams of
Figure \ref{ribbon-kirby14/fig} are related by 1/2-handle addition/deletion, while the two
diagrams of Figure \ref{ribbon-kirby15/fig} by labeled isotopy.

It remains to consider the covering moves \(R1) and \(R2). If $S$ and $S'$ differ by such
a move, then by making the right choices in the construction of $\Theta_n(S)$ and
$\Theta_n(S')$ we get the same result up to labeled isotopy. This is shown in Figure
\ref{ribbon-kirby16/fig} for move \(R1), while the analogous easier case of move
\(R2) is left to the reader.

\begin{Figure}[htb]{ribbon-kirby16/fig}
{}{Realizing move \(R1)}
\centerline{\fig{}{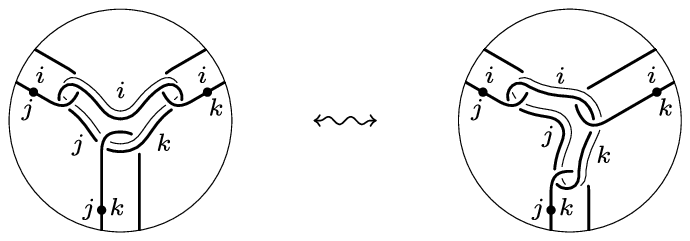}}
\end{Figure}

This completes the proof that $\Theta_n$ is well-defined as a functor from $\S_n$ to
$\K_n$, being\break the identity morphisms and the composition of morphisms clearly
preserved by it.

The fact that $\Theta_n$ is actually a braided monoidal functor follows from a
straightforward verification of the identities $\Theta_n(\gamma_{J_\sigma, J_{\sigma'}}) =
\gamma_{I_{\pi(\sigma)}, I_{\pi(\sigma')}}$ for any $\sigma,\sigma' \in \seq\Gamma_n$ and
$\Theta_n(S \diam S') = \Theta_n (S) \diam \Theta_n(S')$ for any $n$-labeled ribbon
surface tangles.
\end{proof}

It is worth remarking that the Theorem \ref{theta/thm} becomes much simpler if we limit
ourselves to require that the relative 4-dimensional 2-handlebodies represented by
$\Theta_n(S)$ and $\Theta_n(S')$ are diffeomorphic, without insisting that they are
2-equivalent. In fact, labeled isotopy between $S$ and $S'$ (instead of labeled 1-isotopy)
suffices for that, since it induces equivalence between the corresponding branched
coverings, as recalled in Section \ref{coverings/sec}. The relation between isotopy and
1-isotopy of ribbon surfaces in $B^4$ on one hand and diffeomorphism and 2-equivalence of
4-dimensional 2-handlebodies on the other hand, will be discussed in Remark 
\ref{1-isotopy/rem}.

\begin{proposition} \label{theta-c/thm}
For any $n \geq 2$, $\Theta_n$ restricts to a functor\/ $\Theta_n: \S_n^c \to \K_n^c$,
such that $\Theta_n(S\rdiam S') = 
\Theta_n (S) \rdiam \Theta_n (S')$ for any two morphisms $S$ and $S'$ in $\S_n^c$.
\end{proposition}

\begin{proof}
The proposition follows from the monoidality of $\Theta_n$ and from the identities
$\Theta_n(\Delta_{\sigma}) = \Delta_{\pi(\sigma)}$ and $\Theta_n(\gamma_{\sigma,\sigma'})
= \gamma_{\pi(\sigma),\pi(\sigma')}$, which apparently hold for any $\sigma,\sigma' \in
\seq\Gamma_n$.
\end{proof}

\begin{proposition} \label{stab-theta/thm}
For any $n > k \geq 2$ the following diagram commutes.
\vskip9pt
\centerline{\epsfbox{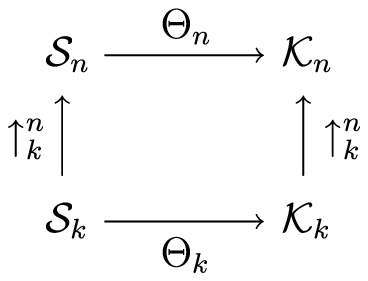}}
\vskip-6pt
\end{proposition}

\begin{proof}
This is a direct consequence of the definitions of the functors involved.
\end{proof}

In order to have an explicit form of $\Theta_n(S)$ for any $S \in \S_n$, we need to choose
a specific adapted 1-handlebody structure on $S$. By Proposition \ref{Sn-category/thm}, it
is enough to specify this choice for the elementary morphisms in $E_n$, represented by the
$n$-labeled versions of the planar diagrams in Figure \ref{ribbon-morph01/fig}. Actually,
we will consider only the elementary morphisms bases on diagrams \(a) to \(g) in that
figure, and use move \(I6) in Figure \ref{ribbon-tang01/fig} to reduce those based on
\(g') to the ones based on \(g).

This is done in Figures \ref{theta01/fig}, \ref{theta02/fig} and \ref{theta03/fig}, where
the 0-handles are denoted with lighter gray and the 1-handles with heavier gray color.
Notice that there are two types of 0-handles: the ones which are neighborhoods of vertices
of the core graph,\break and the others which divide the ribbons in such a way that none
of them contains two boundary arcs. Moreover, since any 1-handle forms at most one ribbon
intersection with any disk 0-handle, and in this case the 0-handle has a single band
attached to it, the choice of the arcs $\alpha$ is essentially unique up to isotopy, so we
omit them.

\begin{Figure}[b]{theta01/fig}
{}{Specifying $\Theta_n$ -- I ($i > j$ and $i' > j'$)}
\centerline{\fig{}{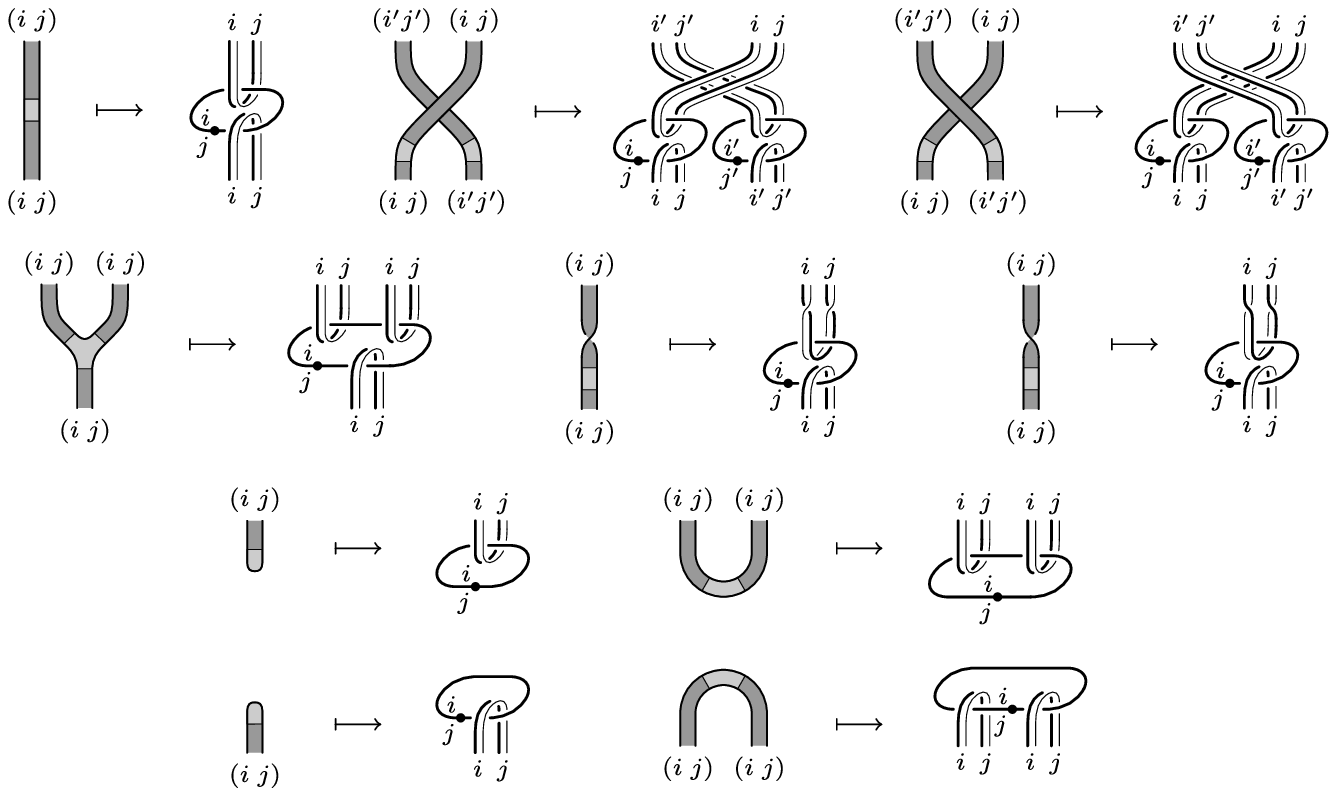}}
\vskip-3pt
\end{Figure}

\begin{Figure}[htb]{theta02/fig}
{}{Specifying $\Theta_n$ -- II
   ($i > j > k$, $h > l$ and $\{i,j\} \cap \{h,l\} = \emptyset$)}
\centerline{\fig{}{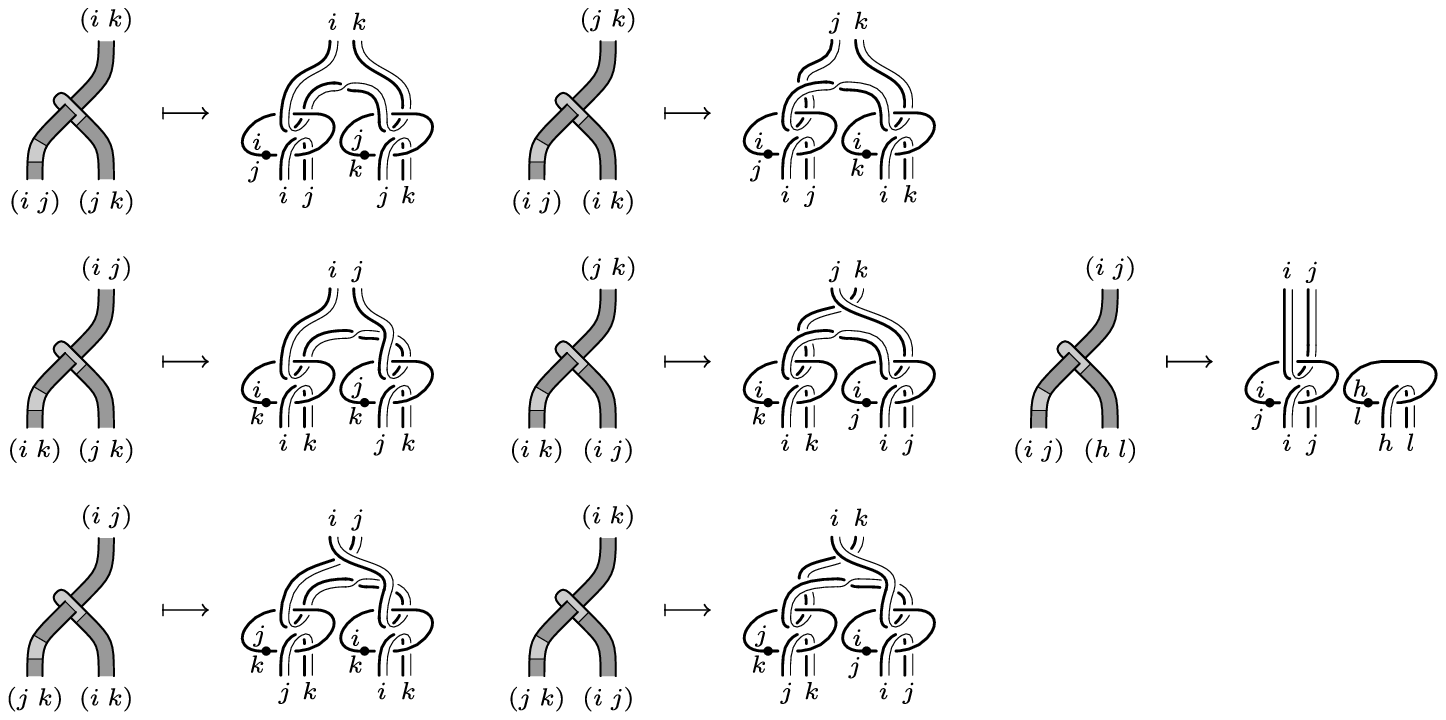}}
\vskip-3pt
\end{Figure}

\begin{Figure}[htb]{theta03/fig}
{}{Specifying $\Theta_n$ -- III ($i > j$)}
\vskip6pt
\centerline{\fig{}{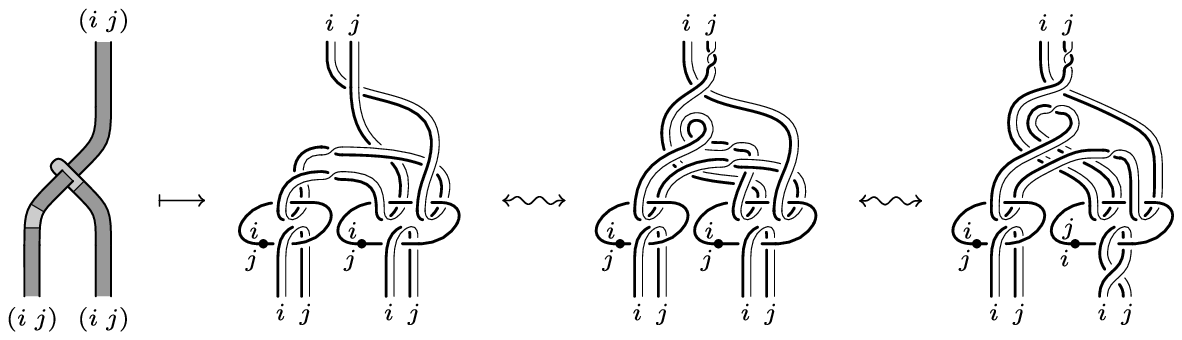}}
\vskip-3pt
\end{Figure}

In the same Figures \ref{theta01/fig}, \ref{theta02/fig} and \ref{theta03/fig}, we also
specify generalized Kirby tangles $\Theta_n(S)$ obtained from the chosen 1-handlebody
structures. For later use the image of the elementary morphism in Figure \ref{theta03/fig}
has been transformed by labeled isotopy and by inverting the polarization of the dotted
unknot on the right.

Then, given $S \in \S_n$ as an iterated product/composition of elementary diagrams, we
construct $\Theta_n(S)$ as the formal iterated product/composition of the corresponding
generalized Kirby tangles depicted in those figures. Theorem \ref{theta/thm} assures us
that such composition is well-defined as a morphism in $K_n$.

\subsection{Fullness of $\Theta_n: \S^c_n \to \K^c_n$ for $n \geq 3$%
\label{fullness/sec}}

As we will see in Proposition 3.4.4, the fullness of $\Theta_n: \S_n^c \to \K_n^c$
for any $n \geq 3$ follows from the fullness of ${\down_1^3} \circ \Theta_3: \S_3^c \to
\K_1$. Hence, we focus on the latter.

Given a Kirby tangle $K: I_{m_0} \to I_{m_1}$ in $\K_1$, we will construct a labeled 
ribbon surface tangle $S_K: J_{\sigma_{3 \red 1}} \!\diam J_{m_0} \to J_{\sigma_{3 \red 
1}} \!\diam J_{m_1}$ in $\S_3^c$ such that $\down_1^3 \Theta_3(S_K) = K$, with all 
intervals in $J_{m_0}$ and $J_{m_1}$ labeled by $\tp12$.

Actually, the notation $S_K$ is somewhat abusive, since the ribbon surface
tangle to which it refers is not uniquely determined by the Kirby tangle $K$, depending on
some choices involved in its construction (at steps 1, 3, 6 and 7). However, in the next
sections (cf. Lemma \ref{SK-welldef/thm} and Proposition \ref{trivial-state/thm}), the
uniqueness will be shown to hold for $\up_3^4 S_K$ up to equivalence of ribbon surface
tangles.

The global structure of $S_K = Q_{m_1} \!\circ R_K \circ \bar Q_{m_0}$ is illustrated in
Figure \ref{kirby-ribbon01/fig}. Here, $\bar Q_{m_0}\!: J_{\sigma_{3 \red 1}} \!\diam
J_{m_0} \to J_{\sigma_{3 \red 1}} \!\diam J_{2m_0}$ and $Q_{m_1}\!: J_{\sigma_{3 \red 1}}
\!\diam J_{2m_1} \to J_{\sigma_{3 \red 1}} \!\diam J_{m_1}$ are standard morphisms in
$\S_{3 \red 1}$ that only depend on $m_0$ and $m_1$ respectively. On the contrary, the
morphism $R_K = \id_{\sigma_{3 \red 1}} \!\diam_{\beta, \gamma, \delta} T_K: J_{\sigma_{3
\red 1}} \!\diam J_{2m_0} \to J_{\sigma_{3 \red 1}} \!\diam J_{2m_1}$ of $\S_{3 \red
1}$, which is obtained by attaching to $\id_{\sigma_{3 \red 1}} \!\diam T_K$ certain
families of bands $\beta, \gamma, \delta$ between $\id_{\sigma_{3 \red 1}}$ and $T_K$ (the
horizontal ones on the left of $T_K$ in Figure \ref{kirby-ribbon01/fig}), depends on the
internal structure of $K$ and on the choice of the bands $\beta, \gamma, \delta$.

\begin{Figure}[htb]{kirby-ribbon01/fig}
{}{The global structure of $S_K$}
\centerline{\kern50pt\fig{}{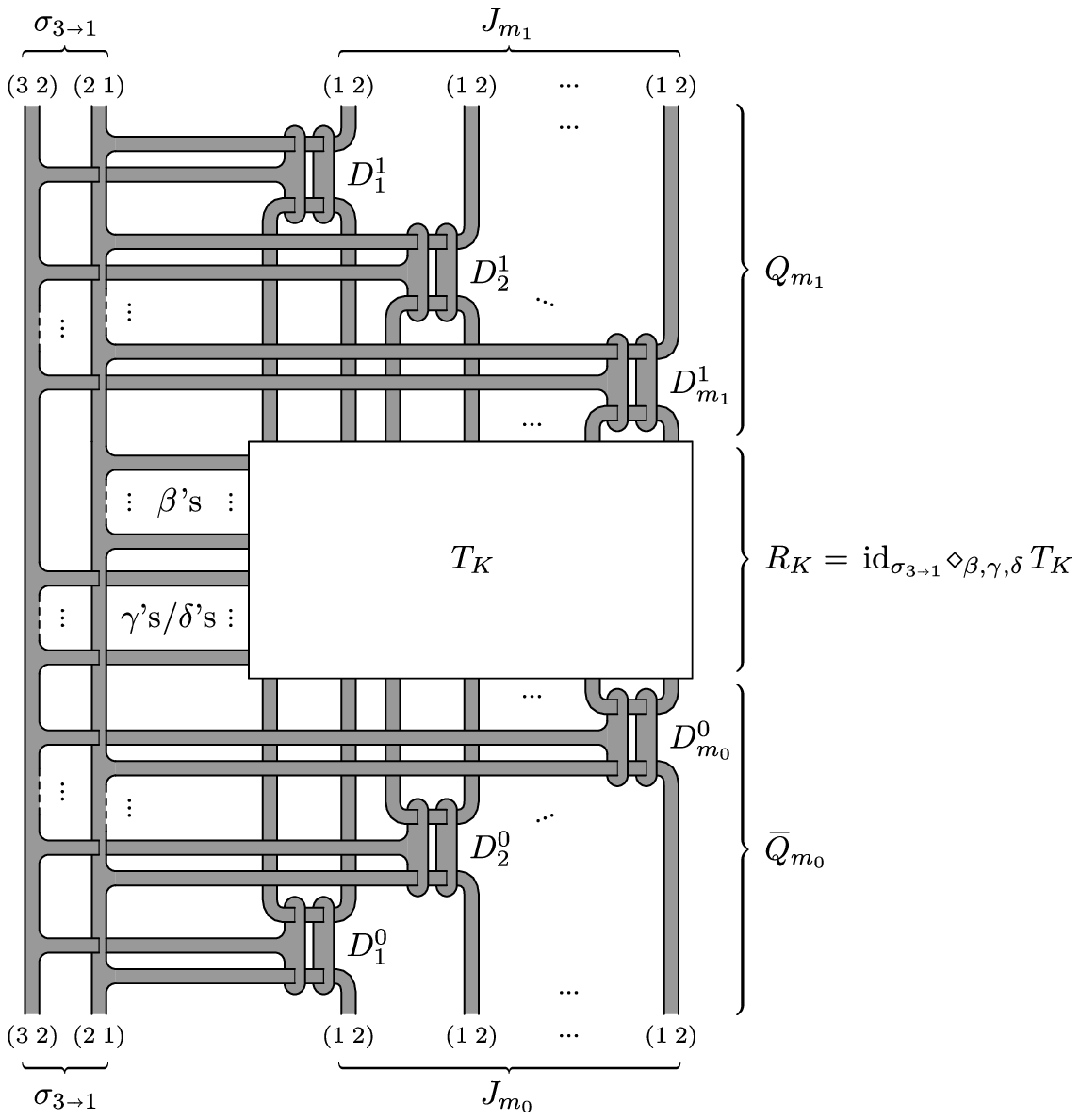}}
\vskip-6pt
\end{Figure}

The main steps in the construction of $S_K$ are explained below (namely, step 1 is a
preparatory one on $K$, step 2 concerns $\bar Q_{m_0}$ and $Q_{m_1}$, steps 3 to 5 deal
with $T_K$, steps 6 and 7 regard the families of bands $\beta, \gamma, \delta$, step 8 is 
devoted to the labeling).
\begin{itemize}
\begin{Figure}[b]{kirby-ribbon02/fig}
{}{Breaking the open framed components of $K$}
\centerline{\fig{}{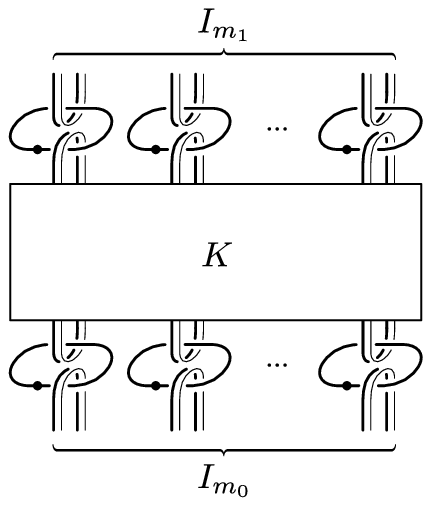}}
\vskip-3pt
\end{Figure}
\item[1)]
Represent $K$ by a strictly regular planar diagram (see Definition
\ref{strictly-regular/def}) and break all the open framed components of $K$, by inserting
dotted components near their ends at $I_{m_0}$ and $I_{m_1}$, as depicted in Figure
\ref{kirby-ribbon02/fig}. This change is achieved by a 2-deformation, namely by inserting
canceling 1/2-pairs and then performing 2-handle slidings, so the resulting Kirby tangle
can still be denoted by $K$. Moreover, denote by $D_1, \dots, D_r \subset E \times [0,1]$
the disjoint disks spanned by the dotted unknots of the original tangle $K$ (meaning
excluded the new ones we have just inserted) and by $L = L_1 \cup \dots \cup L_s$ the
framed link formed by the closed framed components of the modified tangle $K$
(corresponding to all the framed curves, both open and closed, of the original $K$).
\item[2)]
Replace the exterior of the box $K$ in Figure \ref{kirby-ribbon02/fig}, by the standard
ribbon surface tangles $\bar Q_{m_0}$ and $Q_{m_1}$ shown in Figure
\ref{kirby-ribbon01/fig}, and extend their reduction ribbons on the left of the box, to
compose them into a unique copy of $\id_{\sigma_{3 \red 1}}$. We observe that $\bar
Q_{m_0}$ and $Q_{m_1}$ are symmetric to each other (except for the fact that $m_0$ and
$m_1$ can be different), but they do no represent inverse morphisms (even if $m_0 = m_1$).
For future references, we introduce the following notations: $\tilde{\id}_{\tp21}$ for the
union of $\id_{\tp21}$ with the bands connecting it to $J_{m_0}$ and $J_{m_1}$ in $\bar
Q_{m_0}$ and $Q_{m_1}$; $\tilde{\id}_{\tp32}$ for the union of $\id_{\tp32}$ with the
parts of $\bar Q_{m_0}$ and $Q_{m_1}$ connected to it (the horizontal bands springing from
it and the vertical disks where they end); $D^0_1, \dots, D^0_{m_0}$ and $D^1_1, \dots,
D^1_{m_1}$ for the remaining vertical disks in $\bar Q_{m_0}$ and $Q_{m_1}$ respectively
(see Figure \ref{kirby-ribbon01/fig}).
\item[3)]
Choose a trivial state for the diagram of the unframed link $|L|$ (cf. Section
\ref{links/sec}), and let $L' = L'_1 \cup \dots \cup L'_s$ be the new framed link obtained
from $L$ by a regular vertical homotopy, whose corresponding unframed link $|L'|$ is
represented by that trivial state. Moreover, assume $L'$ to coincide with $L$ outside $F_1
\cup \dots \cup F_l$, where each $F_i$ is a cylinder projecting onto a small circular
neighborhood of a changing crossing. Such a cylinder $F_i$, together with the relative
portion of diagram, is depicted in Figure \ref{kirby-ribbon03/fig} \(a) and \(b), where
$j$ and $k$ may or may not be distinct. Here $C_i \subset F_i$ is a regularly embedded
disk without vertical tangencies, separating the two bands of $L \cap F_i$ and forming
four transversal intersection arcs with $L'$.
\begin{Figure}[htb]{kirby-ribbon03/fig}
{}{The disk $C_i$ and the framed links $L$ and $L'$ at a changing crossing}
\centerline{\kern-5pt\fig{}{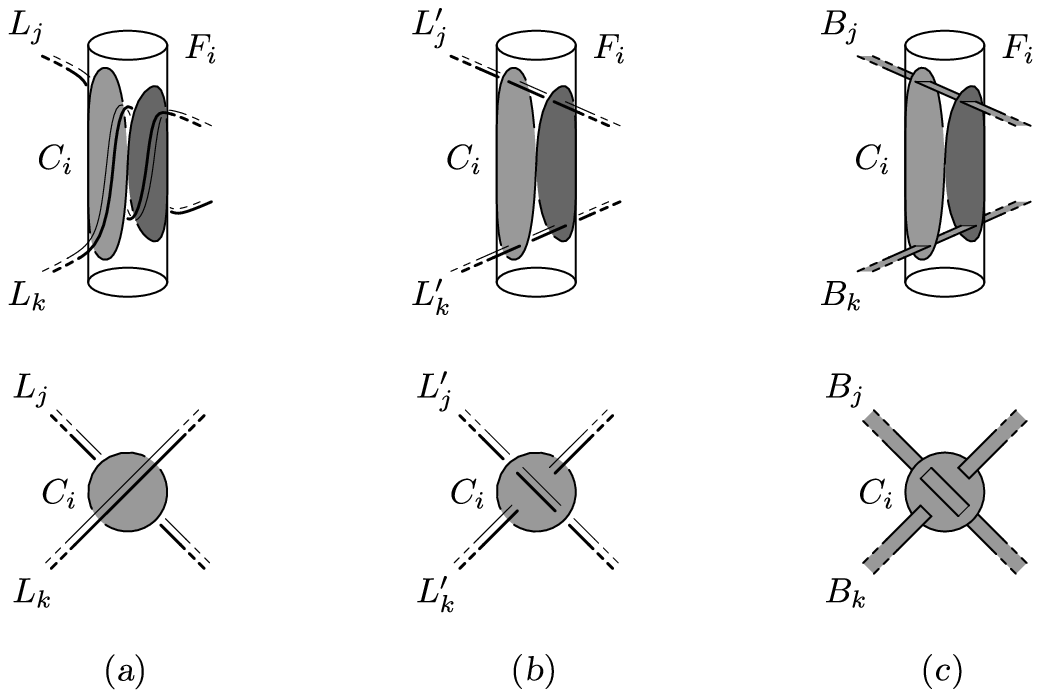}}
\vskip-3pt
\end{Figure}
\item[4)]
For each disk $D_i$, take a parallel copy $D'_i$ on one of the two sides of it (see Figure
\ref{kirby-ribbon04/fig} \(b)). Up to 1-isotopy, it does not matter what side of $D_i$ is
chosen for $D'_i$, since the moves in Figure \ref{ribbon-surf13/fig} allows
us to push one disk through the other. Denote by $G_i$ the cylinder between $D_i$ and
$D'_i$ and assume that $L' \cap G_i = L \cap G_i$ consists of trivial framed arcs as shown
in Figure \ref{kirby-ribbon04/fig} \(b). Of course, the height function of the disks $D_i$
and $D'_i$ varies according to that of such arcs.
\begin{Figure}[htb]{kirby-ribbon04/fig}
{}{The parallel disks $D_i$ and $D'_i$ for a dotted unknot of $K$}
\centerline{\kern15pt\fig{}{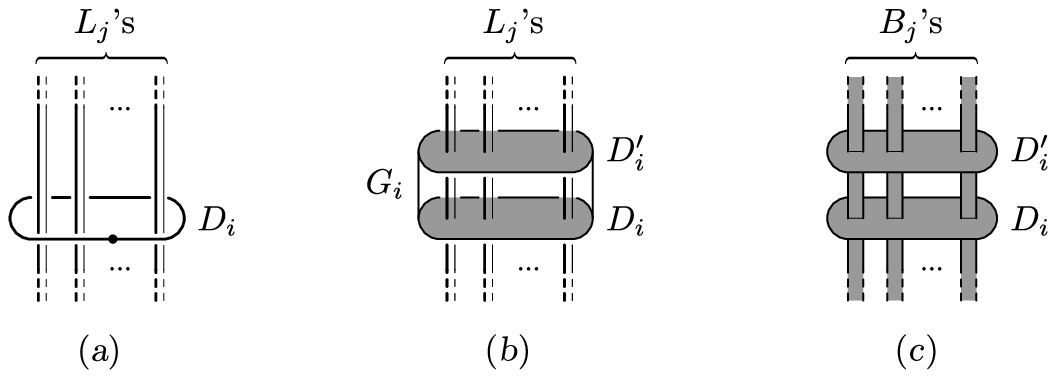}}
\vskip-3pt
\end{Figure}
\item[5)]
For each framed component $L_j$, consider a (possibly non-orientable) narrow closed band
$B_j$, whose core is the base curve $|L'_j|$ of $L'_j$ and whose framing number is
$\fr(L'_j)/2$. In particular, $B_j$ is orientable when $\fr(L'_j)$ is even, while it is
non-orientable when $\fr(L'_j)$ is odd. In both cases, $\fr(L'_j)$ coincides with
$\lk(L'_j,\Bd B_j)$ if $L'_j$ and $\Bd B_j$ are coherently oriented. The bands $B_j$ are
assumed to form with the $D_i$'s, the $D'_i$'s and the $C_i$'s only ribbon intersections,
as shown in Figures \ref{kirby-ribbon03/fig} \(c) and \ref{kirby-ribbon04/fig} \(c).
Furthermore, all the portions of the $B_j$'s outside of the box $K$ in Figure
\ref{kirby-ribbon02/fig} are assumed to be blackboard parallel and to coincide with
components of $\bar Q_{m_0}$ or $Q_{m_1}$ attached to the box $T_K$ in Figure
\ref{kirby-ribbon01/fig}.
\item[6)]
Connect each band $B_j$ to the boundary component on the right of the reduction ribbon
$\id_{\tp21}$ by a narrow band $\beta_j$, to get a connected non-singular surface $$B =
{\tilde{\id}_{\tp21}} \cup \beta_1 \cup \dots \cup \beta_s \cup B_1 \cup \dots \cup
B_s\,.$$ The bands $\beta_1, \dots, \beta_s$ are assumed to be disjoint from the $F_i$'s
and the $G_i$'s defined above and from an arbitrary family of disjoint spanning disks
$A_1, \dots, A_s \subset E \times [0,1] - \tilde{\id}_{\tp21}$ for the components $|L'_1|,
\dots, |L'_s|$ of the trivial link $|L'|$.\break The surface $B$ is assumed to be entirely
contained in the box $T_K$, except for $\tilde{\id}_{\tp21}$ and for the portions of the
$B_j$'s mentioned at the end of the previous step 5 and those of the $\beta_j$'s coming
out from the left side of the box.
\item[7)]
Connect each disk $C_i$ and each disk $D'_i$ to the boundary component on the right of the
reduction ribbon $\id_{\tp32}$ by a narrow bands $\gamma_i$ and $\delta_i$ respectively,
to get a connected non-singular surface $$C = {\tilde{\id}_{\tp32}} \cup \gamma_1 \cup
\dots \cup \gamma_l \cup \delta_1 \cup \dots \cup \delta_r \cup C_1 \cup \dots \cup C_l
\cup D'_1 \cup \dots \cup D'_r\,.$$ The bands $\gamma_1, \dots, \gamma_l$ and $\delta_1,
\dots, \delta_r$ are assumed to be disjoint from $B \cup D_1 \cup \dots \cup D_r$ and from
the interiors of the $F_i$'s and the $G_i$'s, except for the ribbon intersections with the
reduction ribbon $\id_{\tp21}$ shown in Figure \ref{kirby-ribbon01/fig}. The surface $C$
is assumed to be entirely contained in the box $T_K$, except for
$\smash{\tilde{\id}}_{\tp32}$ and for the portions the $\gamma_i$'s and $\delta_i$'s
coming out from the left side of the box.
\item[8)]
Finally, define $S_K$ to be the labeled ribbon surface tangle given by the union $$S_K = B
\cup C \cup D_1 \cup \dots \cup D_r \cup D^0_1 \cup \dots \cup D^0_{m_0} \cup D^1_1 \cup
\dots \cup D^1_{m_1}$$ with the unique labeling extending that one we already have at the
source and at the target (cf. Figure \ref{kirby-ribbon01/fig}). In particular, the
labeling around the disks $C_i$ and $D_i$ results as in Figure \ref{kirby-ribbon05/fig}.
Then, letting $T_K: J_{2m_0} \to J_{2m_1}$ be the union of the part of $B_1 \cup \dots
\cup B_s$ inside the box in Figure \ref{kirby-ribbon01/fig} with the disks $C_1, \dots,
C_l$, $D_1, \dots, D_r$ and $D'_1, \dots, D'_r$, we have $S_K = Q_{m_1} \!\circ R_K \circ
\bar Q_{m_0}$ with $R_K = \id_{\sigma_{3 \red 1}} \!\diam_{\beta, \gamma, \delta} T_K$.
\begin{Figure}[htb]{kirby-ribbon05/fig}
{}{The labeling around the disks $C_i$ and $D_i$}
\vskip-6pt
\centerline{\fig{}{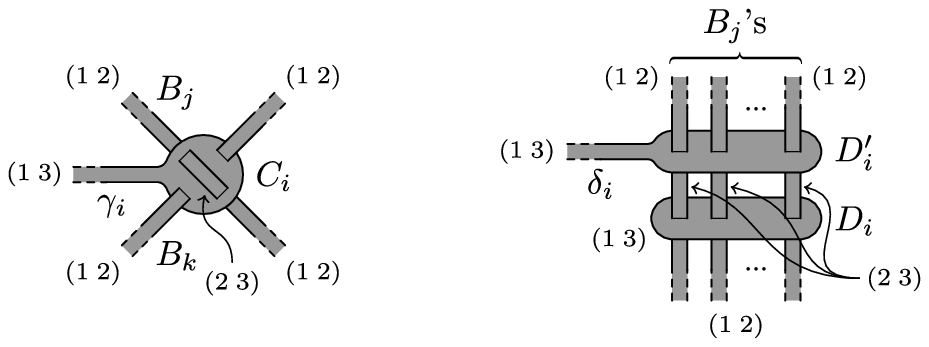}}
\vskip-3pt
\end{Figure}
\end{itemize}

\begin{remark}\label{crossing/rem}
We recall that, up to 2-deformation, any crossing in a Kirby diagram can be inverted by
adding a suitable pair of 1/2-handles, as in Figure \ref{crossing01/fig}. Actually,
up to handle trading this is the trick used in \cite{Mo80} to symmetrize framed links in
order to represent closed 3-manifolds as branched covers of $S^3$.

\begin{Figure}[htb]{crossing01/fig}
{}{Inverting a crossing up to 2-deformation}
\centerline{\fig{}{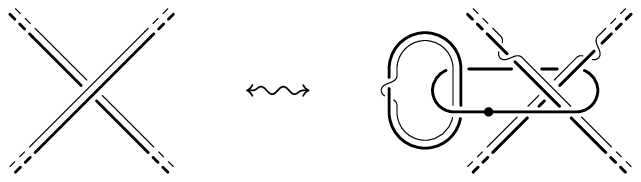}}
\end{Figure}

\begin{Figure}[b]{crossing02/fig}
{}{Interpreting the disks $C_i$ in terms as a pair of 1/2-handles}
\centerline{\fig{}{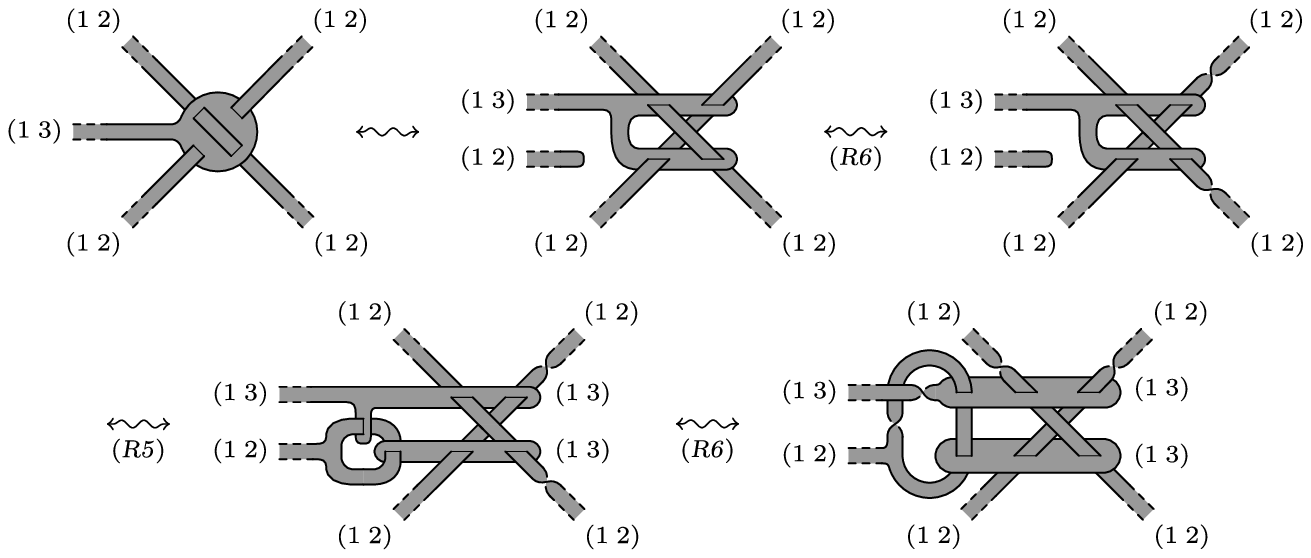}}
\vskip-3pt
\end{Figure}

Figure \ref{crossing02/fig} shows, how to interpret in this way the disks $C_i$ we insert
at the changing crossings in step 3 of our construction of $S_K$. Apart from the indicated 
moves only diagram isotopy is needed for that.
\end{remark}

\begin{remark}\label{non-monoidal/rem}
The definition of $S_K$ does not preserve the corresponding products. The most we can say
about it is that, if $K:I_{m_0} \to I_{m_1}$ and $K': I_{m'_0} \to I_{m'_1}$\break are two
Kirby tangles in $\K_1$, then $R_{K \diam K'} = T_K \diam T_{K'}$, hence $R_{K \diam K'} =
R_K \rdiam R_{K'}$ and $S_{K \diam K'} = Q_{m_1 + m'_1} \circ (R_K \rdiam R_{K'}) \circ
\bar Q_{m_0 + m'_0}$ (cf. Figure \ref{kirby-ribbon01/fig}).
\end{remark}

An example of the construction of $S_K$ is presented in Figure \ref{kirby-ribbon06/fig}.
Here, all the framed components in the Kirby tangle $K$ on the left side are blackboard
parallel, except for one half-twist needed to satisfy the last requirement in point 2 of
Definition \ref{kirby-tangle/def}. We assume such half-twist to be positive for the two
components ending at the top and negative for the component ending at the bottom, in such
a way that they cancel with the extra half-twists that appear when closing those
components in step 1 of the construction of $S_K$. Notice the half-twist introduced along
the trefoil band in $S_K$ according to step 5 of the construction.

\begin{Figure}[htb]{kirby-ribbon06/fig}
{}{An example of labeled ribbon surface tangle $S_K$}
\centerline{\fig{}{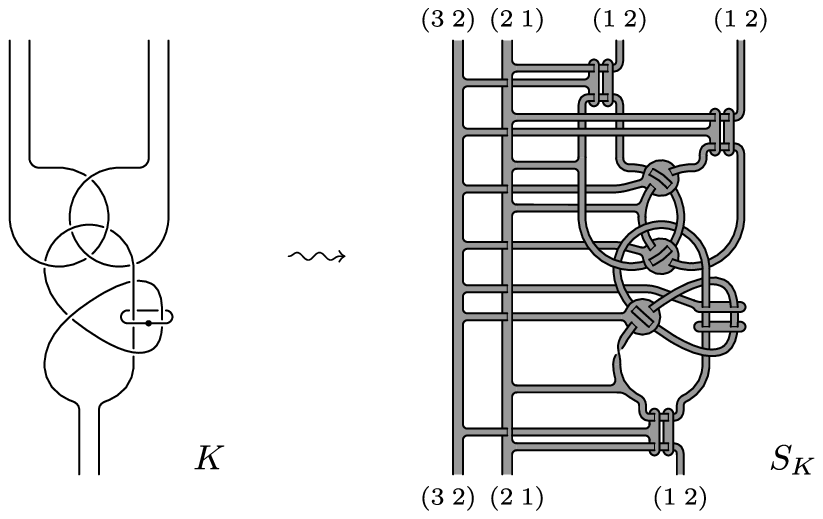}}
\vskip-3pt
\end{Figure}

\begin{proposition} \label{full-theta3/thm}
$\down_1^3 \Theta_3 (S_K) = K$ for any Kirby tangle $K \in \K_1$.
\end{proposition}

\begin{proof}
Recall that the construction of $\Theta_3(S_K)$ involves some choices. In particular, we  
need to choose an adapted 1-handlebody decomposition of $S_K$ and a polarization for the 
0-handles of such decomposition.

To this aim, we first decompose the reduction ribbons $\id_{\tp32}$ and $\id_{\tp21}$ as a
single long 0-handle and two short 1-handles connecting it with the collars of the source
and the target. We denote the two 0-handles by $R^0_{\tp32}$ and $R^0_{\tp21}$ and let
them contain the attaching arcs of all the horizontal bands attached to the ribbons,
including those in $\bar Q_{m_0}$ and $Q_{m_1}$, as well as all the ribbon intersections
between such bands and $\id_{\tp21}$.

We also decompose each of the closed bands $B_1, \dots, B_s$ defined in step 5 of the
construction of $S_K$, by putting $B_i = B_i^0 \cup B_i^1$, with $B_i^0$ a small 0-handle
containing the attaching arc of $\beta_i$ and $B_i^1$ a 1-handle attached to $B_i^0$.

Then, we consider the adapted 1-handlebody structure of $S_K$, whose 0-handles are
$R^0_{\tp32}$, $R^0_{\tp21} \cup \beta_1 \cup \dots \cup \beta_s \cup B_1^0 \cup \dots 
\cup B_s^0$, all the vertical disks in $\bar Q_{m_0}$ and $Q_{m_1}$ and the disks $D_1, 
\dots, D_r$, $D'_1, \dots, D'_r$, $C_1, \dots, C_l$ discussed in steps 3 and 4 of the 
construction of $S_K$. Consequently, as the 1-handles we have those in $\id_{\tp32}$ and 
$\id_{\tp21}$, the bands $\gamma_1, \dots, \gamma_l, \delta_1, \dots, \delta_r$ and all 
the horizontal bands in $\bar Q_{m_0}$ and $Q_{m_1}$.

Concerning the polarizations, we put the greater label on the top face for all the
0-handles, except in the case of the disks $D^0_1, \dots, D^0_{m_0}$ and $D^1_1, \dots,
D^1_{m_1}$ in $\bar Q_{m_0}$ and $Q_{m_1}$, for which we put the label 3 on the bottom
face, and the disks $D_i$ and $D'_i$, for which we put the label 3 on the face internal to
the cylinder $G_i$.

The first diagram in Figure \ref{kirby-ribbon07/fig} shows the generalized Kirby diagram
$\Theta_3(S_K)$ based on the handlebody structure described above, in the case when $K$
and $S_K$ are the Kirby tangle and the ribbon surface tangle presented in Figure
\ref{kirby-ribbon06/fig}. In this figure, as well as in the next ones of this proof, we
omit to draw the framings for the sake of readability.

\begin{Figure}[htb]{kirby-ribbon07/fig}
{}{Reducing $\Theta_3(S_K)$ to $K$: overview}
\centerline{\fig{}{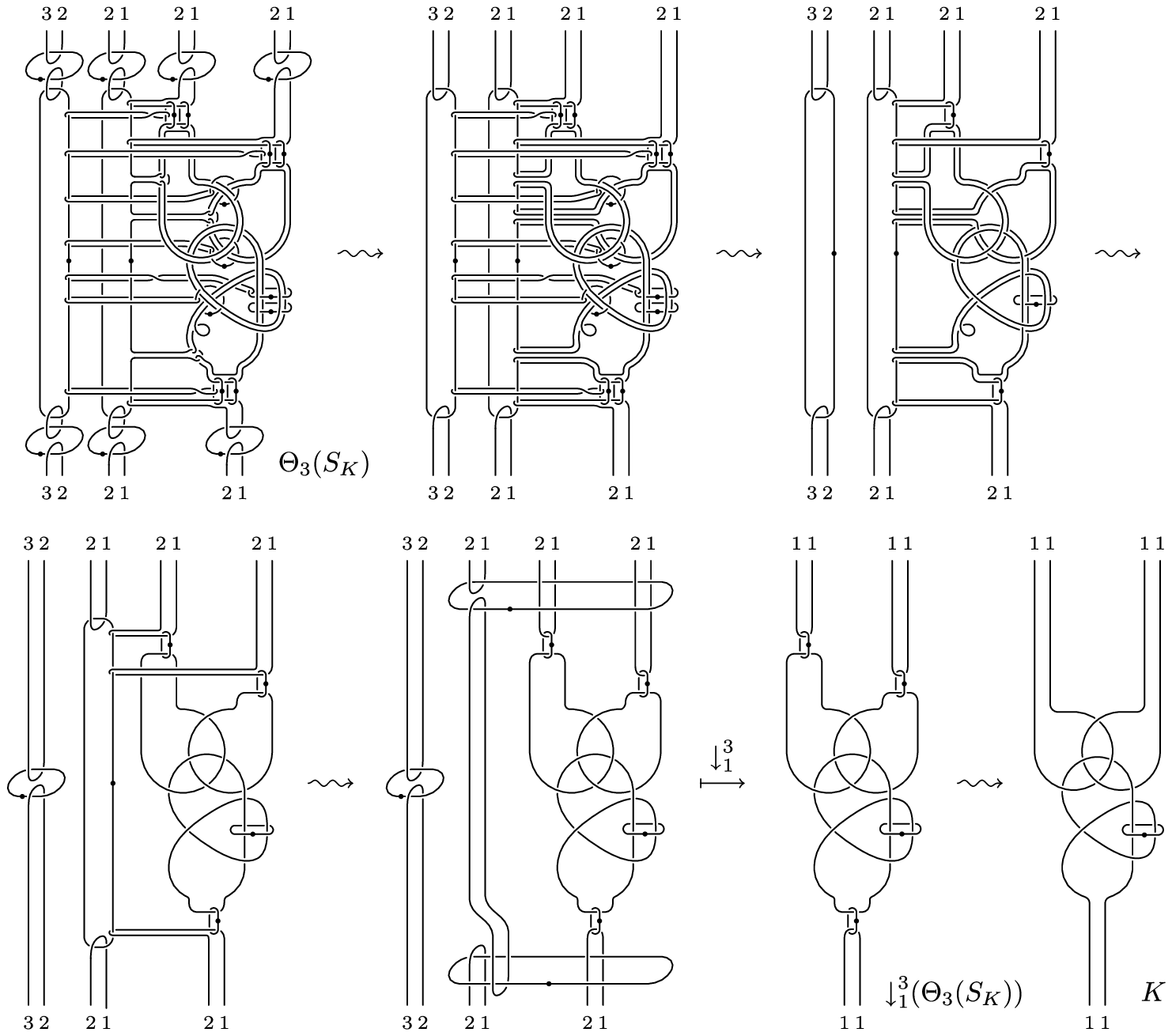}}
\end{Figure}

It remains to prove that $\Theta_3(S_K)$ can be reduced to $K$ via $\down_1^3$. The
details of this reduction process are given below. The whole process is outlined in Figure 
\ref{kirby-ribbon07/fig} for the specific example considered there.

As the first step, we slide the open framed components over the corresponding closed ones
linked to the same dotted unknot on the top and on the bottom of the diagram and then
eliminate such dotted unknots by 0/1-handle cancelation. In the same step, we also perform
at each band $\beta_j$ the isotopy modification described in Figure
\ref{kirby-ribbon08/fig} (cf. second diagram in Figure \ref{kirby-ribbon07/fig}).

\begin{Figure}[htb]{kirby-ribbon08/fig}
{}{Reducing $\Theta_3(S_K)$ to $K$: the bands $\beta_j$}
\centerline{\fig{}{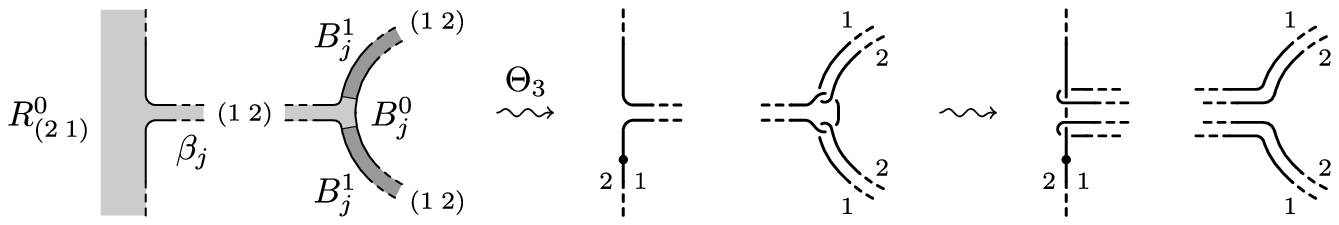}}
\vskip-3pt
\end{Figure}

As the second step, we perform a 1-handle sliding of the dotted unknots spanning the disks
$D'_i$ over the corresponding ones spanning the disks $D_i$, as it is shown in Figure
\ref{kirby-ribbon09/fig}, and then we eliminate them by 0/1-handle cancelation. We
eliminate in the same way one dotted unknot from each pair deriving from a pair of
parallel vertical disks in $\bar Q_{m_0}$ and $Q_{m_1}$.

\begin{Figure}[htb]{kirby-ribbon09/fig}
{}{Reducing $\Theta_3(S_K)$ to $K$: the disks $D_i$}
\vskip-9pt
\centerline{\fig{}{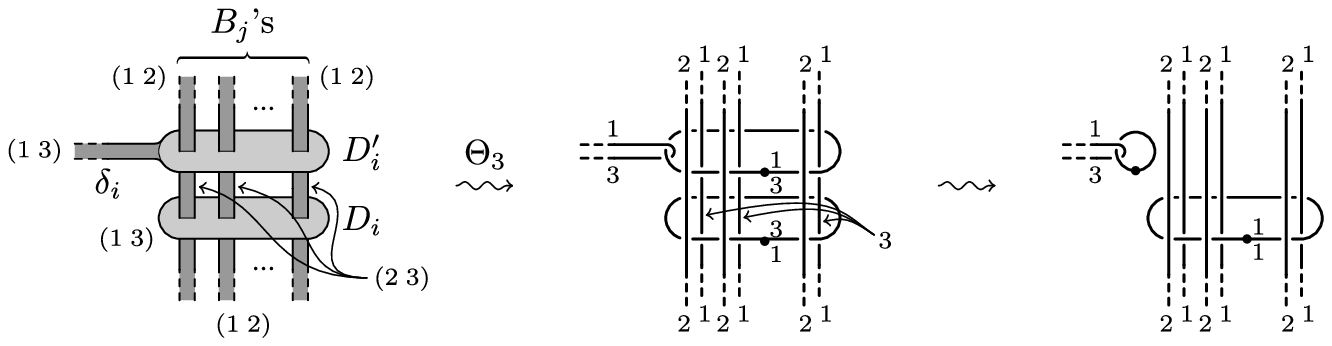}}
\vskip-3pt
\end{Figure}

Moreover, at each band crossing we modify the Kirby diagram by the crossing changes 
depicted in Figure \ref{kirby-ribbon10/fig} and then we use 0/1-handle cancelation once 
again to eliminate the dotted unknots spanning the disks $C_i$.

\begin{Figure}[htb]{kirby-ribbon10/fig}
{}{Reducing $\Theta_3(S_K)$ to $K$: the crossings}
\centerline{\fig{}{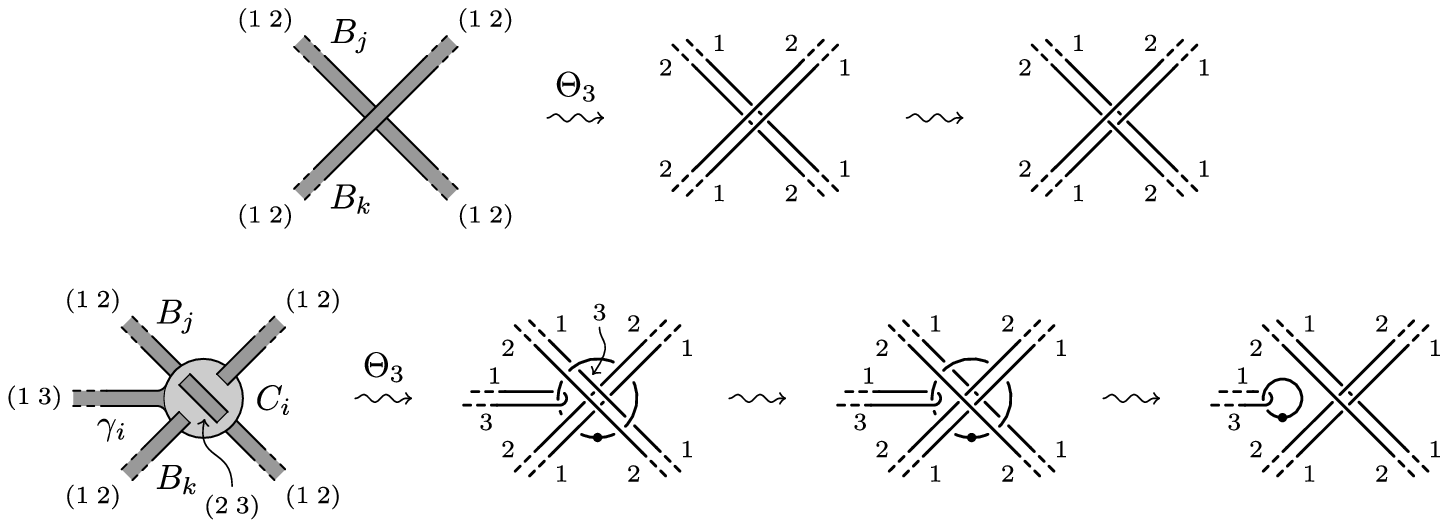}}
\vskip-6pt
\end{Figure}

At this point, we are left with a generalized Kirby diagram whose framed link $\bar L$ is
a componentwise band connected sum of the original framed link $L$, labeled by 1, and a
parallel copy $L''$ of $L'$, labeled by 2, with a certain number of extra half-twists
added to the framing. Namely, each component $\bar L_j$ of $\bar L$ is the band connected
sum of $L_j$ and the parallel copy $L''_j$ of the corresponding component $L'_j$, through
connecting bands running back and forth on the two sides of $\beta_j$, with $-2\fr(L'_j)$
extra half-twists added to the framing (cf. step 5 in the construction of $S_K$ and Figure
\ref{ribbon-kirby07/fig}). Looking at the rightmost diagrams in Figures
\ref{kirby-ribbon09/fig} and \ref{kirby-ribbon10/fig}, we see that the part of $\bar L$
labeled by 2 has been pulled up over everything else, including the dotted unknots (cf.
third diagram in Figure \ref{kirby-ribbon07/fig}).

Since the unframed link $|L'|$ is trivial and the bands $\beta_i$ are disjoint from a set
of trivializing disks for it, we can isotope $|\bar L|$ to get back $|L|$ entirely labeled
by 1, without moving the rest of the diagram (cf. fourth diagram in Figure
\ref{kirby-ribbon07/fig}).

We want to show that the last isotopy actually takes the framed link $\bar L$ to $L$, or
equivalently that the equality $\fr(\bar L_j) = \fr(L_j)$ holds for every $j = 1, \dots,
s$. In fact $\fr(\bar L_j) = \fr(L_j) + \fr(L''_j) - \fr(L'_j)$, with the last term
resulting from the $-2 \fr(L'_j)$ extra half-twists. Then, it suffices to observe that
$\fr(L''_j) = \fr(L'_j)$, being $L''_j$ a parallel\break copy of $L'_j$.

As it is illustrated by the fifth diagram in Figure \ref{kirby-ribbon07/fig}, the
resulting Kirby tangle is 2-equivalent to $\up_2^3\, (\xi^{(2,1)} \circ {\up_1^2} K \circ
(\xi^{(2,1)})^{-1})$, where $\xi^{(2,1)}$ is the natural equivalence defined in Section
\ref{K/sec} (see Lemma \ref{K-reduction1/thm}). Thus, we have $\down_1^3 \Theta_3(S_K) =
\down_1^3 \up_2^3\, (\xi^{(2,1)} \circ {\up_1^2} K \circ (\xi^{(2,1)})^{-1}) = \down_1^3
\up_2^3 \up_1^2 \down_1^2 {\up_1^2} K = K$, by Proposition \ref{K-reduction/thm} (cf. last
two diagrams in Figure \ref{kirby-ribbon07/fig}).
\end{proof}

\begin{proposition} \label{full-theta/thm}
The functor $\Theta_n: \S_n^c \to \K_n^c$ is full for any $n \geq 3$.
\end{proposition}

\begin{proof}
First of all, we observe that $S_K \in \S_3^c$. Indeed, it can be put in the form of
Figure \ref{ribbon-stab01/fig} with $n = 3$ and $k = 1$, through a quite obvious labeled
isotopy. Then, the fullness of $\Theta_3: \S_3^c \to \K_3^c$ follows from Proposition
\ref{full-theta3/thm} and the fact that the reduction functor $\down_1^3: \K_3^c \to \K_1$
is a category equivalence (Proposition \ref{K-reduction/thm}).

For $n > 3$, we have that $\Theta_n (\up_3^n S_K) = \up_3^n \Theta_3 (S_K)$ by Proposition
\ref{stab-theta/thm}. Hence, taking into account that $\up_3^n$ is a category equivalence
(Proposition \ref{K-reduction/thm}), we can derive the fullness of $\Theta_n: \S_n^c \to
\K_n^c$ from the one of $\Theta_3: \S_3^c \to \K_3^c$.
\end{proof}

\subsection{The functor $\Xi_n: \K_1 \to \S^c_n$ for $n \geq 4$%
\label{Xi/sec}}

As discussed at the beginning of the previous section, the ribbon surface tangle $S_K$
does not depend only to the given Kirby diagram $K$, but also on the various choices
involved in its definition. In particular, in step 3 of the construction of $S_K$ we
made the choice of a trivial state $|L'|$ of for the diagram of the link $|L|$ contained
in a strictly regular diagram of $K$.

Let $\check S_K$ denote any ribbon surface tangle resulting from that construction,
under the extra assumption that the trivial state $|L'|$ is vertically trivial.

In this section, we will show that $\up_3^4 \check S_K$ is well-defined up to labeled
1-isotopy and moves \(R1) and \(R2), in other words it does not depend on the choices
involved in the construction of $\check S_K$, and that it is invariant under
2-deformations of $K$.

This will give us a functor $\Xi_4: \K_1 \to \S^c_4$ defined by $\Xi_4(K) = \up_3^4
\check S_K$, from the category $\K_1$ of ordinary Kirby tangles to the category $\S_4^c =
\S_{4 \red 1}$ of 1-reducible $4$-labeled ribbon surface tangles. Then, by composing with
$\up_4^n$, we will get an analogous functor $\Xi_n: \K_1 \to \S^c_n$ also for $n > 4$.

Eventually, in the next section we will see that the functor $\Xi_4$ is full (Proposition
\ref{full-xi/thm}), which together with Proposition \ref{full-theta3/thm} implies that
$\Xi_4(K) = S_K$, where $S_K$ is constructed using an arbitrary trivial state (Proposition
\ref{trivial-state/thm}).

\begin{lemma}\label{SK-welldef/thm}
Let $K$ be a given Kirby tangle in $\K_1$. Up to labeled 1-isotopy and moves \(R1) and
\(R2), the labeled ribbon surface tangle $\up_3^4 \check S_K$ does not depend on the
choices of the strictly regular planar diagram of $K$, of the vertically trivial state of
the link $|L|$ and of the bands $\beta_1, \dots, \beta_s$, $\gamma_1, \dots, \gamma_l$ and
$\delta_1, \dots, \delta_r$ involved in the construction of $\check S_K$ (cf. steps 1, 3,
6 and 7).
\end{lemma}

\begin{proof}
First of all, we observe that the spanning disks $A_1, \dots, A_s$ in step 6 of the 
construction of $\check S_K$ can be always assumed to satisfy the following conditions 
(possibly after a small perturbation of them):
\begin{itemize}
\item[\(a)] 
$A_i \cap B_i$ consists of $|L'_i|$ and a certain number of disjoint clasps connecting
$|L'_i|$ with the boundary of $A_i$, in such a way that $A_i \cup B_i$ collapses to $A_i$, 
for every $i = 1, \dots, s$; while $A_i \cap B_j = \emptyset$ for every $i \neq j$;
\item[\(b)]
$A_1 \cup \dots \cup A_s$ forms with each $C_i$ four clasps and some (possibly none) 
ribbon intersections, as shown in Figure \ref{xi01/fig} (left side);
\item[\(c)] 
$A_1 \cup \dots \cup A_s$ forms with each $D_i \cup D'_i$ two clasps for each intersection 
point in $|L| \cap D_i = |L'| \cap D_i$ and some (possibly none) ribbon intersections, as 
shown in Figure \ref{xi01/fig} (right side);
\item[\(d)]
the $\gamma_i$'s and the $\delta_i$'s may pass through $A_1 \cup \dots \cup A_s$ forming 
only ribbon intersections with them.
\end{itemize}

\begin{Figure}[htb]{xi01/fig}
{}{Intersections between $A_1 \cup \dots \cup A_s$ and the disks $C_i$, $D_i$ and $D'_i$}
\centerline{\fig{}{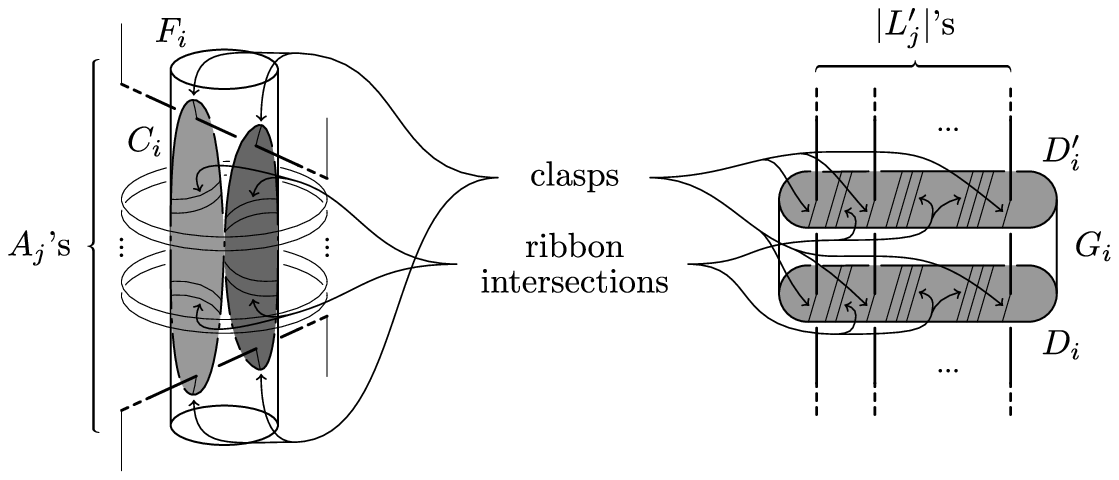}}
\vskip-3pt
\end{Figure}

At this point we pass to the core of the proof. We prove that $\check S_K$ is independent
(up to labeled 1-isotopy and moves \(R1) and \(R2)) on the choices listed in the
statement, by proceeding in the reverse order and assuming each time that all previous
choices have been fixed. By Proposition \ref{moves-aux/thm}, in addition to moves \(R1)
and \(R2), we can use also the moves \(R3) to \(R6) in Figure \ref{ribbon-moves02/fig}.

Concerning the $\gamma_i$'s and the $\delta_i$'s, it suffices to prove that labeled
1-isotopy and the moves above enable us to change them one by one.

We start by showing how to replace a band $\gamma_i$ by a different band $\gamma'_i$.
Observe that, up by a 3-dimensional isotopy, we can assume $\gamma_i$ and $\gamma_i'$ to
be disjoint. Then, Figure \ref{xi02/fig} illustrates the sequence of moves realizing the
replacement. First we modify the diagram in \(a) to get the labeled ribbon surface tangle
in \(c), where both $\gamma_i$ and $\gamma_i'$ are attached to $C_i$, but $\gamma_i'$
passes through two small disks $D$ and $D'$ labeled $\tp34$, with $D'$ attached by a
narrow band $\delta$ to $\id_{\tp43}$. The tangle in \(c) is obviously equivalent to
$\check S_K$, since we can use move \(R3) to cut $\gamma_i'$ (see diagram \(b)), retract
the resulting tongue to $\id_{\tp32}$ and then retract the tongue $D' \cup \delta$ to
$\id_{\tp43}$. Now, since the label $\tp12$ of the bands $B_j$ and $B_k$ is disjoint from
$\tp34$, we can move the disks $D$ and $D'$ in turn, to let $C_i$ pass through them by
1-isotopy and four moves \(R2), obtaining in this way \(d) and then \(e). Finally, the
same procedure described above to see the equivalence between \(a) and \(c), but with the
roles of $\gamma_i$ and $\gamma'_i$ interchanged, gives the equivalence between \(e) and
\(f).

\begin{Figure}[htb]{xi02/fig}
{}{Independence of $\check S_K$ on the bands $\gamma_i$}
\centerline{\fig{}{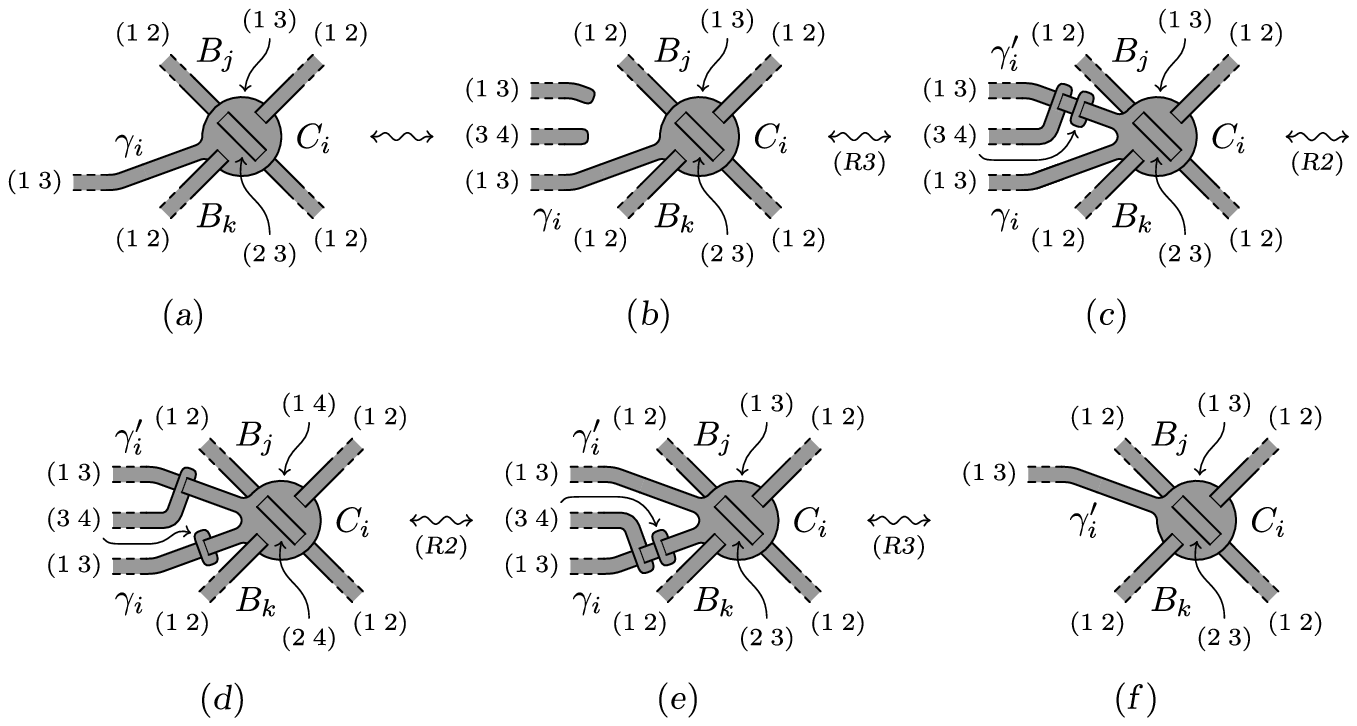}}
\vskip-3pt
\end{Figure}

The same argument works to change a band $\delta_i$ into a different band $\delta'_i$. The
process begins and ends like above, with $\delta_i$ and $\delta'_i$ in place of $\gamma_i$
and $\gamma'_i$, while the central part is illustrated by Figure \ref{xi03/fig}. In this
case, two \(R2) moves for each arc passing through $D_i$ are needed in order to move the
disks $D$ and $D'$ from $\delta_i'$ to $\delta_i$.

\begin{Figure}[htb]{xi03/fig}
{}{Independence of $\check S_K$ on the bands $\delta_i$}
\centerline{\fig{}{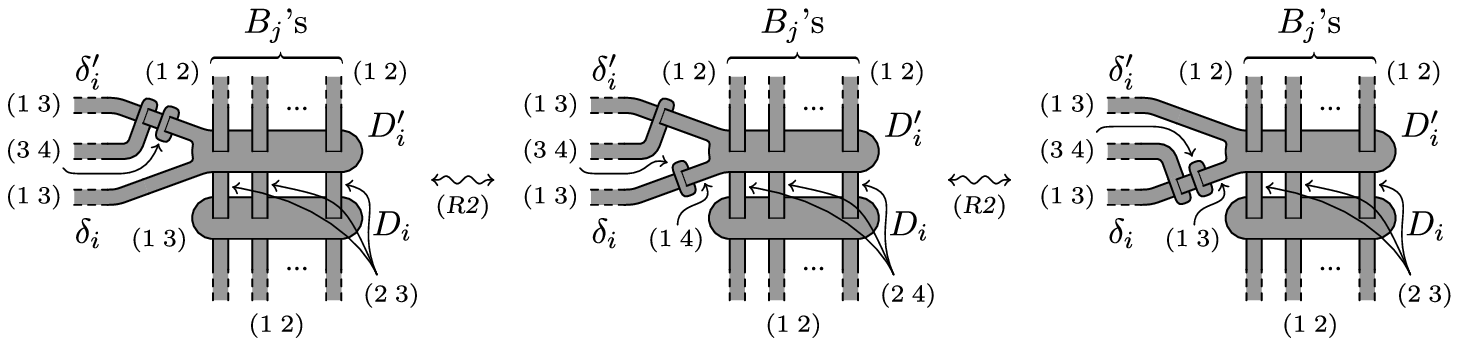}}
\vskip-3pt
\end{Figure}

The proof of the independence of $\check S_K$ on the $\beta_i$'s is more involved, but
still based on the same idea. We want to replace certain bands $\beta_1, \dots, \beta_s$
with different bands $\beta'_1, \dots, \beta'_s$. According to step 6 in the construction
of $\check \S_k$, there exist two families of disjoint spanning disks $A_1, \dots, A_s$
and $A'_1, \dots, A'_s$ for the link $|L'|$ which are disjoint from the $\beta_i$'s and
from the $\beta'_i$'s respectively.

We first consider the special case when $A_i = A'_i$ for every $i = 1, \dots, s$. In this
case (but not in general, as we will see) we can replace the bands one by one.

The idea of the proof that a given band $\beta_i$ can be replaced with a different band
$\beta'_i$, is presented in Figure \ref{xi04/fig}, even if that figure is much more
sketchy than Figure \ref{xi02/fig}. In fact, instead of the disk $C_i$ we have here the
complex $A_i \cup B_i$ that can be large and complicated, although still collapsible, as
it follows from assumption \(a) at the beginning of the proof. The first step, to get the
labeled ribbon surface tangle in Figure \ref{xi04/fig} \(a) from $\check S_K$, and in
particular the disks $D$ and $D'$, and the last step, to get the final result from the
labeled ribbon surface tangle in \(e), are once again similar to the above ones. But this
time the bands $\beta_i$ and $\beta'_i$ in place of $\gamma_i$ and $\gamma'_i$ originate
from $\id_{\tp21}$ and are labeled $\tp12$, while the band $\delta$ connects the disk $D'$
to $\id_{\tp43}$ passing through $\id_{\tp32}$ to form a ribbon intersection. We obtain in
this way the label $\tp24$ for the disks $D$ and $D'$, which is disjoint from the label
$\tp13$ of the disks $D_j$, $D_j'$ and the bands $\gamma_j$ and $\delta_j$. This permits
to get \(e) from \(a) by moving $D$ and $D'$ in turn by 1-isotopy and some moves \(R2) and
\(R4), within a neighborhood of $A_i$ and letting $A_i$ pass through them. In particular,
the moves \(R2) are needed when the disks cross the clasps that $A_i$ forms with the
$D_j$'s, the $D'_j$'s and the $C_k$'s (cf. assumptions \(b) and \(c) at the beginning of
the proof). These clasps always appear in pairs, each pair being formed with one of the
$C_j$'s or with a $D_k$ and the corresponding $D'_k$ or (cf. Figure
\ref{kirby-ribbon05/fig}). Each pair looks like the only one depicted on the right part of
the diagrams in Figure \ref{xi04/fig}. Comparing \(b) and \(c), we see that the disk $D$
can be pushed beyond such pair by two \(R2) moves. On the other hand, the moves \(R4) are
used in the obvious way to push $D$ beyond each ribbon intersection that the interior of
$A_i$ forms with the $\gamma_j$'s and the $\delta_k$'s (cf. assumption \(d) at the
beginning of the proof). Finally, once the disk $D$ has been completely moved from
$\beta_i$ to $\beta'_i$ (diagram \(d) in the figure), we move in the same way the disk
$D'$ to get \(e) and we are done.

\begin{Figure}[htb]{xi04/fig}
{}{Independence of $\check S_K$ on the bands $\beta_i$}
\centerline{\fig{}{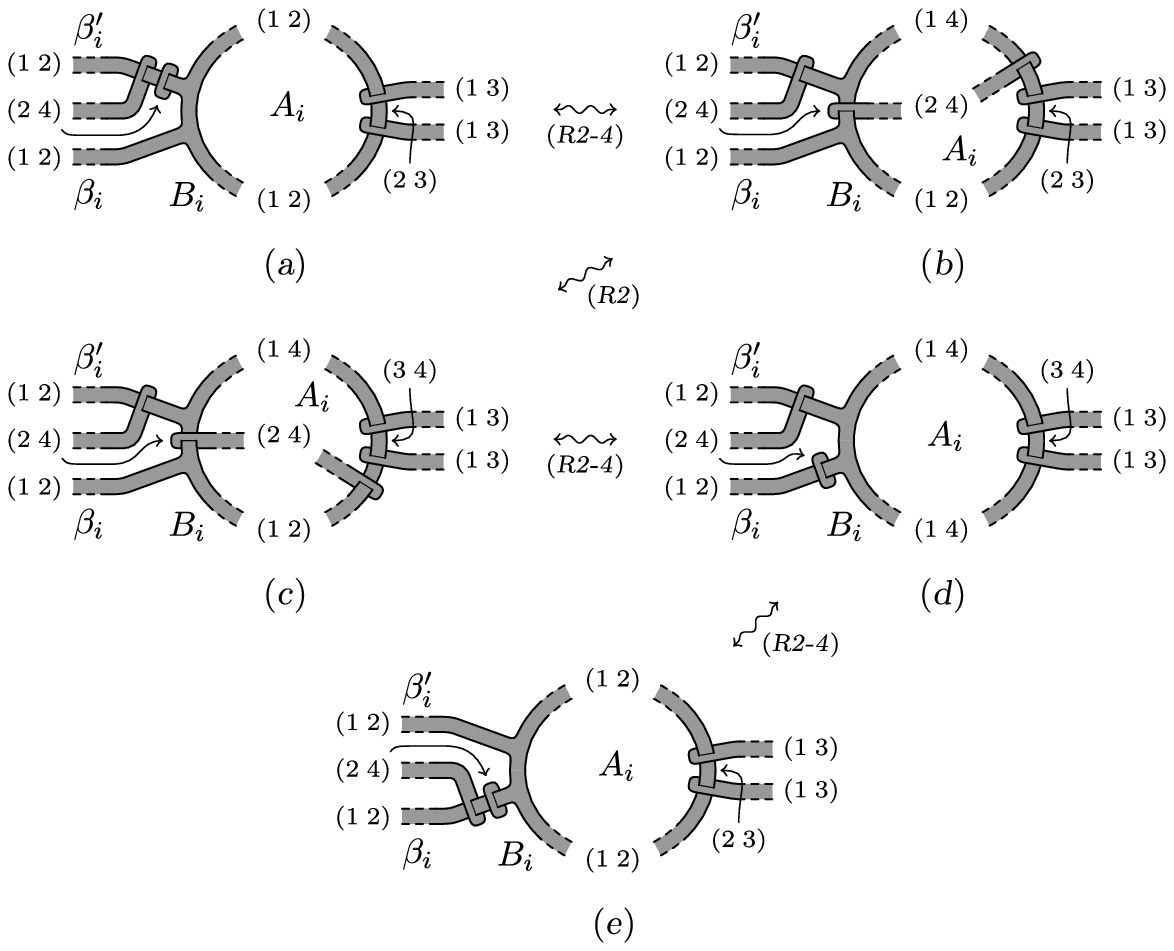}}
\vskip-3pt
\end{Figure}

\begin{statement}
{Remark}\label{gen-beta/rem}\ns
The above argument without any modification with the exception of some extra 1-isotopy
moves, proves that in the special case $\beta_i$ can be replaced by $\beta'_i$ even
if the bands $\beta_1, \dots, \beta_s$ and $\beta'_1, \dots, \beta'_s$ form ribbon
intersections with the disks $A_j$ with $j \neq i$. Of course, in this case the labeled
ribbon surface tangles involved are not $\check S_K$'s, but we need to extend our argument
to include such kind of tangles, since they will appear in the following as intermediate
stages between genuine $\check S_K$'s.
\end{statement}

In order to deal with the general case, we need a preliminary adjustment. Up to a small
isotopy inside a tubular neighborhood of $|L'|$, we separate the intersection $\Int(A_1
\cup \dots \cup A_s) \cap \Int(A'_1 \cup \dots \cup A'_s)$ from $|L'|$. After that, we put
$\Int(A_1 \cup \dots \cup A_s)$ and $\Int(A'_1 \cup \dots \cup A'_s)$ in general position,
so that they meet transversally along a finite family of closed curves. Then, we proceed
by induction on the number of components of $\Int(A_1 \cup \dots \cup A_s) \cap \Int(A'_1
\cup \dots \cup A'_s)$.

The induction starts from the case when such intersection is empty. In this case, we have
pairwise disjoint spheres $S_i = A_i \cup A'_i$ with $i = 1, \dots, s$. Then, there exist
new bands $\beta''_i$ between $B_i$ and $\id_{\tp21}$, such that: 1) $\beta''_i \cap S_j =
\emptyset$ if either $i = j$ or $i \neq j$ and $S_i$ lies in the exterior $E(S_j)$ of
$S_j$; 2) $\beta''_i \cap S_j$ consists of a single ribbon intersection if $i \neq j$ and
$S_i$ lies in the interior $I(S_j)$ of $S_j$. We choose an indexing such that case 2
possibly happens only for $i > j$. The $\beta''_i$'s are not legitimate bands for a
genuine $\check S_K$, but according to the remark above about the proof of the special
case, we can still apply that argument to change the $\beta_i$'s into the $\beta''_i$'s
one by one, by using the $A_i$'s, provided we proceed in the order given by the chosen
indexing. After that, we can change the $\beta''_i$'s into the $\beta'_i$ in the same way,
but using the $A'_i$'s in place of the $A_i$'s and proceeding in the reversed order.

We are left with the inductive step. Choose a closed curve $C$ among the components of
$\Int(A_1 \cup \dots \cup A_s) \cap \Int(A'_1 \cup \dots \cup A'_s)$, which is contained
in $\Int A_i \cap \Int A'_j$ and is an innermost one in $\Int A'_j$ for $i,j \leq s$. We
apply the usual cut and paste technique to $A_i$ to remove $C$ from the intersection.
Namely, if $D \subset A_i$ and $D' \subset A'_j$ are the subdisks spanned by $C$, then we
replace a subdisk of $A_i$ slightly larger of $D$ with a disk parallel to $D'$, to get a
new spanning disk $A''_i$. Then, by putting $A''_k = A_k$ for any $k \neq i$, we have a
new family $A''_1, \dots, A''_s$ of spanning disks for $|L'|$, whose interiors intersect
$\Int(A'_1 \cup \dots \cup A'_s)$ in a smaller number of curves than the original
disks\break $A_1, \dots, A_s$. Let $\beta''_1, \dots, \beta''_s$ be any family of bands
disjoint from the disks $A''_1, \dots, A''_s$ (observe that $\beta_1, \dots, \beta_s$ may
meet $A''_i$ in $D'$). Since the disks $A''_1, \dots, A''_s$ can be obviously perturbed to
make their interiors disjoint from $\Int(A_1 \cup \dots \cup A_s)$, the starting step of
the induction applies to transform the bands $\beta_1, \dots, \beta_s$ into $\beta''_1,
\dots, \beta''_s$. Finally, these latter can be made into the bands $\beta'_1, \dots,
\beta'_s$, by the induction hypothesis. This completes the proof of the independence of
$\check S_K$ on the $\beta_i$'s.

Now we pass to the vertically trivial state $|L'|$. Recall that we are thinking of it as a
vertically trivial link, that is a vertically trivial diagram together with a compatible
height function. Of course, different choices of the height function, compatible with the
same diagram, are related by a vertical diagram isotopy and such an ambient isotopy can be
used to relate the resulting labeled ribbon surface tangles.

So, we have only to consider the case of different choices for the vertically trivial
state $|L'|$. According to Proposition \ref{vert-state/thm}, any two such choices are
related by a sequence of the following two moves: the first one is a single self-crossing
change of a given component; the second one is a simultaneous change of all the crossings
between two vertically adjacent components.

We first address the second move, which changes the vertical order of two vertically
adjacent components. Thanks to what we have already proved, we can choose the disks $A_1,
\dots, A_s$ to be perturbations of those generated by the horizontal intervals with
endpoints in $|L'|$ (cf. Section \ref{links/sec}) and the bands $\beta_1, \dots, \beta_s$
to satisfy the following conditions:
\begin{itemize}
\item[1)]
the intervals $[a_i,b_i] = h(A_i \cup \beta_i)$, with $i= 1, \dots, s$ and $h$ being the
height function in the 3-dimensional space, are pairwise disjoint;
\item[2)]
the planar diagram of the band $\beta_i$ is disjoint from the projection in the diagram's
plane of the corresponding disk $A_i$, for every $i = 1, \dots, s$.
\end{itemize}

Assuming that the components of $|L'|$ are numbered according to their vertical order, let
$|L'_i|$ and $|L'_{i+1}|$ be the two components involved in the move. We attach to $L_i$
an arbitrary auxiliary band $\beta'_i$ as in Figure \ref{xi04/fig} \(a) and then perform
the indicated moves which lead to diagram \(d) in the same figure. As a result, the
labeling of the entire band $B_i$ changes from $\tp12$ to $\tp14$, while the labeling of
$B_{i+1}$ is left unchanged. This allows us to perform the required crossing changes,
bringing $B_i$ on the top of $B_{i+1}$. Figure \ref{xi05/fig} describes the main steps in
the realization of such crossing changes in the case when a disk $C_j$ is present in the
original crossing. Here, apart from 1-isotopy, we have used one move \(R4) relating the
second and the third diagram. The other case, when the disk $C_j$ is not present in the
original crossing, is obtained by the same steps in the reverse order with the roles of
$B_i$ and $B_{i+1}$ (and their labels) exchanged. Once all the crossing changes have been
performed, we restore the original labeling of $B_i$ and cancel the auxiliary band
$\beta'_i$, by reversing the steps from \(a) to \(d) of Figure \ref{xi04/fig}. At this
point, we are in position to push the disk $A_i$ and the band $B_i$ in the height interval
$\left] b_{i+1}, a_{i+2} \right[\,$ through a vertical isotopy (put $a_{i+2} = \infty$ if
$i+1 = s$). We can assume that during this isotopy $A_k$, $B_k$ and $\beta_k$ with $k \neq
i$ are kept fixed while $A_i$ and $B_i$ are being moved and the necessary vertical
deformations are performed on $\beta_i$ and on those $C_j$'s, $D_j$'s and $D'_j$'s which
form ribbon intersections with $B_i$. The only problem which can arise here, is that after
such isotopy $\beta_i$ could intersect $A_{i+1}$, so that the resulting labeled ribbon
surface tangle would not be a genuine $\check S_K$. To fix this, we can replace the
deformed $\beta_i$ by a legitimate band, thanks to the remark at page
\pageref{gen-beta/rem}.

\begin{Figure}[htb]{xi05/fig}
{}{Independence of $\check S_K\!$ on the vertical order of the components of $|L'|\!$}
\centerline{\fig{}{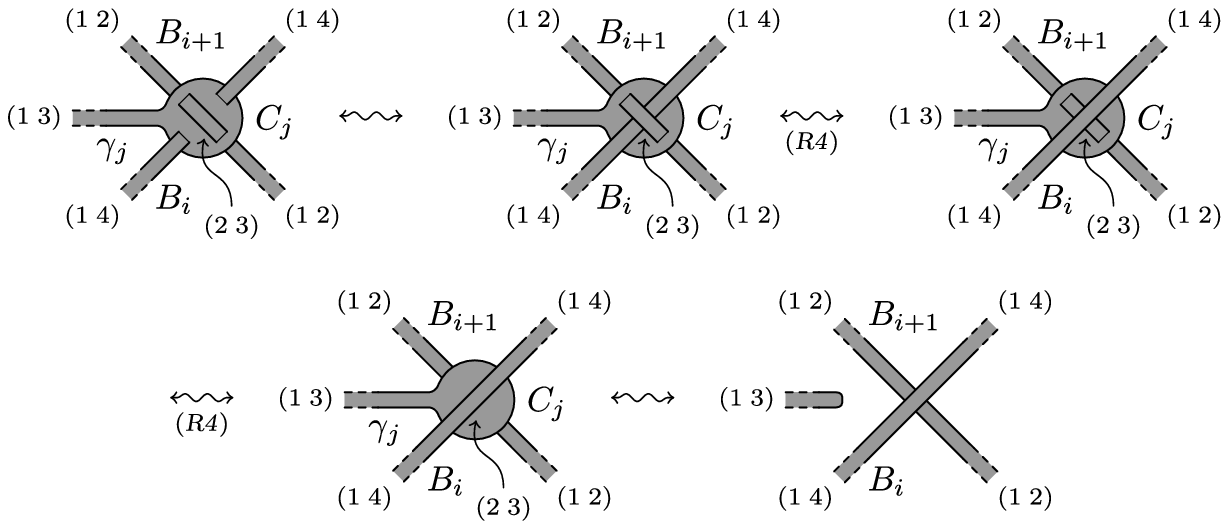}}
\vskip-3pt
\end{Figure}

Concerning a single crossing change making a vertically trivial component $|L'_i|$ into a
different vertically trivial state of $|L_i|$, there are four cases to be considered,
depending on sign of the crossing and on whether $|L'_i|$ coincides with $|L_i|$ at that
crossing or not. In all cases, according to Proposition \ref{vert-state/thm}, there exists
a disk $A_i'$ bounded by one of the two arcs of $|L'_i|$ determined by the two points
projecting to the changing crossing, and the vertical segment joining these two points.
Actually, such a disk can be thought as a perturbation of a subdisk of $A_i$, in such a
way that it does not meet the bands $\beta_1, \dots, \beta_s$. Moreover, we can choose
$\beta_i$ to be attached to $B_i$ in the portion of it corresponding to $|L'_i| - \Bd
A'_i$.

Figure \ref{xi06/fig} indicates how to realize the crossing change in one of the four
cases. For the other three cases it suffices to apply a mirror symmetry to all the stages
and/or reverse their order. In step \(a), as it was done above, we create a disk $D$ of
label $\tp24$ attached to $\id_{\tp43}$ through a narrow band $\delta$ passing through
$\id_{\tp32}$. Then the disk $D$ is moved within a neighborhood of $A'_i$ to let $A'_i$
pass through it, by 1-isotopy and some moves \(R2) and \(R4) like in Figure
\ref{xi04/fig}, until diagram \(b) is achieved. After that, we get \(e) as described in
the intermediate steps. At this point, we use move \(R6) to transfer the two negative
half-twists from $D$ to $B_i$ and we push back $D$ to its original position, by reversing
the process from \(a) to \(b). We remind that the additional negative full twist which
appears on $B_i$ compensates the change of crossing (cf. step 5 of the construction of
$S_K$). Indeed, the move changes the framing of $L'_i$ by $-2$ and therefore the framing
of $B_i$ should change by $-1$ as it does.

\begin{Figure}[htb]{xi06/fig}
{}{Independence of $\check S_K$ on the vertical trivial component $|L'_i|$}
\centerline{\fig{}{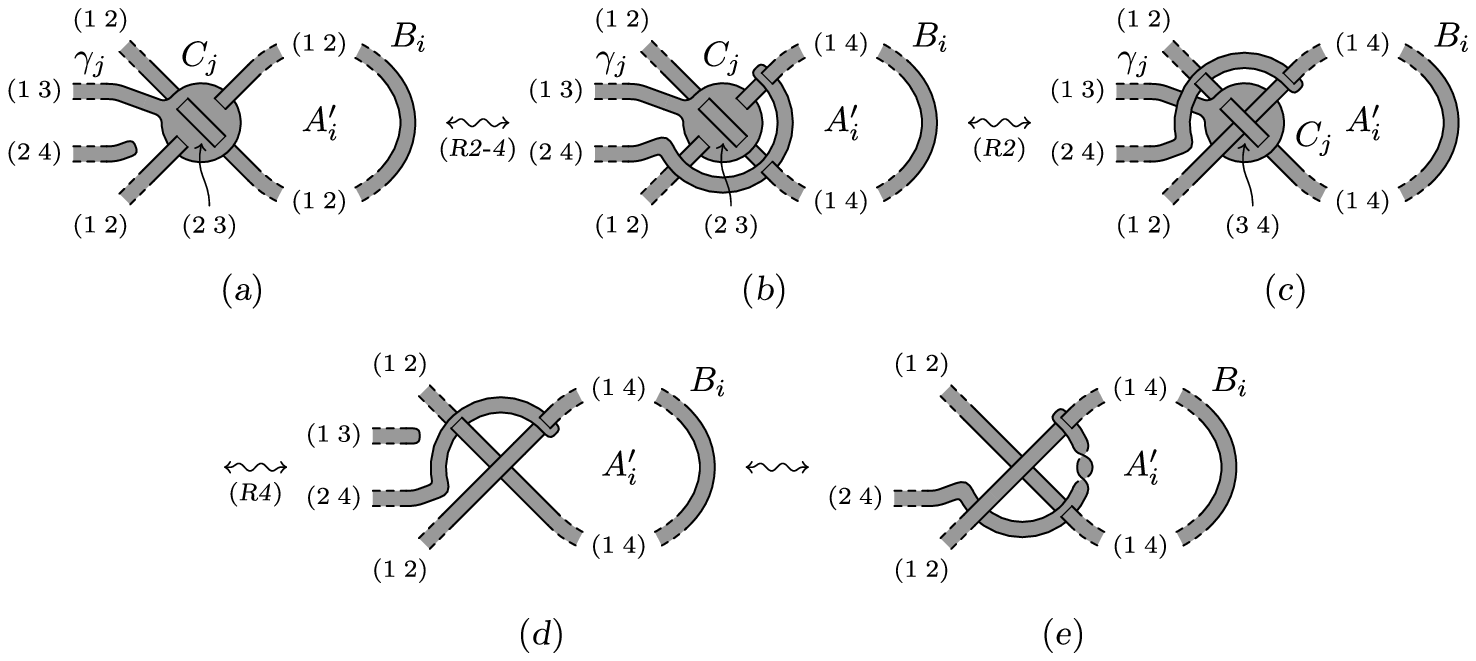}}
\vskip-3pt
\end{Figure}

Finally, we prove the independence of $\check S_K$ on the strictly regular planar diagram 
of $K$. We recall that two strictly regular planar diagrams represent isotopic Kirby
tangles if they are related through planar isotopy, Reidemeister moves and the moves 
presented in Figure \ref{kirby-tang06/fig}.

\begin{Figure}[b]{xi07/fig}
{}{The exceptional Reidemeister move}
\centerline{\fig{}{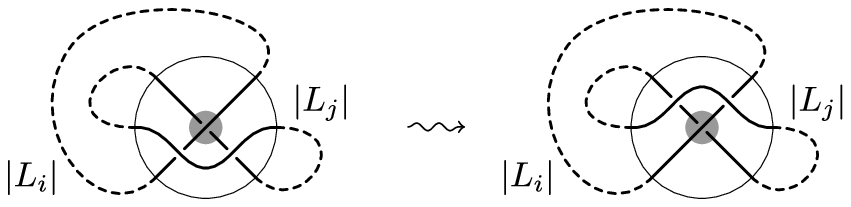}}
\vskip-6pt
\end{Figure}

We observe that any Reidemeister move on the link $L$ of closed framed components of $K$
induces the same move on its vertically trivial state $L'$, and so just a diagram isotopy
on $\check S_K$, provided that none of the involved crossings (before and after the move)
has been changed when passing from $|L|$ to $|L'|$. The reason is that in this case the
two links coincide inside a small 3-cell where the move takes place and such 3-cell is
free from the $C_i$'s. We leave to the reader the straightforward verification that, with
the only exception depicted in Figure \ref{xi07/fig}, a vertically trivial state $|L'|$ of
$|L|$ that satisfies this property can be always achieved by a suitable application of
the naive unknotting procedure described in Section \ref{links/sec} with height function
on each component as in Figure \ref{links01/fig} \(a) or \(c), depending on the move. In
the remaining case of Figure \ref{xi07/fig}, we need to invert at least one of the two
crossings formed by $|L_i|$ and $|L_j|$, in order to get the corresponding components
$|L'_i|$ and $|L'_j|$ of the vertically trivial state $|L'|$. Nevertheless, assuming that
we invert the crossing inside the shaded circle, the move still induces diagram isotopy on
$\check S_K$.

\begin{Figure}[htb]{xi08/fig}
{}{Moving an inverted crossing over/under a pair $D_i$ and $D'_i$}
\centerline{\fig{}{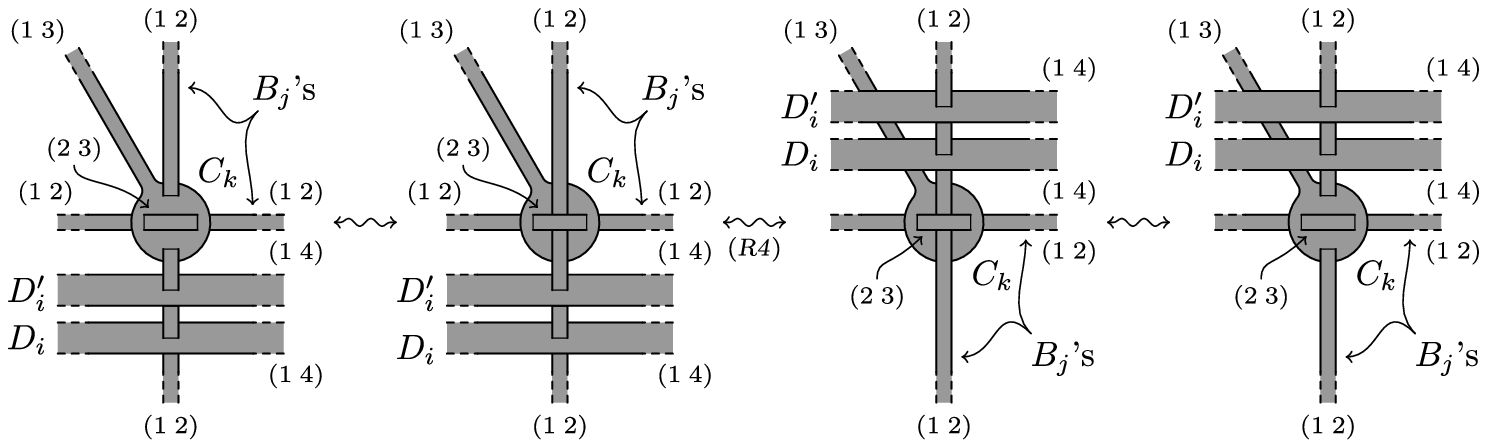}}
\vskip-3pt
\end{Figure}

To complete the proof, we consider the isotopy moves of Figure \ref{kirby-tang06/fig},
which involve the dotted components of $K$. Moves \(a) and \(e) in the figure clearly
induce labeled diagram isotopy on $\check S_K$. The same is true for move \(d) once a
suitable vertically trivial state has been chosen leaving unchanged the involved crossing
between framed arcs. Also moves \(b) and \(c) induce labeled diagram isotopy on $\check
S_K$ when $L'$ coincides with $L$ at all the involved crossings between framed arcs.
Otherwise, we perform the first move in Figure \ref{xi03/fig} on the disks $D_i$ and
$D_i'$ originating from the one handle, changing in this way their labels from $\tp13$ to
$\tp14$. Then, at any crossing where a disk $C_k$ appears, we proceed as indicated in
Figure \ref{xi08/fig}. Here, four moves \(R4) occur in the second step, while labeled
1-isotopy suffices for the other steps. Eventually, we perform backwards the first move in
Figure \ref{xi03/fig} to restore the original labels of $D_i$ and $D_i'$.
\end{proof}

\begin{remark}\label{gen-bands/rem}
The proof of the independence of $\check S_K$ on the choice of the bands $\beta$, $\gamma$
and $\delta$ in Lemma \ref{SK-welldef/thm}, works in the following much more general 
context.

Let $S$ be any labeled ribbon surface tangle in $\S_4$ whose set of labels is complete,
i.e. it generates the whole permutation group $\Sigma_4$. If $S$ contains a band
$\gamma_i$ (resp. $\delta_i$) attached to a local configuration (included the labeling) as
in the left (resp. right) side of Figure \ref{kirby-ribbon05/fig}, then up equivalence
moves, such band can be replaced by any other band $\gamma'_i$ (resp. $\delta'_i$)
attached in the same way to that local configuration. Analogously, if $S$ contains a band
$\beta_i$ attached to a (possibly non-orientable) closed band $B_i$ labeled $\tp12$ whose 
core spans a disk $A_i$ as in Figure \ref{xi04/fig}, such that $B_i$ (resp. $A_i$) forms 
ribbon intersections only passing through disks (resp. being passed through by bands) of 
label $\tp13$, then up equivalence moves, $\beta_i$ can be replaced by any other band 
$\beta'_i$ attached in the same way to $B_i$.

Indeed, thanks to the completeness of the labeling, we can always create a tongue of label
$\tp34$ (resp. $\tp24$) and use it to perform the modification described in Figures
\ref{xi02/fig} or \ref{xi03/fig} (resp. Figure \ref{xi04/fig}) in order to replace the
band $\gamma_i$ or $\delta_i$ (resp. $\beta_i$). We also observe the same argument still
works with any labeling obtained from the specific one considered above by conjugation in
$\Sigma_4$.
\end{remark}

\begin{lemma}\label{SK-invariance/thm}
Up to labeled 1-isotopy and moves \(R1) and \(R2), the labeled ribbon surface tangle 
$\up_3^4 \check S_K$ is invariant under 2-equivalence of $K$.
\end{lemma}

\begin{proof}
According to Definition \ref{kt-equivalence/def}, in order to prove the invariance of
$\check S_K$ under 2-equivalence of $K$, it is enough to consider the moves in Figure
\ref{kirby-tang01/fig}, since we have already proved the invariance under isotopy.
Moreover, in our case $K$ is an ordinary Kirby tangle and therefore no crossing change is
possible, while the other three moves reduce to the ordinary ones without labels. We also
recall that in addition to moves \(R1) and \(R2), we can use the moves \(R3) to \(R6)
introduced in the previous section (cf. Proposition \ref{moves-aux/thm}).

A pushing through 1-handle move on $K$ trivially induces labeled isotopy on $\check S_K$.
Similarly, the addition/deletion of a canceling 1/2-pair can be easily interpreted in
terms of equivalence moves on $\check S_K$. In fact, if the disk $D_i$ and the loop $L_j$
represent a canceling pair in $K$, then in $\check S_K$ the band $B_j$ passes only once
through $D_i$. Thus, by a\break move \(R3) we can remove $D_i$ and break $B_j$ into two
tongues. At this point, by labeled 1-isotopy we can retract such tongues to
$\id_{\tp21}$ and after that to retract to $\id_{\tp32}$ the tongue $\beta_i \cup D'_i$
and all the $\gamma_k \cup C_k$'s related to crossing changes involving $L_j$.

The case of a 2-handle sliding requires some preliminaries. First of all, we number the
components of $L$ starting from the two ones involved in the sliding, in such a way that
$L_1$ slides over $L_2$. In terms of Kirby tangles, this means to replace $L_1$ with the
band connected sum $L_1 \#_\delta \bar L_2$, where $\bar L_2$ is a parallel copy of $L_2$,
and $\delta$ is a band connecting $L_1$ to $\bar L_2$. Since we have already proved the
invariance of $\check S_K$ under isotopy, we can isotope $K$ in such a way that in its
planar diagram $\delta$ is a blackboard parallel band, which does not form any crossing
with the $L_i$'s and lies in a neighborhood of $\id_{\tp21}$ as shown in Figure
\ref{xi09/fig} \(a).

\begin{Figure}[htb]{xi09/fig}
{}{Standard set up for a 2-handle sliding}
\centerline{\fig{}{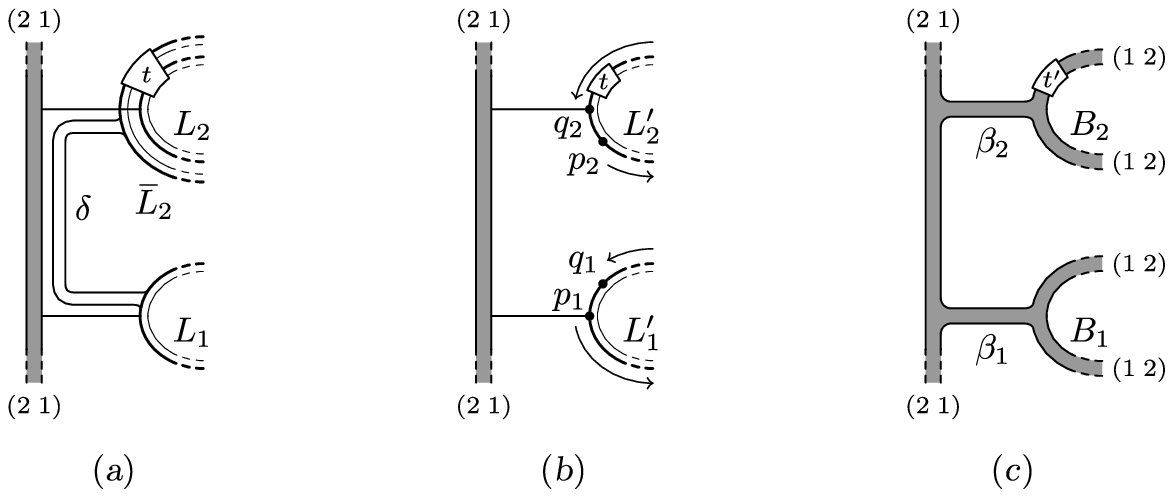}}
\vskip-3pt
\end{Figure}

Then, in the construction of $\check S_K$ we choose a vertically trivial state $|L'|$ such
that: 1) the vertical order of the components is the one given by the numbering ($|L'_i|$
lies under $|L'_j|$, for $i < j$); 2) the minimum point $p_1$ (resp. $p_2$) and the
maximum point $q_1$ (resp. $q_2$) of the height function $h$ on $|L'_1|$ (resp. $|L'_2|$)
coincide with the end points of the attaching arc of $\delta$ to $L_1$ (resp. $\bar L_2$),
as in Figure \ref{xi09/fig} \(b). Here, the arrows indicate the orientations that we will
use in the framing computation at the end of the proof, so they are not relevant for the
moment. Finally, we choose $\beta_1$ and $\beta_2$ to be blackboard parallel bands, such
that $\delta$ can be thought to run parallel to them and to the part of the boundary of
$\id_{\tp21}$ between them, as in Figure \ref{xi09/fig} \(c). For the sake of convenience,
the framed curves $L_2$ and $L'_2$ and the ribbon $B_2$ are assumed to be blackboard
parallel outside the twist boxes $t$ and $t'$ in Figure \ref{xi09/fig}. The reader should
be aware that the numbers of twists inside those boxes may differ, accordingly to step 5
of the construction of $\check S_K$.

\begin{Figure}[htb]{xi10/fig}
{}{A 2-handle sliding in terms of moves on $\up_3^4 \check S_K$}
\vskip-3pt
\centerline{\fig{}{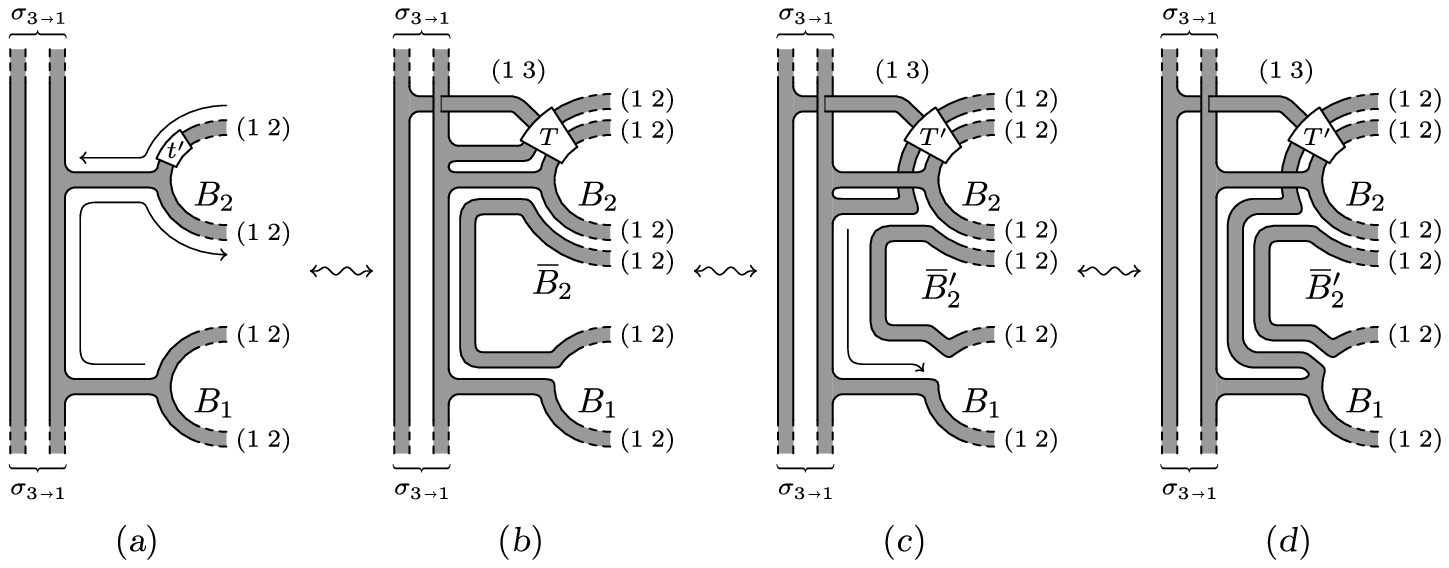}}
\end{Figure}

Once it has been set up in this way, the sliding can be interpreted in terms of moves on
$\up_3^4 \check S_K$, as sketched in Figure \ref{xi10/fig}. Here, we omit to draw
the reduction ribbon $\id_{\tp43}$ that is only implicitly involved in the step from \(b)
to \(c).

We think of $B_1$ as a 1-handle attached to $\id_{\tp21}$ (like $B_j$ in Figure
\ref{kirby-ribbon08/fig}) and slide one of its attaching arcs as indicated by the arrows
in \(a), to form a new ribbon $\bar B_2$ parallel to $B_2$. Before reaching the twist box
$t'$, this sliding can be entirely realized by labeled diagram isotopy, except for the
labeled 1-isotopy moves needed to pass through the $D_i$'s and $C_i$'s encountered by
$B_2$. Each time a disk $C_i$ is passed through, two new ribbon intersections appear as
in the first diagram of Figure \ref{xi11/fig}. Then, we use again 1-isotopy to split
$C_i$ into two twin disks similar to the original one, as suggested by the rest of Figure
\ref{xi11/fig}.

\begin{Figure}[htb]{xi11/fig}
{}{}
\centerline{\fig{}{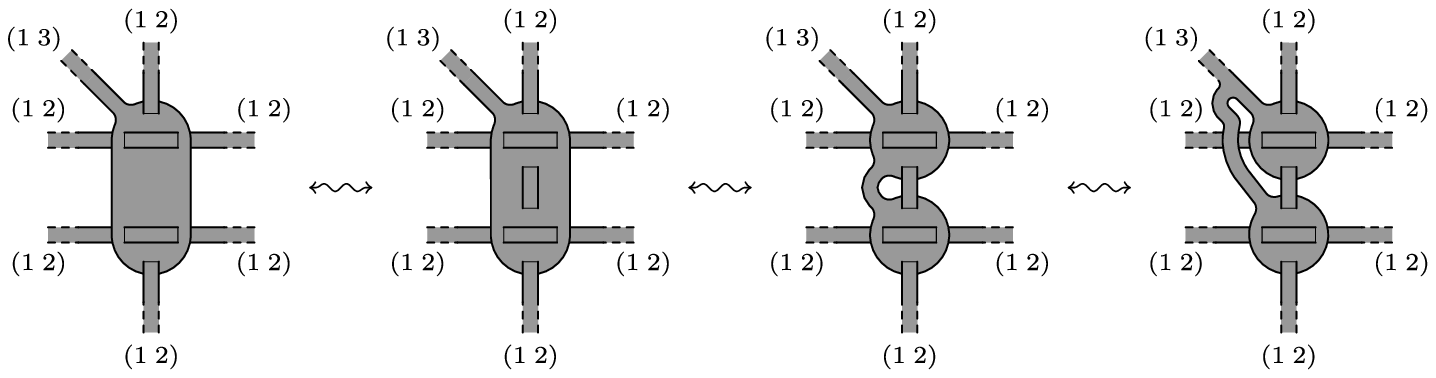}}
\vskip-3pt
\end{Figure}

To get the twist box $T$ in \(b), after having followed all the twists of $B_2$ in the
twist box $t'$, we add some further crossings between $B_2$ and $\bar B_2$ (together with
further $C_i$'s), in order to make them unlinked. Figure \ref{xi12/fig} shows how to
add a positive crossing; for a negative one it suffices to mirror the figure. Here, some
moves other than 1-isotopy are needed: one move \(R5) in the first step; two opposite
moves \(R6) in the third; two moves \(R1) in the fourth. In particular, we observe that
$\bar B_2$ enters and leaves the box $T$ in Figure \ref{xi10/fig} \(b) on the same
side of $B_2$.

\begin{Figure}[htb]{xi12/fig}
{}{}
\centerline{\fig{}{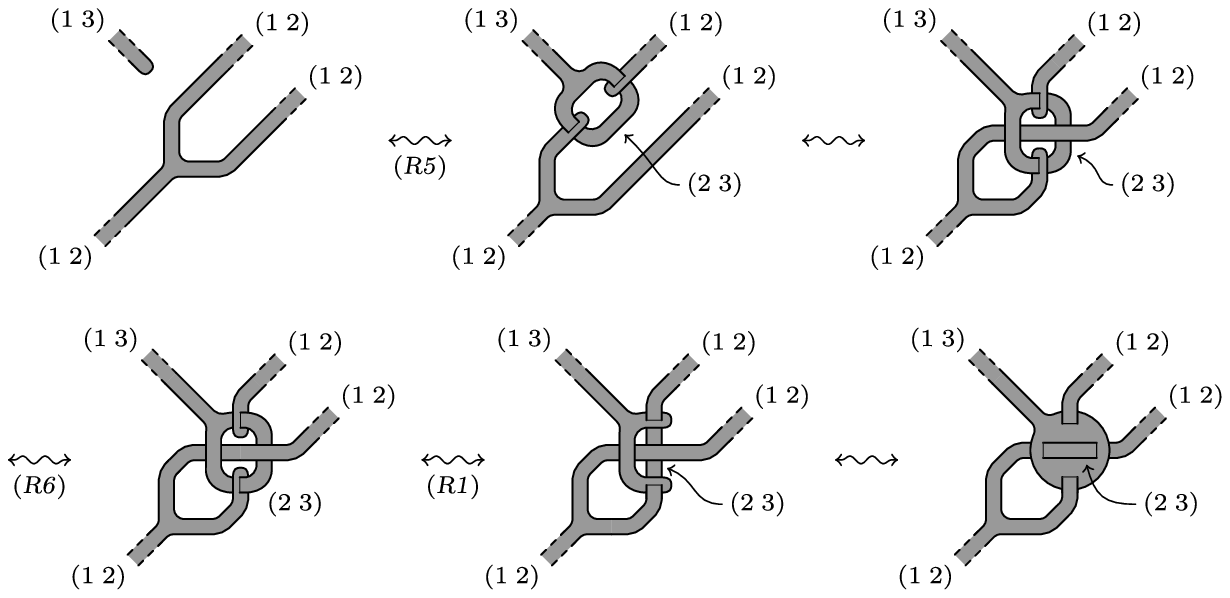}}
\vskip-3pt
\end{Figure}

Now, we consider a disk $A_2$ spanned by $|L'_2|$ as in step 6 of the construction of
$\check S_K$, and perturb it near $B_2$ in such a way that it becomes disjoint from $\bar
B_2$, while remaining disjoint from all the other $B_i$'s and continuing to form only
clasps and ribbon intersections with the rest of the ribbon surface tangle. Indicating the
perturbed disk still by $A_2$, we use it to replace the band $\beta_2$ in \(b) by that in
\(c) and at the same time pull $\bar B_2$ below $B_2$. To do that we first perform steps
\(a) to \(d) in Figure \ref{xi04/fig} with $i = 2$, to change the labeling of $B_2$ from
$\tp12$ to $\tp14$. Then we change all the crossings where $\bar B_2$ passes over $B_2$,
by operating as in Figure \ref{xi05/fig}. Finally, we perform the last step in Figure
\ref{xi04/fig} obtaining \(e) and eventually cut and retract $\beta_2$. In this way, we
get the diagram in Figure \ref{xi10/fig} \(c), where $\bar B'_2$ and $T'$ differ from
$\bar B_2$ and $T$ only by the performed crossing changes. Eventually we get \(d) by
completing the sliding as indicated by the arrow in \(c).

We claim that the labeled ribbon surface tangle in Figure \ref{xi10/fig} \(d)
coincides up to 3-dimensional isotopy with $\check S_{K'}$, where $K'$ is the ordinary
Kirby diagram obtained from $K$ by replacing $L_1$ with $L_1 \#_\delta \bar L_2$.

To prove this claim, let $\bar L'_2$ be the framed unknot whose base curve is the core of
the ribbon $\bar B'_2$ and whose framing is the double of that represented by $\bar B'_2$.
Taking into account our starting assumptions on the height function of $|L'|$, the link
formed by $|L'_1 \#_\delta \bar L'_2|, |L'_2|, \dots, |L'_s|$ is vertically trivial for a
suitable height function compatible with the planar diagram in Figure \ref{xi10/fig}
\(d). Moreover, $|L'_1 \#_\delta \bar L'_2|$ is the core of the ribbon $B_1 \#_\delta \bar
B'_2$ and $\fr(L'_1 \#_\delta \bar L'_2) = \fr(L'_1) + \fr(\bar L'_2)$ is the double of
the framing represented by $B_1 \#_\delta \bar B'_2$. Therefore, still referring to
diagram \(d) in Figure \ref{xi10/fig}, our claim amounts to say that by inverting
the crossings marked by the presence of a disk $C_i$ in the framed link formed by $L'_1
\,\#_\delta \bar L'_2, L'_2, \dots, L'_s$, we get the framed link formed by $L_1
\,\#_\delta \bar L_2, L_2, \dots, L_s$, up to diagram isotopy of $K$. It is clear from the
construction that such crossing inversions produce a framed link of components $L_1
\#_\delta \hat L_2, L_2, \dots, L_s$, where $\hat L_2$ is a framed knot whose base curve
$|\hat L_2|$ is parallel to $|L_2|$. Hence, it is left to prove that $\lk(|L_2|, |\hat
L_2|) = \lk(|L_2|, |\bar L_2|) = \fr(L_2)$ (in other words, $|\hat L_2|$ and $|\bar L_2|$
represent the same framing of $|L_2|$) and $\fr(\hat L_2) = \fr(\bar L_2) = \fr(L_2)$,
where in both cases the second equality directly derives from the construction.

Since $B_2$ and $\bar B'_2$ are vertically separated, $\lk(|L'_2|,|\bar L'_2|) = 0$ and
thus $\lk(|L_2|, |\hat L_2|)$ is the opposite of the signed number of crossings between
$B_2$ and $\bar B'_2$ marked by a $C_i$ in \(d). On the other hand, being also $B_2$ and
$\bar B_2$ unlinked, this equals the opposite of the signed number of crossings between
$B_2$ and $\bar B_2$ marked by a $C_i$ in \(b). According to the construction of the
diagram in \(b), there are $-\fr(L'_2)$ such crossings inside the twist box $T$, while the
number of them outside the twist box $T$ is the double of the signed number of the
self-crossings of $B_2$ marked by a $C_i$ in $S_K$. But this last number is equal to the
difference $\fr(L'_2) - \fr(L_2)$, so we can immediately conclude that $\lk(|L_2|, |\hat
L_2|) = \fr(L_2)$.

Finally, the equality $\fr(\hat L_2) = \fr(L_2)$ can be derived in a similar way from
$\fr(\bar L'_2) = 2 \fr(\bar B'_2) = 2 \fr(\bar B_2) = 2 \fr(B_2) = \fr(L'_2)$, once one
observes that the signed number of the self-crossings of $\bar B'_2$ marked by a $C_i$ in
\(d) coincides with that of the self-crossings of $B_2$ marked by a $C_i$ in $S_K$.
\end{proof}

\begin{proposition}\label{xi4/thm}
The map $\,\Xi_4$ defined $\,\Xi_4(I_m) = J_{\sigma_{4 \red 1}} \!\diam J_m$ for every $m
\geq 0$ and $\,\Xi_4(K) = \up_3^4 \check S_K$ for every Kirby tangle $K$ in $\,\K_1$ (cf.
Proposition \ref{trivial-state/thm}), represents a faithful functor $\,\Xi_4: \K_1 \to
\S_4^c$, such that $\,\Xi_4(\K_1) \subset \S_{4 \red 1}^2$ and ${\down^4_1} \circ \Theta_4
\circ \Xi_4 = \id_{\K_1}$. Moreover, for any two morphisms $K: I_{m_0} \to I_{m_1}$ and
$K': I_{m'_0} \to I_{m'_1}$ in $\K_1$ we have $\,\Xi_4(K \diam K') = \up_3^4\,(Q_{m_1 +
m'_1} \circ (R_K \rdiam R_{K'}) \circ \bar Q_{m_0 \diam m'_0})$ (cf. Figure
\ref{kirby-ribbon01/fig}).
\end{proposition}

\begin{proof}
In the light of the previous Lemmas \ref{SK-welldef/thm} and \ref{SK-invariance/thm}, 
$\Xi_4$ is well-defined as a map. To prove that it is a functor, we have to show 
that it preserves the identity morphisms and the composition of morphisms.

\begin{Figure}[htb]{xi13/fig}
{}{$\Xi_4$ preserves the identity morphisms}
\centerline{\fig{}{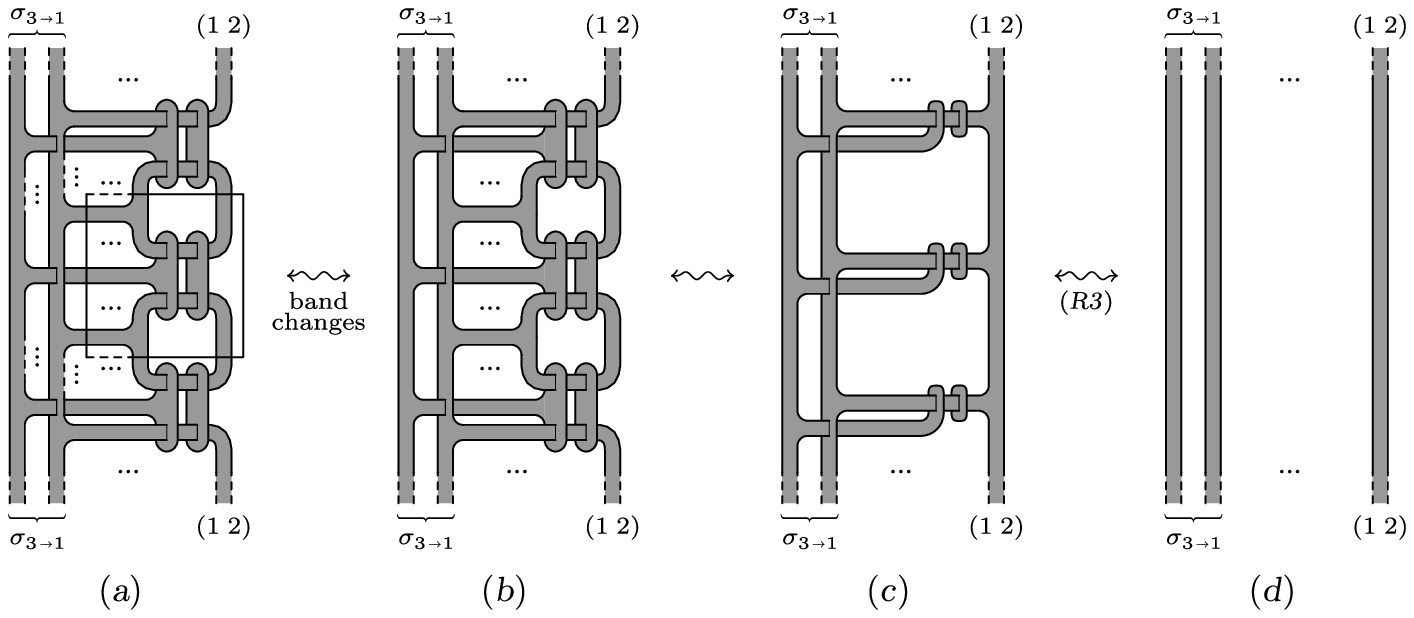}}
\vskip-3pt
\end{Figure}

Concerning the identities, we observe that for $K = \id_{I_m}$ the labeled ribbon
sur\-face tangle $\check S_K$ looks like in Figure \ref{xi13/fig} \(a), where only the
rightmost of the $m$ blocks forming $T_K$ (in the box), $Q_m$ and $\bar Q_m$ is drawn. The
rest of Figure \ref{xi13/fig} outlines how we can reduce this block to a single identity
ribbon. Namely, to get \(b) we change the $\beta_i$'s and the $\delta_i$'s relative to the
other blocks and laying between the uppermost and lowermost horizontal bands of label
$\tp12$ in the figure, in such a way that they allow the labeled 1-isotopy from \(b) into
\(c). This can be done by the same arguments used in the proof of Lemma
\ref{SK-welldef/thm} thanks to Remark \ref{gen-bands/rem}, by using the stabilization
ribbon $\id_{\tp43}$ that is there even if not drawn in the figure. Then, we use move
\(R3) to cut the horizontal bands in \(c) and we retract the resulting tongues to get
\(d). Once the reduction of the rightmost block has been completed, the process can be
iterated to reduce one by one, from right to left, also the other $m-1$ blocks to identity
ribbons, hence all $\Xi_4(K)$ to the identity.

Now let us consider the composition. Given two morphisms $K_1: I_{m_0} \to I_{m_1}$ and
$K_2: I_{m_1} \to I_{m_2}$ in $\K_1$, we have to prove that $\up_3^4 (\check S_{K_2} \circ
\check S_{K_1})$ and $\up_3^4 \check S_{K_2\circ K_1}$ are equivalent up to 1-isotopy and
moves \(R1) and \(R2). This will follow once show that $ R_{K_2} \circ \bar Q_{m_1} \circ
Q_{m_1} \circ R_{K_1}$ is equivalent to $R_{K_2\circ K_1}$. Apart from the two reduction
ribbons on the left, the composition $\bar Q_{m_1} \!\circ Q_{m_1}$ consists of $m_1$
blocks like the one shown in Figure \ref{xi14/fig} \(a). Each block contains two arcs
labeled $\tp12$, ending one in the source and the other in the target and originating from
two open components in $K_1$ and $K_2$, which create a single closed component in
$K_2\circ K_1$. Figure \ref{xi14/fig} describes in a schematic way the transformation of
these blocks one by one, starting from the left. Let $B_1$ and $B_2$ be the closed bands
in $R_{K_1}$ and $R_{K_2}$ to which the two arcs of the block shown in \(a) belong. By
replacing bands if necessary, we can assume that their attaching bands $\beta_1$ and
$\beta_2$ are the as shown in \(b), where we have also slided the two innermost horizontal
bands in \(a) until they form an extra closed ribbon connected to $\id_{\tp21}$. In the
last step from \(c) to \(d) we first use two \(R3) moves to cut and retract the horizontal
bands passing through disks, so that in the resulting tangle $B_1$ and $B_2$ are replaced
with a single closed band attached to $\id_{\tp12}$ by a band $\beta$. Then, we change
such $\beta$ in order to allow the iteration of the whole process to the subsequent blocks 
of $\bar Q_{m_1} \!\circ Q_{m_1}$.

\begin{Figure}[htb]{xi14/fig}
{}{$\Xi_4$ preserves the composition of morphisms}
\centerline{\fig{}{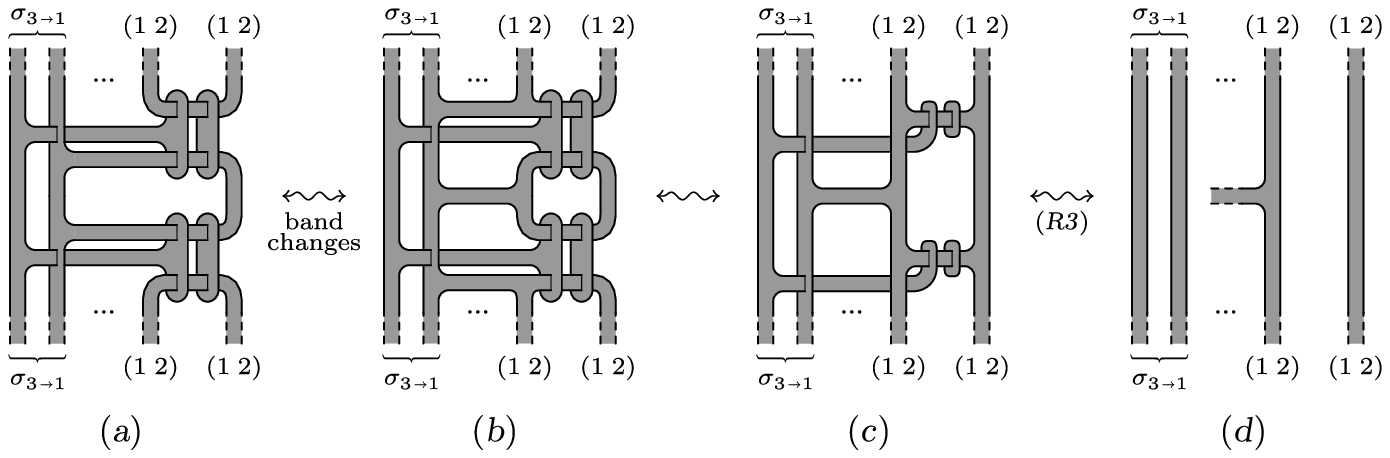}}
\vskip-3pt
\end{Figure}

Finally, the equality $\down^4_1 \Theta_4 \circ \Xi_4 = \id_{\K_1}$ and hence the
faithfulness of $\Xi_4$ immediately follow from Propositions \ref{stab-theta/thm} and
\ref{full-theta3/thm}, while the second part of the statement is just a reinterpretation
of Remark \ref{non-monoidal/rem} under the stabilization $\up_3^4$.
\end{proof}

We note that $\Xi_4: \K_1 \to \S^c_4$ will be proved to be a category equivalence in the
next section (cf. Theorem \ref{ribbon-kirby/thm}), and therefore it induces a monoidal
structure on $\S^c_4$, which we do not provide in explicit form.

\begin{proposition}\label{xi/thm}
For any $n \geq 4$, $\Xi_n = {\up_4^n} \circ \Xi_4: \K_1 \to \S_n^c$ is a faithful functor
and satisfies the identity ${\down^n_1} \circ \Theta_n \circ \Xi_n = \id_{\K_1}$.
\end{proposition}

\begin{proof}
This is a direct consequence of Propositions \ref{stab-theta/thm}, \ref{full-theta3/thm}
and \ref{xi4/thm}.
\end{proof}

\subsection{Equivalence between $\K^c_n$ and $\S^c_n$ for $n \geq 4$%
\label{K=S/sec}}

We want to prove that the functors $\Theta_n: \S^c_n \to \K^c_n$ and $\Xi_n: \K_1 \to
\S^c_n$ are equivalences of categories for any $n \geq 4$. Taking into account
Propositions \ref{cat-equiv/thm} and \ref{xi/thm}, to achieve this result it is enough to
show that the functor $\Xi_4: \K_1 \to \S^c_4$ is full and that any object in $\S^c_n$ is
isomorphic to one in its image, i.e. that the inclusion of $\Xi_4 (\K_1)$ in $\S^c_4$ is
an equivalence of categories.

\begin{lemma}\label{special-ribbon/thm}
Any labeled ribbon surface tangle $T: J_{\sigma_{3 \red 1}} \diam J_{\sigma_0} \to
J_{\sigma_1}$ in $\S_3$ is equivalent to a composition $T_l \circ \dots \circ T_1$, where
each $T_k$ is an expansion (that is a product with some identity ribbons on the left
and/or on the right) of one of the special elementary ribbon surface tangles presented in 
Figure \ref{xi-full01/fig}.
\end{lemma}

\begin{Figure}[htb]{xi-full01/fig}
{}{Special elementary morphisms in $\S_3$ ($\tp{i}{j} \in \Gamma_3$)}
\vskip-9pt
\centerline{\fig{}{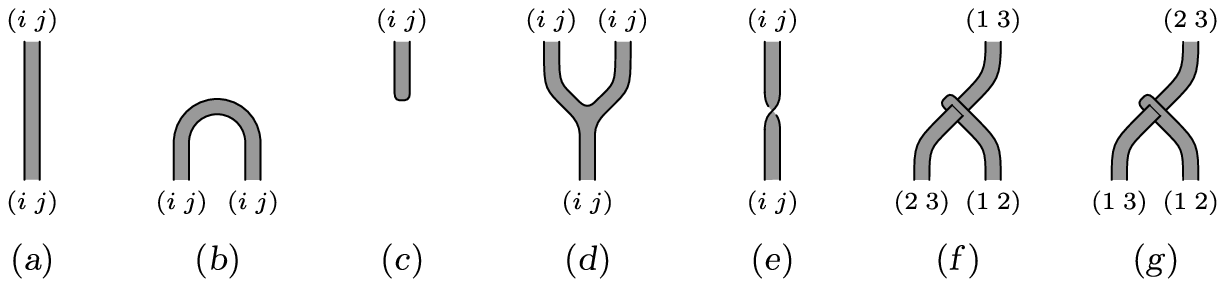}}
\vskip-3pt
\end{Figure}

\begin{proof}
According to Proposition \ref{planar-diagram/thm}, we can assume that $T$ is presented by
a\break 3-labeled special planar diagram, that is a planar diagram consisting of some
spots like \(a) to \(e) and \(h) in Figure \ref{ribbon-surf06/fig} and some flat bands,
with a compatible label\-ing in $\Gamma_3$. Moreover, we can convert all the negative
half-twists occurring in such a diagram into positive ones and crossings, by using labeled
versions of the moves \(S14) and \(S17) in Figure \ref{ribbon-surf11/fig}.

\begin{Figure}[b]{xi-full02/fig}
{}{}
\centerline{\fig{}{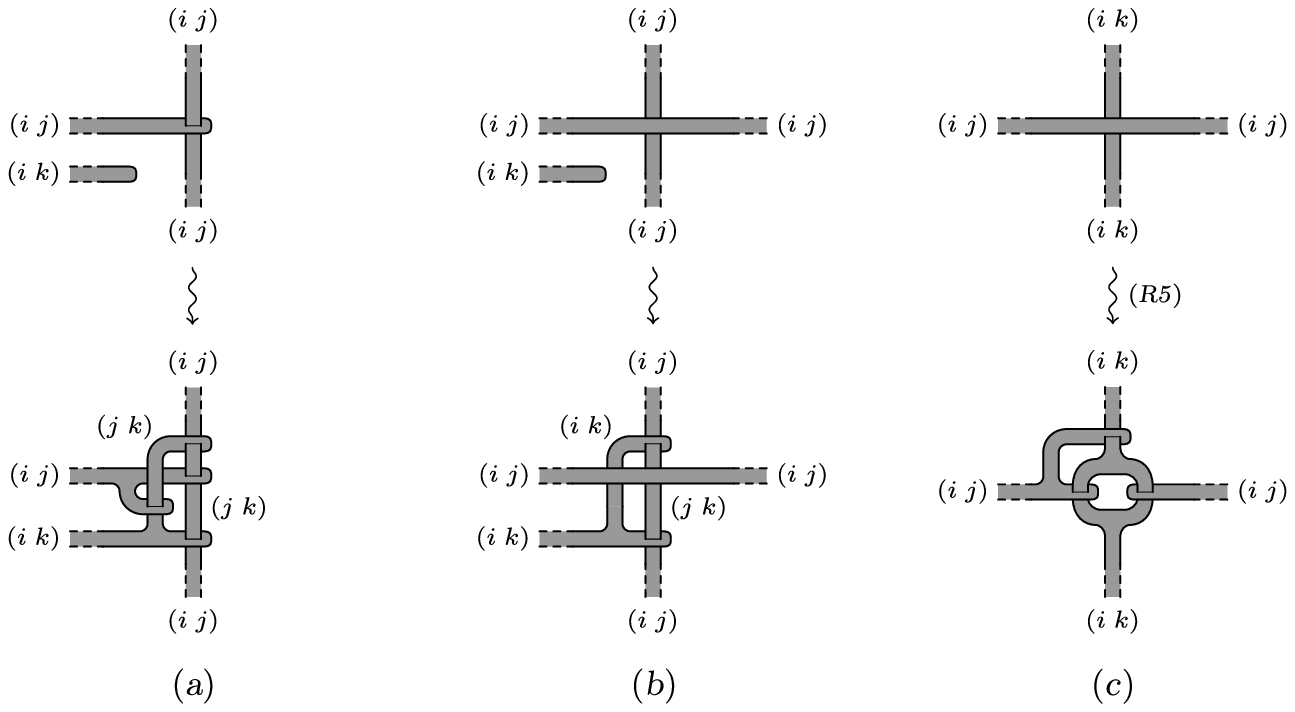}}
\vskip-3pt
\end{Figure}

Then, we proceed as shown in Figure \ref{xi-full02/fig} to eliminate all the monocromatic
ribbon intersections and all the crossings. More precisely, we first eliminate the
monocromatic ribbon intersections and crossings by performing the labeled 1-isotopies
described in \(a) and \(b) respectively. Here, the tongues labeled $\tp{i}{k}$ have been
pulled out of the closest ribbon with that label. Observe that such ribbon always be
created, due to the form of the source. Subsequently, we eliminate all the
remaining crossings, including the new ones deriving from \(a) and \(b), by the moves
depicted in \(c). Eventually, we put again the diagram into special form by move \(S2),
which only introduces positive half-twists.

After that, all the three transpositions of $\Gamma_3$ are involved in the labeling of 
any surviving ribbon intersection, so we can use move \(R1) to adjust it in such a way 
that the disk which is passed through is the one with label $\tp12$.

Finally, by using planar isotopy as in the proof of Proposition 3.1.2, we can express $T$
as $T_l \circ \dots \circ T_1$, where each $T_k$ is an expansion of one of the elementary
morphisms of Figure \ref{ribbon-morph01/fig}. The previous form of $T$ guarantees that no
spot like \(e), \(e') and \(f') occurs. Moreover, we can also eliminate all the spots like
\(g'), \(c') and \(b) in the order, by using move \(I6), \(I3) and \(I16) respectively.
This leaves us with the required expression of $T$.
\end{proof}

\begin{proposition}\label{full-xi/thm}
The inclusion of  $\,\Xi_4(\K_1)$ in $\S_4^c$ is a category equivalence.
\end{proposition}

\begin{proof}
We observe that $\Xi_4(\K_1) \subset \up^4_3 \S_3^c$. Moreover, according to Proposition
\ref{S-reduction/thm}, the inclusion of $\up^4_3 \S_3^c$ in $\S_4^c$ is a category
equivalence. Therefore, it suffices to show that the inclusion $\Xi_4(\K_1) \subset
\up^4_3 \S_3^c$ is an equivalence as well. By Proposition \ref{subcat-equiv/thm}, this
amounts to say that any object in $\up^4_3 \S_3^c$ is isomorphic to one in the image of
$\Xi_4$ and that any morphisms in $\up^4_3 \S_3^c$ between two objects in the image of
$\Xi_4$ is the image of a morphism in $\K_1$ under $\Xi_4$.

About objects, we recall that those in the image of $\Xi_4$ are given by $J_{\sigma_{4
\red 1}} \diam J_m$ with $m \geq 0$, where $J_m$ denotes the sequence of $m$ intervals
each labeled by $\tp12$.\break For any $\sigma = (\tp{i_1}{j_1}, \dots, \tp{i_m}{j_m}) \in
\seq\Gamma_3$, we define $\zeta_\sigma: J_{\sigma_{3 \red 1}} \!\diam J_{\sigma} \to
J_{\sigma_{3 \red 1}} \!\diam J_m$, as the composition $\zeta_\sigma = \zeta_{\sigma,m}
\circ \dots \circ \zeta_{\sigma,1}$, where $\zeta_{\sigma,h}$ is the identity if
$\tp{i_h}{j_h} = \tp12$, while it is illustrated in Figure \ref{xi-full03/fig} for
$\tp{i_h}{j_h}$ equal to $\tp13$ and $\tp23$. Here, the horizontal tongues pass in front
of the first $h-1$ vertical ribbons originating in $J_\sigma$, and form ribbon
intersections with the $h$-th one. The $\zeta_{\sigma,h}$'s are isomorphisms, their
inverses being obtained by vertical reflection, hence $\zeta_\sigma$ is an isomorphism as
well. Then, the desired isomorphism between the generic object $J_{\sigma_{4 \red 1}}
\diam J_\sigma$ in $\up^4_3 \S_3^c$ and an object in the image of $\Xi_4$ is represented
by $\id_{4 \red 3} \diam \zeta_\sigma: J_{\sigma_{4 \red 1}} \diam J_\sigma\to
J_{\sigma_{4 \red 1}} \diam J_m$.

\begin{Figure}[htb]{xi-full03/fig}
{}{The isomorphism $\zeta_{\sigma,h}$}
\centerline{\fig{}{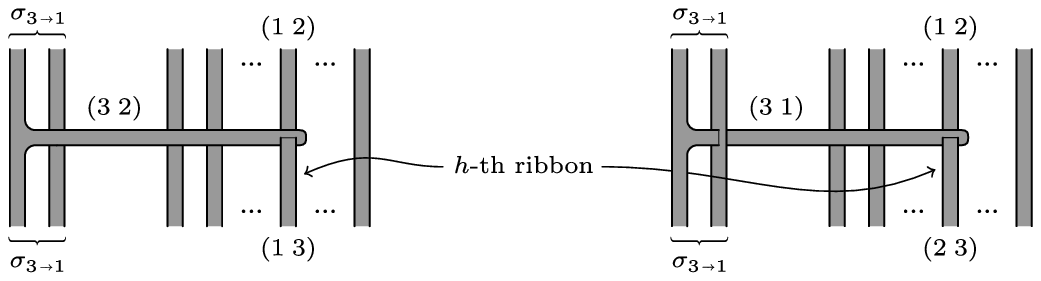}}
\vskip-6pt
\end{Figure}

The proof that any morphism $S: J_{\sigma_{4 \red 1}} \!\diam J_{m_0} \to J_{\sigma_{4
\red 1}} \!\diam J_{m_1}$ in $\up_3^4 \S_3^c$ is in the image if $\Xi_4$, requires more
work. By definition, $S = (\id_{\sigma_{4 \red 1}} \!\diam T) \circ (\id_{\tp43} \diam
\Delta_{\sigma_{3 \red 1}} \!\diam \id_{m_0})$ for some $T: J_{\sigma_{3 \red 1}} \!\diam
J_{m_0} \!\to J_{m_1}$ in $\S_3$. Then, we have the decomposition
$$S = \id_{\tp43} \diam (\zeta_{\sigma_1} \circ (\id_{\sigma_{3
\red 1}} \! \diam T) \circ \zeta_{\sigma_{3 \red 1} \diam \sigma_0}^{-1} \circ
\zeta_{\sigma_{3 \red 1} \diam \sigma_0} \circ (\Delta_{\sigma_{3 \red 1}} \!\diam
\id_{\sigma_0}) \circ \zeta_{\sigma_0}^{-1})\,,$$
\pagebreak
where $\zeta_{\sigma_1}$ and $\zeta_{\sigma_0}^{-1}$ act as identities, being the elements
of $J_{m_0}$ and $J_{m_1}$ all labeled by $\tp12$. Moreover, by Lemma
\ref{special-ribbon/thm} we can assume that $T$ is a composition $T_l \circ \dots \circ
T_1$, where each $T_k \in \S_3$ is an expansion of one of the special elementary morphisms
in Figure \ref{xi-full01/fig}. Therefore, if $J_{\sigma_0^k}$ and $J_{\sigma_1^k}$ are
respectively the source and the target of $T_k$, by inserting in $\id_{\sigma_{3 \red 1}}
\! \diam T = (\id_{\sigma_{3 \red 1}} \! \diam T_l) \circ \dots \circ (\id_{\sigma_{3 \red
1}} \! \diam T_1)$ the canceling pair $\zeta_{\sigma_0^{k+1}}^{-1} \circ
\zeta_{\sigma_1^k}$ between $\id_{\sigma_{3 \red 1}} \! \diam T_{k+1}$ and $\id_{\sigma_{3
\red 1}} \! \diam T_k$ for any $k = 1,\dots, l-1$, we get
$$\zeta_{\sigma_1} \circ (\id_{\sigma_{3 \red 1}} \!\diam T) \circ \zeta_{\sigma_{3 \red
1} \diam \sigma_0}^{-1} = (\zeta_{\sigma_1^l} \circ (\id_{\sigma_{3 \red 1}} \!\diam T_l)
\circ \zeta_{\sigma_0^l}^{-1}) \circ \dots \circ (\zeta_{\sigma_1^1} \circ (\id_{\sigma_{3
\red 1}} \!\diam T_1) \circ \zeta_{\sigma_0^1}^{-1})\,.$$
On the other hand, Figure \ref{xi-full04/fig} shows that $\zeta_{\sigma_{3 \red 1} \diam
\sigma_0} \circ (\Delta_{\sigma_{3 \red 1}} \!\diam \id_{\sigma_0}) \circ
\zeta_{\sigma_0}^{-1}$ is equivalent to the composition of $\id_{\tp32} \diam
\Delta_{\tp12} \diam \id_{\sigma_0}$ and two morphisms having the same form as the
$\id_{\sigma_{3 \red 1}} \!\diam T_k$'s above.

\begin{Figure}[htb]{xi-full04/fig}
{}{}
\centerline{\fig{}{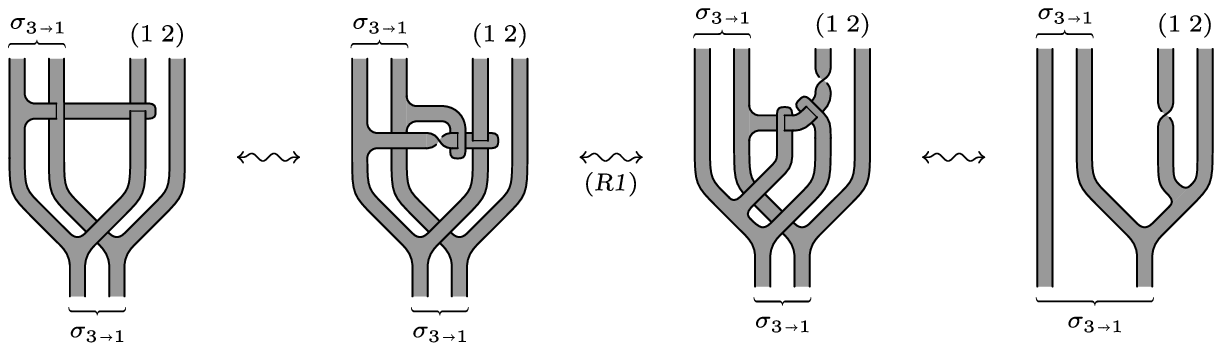}}
\vskip-6pt
\end{Figure}

In conclusion, it is enough to prove that a ribbon surface tangle $S \in \up^4_1 \S_{3
\red 1}$ is in the image of $\Xi_4$ in the special cases when $S = \id_{\sigma_{4 \red 2}}
\!\diam \Delta_{\tp12} \diam \id_m$ for some $m \geq 0$, or $S = \id_{\tp43} \diam
(\zeta_{\sigma_1} \circ (\id_{\sigma_{3 \red 1}} \!\diam T) \circ \zeta_{\sigma_0}^{-1})$,
with $T: J_{\sigma_0} \to J_{\sigma_1}$ being any expansion in $\S_3$ of one of the
special elementary morphisms in Figure \ref{xi-full01/fig}.

\smallskip

We start with the simplest cases when $m = 0$ and $T$ is an elementary morphism. Namely,
Figure \ref{xi-full09/fig} shows that $\id_{\tp42} \diam \Delta_{\tp12}$ is equivalent to
$\Xi_4(\eta_1)$, where $\eta_1$ is a single framed arc (cf. Figure
\ref{kirby-morph01/fig}). Here, we first modify the ribbon surface tangle by 1-isotopy and
then use move \(R3) to eliminate the pair of vertical disks and the horizontal band
passing through them.

\begin{Figure}[htb]{xi-full09/fig}
{}{}
\centerline{\fig{}{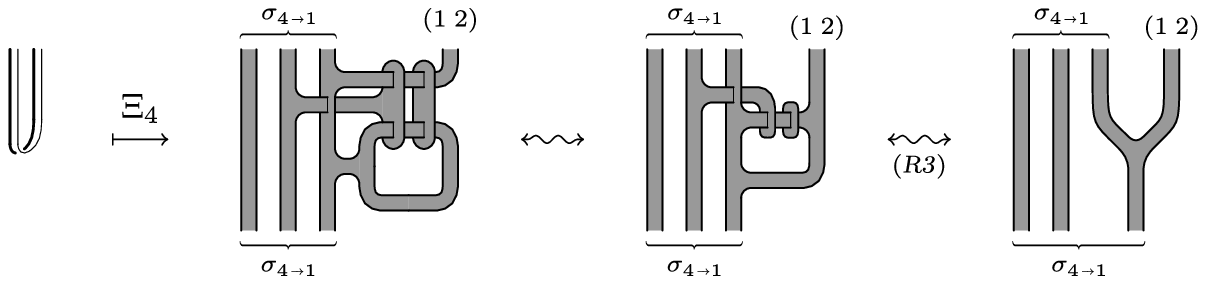}}
\vskip-3pt
\end{Figure}

Concerning $\id_{\tp43} \diam (\zeta_{\sigma_1} \circ (\id_{\sigma_{3 \red 1}} \!\diam T)
\circ \zeta_{\sigma_0}^{-1})$, we first observe that, if $T$ is one of the tangles \(a) to
\(e) in Figure \ref{xi-full01/fig}, then the equality $\id_{\tp43} \diam (\zeta_{\sigma_1}
\circ (\id_{\sigma_{3 \red 1}} \!\diam T) \circ \zeta_{\sigma_0}^{-1}) = \id_{\sigma_{4
\red 1}} \!\diam T$ can be immediately seen to hold just by using 1-isotopy. So, for those
$T$ it is enough to treat the case $\tp{i}{j} = \tp12$.
This has been already done for $T = \id_{\tp12}$\break in the proof of Proposition
\ref{xi4/thm}, while it is done in Figure \ref{xi-full10/fig} for the other $T$'s, from
\(b) to \(e). This time, the first modification of each ribbon surface tangle consists in
\pagebreak
the elimination of the pairs of vertical disks on the top and on the bottom (belonging to
the $Q_{m_1}$'s and the $\bar Q_{m_0}$'s), realized by moves \(R3) after suitable
1-isotopy (cf. Figure \ref{xi-full09/fig}, where an analogous elimination is detailed into
two steps).

\begin{Figure}[htb]{xi-full10/fig}
{}{}
\centerline{\fig{}{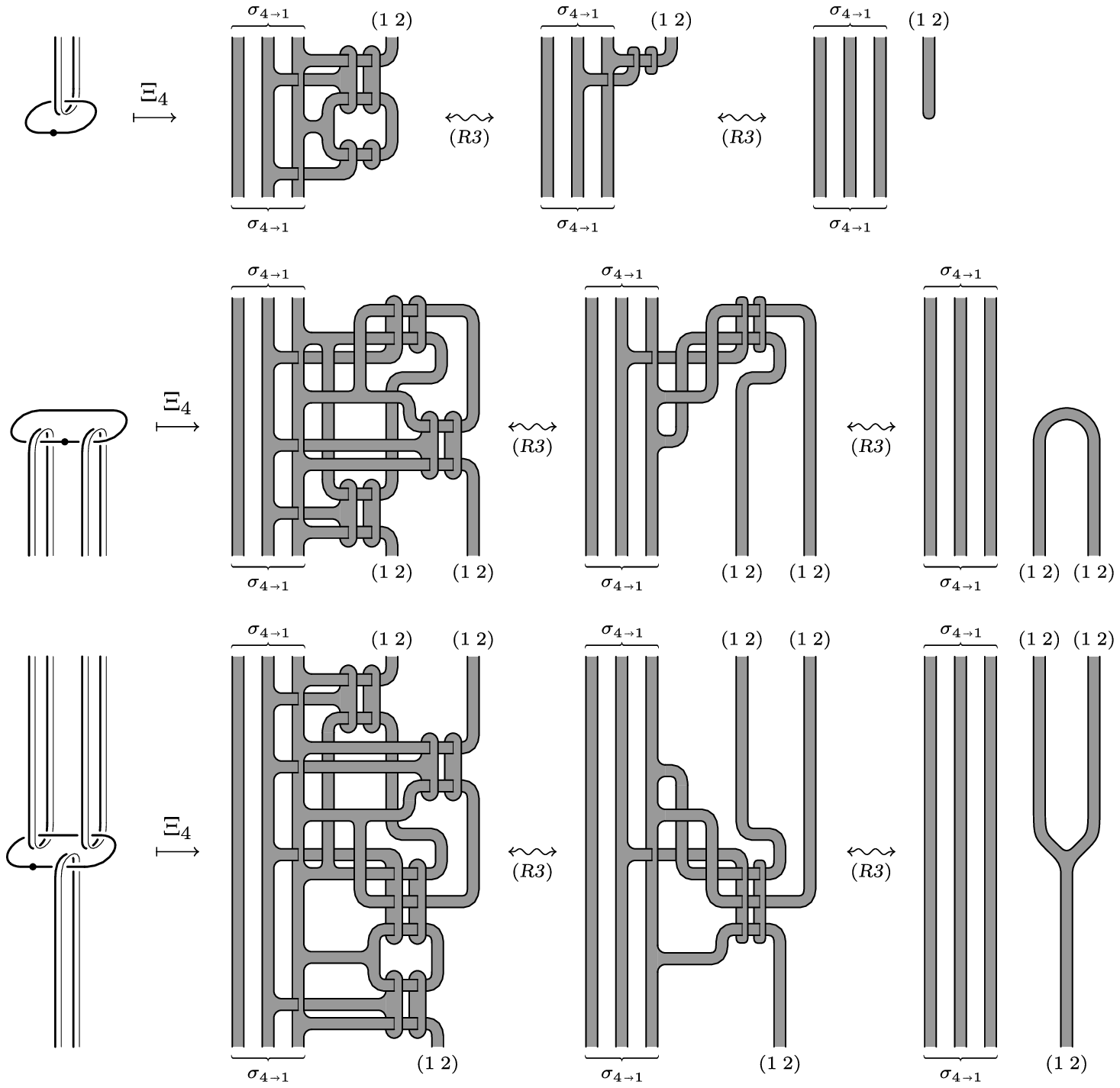}}
\vskip-3pt
\end{Figure}

The cases when $T$ coincides with \(f) and \(g) are shown in Figure \ref{xi-full11/fig}.
This figure is more sketchy then the previous ones. In particular, in the starting ribbon
surface tangles, the two pairs of vertical disks belonging to $\bar Q_2$ are omitted,
assuming that they have been previously eliminated as above. The first step in both cases
starts with the further elimination of two of the three pairs of vertical disks, once
again by the same procedure as above even if a band replacement is needed in the first
case. This is followed by the inversion of some of the ribbon intersection formed by the
surviving pair, like in move \(S2). The resulting half-twists, together with the
preexisting one in the first case, are then canceled with the help of move \(R6).\break
At this point, the second step just consists of three \(R1) moves and 1-isotopy.

\begin{Figure}[htb]{xi-full11/fig}
{}{}
\vskip-6pt
\centerline{\fig{}{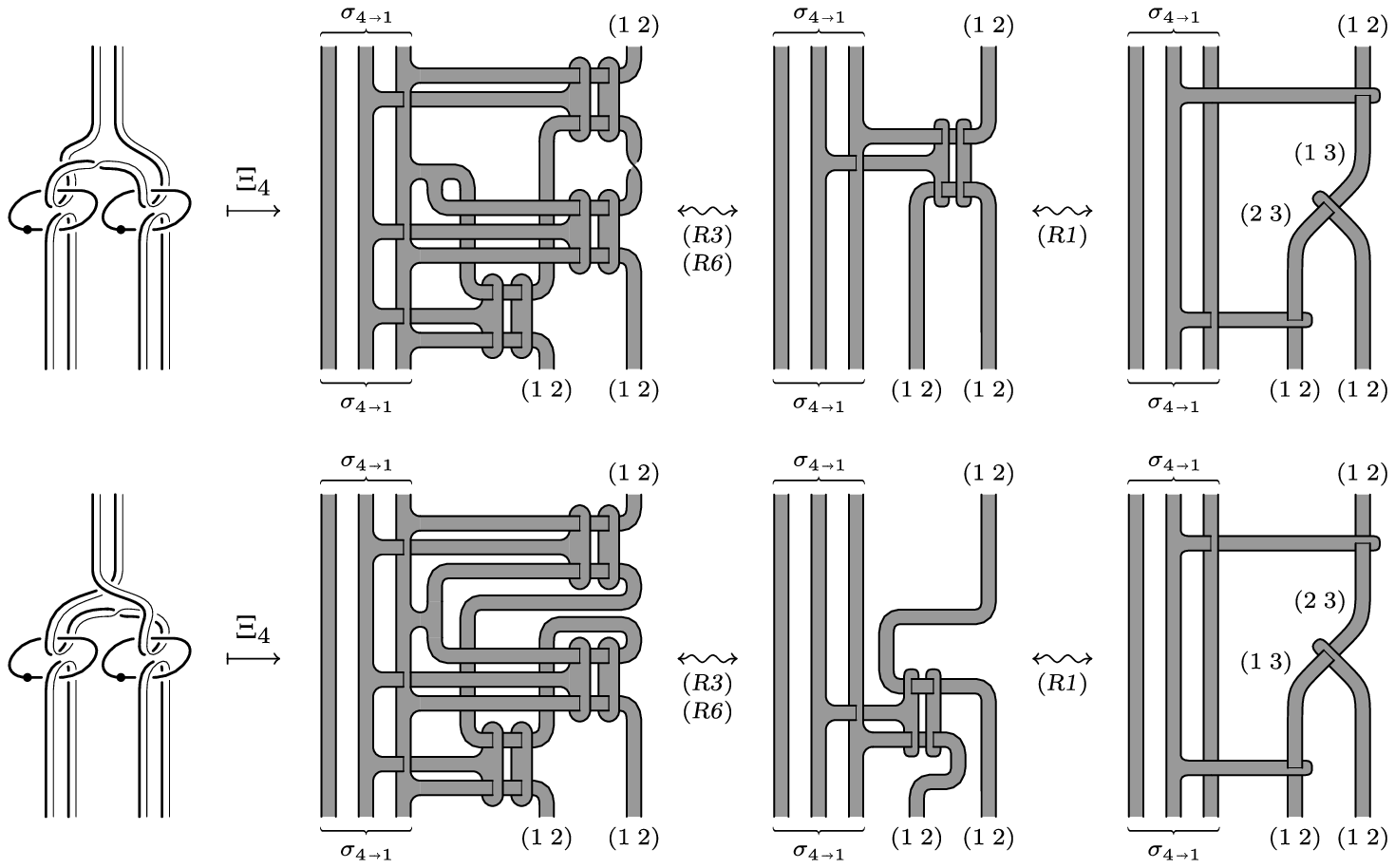}}
\vskip-3pt
\end{Figure}

Now, it remains to consider the cases of $\id_{\sigma_{4 \red 2}} \!\diam \Delta_{\tp12}
\diam \id_m$ with $m > 0$, and $\id_{\tp43} \diam (\zeta_{\sigma_1} \circ (\id_{\sigma_{3
\red 1}} \!\diam T) \circ \zeta_{\sigma_0}^{-1})$ with $T$ being a non-trivial expansion
in $\S_3$ of one of the special elementary morphisms in Figure \ref{xi-full01/fig}.

The following claim allows us to deduce these cases from the spacial ones we have just
considered, when $m = 0$ and $T$ is an elementary morphism, by a straightforward inductive
argument based on the number of the expansion ribbons.

Given two morphisms $K$ and $K'$ in $\K_1$, the claim relates the image of $K \diam K'$
under $\Xi_4$, with the images of $K$ and $K'$. Unfortunately, we do not have an explicit
monoidal structure on $\up^4_1 \S_{3 \red 1}$, and this makes the statement quite
technical. Actually, for the present aim the claim could be restricted by the assumption
that at least one of $T$ and $T'$ is an identity morphism, but this would not make its
proof simpler.

\begin{Figure}[b]{xi-full07/fig}
{}{The framed $2(m+m')$-braid $b_h$ for $\tp{i'_h}{j'_h} = \tp23$}
\vskip-3pt
\centerline{\fig{}{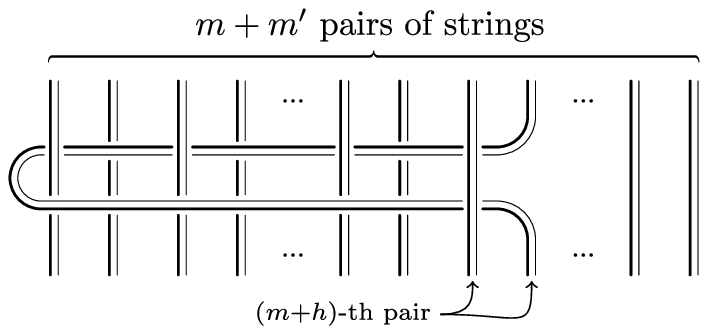}}
\vskip-3pt
\end{Figure}

\begin{statement}{Claim}
Let $T: J_{\sigma_{3 \red 1}} \!\diam J_{\sigma_0} \to J_{\sigma_{3 \red 1}} \!\diam
J_{\sigma_1}\!$ and $T': J_{\sigma'_0} \to J_{\sigma'_1}\!$ be morphisms in $\S_3$, such
that $\id_{\tp43} \diam (\zeta_{\sigma_1} \circ T \circ \zeta_{\sigma_0}^{-1}) = \Xi_4(K)$
and $\id_{\tp43} \diam (\zeta_{\sigma'_1} \circ (\id_{\sigma_{3 \red 1}} \!\diam T') \circ
\zeta_{\sigma'_0}^{-1}) =\break \Xi_4(K')$ for some $K$ and $K'$ in $\K_1$. Then
$$\id_{\tp43} \diam (\zeta_{\sigma_1 \diam \sigma'_1} \circ (T \diam T') \circ
\zeta_{\sigma_0 \diam \sigma'_0}^{-1}) = \Xi_4(B_{\sigma_1, \sigma'_1} \circ (K \diam K')
\circ \bar B_{\sigma_0 , \sigma'_0}^{-1})\,.$$
Here, given $\sigma = (\tp{i_1}{j_1}, \dots, \tp{i_m}{j_m})$ and $\sigma' =
(\tp{i'_1}{j'_1}, \dots, \tp{i'_{m'}}{j'_{m'}})$ in $\seq\Gamma_3$, we denote by
$B_{\sigma,\sigma'}$ the framed $2(m+m')$-braid defined as $B_{\sigma,\sigma'} = b_{m'}
\circ \dots \circ b_1$, where $b_h$ is the framed $2(m+m')$-braid shown in Figure
\ref{xi-full07/fig} if $\tp{i'_h}{j'_h} = \tp23$, while it is the identity otherwise.
\end{statement}

To prove the claim, let us consider $\zeta_{\sigma_1 \diam \sigma'_1} \circ (T \diam T')
\circ \zeta_{\sigma_0 \diam \sigma'_0}^{-1}$ and look at Figure \ref{xi-full05/fig},
where: $m_0,m'_0,m_1,m'_1$ denote the lengths of $\sigma_0,\sigma'_0,\sigma_1,\sigma'_1$
respectively. Here once again the reduction ribbon $\id_{\tp43}$ is omitted, being
involved only in performing the necessary band replacements as in Remark
\ref{gen-bands/rem}.

\begin{Figure}[htb]{xi-full05/fig}
{}{}
\centerline{\fig{}{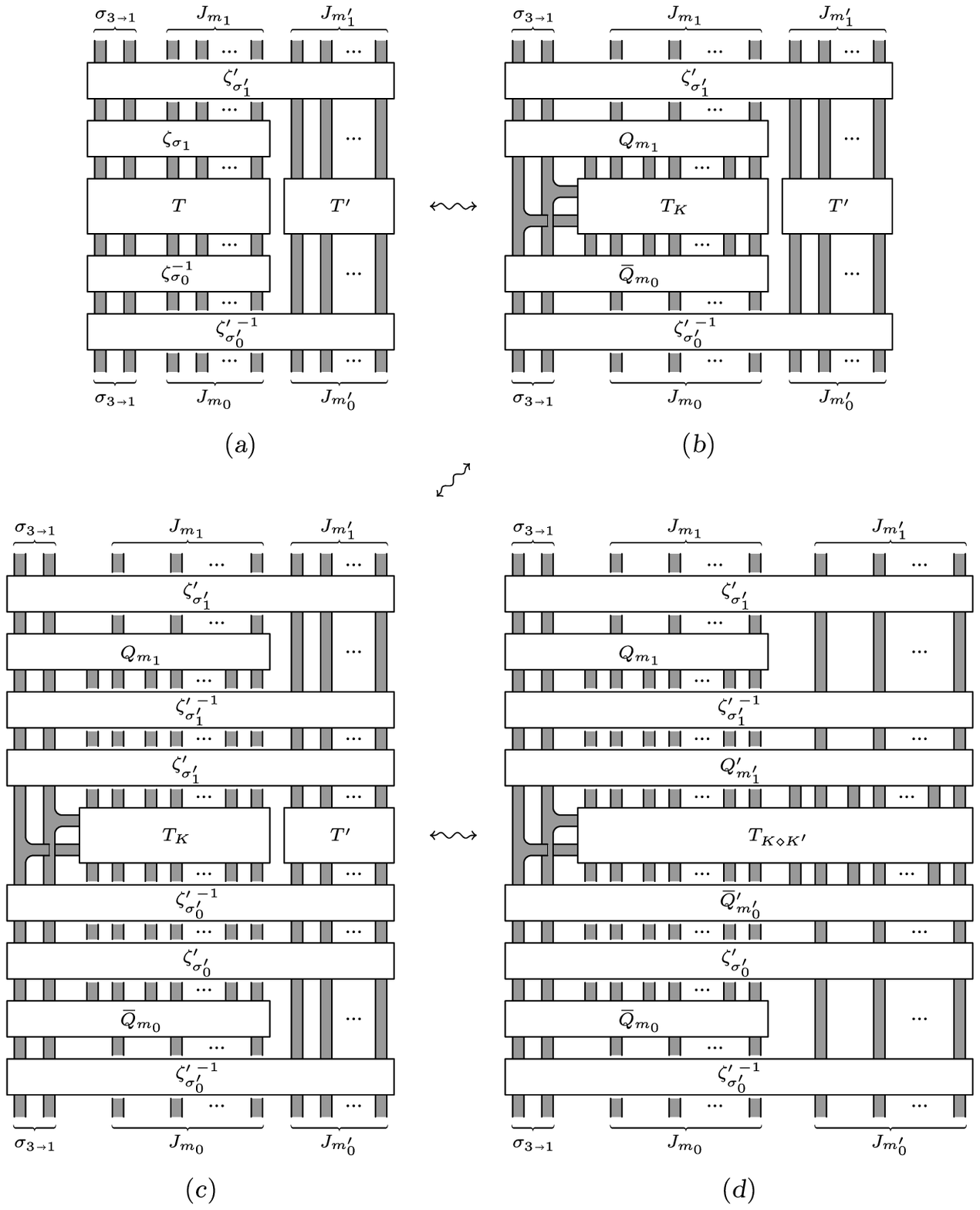}}
\vskip-3pt
\end{Figure}

Diagram \(a) in the figure is obtained by decomposing the natural isomorphisms
$\zeta_{\sigma_1 \diam \sigma'_1}$ and $\zeta_{\sigma_0 \diam \sigma'_0}^{-1}$, according
to the identity $\zeta_{\sigma \diam \sigma'} = \zeta'_{\sigma'} \circ (\zeta_\sigma \diam
\id_{\sigma'})$, with $$\zeta'_{\sigma'} = (\id_{4 \red 1} \diam \gamma_{\sigma,\sigma'})
\circ (\zeta_{\sigma'} \diam \id_{\sigma}) \circ (\id_{4 \red 1} \diam
\gamma_{\sigma,\sigma'}^{-1})=\zeta_{\sigma \diam \sigma',m+m'} \circ \dots \circ
\zeta_{\sigma \diam \sigma',m+1},$$ where $m$ and $m'$ denote the lengths of $\sigma$ and
$\sigma'$ respectively. Then, we get \(b) by using the hypothesis that $\id_{\tp43} \diam
(\zeta_{\sigma_1} \circ T \circ \zeta_{\sigma_0}^{-1}) = \Xi_4(K)$, and \(c) just by
inserting the canceling pairs ${\zeta'_{\sigma'_0}\!\!}^{-1} \circ \zeta'_{\sigma'_0}$ and
${\zeta'_{\sigma'_1}\!\!}^{-1} \circ \zeta'_{\sigma'_1}$. Here, the horizontal bands
forming the pseudo-products $\diam_{\beta,\gamma,\delta}$ are fused all together in their
terminal parts outside the boxes. In order to get \(d), we first change those bands into
new ones which pass under ${\zeta'_{\sigma'_0}\!\!}^{-1}$ and reach the reduction ribbons
$\id_{\sigma_{3 \red 1}}$ in the region between ${\zeta'_{\sigma'_0}\!\!}^{-1}$ and
${\zeta'_{\sigma'_0}\!\!}$ and then apply the hypothesis that $\id_{\tp43} \diam
(\zeta_{\sigma'_1} \circ (\id_{\sigma_{4 \red 1}} \!\diam T') \circ
\zeta_{\sigma'_0}^{-1}) = \Xi_4(K')$ and the fact that $R_{K}\rdiam R_{K'}=R_{K\diam K'}$
(see Remark \ref{non-monoidal/rem}). In diagram \(d), $Q'_{m'_1} = (\id_{3 \red 1} \diam
\gamma_{2m'_1, 2m_1}) \circ (Q_{m'_1} \diam \id_{2m_1}) \circ (\id_{3 \red 1} \diam
\gamma_{2m'_1, 2m_1}^{-1})$ coincides with the factor of $Q_{m_1 + m'_1}$ relative to
$J_{m'_1}$ in the decomposition $Q_{m_1 + m'_1} = (Q_{m_1} \diam \id_{m'_1}) \circ
Q'_{m'_1}$, while $\bar Q'_{m'_0} = (\id_{3 \red 1} \diam \gamma_{2m'_0, 2m_0}) \circ
(\bar Q_{m'_0} \diam \id_{2m_0}) \circ (\id_{3 \red 1} \diam \gamma_{2m'_0, 2m_0}^{-1})$
coincides with the factor of\break $\bar Q_{m_0 + m'_0}$ relative to $J_{m'_0}$ in the
decomposition $\bar Q_{m_0 + m'_0} = \bar Q'_{m'_0} \circ (\bar Q_{m_0} \diam
\id_{m'_0})$.

Now, we want to show that, in the presence of the reduction ribbon $\id_{\tp43}$, the
subtangles $\zeta'_{\sigma'_1}\!\circ (Q_{m_1} \!\diam \id_{m'_1}) \circ
{\zeta'_{\sigma'_1}\!\!}^{-1} \circ Q'_{m'_1}$ and $\bar Q'_{m'_0} \!\circ
{\zeta'_{\sigma'_0}\!\!}^{-1} \circ (\bar Q_{m_0} \diam \id_{m'_0}) \circ
\zeta'_{\sigma'_0}$ in Figure \ref{xi-full05/fig} \(d), can be respectively replaced by
$Q_{m_1 + m'_1} \circ Z_{\sigma_1, \sigma'_1}$ and $Z_{\sigma_0, \sigma'_0}^{-1} \circ
\bar Q_{m_0 + m'_0}$, with the tangles $Z$ defined as follows. Given $\sigma =
(\tp{i_1}{j_1}, \dots, \tp{i_m}{j_m})$ and $\sigma' = (\tp{i'_1}{j'_1}, \dots,
\tp{i'_{m'}}{j'_{m'}})$ in $\seq\Gamma_3$, $Z_{\sigma,\sigma'} = z_{m'} \circ \dots \circ
z_1$, where $z_h$ is the tangle shown in Figure \ref{xi-full06/fig} if $\tp{i'_h}{j'_h}
= \tp23$, while it is the identity otherwise. Observe that the $z_h$'s, hence
$Z_{\sigma,\sigma'}$ as well, are isomorphisms with their inverses obtained by vertical
reflection.

\begin{Figure}[htb]{xi-full06/fig}
{}{The 3-labeled ribbon surface tangle $z_h$ for $\tp{i'_h}{j'_h} = \tp23$}
\vskip-6pt
\centerline{\fig{}{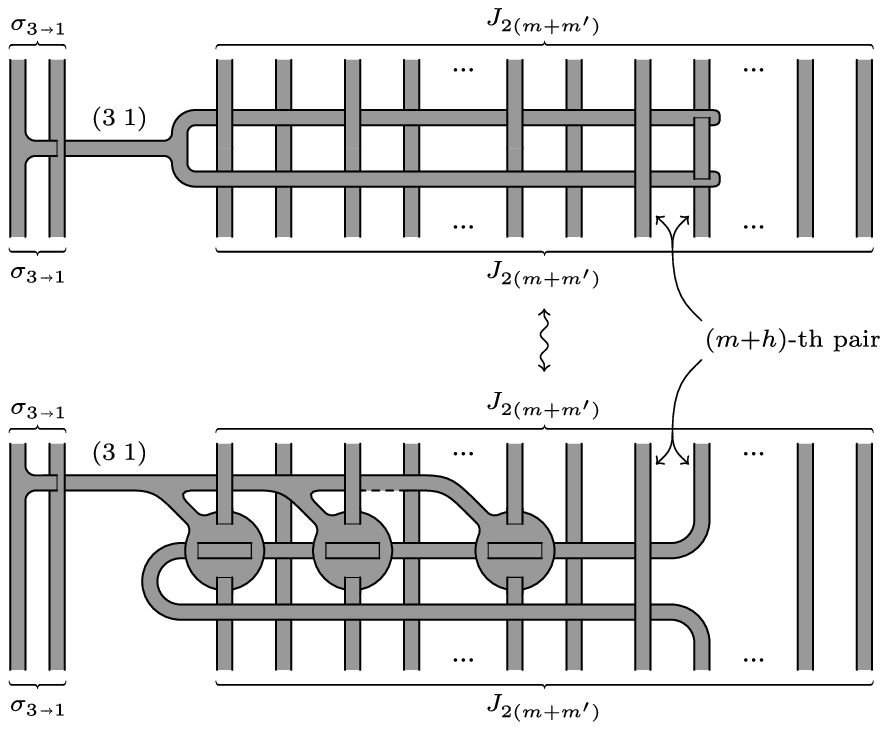}}
\vskip-3pt
\end{Figure}

The proof that $\zeta'_{\sigma'_1} \!\circ (Q_{m_1} \!\diam \id_{m'_1}) \circ
{\zeta'_{\sigma'_1}\!\!}^{-1} \circ Q'_{m'_1}$ can be replaced by $ Q_{m_1 + m'_1} \circ
Z_{\sigma_1, \sigma'_1}$ goes by induction on the length $m_1'$ of $\sigma_1'$, starting
with the trivial case of $m_1' = 0$.

The inductive step for $m_1' \geq 1$ presents some difficulties only when
$\tp{i'_{m_1'}}{j'_{m_1'}} = \tp23$. Indeed, if $\tp{i'_{m_1'}}{j'_{m_1'}} = \tp12$, both
$\zeta_{\sigma_1 \diam \sigma'_1,m_1 + m_1'}$ and $\zeta_{\sigma_1 \diam \sigma'_1, m_1 +
m_1'}^{-1}$ are identities, so there is nothing to prove. While if
$\tp{i'_{m_1'}}{j'_{m_1'}} = \tp13$, after a suitable replacement of the bands attached to
$\id_{\tp32}$, $\zeta_{\sigma_1 \diam \sigma'_1,m_1 + m_1'}$ and $\zeta_{\sigma_1 \diam
\sigma'_1,m_1 + m_1'}^{-1}$ can be moved to be contiguous and then canceled.

\begin{Figure}[b]{xi-full08/fig}
{}{}
\vskip-3pt
\centerline{\fig{}{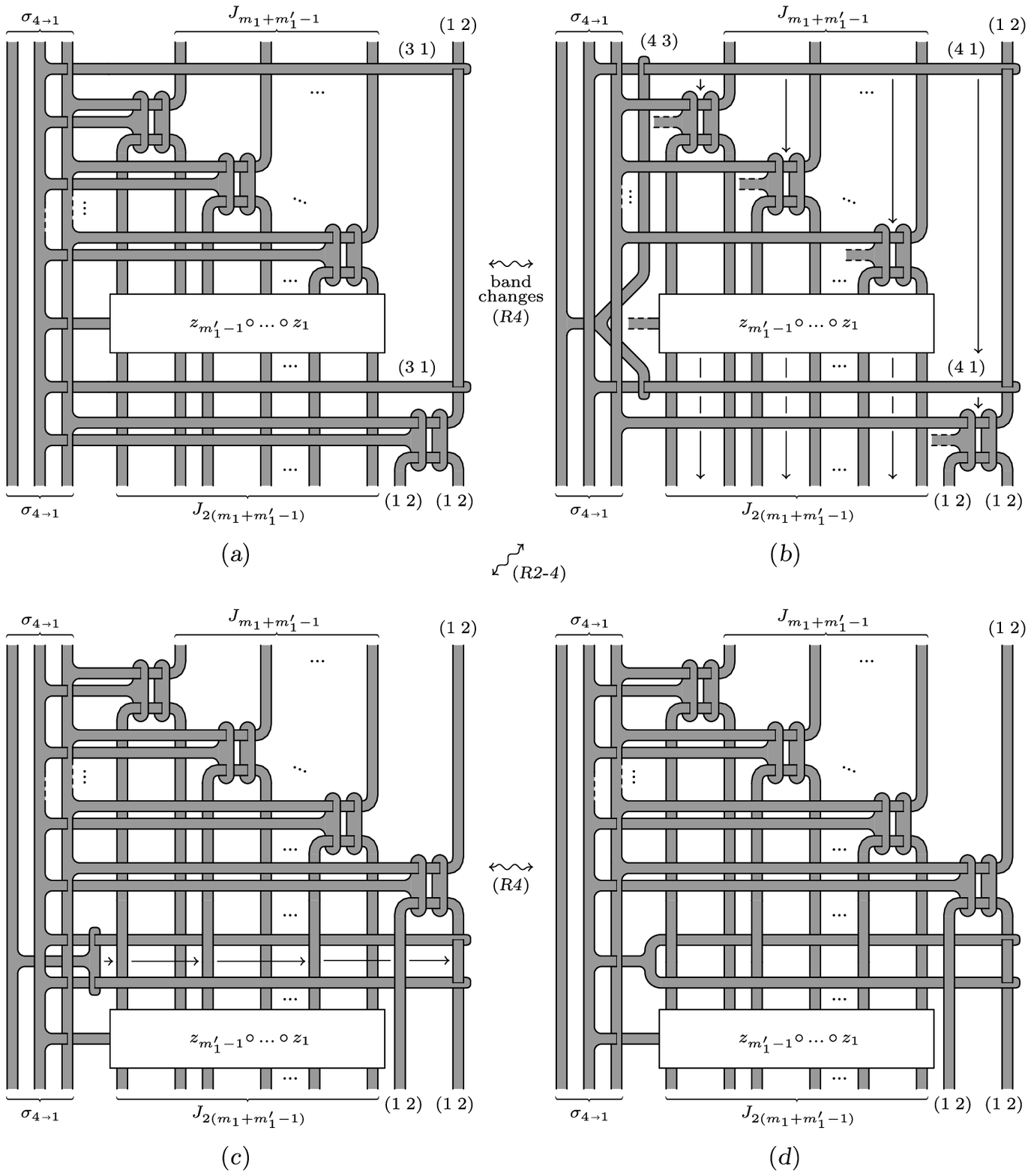}}
\vskip-3pt
\end{Figure}

The case of $\tp{i'_{m_1'}}{j'_{m_1'}} = \tp23$ is treated in Figure \ref{xi-full08/fig}.
Here, diagram \(a) is obtained by applying the induction hypothesis to the first $m_1'-1$
blocks of $Q'_{m_1'}$. In the same diagram, the horizontal bands labeled $\tp31$ represent
$\zeta_{\sigma_1 \diam \sigma'_1, m_1 + m_1'}$ and $\zeta_{\sigma_1 \diam \sigma'_1, m_1 +
m_1'}^{-1}$, and all the bands that should be connecting the left side of the box to the
reduction ribbon $\id_{\tp32}$ are fused together. To get \(b), we first replace the bands
connecting the vertical disks to $\id_{\tp32}$ with new ones, whose attaching arcs in
$\id_{\tp32}$ lie over that of $\zeta_{\sigma_1 \diam \sigma'_1, m_1 + m_1'}$, on the top
of the diagram. Similarly, we replace the bands coming out from the left side of the box,
moving their attaching arcs in $\id_{\tp32}$ under that of $\zeta_{\sigma_1 \diam
\sigma'_1, m_1 + m_1'}^{-1}$, on the bottom of the diagram. All those replacements can be
done thanks to Remark \ref{gen-bands/rem}. The new bands (not drawn in the diagram) can be
suitably chosen, in such a way that they allow the slidings we are going to perform in
order to get \(c). Moreover, we expand a tongue from $\id_{\tp43}$ and move it by
1-isotopy and move \(R4) until it reaches the final position in diagram \(b). As a
consequence, the horizontal bands forming $\zeta_{\sigma_1 \diam \sigma'_1, m_1 + m_1'}$
and $\zeta_{\sigma_1 \diam \sigma'_1, m_1 + m_1'}^{-1}$ acquire labels $\tp41$. Then, we
slide down those bands and the box as suggested by the arrows, up to their new position in
\(c). In doing that, we use move \(R4) to let the topmost band pass through the two short
bands labeled $\tp23$ in the middle of each of the first $m_1+m_1'-1$ pairs of vertical
disks. While for the $(m_1+m_1')$-th pair, we use move \(R2) to let the rightmost vertical
disk pass through both the $\zeta$'s. To pass from \(c) to \(d), we disentangle the
vertical disk attached to $\id_{\tp43}$ by moving it to the right, as indicated by the
arrow and using once again move \(R4) to let it pass through the vertical bands, and
finally we retract the resulting tongue to $\id_{\tp43}$.

This completes the proof that $\zeta'_{\sigma'_1} \!\circ (Q_{m_1} \!\diam \id_{m'_1})
\circ {\zeta'_{\sigma'_1}\!\!}^{-1} \circ Q'_{m'_1}$ can be replaced with $Q_{m_1 + m'_1}
\circ Z_{\sigma_1, \sigma'_1}$. The replacement of $\bar Q'_{m'_0} \!\circ
{\zeta'_{\sigma'_0}\!\!}^{-1} \circ (\bar Q_{m_0} \diam \id_{m'_0}) \circ
\zeta'_{\sigma'_0}$ by $Z_{\sigma_0, \sigma'_0}^{-1} \circ \bar Q_{m_0 +m'_0}$ is
symmetric and it is left to the reader.

Once both those replacements have been performed, we are left with $Z_{\sigma_1,
\sigma'_1}\circ (\id_{\sigma_{3\red 1}}\diam_{\beta,\gamma,\delta} T_{K\diam K'})\circ
Z_{\sigma_0, \sigma'_0}^{-1}$, which is equal to $\check S_{K''}$ for
$K'' = B_{\sigma_1, \sigma'_1} \circ (K \diam K') \circ B_{\sigma_0,\sigma'_0}^{-1}$ and a
suitable choice of the vertically trivial state involved in the construction.
\end{proof}

Before going on to our main theorem, we want to see the effect of Proposition
\ref{full-xi/thm} on the definition of the functor $\Xi_4$ itself. In the previous
section we defined it by putting $\Xi_4(K) = \up_3^4 \check S_K$ for any Kirby tangle $K
\in \K_1$, where the check means that the trivial state in point 3 of the construction of
the ribbon surface tangle $S_K$ (see Section \ref{fullness/sec}) has been actually
required to be vertically trivial. As a consequence of Proposition
\ref{full-xi/thm}, we can now relax such extra requirement and remove the check from
the definition of $\Xi_4$. This is the content of the next proposition.

\begin{proposition}\label{trivial-state/thm}
Given a Kirby tangle $K \in \K_1$, the equivalence class of the labeled ribbon surface
tangle $\up_3^4 S_K$ does not depend of the choices involved in the construction of $S_K$.
In particular, it does not depend on the choice of the trivial state occurring in point 3
of that construction. Then, we can write $\Xi_4(K) = \up_3^4 S_K$ without requiring
any more the vertically triviality of such state.
\end{proposition}

\begin{proof}
Let $K$ and $S_K$ be as in the statement. Since the equivalence class of $\up_3^4 S_K$ is
a morphism between two objects in the image of $\Xi_4$, Proposition \ref{full-xi/thm}
tells us that it can be written in the form $\Xi_4(K') = \up_3^4 \check S_{K'}$ for some
Kirby tangle $K' \in \K_1$. Then, Propositions \ref{K-reduction/thm}, \ref{theta/thm},
\ref{stab-theta/thm} and \ref{full-theta3/thm}, give us the chain of equalities $K =
\Theta_3(S_K) = \down_1^4 \Theta_4(\up_3^4 S_K) = \down_1^4 \Theta_4(\up_3^4 \check
S_{K'}) = \Theta_3(\check S_{K'}) = K'$ holding in $\K_1$. Moreover, from $K = K'$ in
$\K_1$ we get $\up_3^4 \check S_K = \up_3^4 \check S_{K'}$ in $\up^4_1 S_{3\red 1}$,
thanks to Proposition \ref{xi4/thm}. Thus, the starting ribbon surface tangle $S_K$ turns
out to be equivalent to $\check S_K$. At this point, Lemmas \ref{SK-welldef/thm} and
\ref{SK-invariance/thm} allow us to conclude the proof.
\end{proof}

Finally, let us state and prove the main equivalence theorem.

\begin{theorem}\label{ribbon-kirby/thm}
For any $n \geq 4$, the functor $\,\Xi_n: \K_1 \to \S_n^c$ and the braided monoidal
functor $\Theta_n: \S_n^c\to \K_n^c$ (cf. Proposition \ref{theta/thm}) are category
equivalences. Moreover, ${\down_1^n} \circ \Theta_n \circ \Xi_n = \id_{\K_1}$, while
$\,\Xi_n \circ {\down_1^n} \circ \Theta_n$ is naturally equivalent to $\id_{\S_n^c}$.
\end{theorem}

\begin{proof}
According to Proposition \ref{xi/thm}, ${\down_1^n} \circ \Theta_4 \circ \Xi_4 =
\id_{\K_1}$, hence $\Xi_4: \K_1 \to \S_n^c$ is a faithful functor. Moreover, Proposition
\ref{full-xi/thm} implies that $\Xi_4$ is full and that any object in $\S_n^c$ is
isomorphic to one in its image. Then, $\Xi_4$ is a category equivalence by Proposition
\ref{cat-equiv/thm}. Since $\up_1^n$ is also a category equivalence by Proposition
\ref{K-reduction/thm}, we have that $\Xi_n = {\up_4^n} \circ \Xi_4: \K_1 \to \S_n^c$ is
such for any $n\geq 4$. Since ${\down_1^n} \circ \Theta_n \circ \Xi_n =
\id_{\K_1}$ by Proposition \ref{xi/thm}, $\Theta_n: \S_n^c \to \K_n^c$ is a category
equivalence as well.
\end{proof}

\newpage

\section{Universal groupoid ribbon Hopf algebra%
\label{algebra/sec}}

This chapter is dedicated to the construction and the study of the algebraic categories
$\H_n^{r}$ and of the functors $\Phi_n: \H_n^{r} \to \K_n$. For any $n \geq 1$, we will
define a suitable subcategory $\H_n^{r,c}\subset \H_n^{r}$ together with an equivalence
reduction functor\break $\down_1^n: \H_n^{r,c} \to \H_1^r$, and show that the restriction
$\Phi_n: \H_n^{r,c} \to \K_n^c$ is a category equivalence. In particular, the algebra
$\H^r = \H_1^r$ will give an algebraic characterization of the category $\Chb^{3+1} =
\Chb_1^{3+1}$ of 4-dimensional relative 2-handlebody cobordisms, in the sense of Kerler
(\cite{Ke02}).

The proof is based on the factorization of the category equivalence $\Theta_n: \S_n^c \to
\K_n^c$ defined in Section \ref{Theta/sec}, as a composition $\Phi_n \circ \Psi_n$, where
for $n\geq 4$ the functor $\Psi_n: \S_n^c \to\H_n^{r,c}$ is shown to be full. This implies
that both $\Psi_n $ and $\Phi_n$ are category equivalences. In particular, $\Psi_n$ gives
an algebraic interpretation of the simple branched covering representation of
4-dimensional relative 2-handlebody cobordisms.

$\H_n^r$ will be defined as the strict monoidal braided category freely generated by a
groupoid Hopf algebra for the groupoid $\G_n = \{1, \dots, n\}^2$, thought with its 
natural composition given by $(i,j) (j,k) = (i,k)$ for any $1 \leq i,j,k \leq n$.

Before going on, let us see how this groupoid structure of $\G_n$ fits into the picture.
As it has been established in \cite{Ke02, Ha00}, $I_{(1,1)} \in \Obj \K_1$ is a braided
Hopf algebra object with the comultiplication morphism $\Delta_{(1,1)}$ described in
Section \ref{K/sec}. Actually, by Proposition \ref{K-coassociativity/thm},
$\Delta_{(i,j)}: I_{(i,j)} \to I_{(i,j)} \diam I_{(i,j)}$ makes any $I_{(i,j)} \in
\Obj\K_n$ into a coalgebra object. From the topological point of view, $\Delta_{(i,j)}$
represents a single 1-handle along which run the attaching maps of three 2-handles. The
dual notion of multiplication morphism (see below) corresponds to a 2-handle which runs
along exactly three 1-handles. Now, while in $\K_1$ one can always attach such a 2-handle
along any given three 1-handles (since they are all attached on the same 0-handle), the
same is not true in $\K_n$. Here, to be able to close the loop of the attaching sphere,
the indices of the 1-handles need to be in the order $(i,j)$, $(j,k)$ and $(i,k)$. Then,
the algebraic structure of $\K_n$ involves multiplication morphisms $m_{(i,j),(i',j')}:
I_{(i,j)} \diam I_{(i',j')} \to I_{(i'',j'')}$ defined only for $i' = j$, $i'' = i$ and
$j'' = j'$, in other words when $(i,j)$ and $(i',j')$ are composable and $(i,j) (i',j') =
(i'',j'')$ in the groupoid $\G_n$. Moreover, the morphisms $\eta_i: \emptyset \to
I_{(i,i)}$, each given by a 2-handle that cancels against the 1-handle represented by
$I_{(i,i)}$, act as units for the multiplication. Therefore, in $\H^r_n$ we will have a
family of objects labeled by $\G_n$ and a family of multiplication and unit morphisms
between them, which reflects the groupoid structure of $\G_n$.

Actually, $\H_n^r$ will be introduced as a specialization of the more general notion of
universal groupoid ribbon Hopf algebra $\H^r(\G)$, with $\G$ being an arbitrary finite
groupoid. Disregarding the ribbon structure, the notion of groupoid braided Hopf algebra
$\H(\G)$ extends the one of group braided Hopf algebra, used in \cite{Vi01, Tu00} to
define TQFT invariants of regular (unbranched) coverings of three manifolds, just because
any group can be considered as a groupoid. The study the algebras $\H(\G)$ and $\H^r(\G)$
for an arbitrary groupoid $\G$ does not bring to any additional technical difficulties
with respect to the one of $\H_n^r$, allowing on the contrary somewhat simpler notations.

The total list of axioms of a groupoid ribbon Hopf algebra is quite long (11 elementary
morphisms and 34 relations between them). Moreover, these axioms have many important
consequences, to which we will refer as properties of the algebra. Of course, a given
property usually depends only on a small subset of axioms. So, in order to make more clear
the logical structure of such implications, we decided to introduce the axioms in four
steps. We start by presenting the axioms of a braided Hopf algebra and showing that the
corresponding universal category $\H(\G)$ carries an autonomous structure. Then, we
require the unimodularity condition, which leads to a tortile category $\H^u(\G)$. Here,
we require the existence of ribbon morphisms satisfying Kerler's axioms in \cite{Ke02}
with the exception of the self-duality condition (we remind that we are looking for an
algebra which describes 4-dimensional 2-handlebodies; self-duality will be considered
later in Chapter \ref{boundaries/sec}, where we study the category of framed 3-dimensional
cobordisms). The universal unimodular Hopf algebra with such ribbon morphisms will be
called a pre-ribbon Hopf algebra and will be denoted by $\H^u_v(\G)$. Finally, the
universal ribbon Hopf algebra $\H^r(\G)$ is obtained by adding two new axioms which relate
the ribbon morphism with the coalgebraic and braided structures.

The diagrammatic language, developed and used in the literature (see for example
\cite{Ha00, Ke02}), will be our main tool in representing the morphisms in the algebraic
categories and in the study of the functors between the geometric and the algebraic
categories. The next sections contain widespread references to the axioms and properties
of the universal ribbon Hopf algebra. In order to simplify the search of such references,
we assign a compact name (a letter and a number) to each axiom and property and collect
all the corresponding diagrams in few tables.

\subsection{The universal groupoid Hopf algebras $\H(\G)$ and $\H^u(\G)$%
\label{HG/sec}}

Let $\G$ be a groupoid, that is a small category whose morphisms are all invertible. Will
denote by $\G$ also the set of the morphisms of $\G$, endowed with the partial binary
operation given by the composition, for which we adopt the multiplicative notation from
left to right. The identity of $i \in \Obj\G$ will be denoted by $1_i \in \G$, while the
inverse of $g \in \G$ will be denoted by $\bar g \in \G$. For $i,j \in \Obj\G$, we denote
by $\G(i,j) \subset \G$ the subset of morphisms from $i$ to $j$. Consequently, if $g \in
\G(i,j)$ and $h \in \G(j,k)$ then $gh \in \G(i,k)$. In particular, $g \bar g = 1_i$ and
$\bar g g = 1_j$, and sometimes the identity morphisms will be represented in this way.

A groupoid is called {\sl connected} if $\G(i,j)$ is non-empty for any $i,j \in \Obj\G$.
Given two groupoids $\G \subset \G'$, we say that $\G$ is {\sl full} in $\G'$ if $\G$ is a
full subcategory of $\G'$, i.e. $\G(i,j) = \G'(i,j)$ for all $i,j \in \Obj \G$. Moreover,
given $k\in\Obj \G$, we denote by $\G^{\bs k}$ the full subgroupoid of $\G$ with $\Obj
\G^{\bs k} = \Obj\G - \{k\}$

\medskip

Now, we start with the definition of the notion of Hopf $\G$-algebra in a braided monoidal
category $\C$. This involves a family $H = \{H_g\}_{g \in \G}$ of objects of $\C$ and
certain families of morphisms indexed by (possibly pairs of) such objects.

\medskip

{\sl Here and in the sequel, we will write $g$ instead of $H_g$ in the subscripts of the
notation for morphisms of $\C$. For example, we will use the notations $\id_g = \id_{H_g}$
and $\gamma_{g,h} = \gamma_{H_g,H_h}$. Moreover, based on the MacLane's coherence result
for monoidal categories (p. 161 in \cite{McL}), we will omit the associativity morphisms
since they can be filled in a unique way}.

\begin{definition}\label{hopf-algebra/def}
Given a groupoid $\G$ and a braided monoidal category $\C$, a {\sl Hopf $\G$-algebra} in
$\C$ is a family of objects $H = \{H_g\}_{g \in \G}$ in $\C$, equipped with the families
of morphisms in $\C\,$ described below.

\smallskip\noindent
A {\sl comultiplication} $\Delta = \{\Delta_g: H_g \to H_g \diam H_g\}_{g \in \G}$,
such that for any $g \in \G$:
\vskip-4pt
$$ (\Delta_g\diam\id_g) \circ \Delta_g = (\id_g\diam\Delta_g) \circ \Delta_g;
\eqno{\(a1)}$$
\vskip4pt

\smallskip\noindent
a {\sl counit} $\epsilon = \{\epsilon_g: H_g \to \one\}_{g \in \G}$, such that
for any $g \in \G$:
\vskip-4pt
$$(\epsilon_g\diam\id_g)\circ\Delta_g = \id_g = (\id_g\diam\epsilon_g)\circ\Delta_g;
\eqno{\(a2-2')}$$
\vskip4pt

\smallskip\noindent
a {\sl multiplication} $m = \{m_{g,h}: H_g \diam H_h \to H_{gh}\}_{g,h,gh \in
\G}$ (notice that $m_{g,h}$ is defined only when $g$ and $h$ are composable in $\G$), such 
that for any $f,g,h,fgh \in \G$:
\vskip-4pt
$$m_{fg,h} \circ (m_{f,g} \diam \id_h) = m_{f,gh} \circ (\id_f \diam m_{g,h}),
\eqno{\(a3)}$$
$$(m_{g,h} \diam m_{g,h}) \circ (\id_g \diam \gamma_{g,h}\diam \id_h) \circ
(\Delta_g \diam \Delta_h) = \Delta_{gh} \circ m_{g,h},
\eqno{\(a5)}$$
$$\epsilon_{gh} \circ m_{g,h} = \epsilon_g \diam \epsilon_h;                  
\eqno{\(a6)}$$
\vskip4pt

\smallskip\noindent
a {\sl unit} $\eta = \{\eta_i: \one \to H_{1_i}\}_{i \in \Obj \G}$, such
that for any $g \in \G(i,j)$:
\vskip-4pt
$$m_{g,{1_j}}\circ(\id_g\diam\eta_j) = \id_g = m_{{1_i},g}\circ(\eta_i\diam\id_g),
\eqno{\(a4-4')}$$
$$\Delta_{1_i} \circ \eta_i = \eta_i \diam \eta_i,
\eqno{\(a7)}$$
$$\epsilon_{1_i} \circ \eta_i = \id_{\one};
\eqno{\(a8)}$$
\vskip4pt

\smallskip\noindent
 an {\sl antipode} $S = \{S_g: H_g \to H_{\bar g}\}_{g \in \G}$ and its inverse $\bar 
S = \{\bar S_g: H_g \to H_{\bar g}\}_{g \in \G}$, such that for any $g\in \G(i,j)$:
\vskip-12pt
$$m_{\bar g,g}\circ(S_g\diam\id_g)\circ\Delta_g = \eta_{1_j}\circ\epsilon_g,
\eqno{\(s1)}$$
$$m_{g,\bar g}\circ(\id_g\diam S_g)\circ\Delta_g = \eta_{1_i} \circ \epsilon_g,
\eqno{\(s1')}$$
$$S_{\bar g}\circ\bar S_g =\bar S_{\bar g}\circ S_g=\id_g.
\eqno{\(s2-2')}$$
\vskip-9pt
\end{definition}

We observe that an ordinary braided Hopf algebra in $\C$ is a Hopf $\G_1$-algebra, where
$\G_1$ is the trivial groupoid with a single object and a single morphism. In particular,
$H_{1_i}$ is a braided Hopf algebra in $\C$ for any $i \in \Obj \G$.

\begin{definition}\label{integral/def}
Let $\C$ be a braided monoidal category and $H = \{H_g\}_{g \in \G}$ be a Hopf
$\G$-algebra in $\C$. By a categorical {\sl left} (resp. {\sl right}\/) {\sl cointegral}
of $H$ we mean a family $l = \{l_i: H_{1_i} \to \one \}_{i \in \Obj \G}$ of morphisms in
$\C$, such that for any $i \in \Obj \G$
$$
\begin{array}{c}
 (\id_{1_i} \diam l_i) \circ \Delta_{1_i} = \eta_i \circ l_i: 
 H_{1_i} \to H_{1_i}\\[6pt]
 (\text{resp. } (l_i \diam \id_{1_i}) \circ \Delta_{1_i} = \eta_i \circ l_i: 
 H_{1_i} \to H_{1_i}).
\end{array} 
\eqno{\(i1-1')}
$$

\noindent
On the other hand, by categorical {\sl right} (resp. {\sl left}\/) {\sl integral} of
$H$ we mean a family $L = \{L_g: \one \to H_g\}_{g \in \G}$ of morphisms in $\C$, such
that if $g,h,gh \in \G$ then
$$
\begin{array}{c}
 m_{g,h} \circ (L_g \diam \id_h) = L_{gh} \circ \epsilon_h: 
 H_h \to H_{gh}\\[6pt]
 (\text{resp. } m_{g,h} \circ (\id_g \diam L_h) = L_{gh} \circ \epsilon_h: 
 H_g \to H_{gh}).
\end{array}
\eqno{\(i2-2')}
$$
If $l$ (resp. $L$) is both right and left categorical cointegral (integral) of $H$, we
call it simply a cointegral (integral) of $H$.
\end{definition}

\begin{Table}[p]{table-Hdefn/fig}
{}{}
\centerline{\fig{}{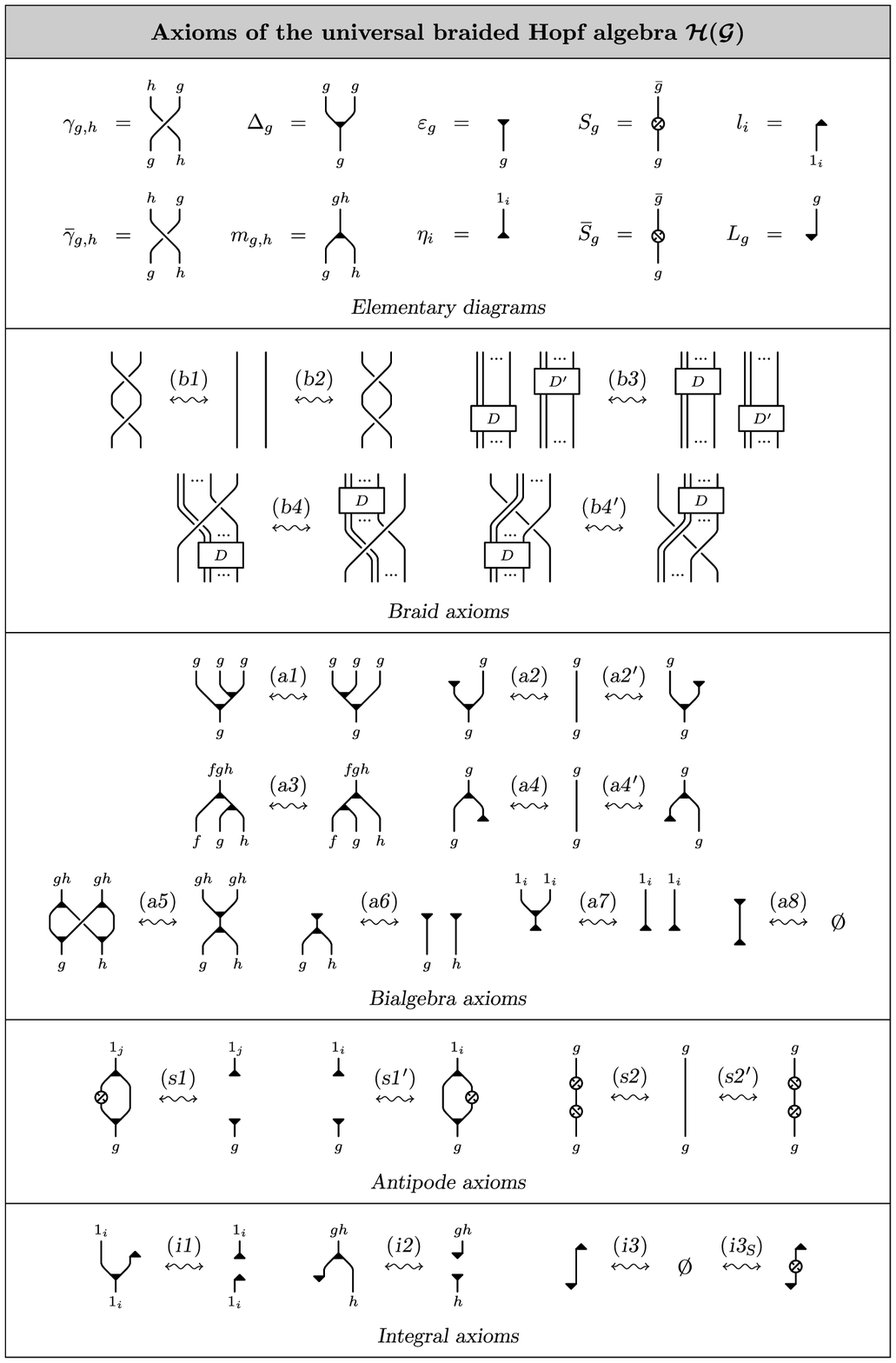}}
\vskip-3pt
\end{Table}

\begin{definition}\label{HG/def}
Given a groupoid $\G$, the {\sl universal Hopf $\G$-algebra} $\H(\G)$ is the strict
braided monoidal category freely generated by a Hopf $\G$-algebra $H =\break \{H_g\}_{g
\in \G}$ with a left cointegral $l$ and a right integral $L$, modulo the following
additional relations for any $i \in \Obj\G$
$$l_i \circ L_{1_i} = \id_{\one} = l_i \circ S_{1_i} \circ L_{1_i}\,.
\eqno{\(i3-3{$_{\text{S}}$})}$$
\vskip-12pt
\end{definition}

According to Definition \ref{presentation/def}, $\Obj \H(\G)$ is the free
monoid $\seq H = \cup_{m=0}^\infty H^m$ generated by $H = \{H_g\}_{g \in \G}$ and
consisting of all (possibly empty) finite sequences (that is products) of elementary
objects. We will use the notation $$H_\pi = H_{g_1} \!\diam \dots \diam H_{g_m}$$ for the
sequence corresponding to $\pi = (g_1, \dots, g_m) \in \seq\G$, in such a way that
$H_\emptyset$ is the unit object and $H_\pi \diam H_{\pi'} = H_{\pi \diam \pi'}$ for any
$\pi,\pi' \in \seq\G$.

\medskip

{\sl Extending the notational convention made above, we will write $\pi$ instead of
$H_\pi$ in the subscripts of the notation for morphisms of $\H(\G)$. For example, we will
use the notations $\id_\pi = \id_{H_\pi}$, $\gamma_{\pi,\pi'} = \gamma_{H_\pi,H_{\pi'}}$,
$\bar \gamma_{\pi,\pi'} = \bar \gamma_{H_\pi,H_{\pi'}}$ for any $\pi,\pi' \in \seq\G$.}

\medskip

On the other hand, $\Mor \H(\G)$ consists of all the compositions of products of
identities and one of the elementary morphisms $\gamma_{g,h}, \bar\gamma_{g,h}, \Delta_g,
\epsilon_g, m_{g,h}, \eta_i, S_g, \bar S_g, L_g, l_i$ as in Definition
\ref{hopf-algebra/def}, modulo the defining axioms for a braided structure and for a Hopf
$\G$-algebra with integrals listed in Definitions \ref{hopf-algebra/def} and
\ref{integral/def} (cf. Table \ref{table-Hdefn/fig}).

Moreover, $\H(\G)$ satisfies the following universal property: if $\C$ is any braided
monoidal category with a Hopf $\G$-algebra $H' = \{H'_g\}_{g \in \G}$ in it, and $H'$ has
a left cointegral and a right integral related by conditions \(i3-3{$_{\text{S}}$}), then
there exists a braided monoidal functor $\H(\G) \to \C$ sending $H_g$ to $H'_g$.

Analogously to \cite{Ke02}, $\H(\G)$ can be described as a category of planar diagrams in
$[0,1] \times [0,1]$. The objects of $\H(\G)$ are sequences of points in $[0,1]$ labeled
by elements in $\G$, and the morphisms are iterated products and compositions of the
elementary diagrams presented in Table \ref{table-Hdefn/fig}, modulo the relations
presented in the same figure and plane isotopies which preserve the $y$-coordinate. We
remind that the composition of diagrams $D_2 \circ D_1$ is obtained by stacking $D_2$ on
the top of $D_1$ and then rescaling, while the product $D_1 \diam D_2$ is given by
the horizontal juxtaposition of $D_1$ and $D_2$ and rescaling.

The plane diagrams and the relations between them are going to be our main tool. Observe
that the diagrams we use consist in projections of embedded graphs in $R^3$ with uni-, bi-
and tri-valent vertices. The vertices correspond to the defining morphisms in the algebra,
and we will call them with the name of the corresponding morphism. For example the
bi-valent vertices (which have one incoming and one outgoing edge) will be called antipode
vertices. Except the antipode ones, the rest of the vertices are represented by triangles
that point up (positively polarized) or point down (negatively polarized). The uni-valent
vertices are divided in unit vertices (corresponding to $\eta$ and $\epsilon$) and
integral vertices (corresponding to $l$ and $\Lambda$), while the positively (resp.
negatively) polarized tri-valent vertices will be called multiplication (resp.
comultiplication) vertices. Observe that the choice of polarization of the vertices is not
arbitrary.
\pagebreak
Indeed, as we will see in Section \ref{Phi/sec}, in the category of generalized
Kirby tangles $\K_n$ the univalent vertices with same polarization correspond to morphisms
with the same handle structure (upside down).

Most of our proofs consist in showing that some morphisms in the universal algebra are
equivalent, meaning that the graph diagram of one of them can be ob\-tained from the graph
diagram of the other by applying a sequence of the defining relations (moves) of the
algebra axioms. We will outline the main steps in this procedure by drawing in sequence
some intermediate diagrams, and for each step we will indicate in the corresponding order,
the main moves needed to transform the diagram on the left into the one on the right.
Actually, some steps can be understood more easily by starting from the diagram on the
right and reading the moves in the reverse order.\break Notice, that the moves represent
equivalences of diagrams and we use the same notation for them and their inverses. In the
captions of the figures the reader will find (in square brackets) the reference to the
pages where those moves are defined. As an example, the reader can see the proof of
Proposition \ref{antipode/thm}, where we have added some additional comments in order to
make clearer the interpretation of the figures.

We now proceed with the study of the properties of the category $\H(\G)$ listed in Table
\ref{table-Hprop/fig}. Such properties are divided in two sets. The first one generalizes
to the case of a groupoid Hopf algebra the well-known properties of the antipode (see
\cite{Ke02} and \cite{Vi01} for the case of braided and group Hopf algebras).

\begin{Table}[b]{table-Hprop/fig}
{}{}
\centerline{\fig{}{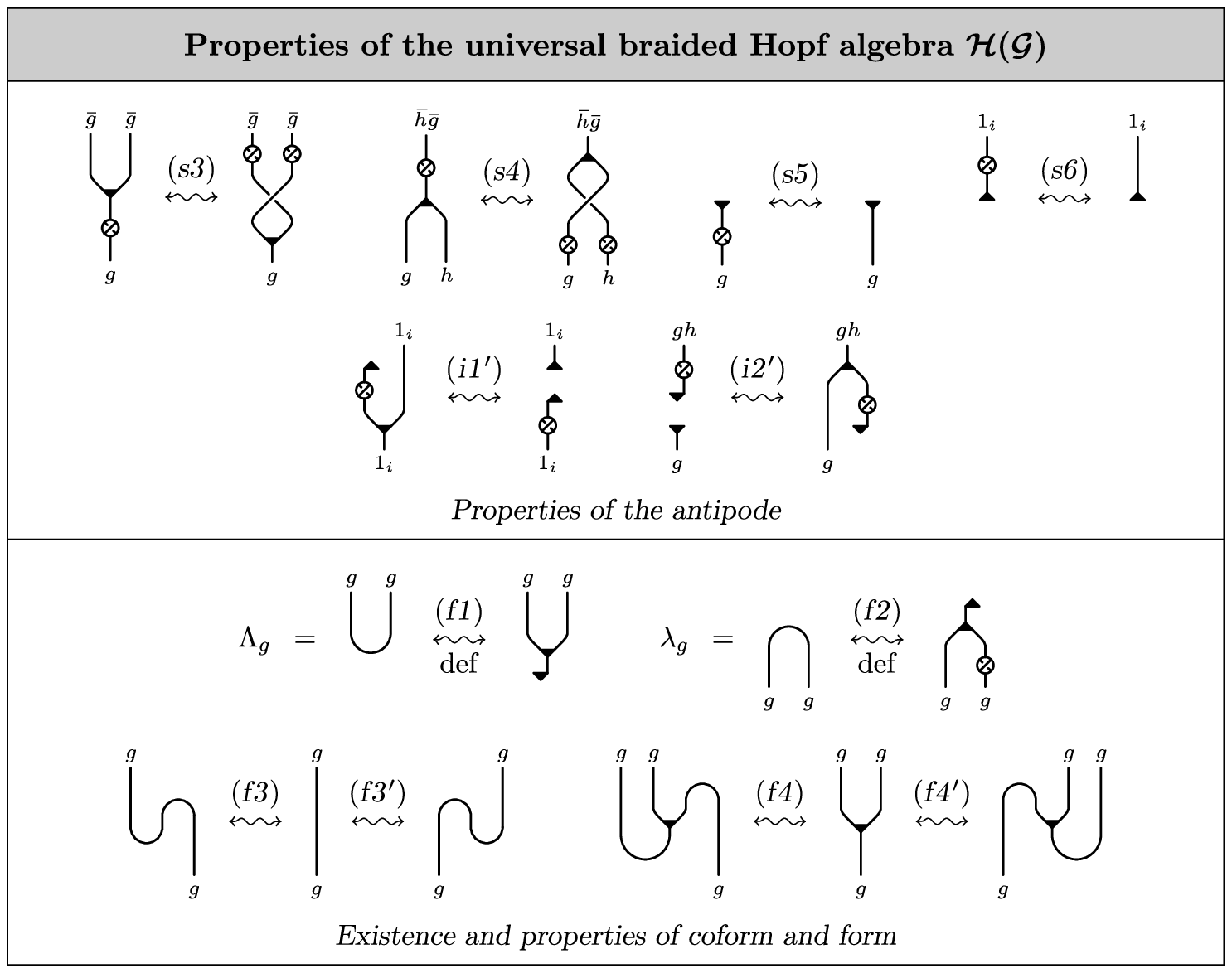}}
\vskip-3pt
\end{Table}

\begin{proposition}\label{antipode/thm}
The following properties of the antipode in $\H(\G)$, hold for any $g,h \in \G$ such that
$gh$ is defined, and for any $i \in \G$:
\vskip-4pt
$$\Delta_{\bar g} \circ S_g = (S_g \diam S_g) \circ \gamma_{g,g} \circ \Delta_g: 
H_g \to H_{\bar g} \diam H_{\bar g},
\eqno{\(s3)}$$
$$S_{gh} \circ m_{g,h} = m_{\bar h,\bar g} \circ (S_h \diam S_g) \circ \gamma_{g,h}:
H_g \diam H_h \to H_{\bar h \bar g},
\eqno{\(s4)}$$
$$\epsilon_{\bar g} \circ S_g = \epsilon_g,
\eqno{\(s5)}$$
$$S_{1_i} \circ \eta_i = \eta_i.
\eqno{\(s6)}$$
\vskip4pt\noindent
Moreover, if $l = \{l_j: H_{1_j} \to \one \}_{j \in \Obj \G}$ and $L = \{L_g: \one \to
H_g\}_{g \in \G}$ are respectively a left cointegral and a right integral of $H$, then:
\vskip-4pt
$$
l \circ S = \{l_j \circ S_{1_j}: H_{1_j} \to \one \}_{j \in \Obj \G}
\text{ is a right cointegral of $H$,} \eqno{\(i1')}
$$
$$
S \circ L = \{S_g \circ L_g: \one \to H_{\bar g}\}_{g \in \G}
\text{ is a left integral of $H$.} \eqno{\(i2')}
$$
\end{proposition}

\begin{proof}
\(s3) is proved in Figure \ref{h-antipode01/fig}. In the first step we obtain the diagram
on the left from the one on the right by applying in the order moves \(a1-3), move \(s1)
and \(a2-4'). To be precise, before applying \(s1) and \(a4'), we also use the braid
axioms presented in Table \ref{table-Hdefn/fig}, but this in general will not be
indicated.

\begin{Figure}[htb]{h-antipode01/fig}
{}{Proof of \(s3) 
   [{\sl a-s}/\pageref{table-Hdefn/fig}]}
\centerline{\fig{}{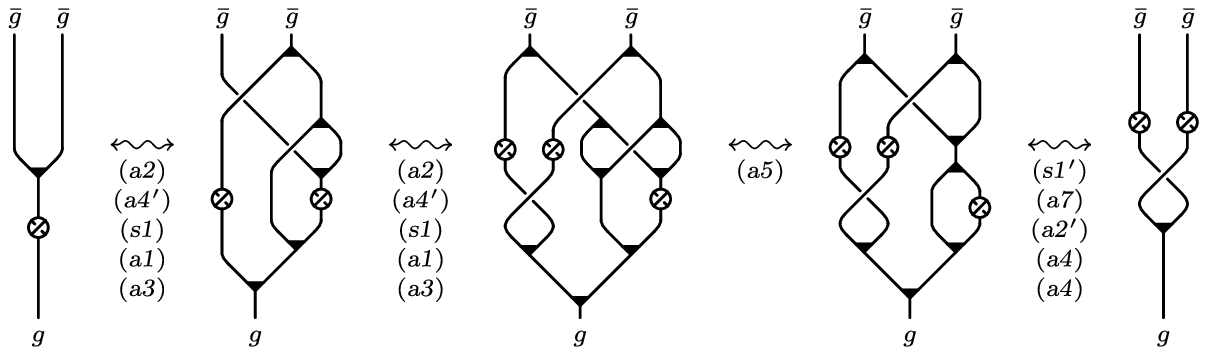}}
\vskip-3pt
\end{Figure}

\begin{Figure}[htb]{h-antipode02/fig}
{}{Proof of \(s5) and \(i1')
   [{\sl a-i}/\pageref{table-Hdefn/fig}, 
    {\sl s}/\pageref{table-Hdefn/fig}-\pageref{table-Hprop/fig}]}
\vskip-12pt
\centerline{\fig{}{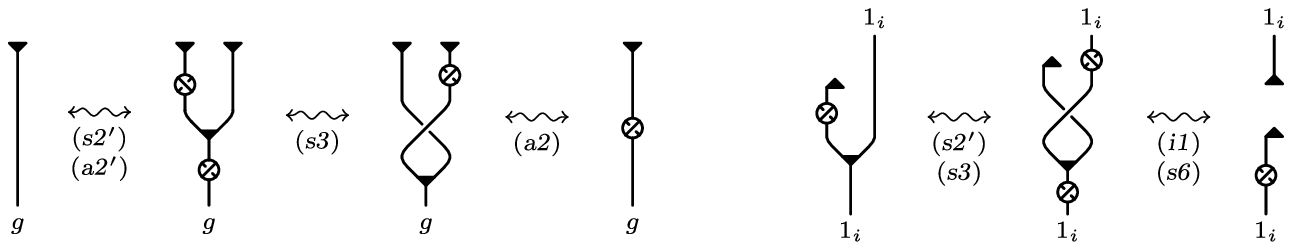}}
\vskip-3pt
\end{Figure}

Property \(s5) is proved in the left side of Figure \ref{h-antipode02/fig}. Then \(s4)
and \(s6) are obtained by rotating the diagrams in Figures \ref{h-antipode01/fig} and
\ref{h-antipode02/fig} upside down. Eventually, using \(s3) and \(s6) one obtains
\(i1') as shown in the right side of Figure \ref{h-antipode02/fig} and \(i1) by
rotating that figure.
\end{proof}

The second set of properties in Table \ref{table-Hprop/fig} states the existence of an
autonomous structure on $\H(\G)$ and describes the relation between such structure and the
algebraic one. The proposition below proves properties \(f1), \(f2) and \(f3-3'), 
extending the result in Lemma 7 of \cite{Ke02} to possibly non-unimodular categories.

Note that in the diagrams representing the morphisms of $\H(\G)$, it is appropriate to use
for the coform and the form the notations \(f1) and \(f2), presented in Table
\ref{table-Hprop/fig}. In fact, the relations \(f3-3') reduce to the standard ``pulling
the string'', which together with the braid axioms in Table \ref{table-Hdefn/fig} realize
regular isotopy of strings.

\begin{proposition}\label{autonomous/thm}
Given a groupoid $\G$, the universal Hopf $\G$-algebra $\H(\G)$ is an autonomous category,
with $H_\pi^{\,\ast} = H_{\pi^\ast}$ for every $\pi \in \seq\G$, where $\pi^\ast$ is the
sequence obtained by re\-versing the order of $\sigma$. In particular $H_g^{\,\ast} = H_g$
for $g \in \G$, while coform and form are defined by
\vskip-8pt
$$\Lambda_{H_g} \!= \Lambda_g = \Delta_g \circ L_g \,,
\eqno{\(f1)}$$
$$\lambda_{H_g} \!= \lambda_g = l_{g \bar g} \circ m_{g, \bar g} \circ (\id_g \diam S_g) 
\,,
\eqno{\(f2)}$$
\vskip4pt\noindent
for any $g \in \G$, and by the following recursive formulas for $\pi = \pi' \diam \pi''
\in \seq\G$ (note\break that the definition is well-posed, giving equivalent results for
different decompositions $\pi = \pi' \diam \pi''$)
\vskip-8pt
$$\Lambda_{H_\pi} \!= \Lambda_\pi = (\id_{(\pi'')^\ast} \diam \Lambda_{\pi'} \diam
\id_{\pi''}) \circ \Lambda_{\pi''}\,,$$
$$\lambda_{H_\pi} \!= \lambda_\pi = \lambda_{\pi'} \circ (\id_{\pi'} \diam \lambda_{\pi''}
\diam \id_{(\pi')^\ast})\,.$$
\vskip4pt\noindent
Hence, properties \(f3-3') in Table \ref{table-Hprop/fig} hold in $\H(\G)$.
\end{proposition}

\begin{proof}
Figure \ref{h-form01/fig} shows that $\Lambda_g$ and $\lambda_g$ satisfy the relations in
the Definition \ref{autonomous/def}, that is the properties \(f3-3'). Then, it suffices to
observe that such relations propagate to $\Lambda_\pi$ and $\lambda_\pi$ for any $\pi \in 
\G$, by a simple induction on the length of $\pi$.
\end{proof}

\begin{Figure}[htb]{h-form01/fig}
{}{Proof of \(f3-3') 
   [{\sl a}/\pageref{table-Hdefn/fig}, 
    {\sl i-s}/\pageref{table-Hdefn/fig}-\pageref{table-Hprop/fig}]}
\vskip-3pt
\centerline{\fig{}{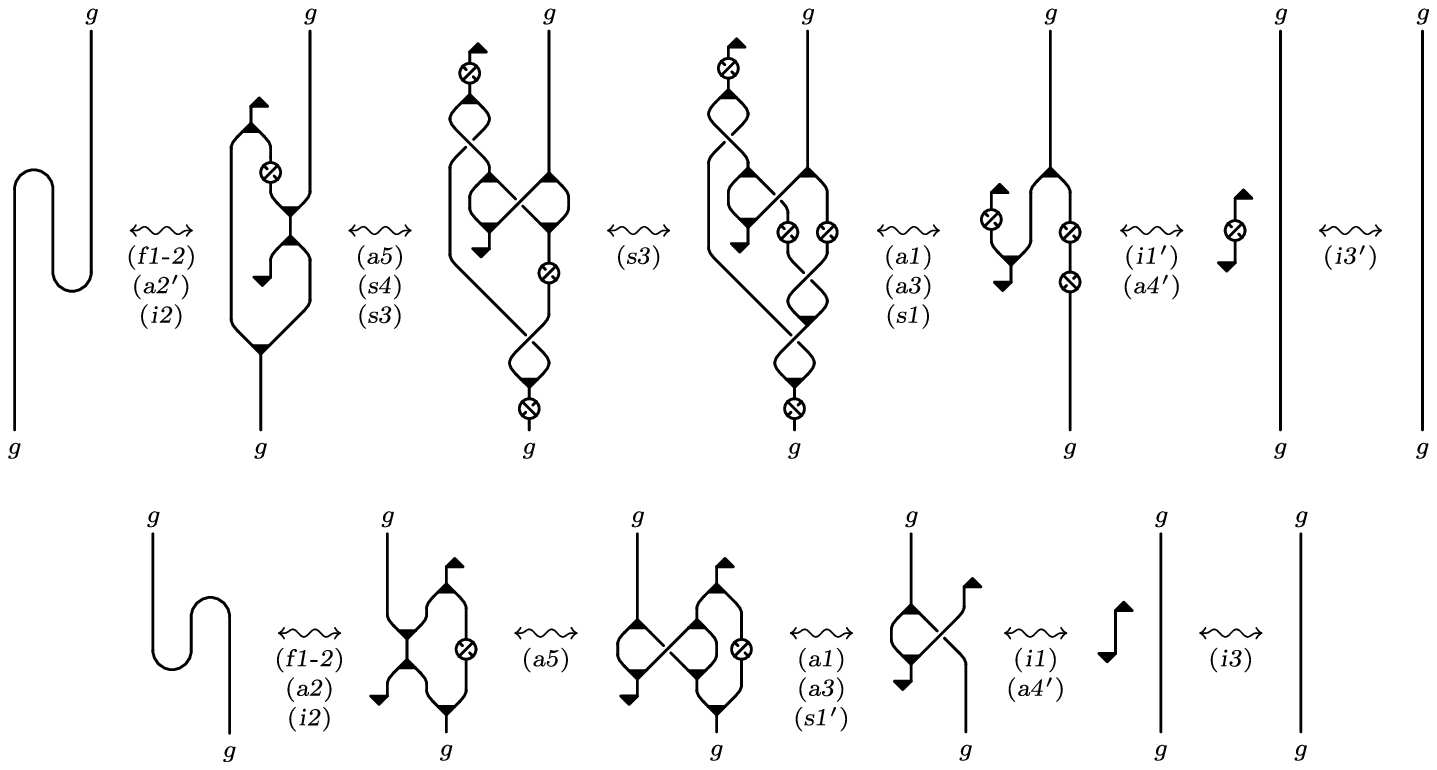}}
\vskip-3pt
\end{Figure}

The remaining two properties \(f4-4') in Table \ref{table-Hprop/fig} can be expressed in 
terms of the right and left ``rotation'' maps:
\begin{eqnarray*}
&\rot_r: \Mor_{\H(\G)}(H_{h_0 \diam \pi_0},H_{\pi_1 \diam h_1}) 
 \to \Mor_{\H(\G)}(H_{\pi_0 \diam h_1}, H_{h_0 \diam \pi_1}),&\\[2pt]
&\rot_l: \Mor_{\H(\G)}(H_{\pi_0 \diam h_0}, H_{h_1 \diam \pi_1}) 
 \to \Mor_{\H(\G)}(H_{h_1 \diam \pi_0}, H_{\pi_1 \diam h_0}),
\end{eqnarray*}
defined for any $h_0,h_1 \in \G$ and $\pi_0,\pi_1 \in \seq\G$ by the identities (see 
Figure \ref{h-form02/fig}):
\begin{eqnarray*}
&\rot_r(F) = (\id_{h_0 \diam \pi_1} \!\diam \lambda_{h_1}) \circ (\id_{h_0} \!\diam F 
 \diam \id_{h_1}) \circ (\Lambda_{h_0} \!\diam \id_{\pi_0 \diam h_1}),&\\[2pt]
&\rot_l(F) = (\lambda_{h_1} \!\diam \id_{\pi_1 \diam h_0}) \circ (\id_{h_1} \!\diam F
 \diam \id_{h_0}) \circ (\id_{h_1 \diam \pi_0} \!\diam \Lambda_{h_0}).&
\end{eqnarray*}

\begin{Figure}[htb]{h-form02/fig}
{}{The rotation maps $\rot_r$ and $\rot_l$}
\centerline{\fig{}{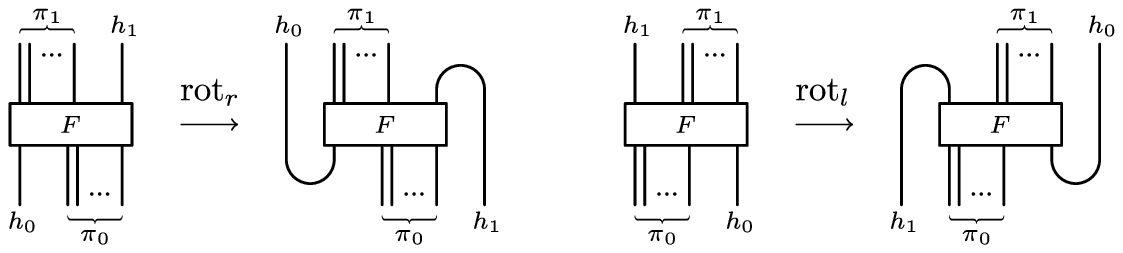}}
\vskip-3pt
\end{Figure}

Proposition \ref{autonomous/thm} implies that the maps $\rot_r$ and $\rot_l$ are inverses
to each other.\break Then properties \(f4-4') in Table \ref{table-Hprop/fig} state that
the comultiplication vertices are invariant under such rotation maps. We will see later
that the same is not true for the multiplication vertices. The reason for which we study
the action of the rotation maps to these types of vertices will become clear in Section
\ref{Psi/sec}, where the functors $\Psi_n$ will relate them to moves \(I5-6) and \(R1) in
Figures \ref{ribbon-moves01/fig} and \ref{ribbon-tang01/fig} in the category of labeled
ribbon surface tangles.

\begin{proposition}\label{rot-delta/thm}
In $\H(\G)$, for any $g \in \G$ we have
\vskip-6pt
$$\rot_r(\Delta_g) = \rot_l(\Delta_g) = \Delta_g: H_g \to H_g \diam H_g,
\eqno{\(f4-4')}$$
\vskip-12pt
\end{proposition}

\begin{proof}
\(f4) is shown in  Figure \ref{h-form03/fig}. \(f4') is analogous and left to 
the reader.
\end{proof}

\begin{Figure}[htb]{h-form03/fig}
{}{Proof of \(f4) 
   [{\sl a}/\pageref{table-Hdefn/fig}, {\sl f}/\pageref{table-Hprop/fig}]}
\vskip-9pt
\centerline{\fig{}{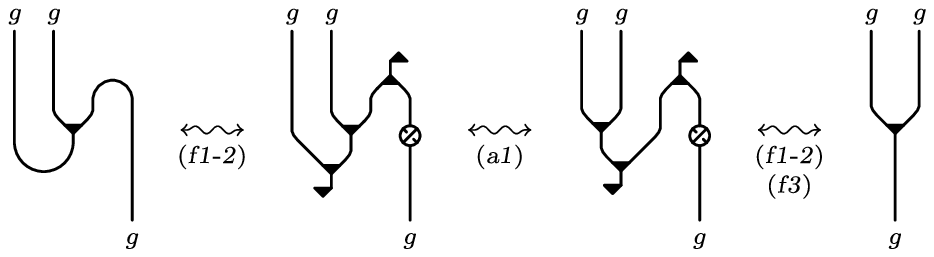}}
\vskip-3pt
\end{Figure}

\begin{definition}\label{unimodular}
A Hopf $\G$-algebra in a braided monoidal category is called {\sl unimodular}
if it has $S$-invariant (2-sided) integral and cointegral, that is:
\vskip-4pt
$$S_g \circ L_g = L_{\bar g},\eqno{\(i4)}$$ 
$$l_{\bar g} \circ S_g = l_g.\eqno{\(i5)}$$
\vskip4pt\noindent
Then, given a groupoid $\G$, the {\sl universal unimodular Hopf $\G$-algebra} $\H^u(\G)$
is the quotient of $\H(\G)$ modulo the relations \(i4) and \(i5) above.
\end{definition}

The diagrammatic presentation of the relation \(i4) and \(i5) above can be found in Table
\ref{table-Hu/fig}. Moreover, as it is indicated there, in $\H^u(\G)$ we change the
notation for the integral and cointegral vertices by connecting the edge to the middle
point of the base of the triangle, to reflect that the corresponding integral is
two-sided.\break In the same Table \ref{table-Hu/fig}, we also rewrite the integral axioms
\(i1-3) (notice that here \(i3{$_{\text{S}}$}) is redundant) and the properties \(i1') and
\(i2'), as well as the definitions \(f1-2) of form and coform, according to the new
notation.

\begin{Table}[p]{table-Hu/fig}
{}{}
\centerline{\fig{}{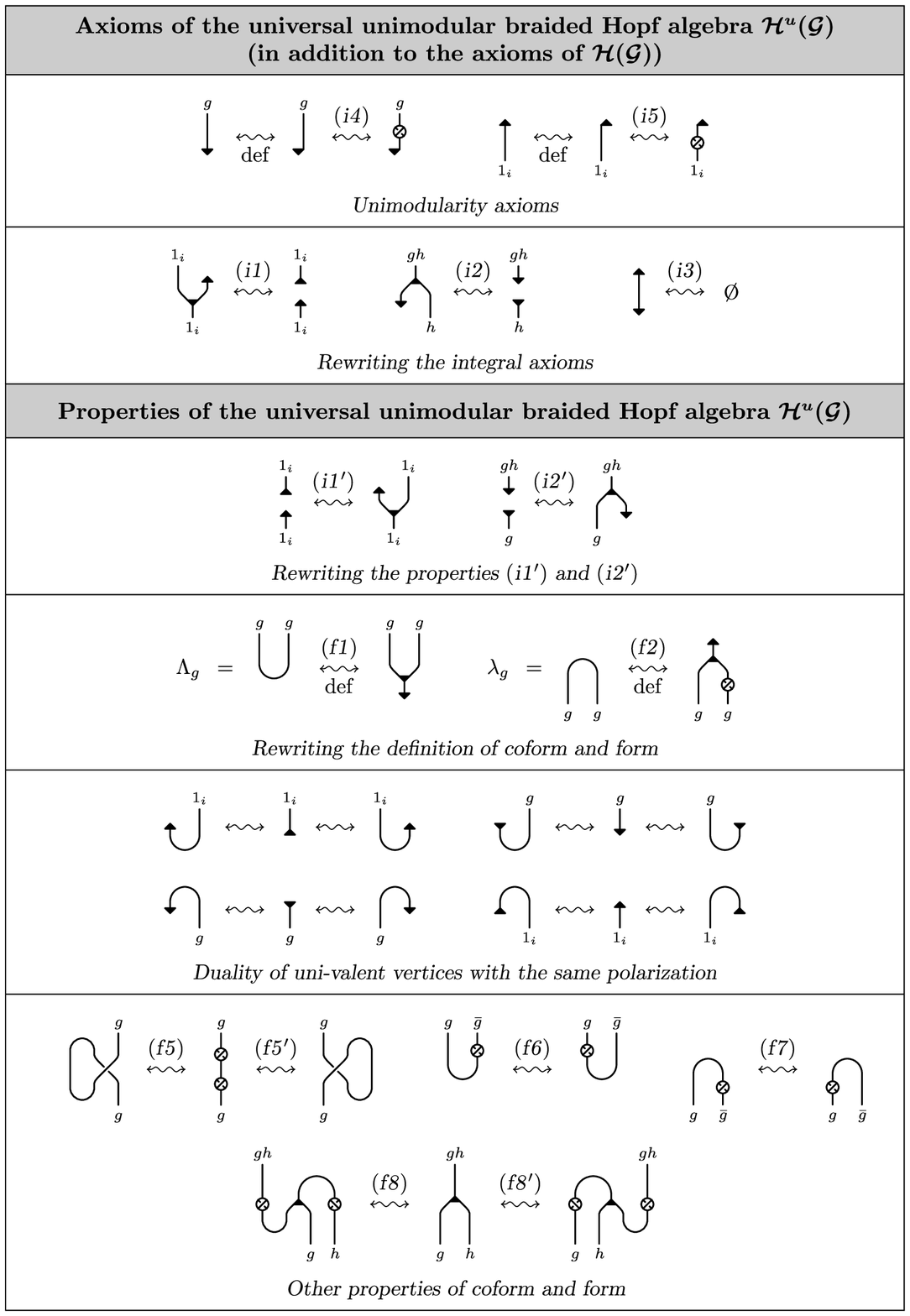}}
\vskip-3pt
\end{Table}

Adding the unimodularity condition to the algebra axioms brings to the symmetry of the 
structure with respect to the rotation around a vertical axis, which inverts the labels in 
$\G$ and reverse the products of objects and morphisms. In fact, axioms \(i4-5) make 
properties \(i1'-2') symmetric to axioms \(i1-2), completing the preexisting symmetry of 
the other axioms. This is the content of next proposition.

\pagebreak

\begin{proposition}\label{symmetry/thm}
Given a groupoid $\G$, there is an involutive antimonoidal category equivalence $\sym:
\H^u(\G) \to \H^u(\G)$, uniquely determined by the following identities, for any $g,h \in 
\G$ (composable for $m_{g,h}$), any $H_\pi,H_{\pi'} \in \Obj \H^u(\G)$ and any $F,F' \in 
\Mor \H^u(\G)$ :
\begin{eqnarray*}
&\sym (H_g) = H_{\bar g}\,,
  \,\sym(\gamma_{g,h}) = \gamma_{\bar h, \bar g}\,,
  \,\sym(\bar\gamma_{g,h}) = \bar\gamma_{\bar h, \bar g}\,,&\\[2pt]
&\sym(\Delta_g) = \Delta_{\bar g}\,,
  \,\sym(\epsilon_g) = \epsilon_{\bar g}\,,
  \,\sym(m_{g,h}) = m_{\bar h,\bar g}\,,
  \,\sym(\eta_i) = \eta_i\,,&\\[2pt]
&\sym(S_g) = S_{\bar g}\,,
  \,\sym(\bar S_g) = \bar S_{\bar g}\,,
  \,\sym(l_i) = l_i\,,
  \,\sym(L_g) = L_{\bar g}\,,&\\[2pt]
&\sym(H_\pi \diam H_{\pi'}) = \sym(H_{\pi'}) \diam \sym (H_\pi)
  \ \ \text{and} \ \
  \sym(F \diam F') = \sym(F') \diam \sym (F)\,.&
\end{eqnarray*}
\vskip-12pt
\end{proposition}

\begin{proof}
The universal property of $\H^u(\G)$ allows us to define the wanted functor by propagating
the identities in the statement over compositions, once we show that this preserves the
axioms of $\H^u(\G)$. This is indeed the case, being all the axioms invariant or
interchanged with their primed versions (possibly up to inversion). That such functor is
an involution and hence a category equivalence, follows from its involutive action on the
elementary morphisms and products.
\end{proof}

Some further properties of $\H^u(\G)$ are listed in Table \ref{table-Hu/fig}. First of
all, using the integral axioms \(i1) to \(i5) and the bialgebra axiom \(a8) in Table
\ref{table-Hdefn/fig}, it is easy to see that the uni-valent vertices of the same
polarization are dual to each other with respect to the form/coform. Moreover,
unimodularity leads to the existence of a tortile structure (cf. Definition
\ref{tortile/def}) with some additional properties, as stated by the next proposition,
which is a version of Lemma 8 in \cite{Ke02}.

\begin{Figure}[b]{h-form04/fig}
{}{Proof of \(f5) 
   [{\sl f-i}/\pageref{table-Hu/fig}, {\sl s}/\pageref{table-Hprop/fig}]}
\centerline{\fig{}{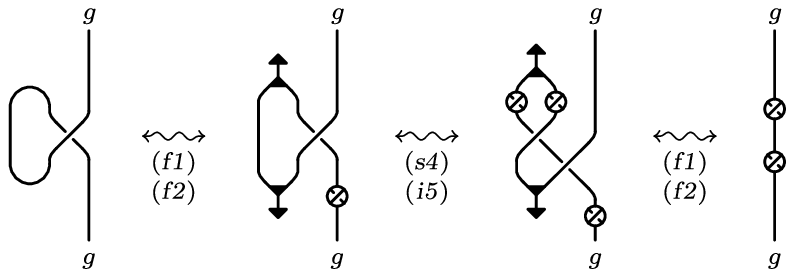}}
\vskip-3pt
\end{Figure}

\begin{proposition} \label{H-tortile/thm}  
Given a groupoid $\G$, the universal unimodular Hopf $\G$-algebra $\H^u(\G)$ is a tortile
category, with the twist $\theta_{H_\pi}\!$ defined for any $\pi \in \seq\G$\nobreak\ by
$$\theta_{H_\pi} = \theta_\pi = (\lambda_{\pi^\ast} \!\diam \id_\pi) \circ (\id_{\pi^\ast} 
\!\diam \gamma_{\pi,\pi}) \circ (\Lambda_\pi \diam \id_\pi)\,.$$
Moreover, the following properties (cf. Table \ref{table-Hu/fig}) hold for any $g\in\G$:
\vskip-4pt
$$\theta_g = S_{\bar g} \circ S_g= \theta_g^{\,\ast}, 
\eqno{\(f5-5')}$$
$$(\id_g \diam S_g) \circ \Lambda_{\bar g} = (S_{\bar g} \diam \id_g) \circ \Lambda_g\,,
\eqno{\(f6)}$$
$$\lambda_g \circ(\id_g \diam S_{\bar g}) = \lambda_{\bar g} \circ (S_g \diam \id_{\bar 
g}) \,.
\eqno{\(f7)}$$
\vskip-12pt
\end{proposition}

\begin{proof}
Observe that the definition of $\theta_\pi$ guarantees that $\theta_\one = \id_\one$ and
that the identity $\theta_{\pi \diam \pi'} = \gamma_{\pi',\pi} \circ (\theta_{\pi'} \diam
\theta_\pi) \circ \gamma_{\pi,\pi'}$ holds, up to isotopy moves, for any $\pi,\pi' \in
\seq\G$. Therefore, in order to see that $\theta$ makes $\H^u(\G)$ into a tortile
category, it is enough to show its naturality and that $\theta_{\pi^\ast} =
\theta_\pi^{\,\ast}$ for any object $H_\pi$ in $\H^u(\G)$.

\begin{Figure}[htb]{h-form05/fig}
{}{Proof of \(f6)
   [{\sl f-i}/\pageref{table-Hu/fig},
    {\sl s}/\pageref{table-Hdefn/fig}-\pageref{table-Hprop/fig}]}
\centerline{\fig{}{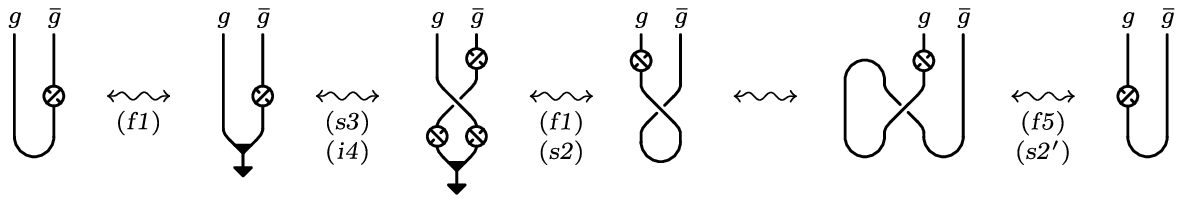}}
\vskip-3pt
\end{Figure}

We will first prove the last identity. Through an induction argument, one can easily see
that the general case follows from \(f5-5'). Figure \ref{h-form04/fig} shows that
$\theta_g$, represented by the diagram on the left side of \(f5), is equivalent to
$S_{\bar g} \circ S_g$, which proves \(f5). Then also $\theta_g^{\,\ast}$, represented up
to isotopy moves by the diagram on the right side of \(f5'), is equivalent to $S_{\bar g}
\circ S_g$, being properties \(f5) and \(f5') symmetric to each other under the category
equivalence in Proposition \ref{symmetry/thm}. Moreover, in Figure \ref{h-form05/fig} we
see that \(f5) implies \(f6), while \(f7) immediately follows from \(f6) and the
properties \(f3-3') of the form and coform presented in Table \ref{table-Hprop/fig}.

It is left to prove the naturality of $\theta$, i.e. that $\theta_{\pi_1} \circ F = F
\circ \theta_{\pi_0}$ for any morphism $F: H_{\pi_0} \to H_{\pi_1}$ in $\H^u(\G)$. Since
any morphism in $\H^u(\G)$ is a composition of expansions (i.e. products with identities)
of elementary morphisms, by using the identity $\theta_{\pi \diam \pi'} =
\gamma_{\pi',\pi} \circ (\theta_{\pi'} \diam \theta_\pi) \circ \gamma_{\pi,\pi'}$ and the
isotopy moves, we can reduce ourselves to the case when $F$ is any elementary morphism.
This case easily follows from \(f5) and the properties \(s3-6) in Table
\ref{table-Hprop/fig} and \(i4-5) in Table \ref{table-Hu/fig}
\end{proof}

As we said above, under the right and left rotation maps the multiplication vertices do
not remain invariant, but change as described by properties \(f8-8') in Figure
\ref{table-Hu/fig}, which are proved in the next proposition. We observe that \(f8) still 
holds in the non-unimodular case (in $\H(\G)$), while \(f8') needs the unimodularity 
axioms.

\begin{proposition}\label{rot-m/thm}
In $\H^u(\G)$, if $g,h,gh\in\G$ ($g$ and $h$ are composable) then
\vskip-4pt
$$\rot_r(m_{g,h}) = \bar S_{\bar g} \circ m_{h,\bar h \bar g} \circ (\id_h \diam
S_{gh}) :  H_h \diam H_{gh} \to H_g\,,
\eqno{\(f8)}$$
$$\rot_l(m_{g,h}) = \bar S_{\bar h} \circ m_{\bar h \bar g,g} \circ (S_{gh}  \diam
\id_g) :  H_h \diam H_{gh} \to H_g\,.
\eqno{\(f8')}$$
\end{proposition}

\begin{proof}
See Figure \ref{h-form06/fig} for property \(f8). Then property \(f8') follows by 
symmetry, according to Proposition \ref{symmetry/thm}.
\end{proof}

\begin{Figure}[htb]{h-form06/fig}
{}{Proof of \(f8) 
   [{\sl a}/\pageref{table-Hdefn/fig}, {\sl f}/\pageref{table-Hu/fig}, 
    {\sl s}/\pageref{table-Hdefn/fig}]}
\vskip-6pt
\centerline{\fig{}{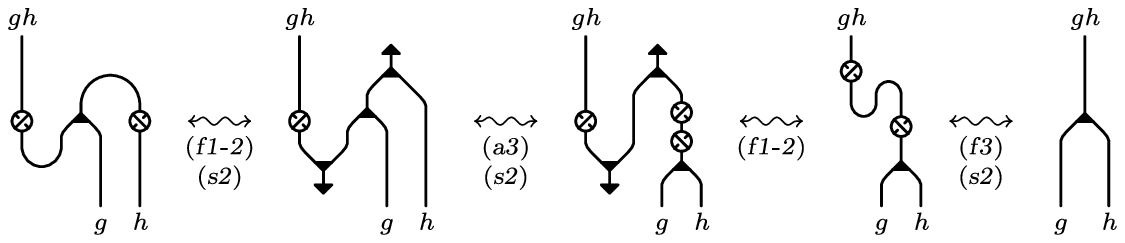}}
\vskip-3pt
\end{Figure}

\subsection{The universal groupoid ribbon Hopf algebra $\H^r(\G)$%
\label{HrG/sec}}

In this section we want to provide the universal unimodular Hopf $\G$-algebra $\H^u(\G)$
with a ribbon structure. The first step in this direction is to postulate the existence of
a family of ribbon morphisms, in the sense of the following definition.

\begin{definition}\label{ribbon-morphisms/def}
Given a unimodular Hopf $\G$-algebra $H$ in a braided monoid\-al category $\C$, a {\sl
family of ribbon morphisms} of $H$ is a set $v =\{v_g: H_g\to H_g\}_{g \in \G}$ of
invertible morphisms in $\C$, such that for any $g,h,gh \in \G$:
\vskip-6pt 
$$S_g \circ v_g = v_{\bar g} \circ S_g\,,
\eqno{\(r3)}$$
$$\epsilon_g \circ v_g = \epsilon_g\,,
\eqno{\(r4)}$$
$$m_{g,h} \circ (v_g \diam \id_h) = v_{gh} \circ m_{g,h}\,,
\eqno{\(r5)}$$
\vskip4pt\noindent
and the family of morphisms $\sigma = \{\sigma_{i,j}: \one \to H_{1_i} \diam 
H_{1_j}\}_{i,j \in \Obj \G}$, defined by
$$\arraycolsep0pt
\sigma_{i,j} = \left\{
\begin{array}{ll} 
 (v_{1_i}^{-1} \diam (v_{1_i}^{-1} \circ S_{1_i})) \circ \Delta_{1_i} \circ v_{1_i}
 \circ \eta_i &\quad \text{if } i = j,\\[2pt]
 \eta_i \diam \eta_j &\quad \text{if } i\neq j,
\end{array}\right.
\eqno{\(r6)}$$
\vskip8pt\noindent
satisfies the following identity for any $i \in \Obj \G$:
\vskip-6pt
$$(\Delta_{1_i} \diam \id_{1_i}) \circ \sigma_{i,i} = (\id_{1_i} \diam \id_{1_i} \diam 
m_{1_i,1_i}) \circ (\id_{1_i} \diam \sigma_{i,i} \diam \id_{1_i}) \circ \sigma_{i,i}.
\eqno{\(r7)}$$
The axioms above are presented in Table \ref{table-Huvdefn/fig}, where we use thinner
lines to draw the optional part in the diagram on the righthand side of \(r6). Observe
that when $i = j$, the thicker units can be deleted by axioms \(a4-4') in Table
\ref{table-Hdefn/fig}, and therefore $\sigma_{i,i}$ is given just by the thinner diagram.
\end{definition}

\begin{Table}[htb]{table-Huvdefn/fig}
{}{}
\centerline{\fig{}{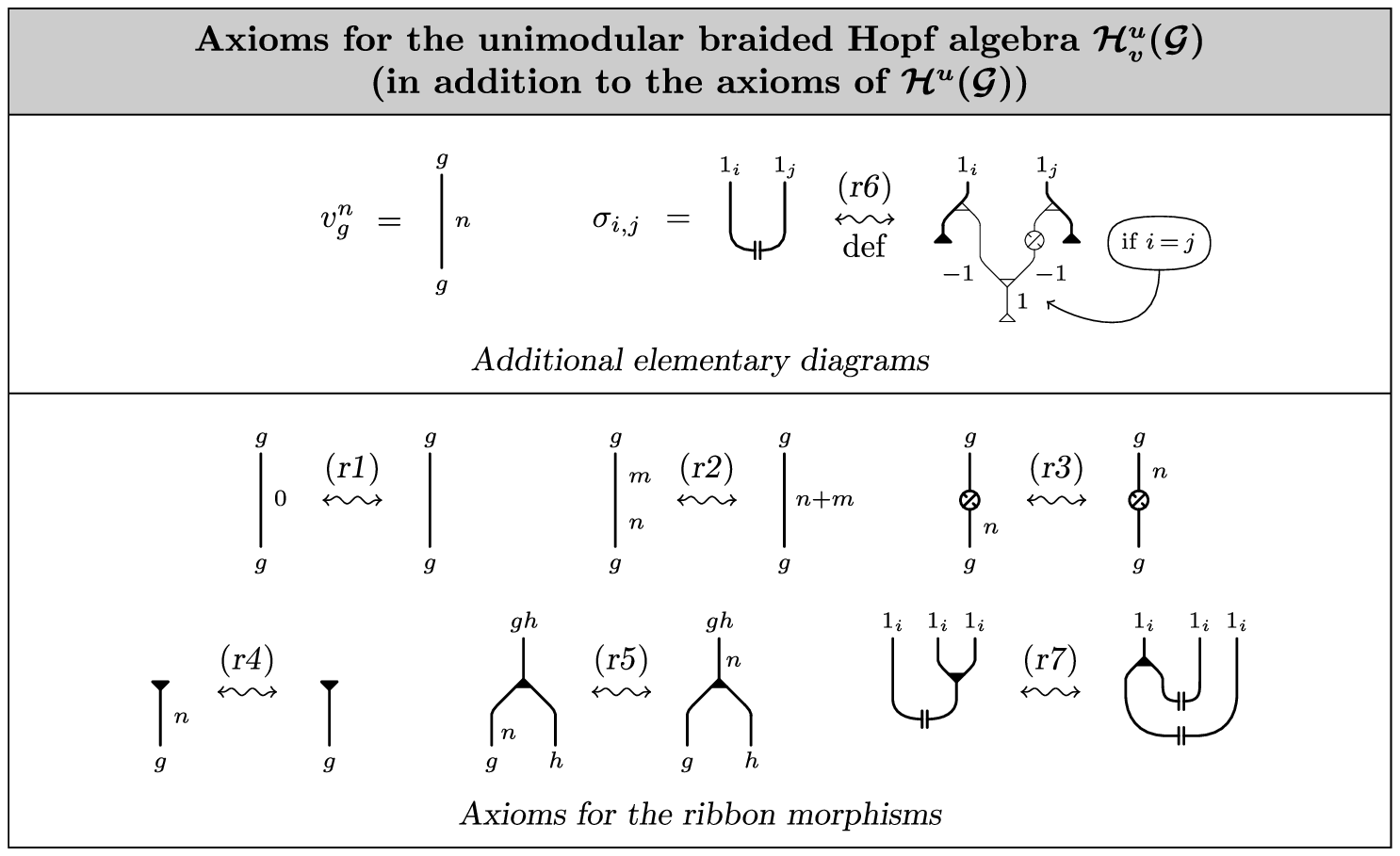}}
\vskip-3pt
\end{Table}

Usually in the literature (see Section 4 in \cite{Ke02} for the case of a trivial
groupoid), instead of ribbon morphisms are introduced ribbon elements, as the ones defined
below in the general case of a groupoid Hopf algebra.

\begin{definition}\label{ribbon-elements/def}
Given a unimodular Hopf $\G$-algebra $H$ in a braided monoid\-al category $\C$, a family
of {\sl ribbon elements} of $H$ consists of two sets $V =\{V_i: \one \to H_{1_i}\}_{i \in
\Obj\G}$ and $\bar V =\{\bar V_i: \one\to H_{1_i}\}_{i \in \Obj\G}$ of morphisms in $\C$,
such that for any $i \in \Obj \G$ and $g \in \G(i,j)$:
\vskip-12pt
$$S_{1_i} \circ V_i = V_i
\ \ \text{and} \ \ 
\epsilon_{1_i} \circ V_i = \one\,,$$
$$m_{1_i,1_i} \circ (V_i \diam \bar V_i) = \eta_i
\ \ \text{and} \ \
m_{1_i,g} \circ (V_i \diam \id_g) = m_{g,1_j} \circ (\id_g \diam V_j)\,,$$
\vskip4pt\noindent
and the family of morphisms $\sigma = \{\sigma_{i,j}: \one \to H_{1_i} \diam 
H_{1_j}\}_{i,j \in \Obj \G}$, defined by
$$\arraycolsep0pt
\sigma_{i,j} = \left\{
\begin{array}{ll} 
(m_{1_i,1_i}\diamond m_{1_i,1_i})\circ (\bar V_i \diam \id_{1_i}\diam S_{1_i} \diam \bar 
V_i) \circ \Delta_{1_i} \circ V_i &\quad \text{if } i = j,\\[2pt]
\eta_i \diam \eta_j &\quad \text{if } i\neq j,
\end{array}\right.
$$
\vskip4pt\noindent
satisfies the following identity:
\vskip-6pt
$$(\Delta_{1_i} \diam \id_{1_i}) \circ \sigma_{i,i} = (\id_{1_i} \diam \id_{1_i} \diam 
m_{1_i,1_i}) \circ
(\id_{1_i} \diam \sigma_{i,i} \diam \id_{1_i}) \circ \sigma_{i,i}.
$$\vskip-12pt
\end{definition}

Actually, next proposition states the equivalence of the two approaches. Moreover, the
axioms for ribbon elements can be considered conceptually simpler, in sense that when $H$
is the braiding of an ordinary Hopf algebra $\bar H$ the ribbon element is a special
central element in $\bar H$ (cf. \cite{Vi01}). Nevertheless, as we will discuss below, the
approach base on ribbon morphisms seems to be preferable in the present context.

\begin{proposition}\label{ribbon-elements/thm}
For any unimodular Hopf $\G$-algebra $H$ in a braided mon\-oidal category $\C$, there is a
bijective correspondence between the set of families of ribbon morphisms and the set of
families of ribbon elements of $H$, given by the map $v \mapsto \{V,\bar V\}$ defined by
$V_i = v_{1_i} \circ \eta_i$ and $\bar V_i = v_{1_i}^{-1} \circ \eta_i$ for every $i \in
\Obj \G$, whose inverse $\{V,\bar V\} \mapsto v$ is defined by $v_g = m_{1_i,g} \circ (V_i
\diam \id_g)$ for every $g \in \G(i,j)$.
\end{proposition}

\begin{proof}
The only non-trivial point is to prove the identity $m_{1_i,g} \circ (V_i \diam \id_g) =
m_{g,1_j} \circ (\id_g \diam V_j)$, when $V_i = v_{1_i} \circ \eta_i$, $V_j = v_{1_j}
\circ \eta_j$ and $g \in \G(i,j)$. This requires the property \(r5'), which can be
obtained from \(r5) by using \(s4) (see below). The rest is straightforward and it is left
to the reader.
\end{proof}

As we will see in the next section, the image of the ribbon morphisms/elements in the
category $\K_n$ are twists in the attaching map of the corresponding 2-handles. In
particular, using ribbon elements, $k$ full twists along the framing of a 2-handle are
represented by $m_{1_i} \circ (\id_{1_i} \diam m_{1_i}) \circ \dots \circ (\id_{1_i} \diam
\dots \diam \id_{1_i} \diam m_{1_i}) \circ V_i^{\diam k}$, and the corre\-sponding graph
diagram is quite heavy. On the other hand, using ribbon morphisms, the same $k$ full
twists are represented by $v_{1_i}^k \circ \eta_i$, which seems to us much simpler. This
is the main reason why we decided to follow this last approach.

\begin{definition}\label{pre-ribbon/def}
Given a groupoid $\G$, we define the {\sl universal pre-ribbon Hopf $\G$-algebra}
$\H^u_v(\G)$ to be the braided strict monoidal category freely generated by a unimodular
Hopf $\G$-algebra $H$ with a family of ribbon morphisms $v = \{v_g: H_g \to H_g\}_{g \in
\G}$. Then, $\H^u_v(\G)$ has the same objects as $\H^u(\G)$, while its elementary
morphisms are those of $\H^u(\G)$ with the addition of the diagrammatic representations of
$v^n$ and $\sigma_{i,j}$ shown in Table \ref{table-Huvdefn/fig}. Moreover, the defining
relations of $\H_v^u(\G)$ are the axioms of $\H^u(\G)$ (where now in the braid axioms in
Table \ref{table-Hdefn/fig}, $D$ can be also a ribbon morphism) plus the axioms for the
ribbon morphisms presented in Table \ref{table-Huvdefn/fig}.
\end{definition}

\begin{Table}[b]{table-Huvprop/fig}
{}{}
\vskip12pt
\centerline{\fig{}{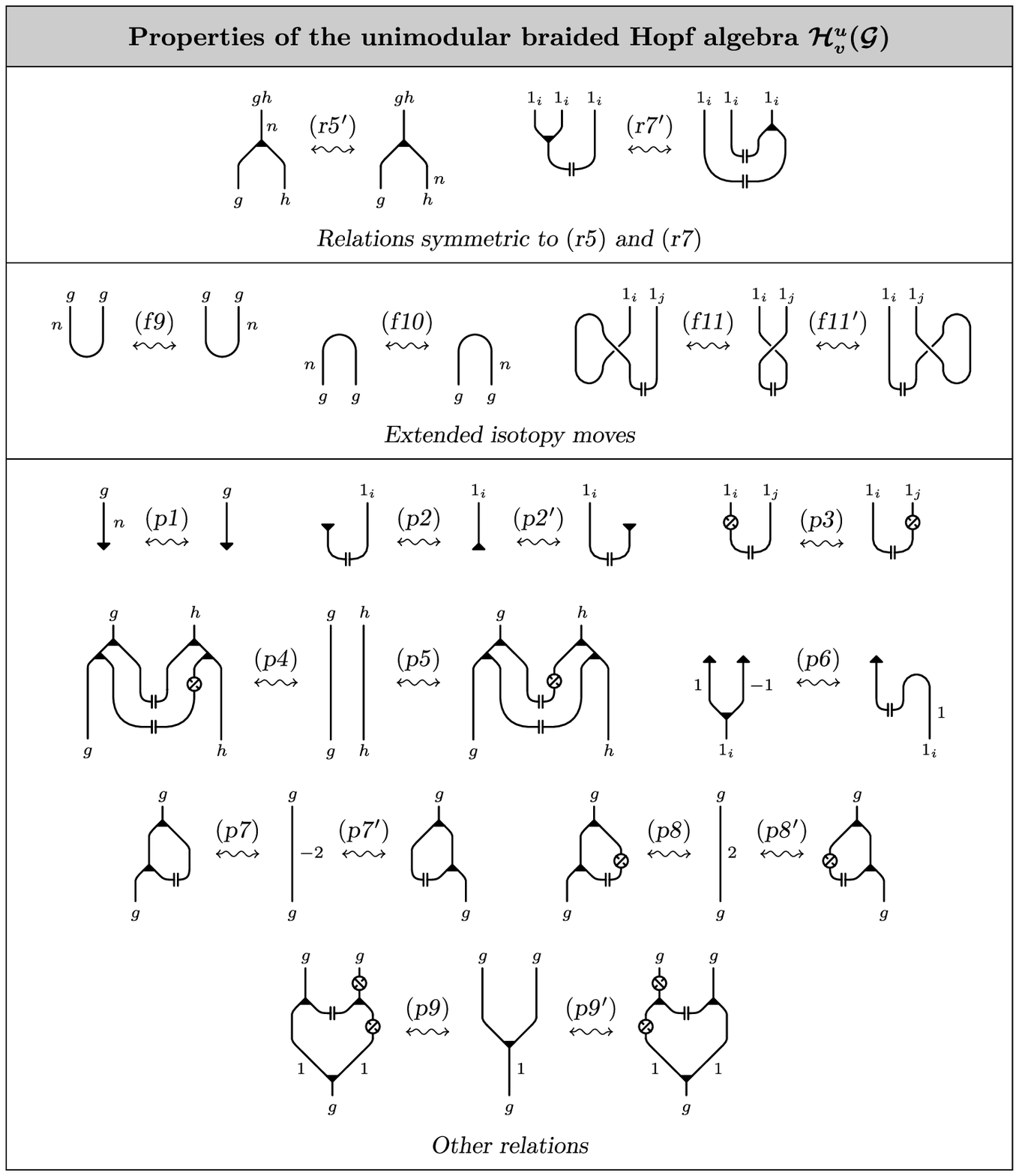}}
\vskip-3pt
\end{Table}

The category $\H^u_v(\G)$ has many important properties presented in Table
\ref{table-Huvprop/fig}. Before proving them, we make here few observations about the
basic relations of the pre-ribbon algebra and their diagrammatic representation.
\begin{itemize}
\item[\(a)]\vskip-\lastskip\smallskip
In the diagrams we represent $v_g^n$ by an edge with weight $n \in \Bbb Z$. In particular,
\(r1) states that an edge with weight 0 is the same as an edge without any weight.
\item[\(b)]
Axiom \(r4) says that the weight of an edge attached to a counit vertex can be changed
arbitrarily. By property \(p1), the same is true for the integral
vertex.
\item[\(c)] 
Axiom \(r5) and property \(r5') imply that the weight of an edge attached to a\break
multiplication vertex can be moved to any other edge attached to that vertex.
\item[\(d)] 
We remind that, given two Hopf algebras over the trivial groupoid $(A, m^A,\eta^A,\break
\Delta^A, \epsilon^A, S^A)$ and $(B, m^B,\eta^B, \Delta^B, \epsilon^B, S^B)$ in the same
braided monoidal category $\C$, a morphism $\sigma_{A,B}: \one \to A \diam B$ is called a
{\sl Hopf copairing} if the following conditions are satisfied:
\vskip-6pt
$$(\Delta^A \diam \id_B) \circ \sigma_{A,B} = (\id_{A \diam A} \diam m^B) \circ
(\id_A \diam \sigma_{A,B} \diam \id_B) \circ \sigma_{A,B},$$
$$(\id_A \diam \Delta^B) \circ \sigma_{A,B} = (m^A \diam \id_{B \diam B})
\circ (\id_A \diam \sigma_{A,B} \diam \id_B) \circ \sigma_{A,B},$$
$$(\epsilon^A \diam \id_B) \circ \sigma_{A,B} = \eta^B \quad \text{and} \quad 
(\id_A \diam \epsilon^B) \circ \sigma_{A,B} = \eta^A.$$
\vskip4pt\noindent
Moreover, the Hopf copairing $\sigma_{A,B}$ is called {\sl trivial} if $\sigma_{A,B} =
\eta^A \diam \eta^B$. Therefore, axioms \(r6) and \(r7) together with properties \(r7')
and \(p2-2'), imply that $\sigma_{i,j}: \one \to H_{1_i}\diamond H_{1_j}$ is a Hopf
pairing, and that such pairing is trivial for $i \neq j$. For this reason we will refer to
the morphisms $\sigma_{i,j}$'s as {\sl copairing morphisms} or simply {\sl copairings}.
\item[\(e)]
Property \(p3) tell us that the copairing is symmetric with respect to the rotation around
a vertical axis, while properties \(r5') and \(r7') are the symmetric version of axioms
\(r5) and \(r7). Hence, once those properties will be proved in Proposition
\ref{h-prop/thm}, they will imply that the functor $\sym: \H^u(\G) \to \H^u(\G)$ of
Proposition \ref{symmetry/thm} induces an analogous symmetry functor $\sym: \H_v^u(\G) \to
\H_v^u(\G)$, which fixes the ribbon morphisms $v_g$ for $g \in \G$.
\end{itemize}

We now proceed with the proof of the properties of $\H^u_v(\G)$ listed in Table
\ref{table-Huvprop/fig}, beginning with those which follow directly from the definition
\(r6) of $\sigma_{i,j}$ and the axioms \(r1-5), but do not depend on the copairing
property of $\sigma_{i,j}$ (axiom \(r7)).

\begin{lemma}\label{h-prop0/thm}
Properties \(r5'), \(f9), \(f10), \(f11-11'), \(p1), \(p2-2'), \(p7-7') and \(p9)
in Table \ref{table-Huvprop/fig}, hold in $\H_v^u(\G)$. Actually, they can be proved
without using axiom \(r7).
\end{lemma}

\begin{proof}
\(r5') can be shown to be equivalent to \(r5), by using the property \(s4) of the antipode
(see Table \ref{table-Hprop/fig}). \(f10) is a direct consequence of definition \(f2)
in Table \ref{table-Hprop/fig} and of the relations \(r3) and \(r5-5').

\(f11-11') are trivial when $i \neq j$, due to the definition \(r6) and the duality of 
the corresponding univalent vertices in Table \ref{table-Hprop/fig}, while the proofs for 
$i = j$ are presented in Figure \ref{h-form07/fig}. Here, we first prove \(f11') in the 
top line and then we show how it implies \(f11) up to isotopy moves in the bottom line. 

\begin{Figure}[htb]{h-form07/fig}
{}{Proof of \(f11-11') for $i = j$
   [{\sl f}/\pageref{table-Hu/fig}-\pageref{table-Huvprop/fig},
    {\sl r}/\pageref{table-Huvdefn/fig}, {\sl s}/\pageref{table-Hprop/fig}]}
\centerline{\fig{}{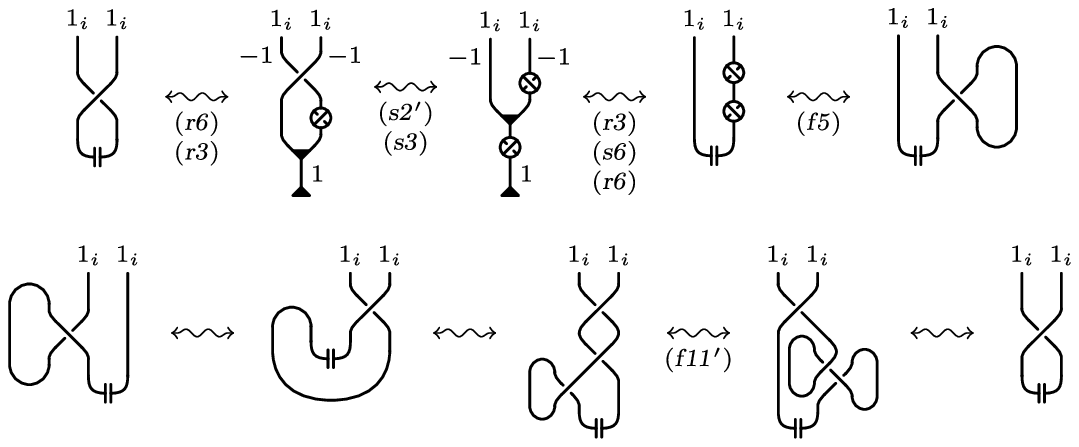}}
\vskip-3pt
\end{Figure}

\begin{Figure}[htb]{h-prop01/fig}
{}{Proof of \(p7) and \(p9)
   [{\sl a}/\pageref{table-Hdefn/fig},
    {\sl r}/\pageref{table-Huvdefn/fig}-\pageref{table-Huvprop/fig},
    {\sl s}/\pageref{table-Hdefn/fig}]}
\centerline{\fig{}{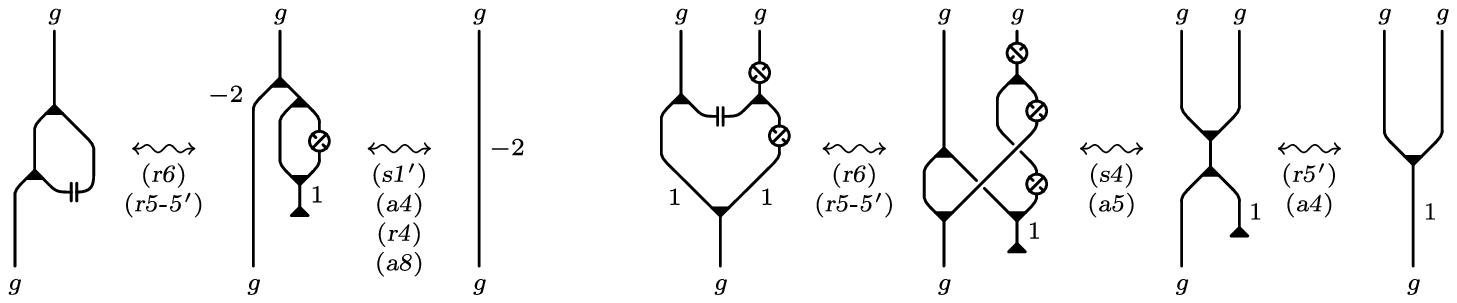}}
\vskip-3pt
\end{Figure}

\begin{Figure}[htb]{h-prop02/fig}
{}{Proof of \(f9) 
   [{\sl a}/\pageref{table-Hdefn/fig},
    {\sl f}/\pageref{table-Hu/fig}-\pageref{table-Huvprop/fig}, 
    {\sl i}/\pageref{table-Hu/fig}, {\sl p}/\pageref{table-Huvprop/fig},
    {\sl r}/\pageref{table-Huvdefn/fig}-\pageref{table-Huvprop/fig},
    {\sl s}/\pageref{table-Hdefn/fig}-\pageref{table-Hprop/fig}]\kern-3pt}
\vskip3pt
\centerline{\fig{}{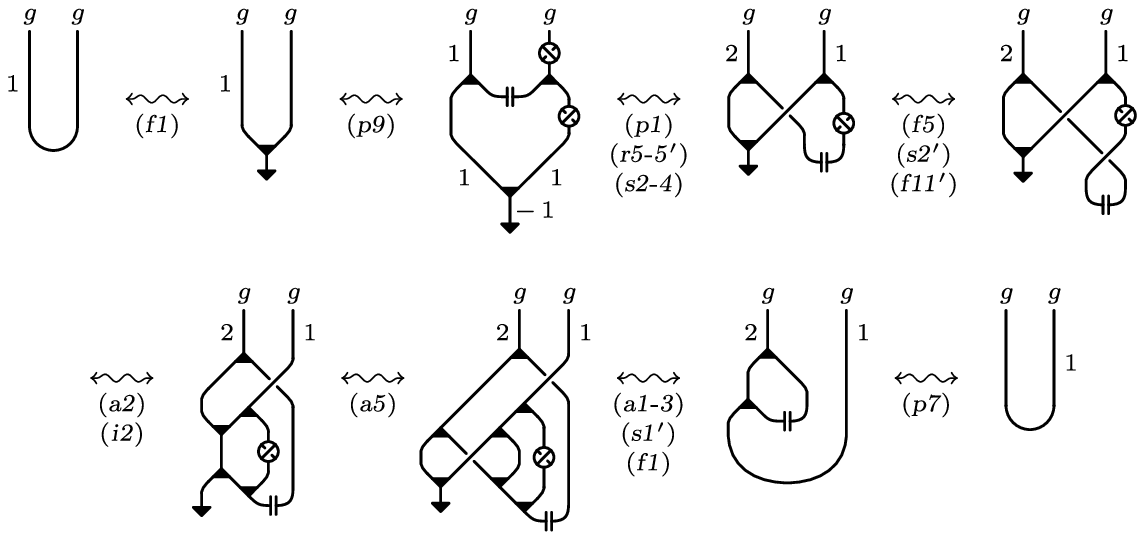}}
\vskip-3pt
\end{Figure}

\(p1) can be derived from \(r4), by using \(r5) and the duality of the negative uni-valent
vertices in Table \ref{table-Hprop/fig}. \(p2-2') immediately follow from the definition
\(r6) of the copairing, the axiom \(r2) and the relations \(a2-2') and \(s5-6) in
Tables \ref{table-Hdefn/fig} and \ref{table-Hprop/fig}. \(p7) and \(p9) are proved in
Figure \ref{h-prop01/fig}. \(p7') can be proved in a completely analogous way as \(p7).
\(f9) is proved in Figure \ref{h-prop02/fig}.
\end{proof}

\begin{lemma}\label{mu/thm}
For every $g \in \G(i,j)$ and $h \in \G(j,k)$, $$\mu_{g,h} = (m_{g,1_i} \!\diam m_{1_j,h})
\circ (\id_g \diam \sigma_{i,j} \diam \id_h) : H_g \diam H_h \to H_g \diam H_h$$ is an
isomorphism in $\H^u_v(\G)$, and $$\mu^{-1}_{g,h} = (m_{g,1_i} \!\diam m_{1_j,h}) \circ
(\id_g \diam ((\id_{1_i} \!\diam S_{1_j}) \circ \sigma_{i,j}) \diam \id_h): H_g \diam H_h
\to H_g \diam H_h$$ is its inverse (see Figure \ref{h-prop-mu/fig}). In other words,
properties \(p4-5) hold in $\H^u_v(\G)$.
\end{lemma}

\begin{Figure}[htb]{h-prop-mu/fig}
{}{The isomorphisms $\mu_{g,h}$ and $\mu^{-1}_{g,h}$}
\centerline{\fig{}{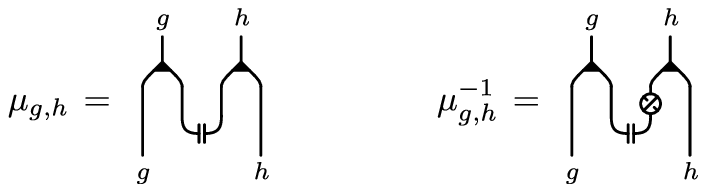}}
\vskip-3pt
\end{Figure}

\begin{proof}
Property \(p4) is proved in Figure \ref{h-prop03/fig}. The proof of \(p5) in analogous, 
using \(s1) in place of \(s1').
\end{proof}

\begin{Figure}[htb]{h-prop03/fig}
{}{Proof of \(p4)
   [{\sl a-s}/\pageref{table-Hdefn/fig}, {\sl r}/\pageref{table-Huvprop/fig}]}
\centerline{\fig{}{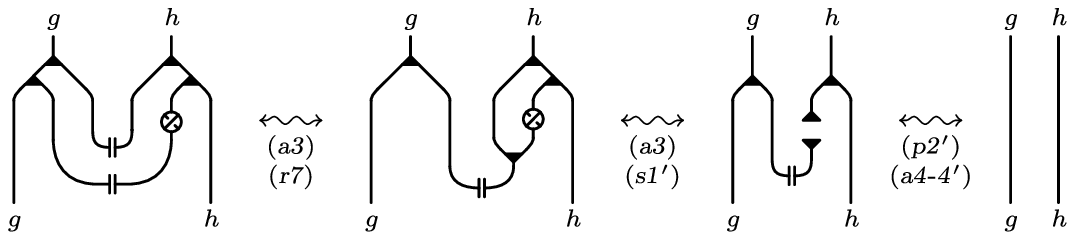}}
\vskip-3pt
\end{Figure}

\begin{proposition}\label{h-prop/thm}
Given a groupoid $\G$, all the relations in Table \ref{table-Huvprop/fig} are satisfied 
in the universal pre-ribbon Hopf $\G$-algebra $\H^u_v(\G)$. Moreover, there is an 
involutive antimonoidal category equivalence $\sym: \H_v^u(\G) \to \H_v^u(\G)$ uniquely 
determined by the identities in Proposition \ref{symmetry/thm} and by $\sym(v_g) = 
v_{\bar g}$.
\end{proposition}

\begin{proof}
In the light of the previous lemmas, we are left to prove the relations \(r7'), \(p3), 
\(p6), \(p8-8') and \(p9').

\(r7') is derived from \(r7) in Figure \ref{h-prop04/fig}, by using the relation \(f11'), 
the braid axioms and the properties of the antipode.

\begin{Figure}[htb]{h-prop04/fig}
{}{Proof of \(r7')
   [{\sl f}/\pageref{table-Hu/fig}-\pageref{table-Huvprop/fig},
    {\sl r}/\pageref{table-Huvprop/fig}
    {\sl s}/\pageref{table-Hdefn/fig}-\pageref{table-Hprop/fig}]}
\centerline{\fig{}{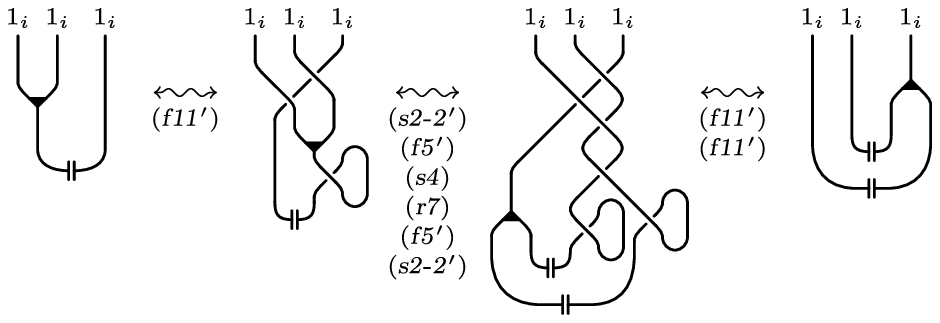}}
\vskip-3pt
\end{Figure}

To prove \(p3), let us consider the morphism $$\sym(\mu^{-1}_{g,h}) = (m_{g,1_i} \!\diam
m_{1_j,h}) \circ (\id_g \diam ((S_{1_i} \!\diam \id_{1_j}) \circ \sigma_{i,j}) \diam
\id_h): H_g \diam H_h \to H_g \diam H_h,$$ where for the moment $\sym(\mu^{-1}_{g,h})$ is
just as a notation for the symmetric of $\mu^{-1}_{g,h}$, with the antipode moved to the
left side of $\sigma_{i,j}$.
Then, by replacing \(r7'), \(p2') and \(s1') in Figure \ref{h-prop03/fig} respectively
with \(r7), \(p2) and \(s1), we obtain that \(p4) and \(p5) are still valid if
$\mu^{-1}_{g,h}$ is replaced by $\sym(\mu^{-1}_{g,h})$. Hence $\sym(\mu^{-1}_{g,h})$ and
$\mu^{-1}_{g,h}$ are equal, being both two-sided inverses of $\mu_{g,h}$. This gives
\(p3), by composing with $\eta_g \diam \eta_h$.

\begin{Figure}[b]{h-prop05/fig}
{}{Proof of \(p6)
   [{\sl a}/\pageref{table-Hdefn/fig},
    {\sl f}/\pageref{table-Hu/fig}-\pageref{table-Huvprop/fig},
    {\sl i}/\pageref{table-Hu/fig}, {\sl p}/\pageref{table-Huvprop/fig},
    {\sl r}/\pageref{table-Huvdefn/fig},
    {\sl s}/\pageref{table-Hdefn/fig}-\pageref{table-Hprop/fig}]}
\centerline{\fig{}{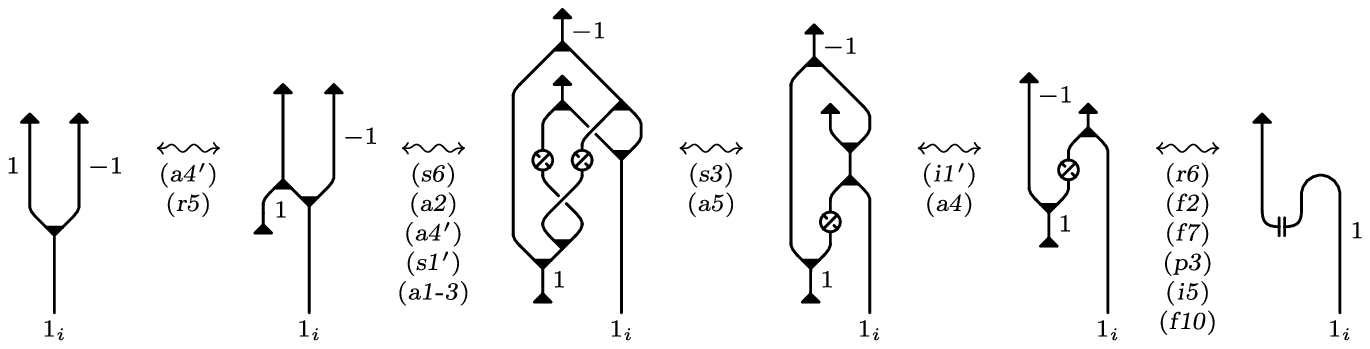}}
\vskip-3pt
\end{Figure}

\begin{Figure}[htb]{h-prop06/fig}
{}{Proof of \(p8)
   [{\sl a}/\pageref{table-Hdefn/fig}, {\sl p}/\pageref{table-Huvprop/fig},
    {\sl r}/\pageref{table-Huvdefn/fig}]}
\centerline{\fig{}{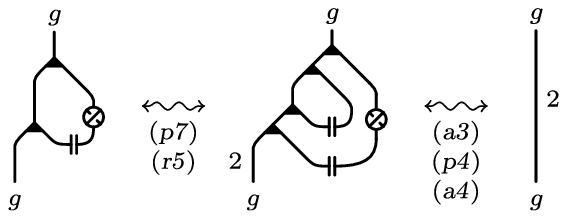}}
\vskip-3pt
\end{Figure}

At this point, we are ready to prove that $\sym: \H_v^u(\G) \to \H_v^u(\G)$ defines an
involutive antimonoidal category equivalence. According to Proposition \ref{symmetry/thm},
it is enough to check that it preserves the additional ribbon axioms. Indeed, property
\(p3) implies that the copairing is symmetric, i.e. the antipode which appears in its
definition \(r6) can be put on either side. Moreover, all ribbon axioms are left invariant
by $\sym$ with the exceptions of axioms \(r5) and \(r7), which are mapped to properties 
\(r5') and \(r7').

Finally, \(p6) and \(p8) are proved in Figures \ref{h-prop05/fig} and \ref{h-prop06/fig} 
respectively, while \(p8') and \(p9') follow from \(p8) and \(p9) by symmetry.
\end{proof}

We observe that the morphisms in $\H^u_v(\G)$ are represented by formal graph diagrams,
which can be interpreted as plane projection of uni/tri-valent graphs embedded in $R^3$.
Under this interpretation, some of the equivalence moves between such diagrams correspond
to graph isotopies in $R^3$. We collect these moves in the following definition.

\begin{definition}\label{isotopy/def}
Two diagrams representing morphisms in $\H^u_v(\G)$, as well as in its quotient $\H^r(\G)$
we will define below, will be called {\sl isotopic}, or obtainable from each other
through {\sl isotopy}, if they are related by a sequence of moves \(b1) to \(b4') in Table
\ref{table-Hdefn/fig}, \(f3-3') in Table \ref{table-Hprop/fig}, \(f5-5') and \(f6-7) in
Table \ref{table-Hu/fig}, \(f9), \(f10) and \(f11-11') in Table \ref{table-Huvprop/fig}.
\end{definition}

An example of diagram isotopy is shown in Figure \ref{h-form08/fig}. The reader can check 
that here only the braid axioms and moves \(f11-11') are needed.

\begin{Figure}[htb]{h-form08/fig} 
{}{An example of diagram isotopy}
\centerline{\fig{}{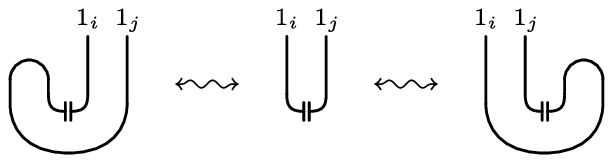}}
\vskip-3pt
\end{Figure}

{\sl Since the application of the isotopy moves will be frequent and quite intuitive, we 
will usually omit to explicitly indicate them in the diagrammatic proofs.}

\medskip

As we already mentioned, the universal algebra $\H^u_v(\G)$ is the algebra $\Alg$
introduced by Kerler in \cite{Ke02}, without the self-duality axiom, or equivalently
without the requirement that the copairing is non-degenerate. We do not impose yet the
self-duality, since for now we want to interpret the algebraic structure of cobordisms of
relative 4-dimensional 2-handlebodies. Self-duality will be added only later in Chapter
\ref{boundaries/sec}, when we will study 3-dimensional boundaries of such handlebodies.

The axioms of $\H^u_v(\G_n)$ are clearly too weak for what we need. In fact, they are
compatible with the trivial choice $v_g = \id_g$ of the ribbon morphisms, which brings to
trivial copairings. On the contrary, the ribbon morphisms of the Hopf algebra of
cobordisms are 2-handles with non-trivial framings and the copairing is the one introduced
by Lyubashenko in \cite{Ly95} (see \cite{Ke02} and Section \ref{Phi/sec} below).
Therefore, the universal algebraic category equivalent to such cobordism category
necessarily contains additional axioms connecting the ribbon structure to the braiding, in
such a way that if the braiding structure of the category is non-trivial, then the ribbon
one is forced to be non-trivial as well. Before introducing those axioms, we make the
following observation.

\begin{proposition}\label{rho/thm} 
Given any any $i \in \Obj\G$ and any $g \in \G(j,k)$, consider the morphisms $\rho_{i,g}$ 
and $\rho_{g,i} = \sym(\rho_{i,\bar g})$ of $\,\H^u_v(\G)$ defined by (see Figure 
\ref{h-prop-rho/fig})
$$
\begin{array}{rl}
 \rho_{i,g} &= (\id_{1_i} \diam m_{1_j,g}) \circ 
 (\sigma_{i,j} \diam \id_g): H_g \to H_{1_i} \diam H_g,\\[4pt]
 \rho_{g,i} &= (m_{g,1_k} \diam \id_{1_i}) \circ 
 (\id_g \diam \sigma_{k,i}): H_g \to H_g \diam H_{1_i}.
\end{array}
$$
Then $\rho_{i,g}$ (resp. $\rho_{g,i}$) makes $H_g$ into a left (resp. right)
$H_{1_i\!}$-comodule.
\begin{Figure}[htb]{h-prop-rho/fig}
{}{The morphisms $\rho_{g,k}$ and $\rho_{k,g}$}
\vskip-3pt
\centerline{\fig{}{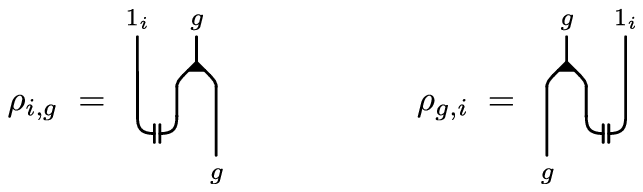}}
\vskip-3pt
\end{Figure}
\end{proposition}

\vskip-12pt\vskip0pt
\begin{proof}
The proposition is a direct consequence of \(p2-2') and \(r7-7').
\end{proof}

\begin{definition}\label{ribbon/def}
Given a groupoid $\G$, a {\sl ribbon Hopf $\,\G$-algebra} in a braided monoidal category
$\C$ is a unimodular Hopf $\G$-algebra $H$ with a family of ribbon morphisms $v$, which
satisfies the two additional conditions (cf. Table \ref{table-Hr/fig}):
\vskip-4pt
$$\Delta_g \circ v_g^{-1} = \mu_{g,g} \circ (v_g^{-1} \diam
v_g^{-1}) \circ \bar \gamma_{g,g} \circ \Delta_g: H_g \to H_g \diam H_g,
\eqno{\(r8)}$$
$$\arraycolsep0pt
\begin{array}{rl}
(m_{1_k,h} \diam m_{g,1_j}) \circ {}& (S_{1_k} \diam (\mu_{h,g} \circ \bar
\gamma_{g,h} \circ \mu_{g,h}) \diam S_{1_j}) \circ (\rho_{k,g} \diam
\rho_{h,j}) = {} \\[4pt] & {} = \gamma_{g,h}: H_g \diam H_h \to H_h \diam H_g,
\end{array}
\eqno{\(r9)}$$
\vskip4pt\noindent
for any $g \in \G(i,j)$ and $h \in \G(k,l)$.
\end{definition}

\begin{Table}[p]{table-Hr/fig}
{}{}
\centerline{\fig{}{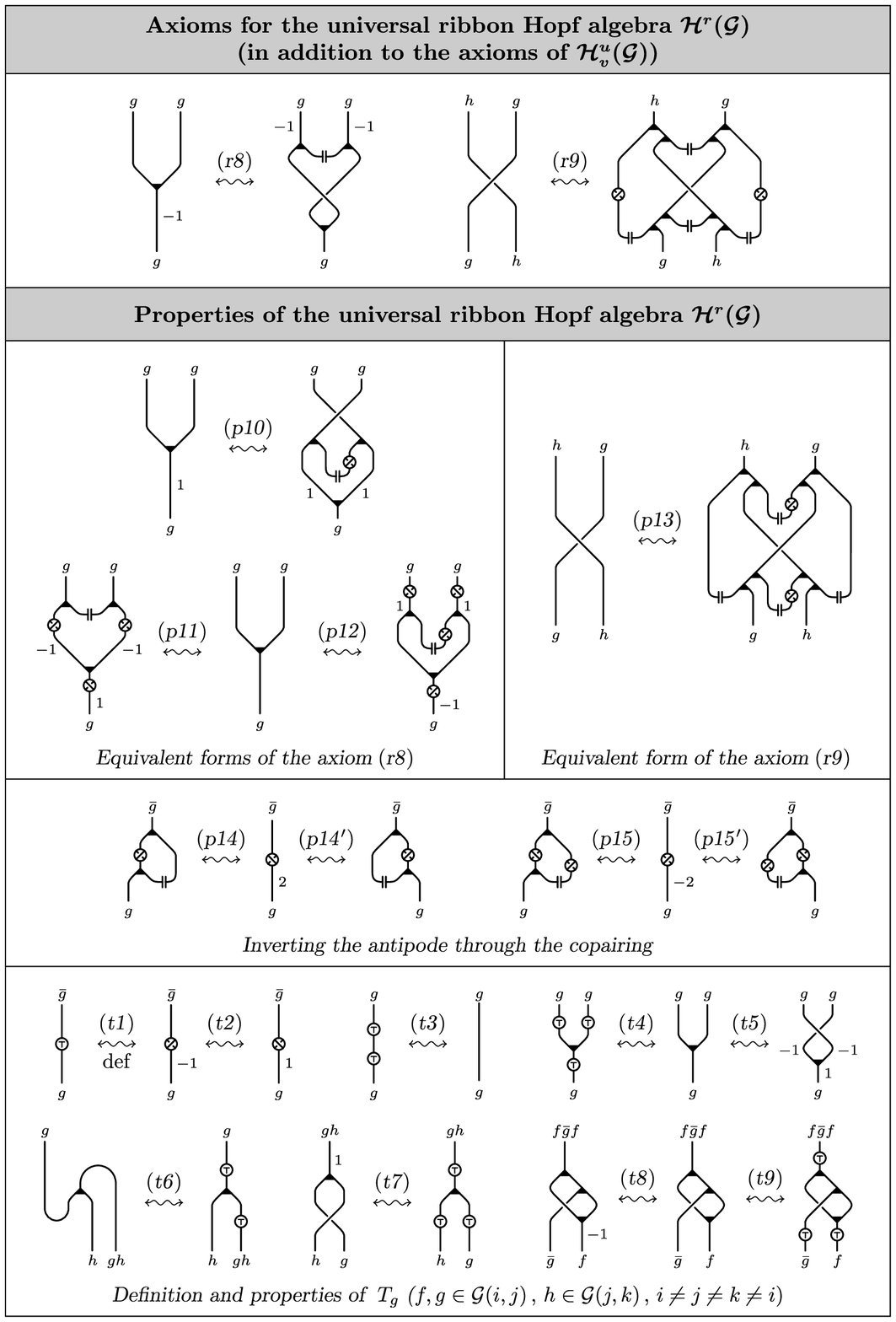}}
\end{Table}

\label{r9/rem}
We observe that the relations \(r8) and \(r9) simplify for some combinations of the
labels, because trivial copairings appear. In particular, move \(r9) reduces to a
crossing change when $g \in \G(i,j)$ and $h \in \G(k,l)$ with $\{i,j\} \cap \{k,l\} =
\emptyset$.

Moreover, for a ribbon Hopf algebra whose ribbon (and the copairings) morphisms are all
trivial the new axioms imply cocommutativity and trivial braiding. This proves the
independence of the two new axioms from the rest of the axioms of the algebra. However,
the same argument is not valid any more in the presence of the self-duality (see Section
\ref{Hb/sec}), which is not compatible with the trivial choice for the $v_g$'s. Therefore,
in the self-dual case, which refers to the algebraic description of 3-dimensional framed
cobordisms, we do not claim the independence of the axioms \(r8) and \(r9), even if we
are convinced that this is still true.

\begin{proposition}\label{hr-axioms/thm}
Modulo the axioms for $\H^u_v(\G)$, \(r8) is equivalent to either \(p10), \(p11) or
\(p12), while \(r9) is equivalent to \(p13), all these relations being defined in Table
\ref{table-Hr/fig}.
\end{proposition}

\pagebreak

\begin{proof}
The equivalence between \(r8) and \(p11) derives from \(s3) in Table \ref{table-Hprop/fig}
and \(r3) in Table \ref{table-Huvdefn/fig}, after having composed both sides of \(r8) with
$v_g$ on the bottom. \(p10) is obtained from \(r8) by composing both sides with $v_g$ on
the bottom and with the invertible morphism $\gamma_{g,g} \circ \mu^{-1}_{g,g} \circ (v_g
\diam v_g)$ (see Proposition \ref{mu/thm}) on the top. Analogously, \(p12) is obtained
from \(p11) by replacing $g$ with $\bar g$ and composing both sides with the invertible
morphisms $S_g \circ v_g^{-1}$ and $(\bar S_g \diam \bar S_g) \circ (v_g \diam v_g) \circ
\mu^{-1}_{g,g}$ respectively on the bottom and on the top.

To see that \(r9) implies \(p13), it suffices to observe that the diagram on the right
side of \(p13) can be reduced to the single crossing on the left side, by applying \(r9)
at the crossing in the middle of it, and then using one move \(p3) and four moves \(p4-5)
to cancel the corresponding copairings. The opposite argument shows that \(p13) implies
\(r9) as well.
\end{proof}

\begin{definition}\label{defnHr/def}
Given a groupoid $\G$, we define the {\sl universal ribbon Hopf $\,\G$-algebra} $\H^r(\G)$
as the braided strict monoidal category freely generated by a ribbon Hopf $\G$-algebra $H$
with a family of ribbon morphisms $v = \{v_g: H_g \to H_g\}_{g \in \G}$. Equiv\-alently,
$\H^r(\G)$ is the quotient of $\H^u_v$ modulo the relations \(r8) and \(r9), i.e. the
objects and elementary morphisms of $\H^r(\G)$ are the same as $\H^u_v(\G)$, while the
relations are those of $\H^u_v(\G)$ plus the additional axioms for the ribbon morphisms
\(r8) and \(r9) in Table \ref{table-Hr/fig}.
\end{definition}

\begin{proposition}\label{hr-prop/thm}
Given a groupoid $\G$, all the relations in Table \ref{table-Hr/fig} hold in the universal 
ribbon Hopf $\,\G$-algebra $\H^r(\G)$. Moreover, the category equivalence in Proposition 
\ref{h-prop/thm} passes to the quotient to give an involutive antimonoidal category 
equivalence $\sym: \H^r(\G) \to \H^r(\G)$.
\end{proposition}

\begin{proof}
The existence of the induced involutive antimonoidal category equivalence $\sym: \H^r(\G)
\to \H^r(\G)$ is due to the invariance of the axioms \(r8) and \(r9) under
$\sym: \H_v^u(\G) \to \H_v^u(\G)$.

Relations \(p10) to \(p13) were considered in Proposition \ref{hr-axioms/thm}. Figure
\ref{h-prop07/fig} proves \(p15), while \(p15') follows by symmetry. Then \(p14-14') can
be derived from \(p15-15') in the same way as \(p8) was derived from \(p7) (see Figure
\ref{h-prop01/fig}) and we leave their proof to the reader.

\begin{Figure}[htb]{h-prop07/fig}
{}{Proof of \(p15)
   [{\sl a}/\pageref{table-Hdefn/fig}, {\sl p}/\pageref{table-Hr/fig},
    {\sl r}/\pageref{table-Huvdefn/fig}-\pageref{table-Huvprop/fig},
    {\sl s}/\pageref{table-Hdefn/fig}-\pageref{table-Hprop/fig}]}
\centerline{\fig{}{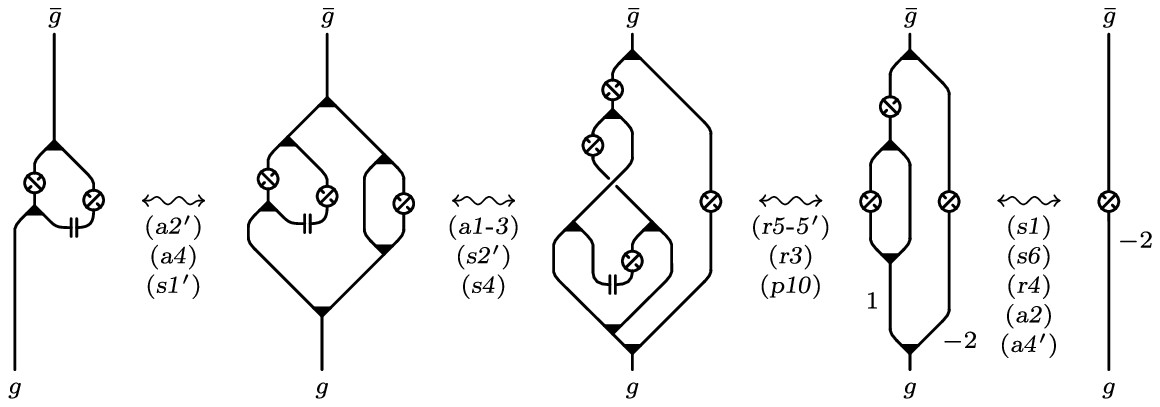}}
\vskip-3pt
\end{Figure}

Properties \(t2) to \(t9) concern the two morphisms $$T_g = S_g \circ v_g^{-1} \quad
\text{and} \quad \bar T_g = \bar S_g \circ v_g$$ for $g \in \G(i,j)$ with $i \neq j$.
To prove these properties, notice that for an arbitrary $g \in \G$ we have
$T_g^{-1} = \bar T_{\bar g}$, while $T_{\bar g} \circ T_g$ (resp. $\bar T_{\bar g} \circ
\bar T_g$) gives a full positive (resp. negative) twist of an edge with weight $-2$ (resp.
$+2$), as shown in Figure \ref{h-prop08/fig}.

\vskip0pt

\begin{Figure}[htb]{h-prop08/fig}
{}{$T_{\bar g} \circ T_g$ and $\bar T_{\bar g} \circ \bar T_g$
   [{\sl f}/\pageref{table-Hu/fig},
    {\sl r}/\pageref{table-Huvdefn/fig}]}
\centerline{\fig{}{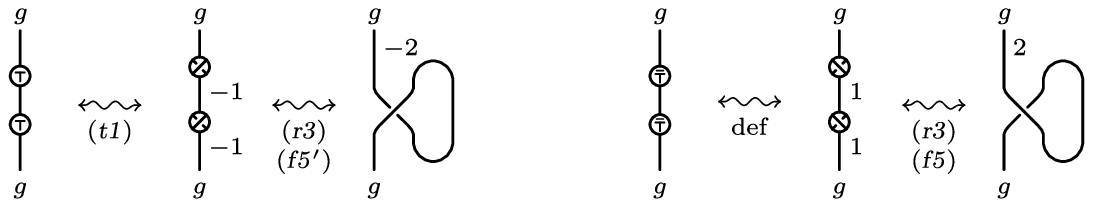}}
\vskip-3pt
\end{Figure}

Now, taking into account the triviality of $\sigma_{i,j}$ for $i \neq j$, we see that 
\(t2-3) immediately follow from relation \(p14), \(t4-5) rewrite \(p11) and \(r8) when $g 
\in \G(i,j)$ with $i\neq j$, and \(t6-7) rewrite \(f8) and \(s4) under the further 
assumption that $h \in \G(j,k)$ and $j \neq k \neq i$. Finally, \(t8-9) are proved in 
Figure \ref{h-prop09/fig}.
\end{proof}

\begin{Figure}[htb]{h-prop09/fig}
{}{Proof of \(t8-9)
   [{\sl a}/\pageref{table-Hdefn/fig},
    {\sl p}/\pageref{table-Huvprop/fig}-\pageref{table-Hr/fig},
    {\sl r}/\pageref{table-Huvdefn/fig}-\pageref{table-Huvprop/fig}, 
    {\sl s}/\pageref{table-Hprop/fig},
    {\sl t}/\pageref{table-Hr/fig}]}
\vskip-3pt
\centerline{\fig{}{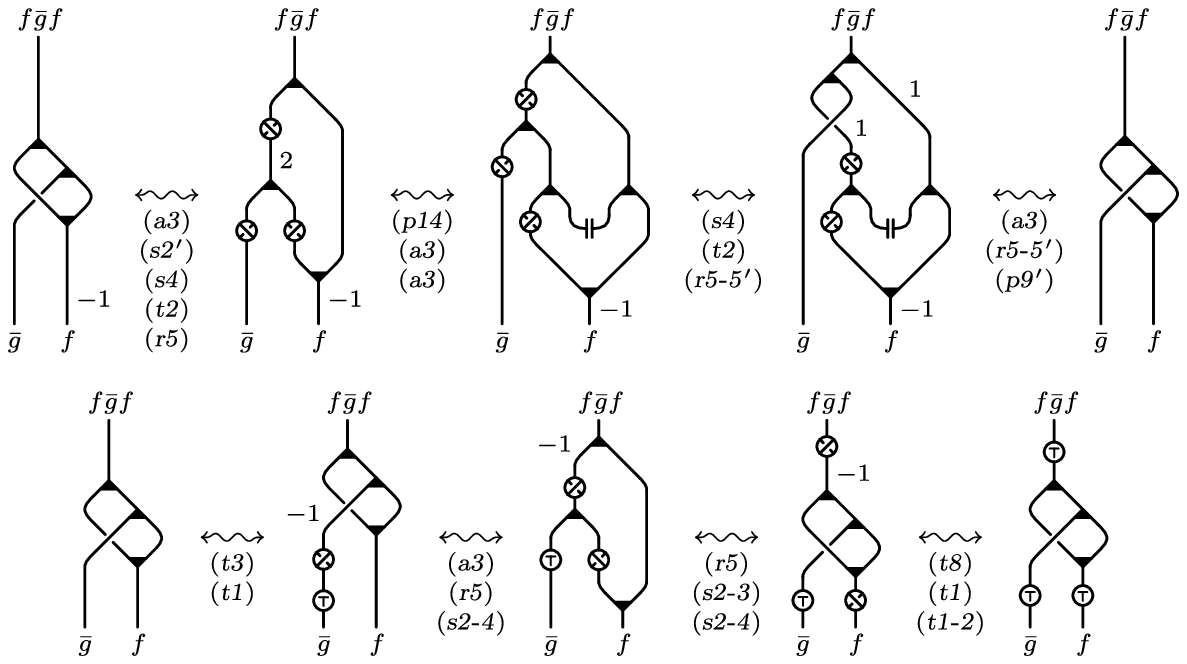}}
\vskip-3pt
\end{Figure}

We observe that, thanks to properties \(t3-4), when $g \in \G(i,j)$ with $i \neq j$ the
morphism $T_g$ is a coalgebra isomorphism, i.e. $\Delta_{\bar g}\circ T_g=(T_{\bar g}
\diam T_{\bar g}) \circ \Delta_g$.

\medskip

The next proposition tells us that an inclusion of groupoids $\G \subset \G'$ induces an
inclusion between the correspoding universal ribbon Hopf $\G$-algebras, hence we can
write $\H^r(\G) \subset \H^r(\G')$.

\begin{proposition}\label{formalext/thm} 
Any functor $\phi: \G \to \G'$ between groupoids which is injective on the set of objects
can be extended to a functor $\Upsilon_\phi: \H^r(\G) \to \H^r(\G')$. Moreover, if $\phi$
is faithful (an embedding) then $\Upsilon_\phi$ is also faithful.
\end{proposition}

\begin{proof}
We formally define $\Upsilon_\phi$ by applying $\phi$ to the indices of all the elementary
morphisms, that is $\Upsilon_\phi(\gamma_{g,h}) = \gamma_{\phi(g),\phi(h)}$,
$\Upsilon_\phi(\Delta_g) = \Delta_{\phi(g)}$, $\Upsilon_\phi(\eta_i) = \eta_{\phi(i)}$,
etc.

To see that $\Upsilon_\phi$ is well-defined, we need to check that all relations for
$\H^r(\G)$ are satisfied in the image. The only problem we might have would be with
relations \(r8) and \(r9) in Table \ref{table-Hr/fig}. Those relations will be satisfied
if we have that $\Upsilon_\phi(\sigma_{i,j}) = \sigma_{\phi(i),\phi(j)}$, which is always
true since $\phi$ is injective on objects. This concludes the first part of the
proposition since the functoriality of $\Upsilon_\phi$ is obvious.

At this point, it is left to show that when $\phi$ is injective on morphisms, then
$\Upsilon_\phi$ is injective on morphisms as well. In this case $\phi$ induces an
isomorphism of categories $\hat{\phi}: \G \to \phi(\G)$ and $\Upsilon_{\hat\phi}: \H(\G)
\to \H(\phi(\G))$ is an isomorphism of categories as well, being $\Upsilon_{\hat\phi}$ and
$\Upsilon_{\hat\phi^{-1}}$ inverse of each other by construction. Moreover, $\Upsilon_\phi
= \Upsilon_\iota \circ \Upsilon_{\hat\phi}$ where $\iota:\phi(\G) \subset \G'$ is the
corresponding inclusion. Hence, it suffices to show that the functor $\Upsilon_\iota$ is
injective on morphisms.

Let $F,F':A \to B$ be morphisms of $\H(\phi(\G))$ with $\Upsilon_\iota(F) =
\Upsilon_\iota(F')$. Then, they are represented by diagrams labeled in $\phi(\G)$ and
related by a sequence of moves in $\H(\G')$. Now, when we apply a relation move to a
diagram representing a morphism of $\H(\G')$, the only new labels that can appear are
identities of $\G'$ and products of labels already occurring in it or their inverses.
Therefore, since $\phi(\G)$ is a subcategory of $\G'$, the only labels not belonging to
$\phi(\G)$ that can occur in the intermediate diagrams of the sequence are identities
$1_i$ with $i \in \Obj \G' - \Obj \phi(\G)$. The parts of the diagram carrying such labels
interact with the rest of the diagram only through move \(r9) in Table
\ref{table-Hr/fig}. Hence, by applying move \(r6) to change the trivial copairings into
two units, we can disconnect those parts from the rest of the intermediate diagrams. After
that, we can delete them to get a new sequence of diagrams between $F$ and $F'$ related by
moves in $\H(\phi(\G))$, which proves that $F = F'$ in $\H(\phi(\G))$.
\end{proof}

\begin{Figure}[b]{phi01/fig}
{}{The functor $\Phi_n\!:\H_n^r \to \K_n$}
\centerline{\fig{}{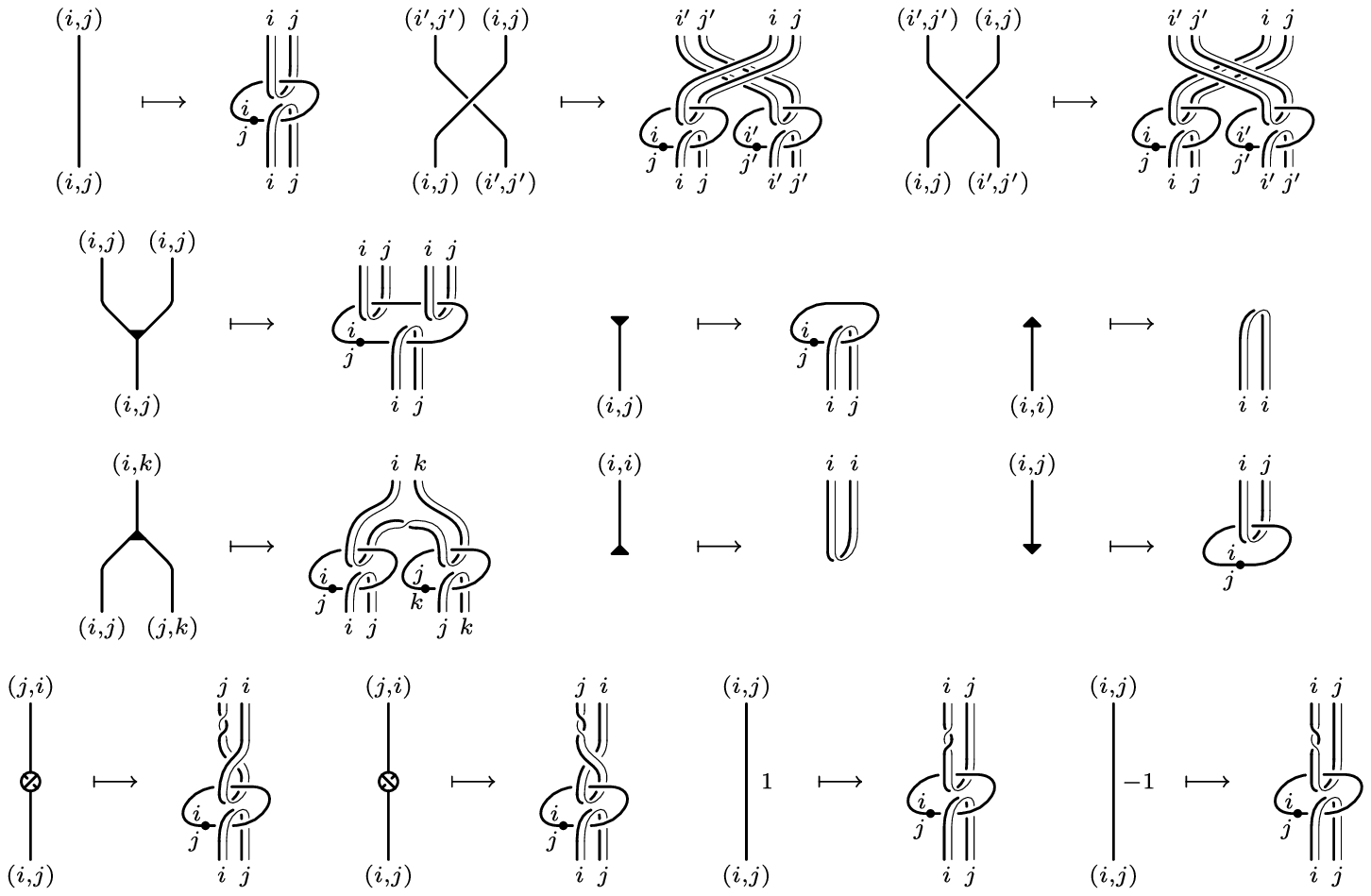}}
\vskip-3pt
\end{Figure}

\subsection{The functors $\Phi_n: \H^r_n \to \K_n$%
\label{Phi/sec}}

In the introduction of this chapter we argued that the category of generalized Kirby
tangles $\K_n$ contains a ribbon Hopf $\G_n$-algebra. We will prove this fact here by
constructing a functor from the universal ribbon Hopf $\G_n$-algebra to $\K_n$. Being this
algebra our main object of study we will simplify the notation as follows.

\medskip

For $n \geq 1$, we denote by $\H^r_n$ the universal ribbon Hopf $\G_n$-algebra
$\H^r(\G_n)$ associated to the groupoid $\G_n$, consisting of the set $\{1, \dots, n\}^2$
with the natural groupoid structure given by $(i,j) (j,k) = (i,k)$ for any $1 \leq i,j,k
\leq n$.

\medskip

The next theorem is an extension of the well-known fact that the category of
ad\-missible tangles contains a braided Hopf algebra (see \cite{Ke02, Ha00}). Indeed,
it shows that the categories of generalized Kirby tangles contain groupoid ribbon Hopf
algebras.

\begin{theorem}\label{phi/thm}
There exists a braided monoidal functor $\Phi_n: \H^r_n \to \K_n$, which sends every
object $H_{(i,j)}$ to $I_{(i,j)}$ (defined in Section \ref{Kirby/sec}) and the elementary
morphisms of $\H_n^r$ to the generalized Kirby tangles described in Figure 
\ref{phi01/fig}.
\end{theorem}

Before proving the theorem, we observe that the images through $\Phi_n$ of the form
$\lambda_{(i,j)}$, the coform $\Lambda_{(i,j)}$, the copairing $\sigma_{i,j}$ and the
morphism $T_{(i,j)}$, are equivalent in $\K_n$ to the tangles presented in Figure
\ref{phi02/fig}. The case of $\sigma_{i,i}$ is shown in Figure \ref{phi03/fig}, where some
1/2-handle cancelation is understood in the first diagram, while the other cases are
easier and left to the reader. Notice that the image of the copairing $\sigma_{i,i}$ is
exactly the Lyubashenko's copairing defined in \cite{Ly95}.

\begin{Figure}[htb]{phi02/fig}
{}{Images under $\Phi_n$ of $\lambda_{(i,j)}\,, \Lambda_{(i,j)}\,, \sigma_{i,j}\,, 
   T_{(i,j)}$ and $\bar T_{(i,j)}$}
\centerline{\fig{}{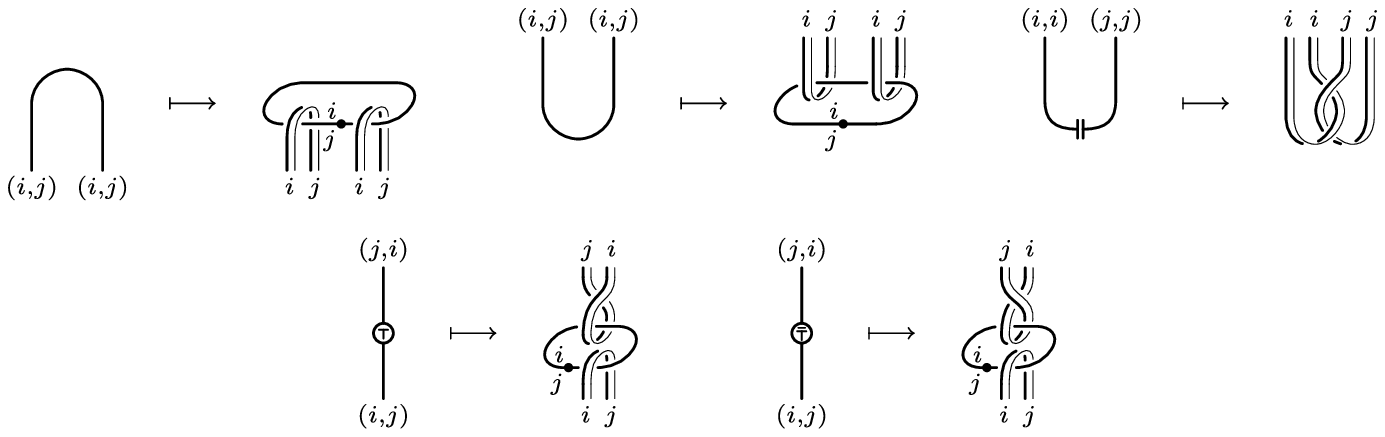}}
\vskip-3pt
\end{Figure}

\begin{Figure}[htb]{phi03/fig}
{}{Deriving $\Phi_n(\sigma_{i,i})$ from \(r6)}
\centerline{\fig{}{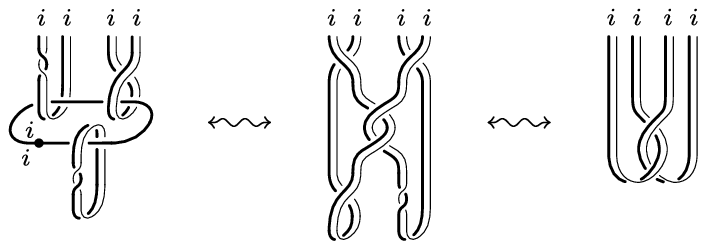}}
\vskip-3pt
\end{Figure}

\begin{proof}[Theorem \ref{phi/thm}]
We have to verify that the definition of $\Phi_n$ on the elementary morphisms is
compatible with the axioms for $\H_n^r$, namely that it determines equivalent images
in $\K_n$ for the two diagrams involved in each of those axioms.

This is easy to check for most of the unimodular braided Hopf algebra axioms in Tables
\ref{table-Hdefn/fig} and \ref{table-Hu/fig}. In particular, it reduces to isotopy for the
braid axioms and to deletion of canceling 1/2-pairs for the axioms \(a2-2'), \(a6), \(a8),
\(i3) and \(i4), while one needs to make one or two handle slides before deleting for the
axioms \(a1), \(a3), \(a4-4'), \(a7), \(s2-2'), \(i1), \(i2) and \(i5).

Axioms \(a5) and \(s1) are considered in Figures \ref{phi04/fig} and \ref{phi05/fig}
respectively. Axiom \(s1') can be treated similarly to \(s1).

\begin{Figure}[htb]{phi04/fig}
{}{The definition of $\Phi_n$ is compatible with \(a5)}
\centerline{\fig{}{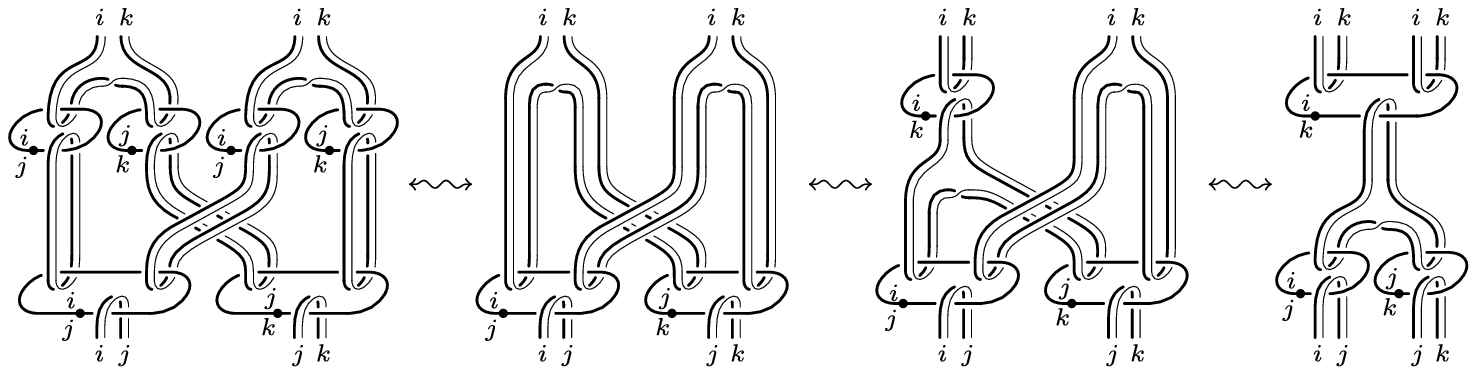}}
\vskip-3pt
\end{Figure}

\begin{Figure}[htb]{phi05/fig}
{}{The definition of $\Phi_n$ is compatible with \(s1)}
\vskip-3pt
\centerline{\fig{}{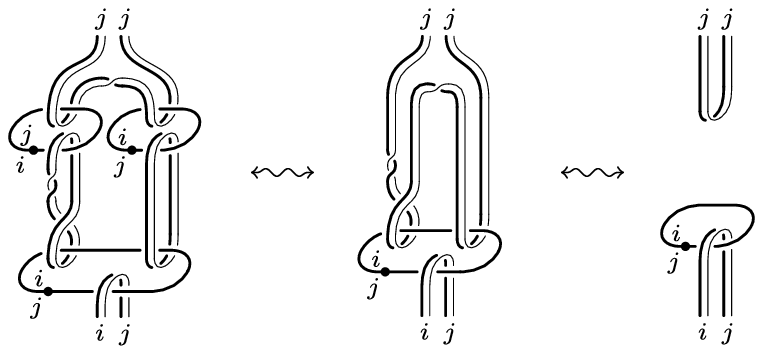}}
\vskip-3pt
\end{Figure}

\begin{Figure}[b]{phi06/fig}
{}{The definition of $\Phi_n$ is compatible with \(r7)}
\vskip-6pt
\centerline{\fig{}{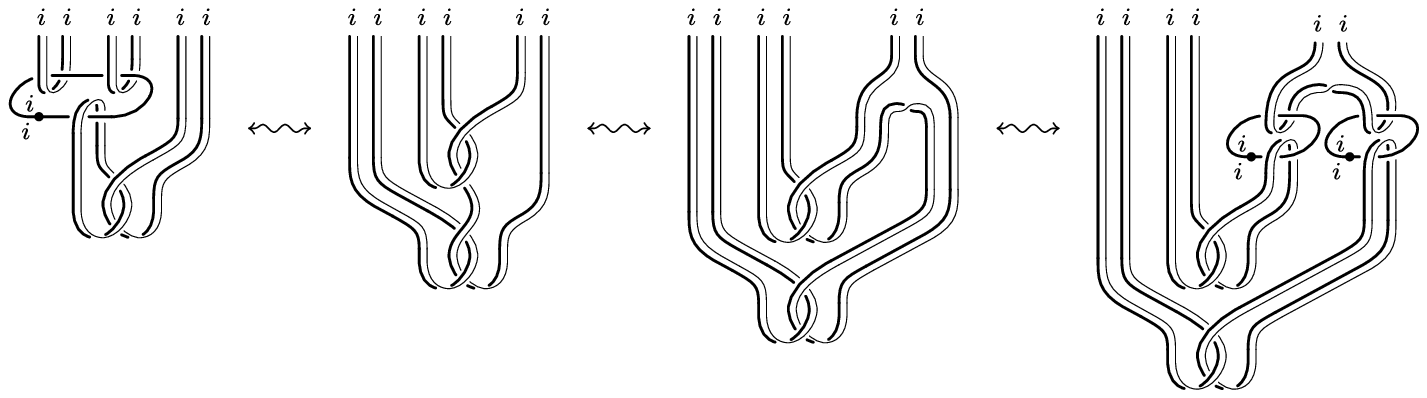}}
\vskip-3pt
\end{Figure}

\begin{Figure}[b]{phi07/fig}
{}{The definition of $\Phi_n$ is compatible with \(r8)}
\vskip-3pt
\centerline{\fig{}{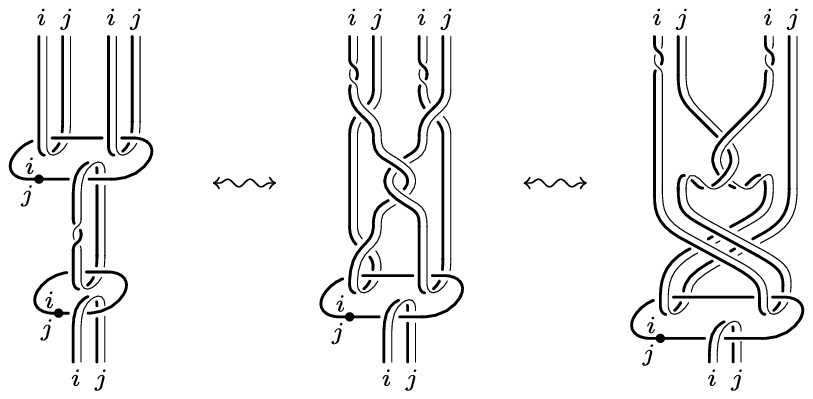}}
\vskip-3pt
\end{Figure}

\begin{Figure}[htb]{phi08/fig}
{}{The definition of $\Phi_n$ is compatible with \(r9)}
\centerline{\fig{}{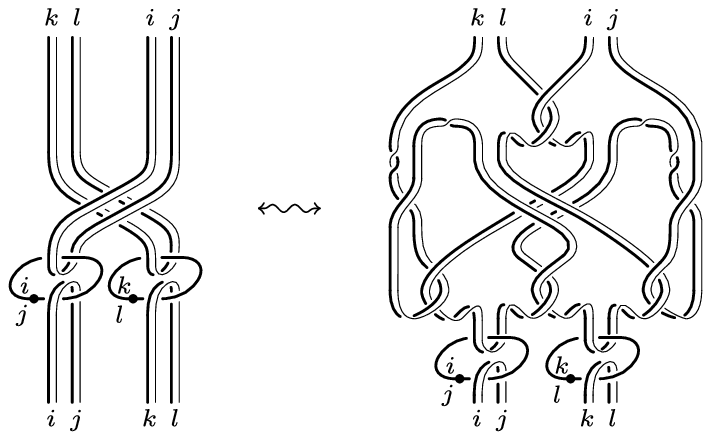}}
\vskip-3pt
\end{Figure}

Now, let us pass to the ribbon axioms in Tables \ref{table-Huvdefn/fig} and
\ref{table-Hr/fig}. The compatibility with the axioms \(r1) to \(r5) can be easily proved
by applying once again cancelations of 1/2-pairs after suitable handle slidings. The rest
of the ribbon axioms are dealt with in Figures \ref{phi06/fig}, \ref{phi07/fig} and
\ref{phi08/fig}. Here, in the rightmost diagrams of the last two figures some cancelations
of 1/2-pairs and some crossing changes moves between components with different labels
(when they appear) have been\break performed.
\end{proof}

\subsection{The adjoint morphisms%
\label{adjoint/sec}}

As we have seen in Section \ref{K/sec}, the {\sl pushing through an 1-handle} move plays
an essential role in the definition of the reduction functor $\down_k^n: \K_{n \red k} \to
\K_k$ for Kirby tangles and in the proof that such functor is a category equivalence.

Now we introduce and study the algebraic analog of such move: the intertwining of a
morphism with the adjoint action of the ribbon Hopf algebra on itself. This will be used
in Section \ref{reduction/sec} to define the reduction functor between the corresponding
algebraic categories.

\medskip

We remind that the (right) {\sl adjoint action} of a group $G$ on itself is defined by
$\ad(x): g \mapsto g^x = \bar x g x$ for any $x,g \in G$. This definition can be
extended to an arbitrary groupoid $\G$ just by deleting $x$ and/or $\bar x$ from $\bar x g
x$ when the corresponding composition with $g$ is not defined.

\begin{proposition}\label{red-groupoid/thm}
Let $\G$ be a groupoid. Given $x \in
\G(i_0,j_0)$, let $\_^x: \G \to \G$ be defined on the objects as $i_0^x = j_0$ and
$i^x = i$ for $i \neq i_0$, and on the morphisms as
\vskip-4pt
$$g^x = \left\{
\arraycolsep0pt
\begin{array}{ll}
 \bar x g  x &\quad\text{if } g \in \G(i_0,i_0),\\
 \bar x g &\quad\text{if } g \in \G(i_0,j) \text{ with } j \neq i_0,\\
 g x &\quad\text{if } g \in \G(i,i_0) \text{ with } i \neq i_0,\\
 g &\quad\text{if } g \in \G(i,j) \text{ with } i,j \neq i_0.
\end{array}
\right.$$
\vskip4pt

Denoting by $\G^{\bs i} \subset \G$ the full subgroupoid of $\G$ with $\Obj \G^{\bs i} =
\Obj\G - \{i\}$ for $i \in \Obj\G$, the following statements hold:
\begin{itemize}
\item[\(a)]
$\_^x: \G \to \G$ is a functor, in particular $(gh)^x = g^{x}h^x$ and $\bar{g^x} = \bar
g^x$ for any $g,h \in \G$;
\item[\(b)]
$(g^x)^y = g^{xy}$ for any $x,y,g\in \G$, that is the functors $\_^x$ with $x \in \G$ give
a right action of $\G$ on itself;
\item[\(c)] 
if $i_0 \neq j_0$, then the image $\G^x$ of $\_^x$ is the subgroupoid $\G^{\bs i_0}$ and
$\_^x: \G \to \G^{\bs i_0}$ is a left inverse of the inclusion $\G^{\bs i_0} \subset \G$;
\item[\(d)] 
$\_^x$ restricts to an equivalence of categories $\G^{\bs j_0} \to \G^{\bs i_0}$, whose
inverse $\G^{\bs i_0} \to \G^{\bs j_0}$ is the corresponding restriction of $\_^{\bar x}$
(both are the identity of $\G^{\bs i_0}$ if $i_0 = j_0$);
\item[\(e)] 
for any $x \in \G(i_0,j_0)$ and $y \in \G(i_0,k_0)$, there exists a natural
transformation $N: \_^x \to \_^y$ such that $N(i_0) = \bar x y$ and $N(i) = 1_i$ for
$i \neq i_0$.
\end{itemize}
\vskip-12pt
\end{proposition}

\begin{proof}
All statements are straightforward and left to the reader.
\end{proof}

Given $x \in \G(i_0, j_0)$, the functor $\_^x: \G \to \G$ uniquely extends to a monoidal
map $\_^x: \seq\G \to \seq\G$ given by $\pi^x = (g_1^x, \dots, g_m^x)$ for any $\pi =
(g_1, \dots, g_m) \in \seq\G$, and this in turn induces a monoidal map $\_^x: \Obj\H^r(\G)
\to \Obj\H^r(\G)$, by putting $H_\pi^x = H_{\pi^x}$. Note that the specialization of
$\_^x$ to the case of the groupoid $\G_n$, coincides the homonymous map defined in Lemma
\ref{K-push/thm}, i.e. given a sequence $\pi\in\seq\G_n$, $\pi^x$ is obtained by changing
all elements $i_0$ to $j_0$.

The main goal of this section is to construct a monoidal functor $\_^x: \H^r(\G) \to
\H^r(\G)$ which coincides with this map on the objects and it is the algebraic analog of
the functor $\_^x: \K_n \to \K_n$ defined in Lemma \ref{K-push/thm}. In particular, we
require that the following conditions are satisfied:
\begin{itemize}
\item [\(a)]
if $x = 1_{i_0}$ then $\_^x$ is the identity functor;
\item[\(b)] 
given $y \in \G(i_0,k_0)$, there is a natural equivalence $$\xi^{x,y}: \id_{\bar y x}
\diam\, \_^x \to \id_{\bar y x} \diam\, \_^y\,;$$ that is, to any $\pi \in \seq\G$ is
associated an isomorphism $\xi^{x,y}_\pi: H_{\bar yx} \diam H_\pi^x \to H_{\bar yx} \diam
H_\pi^y$, such that for every morphism $F: H_{\pi_0} \to H_{\pi_1}$ in $\H^r(\G)$ we have
$$(\id_{\bar y x} \diam F^y) \circ \xi^{x,y}_{\pi_0} =
\xi^{x,y}_{\pi_1} \circ (\id_{\bar y x} \diam F^x)\,;$$
%
\item[\(c)]
if $\G = \G_n$, then $\Phi_n(F^x) = \Phi_n(F)^x$ and $\Phi_n(\xi^{x,y}_\pi) =
\xi^{x,y}_\pi$ according to Lemma \ref{K-push/thm}, for any $x,y \in \G_n$, $\pi \in \seq
\G_n$ and $F \in \Mor \H^r_n$.
\end{itemize}

\medskip

The natural transformation $\xi^{x,y}$ in \(b) will be expressed in terms of the
categorical generalization of the adjoint action of the ribbon groupoid Hopf algebra on
itself. Before introducing this generalization, we remind the definition of the adjoint
action in the case of the trivial groupoid.

Let $\C$ be a (strict) braided monoidal category, $(H, m_H, \eta_H, \Delta_H, \epsilon_H,
S_H)$ be a braided Hopf algebra in $\C$ over the trivial groupoid, and $A$ be an algebra
in $\C$ with multiplication $m_A$ and unit $\eta_A$. We remind that a morphism $\alpha: H 
\diam A \to A$ defines a left action of $H$ on $A$ if the following conditions hold:
\vskip-4pt
$$\alpha \circ (\eta_H \diam \id_A) = \id_A: A \to A\,;$$
$$\alpha \circ (m_H \circ \id_A) = \alpha \circ (\id_H \diam \alpha): H \diam H \diam A
\to A\,;$$
$$\alpha \circ (\id_H \diam \eta_A) = \eta_A \circ \epsilon_H: H \to A\,;$$
$$\alpha \circ (\id_H \diam m_A) = m_A \circ (\alpha \diam \alpha) \circ (\id_H \diam
\gamma_{H,A} \diam \id_A) \circ (\Delta_H \diam \id_{A \diam A}): H \diam A \diam A \to
A\,.$$
\vskip4pt\noindent
The first two conditions express the fact that $A$ is a left $H$-module, while the last
two state that the action intertwines with the multiplication and the unit of $A$, giving
in this way a left $H$-algebra structure on $A$ (cf. Definition 4.1.2 in \cite{Mon92}).
The notion of right action is symmetric and corresponds to a right $H$-algebra structure 
on $A$.

The (left) adjoint action of $H$ on itself is defined as
$$\ad_H = m_H \circ (m_H \diam S_H) \circ (\id_{H} \diam \gamma_{H,H}) \circ 
(\Delta_H \diam \id_H): H \diam H\to H.$$
One can introduce analogously the right adjoint action of $H$ on itself, by a symmetric
formula with interchanged roles of the two $H$'s in the source. In the case of a group
algebra and trivial braiding, the formula for the right adjoint action of $H$ on itself
coincides with the adjoint action of a group $G$ on itself recalled above.

\begin{Table}[b]{table-adjointdefn/fig}
{}{}
\vskip6pt
\centerline{\fig{}{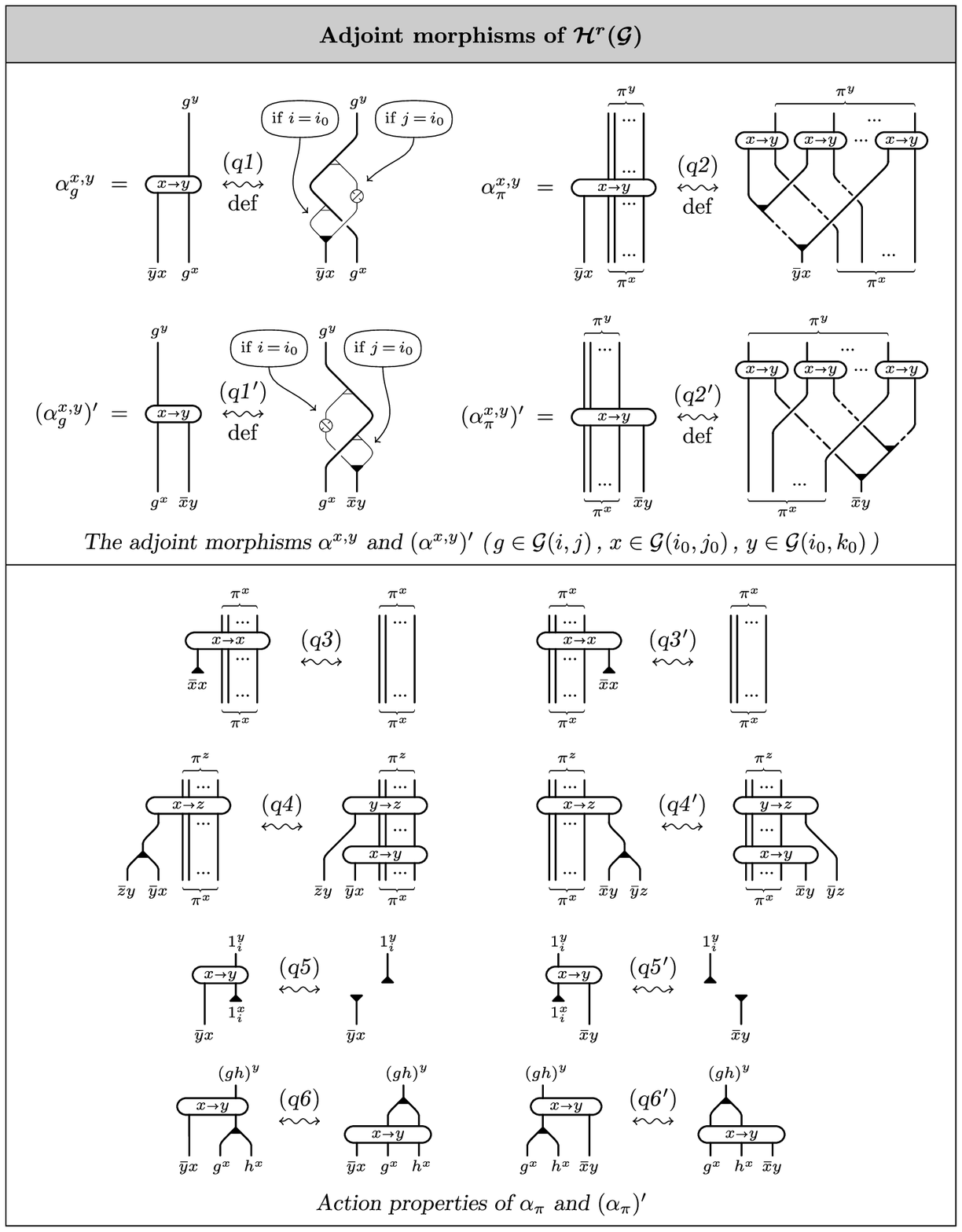}}
\vskip-3pt
\end{Table}

The fact that these are indeed left and right actions is a classical result and the reader
can find the proof in a more general context in Proposition \ref{alpha/thm} below. In
particular, the adjoint action intertwines with the multiplication and the unit morphisms.
On the other hand, in the classical case (trivial braiding) such action is known to
intertwine with the comultiplication and the antipode only when the Hopf algebra is
cocommutative (see Lemma 5.7.2 in \cite{Mon92}). One of the main results of this section
will that, in the case of a ribbon braided Hopf algebra over the trivial groupoid, the
adjoint action intertwines with all elementary morphisms including comultiplication,
antipode and braiding (cf. Proposition \ref{H-push/thm}). Since a ribbon Hopf algebra with
non-trivial braiding is not cocommutative, the proof of this fact necessarily makes use of
the new ribbon axioms \(r8) and \(r9) in Table \ref{table-Hr/fig}.

Now, we proceed with the generalization of the notion of adjoint action to a groupoid Hopf
$\G$-algebra with a possibly non-trivial groupoid $\G$.

\begin{definition}\label{alpha/def}
Let $\G$ be a groupoid. Given $x \in \G(i_0,j_0)$ and $y \in \G(i_0,k_0)$, for any $\pi
\in \seq\G$ we define the left adjoint morphism $\alpha_\pi^{x,y}: H_{\bar y x} \diam
H_\pi^x \to H_\pi^y$ inductively by the following identities (cf. Table
\ref{table-adjointdefn/fig}), where $g \in \G(i,j)$ and $\pi = \pi' \diam \pi''$:
\vskip-6pt
$$\arraycolsep0pt
\alpha_g^{x,y} = \left\{
\begin{array}{ll} 
 \epsilon_{\bar y x} \diam \id_{g^x} &\quad \text{if } i \neq i_0 \neq j,\\[2pt]
 m_{\bar y x,g^x} &\quad  \text{if } i = i_0 \neq j,\\[2pt]
 m_{g^x\!,\bar x y} \circ (\id_{g^x} \diam S_{\bar y x}) \circ 
 \gamma_{\bar y x,g^x} &\quad  \text{if } i \neq i_0 = j,\\[2pt]
 m_{\bar y x g^x\!, \bar x y} \circ (m_{\bar y x,g^x} \diam S_{\bar y x}) \circ {}\\[2pt]
 \quad {} \circ (\id_{\bar y x} \diam \gamma_{\bar y x,g^x}) \circ (\Delta_{\bar y x}
 \diam \id_{g^x}) 
 &\quad \text{if } i = i_0 = j;
\end{array}\right.
\eqno{\(q1)}
$$
$$
\alpha_\pi^{x,y} = \alpha_{\pi' \diam \pi''}^{x,y} = (\alpha_{\pi'}^{x,y} \diam 
 \alpha_{\pi''}^{x,y}) \circ (\id_{\bar y x} \diam \gamma_{\bar y x, \pi'^x} \diam 
 \id_{\pi''^x}) \circ (\Delta_{\bar y x} \diam \id_{\pi'^x \diam \pi''^x}).
 \eqno{\(q2)}
$$
\vskip4pt\noindent
We also define the symmetric right adjoint morphism $(\alpha_\pi^{x,y})': H_\pi^x \diam
H_{\bar x y} \to H_\pi^y$ by the following identity (cf. Table
\ref{table-adjointdefn/fig}), where $\sym(\pi) = \sym(g_1, \dots, g_m) = (\bar g_m, \dots,
\bar g_1)$:
$$
(\alpha_\pi^{x,y})' = \sym(\alpha_{\sym(\pi)}^{x,y})\,.
$$
We will use the simplified notation $\alpha_\pi^x$ and $(\alpha_\pi^y)'$ respectively
for $\alpha_\pi^{\smash{x,1_{i_0}}}\!: H_x \diam H_\pi^x \to H_\pi$ and $
(\alpha_\pi^{\smash{1_{i_0},y}})': H_\pi^y \diam H_y \to H_\pi$.
\end{definition}

{\sl We emphasize that in the definition above $H^x_\pi$ and $H^y_\pi$ should not be
thought as objects of $\,\H^r(\G)$, but instead as pairs $(H_\pi,x)$ and $(H_\pi,y)$, and
similarly $\pi^x$ and $\pi^y$ in the corresponding diagrams should not be thought as
sequences in $\seq\G$, but instead as pairs $(\pi,x)$ and $(\pi,y)$. This requires that
both the algebraic and the graphical notation keep track of such pairs. Actually, it is
enough to specify the target or the source, since one of them determines the other. The
only exception to such a rule will be in the case of $x = y = 1_i$, when we will write
simply $\alpha_\pi^{\smash{1_i}}\!: H_{1_i} \diam H_\pi \to H_\pi$ and
$(\alpha_\pi^{\smash{1_i}})': H_\pi \diam H_{1_i}\! \to H_\pi$, where $H_\pi$ stays for
$H_\pi^{1_i}$. In particular, in the expression $\alpha_\pi^{\smash{1_i}}\!: H_{1_i} \diam
H_{\pi^z} \to H_{\pi^z}$, $\pi^z$ should be interpreted as an element in $\seq\G$, and
$H_{\pi^z}$ stays for $H_{\pi^z}^{1_i}$.}

\begin{proposition}\label{alpha/thm}
The adjoint morphisms $\alpha$ and $\alpha'$ defined by the identities \(q1) and \(q2)
above satisfy all the action properties in Table \ref{table-adjointdefn/fig}, for any $x
\in \G(i_0,j_0)$, $y \in \G(i_0,k_0)$, $z \in \G(i_0,l_0)$ and $\pi \in \seq \G$, with
arbitrary composable $g,h \in \G$ and $i \in \Obj\G$. In particular, $\alpha^{1_i}_\pi$
(resp. $(\alpha^{1_i}_\pi)'\,$) makes $H_\pi$ into a left (resp. right) $H_{1_i}$-module,
for any $i \in \Obj \G$.
\end{proposition}

\begin{proof}
Up to symmetry, the properties \(q3') to \(q6') of the morphisms $\alpha'$ are equivalent
the corresponding properties \(q3) to \(q6) of the morphisms $\alpha$, hence it suffices
to prove the latter ones.

Identity \(q3) is an immediate consequences of axioms \(a4-4') and \(a7) in Table
\ref{table-Hdefn/fig} and property \(s6) in Table \ref{table-Hprop/fig}. To prove \(q4),
we first consider the special case when $\pi = g \in \G$ in Figure \ref{adjoint01/fig}.
Then, the general case follows by the inductive argument shown in Figure
\ref{adjoint02/fig}.

\begin{Figure}[htb]{adjoint01/fig}
{}{Proof of \(q4): the case of $\pi = g \in \G$
   [{\sl a}/\pageref{table-Hdefn/fig}, {\sl s}/\pageref{table-Hprop/fig}]}
\centerline{\fig{}{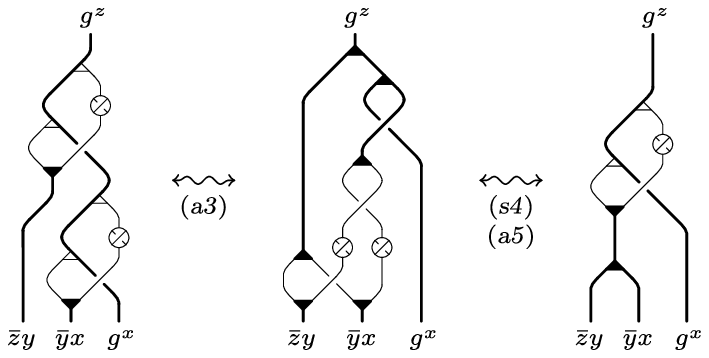}}
\vskip-3pt
\end{Figure}

\begin{Figure}[htb]{adjoint02/fig}
{}{Proof of \(q4): the inductive step 
   [{\sl a}/\pageref{table-Hdefn/fig}, 
    {\sl q}/\pageref{table-adjointdefn/fig}]}
\vskip-9pt
\centerline{\fig{}{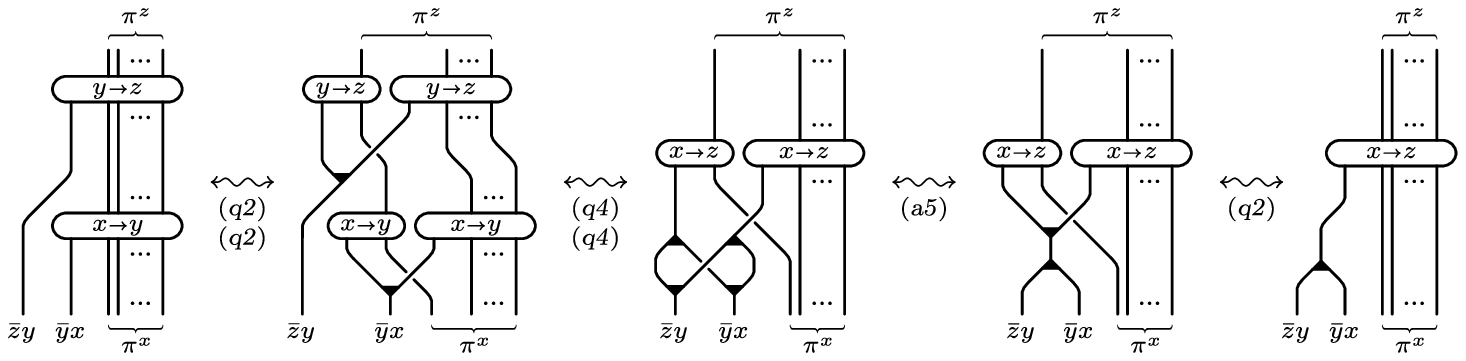}}
\vskip-3pt
\end{Figure}

\begin{Figure}[htb]{adjoint03/fig}
{}{Proof of \(q6) for $g,h \in \G(i_0,i_0)$
   ($x \in \G(i_0,j_0)$ and $y \in \G(i_0,k_0)$) 
   [{\sl a-s}/\pageref{table-Hdefn/fig}]}
\vskip-9pt
\centerline{\fig{}{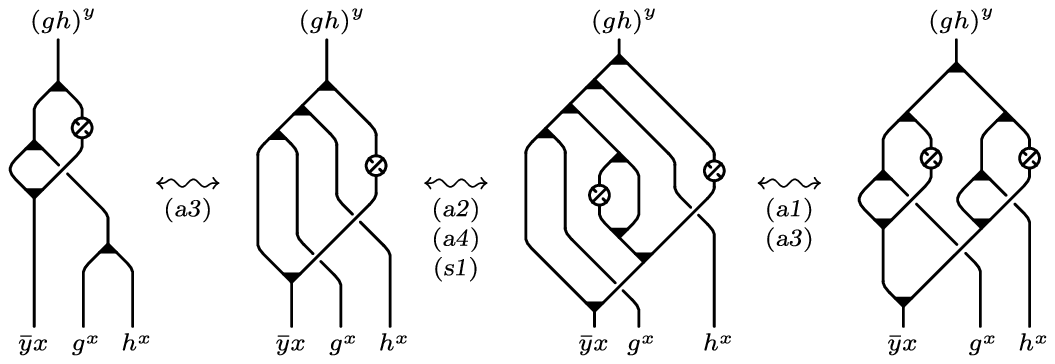}}
\vskip-3pt
\end{Figure}

Identity \(q5) is trivial for $i \neq i_0$, while it follows from axioms \(a4) and \(s1')
in Table \ref{table-Hdefn/fig} for $i = i_0$. Finally, identity \(q6) is considered in
Figure \ref{adjoint03/fig}, in the case when $g,h \in \G(i_0,i_0)$. For the other
cases, one can just delete the proper edges from the diagrams in that figure.
\end{proof}

\begin{remark}\label{adjoint/rem}
In the case when $\G = \G_1$ is the trivial groupoid, the unique left (resp. right) 
adjoint morphism $\alpha = \alpha_1^{(1,1)}$ (resp. $\alpha' = (\alpha_1^{(1,1)})'$) in 
Definition \ref{alpha/def} gives the left (resp. right) adjoint action of $H = H_(1,1)$ on 
itself. In the general case, according to what we said after the definition, the
left (resp. right) adjoint morphisms $\alpha_g^{x,y}$ (resp. $(\alpha_g^{x,y})'$) 
can be thought in a certain sense to act on the product $H \times \G$ of the Hopf 
$\G$-algebra $H = \{H_g\}_{g \in \G}$ with the groupoid $\G$ itself.
\end{remark}

Before going on, we prove the properties of the adjoint morphisms listed in Table 
\ref{table-adjointprop/fig}. These will be used to define the functor $\_^x: \H^r(\G) \to 
\H^r(\G)$ and the natural transformation $\xi^{x,y}$ which relates $\_^x$ and $\_^y$ for 
composable $\bar y$ and $x$.

\begin{Table}[htb]{table-adjointprop/fig}
{}{}
\vskip6pt
\centerline{\fig{}{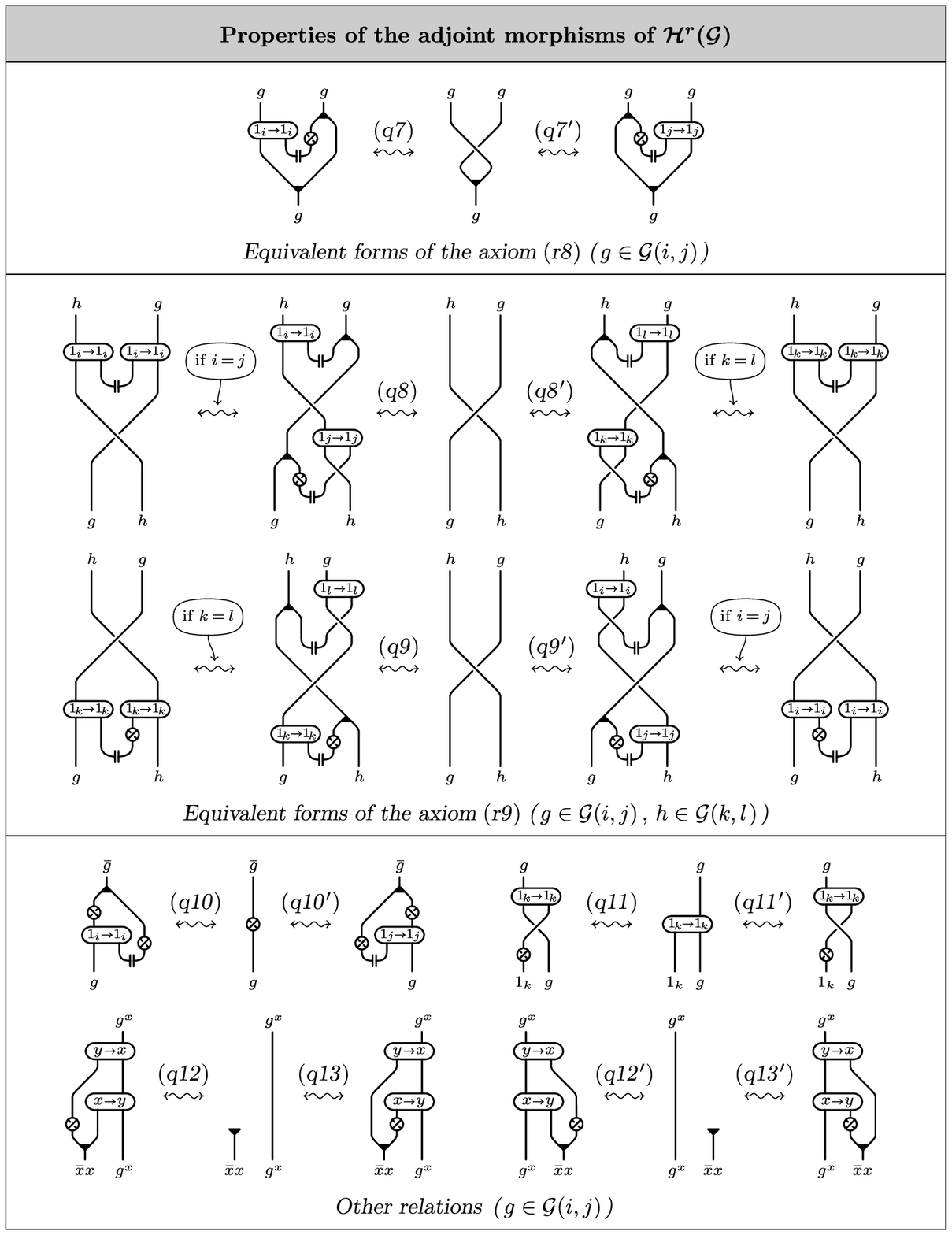}}
\vskip-3pt
\end{Table}

\begin{proposition}\label{adjoint/thm}
The properties in Table \ref{table-adjointprop/fig} hold in $\H^r(\G)$. Moreover, modulo
the other axioms of $\H^r(\G)$, \(r8) is equivalent to either \(q7) or \(q7'), while \(r9) 
is equivalent to either \(q8), \(q8'), \(q9) or \(q9').
\end{proposition}

\begin{proof}
We can limit ourselves to prove properties \(q7) to \(q13), since the corresponding 
primed properties are the symmetric of them under the functor $\sym$. Actually, \(q11')
becomes the symmetric of \(q11) after composition with the isomorphism $\gamma_{g,1_k} 
\circ S_{1_k}$.

Property \(q12) is proved in Figure \ref{adjoint04/fig}, and \(q13) can be proved in a
similar way, by using \(s1') instead of \(s1). We observe that the proof of \(q4), and
therefore that of \(q12), uses only the axioms of $\H^u_v(\G)$.

\begin{Figure}[htb]{adjoint04/fig}
{}{Proof of \(q12)
   [{\sl q}/\pageref{table-adjointdefn/fig}, {\sl s}/\pageref{table-Hdefn/fig}]}
\centerline{\fig{}{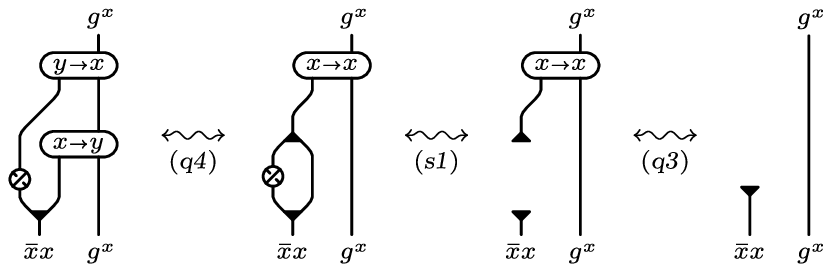}}
\vskip-3pt
\end{Figure}

Relation \(q7) is obtained in the top line of Figure \ref{adjoint05/fig}. The bottom line
of the same figure shows that \(q7) implies \(r8) modulo the relations of $\H^u_v(\G)$,
including \(q13') (see above). Actually, the two moves \(r7') in the figure occur only
when both the involved copairings are non-trivial, that is when $i = j$.

\begin{Figure}[htb]{adjoint05/fig}
{}{Equivalence between \(q7) and \(r8)
   ($g \in \G(i,j)$)
   [{\sl a}/\pageref{table-Hdefn/fig},
    {\sl f}/\pageref{table-Hu/fig}-\pageref{table-Huvprop/fig},
    {\sl p}/\pageref{table-Huvprop/fig}--\pageref{table-Hr/fig},
    {\sl q}/\pageref{table-adjointdefn/fig}-\pageref{table-adjointprop/fig},
    {\sl r}/\pageref{table-Huvdefn/fig}-\pageref{table-Huvprop/fig},
    {\sl s}/\pageref{table-Hdefn/fig}-\pageref{table-Hprop/fig}]}
\centerline{\fig{}{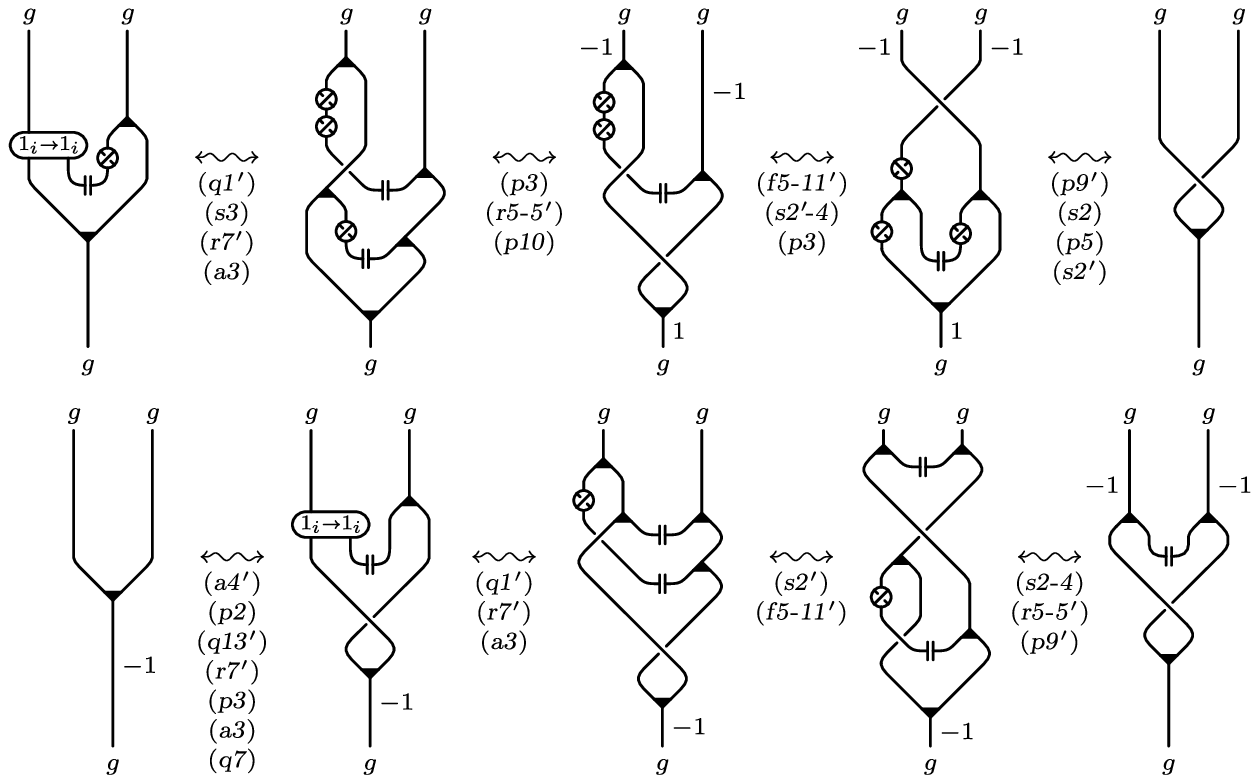}}
\vskip-3pt
\end{Figure}

Then, we can derive \(q11) and \(q10) in the order, as described in Figures
\ref{adjoint06/fig} and \ref{adjoint07/fig} respectively. In steps four and six of the
former figure relation \(t2) is used instead of \(p14') and \(p15') when $i \neq k$, while
in the latter it is used instead of \(p15) when $i \neq j$.

\begin{Figure}[htb]{adjoint06/fig}
{}{Proof of \(q11)
   ($g \in \G(i,j)$)
   [{\sl a}/\pageref{table-Hdefn/fig}, {\sl p}/\pageref{table-Hr/fig},
    {\sl q}/\pageref{table-adjointdefn/fig}-\pageref{table-adjointprop/fig},
    {\sl r}/\pageref{table-Huvprop/fig}, {\sl s}/\pageref{table-Hprop/fig},
    {\sl t}/\pageref{table-Hr/fig}]}
\centerline{\fig{}{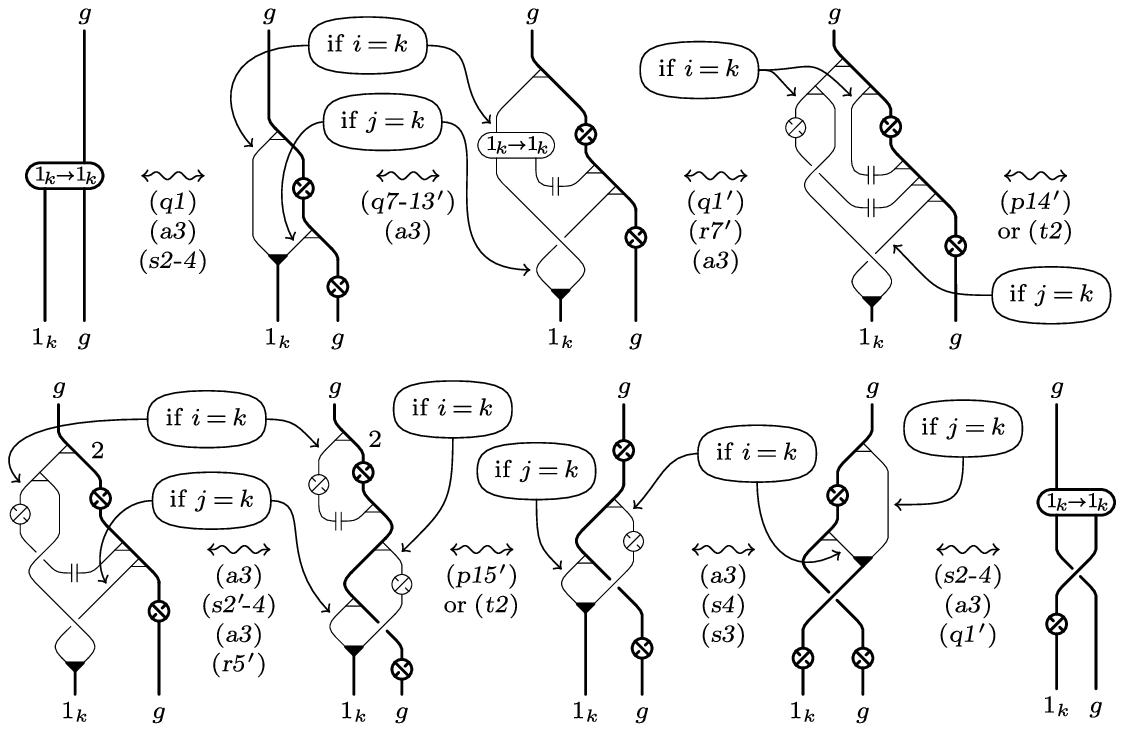}}
\vskip-3pt
\end{Figure}

\begin{Figure}[htb]{adjoint07/fig}
{}{Proof of \(q10)
   ($g \in \G(i,j)$)
   [{\sl a}/\pageref{table-Hdefn/fig},
    {\sl p}/\pageref{table-Huvprop/fig}-\pageref{table-Hr/fig},
    {\sl q}/\pageref{table-adjointdefn/fig}-\pageref{table-adjointprop/fig},
    {\sl r}/\pageref{table-Huvdefn/fig}-\break\pageref{table-Huvprop/fig},
    {\sl s}/\pageref{table-Hprop/fig}]}
\vskip6pt
\centerline{\fig{}{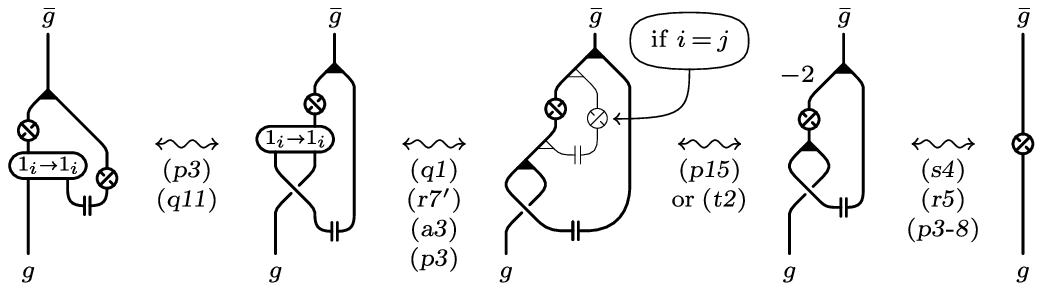}}
\vskip-3pt
\end{Figure}

\begin{Figure}[b]{adjoint08/fig}
{}{Equivalence between \(r9) and \(q8) 
   ($g \in \G(i,j)$)
   [{\sl a}/\pageref{table-Hdefn/fig},
    {\sl f}/\pageref{table-Hu/fig}-\break\pageref{table-Huvprop/fig},
    {\sl p}/\pageref{table-Huvprop/fig},
    {\sl q}/\pageref{table-adjointdefn/fig}-\pageref{table-adjointprop/fig},
    {\sl r}/\pageref{table-Huvdefn/fig}-\pageref{table-Huvprop/fig},
    {\sl s}/\pageref{table-Hdefn/fig}-\pageref{table-Hprop/fig}]}
\vskip-12pt
\centerline{\fig{}{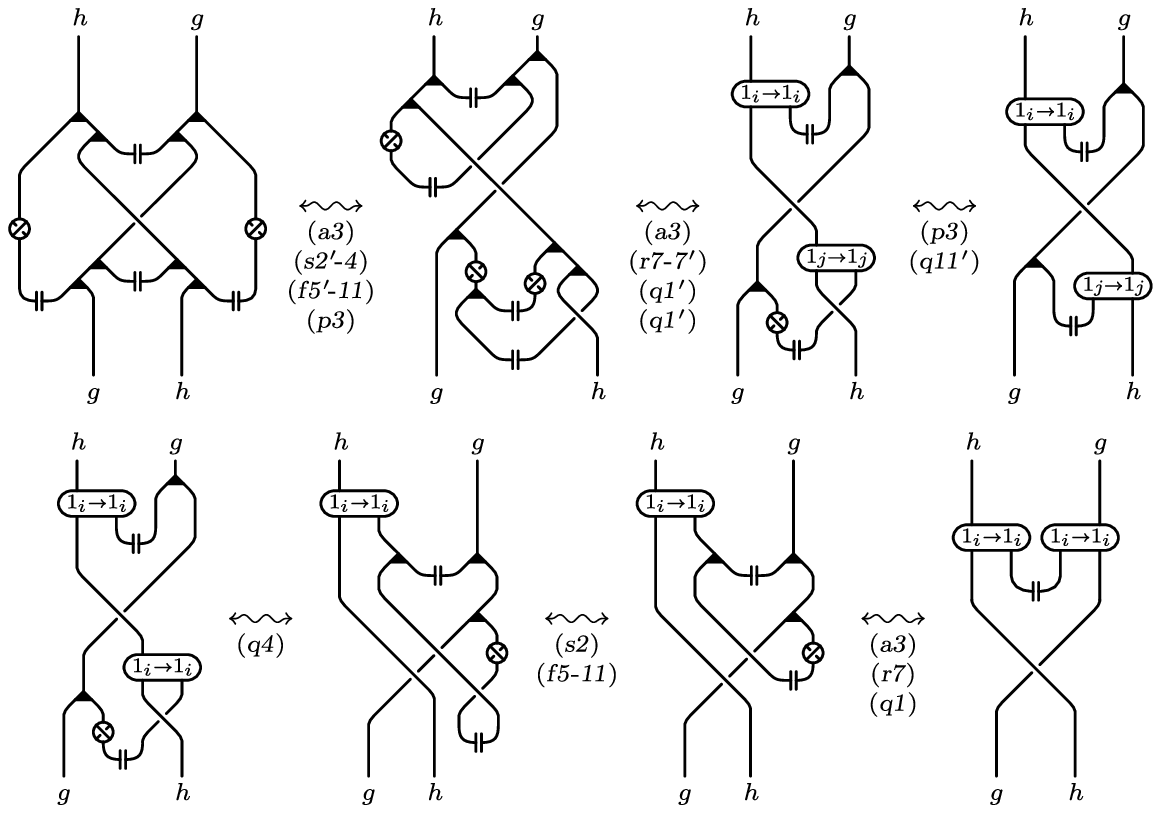}}
\vskip-3pt
\end{Figure}

Finally, we are left with \(q8) and \(q9). A straightforward application of \(q12-12') and
\(q13-13') shows that the diagrams in \(q9) represent the inverse morphisms of those
represented by the diagrams in \(q8), which gives the equivalence between \(q8) and \(q9).
Then, it suffices to verify that \(r9) is equivalent to \(q8) modulo the axioms of
$\H^u_v(\G)$. This is done in the top line of Figure \ref{adjoint08/fig}. Like above, the
moves \(r7-7') occur only when both the involved copairings are non-trivial, that is when
the corresponding optional edges in the resulting $(\alpha_h^{\smash{1_i}})'$ and
$(\alpha_h^{\smash{1_{j}}})'$ are present. In the bottom line of the same figure we obtain
the special form for $i = j$, starting from the last diagram in the top line.
\end{proof}

Now we will use the left adjoint morphisms to define the algebraic analog of the family of
isomorphisms $\xi_\pi^{x,y}$, introduced in Lemma \ref{K-push/thm}. Like in the case of
the category of Kirby tangles, such isomorphisms will eventually define a natural
transformation between the functors $\_^x$ and $\_^y$.

\begin{definition}\label{xinat/def}
Given $x \in \G(i_0,j_0)$ and $y \in \G(i_0,k_0)$, we define the family of morphisms
$\xi_\pi^{x,y} \in \H^r(\G)$ for any $\pi \in \seq\G$ as follows (see Figure
\ref{adjoint09/fig}): $$\xi_\pi^{x,y} = (\id_{\bar yx} \diam \alpha_\pi^{x,y}) \circ
(\Delta_{\bar y x} \diam \id_{\pi^x}): H_{\bar y x} \diam H_\pi^x \to H_{\bar y x} \diam
H_\pi^y\,.$$ We will use the simplified notation $\xi_\pi^x$ for $\xi_\pi^{x,1_{i_0}}\!:
H_x \diam H_\pi^x \to H_x \diam H_\pi$.
\end{definition}

\begin{Figure}[htb]{adjoint09/fig}
{}{}
\vskip-3pt
\centerline{\fig{}{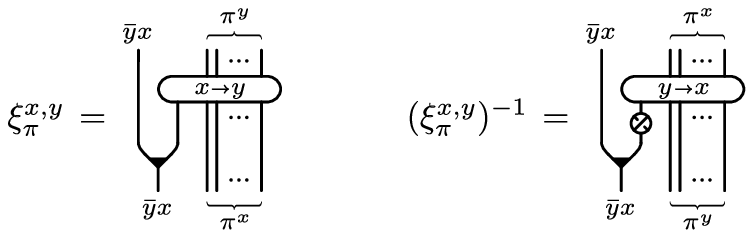}}
\vskip-3pt
\end{Figure}

\begin{proposition}\label{xinat/thm}
For any $x \in \G(i_0,j_0)$, $y \in \G(i_0,k_0)$ and $\pi \in \seq\G$ the mor\-phism 
$\xi_\pi^{x,y}$ is invertible and its inverse is (see Figure \ref{adjoint09/fig}):
$$(\xi_\pi^{x,y})^{-1} = (\id_{\bar yx} \diam \alpha_\pi^{y,x}) \circ (\id_{\bar y x} 
\diam S_{\bar y x} \diam \id_{\pi^y}) \circ (\Delta_{\bar y x} \diam \id_{\pi^y}): 
H_{\bar y x} \diam H_\pi^y \to H_{\bar y x} \diam H_\pi^x\,.$$
\end{proposition}

\begin{proof}
This is a direct consequence of properties \(q12-13) in Table \ref{table-adjointprop/fig}.
\end{proof}

The next proposition makes into a formal statement the idea that the isomorphisms
$\xi^{x,y} \in \Mor \H_n^r$ given by the above definition for $\G = \G_n$ are nothing else
than the algebraic counterpart, under the func\-tors $\Phi_n$ defined in Section
\ref{Phi/sec}, of the homonymous isomorphisms $\xi^{x,y} \in \Mor \K_n$ introduced in
Section \ref{K/sec}.

\begin{proposition}\label{xinat-kirby/thm}
Given $x = (i_0,j_0)$ and $y = (i_0,k_0)$ in $\G_n$, for any $\pi \in \seq\G_n$ 
we have that $\Phi_n(H_\pi^x) = I_\pi^x$ and $\Phi_n(\xi_\pi^{x,y}) = \xi_\pi^{x,y}$
(cf. Lemma \ref{K-push/thm}).
\end{proposition}

\begin{proof}
The identity on the objects follows immediately from the definition of $\Phi_n$. For the
one on the morphisms, a straightforward verification shows that
$$\xi_{\pi' \diam \pi''}^{x,y} = (\xi_{\pi'}^{x,y} \diam 
  \xi_{\pi''}^{x,y}) \circ (\id_{\bar y x} \diam \gamma_{\bar y x, \pi'^x} \diam 
  \id_{\pi''^x}) \circ (\Delta_{\bar y x} \diam \id_{\pi'^x \diam \pi''^x})$$
holds in both the categories $\H^r_n$ and $\K_n$. Hence, it suffices to consider the
elementary case of $\xi_{(i,j)}^{x,y}$ with $(i,j) \in \G_n$. We leave to the reader to
check that $\Phi_n(\xi_{(i,j)}^{x,y}) = \xi_{(i,j)}^{x,y}$ holds for all possible
combinations of indices, but Figure \ref{adjoint10/fig} should be helpful.
\end{proof}

\begin{Figure}[htb]{adjoint10/fig}
{}{($x = (i_0,j_0)$, $y = (i_0,k_0)$, $g = (i_0,i_0)$, $h = (i_0,i)$ with $i \neq i_0$)}
\vskip-3pt
\centerline{\fig{}{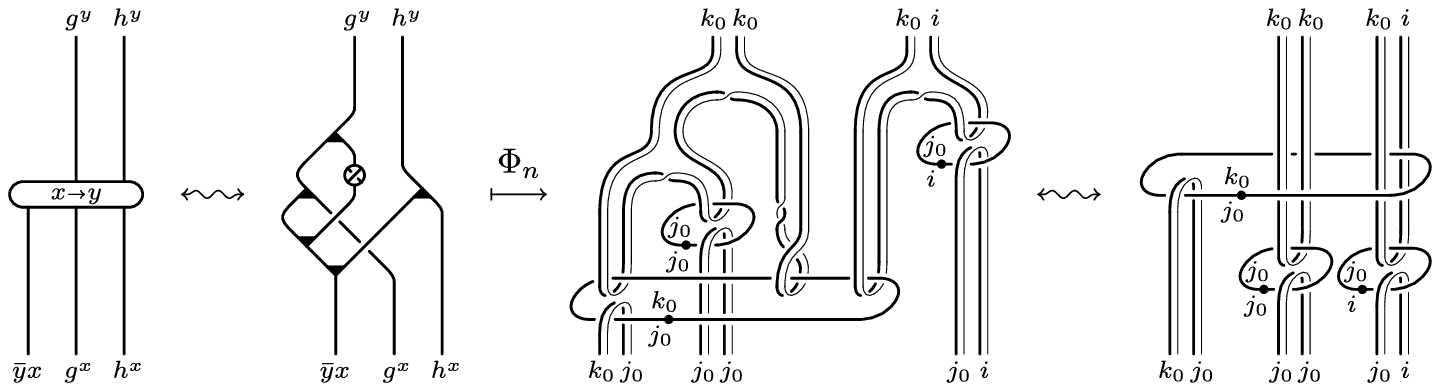}}
\vskip-3pt
\end{Figure}

At this point, we can proceed with the main result of this section, which is the algebraic
analog of Lemma \ref{K-push/thm}.

\medskip

As the first step, given $x \in \G(i_0, j_0)$, we define $F^x$ for any elementary morphism
$F$ of $\H^r(\G)$. The most obvious definition, consisting in the formal extension of
$\_^x: \G \to \G$ with $\_^x$ acting on the indices, works for $i_0 = j_0$ but it runs
into problems for $i_0 \neq j_0$, due to the fact that in the latter case the copairing
$\sigma_{i_0, j_0}$ is trivial, while $\sigma_{i_0^x, j_0^x} = \sigma_{j_0, j_0}$ is not.
Thus, some corrections are needed. Since we want to have $\Phi_n(F)^x = \Phi_n(F^x)$ for
any morphism $F$ in $\H^r_n$, the reader can realize the nature of these corrections by
looking at the $\Phi_n(F)^x$'s in Figures \ref{kirby-stab10/fig}, \ref{kirby-stab11/fig} 
and \ref{kirby-stab12/fig} (cf. Lemma \ref{K-push/thm} and the discussion following it).

\begin{definition}\label{Fx/def}
We define $\_^x$ on the elementary morphisms of $\H^r(\G)$ by putting, for any $i,j,k,l
\in \Obj \G$, $g \in \G(i,j)$ and $h \in \G(k,l)$, with $j = k$ for $m_{g,h}$:
$$\arraycolsep0pt
\begin{array}{c}
  (\eta_i)^x = \eta_{i^x};
  \hfill (m_{g,h})^x = m_{g^x\!,h^x};
  \hfill (\epsilon_g)^x = \epsilon_{g^x};
  \hfill (l_i)^x = l_{i^x};
  \hfill (L_g)^x = L_{g^x};
  \hfill (v_g)^x = v_{g^x};
\\[10pt]
(\Delta_g)^x = \left\{%
  \begin{array}{lr}
    \mu^{-1}_{g^x\!,g^x} \!\circ \Delta_{g^x} 
      = (S_{\bar g^x} \!\circ v_{\bar g^x}^{-1}) \diam (S_{\bar g^x} 
        \!\circ v_{\bar g^x}^{-1}) \circ \Delta_{\bar g^x} 
        \!\circ \bar S_{g^x} \!\circ v_{g^x}
      & \ \ \text{if } i \neq i_0 = j,\\[2pt]
    \Delta_{g^x} &\quad\text{otherwise};
  \end{array}\right.\\[18pt]
(S_g)^x = \left\{%
  \begin{array}{lr}
    m_{\bar g^x\!,1_i} \!\circ (S_{g^x} \diam \id_{1_i}) \circ \rho_{g^x\!,i}
      = \bar S_{g^x} \!\circ v_{g^x}^2 
      & \ \ \text{if } i \neq i_0 = j,\\[2pt]
    S_{g^x} &\text{otherwise};
  \end{array}\right.
\\[15pt]
(\bar S_g)^x = \left\{%
  \begin{array}{lr}
    m_{\bar g^x\!,1_i} \!\circ (\bar S_{g^x} \diam S_{1_i}) \circ \rho_{g^x\!,i}    
      = S_{g^x} \!\circ v_{g^x}^{-2} 
      & \ \ \text{if } i = i_0 \neq j,\\[2pt]
    \bar S_{g^x} &\text{otherwise}.
 \end{array}\right.\\[18pt]
(\gamma_{g,h})^x = \left\{%
  \begin{array}{lr}
    \gamma_{g^x\!,h^x} 
      &\ \ \text{if } i = i_0 = j,\\[6pt]
    ((\alpha^{x,x}_h \!\circ \bar\gamma_{h^x\!,1_{j_0}}) \diam \id_{g^x}) \circ
      (\id_{h^x} \diam \rho_{j_0,g^x}) \circ 
      \gamma_{g^x\!,h^x}
      &\ \ \text{if } i \neq i_0 = j,\\[6pt]
    \gamma_{g^x\!,h^x} \!\circ 
      (\id_{g^x} \diam (\alpha^{x,x}_h \!\circ (S_{1_{j_0}} \!\!\diam \id_{h^x}))) \circ
      (\rho_{g^x\!,j_0} \!\diam \id_{h^x})
      &\ \ \text{if } i = i_0 \neq j,\\[6pt]
    ((\alpha^{x,x}_h \!\circ \bar\gamma_{h^x\!,1_{j_0}}) \diam \id_{g^x}) \circ
      (\id_{h^x} \diam \rho_{j_0,g^x}) \circ \gamma_{g^x\!,h^x} \!\circ {}\\[2pt]
    \hfill {} \circ 
      (\id_{g^x} \diam (\alpha^{x,x}_h \!\circ (S_{1_{j_0}} \!\!\diam \id_{h^x}))) \circ
      (\rho_{g^x\!,j_0} \!\diam \id_{h^x})
      &\ \ \text{if } i \neq i_0 \neq j;
  \end{array}\right.\\[50pt]
(\bar\gamma_{g,h})^x = \left\{%
  \begin{array}{lr}
    \bar\gamma_{g^x\!,h^x} 
      & \ \ \text{if } i = i_0 = j,\\[6pt]
    (\id_{h^x} \diam \alpha^{x,x}_g) \circ
      (\rho_{h^x\!,j_0} \!\diam \id_{g^x}) \circ
      \bar \gamma_{g^x\!,h^x}
      & \ \ \text{if } i \neq i_0 = j,\\[6pt]
    \bar\gamma_{g^x\!,h^x} \!\circ 
      ((\alpha^{x,x}_g \circ \bar\gamma_{g^x,1_{j_0}} \!\!\circ (\id_{g^x}
      \diam \bar S_{1_{j_0}})) \diam \id_{h^x}) \circ
      (\id_{g^x} \diam \rho_{j_0,h^x})
      & \ \ \text{if } i = i_0 \neq j,\\[6pt]
    (\id_{h^x} \diam \alpha^{x,x}_g) \circ
      (\rho_{h^x\!,j_0} \!\diam \id_{g^x}) \circ
      \bar \gamma_{g^x\!,h^x} \circ {}\\[2pt]
      \hfill {} \circ ((\alpha^{x,x}_g \circ \bar\gamma_{g^x,1_{j_0}} \!\!\circ 
      (\id_{g^x} \diam \bar S_{1_{j_0}})) \diam \id_{h^x}) \circ
      (\id_{g^x} \diam \rho_{j_0,h^x})
      & \ \ \text{if } i \neq i_0 \neq j;
  \end{array}\right.
\end{array}$$

Then, according to relation \(r6) in Table \ref{table-Huvdefn/fig}, we also put:
$$\label{sigmax/par}
\arraycolsep0pt
(\sigma_{i,j})^x = \left\{%
  \begin{array}{lr} 
  \eta_{j_0} \!\diam \eta_{j_0} 
    & \ \ \text{if } \{i, j\} = \{i_0, j_0\} \text{ and } i_0 \neq j_0,\\[2pt]
  \sigma_{i^x\!,j^x} &\text{otherwise}.
\end{array}\right.
$$\vskip-12pt
\end{definition}

The diagrams of $(\Delta_g)^x$, $(S_g)^x$ and $(\bar S_g)^x$ are presented in Figure
\ref{adjoint11/fig}, while those of $(\gamma_{g,h})^x$ and $(\bar \gamma_{g,h})^x$ are
presented in Figure \ref{adjoint12/fig} (here, the antipodes have been moved by \(p3) just 
for pictorial convenience). The equivalences in Figure \ref{adjoint12/fig}
can be obtained by moves symmetric to those in the top line of Figure
\ref{adjoint08/fig} performed in the reversed order.

\begin{Figure}[htb]{adjoint11/fig}
{}{$(\Delta_g)^x$, $(S_g)^x$ and $(\bar S_g)^x$ ($g \in \G(i,j)$ and $x \in \G(i_0,j_0)$) 
   [{\sl p}/\pageref{table-Huvprop/fig}-\pageref{table-Hr/fig}]}
\centerline{\fig{}{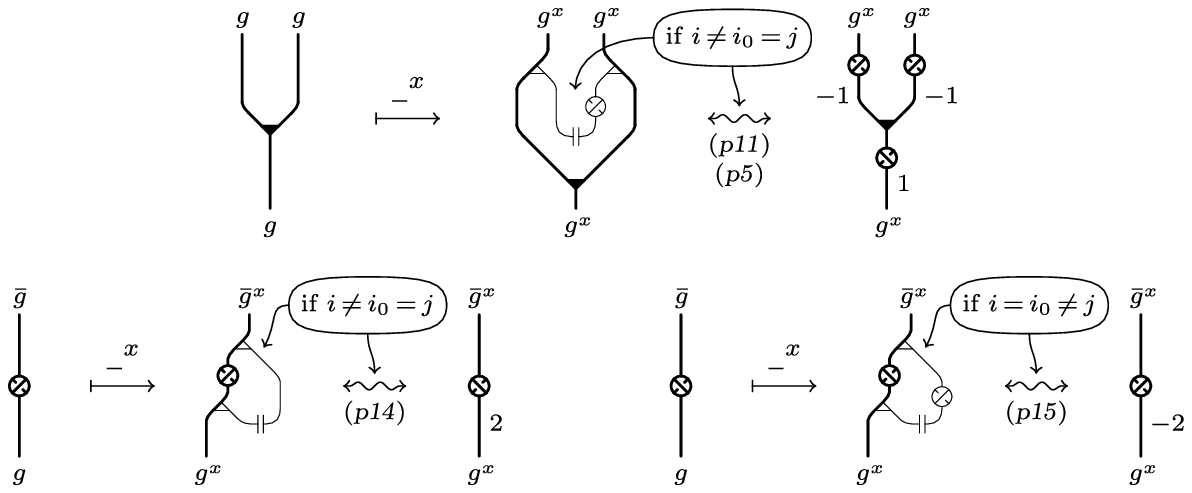}}
\vskip-3pt
\end{Figure}

\begin{Figure}[htb]{adjoint12/fig}
{}{$(\gamma_{g,h})^x$ and $(\bar\gamma_{g,h})^x$ 
   ($g \in \G(i,j)$, $h \in \G(k,l)$ and $x \in \G(i_0,j_0)$)
   [{\sl f}/\pageref{table-Hu/fig}-\pageref{table-Huvprop/fig},
    {\sl q}/\pageref{table-adjointdefn/fig},
    {\sl r}/\pageref{table-Huvdefn/fig}-\pageref{table-Huvprop/fig},
    {\sl s}/\pageref{table-Hdefn/fig}-\pageref{table-Hprop/fig}]}
\centerline{\kern15pt\fig{}{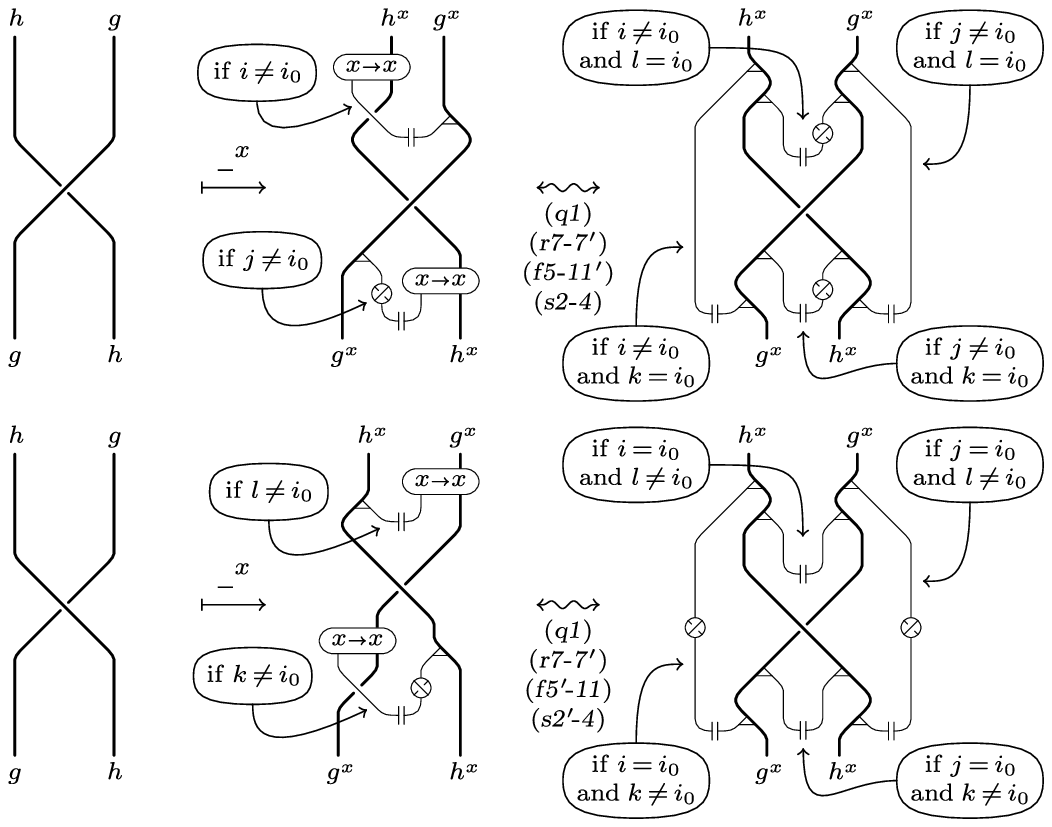}}
\vskip-6pt
\end{Figure}

We observe that all the corrections to the formal extension of $\_^x: \G \to \G$ to the
elementary morphisms are inessential when $i_0 = j_0$, being the involved copairings
trivial in this case. Moreover, if $i_0 \neq j_0$ then all conditions requiring that
certain index is $\neq i_0$ (resp. $\neq j_0$) can be replaced by $= j_0$ (resp. $= i_0$),
because only in that case the copairing involved is non-trivial. The reason for applying
the corrections even in some inessential cases, like we did with crossing changes when
defining $\_^x$ for Kirby tangles in Section \ref{K/sec}, is that this simplifies some
proofs in the following. In particular, the proof of property \(q15) in Table
\ref{table-Fx/fig} below for $i_0 = j_0$, which is far from being trivial in spite of the
triviality of $\_^x$ as a formal extension, will split into subcases of other cases.

It is also worth noticing that the righthand diagrams in Figures \ref{adjoint11/fig} and
\ref{adjoint12/fig} are not symmetric, because of the conditions controlling the presence
of the optional edges. In other words, our definition of $\_^x$ on the elementary
morphisms does not commute with the symmetry functor $\sym$.

\begin{Table}[htb]{table-Fx/fig}
{}{}
\vskip3pt
\centerline{\fig{}{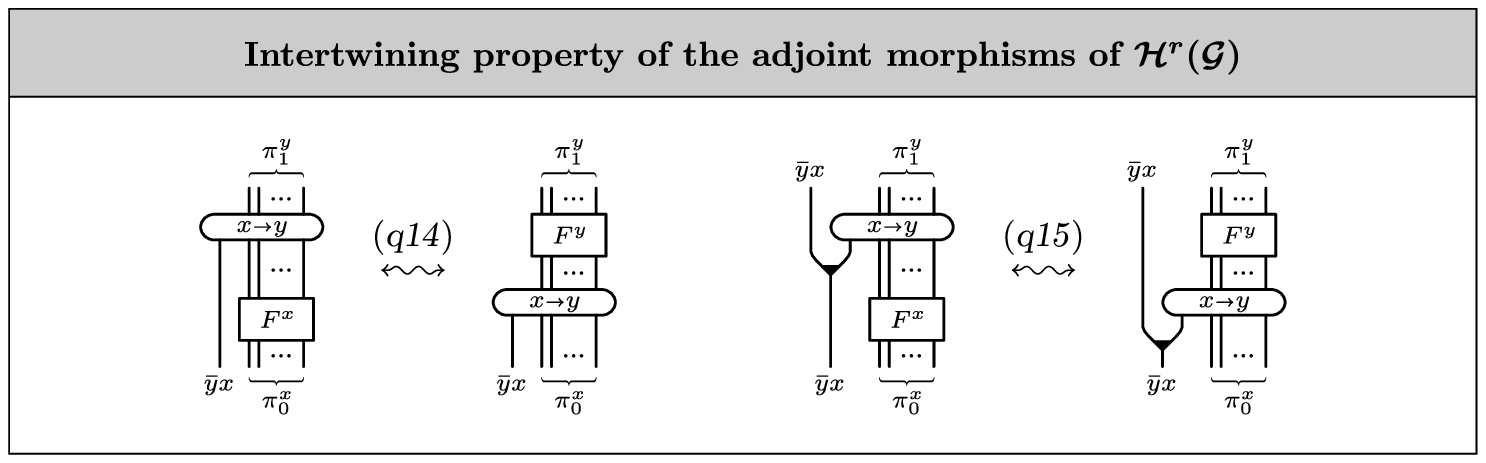}}
\vskip-3pt
\end{Table}

\begin{lemma}\label{intertwining/thm}
Given $x \in \G(i_0,j_0)$ and $y \in \G(i_0,k_0)$, the properties \(q14) and \(q15) shown
in Table \ref{table-Fx/fig} hold for $F$ being any elementary morphism of $\H^r(\G)$.
\end{lemma}

\begin{proof}
We can limit ourselves to the proof of \(q14), since \(q15) can be obtained from it, 
essentially by composing with $\Delta_{\bar y x} \diam \id_{\pi_0^x}$.

For $F = \eta_i$ and $F = m_{g,h}$ \(q14) coincides with action properties \(q5) and \(q6)
respectively, and was already established in Proposition \ref{alpha/thm}. Moreover, the
identity \(q14) for $F = \sigma_{i,j}$ follows from the ones for the other elementary
morphisms, being the definition of $(\sigma_{i,j})^x$ based on \(r6).

\begin{Figure}[b]{natural-xi01/fig}
{}{Proof of \(q14) for $F = l_{i_0}$
   ($x \in \G(k_0,i_0)$ and $y \in \G(j_0,k_0)$)
   [{\sl a}/\pageref{table-Hdefn/fig}, {\sl i-f}/\pageref{table-Hu/fig},
    {\sl s}/\pageref{table-Hdefn/fig}-\pageref{table-Hprop/fig}]}
\centerline{\fig{}{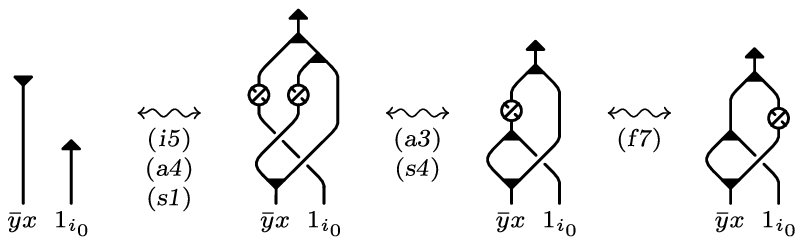}}
\vskip-3pt
\end{Figure}

Before considering the rest of the elementary morphisms, we notice that, when $F$ is
invertible, the identity \(q14) for $F$ implies the one for $F^{-1}$, once we have checked
that $(F^{-1})^x = (F^x)^{-1}$. This is trivially true for $F = S_g$, while it can be
easily verified for $F = \gamma_{g,h}$ by using moves \(r7-7') and \(q12-13) in Tables
\ref{table-Huvdefn/fig}, \ref{table-Huvprop/fig} and \ref{table-adjointprop/fig}.

Hence, we are reduced to proving \(q14) for $F = l_i,\, \epsilon_g,\, L_g,\, v_g,\,
\Delta_g,\, S_g,\, \gamma_{g,h}$ with $i \in \Obj\G$, $g \in \G(i,j)$ and $h \in \G(k,l)$,
and any $x \in \G(i_0,j_0)$ and $y \in \G(i_0,k_0)$.

\smallskip

{\bf $F = l_i$.}
If $i \neq i_0$ there is nothing to prove, being $\xi^{x,y}_{1_i}$ trivial and $F^x = F^y
= F$. The case when $i = i_0$ is shown in Figure \ref{natural-xi01/fig}.


{\bf $F = \epsilon_g, L_g, v_g$.}
The statements follow respectively from \(a6) in Table \ref{table-Hdefn/fig}, \(i2-2') in
Table \ref{table-Hprop/fig} and \(r5-5') in Table \ref{table-Huvdefn/fig} and 
\ref{table-Huvprop/fig}, modulo the relations \(a2-2') in Table \ref{table-Hdefn/fig} and 
\(s5) in Table \ref{table-Hprop/fig} in the first two cases.

\smallskip

\begin{Figure}[b]{natural-xi02/fig}
{}{Proof of \(q14) for $F = \Delta_g$ with $g \in \G(i,i_0)$ and $i \neq i_0$
   ($x \in\break \G(i_0,j_0)$ and $y \in \G(i_0,k_0)$)
   [{\sl a}/\pageref{table-Hdefn/fig}, 
    {\sl r}/\pageref{table-Huvdefn/fig}-\pageref{table-Huvprop/fig},
    {\sl s}/\pageref{table-Hprop/fig}]}
\centerline{\fig{}{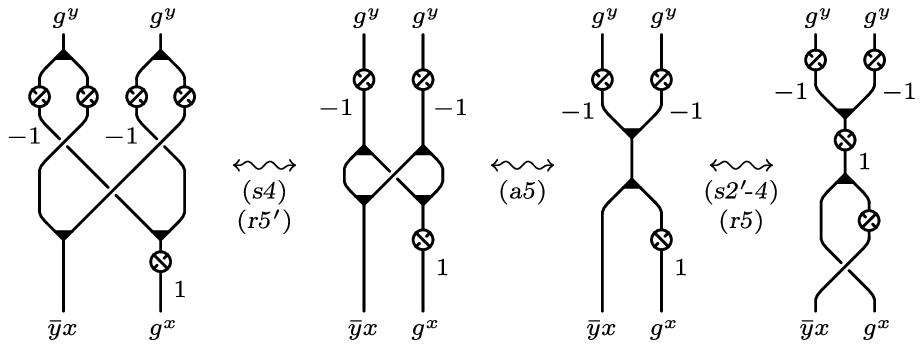}}
\vskip-3pt
\end{Figure}

\begin{Figure}[htb]{natural-xi03/fig}
{}{Proof of \(q14) for $F = \Delta_g$ with $g \in \G(i_0,i_0)$
   ($x \in \G(i_0,j_0)$ and $y \in \G(i_0,k_0)$)
   [{\sl a}/\pageref{table-Hdefn/fig}, {\sl p}/\pageref{table-Huvprop/fig},
    {\sl q}/\pageref{table-adjointdefn/fig}-\pageref{table-adjointprop/fig}, 
    {\sl s}/\pageref{table-Hdefn/fig}-\pageref{table-Hprop/fig}]}
\centerline{\fig{}{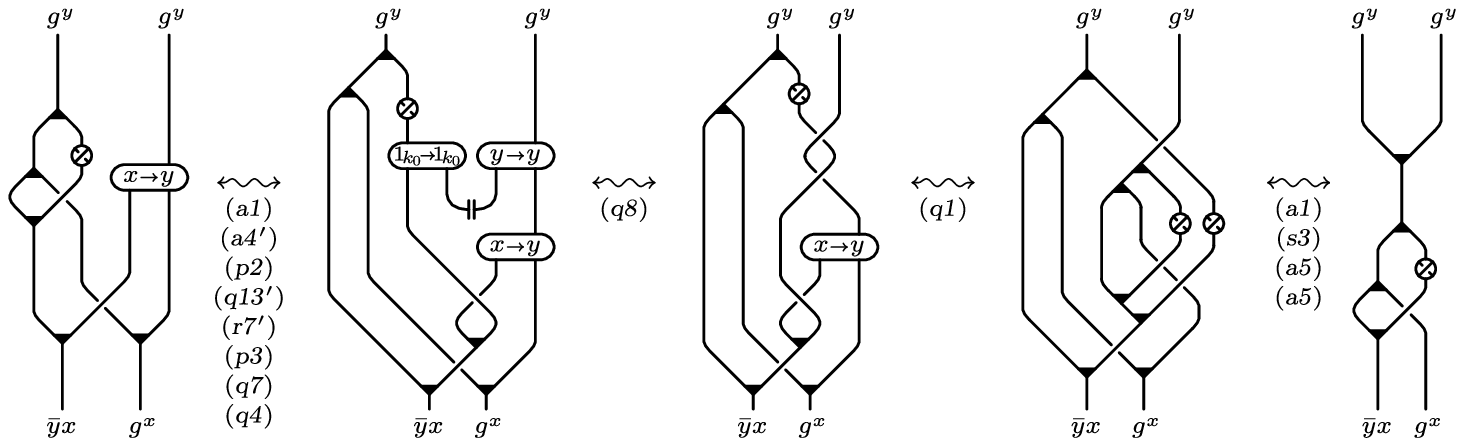}}
\vskip-3pt
\end{Figure}

\begin{Figure}[htb]{natural-xi04/fig}
{}{Proof of \(q14) for $F = S_g$ with $g \in \G(i_0,i_0)$
   ($x \in \G(i_0,j_0)$ and $y \in \G(i_0,k_0)$)
   [{\sl a}/\pageref{table-Hdefn/fig}, 
    {\sl f}/\pageref{table-Hu/fig}-\pageref{table-Huvprop/fig},
    {\sl q}/\pageref{table-adjointprop/fig}, 
    {\sl s}/\pageref{table-Hprop/fig}]}
\centerline{\fig{}{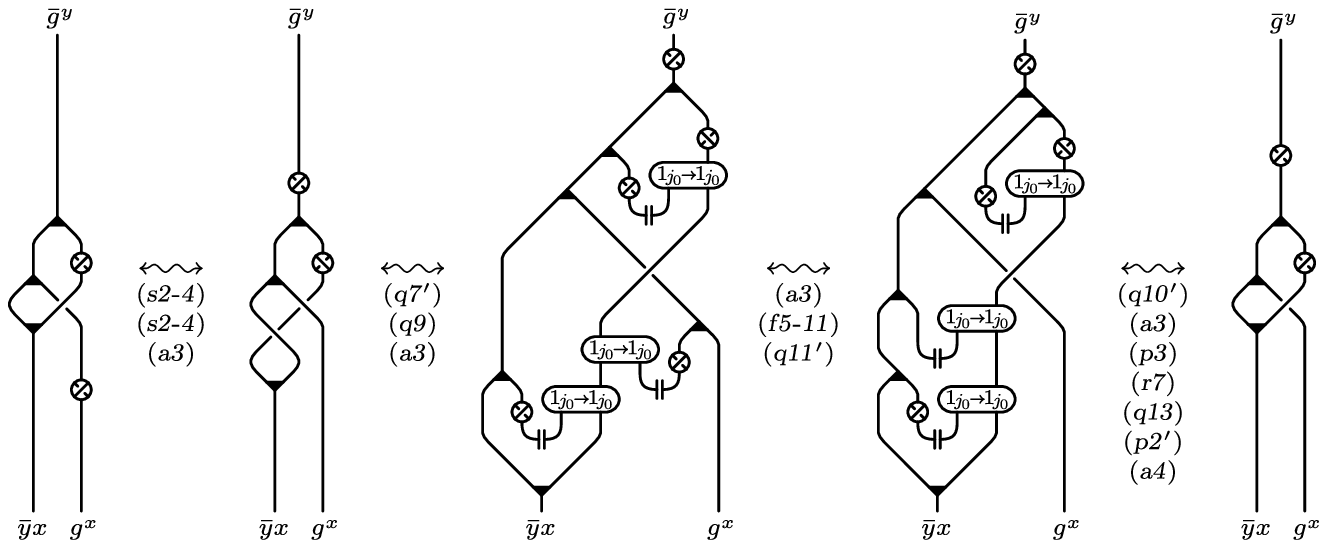}}
\vskip-3pt
\end{Figure}

{\bf $F = \Delta_g$.}
As above, there is nothing to prove for $g \in \G(i,j)$ with $i,j \neq i_0$. For $g \in
\G(i_0,j)$ with $j \neq i_0$, the statement essentially reduces to the relation \(a5) in
Table \ref{table-Hdefn/fig}. The proof for $g \in \G(i,i_0)$ with $i \neq i_0$ is
presented in Figure \ref{natural-xi02/fig}, while the case $g \in \G(i_0,i_0)$ is shown in
Figure \ref{natural-xi03/fig}.
\pagebreak
The first step of the latter figure goes like in the 
first step of the second line of Figure \ref{adjoint05/fig}. Moreover, before applying 
\(q8) in the second step, we replace the morphism $\alpha_g^{y,y}$ with the formally 
identical one $\alpha_g^{\smash{1_{k_0}}}$.

\smallskip

{\bf $F = S_g$.}
If $g \in \G(i,j)$ with $i,j \neq i_0$ there is nothing to prove, being $\alpha^{x,y}_g$
trivial and $(S_g)^x = (S_g)^y = S_g$. When $g \in \G(i,i_0)$ with $i \neq i_0$ or $g \in
\G(i_0,j)$ with $j \neq i_0$, the statement is equivalent to the property \(s4) of the
antipode in Table \ref{table-Hprop/fig}. Figure \ref{natural-xi04/fig} addresses the case 
of $g \in \G(i_0,i_0)$.

\smallskip

{\bf $F = \gamma_{g,h}$.}
We first prove that $(\gamma_{g,h})^x$ and $(\bar\gamma_{g,h})^x$ can be represented by
the diagrams depicted in Figure \ref{natural-xi05/fig}. Since these diagrams can be easily
seen to be inverse to one another by using the relations \(r7'), \(q4) and \(q12-13) in
Tables \ref{table-Huvprop/fig}, \ref{table-adjointdefn/fig} and
\ref{table-adjointprop/fig}, it suffices to consider $(\gamma_{g,h})^x$.

\begin{Figure}[htb]{natural-xi05/fig}
{}{Equivalent forms of $(\gamma_{g,h})^x$ and $(\bar\gamma_{g,h})^x$ 
   ($x \in \G(i_0,j_0)$)}
\vskip3pt
\centerline{\fig{}{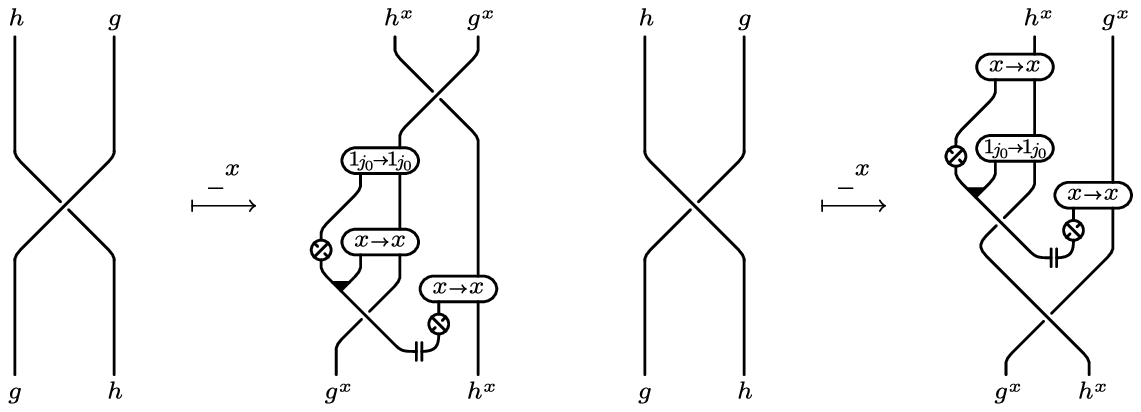}}
\vskip-3pt
\end{Figure}

Look at the diagram corresponding to $(\gamma_{g,h})^x$ in Figure \ref{natural-xi05/fig}.
When $i_0 = j_0$, we can substitute $\alpha^{\smash{1_{j_0}}\vphantom{|}}_{g^x}$ with
$\alpha^{x,x}_g$ and cancel this with the preexisting one by move \(q12) in Table
\ref{table-adjointprop/fig}. After that also $\alpha^{x,x}_h$ can be deleted to get
$\gamma_{g^x,h^x}$, which in this case is equal to $(\gamma_{h,g})^x$ as we noticed above.

\begin{Figure}[b]{natural-xi06/fig}
{}{($g \in \G(i_0,j)$ and $x \in \G(i_0,j_0)$)
   [{\sl a}/\pageref{table-Hdefn/fig}, 
    {\sl q}/\pageref{table-adjointdefn/fig}, 
    {\sl s}/\pageref{table-Hdefn/fig}-\pageref{table-Hprop/fig}]}
\vskip3pt
\centerline{\fig{}{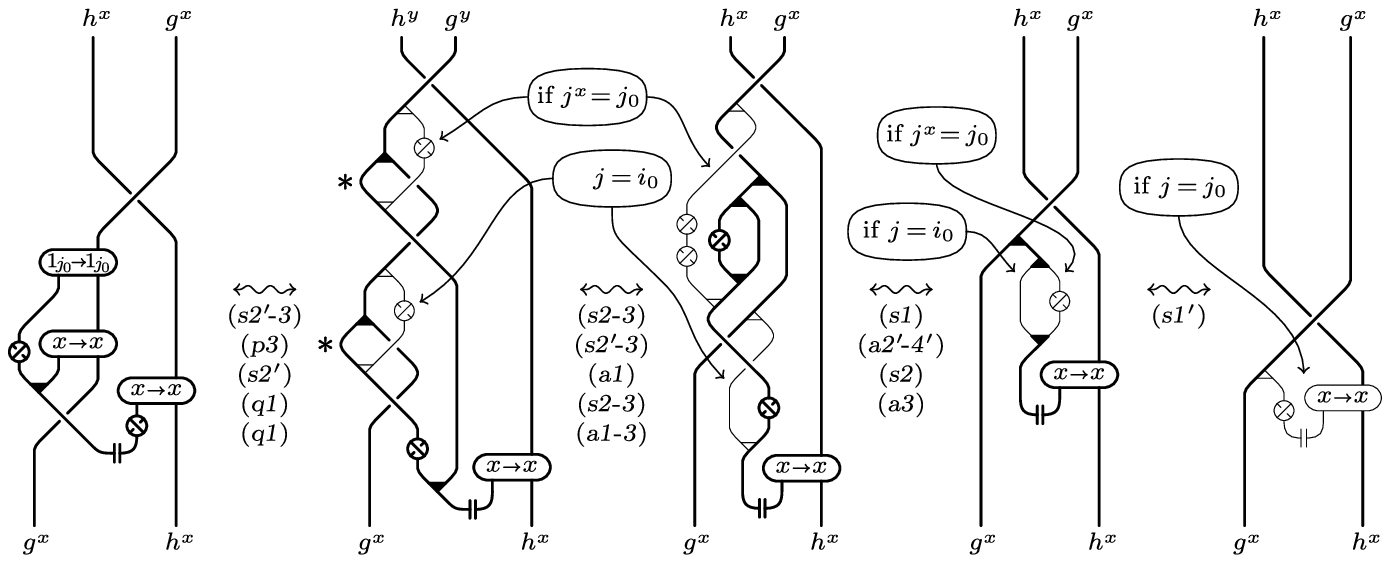}}
\vskip-3pt
\end{Figure}

Now assume $i_0 \neq j_0$. We have already observed that in this case the condition $i
\neq i_0$ (resp. $j \neq i_0$) in the definition of $(\gamma_{g,h})^x$ (cf. Figure
\ref{adjoint12/fig}) is equivalent to $i = j_0$ (resp. $j = j_0$). Then, Figures
\ref{natural-xi06/fig} and \ref{natural-xi07/fig} deal with the cases $g \in \G(i_0,j)$
and $g \in \G(j_0,j)$ respectively. For $g \in \G(i,j)$ with $i \neq i_0,j_0$, we refer
again to Figure \ref{natural-xi06/fig}, noting that the second diagram in the
figure is directly isotopic to the fourth one, once the edges marked with an asterisk in 
the expansions of the $\alpha$'s are deleted.

\begin{Figure}[htb]{natural-xi07/fig}
{}{($g \in \G(j_0,j)$ and $x \in \G(i_0,j_0)$)
   [{\sl a}/\pageref{table-Hdefn/fig},
    {\sl f}/\pageref{table-Hu/fig}-\pageref{table-Huvprop/fig},
    {\sl p}/\pageref{table-Huvprop/fig},
    {\sl q}/\pageref{table-adjointdefn/fig}-\break\pageref{table-adjointprop/fig},
    {\sl r}/\pageref{table-Huvprop/fig},
    {\sl s}/\pageref{table-Hdefn/fig}-\pageref{table-Hprop/fig}]}
\centerline{\fig{}{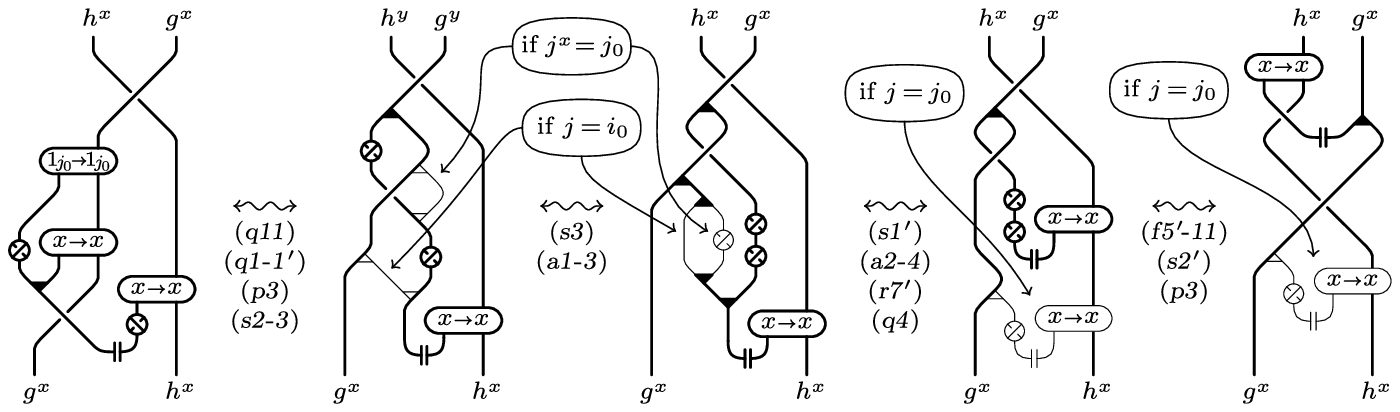}}
\vskip-3pt
\end{Figure}

\begin{Figure}[htb]{natural-xi08/fig}
{}{Proof of \(q14) for $\gamma_{g,h}$ ($x \in \G(i_0,j_0)$ and $y \in \G(i_0,k_0)$)
   [{\sl a}/\pageref{table-Hdefn/fig},
    {\sl f}/\pageref{table-Hu/fig}-\pageref{table-Huvprop/fig},
    {\sl p}/\pageref{table-Huvprop/fig}-\pageref{table-Hr/fig},
    {\sl q}/\pageref{table-adjointdefn/fig}-\pageref{table-adjointprop/fig}%
            -\pageref{table-Fx/fig},
    {\sl r}/\pageref{table-Huvdefn/fig}-\pageref{table-Huvprop/fig},
    {\sl s}/\pageref{table-Hdefn/fig}-\pageref{table-Hprop/fig}]}
\vskip6pt
\centerline{\fig{}{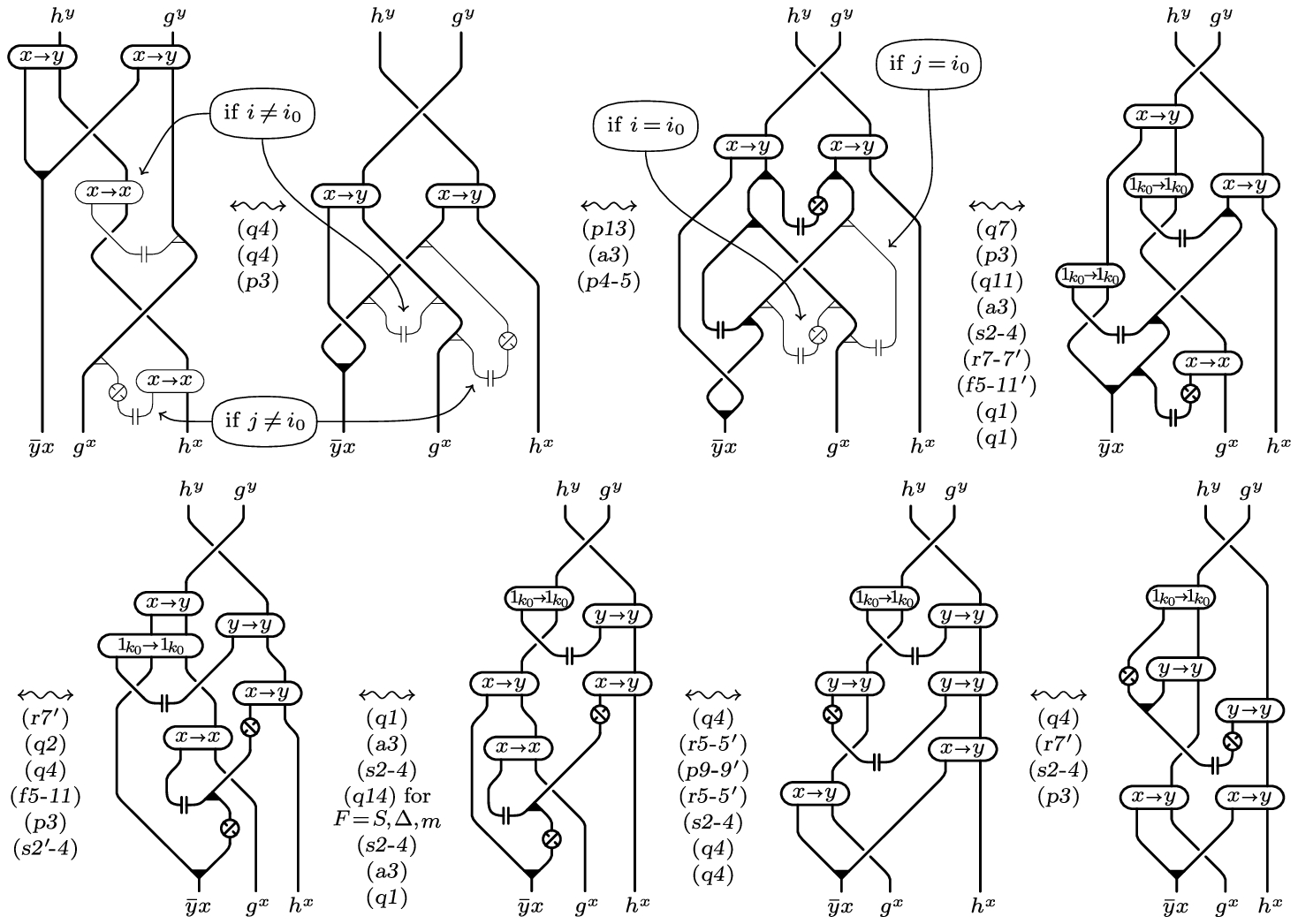}}
\vskip-3pt
\end{Figure}

\pagebreak\vglue0pt
At this point, identity \(q14) for $\gamma_{g,h}$ is proved in Figure
\ref{natural-xi08/fig}, where $(\gamma_{g,h})^x$ in the first diagram has the original
form, while $(\gamma_{g,h})^y$ in the last diagram has the equivalent form given above. We
note that the top line ends with two applications of \(q1) to get $\alpha^{x,x}_g$ and
$\alpha^{\smash{1_{k_0}}\vphantom{|}}_{g^x}$ in the fourth diagram. The conditions of the
corresponding optional edges in the third diagram are explicitly indicated for
$\alpha^{x,x}_g$, while they are implicitly provided by the copairings occurring in them
for $\alpha^{\smash{1_{k_0}}\vphantom{|}}_{g^x}$. Moreover, we observe that the second
diagram in the bottom line is obtained from the first one, by moving $\alpha^{x,y}_g$ down
and letting it pass through $\alpha^{\smash{1_{k_0}}\vphantom{|}}_{\bar yx}$. This can
be done thanks to the identity \(q14) we have already proved for the elementary morphisms
$S$, $\Delta$ and $m$, once $\alpha^{x,y}_g$ has been transformed as in the first step of
Figure \ref{adjoint06/fig}.
\end{proof}

Below, $\G^{\bs i}$ denotes the full subgroupoid of $\G$ with $\Obj \G^{\bs i} = \Obj \G -
\{i\}$ (cf. Proposition \ref{red-groupoid/thm}) and $\H^r(\G^{\bs i}) \subset \H^r(\G)$ is
the universal ribbon Hopf algebra constructed on $\G^{\bs i}$, with the inclusion given by
Proposition \ref{formalext/thm}.

\begin{proposition}\label{H-push/thm}
Let $\G$ be a groupoid. For any $x \in \G(i_0,j_0)$ the map $\_^x: \Obj \H^r(\G) \to \Obj
\H^r(\G)$ defined after Proposition \ref{red-groupoid/thm} extends to a monoidal functor
$$\_^x: \H^r(\G) \to \H^r(\G)\,,$$ 
which transforms the elementary morphisms according to in
Definition \ref{Fx/def}, and satisfies the following properties:
\begin{itemize}
\item[\(a)] 
if $i_0 \neq j_0$ then $(\H^r(\G))^x \subset \H^r(\G^{\bs i_0})$, hence there is an 
induced functor 
$$\_^x:\H^r(\G) \to \H^r(\G^{\bs i_0})\,\text{;}$$
\vskip0pt
\item[\(b)]
$\_^x$ restricts to the identity on $\H^r(\G^{\bs i_0})$ and to an equivalence of
categories $ \H^r(\G^{\bs j_0}) \to \H^r(\G^{\bs i_0})$, whose inverse is given by the
restriction of $\_^{\bar x}$ to $\H^r(\G^{\bs i_0})$;
\item[\(c)]
for any other $y \in \G(i_0,k_0)$, the isomorphisms $\xi_\pi^{x,y}$ introduced in
Definition \ref{xinat/def} (cf. Proposition \ref{xinat/thm}) give a natural equivalence
$$\xi^{x,y}: \id_{\bar yx} \diam \_^x \to \id_{\bar yx} \diam \_^y\,,$$
i.e. for any morphism $F: H_{\pi_0} \to H_{\pi_1}$ in $\H^r(\G)$ we have
(cf. Table \ref{table-Fx/fig}):
$$\xi_{\pi_1}^{x,y} \circ (\id_{\bar y x} \diam F^x) = (\id_{\bar y x} \diam F^y)
\circ \xi_{\pi_0}^{x,y}\,.
\eqno{\(q15)}$$
\vskip-10pt
\end{itemize}
\end{proposition}

\begin{proof}
We want to show that the propagation of the definition of $\_^x$ over products and
compositions is well-defined, i.e. it preserves all the axioms relating the elementary
morphisms in $\H^r(\G)$ presented in Tables \ref{table-Hdefn/fig}, \ref{table-Hu/fig},
\ref{table-Huvdefn/fig} and \ref{table-Hr/fig}. This will give us the functor
$\_^x: \H^r(\G) \to \H^r(\G)$. 

We only need to consider the case $i_0\neq j_0$, since for $i_0=j_0$ the map $\_^x$ acts
as identity on the elementary morphisms.

\begin{Figure}[b]{h-push01/fig}
{}{$(\gamma_{\pi,\pi'})^x$ and $(\bar\gamma_{\pi,\pi'})^x$ ($x \in \G(i_0,j_0)$)}
\vskip-3pt
\centerline{\fig{}{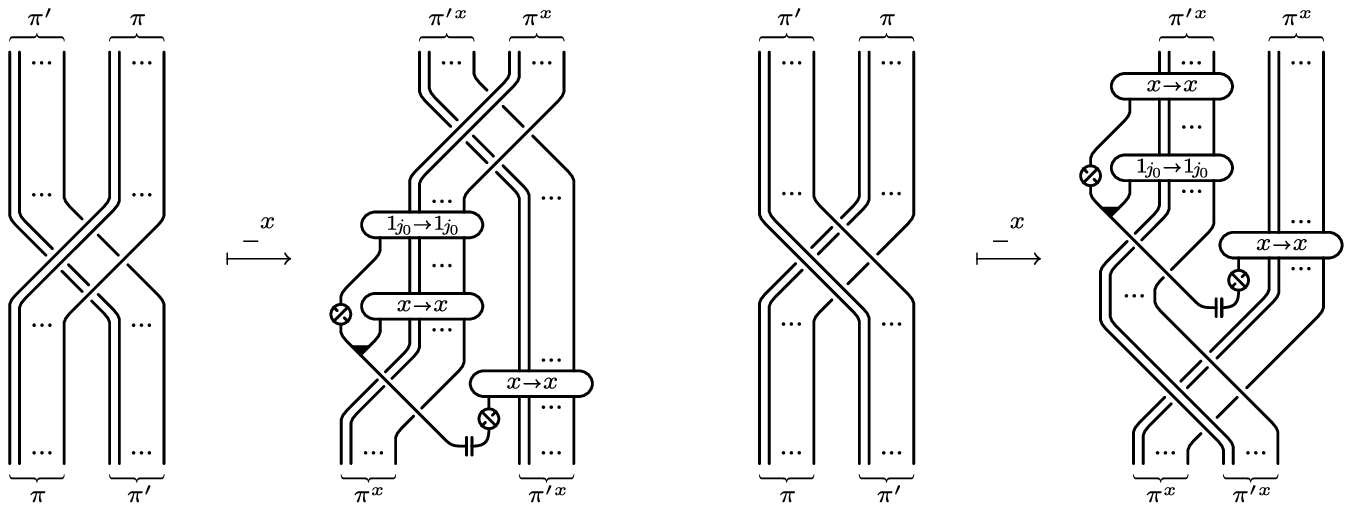}}
\vskip-3pt
\end{Figure}

Let us begin with the braid axioms in Table \ref{table-Hdefn/fig}. The preservation of
\(b1) and \(b2) is equivalent to the identity $(\gamma_{g,h}^{-1})^x =
((\gamma_{g,h})^x)^{-1}$, which was already discussed at the beginning of the proof of
Lemma \ref{intertwining/thm}. Move \(b3) is trivially preserved.

\begin{Figure}[b]{h-push02/fig}
{}{Deriving $(\bar\gamma_{(g_1,g_2),h})^x$ ($x \in \G(i_0,j_0)$)
   [{\sl a}/\pageref{table-Hdefn/fig}, {\sl p}/\pageref{table-Huvprop/fig},
    {\sl q}/\pageref{table-adjointdefn/fig}-\pageref{table-Fx/fig},
    {\sl r}/\pageref{table-Huvdefn/fig},
    {\sl s}/\pageref{table-Hdefn/fig}-\pageref{table-Hprop/fig}]}
\vskip9pt
\centerline{\fig{}{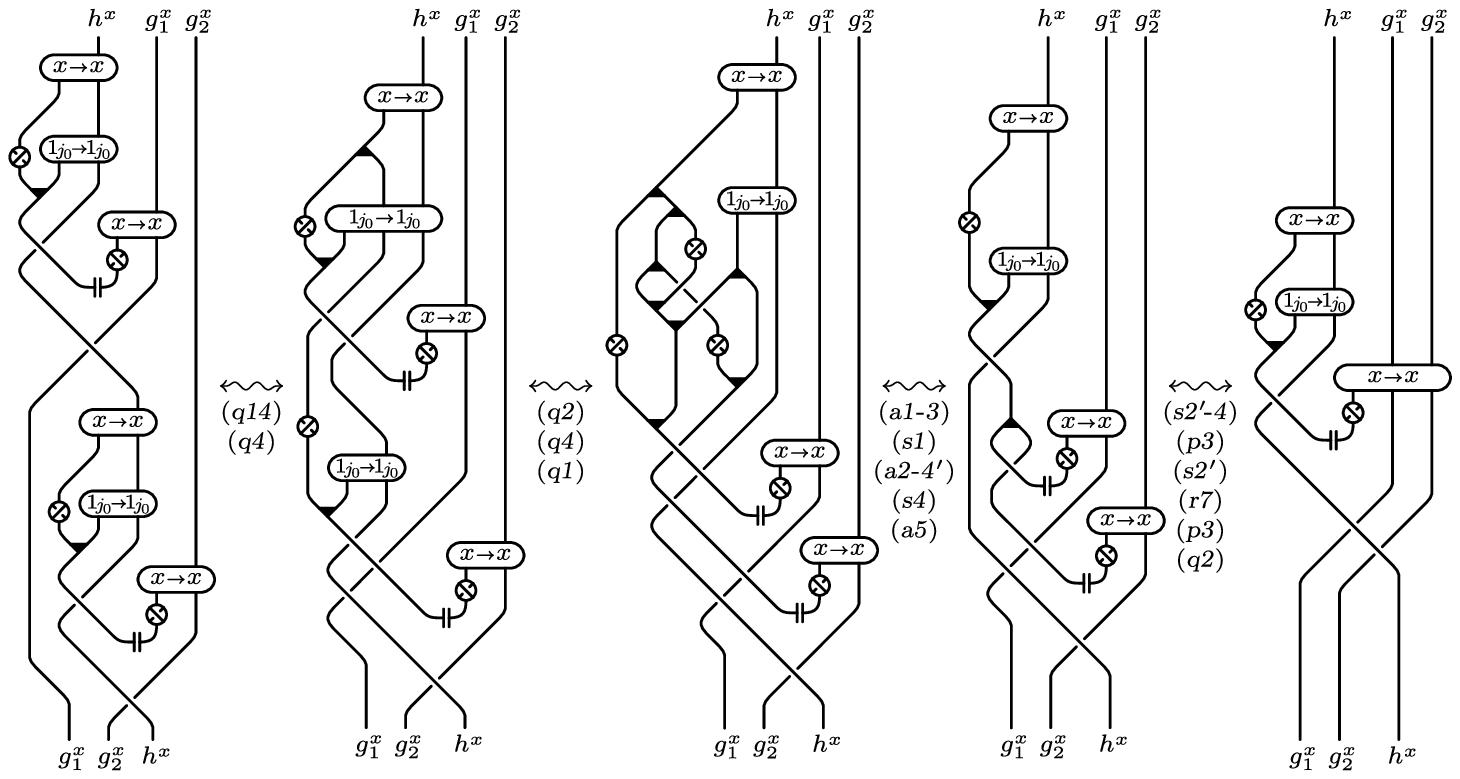}}
\vskip-3pt
\end{Figure}

\begin{Figure}[htb]{h-push03/fig}
{}{Deriving $(\bar\gamma_{g,(h_1,h_2)})^x$ ($x \in \G(i_0,j_0)$)
   [{\sl a}/\pageref{table-Hdefn/fig},
    {\sl f}/\pageref{table-Hu/fig}-\pageref{table-Huvprop/fig}
    {\sl q}/\pageref{table-adjointdefn/fig}-\pageref{table-adjointprop/fig},
    {\sl r}/\pageref{table-Huvprop/fig},
    {\sl s}/\pageref{table-Hdefn/fig}-\pageref{table-Hprop/fig}]}
\vskip9pt
\centerline{\fig{}{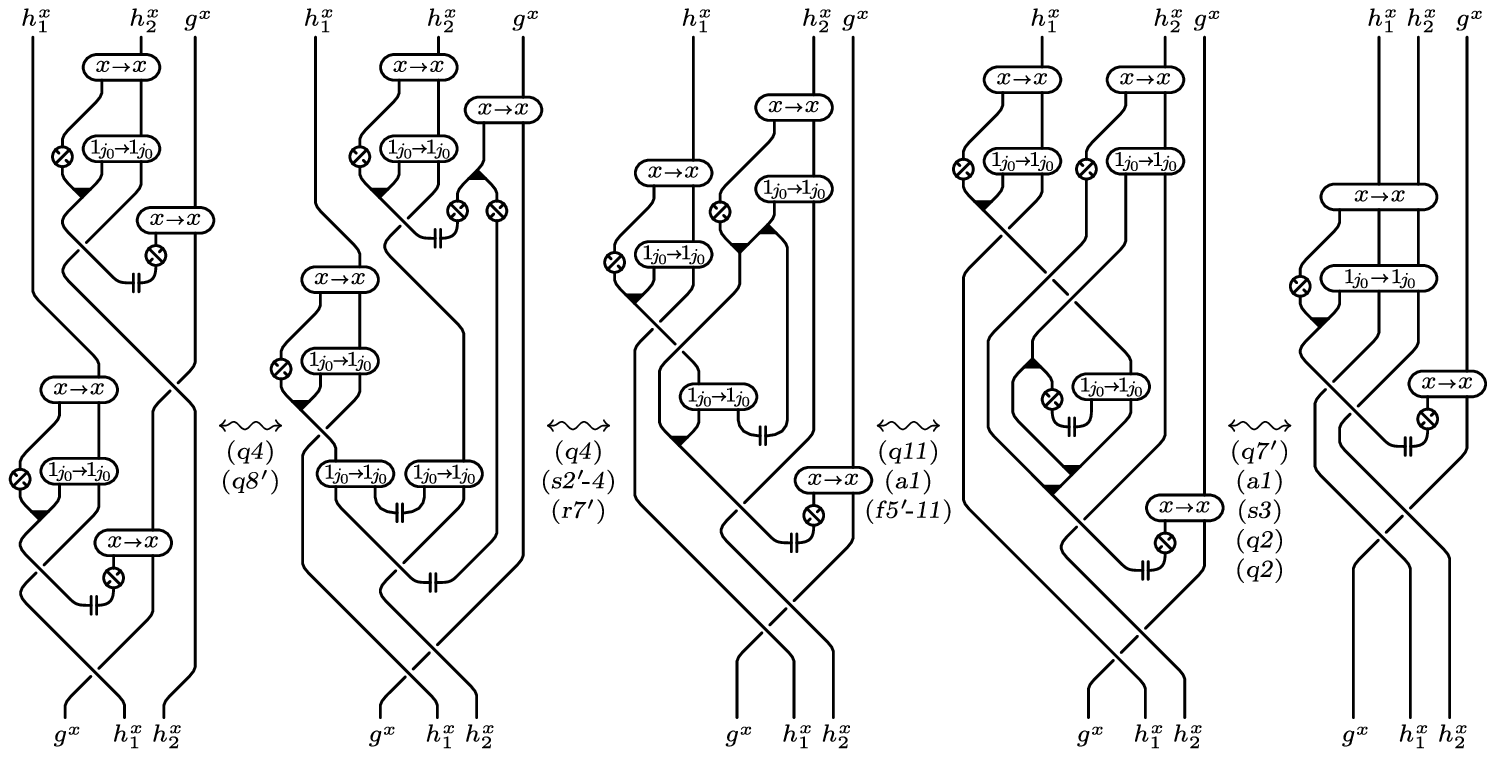}}
\vskip-3pt
\end{Figure}

To deal with moves \(b4-4'), we first prove that the alternative form provided in Figure
\ref{natural-xi05/fig} for $(\gamma_{g,h})^x$ and $(\bar\gamma_{g,h})^x$ can be
generalized to $(\gamma_{\pi,\pi'})^x$ and $(\bar\gamma_{\pi,\pi'})^x$ for any $\pi,\pi'
\in \seq\G$, as shown in Figure \ref{h-push01/fig}. Actually, only the cases when $\pi =
(g_1,g_2)$ and $\pi' = h$ or $\pi = g$ and $\pi' = (h_1,h_2)$ are needed for our purposes.
We derive these cases for $(\bar\gamma_{\pi,\pi'})^x$ in Figures \ref{h-push02/fig} and
\ref{h-push03/fig} respectively, while leaving to the reader to see that the same argument
also allows to prove the general case, by a double induction on the lengths of $\pi$ and
$\pi'$. Then, the expression for $(\gamma_{\pi,\pi'})^x$ follows, being the inverse of
that for $(\bar\gamma_{\pi,\pi'})^x$.

Now, using the form of $(\gamma_{\pi,\pi'})^x$ in Figure \ref{h-push01/fig}, the
preservation of \(b4-4') follows from the special cases of \(q14) for $\alpha^{x,x}$ and
$\alpha^{\smash{1_{j_0}}}$.

We continue with the bi-algebra axioms in Table \ref{table-Hdefn/fig}. The only
non-trivial ones are \(a1), \(a2) and \(a5) when $g$ or $h$ are in $\G(j_0,i_0)$. In this
case, \(a1) and \(a2) follow directly from the definition of $(\Delta_g)^x$ in the
rightmost diagram in Figure \ref{adjoint11/fig}. The proof of \(a5) for $g \in
\G(j_0,i_0)$ and $h \in \G(i_0,i_0)$ is presented in Figure \ref{h-push04/fig}. For
different choices of $g$ and $h$ we have analogous or simpler proofs.

\begin{Figure}[htb]{h-push04/fig}
{}{$\_^x$ preserves \(a5) for $g \in \G(j_0,i_0)$ and $h \in \G(i_0,i_0)$
   ($x \in \G(i_0,j_0)$)
   [{\sl a}/\pageref{table-Hdefn/fig},
    {\sl p}/\pageref{table-Huvprop/fig}]}
\centerline{\fig{}{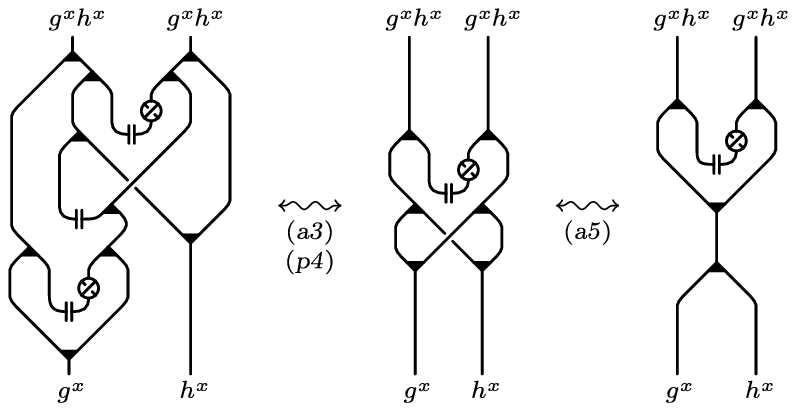}}
\vskip-3pt
\end{Figure}

From the antipode axioms in Table \ref{table-Hdefn/fig} the only non-trivial ones are
\(s1) and \(s1') for $g \in \G(j_0,i_0)$. Figure \ref{h-push05/fig} proves the
preservation of \(s1) in this case, while the proof for \(s1') is obtained by symmetry.

\begin{Figure}[htb]{h-push05/fig}
{}{$\_^x$ preserves \(s1) for $g \in \G(j_0,i_0)$ ($x \in \G(i_0,j_0)$)
   [{\sl a}/\pageref{table-Hdefn/fig},
    {\sl p}/\pageref{table-Huvprop/fig}-\break\pageref{table-Hr/fig},
    {\sl s}/\pageref{table-Hdefn/fig}]}
\centerline{\fig{}{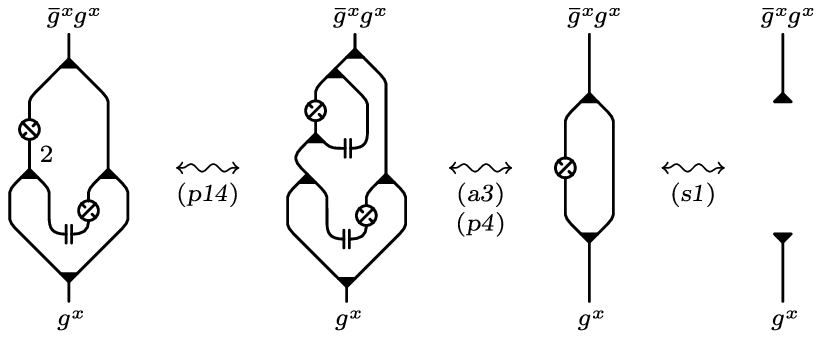}}
\vskip-3pt
\end{Figure}

The integral axioms in Table \ref{table-Hu/fig} are trivially preserved because
$(m_{g,h})^x = m_{g^x,h^x}$ and $(\Delta_{1_i})^x = \Delta_{1_i^x}$.

Also the ribbon axioms in Table \ref{table-Huvdefn/fig} are trivially preserved, while the
only non-trivial cases in Table \ref{table-Hr/fig} are \(r8) when $g \in \G(j_0,i_0)$ and
\(r9) with $g \in \G(i,j)$ and $h \in \G(k,l)$ such that some of $i,j$ are equal to $i_0$
and some of $k,l$ are equal to $j_0$ or vice versa. Figure \ref{h-push06/fig} deals with
the above-metioned case of \(r8). Some of the cases of \(r9) are presented in Figure
\ref{h-push07/fig} (cf. expression of $(\sigma_{i,j})^x$ on page \pageref{sigmax/par}) and
the others are analogous. This concludes the proof of the functoriality of $\_^x$.

\begin{Figure}[htb]{h-push06/fig}
{}{$\_^x$ preserves \(r8) when $g \in \G(j_0,i_0)$ ($x \in \G(i_0,j_0)$)
   [{\sl p}/\pageref{table-Huvprop/fig},
    {\sl r}/\pageref{table-Huvdefn/fig}-\break\pageref{table-Huvprop/fig}%
            -\pageref{table-Hr/fig}]}
\centerline{\fig{}{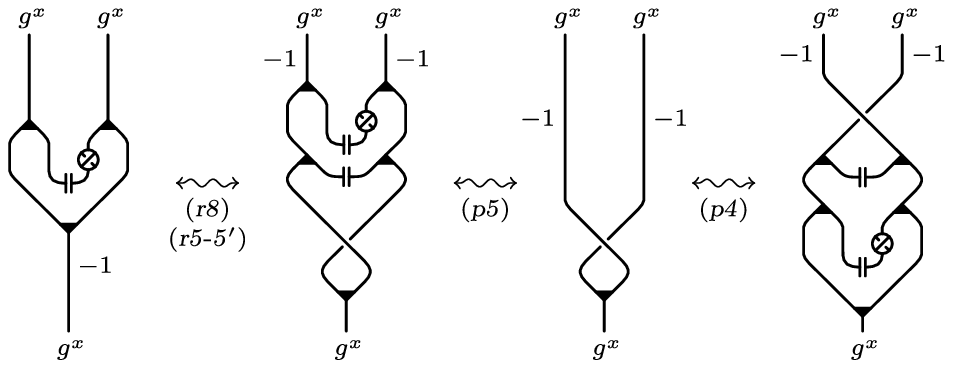}}
\vskip-3pt
\end{Figure}

\begin{Figure}[htb]{h-push07/fig}
{}{$\_^x$ preserves \(r9) ($g \in \G(i,j)$, $h \in \G(k,l)$, $x \in \G(i_0,j_0)$)
   [{\sl p}/\pageref{table-Huvprop/fig}-\break\pageref{table-Hr/fig},
    {\sl r}/\pageref{table-Hr/fig}]}
\centerline{\fig{}{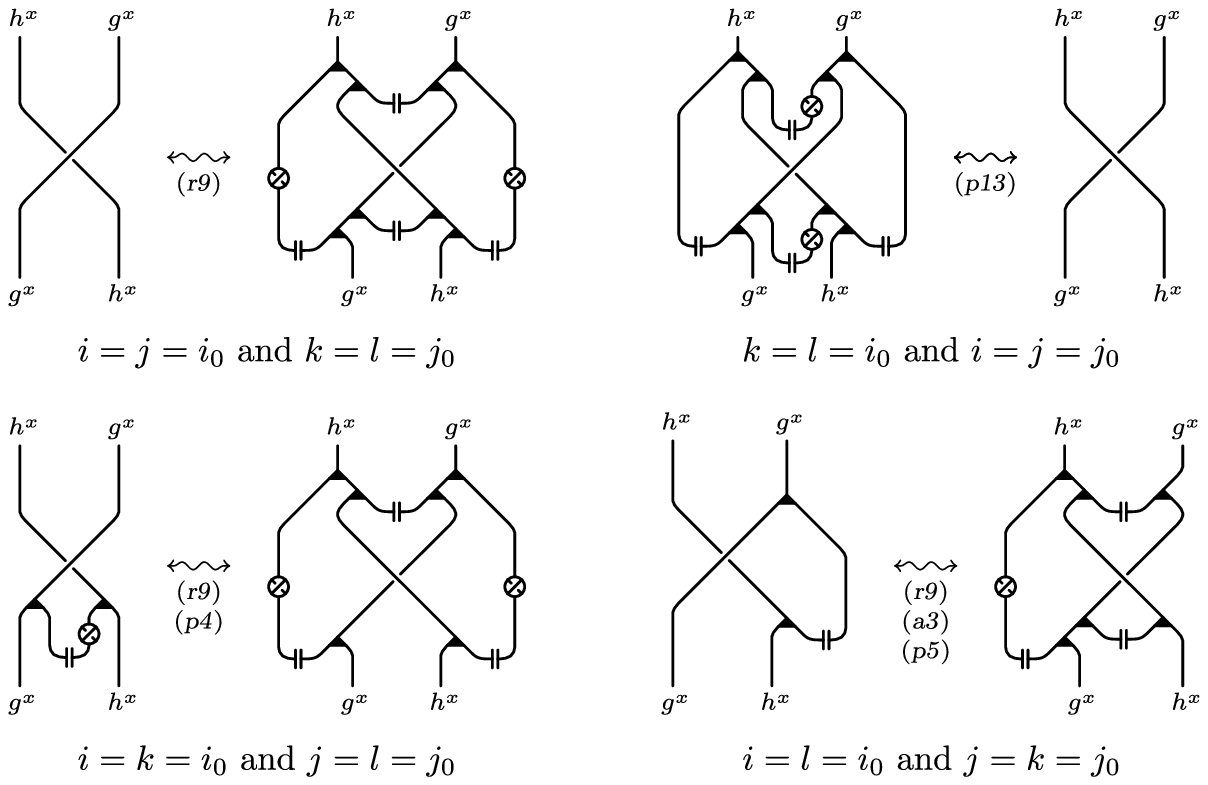}}
\vskip-3pt
\end{Figure}

At this point, thanks to the functoriality of $\_^x$, the validity of identity \(q15),
already known from Lemma \ref{intertwining/thm} for $F$ being an elementary morphism,
extends to any other morphism $F$. This gives the naturality of $\xi^{x,y}$ stated in 
point \(c).

Finally, we oserve that the definition $\_^x$ on the elementary morphisms of $\H^r(\G)$
which do not involve $i_0$ or $j_0$ coincides with the formal extension of $\_^x$ (cf.
Figures \ref{adjoint11/fig} and \ref{adjoint12/fig}). This fact together with
Proposition \ref{red-groupoid/thm} \(d) imply points \(a) and \(b) of the statement.
\end{proof}

\begin{proposition}\label{KtoH-push/thm}
For $x = (i_0,j_0) \in \G_n$, the following diagram commutes.
\vskip9pt
\centerline{\epsfbox{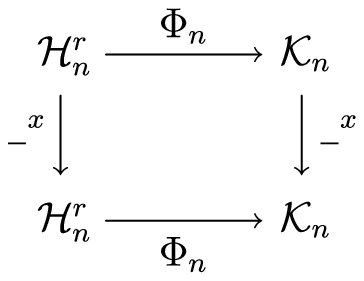}}
\vskip3pt
\end{proposition}

\begin{proof}
We remind that both functors $\_^x: \K_n \to \K_n$ and $\_^x: \H^r_n \to \H^r_n$
are the identities when $i_0 = j_0$. Moreover, when $i_0 \neq j_0$, they still leave
unchanged diagrams and Kirby tangles where no label $i_0$ occurs, while they just replace
any label $i_0$ with $j_0$ in the diagrams and Kirby tangles where no label $j_0$ occurs.

This shows that $\Phi_n(F^x) = \Phi_n(F)^x$ when $x = (i_0,i_0)$ or $F\in \H^r_n$ is
an elementary morphism which does not contain at least one of the labels $i_0$ or $j_0$.
Therefore, it remains to check such identity for elementary morphisms $F \in
\H^r_n$ that contain both $i_0$ and $j_0$ with $i_0 \neq j_0$. 

The cases $F = \Delta_{(j_0,i_0)}, S_{(j_0,i_0)}, \bar S_{(i_0,j_0)}$ can be checked
directly by comparing the $\Phi_n(F)^x$'s in Figures \ref{kirby-stab10/fig} and
\ref{kirby-stab12/fig} with the definitions of $(\Delta_{(j_0,i_0)})^x$,
$(S_{(j_0,i_0)})^x$ and $(\bar S_{(i_0,j_0)})^x$ (see Definition \ref{Fx/def} and Figure
\ref{adjoint11/fig}).

When $F = \gamma_{(i,j),(k,l)}$ or $F = \bar\gamma_{(k,l),(i,j)}$ with some of $k,l$ equal
to $i_0$, we observe that the copairings which appear in $F^x$ (see Definition
\ref{Fx/def} and Figure \ref{adjoint12/fig}) have the effect of pulling the paths labeled
$i_0$ above the ones labeled $j_0$ in $\Phi_n(F)$ (cf. Figure \ref{phi08/fig}), obtaining
this way $\Phi_n(F)^x$. For example, Figure \ref{kirby-stab11/fig} depicts the
$\Phi_n(F)^x$'s corresponding to the $F^x$'s on the left of the arrows in Figure
\ref{h-push07/fig}.

In all the other cases, no corrections occur in Definition \ref{Fx/def} of $F^x$ and also
$\Phi_n(F)^x$ differs from $\Phi_n(F)$ only for the replacement of labels $i_0$ with 
$j_0$, since $\Phi_n(F)$ admits a strictly regular planar diagram where the paths labeled 
$i_0$ pass over the ones of label $j_0$.
\end{proof}

\subsection{The stabilization and reduction functors $\up_X$ and $\down_X$%
\label{reduction/sec}}

Based on the results of the previous section, we are now ready to define the reduction
functors in the context of the universal algebraic categories. This will be done in a way
completely analogous to what we did in Section \ref{K/sec} for Kirby tangles, even if we
will continue to work in the more general framework of Hopf algebras over an arbitrary
groupoid.

We remind that to any (small) category we can associate an oriented graph whose vertices
are the objects of the category and whose edges are its morphisms/arrows. We will use the
same notation for the category and its graph.

\begin{definition}\label{spanning/def}
Let $\G \subset \G'$ be a (strict) full inclusion of groupoids such that\break the
quotient $\G'/\G$ of the graph of $\G'$ by the one of $\G$ is connected. Then a sequence
$X= (x_n, \dots, x_1) \in \seq(\G' - \G)$ will be called a {\sl spanning sequence} for the
pair $(\G',\G)$ if its image $(x_n, \dots, x_1) / \G \subset \G'/\G$ forms a spanning tree
for $\G/\G'$ with all edges oriented towards the root and $(x_i, \dots, x_1) / \G \subset
\G'/\G$ is still a tree for each $i < n$.
\end{definition}

Given a full inclusion of groupoids $\G \subset \G'$ and a spanning sequence $X$ for
$(\G',\G)$, let $\Upsilon_X = \Upsilon_\iota: \H^r(\G)\to\H^r(\G')$ denote the faithful
functor induced by the inclusion $\iota: \G \to \G'$ (cf. Proposition 
\ref{formalext/thm}).

\begin{proposition}\label{H-stabilization0/thm}
Let $\G \subset \G'$ be a full inclusion of groupoids and $X = (x_n, \dots, x_1)$ be a
spanning sequence for $(\G', \G)$. Then, the map:
\vskip-4pt
$$\up_X H_\pi = H_X \diam H_\pi \text{ for any } H_\pi \in \Obj \H^r(\G)\,,$$
$$\up_X F = \id_{X} \diam \Upsilon_X(F) \text{ for any} F \in \Mor \H^r(\G)\,,$$
\vskip4pt\noindent
defines a functor $\up_X:\H^r(\G) \to \H^r(\G')$, called the $X\!$-stabilization functor.
Moreover, if $\G'\subset \G''$ is another full inclusion of groupoids and $Y = (y_m,
\dots, y_1)$ is a spanning sequence for $(\G'', \G')$, then $\up_Y \circ \up_X = \up_{Y
\cup X}$, where $Y \cup X = (y_m, \dots, y_1, x_n, \dots, x_1)$ is a spanning sequence for
$(\G'',\G)$.
\end{proposition}\vskip-\lastskip

\begin{proof}
All statements are straightforward.
\end{proof}

We remind that the definition of the comultiplication can be extended to $\Delta_\pi:
H_\pi \to H_\pi \diam H_\pi$ for any $\pi \in \seq\G$ (see Figure \ref{H-reduction01/fig})
in the following way:
$$
\Delta_\pi = \Delta_{\pi' \diam \pi''} = 
(\id_{\pi'} \diam \gamma_{\pi',\pi''} \diam
\id_{\pi''}) \circ (\Delta_{\pi'} \diam \Delta_{\pi''})\,.
$$

\begin{Figure}[htb]{H-reduction01/fig}
{}{The comultiplication morphism $\Delta_\pi$ in $\H^r(\G)$}
\vskip-9pt
\centerline{\fig{}{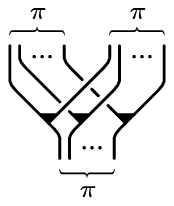}}
\vskip-3pt
\end{Figure}

\begin{definition}\label{H-reducible/def}
Given a full inclusion of groupoids $\G \subset \G'$ and a spanning sequence $X = (x_n,\,
\dots, x_1)$ for $(\G',\G)$, we say that a morphism $F: H_X \diam H_{\pi_0} \to H_X \diam
H_{\pi_1}$ in $\H^r(\G)$ is {\sl $X\!$-reducible} if it is equivalent to one in the form
$$F = (\id_X \diam G) \circ (\Delta_X \diam \id_{\pi_0})\,,$$ for some morphism $G: H_X
\diam H_{\pi_0} \to H_{\pi_1}$ in $\H^r(\G')$ (see Figure \ref{H-reduction02/fig}).

\begin{Figure}[htb]{H-reduction02/fig}
{}{The generic $X$-reducible morphism $F \in \H^r_X(\G')$}
\vskip-3pt
\centerline{\fig{}{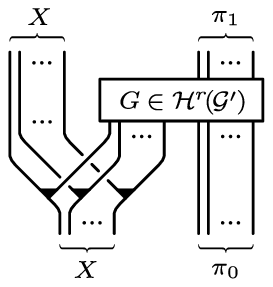}}
\vskip-3pt
\end{Figure}

The composition of two $X$-reducible morphisms is still $X$-reducible (by the
coassociativity) and we denote by $\H^r_X(\G')$ the subcategory of $\H^r(\G')$, whose
objects are $H_X \diam H_\pi$ with $H_\pi \in \Obj \H^r(\G')$ and whose morphisms are
$X$-reducible morphisms.
\end{definition}

\begin{Figure}[b]{H-reduction03/fig}
{}{The product $F \rdiam F'$ of two morphisms in $\H^r_X(\G')$}
\vskip-3pt
\centerline{\fig{}{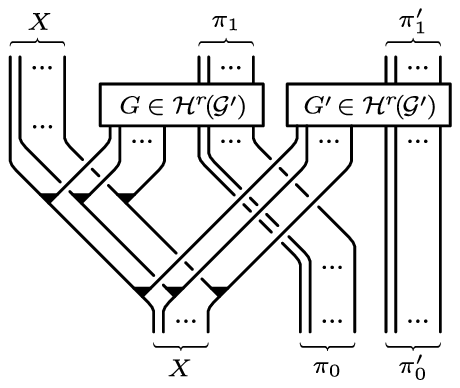}}
\vskip-3pt
\end{Figure}

Analogously to the categories of reducible Kirby tangles and reducible ribbon surface
tangles, also $\Mor_{\H^r_X(\G')}$ can be endowed with a product structure $\rdiam:
\Mor_{\H^r_X(\G')} \times \Mor_{\H^r_X(\G')} \to\Mor_{\H^r_X(\G')}$ as follows. The
product of two morphisms $F = (\id_X \diam G) \circ (\Delta_X \diam \id_{\pi_0})$ and $F'
= (\id_X \diam G') \circ (\Delta_X \diam \id_{\pi'_0})$ of $\H^r_X(\G')$ is defined by
(see Figure \ref{H-reduction03/fig})
\begin{eqnarray*}
F \rdiam F'
&=& F \circ (\id_X \diam \gamma_{\pi'_1,\pi_0}) \circ (F' \diam
\id_{\pi_0}) \circ (\id_X \diam \gamma^{-1}_{\pi'_0,\pi_0})\\
&=& (\id_X \diam G \diam G') \circ (\Delta_X \diam
\gamma_{X,\pi_0} \!\diam \id_{\pi'_0}) \circ (\Delta_X \diam \id_{\pi_0 \diam \pi'_0})\,.
\end{eqnarray*}

The coassociativity property of $\Delta$ implies the associativity of $\,\rdiam\,$, while
$\id_X$ is its unit. Observe also that $\rdiam$ does not define a monoidal structure on
$\H^r_X(\G')$, since it does not intertwine with the composition.

\begin{proposition}\label{H-stabilization/thm}
Given a full inclusion of groupoids $\G \subset \G'$ and a span\-ning sequence $X$ for
$(\G',\G)$, we have $\up_X \H^r(\G) \subset \H^r_X(\G')$. Moreover, for any two morphisms
$F, F'\in \H^r(\G)$, we have $\up_X (F\diam F')=\up_X F\rdiam \up_X F'$. In particular,
$\rdiam$ induces a monoidal structure on $\up_X \H^r(\G)$.
\end{proposition}

\begin{proof}
Figure \ref{H-reduction04/fig} shows that the $X$-stabilization of a morphism $F \in
\H^r(\G)$ is $X$-reducible, in other words $\up_X \H^r(\G)$ is a subcategory of
$\H^r_X(\G')$. The rest of the statement is straightforward.
\end{proof}

\begin{Figure}[htb]{H-reduction04/fig}
{}{Stabilizations are reducible
   [{\sl a}/\pageref{table-Hdefn/fig}]}
\vskip-6pt
\centerline{\fig{}{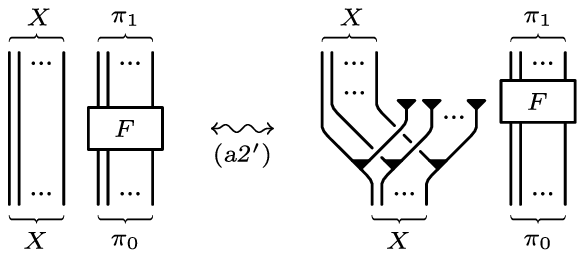}}
\vskip-3pt
\end{Figure}

Our goal is to prove that $\up_X: \H^r(\G) \to \H^r_X(\G')$ is actually an equivalence of
categories. Like we did in Section \ref{K/sec} for Kirby tangles, we will show this by
defining a reduction functor $\down_X: \H^r_X(\G') \to \H^r(\G)$ which, up to natural
equivalence, is the inverse of the stabilization functor. We first consider the case when
the spanning sequence $X$ consists of a single element $x \in \G'$, that is $X = (x)$. In
this case, we use the simplified notations $\H^r_x(\G') = \H^r_{(x)}(\G')$, ${\up_x} = 
{\up_{(x)}}$ and ${\down_x} = {\down_{(x)}}$.

\begin{definition}\label{H-reduction/def}
Let $\G$ be a groupoid. Given $x \in \G(i_0,j_0)$ with $i_0 \neq j_0$, considered as a
spanning sequence for $(\G,\G^{\bs i_0})$, we define the {\sl elementary reduction 
functor} ${\down_x}: \H^r_x (\G) \to \H^r(\G^{\bs i_0})$, by putting
$$\down_x (H_x \diam H_\pi) = H_{\pi^x}$$
for any object $H_x \diam H_\pi$ of $\H^r_x(\G)$, and  (see Figure 
\ref{H-reduction05/fig})
$$\down_x F = (\epsilon_{1_{j_0}} \!\diam \id_{\pi_1^x}) \circ F^x 
\circ (\eta_{j_0} \diam \id_{\pi_0^x}) = G^x \circ (\eta_{j_0} \diam \id_{\pi_0^x})$$
for any morphism $F = (\id_x \diam G) \circ (\Delta_x \diam \id_{\pi_0}): H_x \diam 
H_{\pi_0} \to H_x \diam H_{\pi_1}$. Here, $\_^x: \H^r_x(\G)\to \H^r(\G^{\bs i_0})$ is the 
restriction of the functor defined in Proposition \ref{H-push/thm}.

\begin{Figure}[b]{H-reduction05/fig}
{}{The reduction functor $\down_x$
   [{\sl a}/\pageref{table-Hdefn/fig}]}
\vskip-6pt
\centerline{\fig{}{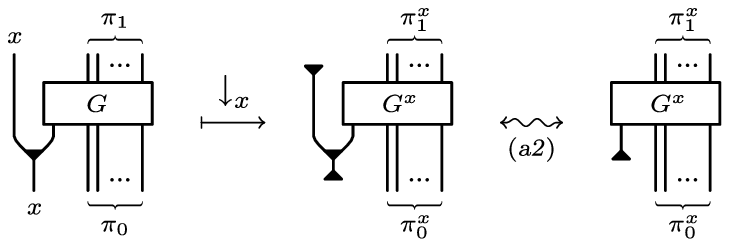}}
\vskip-3pt
\end{Figure}

Given a full inclusion $\G \subset \G'$ of groupoids and a spanning sequence $X = (x_n,
\dots, x_1)$ for $(\G',\G)$, the {\sl reduction functor} ${\down_X}: \H^r_X(\G') \to
\H^r(\G)$ is defined as the composition ${\down_X} = {\down_{x_1}} \!\circ \dots \circ
{\down_{x_n}}$ of elementary reduction functors. Observe that the composition is 
well-defined, since if $x_n \in \G(i_n,j_n)$ then $(x_{n-1}, \dots, x_1)$ forms a spanning 
tree for $(\G')^{\bs i_n}$.
\end{definition}

\begin{lemma}\label{H-reduction1/thm}
For any $x \in \G(i_0,j_0)$, the reduction ${\down_x}: \H^r_x(\G) \!\to \H^r(\G^{\bs
i_0})$ is a functor such that ${\down_x} \circ {\up_x} = \id_{\H^r(\G^{\bs i_0})}$,
while ${\up_x} \circ {\down_x} \simeq \id_{\H^r_x(\G)}$ up to the natural equivalence
$\xi^x = \xi^{x,1_{i_0}}$. Therefore, $\down_x$ and $\up_x$ are category equivalences
between $\H^r_x(\G)$ and $\H^r(\G^{\bs i_0})$.
\end{lemma}

\begin{proof}
That $\down_x$ is well-defined functor follows from the identity $$(\epsilon_{1_{j_0}}
\!\diam \id_{j_0} \!\diam \id_{j_0}) \circ (\Delta_{1_{j_0}} \!\diam \id_{j_0}) \!\circ
\Delta_{1_{j_0}} \!\circ \eta_{j_0}=\eta_{j_0} \!\diam \eta_{j_0}$$ and from the
functoriality of the map $\_^x$.

Looking at Figure \ref{H-reduction05/fig}, we see that $\down_x \circ \up_x =
\id_{\H^r(\G^{\bs i_0})}$. Indeed, if the left\-most diagram in the figure comes from the
stabilization of a morphism in $\H^r(\G^{\bs i_0})$ as in Figure \ref{H-reduction04/fig},
we have ${\pi_0}^x = \pi_0$, ${\pi_1}^x = \pi_1$ and $G^x = G$ by Theorem \ref{H-push/thm}
\(b). Hence, we end up with the rightmost diagram that represents $(\epsilon_{1_{j_0}}
\circ \eta_{j_0}) \diam G = G$ by \(a8) in Table \ref{table-Hdefn/fig}. Finally, Figure
\ref{H-reduction06/fig} shows that $\xi^x$ gives a natural equivalence $\up_x \circ\down_x
\simeq \id_{\H^r_x(\G)}$.
\end{proof}

\begin{Figure}[htb]{H-reduction06/fig}
{}{The natural equivalence $\xi^x: \up_x \circ \down_x \simeq \id_{\H^r_x(\G)}$
   ($x \in \G(i_0,j_0))$\break
   [{\sl a}/\pageref{table-Hdefn/fig},
    {\sl s}/\pageref{table-Hdefn/fig}-\pageref{table-Hprop/fig},
    {\sl q}/\pageref{table-adjointdefn/fig}-\pageref{table-Fx/fig}]}
\vskip-6pt
\centerline{\fig{}{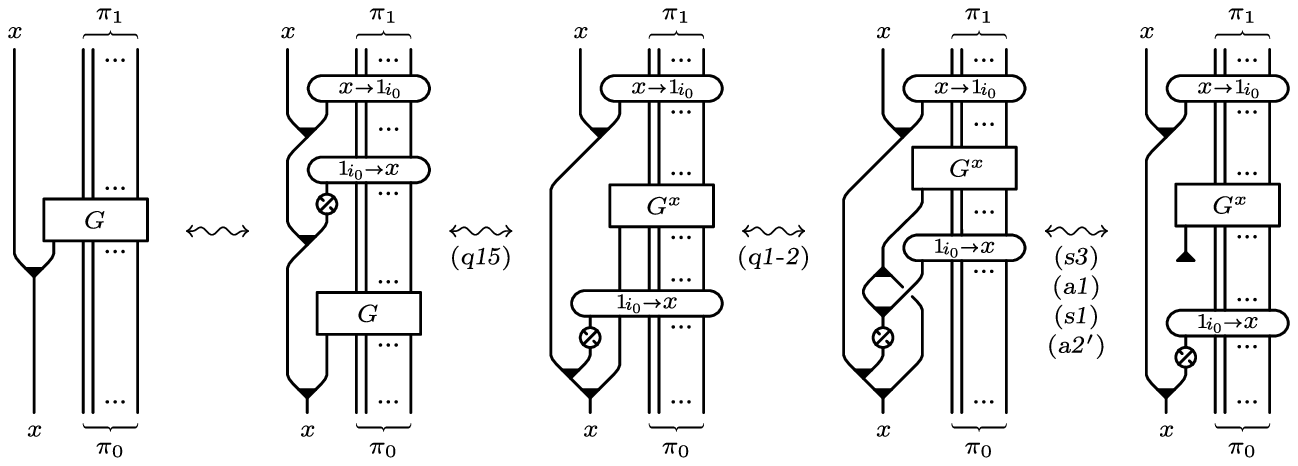}}
\vskip-3pt
\end{Figure}

\begin{proposition}\label{H-reduction/thm}
Given a full inclusion of groupoids $\G \subset \G'$ and a spanning sequence $X = (x_n,
\dots, x_1)$ for $(\G',\G)$, the reduction ${\down_X}: \H^r_X(\G') \to \H^r(\G)$ is a
functor such that ${\down_X} \circ {\up_X} = \id_{\H^r(\G)}$, while there is a natural
equivalence $\xi^X: {\up_X} \circ {\down_X} \simeq \id_{\H^r_X(\G)}$, inductively defined
by $\xi^{(x_1)} = \xi^{x_1}$ and $$\xi^X = \xi^{x_n} \circ (\id_{x_n} \diam \xi^Y)$$ with
$Y = (x_{n-1}, \dots, x_1)$. Therefore, $\down_X$ and $\up_X$ are category equivalences
between $\H^r_X(\G')$ and $\H^r(\G)$.
\end{proposition}

\begin{proof}
We proceed by induction on $n$. For $n = 1$ the statement follows from the previous lemma.
For $n > 1$, we have ${\up_X} = {\up_{x_n}} \circ {\up_Y}$ and ${\down_X} = {\down_Y}
\circ {\down_{x_n}}$ with $Y = (x_1, \dots, x_{n-1})$. Then, taking into account that
$\up_X \H^r(\G) \subset \H^r_X(\G')$, by the induction hypothesis we have $${\down_X}
\circ {\up_X} = {\down_Y} \circ {\down_{x_n}} \circ {\up_{x_n}} \circ {\up_Y} = {\down_Y}
\circ {\up_Y} = \id_{\H^r(\G)}\,.$$ 

\pagebreak

\noindent
Moreover, for any $F \in \H^r_X(\G)$ we can
write $\up_X \down_X F = \id_{x_n} \diam (\up_Y \down_Y (\down_{x_n} F))$, which induces a
natural equivalence $\id_{x_n} \diam \xi^Y: {\up_X} \circ {\down_X} \simeq {\up_{x_n}}
\circ {\down_{x_n}}$. Then, by composing with $\xi^{x_n}: {\up_{x_n}} \circ {\down_{x_n}}
\simeq \id_{\H^r_{x_n}(\G)}$, we get the natural equivalence $\xi^X: {\up_X} \circ
{\down_X} \simeq \id_{\H^r_X(\G)}$.
\end{proof}

Next proposition specializes the results above to the case when $\G = \G_k$, $\G' = \G_n$
and $X = \pi_{n \red k} = ((n,n-1), \dots, (k+1,k))$, with $n > k \geq 1$. In this case,
we will use a notation analogous to the one introduced in the context of the categories of
Kirby tangles. In particular, we put $\up_k^n=\up_{\pi_{n \red k}}$,
$\down_k^n=\down_{\pi_{n \red k}}$, $\xi^{n \red k}=\xi^{\pi_{n \red k}}$ and $\H^r_{n
\red k}=\H^r_{\pi_{n \red k}}(\G_n)$. Then we have the stabilization and reduction
functors:
$$\up_k^n = \up_{n-1}^n \circ \dots \circ \up_k^{k+1}: \H_k^r \to \H_n^r \ \text{ and }
\ \down_k^n = \down_k^{k+1} \circ \dots \circ \down_{n-1}^n: \H_{n \red k}^r \to 
\H_k^r\,.$$

\begin{proposition}\label{Hn-reduction/thm}
For any $n >k \geq 1$, the reduction $\down_k^n: \H^r_{n \red k} \to \H^r_k$ and
the restriction of the stabilization $\up_k^n: \H^r_k \to \H^r_{n \red k}$ are
category equivalences such that $\down_k^n \circ \up_k^n = \id_{\H^r_k}$ and $\xi^{n \red
k}: \up_k^n \circ \down_k^n \simeq \id_{\H^r_{n \red k}}$ is a natural equivalence.
Moreover, we have the following commutative diagrams.
\vskip9pt
\centerline{\epsfbox{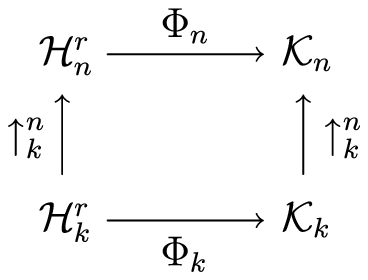}\kern15mm\epsfbox{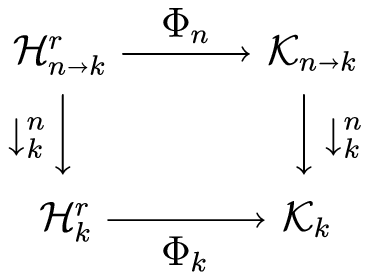}}
\vskip-18pt
\end{proposition}

\begin{proof}
The first part of the statement is just a special case of Propositions 
\ref{H-stabilization/thm} and \ref{H-reduction/thm}. The commutativity of the diagrams 
immediately follows from the definitions of the stabilization and reduction functors on 
both categories and Proposition \ref{KtoH-push/thm}.
\end{proof}

\subsection{The functors $\Psi_n: \S_n \to \H_n^r$%
\label{Psi/sec}}

We remind that in Section \ref{Theta/sec} it was defined a functor $\Theta_n: \S_n \to
\K_n$ from $n$-labeled ribbon surface tangles to $n$-labeled Kirby tangles. The goal of
this section is to show that $\Theta_n$ factorizes through the functor $\Phi_n: \H_n^r \to
\K_n$ introduced in Section~\ref{Phi/sec}. Namely, we will construct a functor $\Psi_n:
\S_n \to \H_n^r$ such that $\Theta_n = \Phi_n \circ \Psi_n$. Moreover, we will show that
the restriction $\Psi_n: \S_n^c \to \H^{r,c}_n$ is well-defined and full for any $n \geq
4$. This fact will allow us to prove the equivalence of the categories $S_n^c$, $\K_n^c$
and $\H^{r,c}_n$ for any $n\geq 4$ and the equivalence between $\K_n^c$ and $\H^{r,c}_n$
for any $n \geq 1$.

First of all, we observe that the objects of $\S_n$ correspond to sequences of
transpositions in $\Gamma_n$, i.e. unordered pairs of different indices, while the objects
of $\H^r_n$ correspond to sequences of morphisms in $\G_n$, i.e. ordered pairs of indices.
Hence, any functor $\S_n \to \H^r_n$ would require a choice of an ordering of the indices.

\begin{theorem}\label{psi/thm} 
For any strict total order $\prec$ on the set of objects of $\,\G_n\!$ with $n \geq
2$, there exists a braided monoidal functor $\Psi_n^\prec: \S_n \to \H^r_n$, such that:
\begin{itemize}
\item[(a)]
$\Psi_n^\prec$ restricts to a monoidal functor $\Psi_n^\prec: \S_{n \to k} \to \H_{n \to 
k}$ for any $1 \leq k < n$;
\item[(b)]
if $\prec'$ is another strict total order on the set of objects of $\,\G_n$, there is a 
natural equivalence $\tau: \Psi_n^\prec \to \Psi_n^{\prec'}$, which is identity on the 
empty set;
\item[(c)]
the following diagram commutes
\vskip9pt
\centerline{\epsfbox{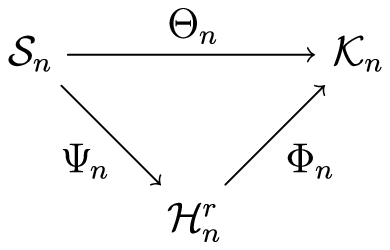}}
\vskip3pt
\noindent
where $\Psi_n = \Psi_n^<$ with $<$ being the natural order on $\Obj\G_n$.
\end{itemize}
\vskip-9pt
\end{theorem}

\begin{proof}
Given a strict total order $\prec$ on $\Obj\G_n$, we define inductively $\Psi_n^\prec:
\S_n \to \H^r_n$ on the objects according to the identities:
\vskip-4pt
$$\Psi_n^\prec(J_{\tp{i}{j}}) = H_{(i, j)} \ \text{ for any } i,j \in \Obj\G_n 
  \text{ with } i \succ j\,,$$
$$\Psi_n^\prec(J_{\sigma}) = \Psi_n^\prec(J_{\sigma'}) \diam \Psi_n^\prec(J_{\sigma''})
  \ \text{ for any } \sigma = \sigma' \diam \sigma'' \in \seq\Gamma_n\,.$$
\vskip4pt

Moreover, given another order $\prec'$ on $\Obj \G_n$, we consider the invertible
morphisms $\tau_\sigma: \Psi^\prec(J_\sigma) \to \Psi^{\prec'}(J_\sigma)$ with $\sigma \in
\seq\Gamma_n$ uniquely determined by:
$$
\tau_{(i \,j)} = \left\{%
\begin{array}{ll}
 \id_{(i,j)}: H_{(i,j)} \to H_{(i,j)} 
   & \text{ if } i \succ j \text{ and } i \succ' j\\[2pt]
 T_{(i,j)}: H_{(i,j)} \to H_{(j,i)} 
   & \text{ if } i \succ j \text{ and } j \succ' i
\end{array}\right.,$$
\vskip-4pt
$$\tau_\sigma = \tau_{\sigma'} \diam \tau_{\sigma''}
  \ \text{ for any } \sigma = \sigma' \diam \sigma'' \in \seq\Gamma_n\,.$$

On the morphisms of $\S_n$ we define $\Psi_n^\prec$ as follows. Figures \ref{psi01/fig}
and \ref{psi02/fig} describe the images under $\Psi_n^\prec$ of any labeling of the
elementary morphisms from \(a) to \(g) in Figure \ref{ribbon-morph01/fig}, while the image
of any labeling of the ribbon surface tangle \(g') in the same Figure
\ref{ribbon-morph01/fig} is defined through relation \(I6) in Figure
\ref{ribbon-tang01/fig}.

\begin{Figure}[htb]{psi01/fig}
{}{The functor $\Psi_n^\prec$ -- I ($i \succ j$, $i' \succ j'$)}
\centerline{\ \fig{}{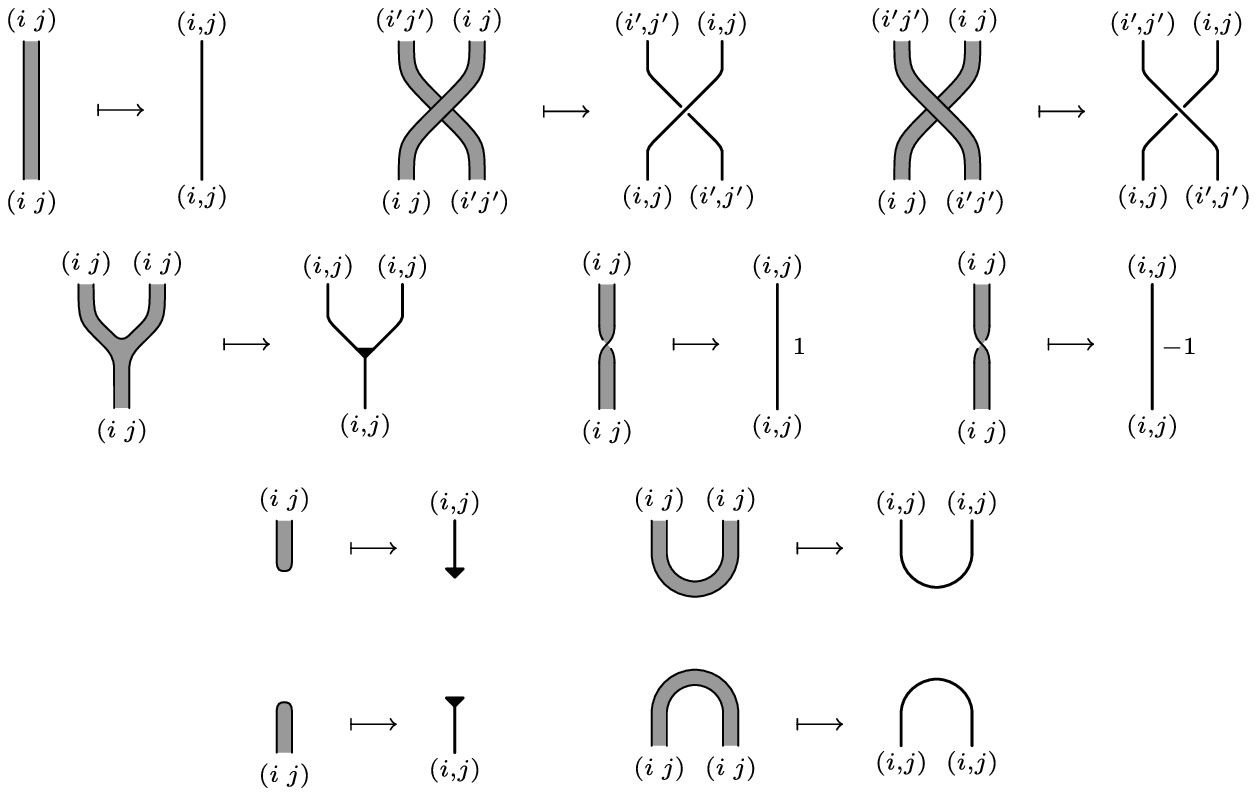}}
\vskip-3pt
\end{Figure}

\begin{Figure}[htb]{psi02/fig}
{}{The functor $\Psi_n^\prec$ -- II
   ($i \succ j \succ k$, $h \succ l$ and $\{i,j\} \cap \{h,l\} = \emptyset$)}
\centerline{\fig{}{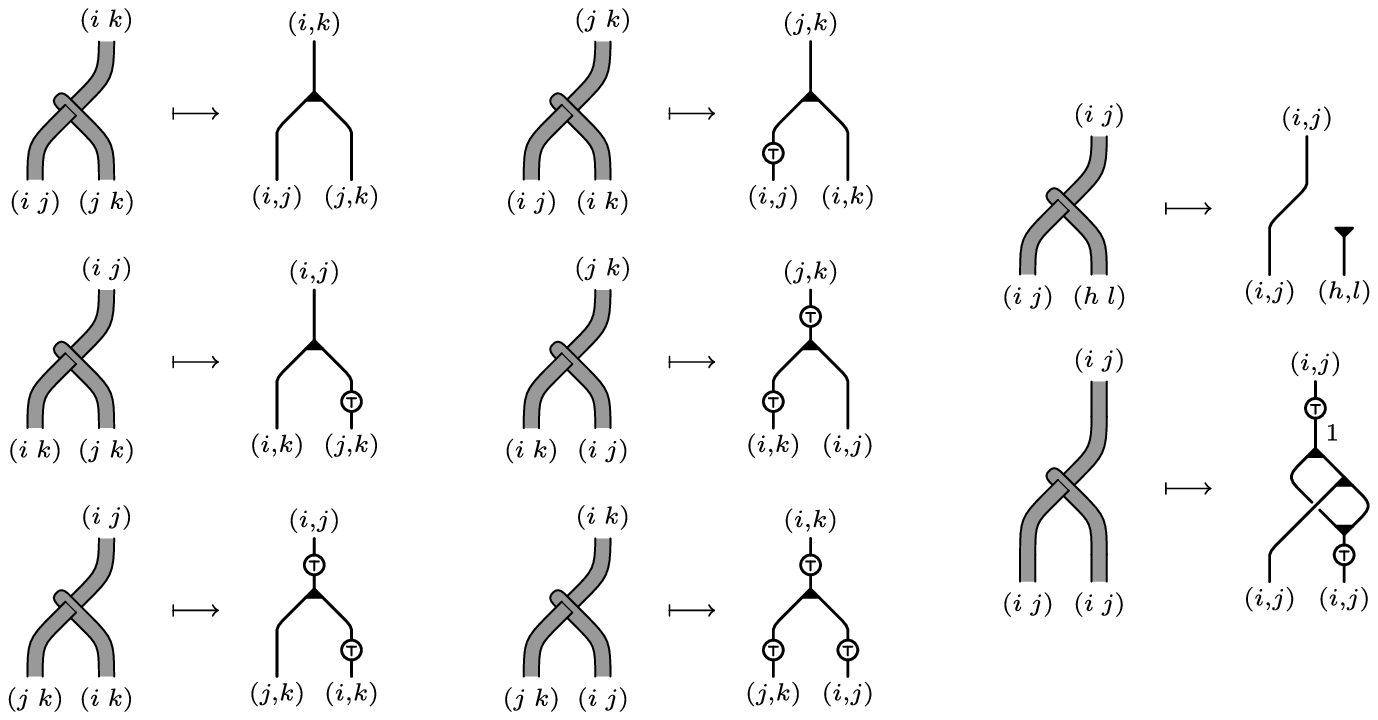}}
\vskip-3pt
\end{Figure}

\pagebreak

Then, we formally propagate this definition over products and compositions of elementary
morphisms. In order to be well-defined on the level of morphisms, such propagation has to
preserve all the defining relations of $\S_n$. Before checking that, we observe that for
any morphism $F: J_{\sigma_0} \to J_{\sigma_1}$ in $\S_n$ represented by a given
composition of products of elementary morphisms, the following diagram commutes.

\vskip12pt
\centerline{\epsfbox{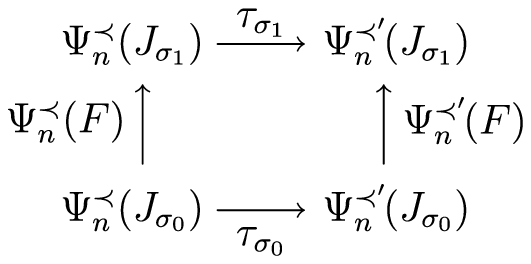}}
\vskip6pt

Indeed, the commutativity in the case when $F$ is a elementary morphism can be easily
derived from the relations \(t1-3-4-9) in Table \ref{table-Hr/fig}, \(s5) in Table
\ref{table-Hprop/fig}, \(i4) and \(f6-7) in Table \ref{table-Hu/fig}, \(r4) in Table
\ref{table-Huvdefn/fig}, \(f9-10) and \(p1) in Table \ref{table-Huvprop/fig}. On the other
hand, the extension to products is trivial and while that to compositions is guaranteed by
the relation \(t3).

Now we show that the formal definition of $\Psi_n^\prec$ on iterated products/compositions
of elementary morphisms actually preserves all the defining relations of $\S_n$, that is
the labeled versions of the 1-isotopy moves in Figures \ref{ribbon-tang01/fig},
\ref{ribbon-tang02/fig}, \ref{ribbon-tang03/fig} and \ref{ribbon-tang04/fig} together with
the two ribbon moves in Figure \ref{ribbon-moves03/fig}. In doing that, by the
commutativity of the above diagram, we can choose the most convenient order $\prec$ for
each single move.

We notice that move \(I6) is trivially preserved, since it was used to define the images
under $\Psi_n^\prec$ of the morphisms on the left side of it. Analogously, \(R2) is
trivially preserved by the definition of $\Psi_n^\prec$.

Most of the other moves essentially rewrite the algebra axioms and the relations of
$\H^r_n$ introduced in Sections \ref{HG/sec} and \ref{HrG/sec}. Namely, we have that:
\(I1), \(I7-7') and \(I8-9) correspond to the braid axioms in Table \ref{table-Hdefn/fig};
\(I18) rewrites the bi-algebra axiom \(a1) in the same Table \ref{table-Hdefn/fig};
\(I2-2'), \(I5) and \(I16) respectively rewrite \(f3-3'), \(f4-4') and \(f1) in Table
\ref{table-Hprop/fig}; \(I3-3') rewrite the bottom-left duality moves in Table
\ref{table-Hu/fig}; \(I4-4') follow from \(f10) in Table \ref{table-Huvprop/fig}; \(I10),
\(I12-12') and \(I13) follow from \(r2), \(r5) and \(p1) in Tables \ref{table-Huvdefn/fig}
and \ref{table-Huvprop/fig}; \(I11) follows from \(f5) and \(t1-3) in Tables
\ref{table-Hu/fig} and \ref{table-Hr/fig}; \(I15) follows from \(t5) in Table
\ref{table-Hr/fig}; \(I20) follows from \(a6), \(s6) and \(r4) in Tables
\ref{table-Hdefn/fig}, \ref{table-Hprop/fig} and \ref{table-Huvdefn/fig}; \(R1) follows
from \(t3) and \(t6) in Table \ref{table-Hr/fig}, taking into account that $T$ propagates
through the form and the coform, due to \(f6-7) in Table \ref{table-Hu/fig} and \(f9-10)
in Table \ref{table-Huvprop/fig}.

\smallskip

Below we indicate how the verification goes for the remaining moves.  

\smallskip

We start with \(I19) in Figure \ref{psi03/fig}. The labeling of this move is unique up to
reordering the indices, while the labelings of other moves present different cases,
depending on the labels at the ribbon intersections. We will call a ribbon intersection
uni-, bi- or tri-labeled according to the number of different labels that occur in it.

\begin{Figure}[htb]{psi03/fig}
{}{Preservation of \(I19) ($i \succ j$) 
   [{\sl a}/\pageref{table-Hdefn/fig}, 
    {\sl f}/\pageref{table-Hprop/fig}-\pageref{table-Hu/fig},
    {\sl r}/\pageref{table-Huvdefn/fig}-\pageref{table-Huvprop/fig}, 
    {\sl s}/\pageref{table-Hdefn/fig}, {\sl t}/\pageref{table-Hr/fig}]}
\centerline{\fig{}{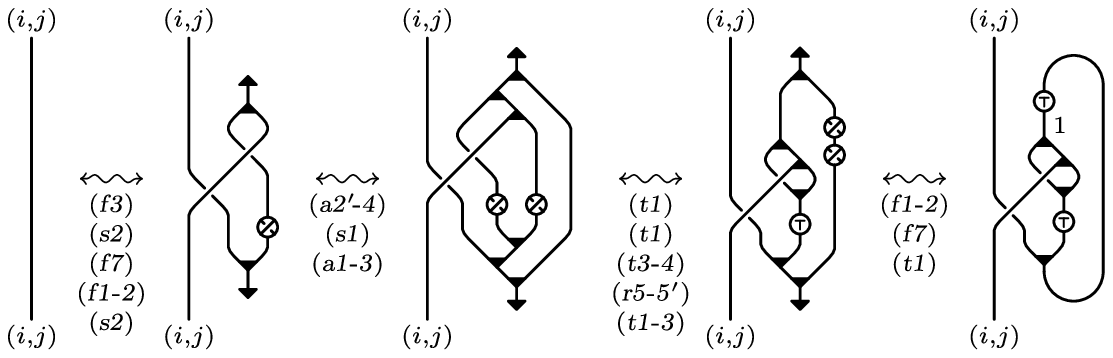}}
\vskip-3pt
\end{Figure}


\begin{Figure}[htb]{psi04/fig}
{}{Preservation of \(I14-14') in the uni-labeled case -- I ($i \succ j$) 
   [{\sl a}/\pageref{table-Hdefn/fig},
    {\sl f}/\pageref{table-Hprop/fig}-\pageref{table-Hu/fig}-\pageref{table-Huvprop/fig}, 
    {\sl r}/\pageref{table-Huvdefn/fig}-\pageref{table-Huvprop/fig}, 
    {\sl s}/\pageref{table-Hdefn/fig}-\pageref{table-Hprop/fig}, 
    {\sl t}/\pageref{table-Hr/fig}]}
\centerline{\fig{}{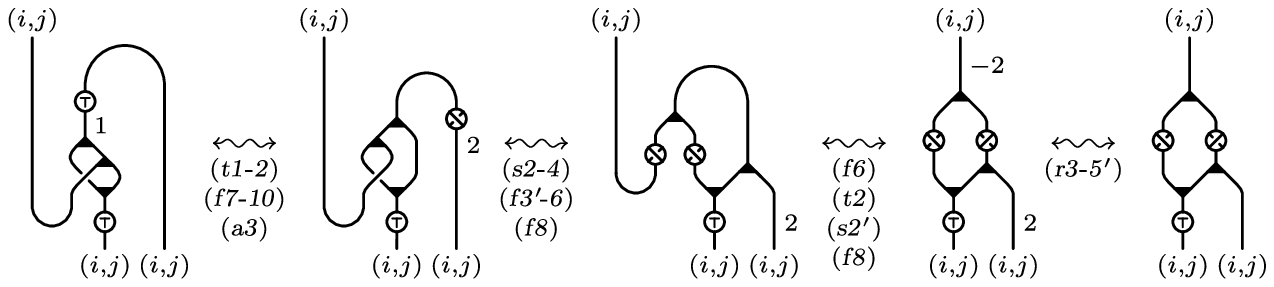}}
\vskip-3pt
\end{Figure}

\begin{Figure}[htb]{psi05/fig}
{}{Preservation of \(I14-14') in the uni-labeled case -- II ($i \succ j$) 
   [{\sl a}/\pageref{table-Hdefn/fig},
    {\sl r}/\pageref{table-Huvdefn/fig}-\pageref{table-Huvprop/fig}, 
    {\sl s}/\pageref{table-Hdefn/fig}-\pageref{table-Hprop/fig}, 
    {\sl t}/\pageref{table-Hr/fig}]}
\centerline{\fig{}{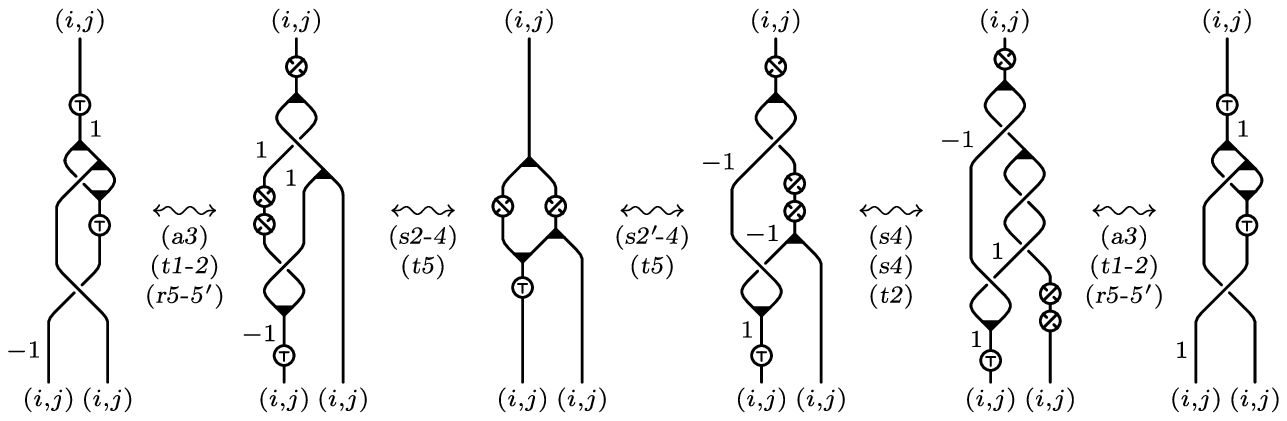}}
\vskip-3pt
\end{Figure}

\break

\(I14-14') for a bi-labeled ribbon intersection are trivial, since they follow from the
braid axioms. For a tri-labeled ribbon intersection they follow directly from \(t3-6) and
\(t7) in Table \ref{table-Hr/fig}. Figures \ref{psi04/fig} and \ref{psi05/fig} deal with
the uni-labeled case. More precisely, Figure \ref{psi04/fig} shows that the image under
$\Psi_n^\prec$ of the labeled ribbon surface tangle in the middle of \(I14-14') is
equivalent to the third graph diagram in Figure \ref{psi05/fig}.

\smallskip

\(I17) for a bi-labeled ribbon intersection reduces to a crossing change, so it follows
from axioms \(r9) in Table \ref{table-Hr/fig}. The tri- and uni-labeled cases are
presented in Figures \ref{psi06/fig} and Figure \ref{psi07/fig} respectively.
 
\begin{Figure}[htb]{psi06/fig}
{}{Preservation of \(I17) in the tri-labeled case ($i \succ j \succ k$)
   [{\sl a}/\pageref{table-Hdefn/fig}, {\sl f-p}/\pageref{table-Huvprop/fig},
    {\sl r}/\pageref{table-Hr/fig},
    {\sl s}/\pageref{table-Hdefn/fig}-\pageref{table-Hprop/fig},
    {\sl t}/\pageref{table-Hr/fig}]}
\centerline{\fig{}{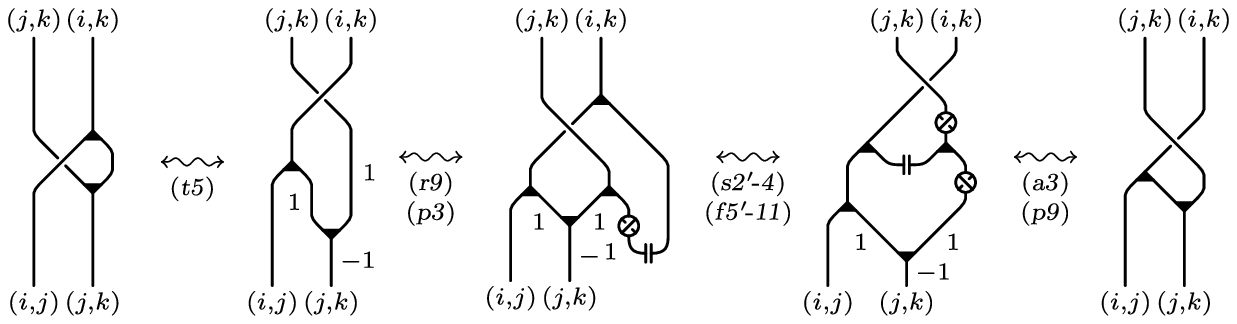}}
\vskip-3pt
\end{Figure}

\begin{Figure}[htb]{psi07/fig}
{}{Preservation of \(I17) in the uni-labeled case ($i \succ j$)
   [{\sl a}/\pageref{table-Hdefn/fig}, 
    {\sl p}/\pageref{table-Huvprop/fig}-\pageref{table-Hr/fig},
    {\sl r}/\pageref{table-Huvdefn/fig}-\pageref{table-Huvprop/fig}, 
    {\sl s}/\pageref{table-Hdefn/fig}-\pageref{table-Hprop/fig}, 
    {\sl t}/\pageref{table-Hr/fig}]}
\vskip-6pt
\centerline{\fig{}{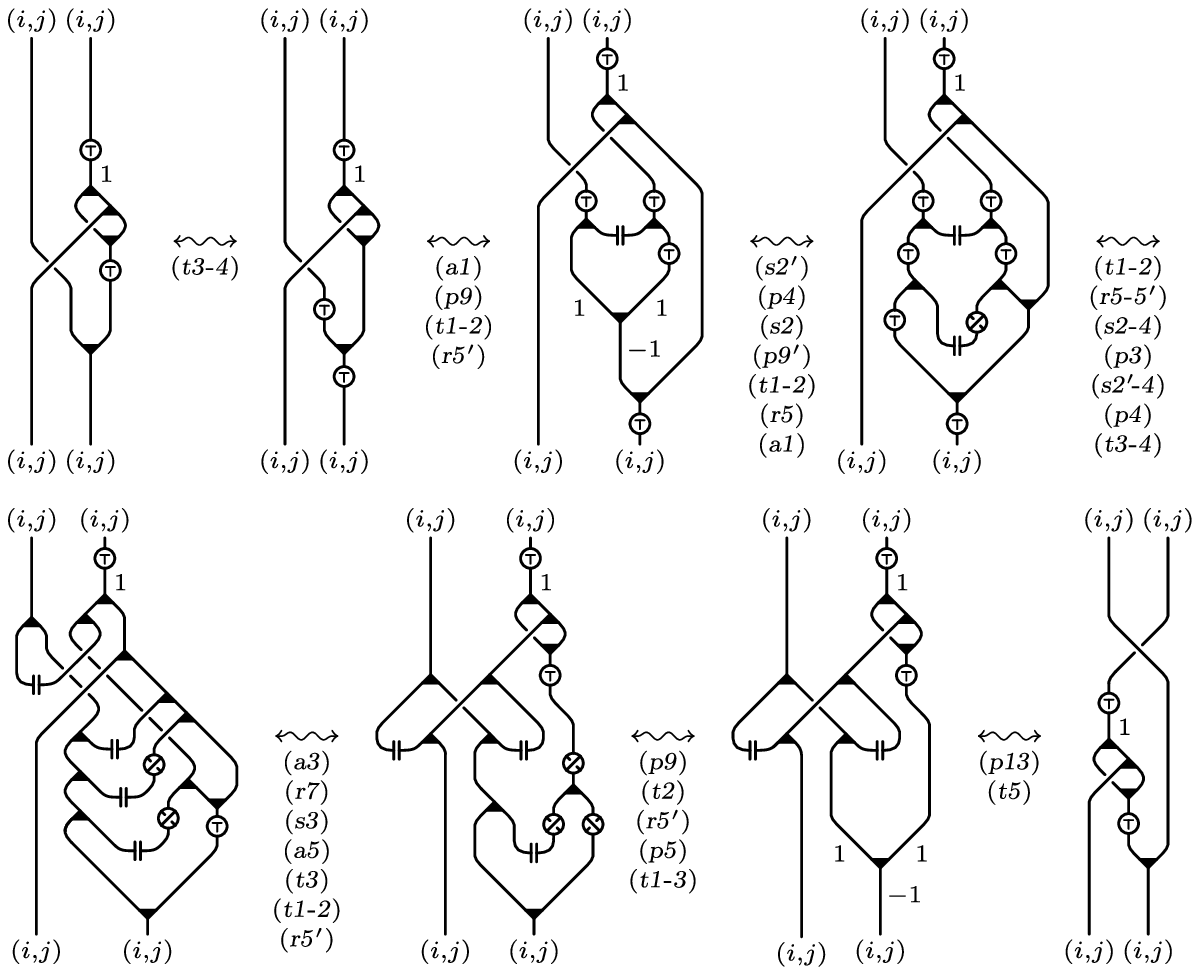}}
\vskip-3pt
\end{Figure}

\(I21) for bi-labeled ribbon intersection is trivial, while for tri-labeled ribbon
intersection, with the proper choice of the order $\prec\,$, it corresponds to the
bi-algebra axiom \(a5) in Table \ref{table-Hdefn/fig}. The uni-labeled case is treated in
Figure \ref{psi08/fig}.

\begin{Figure}[htb]{psi08/fig}
{}{Preservation of \(I21) in the uni-labeled case ($i \succ j$) 
   [{\sl a}/\pageref{table-Hdefn/fig}, 
    {\sl f}/\pageref{table-Hu/fig}-\pageref{table-Huvprop/fig},
    {\sl p}/\pageref{table-Huvprop/fig},
    {\sl r}/\pageref{table-Huvdefn/fig}-\pageref{table-Huvprop/fig}%
           -\pageref{table-Hr/fig}, 
    {\sl s}/\pageref{table-Hdefn/fig}-\pageref{table-Hprop/fig}, 
    {\sl t}/\pageref{table-Hr/fig}]}
\centerline{\ \fig{}{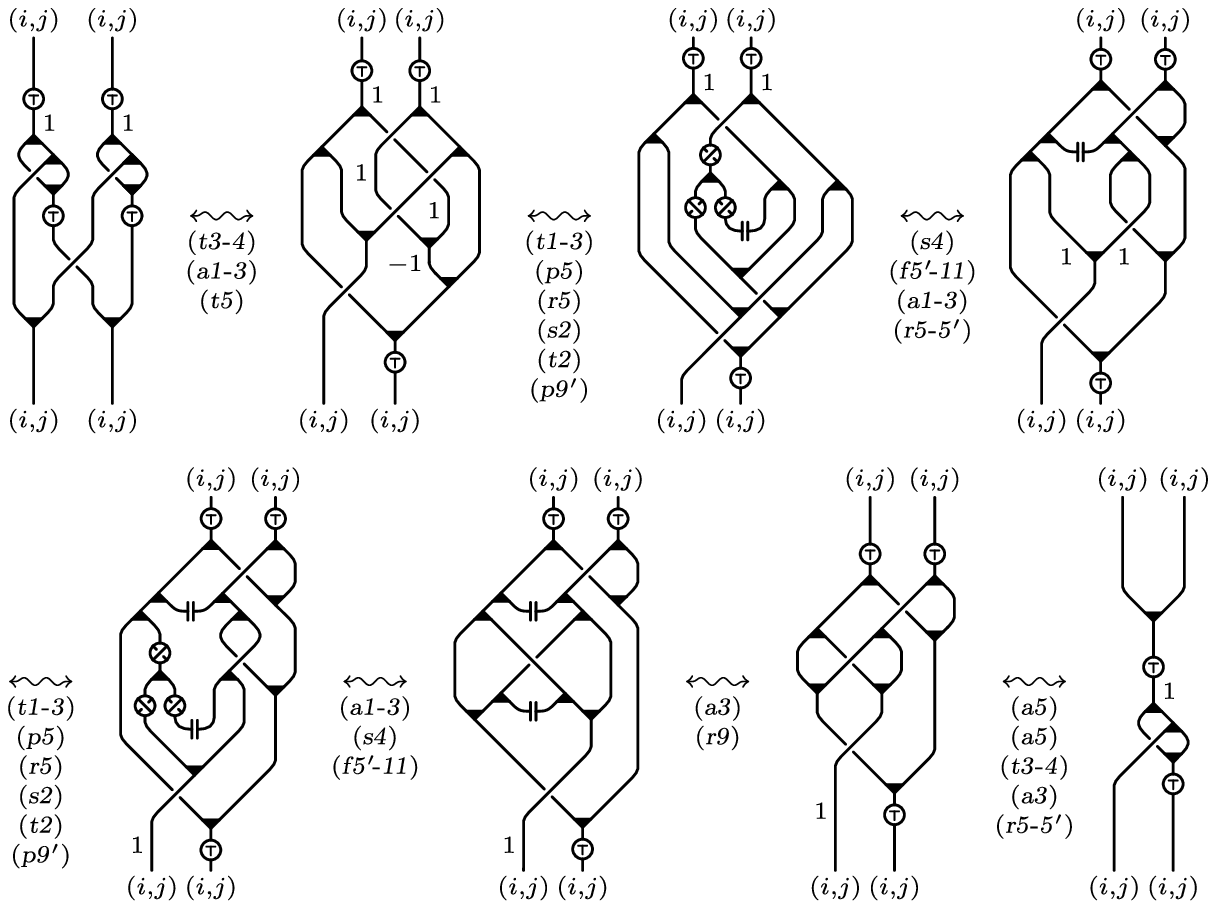}}
\end{Figure}

\begin{Figure}[b]{psi09/fig}
{}{Preservation of \(I22) in the non-trivial cases when a bi-labeled ribbon    
   intersection occurs, ($i \succ j \succ k \succ l$)
   [{\sl a}/\pageref{table-Hdefn/fig},
    {\sl r}/\pageref{table-Huvdefn/fig}-\pageref{table-Huvprop/fig}%
           -\pageref{table-Hr/fig}, 
    {\sl s}/\pageref{table-Hdefn/fig}, {\sl t}/\pageref{table-Hr/fig}]}
\vskip3pt
\centerline{\fig{}{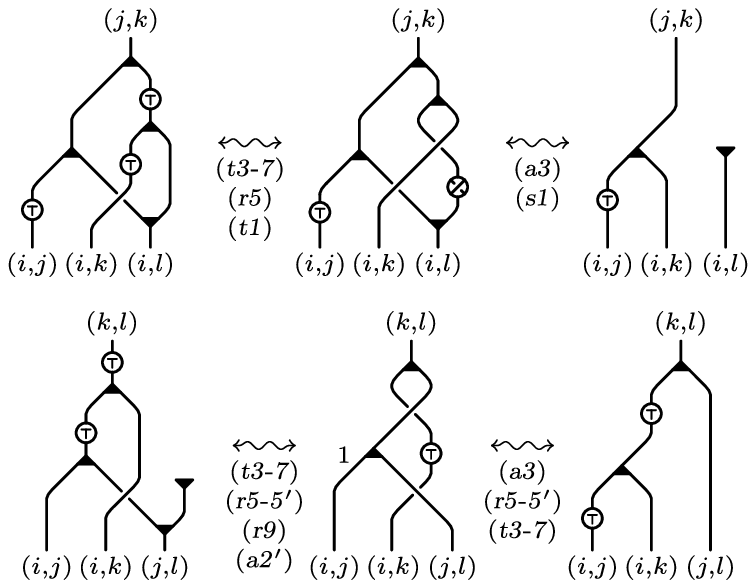}}
\end{Figure}

\pagebreak\vglue-3pt

\(I22) is the most complicated relation to deal with, since the source of the involved
morphisms consists of three intervals which can be labeled independently from each other,
so there are many different cases. First of all, we observe that the presence of disjoint
labels allows us to simplify the relation by using move \(R2) in Figure
\ref{ribbon-moves03/fig} to remove the bi-labeled ribbon intersections. In particular,
when one of those labels is disjoint from both the other two, such simplification reduce
\(I22) to \(I16) and \(I20) modulo \(I2-2'), \(I3-3'), \(I5) and \(I8). Up to
conjugation, the only remaining labelings of the three intervals that include a pair of
disjoint labels, are given by the sequences $(\tp{i}{j}, \tp{i}{k}, \tp{k}{l})$ and
$(\tp{i}{j}, \tp{k}{l}, \tp{i}{k})$ where $i,j,k,l$ are all distinct. Assuming $i \succ j
\succ k \succ l$, after simplification the first case corresponds to the bi-algebra axiom
\(a3) in Table \ref{table-Hdefn/fig}, while the second one reduces to \(I16) and \(I20) as
above. Figure \ref{psi09/fig} concerns the two remaining cases when a bi-labeled ribbon
intersection occurs (even if there is no pair of disjoint labels in the source). The rest
of the cases are presented in Figures \ref{psi10/fig}, \ref{psi11/fig} and \ref{psi12/fig}
respectively, depending on the number of uni-labeled ribbon intersections.

\begin{Figure}[htb]{psi10/fig}
{}{Preservation of \(I22) in the cases when one uni-labeled ribbon intersection occurs 
   ($i \succ j \succ k$)
   [{\sl a}/\pageref{table-Hdefn/fig}, 
    {\sl r}/\pageref{table-Huvdefn/fig}-\pageref{table-Huvprop/fig},
    {\sl s}/\pageref{table-Hdefn/fig}, {\sl t}/\pageref{table-Hr/fig}]}
\centerline{\fig{}{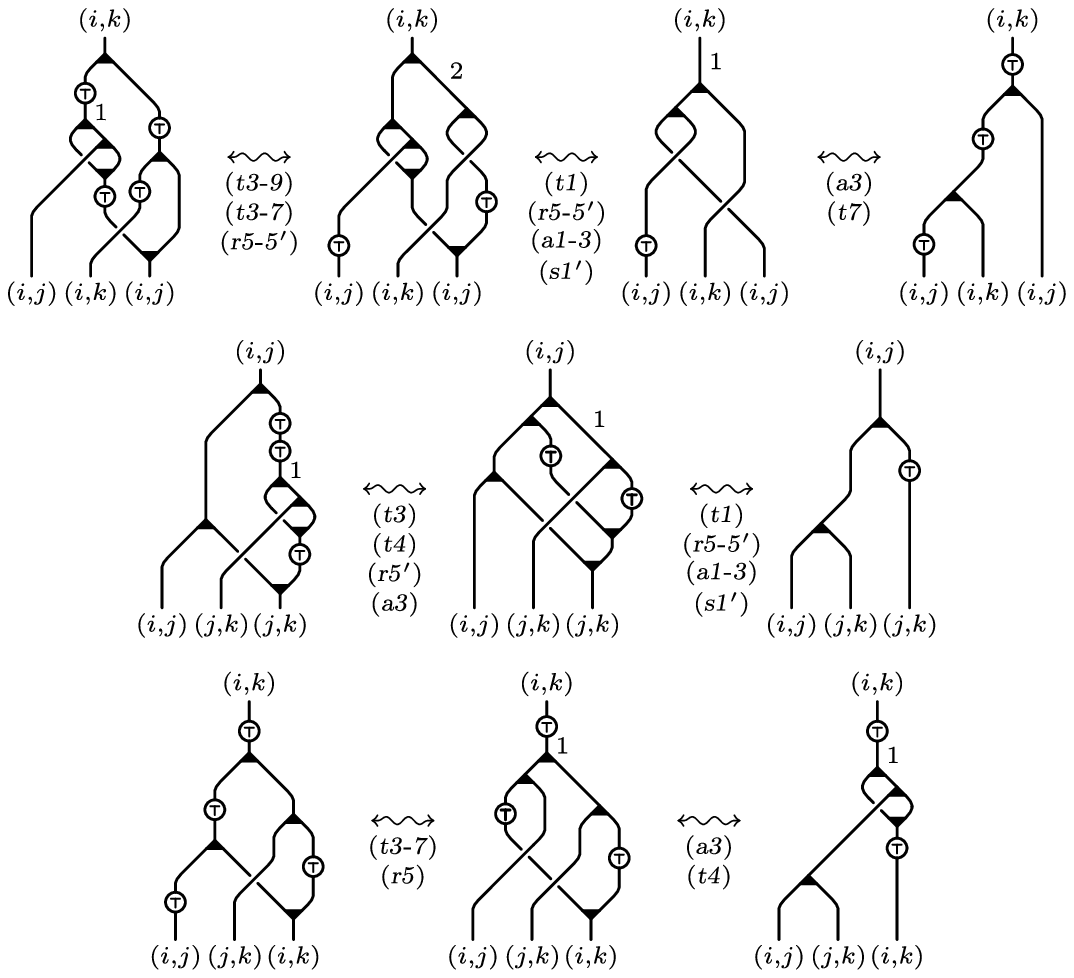}}
\vskip-3pt
\end{Figure}

\begin{Figure}[htb]{psi11/fig}
{}{Preservation of \(I22) in the cases when two uni-labeled ribbon intersections occur
   ($i \succ j \succ k$)
   [{\sl a}/\pageref{table-Hdefn/fig},
    {\sl r}/\pageref{table-Huvdefn/fig}-\pageref{table-Huvprop/fig},
    {\sl s}/\pageref{table-Hdefn/fig}-\pageref{table-Hprop/fig}, 
    {\sl t}/\pageref{table-Hr/fig}]}
\centerline{\fig{}{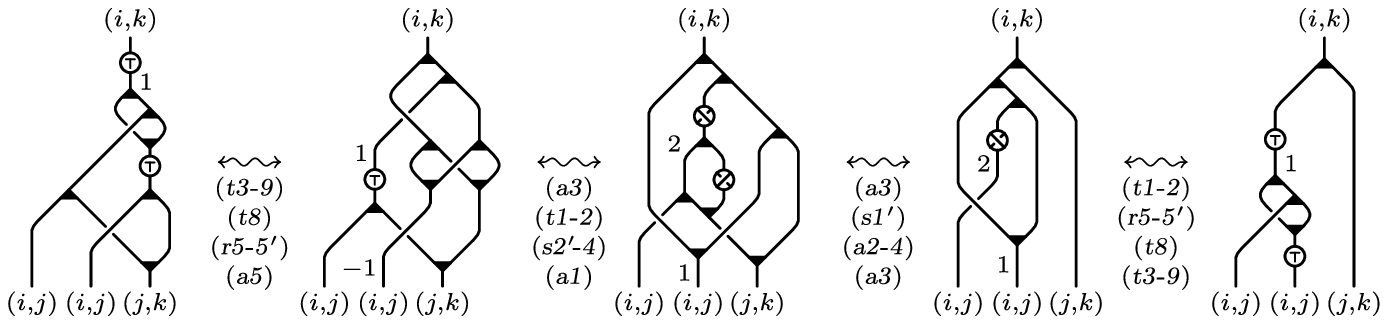}}
\vskip-3pt
\end{Figure}

\begin{Figure}[htb]{psi12/fig}
{}{Preservation of \(I22) in the uni-labeled case ($i \succ j$)
   [{\sl a}/\pageref{table-Hdefn/fig},
    {\sl f}/\pageref{table-Hu/fig}-\pageref{table-Huvprop/fig},
    {\sl p}/\pageref{table-Huvprop/fig}-\pageref{table-Hr/fig},
    {\sl r}/\pageref{table-Huvdefn/fig}-\pageref{table-Huvprop/fig}%
           -\pageref{table-Hr/fig}, 
    {\sl s}/\pageref{table-Hdefn/fig}-\pageref{table-Hprop/fig},
    {\sl t}/\pageref{table-Hr/fig}]}
\centerline{\fig{}{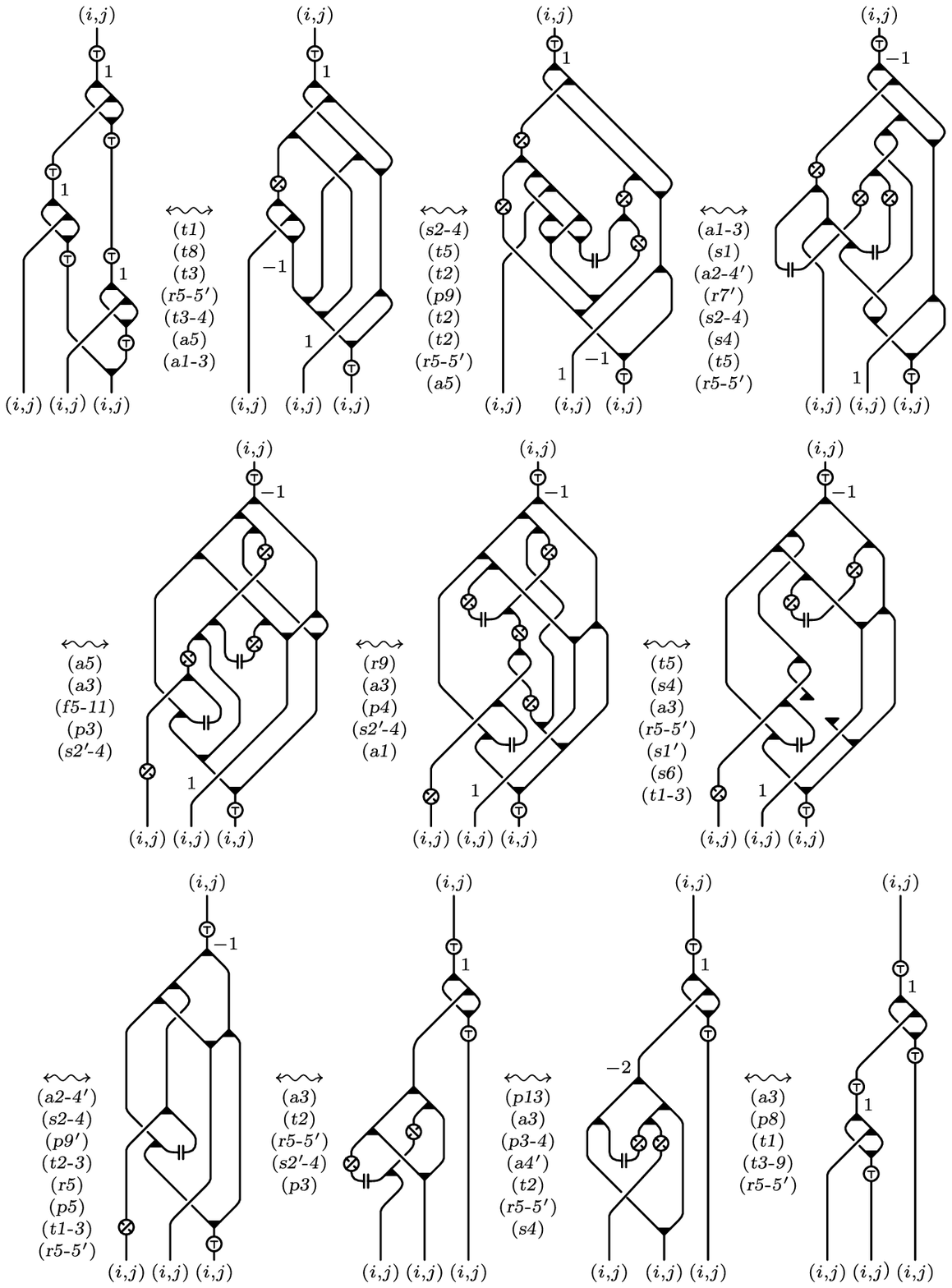}}
\vskip-3pt
\end{Figure}

This completes the proof that $\Psi_n^\prec$ is a well-defined functor. Then, its
monoidality is trivial while the desired natural equivalence $\tau$ in \(b) was defined in
the beginning. Moreover, property \(a) directly derives from the fact that the image of
the elementary morphism $\Delta_{\tp{i}{j}}$ in Figure \ref{ribbon-morph02/fig} is exactly
$\Delta_{(i,j)}$ for $i \succ j$. Finally, the commutativity of the diagram in \(c) can be
\pagebreak
seen by comparing the definitions of $\Phi_n$ in Section \ref{Phi/sec} (cf. Figures
\ref{phi01/fig} and \ref{phi02/fig}) and $\Psi_n$ (cf. Figures \ref{psi01/fig} and
\ref{psi02/fig}) with the definition of $\Theta_n$ in Section \ref{Theta/sec} (cf. Figures
\ref{theta01/fig}, \ref{theta02/fig} and \ref{theta03/fig}).
\end{proof}

Theorem \ref{psi/thm} with $n = 2$ and Proposition \ref{Hn-reduction/thm} imply that there
exists a functor from the category of 1-isotopy equivalence classes of (unlabeled) ribbon
surface tangles $\S$ (cf. Section \ref{surfaces/sec}) to the universal ribbon Hopf algebra
$\H^r = \H_1^r$ (over the trivial groupoid). Observe that the two categories are obviously 
not equivalent, but the corollary below allows us to associate to any braided unimodular
ribbon Hopf algebra an invariant of ribbon surface tangles under 1-isotopy moves.

\pagebreak

\begin{corollary}\label{rstangle-inv/thm}
There exists a monoidal functor $\Psi: \S \to \H^r$, defined by $\Psi(J_m) = H^{\diam m}$
on the objects and by Figure \ref{psi13/fig} on the elementary morphisms.
\end{corollary}

\begin{Figure}[htb]{psi13/fig}
{}{The functor $\Psi: \S \to \H^r$}
\centerline{\fig{}{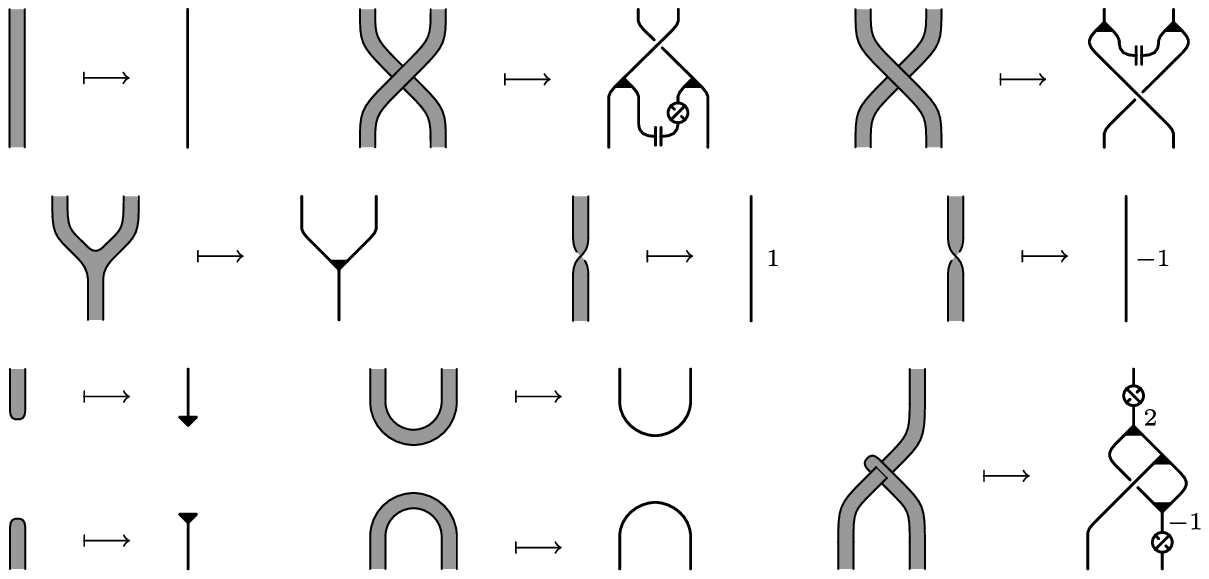}}
\vskip-3pt
\end{Figure}

\begin{proof}
Comparing the presentation of the category $\S$ in Proposition \ref{S-category/thm} with
the one of the category $\S_n$ in Proposition \ref{Sn-category/thm}, we observe that the
the relations of $\S_2$ are just the labeled versions of the relations of $\S$ (see Figure
\ref{ribbon-moves03/fig}). In fact, the extra relations \(R1) and \(R2) do not appear in
$\S_2$ since they involve at least three different labels. Therefore, the category $\S$ is
equivalent to $\S_2$, where the equivalence functor is given by labeling the whole surface
by the transposition $\tp{1}{2}$. Then, we put $\Psi = \down_1^2 \circ \Psi_2$ and the
statement follows from Theorem \ref{psi/thm}, Proposition \ref{Hn-reduction/thm} and the 
identity in Figure \ref{psi14/fig}.
\end{proof}

\begin{Figure}[htb]{psi14/fig}
{}{[{\sl a}/\pageref{table-Hdefn/fig}, {\sl p}/\pageref{table-Huvprop/fig},
    {\sl t}/\pageref{table-Hr/fig}]}
\vskip-3pt
\centerline{\fig{}{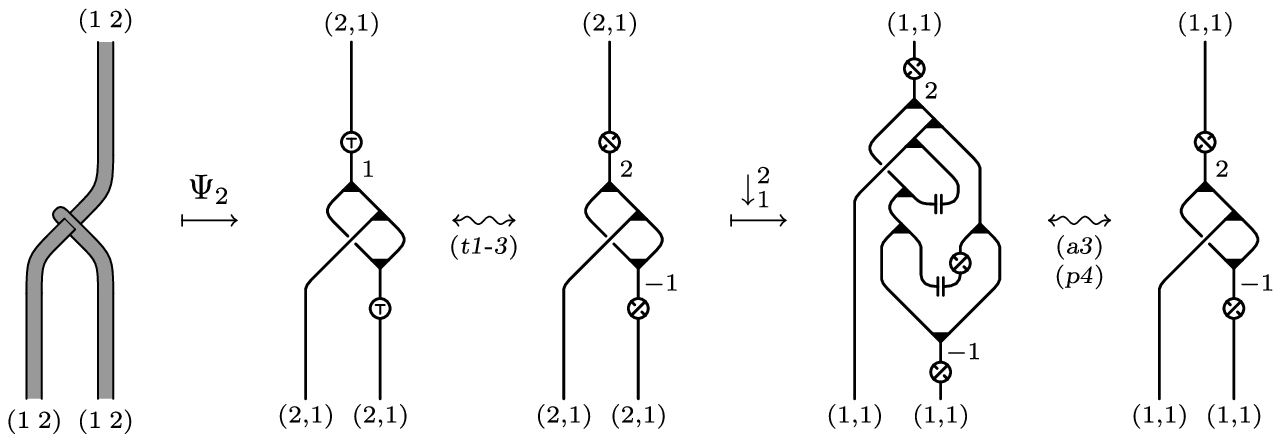}}
\vskip-3pt
\end{Figure}

\subsection{Equivalence between $\K_n^c$ and $\H_n^{r,c}$%
\label{K=H/sec}}

According to point \(a) of Theorem \ref{psi/thm}, we know that the functor $\Psi_n: \S_n
\to \H_n$ restricts to a well-defined functor $\Psi_n: \S_n^c \to \H_n^{r,c}$, where we
use the notation $\S_n^{r,c} = \S_{n \red 1}^r$ and $\H_n^{r,c} = \H_{n \red 1}^r$. Now we
will show that for $n \geq 3$ such restriction is full (cf. Proposition
\ref{full-psi/thm}) and this will allow us to complete the proof of the equivalence of the
categories $\S_n^c$, $\H_n^{r,c}$ and $\K_n^c$ with $n \geq 4$ (cf. Theorem
\ref{equivalence/thm}).

We first need two technical lemmas.

\begin{lemma}\label{ij/thm}
Given $(i,j) \in \G_n$ with $i \neq j$, the identity $\id_{\pi_{n \red 1}}\!$ as a
morphism of $\,\H^{r,c}_n$ can be represented by a diagram containing a counit vertex
$\epsilon_{(i,j)}$ and no edge labeled $(k,k)$ with $1 \leq k \leq n$.
\end{lemma}

\begin{proof} 
Since $\id_{\pi_{n \red 1}}$ contains $\id_{\pi_{i \red j}}$ for $n \geq i \geq j \geq 1$, 
it suffices to prove that $\id_{\pi_{i \red j}}$ can be represented by two diagrams 
containing $\epsilon_{(i,j)}$ and $\epsilon_{(j,i)}$ respectively and no edge labeled 
$(k,k)$. This is done in Figure \ref{equivalence01/fig}.
\end{proof}

\begin{Figure}[htb]{equivalence01/fig}
{}{($i > j$)
   [{\sl a}/\pageref{table-Hdefn/fig}, {\sl s}/\pageref{table-Hprop/fig}]}
\vskip-9pt
\centerline{\fig{}{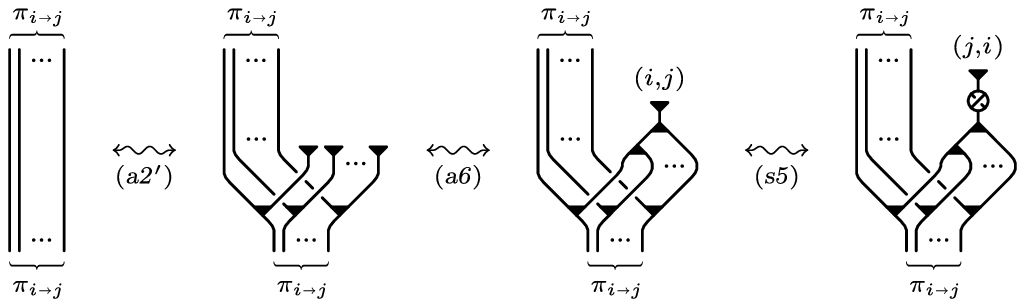}}
\vskip-3pt
\end{Figure}

\begin{lemma}\label{psi-image/thm}
Let $F$ be any morphism of $\,\H^{r,c}_n$ whose source and target are in $\Psi_n(\Obj
\S^c_n)$, i.e. they have the form $H_{\pi_{n \red 1}} \!\diam H_\pi$, with $\pi =
((i_1,j_1), \dots, (i_m,j_m)) \in \seq\G_n$ such that $i_h > j_h$ for $h = 1, \dots, m$.
If $F$ is given by a composition of products of the elementary diagrams presented in
Figure \ref{equivalence02/fig} with $i \neq j \neq k \neq i$ and $i' \neq j'$, then it is 
in the image of $\Psi_n$.
\end{lemma}

\begin{Figure}[htb]{equivalence02/fig}
{}{($i\neq j \neq k \neq i$, $i' \neq j'$)}
\vskip-9pt
\centerline{\fig{}{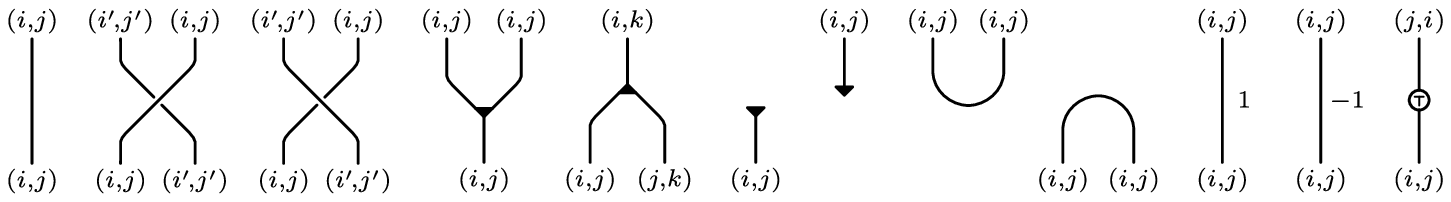}}
\vskip-3pt
\end{Figure}

\begin{proof}
Consider a morphism $F$ as in the statement. We call a label $(i,j)$
in the diagram {\sl good} if $i > j$ and {\sl bad} if $i < j$. We also call a vertex of
the diagram {\sl good} if all the labels of the edges attached to it are good and {\sl
bad} otherwise. Observe that we can transform any comultiplication, counit and any
integral vertex into a good one by applying to it, if necessary, the moves in Figure
\ref{equivalence03/fig}. Therefore, we can assume that any bad vertex of the diagram is a
multiplication vertex.

\begin{Figure}[htb]{equivalence03/fig}
{}{$(i < j)$
   [{\sl i}/\pageref{table-Hu/fig}, {\sl p}/\pageref{table-Huvprop/fig},
    {\sl r}/\pageref{table-Huvdefn/fig}, {\sl s}/\pageref{table-Hprop/fig},
    {\sl t}/\pageref{table-Hr/fig}]}
\centerline{\fig{}{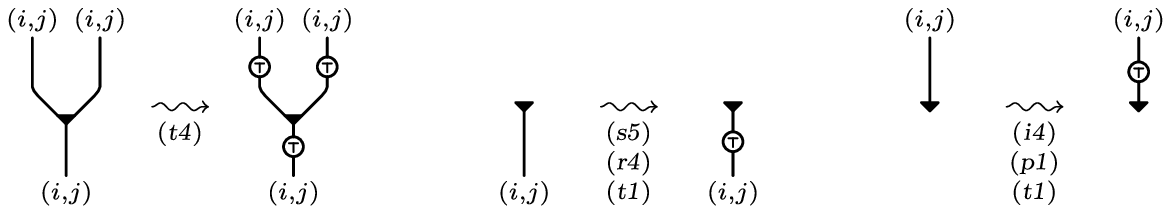}}
\vskip-3pt
\end{Figure}

Now, consider the diagram of $F$ as a planar diagram of a graph in which each edge $e$ is
weighted by an integer and decorated by a certain number $n(e)$ of $T$'s. We observe that
the $T$'s can be slided along the edges by using the braid axioms in Table
\ref{table-Hdefn/fig} and the moves in Figure \ref{equivalence04/fig}. Then, we can always
assume $n(e) = 0,1$, since move \(t3) in Table \ref{table-Hr/fig} allows us to eliminate
any two $T$'s along the same edge, once they are slided next to each other. If both the
ends of an edge $e$ have good (resp. bad) labels, then $n(e) = 0$ and the entire edge $e$
has a good (resp. bad) label. In the bad case, we insert two $T$'s along $e$ by \(t3) and
slide them near to the ends of $e$. On the contrary, if the labels at the ends of an edge
are one good and the other bad, then there is exactly one $T$ along $e$ and we slide it
near to the badly labeled end.

\begin{Figure}[htb]{equivalence04/fig}
{}{$(i \neq j)$
   [{\sl f}/\pageref{table-Hu/fig}-\pageref{table-Huvprop/fig}, 
    {\sl r}/\pageref{table-Huvdefn/fig}, {\sl t}/\pageref{table-Hr/fig}]}
\centerline{\fig{}{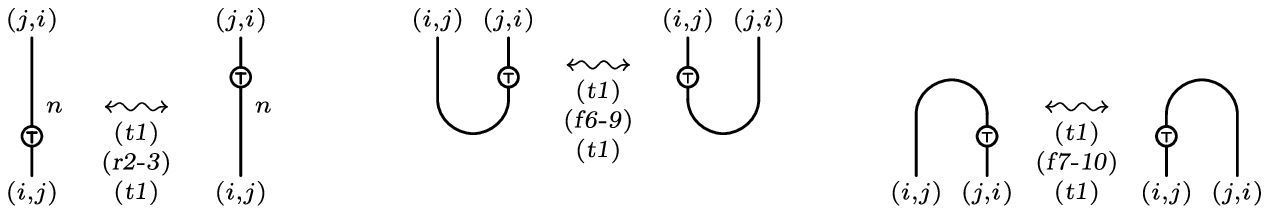}}
\vskip-3pt
\end{Figure}

After that, any bad label in the diagram is confined to a small arc between a $T$ and a
bad multiplication vertex. Taking into account the definition of $\Psi_n$ in Figures
\ref{psi01/fig} and \ref{psi02/fig}, we see that such a diagram is a composition of
products of diagrams each of which is the image of an elementary morphism under the
monoidal functor $\Psi_n$. Therefore, $F$ itself is in the image of $\Psi_n$.
\end{proof}

\begin{proposition}\label{full-psi/thm}
The functor $\Psi_n: \S^c_n \to \H^{r,c}_n$ is full for any $n \geq 3$.
\end{proposition}

\begin{proof}
Let $F$ be a morphism of $\,\H^{r,c}_n$ whose source and target are in $\Psi_n(\Obj
\S^c_n)$. We represent $F$ by a diagram which does not use the copairing and form/coform
notation. This is possible, since those morphisms are defined in terms of the other
elementary morphisms.

Then an edge of the diagram will be called an $i$-{\sl edge}, $1 \leq i \leq n$, if it is
labeled $(i,i)$ and if it does not join a multiplication and a cointegral vertex (cf.
Figure \ref{equivalence05/fig}). Moreover, a vertex will be called an $i$-{\sl vertex} if
it is not cointegral vertex and all edges attached to it are labeled $(i,i)$. In the
figures below we will indicate $i$-edges by thinner lines and $i$-vertices by empty 
triangles.

\begin{Figure}[htb]{equivalence05/fig}
{}{Not an $i$-edge}
\centerline{\fig{}{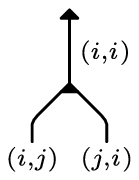}}
\vskip-3pt
\end{Figure}

As a preliminary step, we will show how to transform the diagram representing $F$ into an
equivalent one, where no $i$-vetrices and $i$-edges appear for any $1 \leq i \leq n$.
Actually, the figures below deal only with edges of zero weight and not containing
antipodes, but the generalization to other weights or to the presence of the antipodes is
straightforward. Observe also that since $n \geq 2$, according to Lemma \ref{ij/thm} we
may assume that the diagram of $F$ contains a counit $\epsilon_{(i,j)}$ with $j \neq i$.
Moreover, through isotopy moves such counit can be moved near any given $i$-edge.

We start by eliminating all uni-valent $i$-vertices as described in Figure
\ref{equivalence06/fig}. Then by applying, if necessary, the edge breaking shown in Figure
\ref{equivalence07/fig}, we obtain a diagram where all $i$-edges connect two tri-valent
vertices such that at most one of them is an $i$-vertex. In particular, no $i$-edge
connects two comultiplication vertices.

\begin{Figure}[htb]{equivalence06/fig}
{}{Eliminating the uni-valent $i$-vertices ($i \neq j$)
   [{\sl a-s}/\pageref{table-Hdefn/fig}, {\sl i}/\pageref{table-Hu/fig}]}
\centerline{\fig{}{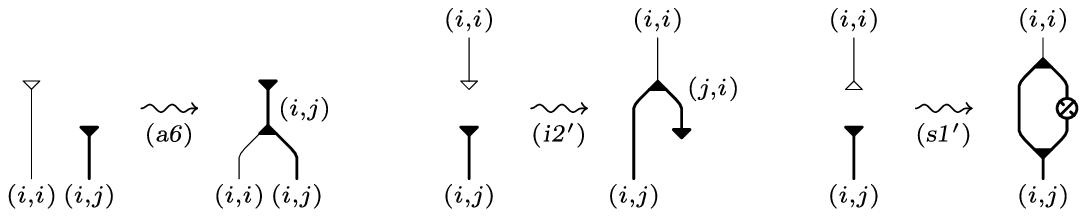}}
\vskip-3pt
\end{Figure}

\begin{Figure}[htb]{equivalence07/fig}
{}{Breaking an $i$-edge ($i \neq j$)
   [{\sl a-s}/\pageref{table-Hdefn/fig}]}
\centerline{\fig{}{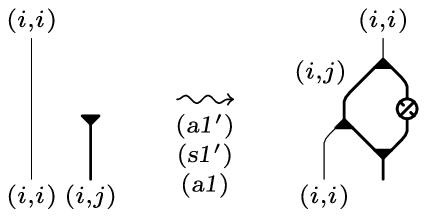}}
\vskip-3pt
\end{Figure}

We proceed by eliminating the tri-valent $i$-vertices, first the comultiplication ones by
the leftmost move in Figure \ref{equivalence08/fig} and then the multiplication ones by
using the two other moves in the same figure (or their vertical reflections).

\begin{Figure}[htb]{equivalence08/fig}
{}{Eliminating tri-valent $i$-vertices ($i \neq j$)
   [{\sl a}/\pageref{table-Hdefn/fig}]}
\centerline{\fig{}{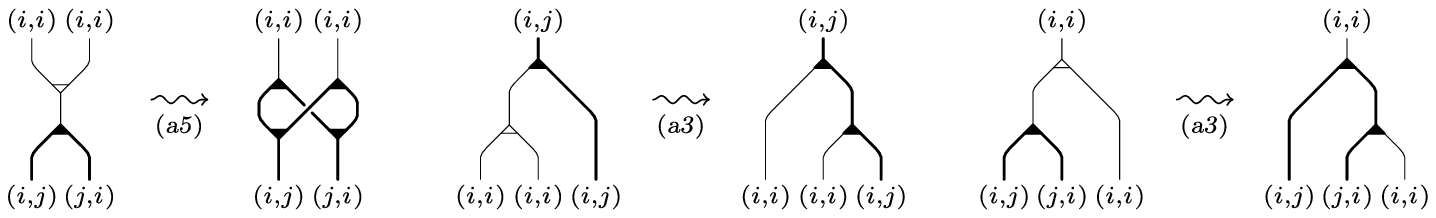}}
\vskip-3pt
\end{Figure}

At this point the only remaining $i$-edges connect two multiplication tri-valent vertices
none of which is an $i$-vertex. Such edges are eliminated through the moves shown in
Figure \ref{equivalence09/fig} (or their vertical reflections). For the move on the right 
side, by applying Lemma \ref{ij/thm} in the case $n \geq 3$, we assume that there is close 
by a counit of label $(j, k)$ with $i \neq j \neq k \neq i$.

\begin{Figure}[hbt]{equivalence09/fig}
{}{Eliminating $i$-edges between two tri-valent vertices none of which is an $i$-vertex 
   ($i \neq j \neq k \neq i$)
   [{\sl a-s}/\pageref{table-Hdefn/fig}]}
\centerline{\fig{}{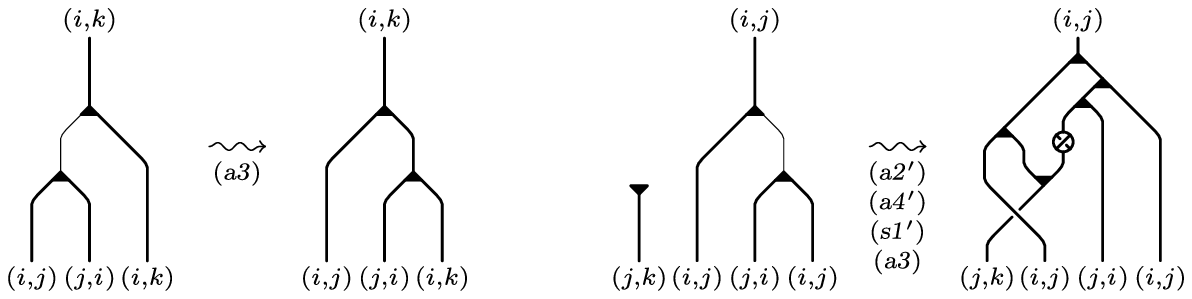}}
\vskip-3pt
\end{Figure}

This gives a diagram, where any edge labeled $(i,i)$ is attached to one multiplication
tri-valent and one cointegral vertex as in Figure \ref{equivalence05/fig} where $i \neq
j$. By using the form notation \(f2) in Table \ref{table-Hu/fig}, we can eliminate those
exceptional edges as well, leaving only labels of the type $(i,j)$ with $i \neq j$.
Finally, we express all antipodes in terms of $T$'s and ribbon morphisms through the
moves \(t1-2) in Table \ref{table-Hr/fig}.

In the end, the resulting diagram of $F$ is a composition of products of the diagrams in
Figure \ref{equivalence02/fig} and the proposition follows from Lemma \ref{psi-image/thm}.

We note that the images of the uni-labeled ribbon intersections do not appear. This is 
because any ribbon surface tangle is equivalent to one which does not contain such ribbon 
interserction.
\end{proof}

\begin{theorem}\label{equivalence/thm}
For any $n \geq 4$, we have the following commutative diagram of equivalence functors:
\vskip3pt
\centerline{\epsfbox{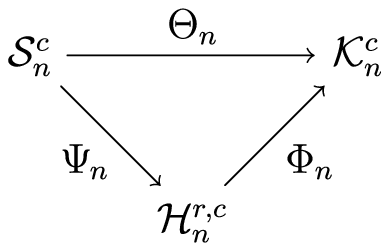}}
\vskip-9pt
\end{theorem}

\begin{proof}
Observe that the commutativity of the diagram has already been estab\-lished in Theorem
\ref{psi/thm} \(c). Moreover, Theorem \ref{ribbon-kirby/thm} tells us that $\Theta_n$ is a
category equivalence for $n \geq 4$. Then, for $n \geq 4$ the functor $\Psi_n$ is
faithful. On the other hand, is it also full, by Proposition \ref{full-psi/thm}.

Therefore, according to Proposition \ref{cat-equiv/thm}, to conclude that $\Psi_n$ is a
category equivalence, and hence so is $\Phi_n$, it is enough to prove that for any object
$H_{\pi_{n \red 1}} \!\diam H_\pi \in \Obj \H^{r,c}_n$ with $\pi = ((i_1,j_1), \dots,
(i_m,j_m)) \in \seq\G_n$ an arbitrary sequence, there exists an isomorphism $\phi_\pi:
H_{\pi_{n \red 1}} \!\diam H_\pi \to H_{\pi_{n \red 1}} \!\diam H_{\pi'}$ with $H_{\pi_{n
\red 1}} \!\diam H_\pi' \in \Psi_n(\Obj \S^c_n)$, that is $i'_k > j'_k$ for $k = 1, \dots,
m$.

\begin{Figure}[b]{equivalence10/fig}
{}{Definition of $\phi_\pi$ when $i_k = j_k < n$}
\centerline{\fig{}{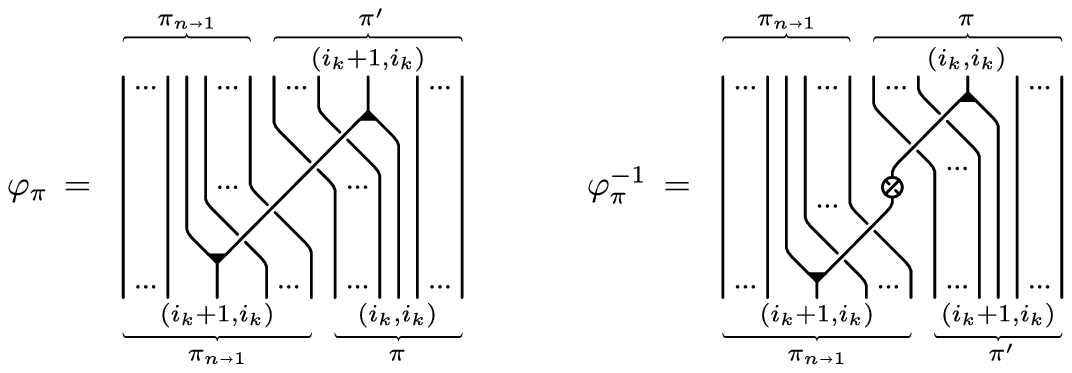}}
\vskip-6pt
\end{Figure}

\begin{Figure}[htb]{equivalence11/fig}
{}{Definition of $\phi_\pi$ when $i_k = j_k = n$.}
\vskip3pt
\centerline{\fig{}{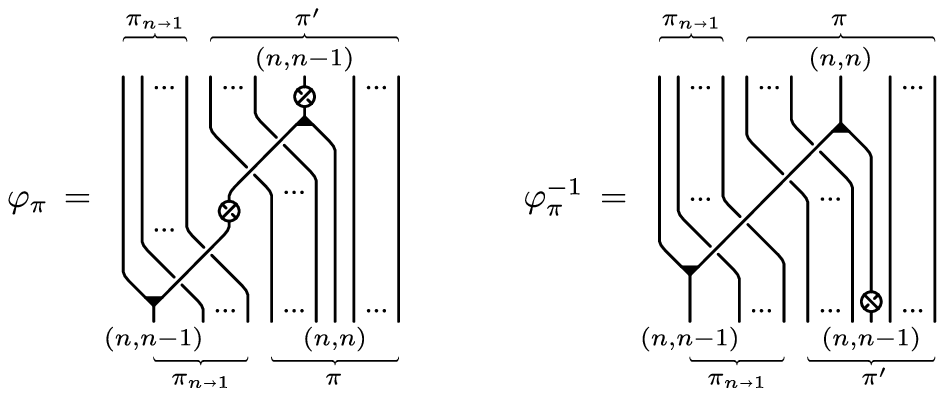}}
\vskip-6pt
\end{Figure}

We call an element $(i_l,j_l) \in \pi$ {\sl bad} if $i_l \leq j_l$. Then, we proceed by
induction on the number of bad elements of $\pi$. The inductive step is provided by the
following claim: if the number of bad elements in $\pi$ is $s > 0$, then there is an
isomorphism $\phi_\pi: H_{\pi_{n \red 1}} \!\diam H_\pi \to H_{\pi_{n \red 1}} \!\diam
H_{\pi'}$, where the number of bad elements in $\pi'$ is $s-1$.

To prove the claim, suppose that $m > 1$ and that the first bad element in $\pi$ is
$(i_k,j_k)$. If $i_k < j_k$ the desired isomorphism $\phi_\pi$ is given by the product of
identity morphisms and the antipode applied on the $k$-th string of $H_\pi$. If $i_k = j_k
< n$ the isomorphism $\phi_\pi$ and its inverse $\phi_\pi^{-1}$ are presented in Figure
\ref{equivalence10/fig}, while if $i_k = j_k = n$ they are presented in Figure
\ref{equivalence11/fig}. The identities $\phi_\pi^{-1} \circ \phi_\pi = \id_\pi$ and
$\phi_\pi \circ \phi_\pi^{-1} = \id_{\pi'}$ directly follow from the axioms \(s1-1') in
Table \ref{table-Hdefn/fig}.
\end{proof}

We are ready now to state the main theorem of this section, proving that the universal
algebraic category $\H^r = \H^r_1= \H^{r,c}_1$ is equivalent to the category of relative
cobordisms of 4-dimensional 2-handlebodies. The elementary diagrams and defining relations
are of $\H^r$ have been collected in Tables \ref{table-Hr1diags/fig} and
\ref{table-Hr1axioms/fig}. Of course, they are simply the specializations of the ones in
Tables \ref{table-Hdefn/fig}, \ref{table-Hu/fig}, \ref{table-Huvdefn/fig} and
\ref{table-Hr/fig} to the case of the trivial groupoid $\G_1$. In this case, all labels 
are equal to $(1,1)$, and are therefore omitted. 
Moreover, we remind that in this case the adjoint morphisms $\alpha:H\diam H\to H$ (resp.
$\alpha':H \diam H \to H$) define the left (resp. right) adjoint action of $H$ on itself,
and the two addition ribbon axioms \(r8) and \(r9) can be equivalently expressed in terms
of such action as shown in Table \ref{table-Hr1adjoint/fig} (cf. Proposition
\ref{adjoint/thm}).

\begin{theorem}\label{alg-kirby-eq/thm}
The functor $\Phi_n: \H_n^{r,c} \to \K_n^c$ is a category equivalence for any $n\geq 1$.
In particular, the universal ribbon Hopf algebra $\H^r$ is equivalent to the category
$\,\Chb^{3+1}$ of relative cobordisms of 4-dimensional 2-handle\-bodies.
\end{theorem}

\begin{proof}
According to Theorem \ref{equivalence/thm}, $\Phi_n$ is a category equivalence for any $n
\geq 4$. Then, for any $1 \leq n \leq 3$ the commutative diagram in Proposition
\ref{Hn-reduction/thm}, implies that $\Phi_n \circ {\down_1^4} = {\down_1^4} \circ
\Phi_4$, where the reduction functors are known to be category equivalences by
Propositions \ref{K-reduction/thm} and \ref{H-reduction/thm}. Therefore $\Phi_n$ is a
category equivalence as well. In particular, for $n = 1$ we obtain that $\Phi_1: \H^r =
\H^r_1 \to \K = \K_1$ is a category equivalence. But according to Theorem
\ref{K-category/thm}, $\K$ is equivalent to $\Chb^{3+1}$, hence $\H^r$ is equivalent to
$\Chb^{3+1}$ as well.
\end{proof}

\begin{Table}[htb]{table-Hr1diags/fig}
{}{}
\centerline{\fig{}{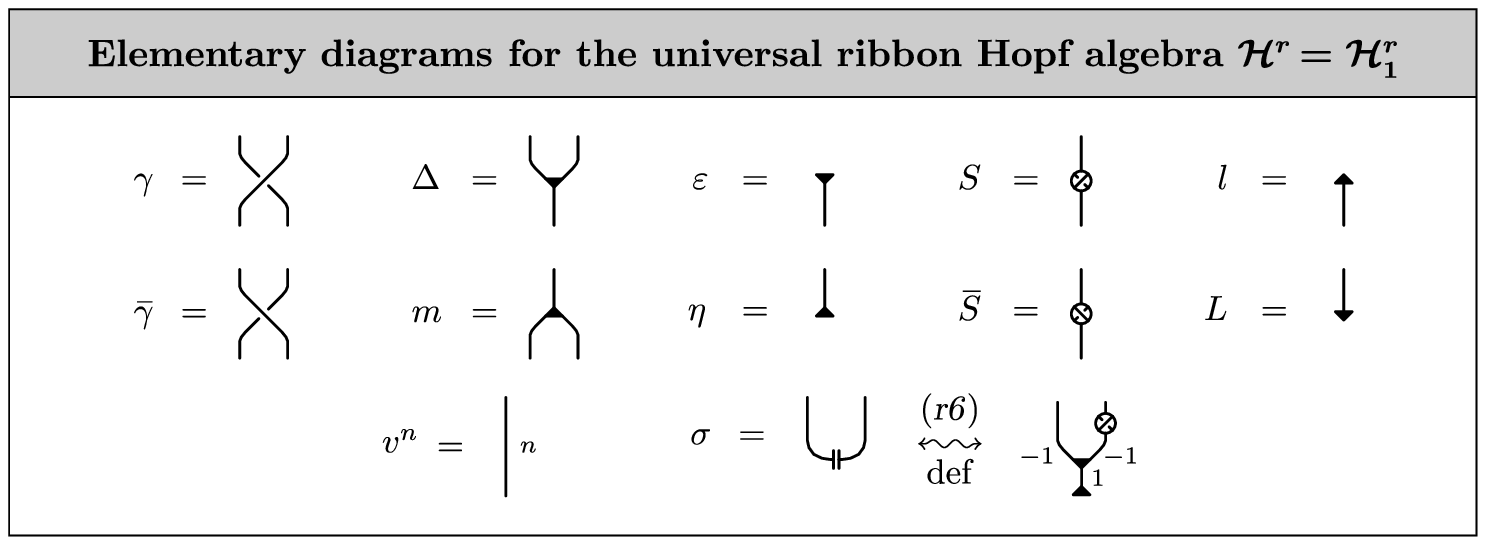}}
\vskip-3pt
\end{Table}

\begin{Table}[p]{table-Hr1axioms/fig}
{}{}
\centerline{\fig{}{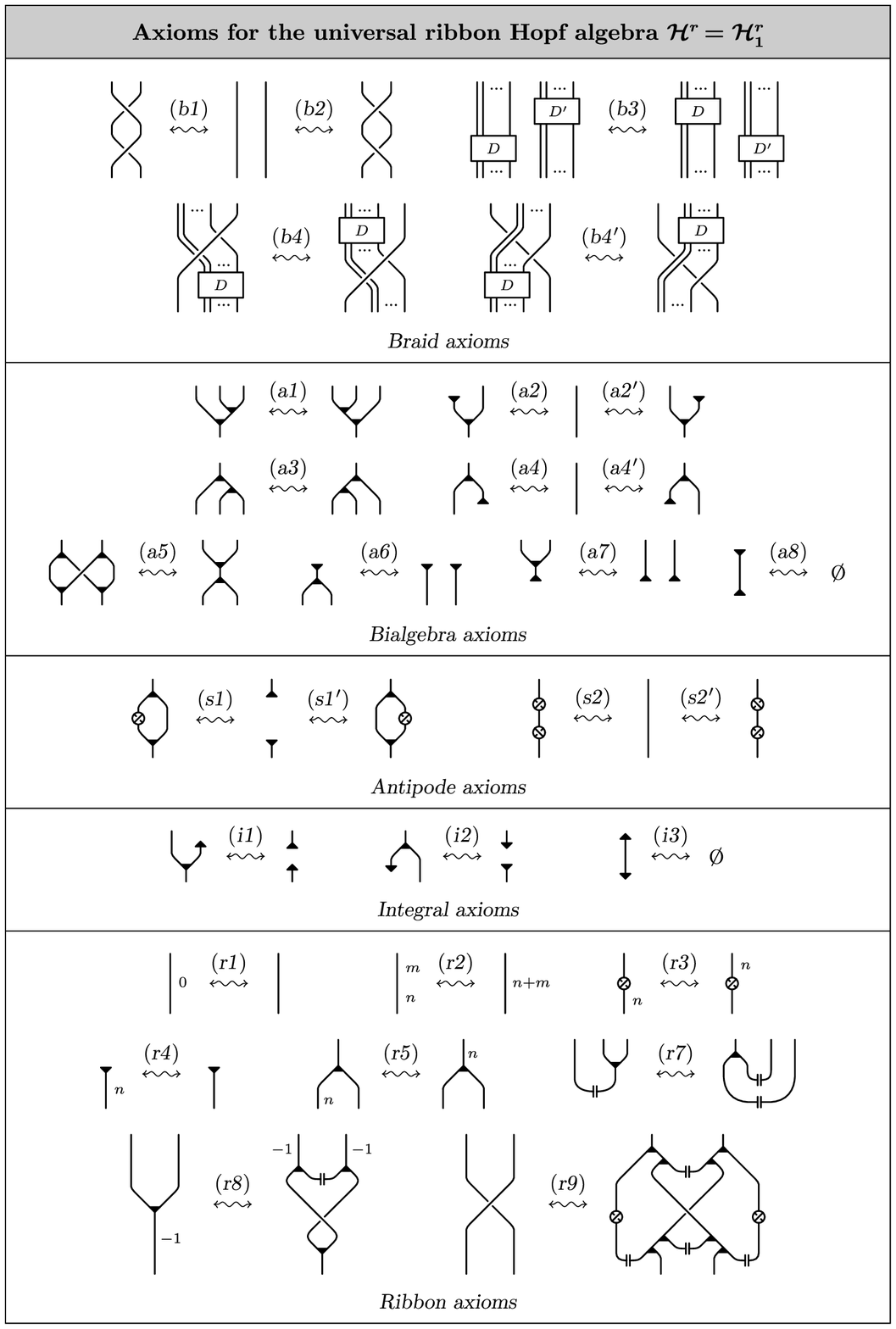}}
\vskip-3pt
\end{Table}

\begin{Table}[htb]{table-Hr1adjoint/fig}
{}{}
\centerline{\fig{}{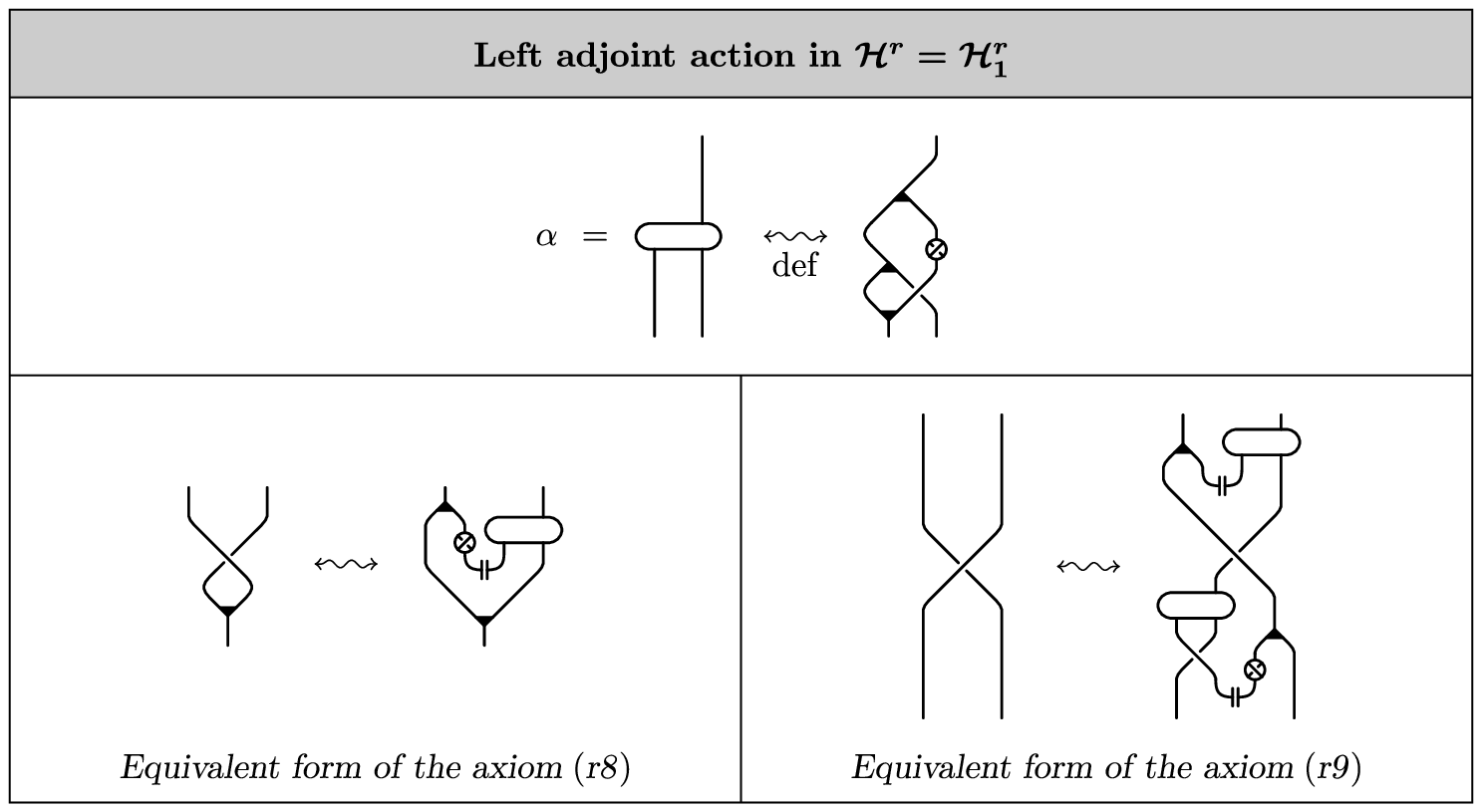}}
\vskip-3pt
\end{Table}

\begin{remark}
The category equivalence established in Theorem \ref{alg-kirby-eq/thm} implies that one
can use braided ribbon Hopf algebras to construct sensitive invariants of 4-dimensional
2-handlebodies. Of course, in order to be able to calculate such invariants, given a Kirby
tangle $K: I_{m_0} \to I_{m_1}$ in $\K = \K_1$, we need an explicit form for the morphism
$F_K$ in $\H^r$ such that $\Phi_1(F_K) = K$. One such form is given by $F_K = \down^3_1
\Psi (S_K)$, where $S_K$ was constructed in Section \ref{fullness/sec}, while $\Psi_n$ and
$\down^3_1$ were defined respectively in Section \ref{Psi/sec} and Section
\ref{reduction/sec}. Of course, such approach requires a good understanding of all
functors used in it, and it is quite long in practice.

\pagebreak

We will outline here a simpler procedure for constructing $F_K$, which is less explicit
but works quite well in many concrete examples.

We start by representing $K$ by a strictly regular planar diagram and use the following
objects (and notations) from the indicated steps of the construction of the surface $S_K$
in Section \ref{fullness/sec}:
\begin{itemize}
\item[1)] the disk $D_1, \dots, D_r$ spanned by the dotted unknots of $K$ and framed link 
$L = L_1 \cup \dots \cup L_s$ introduced in step 1;
\item[2)] a trivial state for the diagram of the unframed link $|L|$ and the framed link
$L' = L'_1 \cup \dots \cup L'_s$, whose corresponding unframed link $|L'|$ is represented
by that trivial state, together with the family of cylinders $F_1, \dots, F_l$ inside
which the trivializing crossing changes take place, all considered in step 3;
\item[3)] a family $A_1, \dots, A_s$ of disjoint spanning disks for the components
$|L'_1|, \dots, |L'_s|$ of the trivial link $|L'|$, as in step 6.
\end{itemize}
The $A_j$'s can be assumed to intersect each cylinder $F_i$ as depicted in the left side
of Figure \ref{xi01/fig} (disregard the disk $C_i$). Then, we can extend them by
introducing one clasp inside each $F_i$, to give a family $\hat A_1, \dots, \hat A_s$ of
spanning disks for the original link $|L|$. Actually, in doing that we possibly introduce
other singularities, which consist of two double arcs meeting at a triple point along the
clasp intersection for each horizontal subdisk of the $A_j$'s inside the $F_i$'s. By
suitable finger moves at the interior of the disks, we can eliminate all the triple points
to leave only ribbon intersections, and then transform each ribbon intersection into a
pair of new clasps.

Moreover, we can assume that the $\hat A_j$'s form with the $D_i$'s only clasp and ribbon
intersections, like in the right side of Figure \ref{xi01/fig}. Also in this case, each
ribbon intersection can be transformed into a pair of new clasps, by a finger move as 
above.

Finally, if $L_j$ is the closure of an open framed component in $K$ joining $a'_{m_c,i_j}$ 
to $a''_{m_c,i_j}$ with $c = 0,1$, then we assume that the portion of $\hat A_j$ outside 
the box $K$ in Figure \ref{kirby-ribbon02/fig} is flat and cut off it (cf. 
Figure \ref{phi-inv01/fig} \(a)).

We still denote by $D_1, \dots, D_r$ and $\hat A_1, \dots, \hat A_s$ the disks resulting 
after the above modifications, which form only clasp singularities as the ones in Figure 
\ref{phi-inv01/fig} \(b).

\begin{Figure}[htb]{phi-inv01/fig}
{}{}
\centerline{\fig{}{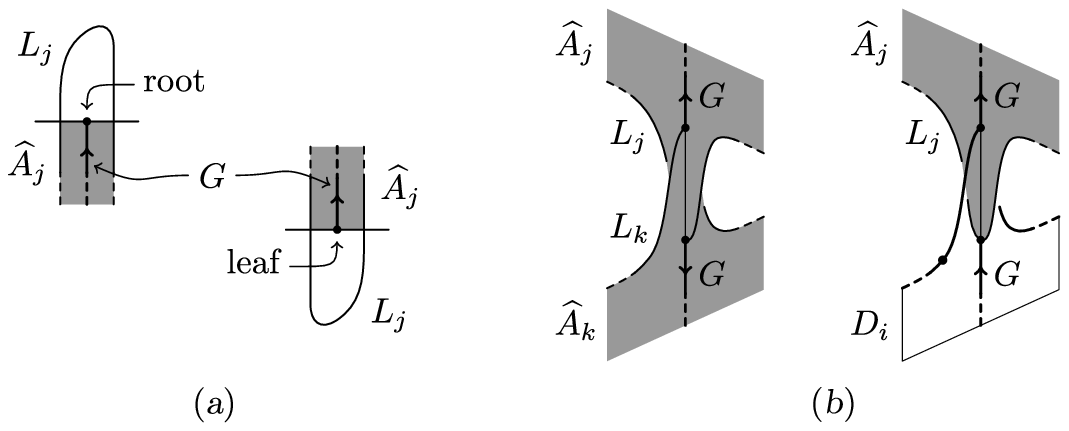}}
\vskip-3pt
\end{Figure}

Now, we consider an oriented graph $G$ consisting of a rooted uni/tri-valent tree embedded
in each of such disks. For the tree $T$ inside any disk, we require that (cf. Figure
\ref{phi-inv01/fig}): the root of $T$ belongs to the boundary of the disk and coincides
the middle point of the segment at the target level if such segment exists (in the case of
a disk $A_j$); the leaves of $T$ include all the end points of the clasp arcs in the
interior of the disk and the middle point of the segment at the source level if such
segment exists (in the case of a disk $A_j$); $T$ does not meet the boundary of the disk
and the clasp arcs at any point other than those already mentioned. Moreover, we orient
the edges of $T$ towards the leaves in the case of a disk $D_i$ and towards the root in
the case of a disk $\hat A_j$.

Let $\bar G$ be the oriented graph obtained by adding to $G$ all the clasp arcs as new
oriented edges. The orientation of each new edge goes from the tree of $D_i$ to that of
$\hat A_j$ if the corresponding clasp involves such disks, while it is arbitrary otherwise
(if the clasp involves two $A_j$'s).

\begin{Figure}[b]{phi-inv02/fig}
{}{}
\centerline{\fig{}{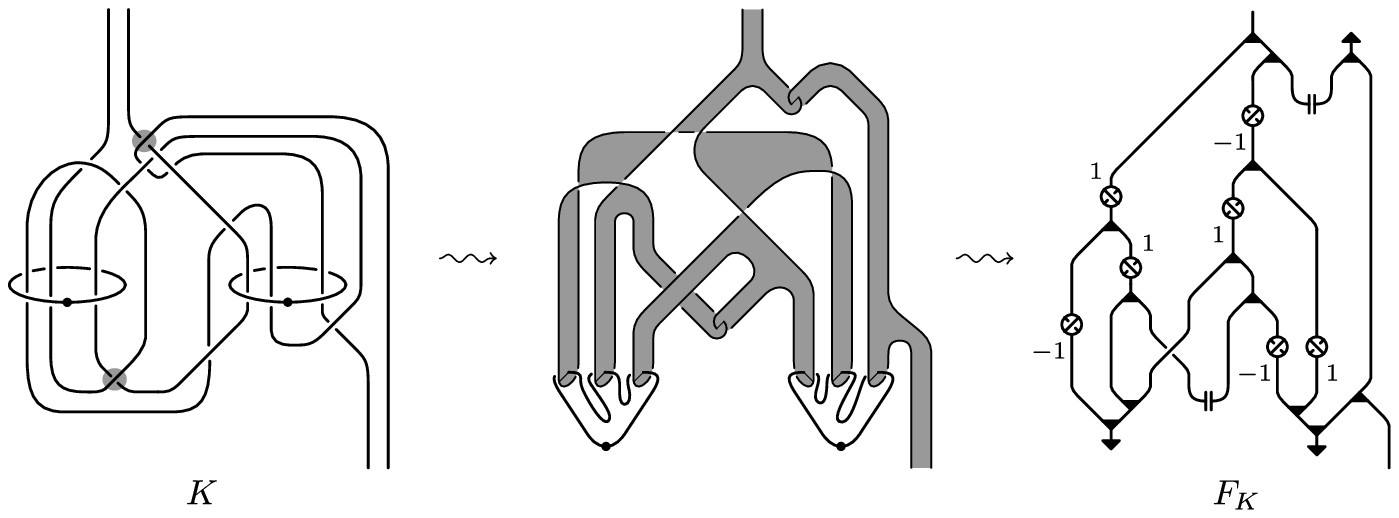}}
\vskip-3pt
\end{Figure}

At this point, a suitable ambient isotopy of $E \times [0,1]$ fixing $E \times \{0,1\}$
allows us to put the oriented graph $\bar G$ in regular position with respect to the
projection plane, in such a way that $y$-coordinate is increasing along the projection of
each oriented edge, and then to deform the union $D_1 \cup \dots \cup D_r \cup \hat A_1
\cup \dots \cup \hat A_s$ to a narrow regular neighborhood $N$ of $\bar G$ in it.

\pagebreak

As a result of such ambient isotopy, the original planar diagram of $K$ is transformed
into an equivalent planar diagram, which can be easily expressed as a composition of
expansions of Kirby tangles as in Figures \ref{phi01/fig} and \ref{phi02/fig}, after
insertion of canceling pairs of 1/2-handles.

Then, the morphism $F_K$ is given by the analogous product of expansions of the
corresponding elementary morphisms of $\H^r_1$ in the same figures, whose images under
$\Phi_1$ are those Kirby tangles. In particular, each claps intersections between two
$A_j$'s gives rise to a copairing morphism in $F_K$.

Figure \ref{phi-inv02/fig} shows the morphism $F_K$ for a very simple Kirby tangle $K$.
This is represented by the leftmost diagram, where the trivializing crossings are
encircled by a small gray disk. The intermediate diagram represent the regular
neighborhood $N$ of the graph $G$ in the construction above.
\end{remark}

\newpage

\section{3-dimensional cobordisms as boundaries%
\label{boundaries/sec}}

This chapter is mainly aimed to prove the Kerler's conjecture (stated in \cite{Ke02}),
that the category of 2-framed 3-dimensional relative cobordisms $\Cobt^{2+1}$ admits a
purely algebraic characterization in terms of a universal algebraic category generated by 
a Hopf algebra object.

Here, we think of $\Cobt^{2+1}$ as the quotient category of $\Chb_1^{3+1}$ modulo
1/2-handle trading. This point of view provides a canonical monoidal functor $\Chb_1^{3+1}
\to \Cobt^{2+1}\,$, through which we will derive the proof of the Kerler's conjecture from
the results of the previous chapter.

An analogous algebraic characterization will be also given for the category of
3-dimensional relative cobordisms $\Cob^{2+1}$, as the further quotient of
$\Cobt^{2+1}$ modulo blowing down/ups.

\subsection{The categories of relative cobordisms $\bfCobt^{2+1}_n$ and $\Cob^{2+1}_n$%
\label{Cob/sec}}

We recall from Section \ref{Chbn/sec} that the objects of $\,\Chb_n^{3+1}$ have been
restricted to be the standard 3-dimensional 1-handlebodies $M^n_\pi$ described in
Definition \ref{standard-hb/def}, where $n$ is the number of 0-handles and $\pi \in \seq
\G_n$. To such an $M^n_\pi$ is associated its front boundary $F^n_\pi = \partial M^n_\pi$,
which is a standard compact oriented surface with marked $(S^1_1 \sqcup \dots \sqcup
S^1_n)$-boundary. In fact, $\Bd F^n_\pi$ has $n$ components, one for each 0-handle of
$M^n_\pi$, and the given numbering of these 0-handles induces the marking of $\Bd F^n_\pi$
(which depends only on $n$ and not on $\pi$).

Moreover, for any morphism $W: M_0 \to M_1$ in $\,\Chb_n^{3+1}$, which is a 4-dimensional
relative 2-handlebody build on $X(M_0,M_1)$, the front boundary $\partial W$ can be seen
as a 3-dimensional relative cobordism $\partial W: \partial M_0 \to \partial M_1$ between
surfaces with marked boundary. Actually, $X(M_0,M_1)$ itself can be seen as such a 
3-dimensional cobordism, and $W$ represents a relative cobordism, build up with only 1- 
and 2-handles, between $X(M_0,M_1)$ and $\partial W$ as oriented 3-manifolds with marked
boundary.

\medskip

Now, consider the handle trading and blowing up/down moves on $W$, whose description in
terms of $n$-labeled Kirby tangle is given in Figure \ref{boundary01/fig}. These are
nothing else than the labeled versions of the well-known Kirby calculus moves and like
them preserve the boundary of $W$. In particular, they preserve the 3-dimensional relative
cobordism $\partial W: \partial M_0 \to \partial M_1$.

\begin{Figure}[htb]{boundary01/fig}
{}{Kirby calculus moves}
\centerline{\fig{}{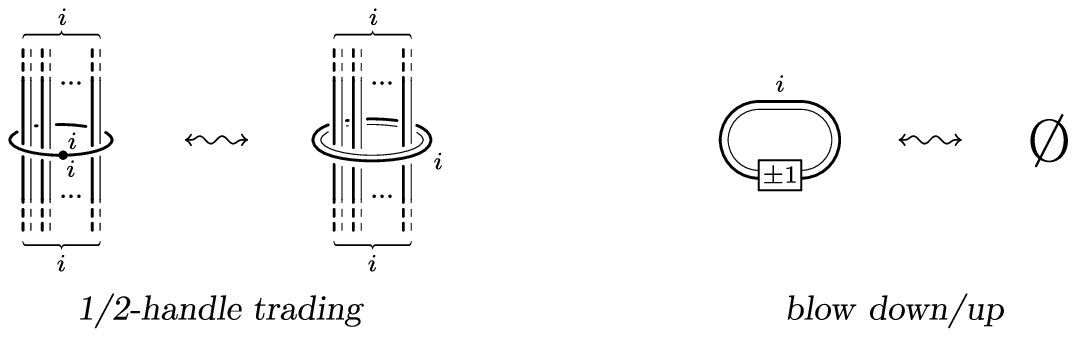}}
\vskip-3pt
\end{Figure}

As discussed in Section \ref{handles/sec}, a 1/2-handle trading changes $W$ into a new
4-dimen\-sional relative 2-handlebody $\bar W$, which is related to $W$ by a relative
5-dimensional cobordism of 4-manifolds with marked boundary. Then, 1/2-handle trading also
preserves the signature of $W$, i.e. $\sigma(\bar W) = \sigma(W)$. According to
\cite{At90} (cf. \cite{Wa91, Ke99, KL01}), the extra information given by $\sigma(W)$ can
be interpreted as a 2-framing on $\partial W$, i.e. a homotopy class of trivializations
of $TM \oplus TM$ as a $\Spin(6)$ bundle. So, we can say that 1/2-handle trading preserves
$\partial W$ as a 2-framed 3-dimensional relative cobordism. On the contrary, a blowing up
transforms $W$ into the connected sum with $W \mathbin{\#} \pm CP^2$, hence it changes the
signature by $\pm1$.

\medskip

In the light of the above considerations, for any $n \geq 1$ we define the category of 
2-framed 3-dimensional relative cobordisms $\Cobt_n^{2+1}$ and the category of 
3-dimensional relative cobordisms $\Cob_n^{2+1}$, directly as quotients of $\Chb_n^{3+1}$.

Namely, the objects of both the categories $\Cobt_n^{2+1}$ and $\Cob_n^{2+1}$ are those of
$\Chb_n^{3+1}$. The morphisms of $\Cobt_n^{2+1}$ are equivalence classes of morphisms of
$\Chb_n^{3+1}$ under the equivalence relation generated by 1/2-handle trading, while the
morphisms of $\Cob_n^{2+1}$ are equivalence classes of morphisms of $\Chb_n^{3+1}$ under
the equivalence relation generated by 1/2-handle trading and blowing down/up.

Notice that composition and product of cobordisms are preserved by 1/2-handle trading and
blow down/up moves, since these take place in the interior leaving the boundary unchanged.
Therefore, $\Cobt_n^{2+1}$ and $\Cob_n^{2+1}$ inherit from $\Chb_n^{3+1}$ a strict
monoidal structure, whose product will be still denoted by $\,\diam\,$, and we have
monoidal quotient functors $$\Chb_n^{3+1} \to \,\Cobt_n^{2+1} \to \,\Cob_n^{2+1}\,.$$
\medskip

In particular, the categories $\Cobt^{2+1} = \Cobt_1^{2+1}$ and $\Cob^{2+1} =
\Cob_1^{2+1}$ are respectively equivalent to the categories $\widetilde{\text{\boldmath
$Cob$}}$ and {\boldmath $Cob$} introduced in \cite{KL01}, admitting the same presentations
in terms of elementary morphisms and relations given for those categories in Chapter 4 of
\cite{KL01} (cf. \cite{Ke02}, where the notation {\boldmath $\cal Cob$} is used for
$\Cobt^{2+1}$).

\medskip

The definitions of $\Cobt_n^{2+1}$ and $\Cob_n^{2+1}$ as quotients of $\Cob_n^{3+1}$ are 
enough for our present aim to prove the equivalence of $\Cobt^{2+1}$ and $\H^r$.

Such definitions are motivated by the fact that the classical Lickorish-Rohlin-Wallace's
theorem \cite{Li62,Ro51,Wa60} asserting that any closed oriented 3-manifold is the
boundary of a 4-dimensional 2-handlebody, and the Kirby calculus \cite{Ki78} relating any
two such handlebodies with diffeomorphic boundaries, can be adapted to the context of the
4-dimensional relative 2-handlebody cobordisms in $\Cob_n^{3+1}$. As a consequence,
$\Cobt_n^{2+1}$ and $\Cob_n^{2+1}$ can be identified as the categories of all the
(2-framed) 3-dimen\-sional cobordisms between surfaces with $n$-component marked boundary
and given 3-dimensional filling, up to diffeomorphisms preserving (the 2-framing and) the
fillings. In fact, the case of $n = 1$ is treated in \cite{KL01} (cf. comment 1. in 0.4.1
at page 8 of the same reference for $n > 1$).

\subsection{The quotient categories $\Kb_n$, $\Kbb_n$, $\Sb_n$ and $\Sbb_n$%
\label{quotients/sec}}

We want to define the quotient categories $\Kb_n$ and $\Kbb_n$ of $\K_n$ corresponding 
to $\Cobt_n^{2+1}$ and $\Cob_n^{2+1}$ respectively. Before doing that, we see how to 
eliminate some redundancy in the 1/2-handle trading and blow down/up moves.

\begin{lemma}\label{red-bdmoves/thm}
Modulo 2-handle sliding, any 1/2-handle trading can be reduced to one where the involved 
1-handle is a canceling one, as in Figure \ref{boundary03/fig} \(a). Modulo 2-handle 
sliding and 1/2-handle trading, positive and negative blow down/up are inverse to one 
another.
\end{lemma}

\begin{proof}
See Figure \ref{boundary02/fig}.
\end{proof}

\begin{Figure}[htb]{boundary02/fig}
{}{Proof of Lemma \ref{red-bdmoves/thm}}
\vskip-6pt
\centerline{\fig{}{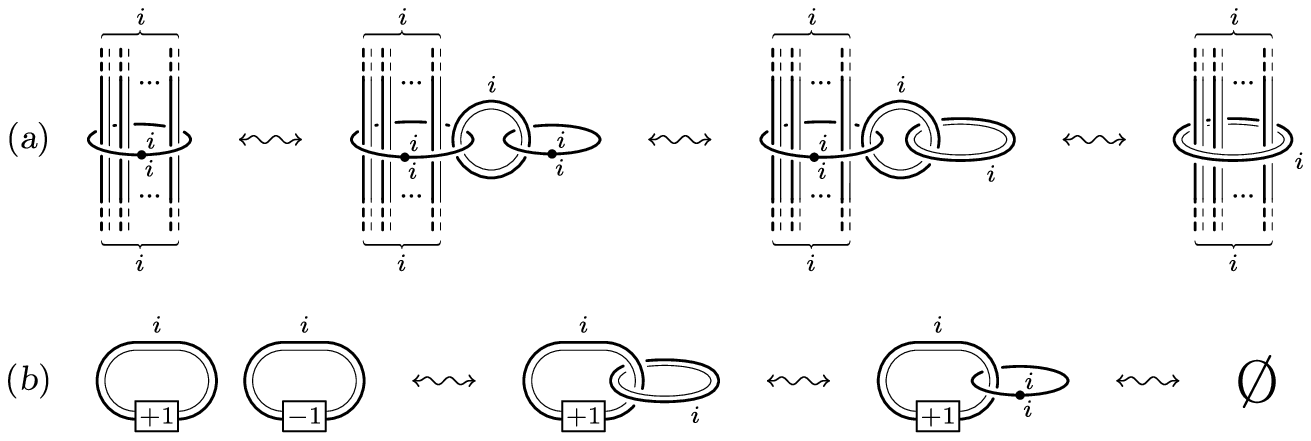}}
\vskip-3pt
\end{Figure}

According to the previous lemma, we define $\Kb_n$ and $\Kbb_n$ as follows. The objects of
both $\Kb_n$ and $\Kbb_n$ coincide with those of $\K_n$. The morphisms of $\Kb_n$
are equivalence classes of morphisms of $\K_n$ modulo the move in Figure
\ref{boundary03/fig} \(a), while the morphisms of $\Kbb_n$ are equivalence classes
of morphisms of $\K_n$ modulo both moves in the same Figure \ref{boundary03/fig}.

\begin{Figure}[htb]{boundary03/fig}
{}{Additional moves in $\Kb_n$ and $\Kbb_n$}
\vskip-3pt
\centerline{\fig{}{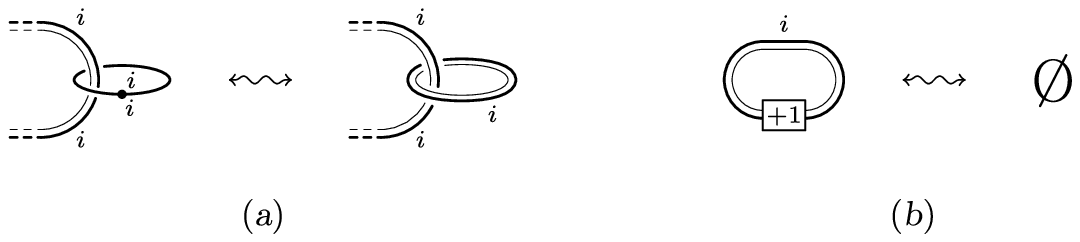}}
\vskip-3pt
\end{Figure}

Since the moves in Figure \ref{boundary03/fig} involve only closed components of the 
diagram, they preserve composition and product of $n$-labeled Kirby tangles. 
Hence, $\Kb_n$ and $\Kbb_n$ inherit from $\K_n$ a strict monoidal structure, whose product 
will be still denoted by $\,\diam\,$, and we have monoidal quotient functors
$$\K_n \to \,\Kb_n \to \,\Kbb_n\,.$$
\medskip

\begin{proposition}\label{Kb-category/thm}
The equivalence of monoidal categories $\,\K_n \cong \Chb_n^{3+1}$ given by Proposition
\ref{K-category/thm} induces equivalences of monoidal categories on the quotients $\,\Kb_n
\cong \Cobt_n^{2+1}$ and $\,\Kbb_n \cong \Cob_n^{2+1}$.
\end{proposition}

\begin{proof}
This is an immediate consequence of Proposition \ref{K-category/thm} and Lemma
\ref{red-bdmoves/thm}, taking into account that the labeled Kirby tangle moves in Figures
\ref{boundary02/fig} and \ref{boundary03/fig} already interpret 1/2-handle trading and
blow down/up under the equivalence of $\K_n$ and $\Chb_n^{3+1}$ in Proposition 
\ref{K-category/thm}.
\end{proof}

Given $n > k \geq 1$, we define the subcategories $\Kb_{n \red k} \subset \Kb_n$ and
$\Kbb_{n \red k} \subset \Kbb_n$ to be the images of $\K_{n \red k} \subset \K_n$ under
the quotient functors. Since the additional moves in Figure \ref{boundary03/fig} involve
only uni-labeled components, when applied to a reducible Kirby tangle $K$ as in Figure
\pagebreak
\ref{kirby-stab02/fig}, they essentially take place inside the box $L$.
Then, $\Kb_{n \red k}$ and $\Kbb_{n \red k}$ can also be though as quotients of $\K_{n
\red k}$, and we have the following commutative diagram of inclusion and quotient
functors.
$$
\begin{array}{ccccc}
\K_{n \red k} & \to & \Kb_{n \red k} & \to & \Kbb_{n \red k}\\[2pt]
\cap & & \cap & & \cap\\[2pt]
\K_n & \to & \Kb_n & \to & \Kbb_n\end{array}
$$

\begin{proposition}\label{Kb-reduction/thm}
For any $n > k \geq 1$, the category equivalences $\up_k^n: \K_k \to \K_{n \red k}$ and
$\down_k^n: \K_{n \red k} \to \K_k$ given by the stabilization and reduction functors
induce well-defined functors on the quotient categories
\vskip-16pt
\begin{eqnarray*}
&\up_k^n: \Kb_k \to \Kb_{n \red k} \text{ \ and }\,
&\up_k^n: \Kbb_k \to \Kbb_{n \red k}\,,\\[2pt]
&\down_k^n: \Kb_{n \red k} \to \Kb_k \text{ \ and }\,
&\down_k^n: \Kbb_{n \red k} \to \Kbb_k\,.
\end{eqnarray*}
Moreover, $\down_k^n \circ \up_k^n$ is equal to $\id_{\Kb_k}$ (resp. $\id_{\Kbb_k}$) 
while $\up_k^n \circ \down_k^n$ is naturally equivalent to $\id_{\Kb_{n \red k}}$ (resp. 
$\id_{\Kbb_{n \red k}}$). Then, $\up_k^n$ and $\down_k^n$ are category equivalence.
\end{proposition}

\smallskip

\begin{proof}
By Proposition \ref{K-reduction/thm}, $\down_k^n \circ \up_k^n = \id_{\K_k}$
while $\up_k^n \circ \down_k^n$ is naturally equivalent to $\id_{\K_{n \red k}}$.
Therefore, the statement will follow once we show that $\up_k^n$ and $\down_k^n$ are
well defined on the quotient categories.

That the stabilization functor $\up_k^n$ is well defined on the quotients is 
straightforward, since by definition $\up_k^n K = \id_{n \red k} \!\diam K$ for any $K \in 
\K_k$.

Concerning the reduction functor $\down_k^n$, it has been defined as composition of
elementary reduction functors (cf. Definition \ref{K-reduction/def}). Then, it is enough
to consider the case of $\down_{n-1}^n$. In this case, for any $K \in \K_{n \red (n-1)}$
we have by definition:
$$\down_{n-1}^n K = (\epsilon_{(n-1,n-1)} \diam \id_{\pi_1^{(n,n-1)}}) \circ K^{(n,n-1)} 
\circ (\eta_{n-1} \diam \id_{\pi_0^{(n,n-1)}})\,,$$
where $K^{(n,n-1)}$ is obtained from $K$ by pulling all tangle components labeled $n$
above the ones labeled $n-1$ and then changing the label $n$ with $n-1$ (cf. proof of
Lemma \ref{K-push/thm}). Since both the additional relations in $\Kb_k$ and $\Kbb_k$
involve only uni-labeled sub-tangles, they are obviously preserved by the map $K \mapsto
K^{(n,n-1)}$, hence $\down_{n-1}^n K$ is well defined, being the composition of
$K^{(n,n-1)}$ with fixed tangles.
\end{proof}

Now we define the quotient categories $\Sb_n$ and $\Sbb_n$ of the category $\S_n$ of 
$n$-labeled ribbon surface tangles, and make some considerations about them analogous to 
those made above for $\Kb_n$ and $\Kbb_n$. In the next section, we will establish the 
relation between $\Sb_n$ (resp. $\Sbb_n$) and $\Kb_n$ (resp. $\Kbb_n$).

\medskip

Once again, the objects of $\Sb_n$ and $\Sbb_n$ are those of $\S_n$. The morphisms of
$\Sb_n$ are equivalence classes of morphisms of $\S_n$ modulo the relation \(T) in Figure
\ref{boundary04/fig}, while the morphisms of $\Sbb_n$ are equivalence classes of morphisms
of $\S_n$ modulo both the relations \(T) and \(P) in the same Figure \ref{boundary04/fig}.
Notice that such moves do not change the boundary of the ribbon surface tangle up to 
isotopy. 

The moves in Figure \ref{boundary04/fig} clearly preserve composition and product of 
$n$-labeled ribbon surface tangles, being local in nature. Hence, $\Sb_n$ and $\Sbb_n$ 
inherit from $\S_n$ a strict monoidal structure, whose product will be still denoted by 
$\,\diam\,$, and we have monoidal quotient functors
$$\S_n \to \,\Sb_n \to \,\Sbb_n\,.$$

\begin{Figure}[htb]{boundary04/fig}
{}{Additional relations in $\Sb_n$ and $\Sbb_n$}
\centerline{\fig{}{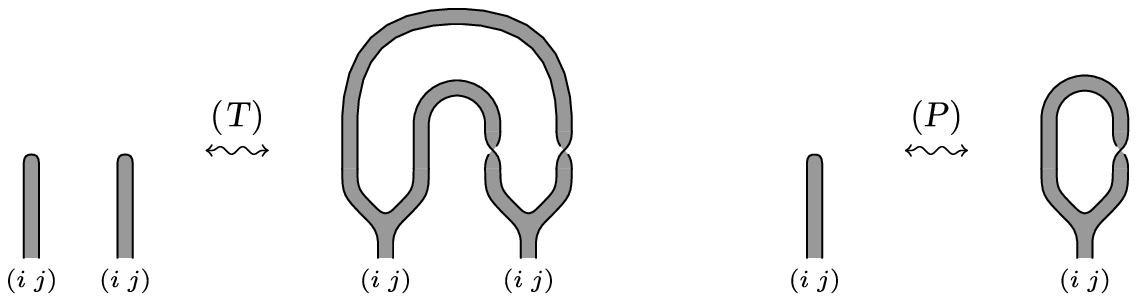}}
\vskip-3pt
\end{Figure}

Given $n > k \geq 1$, we define the subcategories $\Sb_{n \red k} \subset \Sb_n$ and
$\Sbb_{n \red k} \subset \Sbb_n$ to be the images of $\S_{n \red k} \subset \S_n$ under
the quotient functors. Thus, we have quotient functors
$$\S_{n \red k} \to \Sb_{n \red k} \to \Sbb_{n \red k}$$

\medskip

\begin{proposition}\label{Sb-stabilization/thm}
For any $n > k \geq 2$ the functor $\up_k^n: \S_k \to \S_n$ induces well-defined functors 
on the quotient categories, for which we use the same notation:
\begin{eqnarray*}
&\up_k^n: \Sb_k \to \Sb_n \text{ \ and }\,
&\up_k^n: \Sbb_k \to \Sbb_n\,.
\end{eqnarray*}
\vskip-12pt
\end{proposition}

\begin{proof}
The statement is straightforward.
\end{proof}

\subsection{Equivalences $\Kb_n^c \cong \Sb_n^c$ and $\Kbb_n^c \cong \Sbb_n^c$ 
            for $n \geq 4$%
\label{Kb=Sb/sec}}

According to the notation introduced in Sections \ref{K/sec} and \ref{S/sec}, we put
$\Kb_n^c = \Kb_{n \red 1}$, $\Kbb_n^c = \Kbb_{n \red 1}$, $\Sb_n^c = \Sb_{n \red 1}$ and
$\Sbb_n^c = \Sbb_{n \red 1}$.

In Section \ref{Theta/sec} we defined the functor $\Theta_n: \S_n \to \K_n$, whose
restriction $\Theta_n: \S_n^c \to \K_n^c$ was shown to be a category equivalence for every
$n \geq 4$ in Section \ref{K=S/sec} (see Theorem \ref{ribbon-kirby/thm}). The goal of this
section is to prove an analogous fact for the quotients introduced in the previous
section.

\begin{Figure}[b]{boundary05/fig}
{}{The image of move \(T) under the functor $\Theta_n$}
\vskip3pt
\centerline{\fig{}{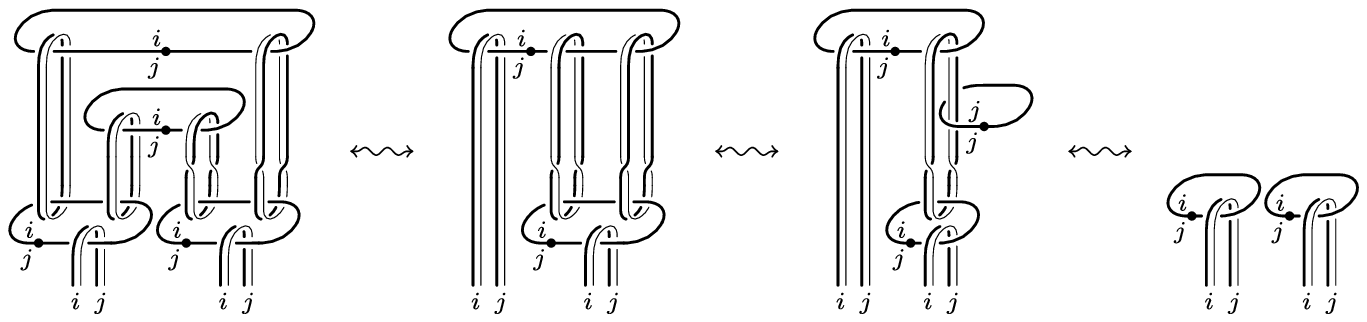}}
\vskip-3pt
\end{Figure}

\begin{proposition}\label{theta-bd/thm}
For any $n \geq 2$ the functor $\Theta_n: \S_n \to \K_n$ induces well-defined braided 
monoidal functors on the quotient categories:
\begin{eqnarray*}
&\bTheta_n: \Sb_n \to \Kb_n \text{ \ and }\,
&\bTheta_n: \Sb_n^c \to \Kb_n^c\,,\\[2pt]
&\bbTheta_n: \Sbb_n \to \Kbb_n \text{ \ and }\,
&\bbTheta_n: \Sbb_n^c \to \Kbb_n^c\,.
\end{eqnarray*}
\vskip-12pt
\end{proposition}

\begin{proof}
Thanks to Propositions \ref{theta/thm} and \ref{theta-c/thm}, it is enough to prove
that if two ribbon surface tangles $S$ and $S'$ in $\S_n$ are related by a \(T) move, then 
their images $\Theta_n(S)$ and $\Theta_n(S')$ in $\K_n$ are related by a handle trading, 
while if $S$ and $S'$ are related by a \(P) move, then $\Theta_n(S)$ and $\Theta_n(S')$ 
are related by a blow down/up. 

\begin{Figure}[htb]{boundary06/fig}
{}{The image of move \(P) under the functor $\Theta_n$}
\centerline{\fig{}{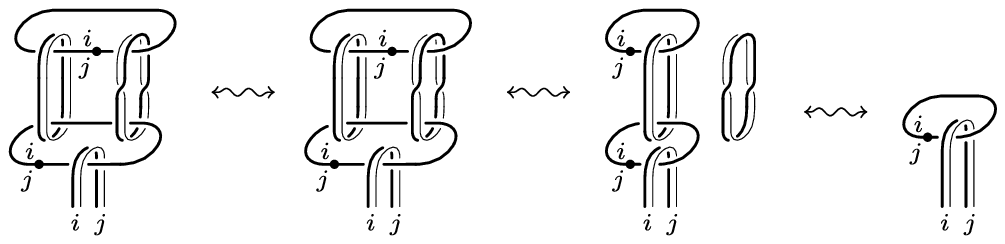}}
\vskip-3pt
\end{Figure}

These facts are shown in Figures \ref{boundary05/fig} and \ref{boundary06/fig}
respectively. In particular, the second equivalence in Figure \ref{boundary05/fig}
consists in a handle sliding followed by a 1/2-handle trading, while a blow down is 
performed in the last step of Figure \ref{boundary06/fig}.
\end{proof}

\begin{Figure}[b]{boundary07/fig}
{}{}
\centerline{\kern12pt\fig{}{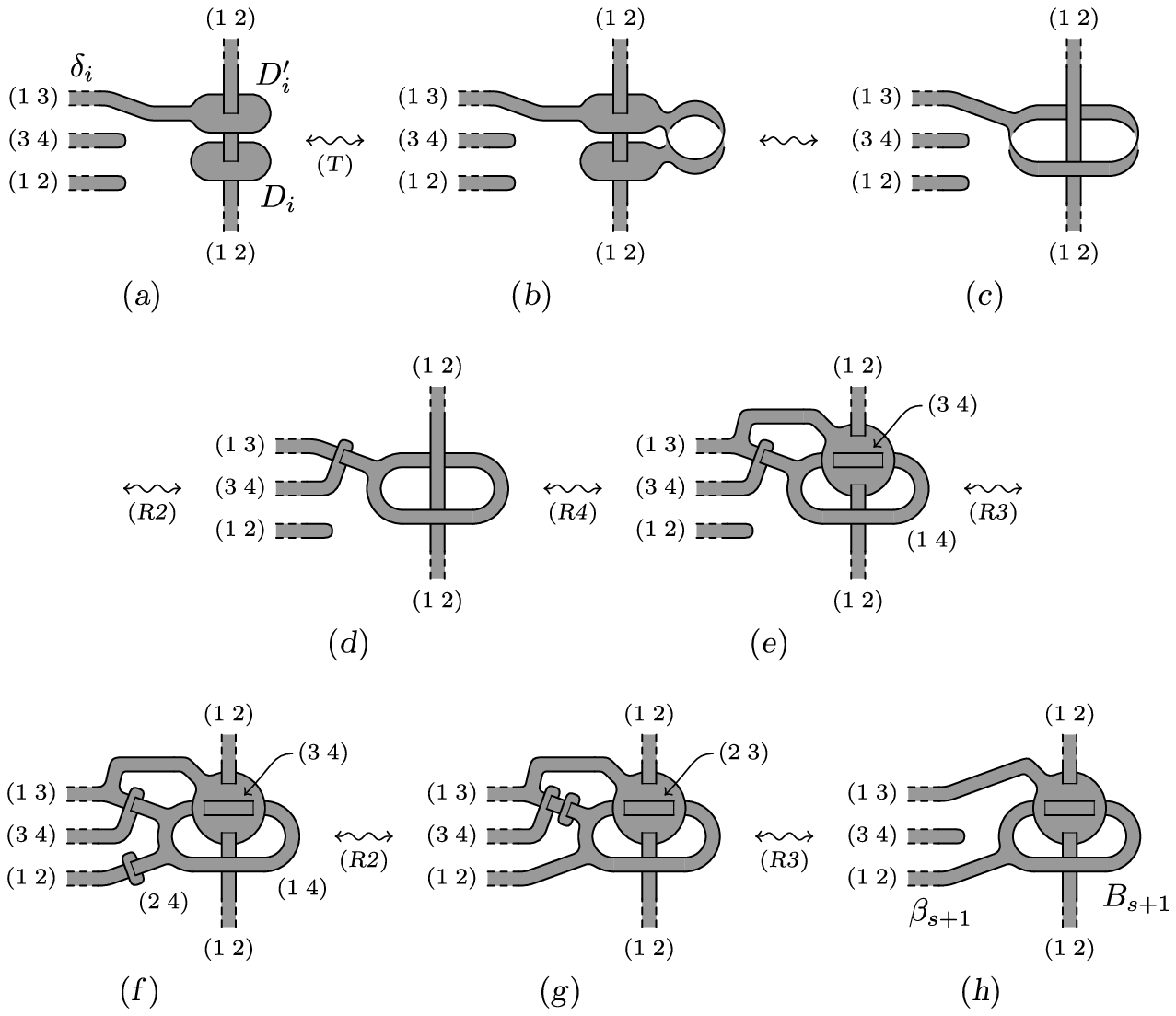}}
\vskip-3pt
\end{Figure}

\begin{proposition}\label{xi-bd/thm}
For any $n \geq 4$, the functor $\Xi_n: \K_1 \to \S_n^c$ induces well-defined functors on
the quotient categories:
\begin{eqnarray*}
&\bXi_n: \Kb_1 \to \Sb_n^c \text{ \ and }\,
&\bbXi_n: \Kbb_1 \to \Sbb_n^c\,.
\end{eqnarray*}
\vskip-12pt
\end{proposition}

\begin{proof}
Taking into account the defining identity $\Xi_n = {\up_4^n} \circ \Xi_4$ (cf. Proposition
\ref{xi/thm}) and Proposition \ref{Sb-stabilization/thm}, it suffices to show that if two
Kirby tangles $K$ and $K'$ in $\K_1$ are related by a 1/2-handle trading as in Figure
\ref{boundary03/fig} \(a), then their images $\Xi_4(K)$ and $\Xi_4(K')$ in $\S_4^c$ are
related by a \(T) move, while if $K$ and $K'$ are related by blowing down/up as in Figure
\ref{boundary03/fig} \(b), then $\Xi_4(K)$ and $\Xi_4(K')$ are related by a \(P) move.

In the light of the construction of the surface $S_K$ in Section \ref{fullness/sec}, the
latter fact is essentially trivial, since the unknot with frame $+1$ involved in the blow
\pagebreak
down/up corresponds in $S_K$ exactly to a positively half twisted closed band as the one
in the right-hand side of the move (P). The former fact is proved in Figure
\ref{boundary07/fig}. This figure shows the sequence of moves needed to replace the disks
$D_i$ and $D'_i$ in diagram \(a), corresponding to a canceling 1-handle of $K$, with the
new ribbon $B_{s+1}$ in diagram \(h), corresponding to the new 2-handle of $K'$ deriving
from the trading. The two bands labeled \tp{1}{2} and \tp{3}{4} in \(a) are assumed to
have been expanded from the stabilization ones. The first band will give rise to the band
$\beta_{s+1}$, the second one is only an auxiliary band, which has to be retracted back in
\(h).
\end{proof}

\begin{theorem}\label{ribbon-kirby-bd/thm}
For any $n \geq 4$, the restriction functor $\bTheta_n: \Sb_n^c\to \Kb_n^c$ (resp.
$\bbTheta_n: \Sbb_n^c\to \Kbb_n^c$) and the functor $\bXi_n: \Kb_1 \to \Sb_n^c$ (resp.
$\bbXi_n: \Kbb_1 \to \Sbb_n^c$) are category equivalences. Moreover, ${\down_1^n} \circ
\bTheta_n \circ \bXi_n$ (resp. ${\down_1^n} \circ \bbTheta_n \circ \bbXi_n$) is equal to
$\id_{\Kb_1}$ (resp. $\id_{\Kbb_1}$), while $\bXi_n \circ {\down_1^n} \circ \bTheta_n$
(resp. $\bbXi_n \circ {\down_1^n} \circ \bbTheta_n$) is naturally equivalent to
$\id_{\Sb_n^c}$ (resp. $\id_{\Sbb_n^c}$).
\end{theorem}

\begin{proof}
It follows immediately by Theorem \ref{ribbon-kirby/thm} and the propositions above.
\end{proof}

\subsection{The quotient categories $\Hb_n$ and $\Hbb_n$%
\label{Hb/sec}}

Analogously to what was done for the categories of Kirby and ribbon surface tangles, we
introduce the quotients of the universal groupoid ribbon Hopf algebra $\H^r(\G)$ by two
additional relations. The first relation states the duality of the algebra integral and
cointegral with respect to the copairing. The second one is a kind of normalization
telling that a specific closed morphism equals the empty morphism $\id_\one$. In
particular, the corresponding quotients of $\H_n^{r,c} \subset \H_n^r$ will be shown to be
equivalent to $\Kb_n^c$ and $\Kbb_n^c$ in the next section.

\begin{Table}[htb]{table-Hrb/fig}
{}{}
\centerline{\fig{}{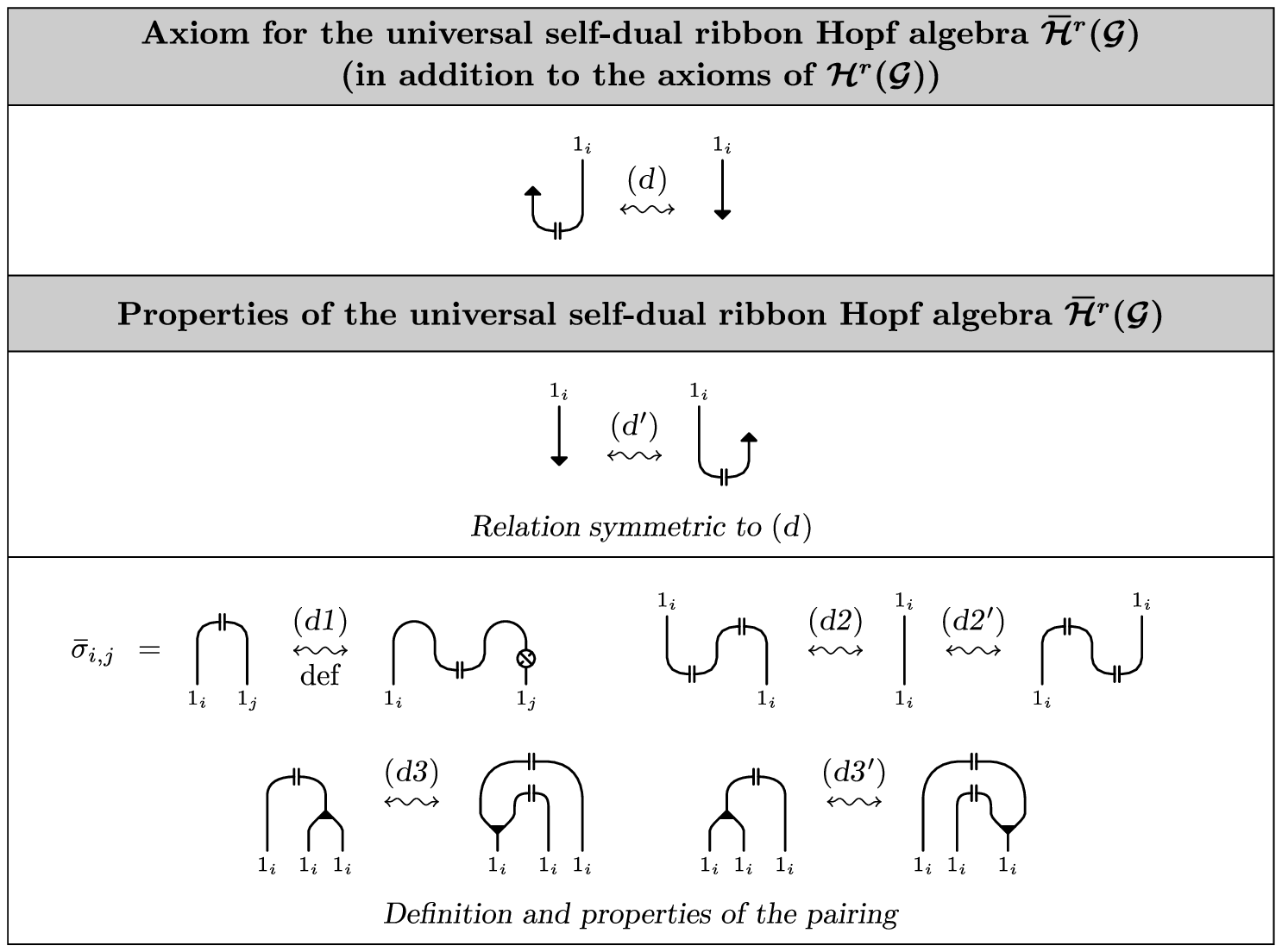}}
\end{Table}

\begin{Table}[htb]{table-Hrbb/fig}
{}{}
\vskip6pt
\centerline{\fig{}{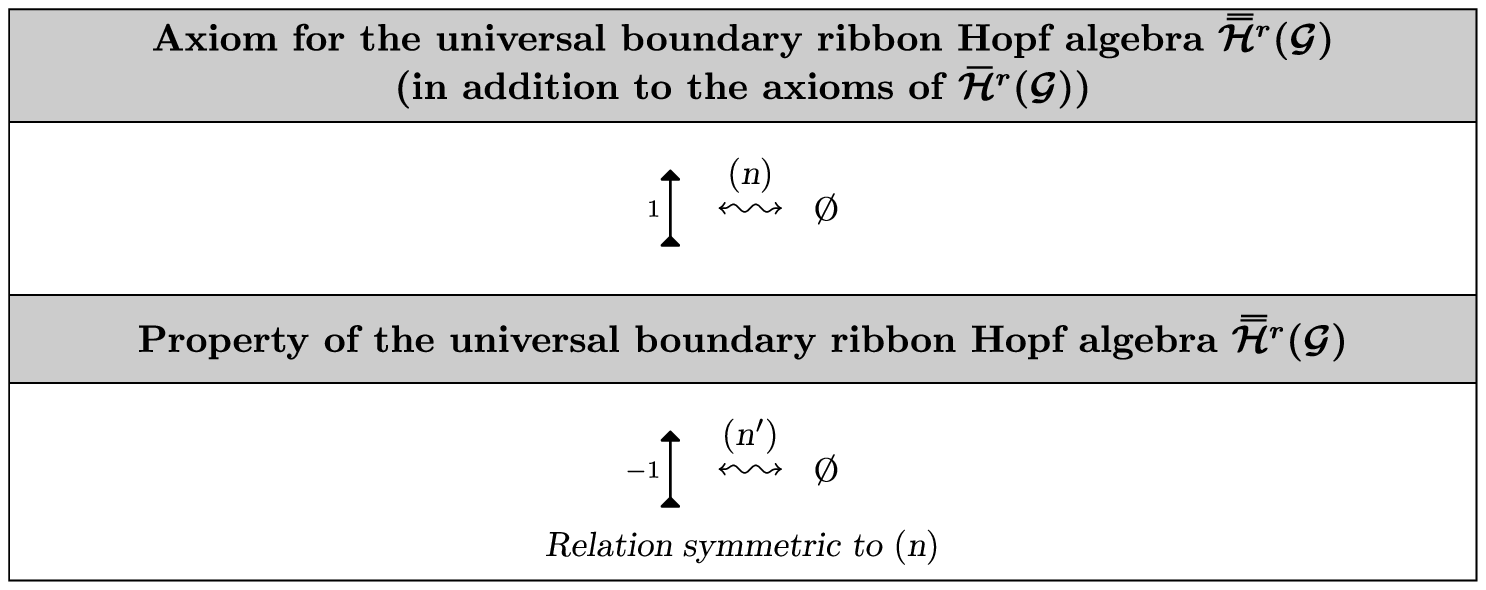}}
\end{Table}

\begin{definition}\label{self-dual/def}
Given a groupoid $\G$, a ribbon Hopf $\G$-algebra $H$ in a braided monoidal category $\C$ 
is called {\sl self-dual} if the following identity (cf. Table \ref{table-Hrb/fig})\break 
holds for every $i \in \G$:
$$(l_i \diam \id_{1_i}) \circ \sigma_{i,i} = L_{1_i}\,.
\eqno{\(d)}$$
\vskip4pt
Moreover, a self-dual ribbon Hopf $\G$-algebra $H$ is called {\sl boundary} if the 
following normalization identity (cf. Table \ref{table-Hrbb/fig}) holds for every $i \in 
\G$: 
$$l_i \circ v_{1_i} \circ \eta_i = \id_\one\,.
\eqno{\(n)}$$

We define the {\sl universal self-dual ribbon Hopf algebra} $\Hb^r(\G)$ to be the quotient
category of $\H^r(\G)$ modulo the relations \(d) presented in Table \ref{table-Hrb/fig},
and the {\sl universal boundary ribbon Hopf algebra} $\Hbb^r(\G)$ to be the quotient
category of $\Hb^r(\G)$ modulo the relations \(n) presented in Tables 
\ref{table-Hrbb/fig}.
\end{definition}

Like for the categories of tangles, the quotients $\Hb^r(\G)$ and $\Hbb^r(\G)$ inherit
from $\H^r(\G)$ a strict monoidal structure, whose product will be still denoted by
$\,\diam\,$, and we have monoidal quotient functors $$\H^r(\G) \to \,\Hb^r(\G)
\to\,\Hbb^r(\G)\,.$$

It is well-known (cf. \cite{Ke02}) that the relation \(d) in Table \ref{table-Hrb/fig}
implies the symmetric relation \(d') and also the non-degeneracy of the copairing and the
duality of the multiplication and comultiplication morphisms in $\Hb^r(\G)$. For the sake
of completeness we present the result below.

\begin{proposition}\label{selfH/thm}
Let $\G$ be any groupoid. Then, the following relation hold in $\Hb^r(\G)$ (cf. Table 
\ref{table-Hrb/fig}):
$$
(\id_{1_i} \diam l_i) \circ \sigma_{i,i} = L_{1_i},
\eqno{\(d')}
$$
Moreover, the pairing morphisms $\bar\sigma_{i,j}: H_{1_i} \diam H_{1_j} 
\to \one$ in $\Hb^r(\G)$ defined by
$$
\bar\sigma_{i,j} = (\Lambda_{1_i} \diam \Lambda_{1_j}) \circ (\id_{1_i} \diam
\sigma_{i,j} \diam S_{1_j}^{-1})
\eqno{\(d1)}
$$
for every $i,j \in \G$, satisfy the following identities \(d2-2') and \(d3-3') in Table 
\ref{table-Hrb/fig}.
\end{proposition}

\begin{proof}
Relation \(d') is equivalent to \(d) modulo \(i5) and \(f5-11) in Tables
\ref{table-Hu/fig} and \ref{table-Huvprop/fig}. Identity \(d2) is proved in Figure
\ref{boundary08/fig}, and then \(d2') follows by symmetry. Finally, by using \(d2-2'), one
can easily derive \(d3-3') from \(r7-7') in Tables \ref{table-Huvdefn/fig} and
\ref{table-Huvprop/fig}.
\end{proof}

\begin{Figure}[htb]{boundary08/fig}
{}{Deriving relation \(d2) in $\Hb^r(\G)$
   [{\sl d}/\pageref{table-Hrb/fig},
    {\sl f}/\pageref{table-Hu/fig}-\pageref{table-Huvprop/fig},
    {\sl p}/\pageref{table-Huvprop/fig}, {\sl r}/\pageref{table-Huvdefn/fig}]}
\vskip-9pt
\centerline{\fig{}{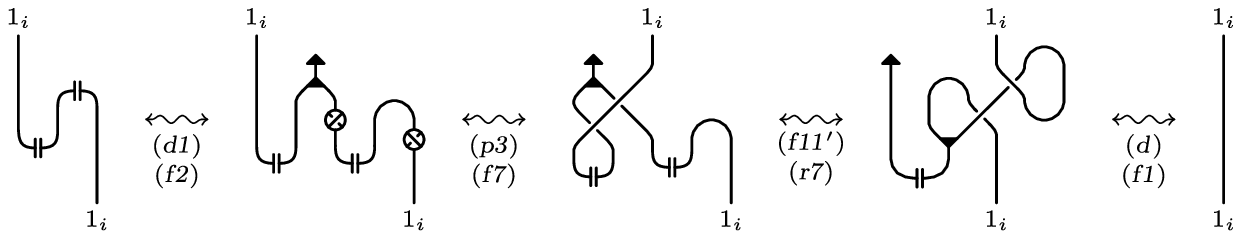}}
\vskip-3pt
\end{Figure}

\begin{proposition}\label{boundary/thm}
Given a groupoid $\G$, the following relation holds in $\Hbb_n^r$ (see Table
\ref{table-Hrbb/fig}):
$$l_i \circ v^{-1}_{1_i} \circ \eta_i = \id_\one.
\eqno{\(n')}$$
\vskip-12pt
\end{proposition}

\begin{proof}
See Figure \ref{boundary09/fig}.
\end{proof}

\begin{Figure}[htb]{boundary09/fig}
{}{Deriving relation \(n') in $\Hbb^r(\G)$
   [{\sl a}/\pageref{table-Hdefn/fig}, {\sl d}/\pageref{table-Hrb/fig},
    {\sl n}/\pageref{table-Hrbb/fig}, {\sl p}/\pageref{table-Huvprop/fig},
    {\sl r}/\pageref{table-Huvdefn/fig}]}
\vskip-6pt
\centerline{\fig{}{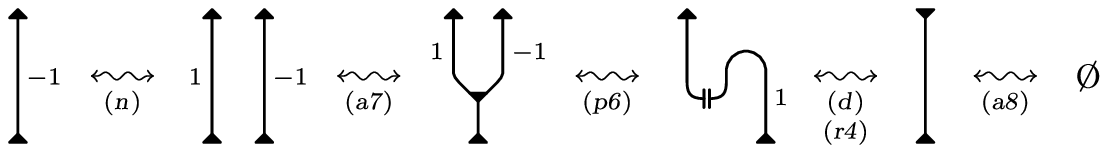}}
\end{Figure}

Given a full inclusion of groupoids $\G \subset \G'$ and a spanning sequence $X = (x_n,
\dots, x_1)$ for the pair $(\G',\G)$, we define the subcategories $\Hb^r_X(\G) \subset
\Hb^r(\G)$ and $\Hbb^r_X(\G) \subset \Hbb^r(\G)$ to be the images of $\H^r_X(\G) \subset
\H^r(\G)$ under the quotient functors. As for Kirby tangles, we have the following
commutative diagram of inclusion and quotient functors
$$
\begin{array}{ccccc}
\H^r_X(\G) & \to & \Hb^r_X(\G) & \to & \Hbb^r_X(\G)\\[2pt]
\cap & & \cap & & \cap\\[2pt]
\H^r(\G) & \to & \Hb^r(\G) & \to & \Hbb^r(\G)\end{array}
$$

\medskip

\begin{proposition}\label{Hb-stabilization/thm}
For any full inclusion of groupoids $\G \subset \G'$ and spanning sequence $X = (x_n,
\dots, x_1)$ for the pair $(\G',\G)$, the stabilization functor $\up_X: \H^r(\G) \to
\H^r_X(\G')$ and the reduction functor $\down_X: \H^r_X(\G') \to \H^r(\G)$ induce
well-defined functors on the quotient categories
\begin{eqnarray*}
&\up_X: \Hb^{r}(\G) \to \Hb^r_X(\G') \text{ \ and }\,
&\down_X: \Hb^r_X(\G') \to \Hb^r(\G),\\[2pt]
&\up_X: \Hbb^r(\G) \to \Hbb^r_X(\G') \text{ \ and }\,
&\down_X: \Hbb^r_X(\G') \to \Hbb^r(\G)\,.
\end{eqnarray*}
Moreover, $\down_X \circ \up_X = \id_{\Hb^r(\G)}$ (resp. $\down_X \circ \up_X =
\id_{\Hbb^r(\G)}$), while $\up_X \circ \down_X$ is naturally equivalent to
$\id_{\Hb^r_X(\G')}$ (resp. $\id_{\Hbb^r_X(\G')}$). In particular, $\down_X$ and $\up_X$
are equivalences of categories.
\end{proposition}

\begin{proof} 
The last part of the statement will follow immediately from Theorem \ref{H-reduction/thm}, 
once we prove that $\up_X$ and $\down_X$ induce well-defined functors between the quotient
categories. This fact is obvious for the stabilization functor $\up_X$. Concerning the 
reduction $\down_X$, we recall from Definition \ref{H-reduction/def} that $\down_X = 
{\down_{x_1}} \!\circ \dots \circ {\down_{x_n}}$. Then, it is enough to consider the 
case of an elementary reduction $\down_x: H^r_x(\G) \to H^r(\G^{\bs i_0})$ for $x \in 
\G(i_0,j_0)$. On the other hand, still referring to Definition \ref{H-reduction/def}, we 
have $\down_x F = (\epsilon_{1_{j_0}} \!\diam \id_{\pi_1^x}) \circ F^x \circ (\eta_{j_0} 
\!\diam \id_{\pi_0^x})$ for any $F: H_x \diam H_{\pi_0} \to H_x \diam H_{\pi_1}$ in 
$\H^r_x(\G)$. Hence, we are reduced to proving that the functor $\_^x: \H^r(\G) \to 
\H^r(\G^{\bs i_0})$ defined in Proposition \ref{H-push/thm} passes to the quotient, giving  
well-defined functors
\begin{eqnarray*}
&\_^x: \Hb^r(\G) \to \Hb^r(\G) \text{ \ and }\, 
&\_^x: \Hbb^r(\G) \to \Hbb^r(\G)\,.
\end{eqnarray*}

This follows from the fact that the additional axioms \(d) and \(n) defining the quotient
categories $\Hb^r(\G)$ and $\Hbb^r(\G)$ have the form $F_1 = F_2$, where all labels
occurring in $F_1$ and $F_2$ are equal to $1_i$ for some $i \in \G$. Then, by Definition
\ref{Fx/def} $F_1^x$ and $F_2^x$ have the same graph diagrams as $F_1$ and $F_2$, but
label $1_{i^x}$ instead of $1_i$. In particular, $F_1^x$ and $F_2^x$ satisfy the same
relation as $F_1$ and $F_2$.
\end{proof}

\subsection{Equivalences $\Kb_n^c \cong \Hb_n^{r,c}$ and $\Kbb_n^c \cong \Hbb_n^{r,c}$%
\label{Kb=Hb/sec}}

Now we want to prove that the commutative diagram in Theorem \ref{equivalence/thm} induces
analogous commutative diagrams of equivalence functors between the quotient categories.
This will imply the equivalence of $\,\Cobt^{2+1}$ (resp. $\Cob^{2+1}$) and $\Hb^r =
\Hb_1^r$ (resp. $\Hbb^r = \Hbb_1^r$).

Analogously to what was done in Sections \ref{reduction/sec} and \ref{K=H/sec}, when $\G =
\G_k$, $\G' = \G_n$ and $X = \pi_{n \red k} = ((n,n-1), \dots, (k+1,k))$ with $n > k \geq
1$, we adopt the notation $\Hb^r_{\pi_{n \red k}} \!= \Hb^r_{n \red k}$ and
$\Hbb^r_{\pi_{n \red k}} \!= \Hbb^r_{n \red k}$, and we put $\Hb^{r,c}_n = \Hb^r_{n \red
1}(\G_n)$ and $\Hbb^{r,c}_n = \Hbb^r_{n \red 1}(\G_n)$.

\begin{proposition}\label{phi-bd/thm} 
For any $n \geq 1$, the functor $\Phi_n: \H_n \to \K_n$ induces well-defined monoidal 
functors on the quotient categories
\begin{eqnarray*}
&\bPhi_n: \Hb^r_n \to \Kb_n \text{ \ and }\,
&\bbPhi_n: \Hbb^r_n \to \Kbb_n\,.
\end{eqnarray*}
Moreover, for any $n > k \geq 1$ we have the following commutative diagrams.
\end{proposition}

\centerline{\epsfbox{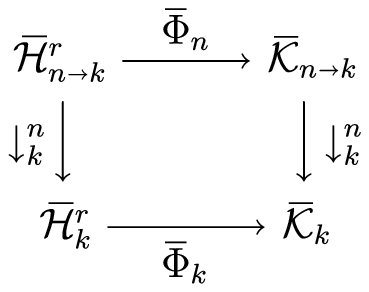}\kern15mm\epsfbox{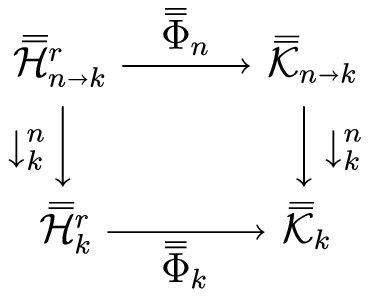}}

\begin{proof}
For the first part of the statement, it suffices to observe that, under the functor
$\Phi_n: \H_n \to \K_n$ defined in Theorem \ref{phi/thm}, the relations \(d) and \(n) in
Tables \ref{table-Hrb/fig} and \ref{table-Hrbb/fig} translate directly in the moves \(a)
and \(b) in Figure \ref{boundary03/fig}. Then, the commutative diagrams are obtained just
by quotienting the second diagram in Proposition \ref{Hn-reduction/thm}.
\end{proof}

\pagebreak

\begin{proposition}\label{psi-bd/thm}
For any $n \geq 4$, the functor $\Psi_n: \S_n \to \H_n^r$ induces well-defined braided 
monoidal functors between the quotient categories:
\begin{eqnarray*}
&\bPsi_n: \Sb_n \to \Hb^r_n \text{ \ and }\,
&\bbPsi_n: \Sbb_n \to \Hbb^r_n\,,\\[2pt]
&\bPsi_n: \Sb_n^c \to \Hb^{r,c}_n \text{ \ and }\,
&\bbPsi_n: \Sbb_n^c\to \Hbb^{r,c}_n\,.
\end{eqnarray*}
\vskip-12pt
\end{proposition}

\begin{proof}
Figure \ref{boundary10/fig} (resp. \ref{boundary11/fig}) respectively show that the
images under $\Psi_n: \S_n \to \H_n^r$ of the two sides of the relation \(T) and (resp.
\(P)) in Figure \ref{boundary04/fig}, according to the definition of are equivalent in 
$\Hb_n^r$ (resp. $\Hbb_n^r$). Therefore, $\Psi_n$ induces well-defined functors between
the quotient categories.
\end{proof}

\begin{Figure}[htb]{boundary10/fig}
{}{The image of move \(T) in $\Hb^r_n$ ($i > j$)
   [{\sl a}/\pageref{table-Hdefn/fig}, {\sl f}/\pageref{table-Hu/fig}, 
    {\sl d}/\pageref{table-Hrb/fig}, {\sl p}/\pageref{table-Huvprop/fig},
    {\sl r}/\pageref{table-Huvdefn/fig}-\pageref{table-Huvprop/fig}, 
    {\sl s}/\pageref{table-Hprop/fig}]}
\vskip-6pt
\centerline{\fig{}{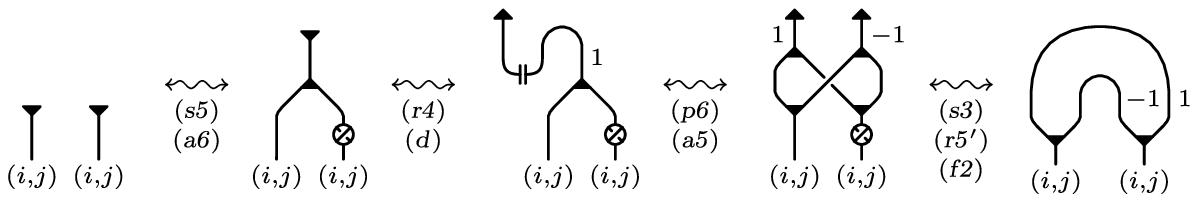}}
\end{Figure}

\begin{Figure}[htb]{boundary11/fig}
{}{The image of move \(P) in $\Hbb^r_n$ ($i > j$)
   [{\sl n}/\pageref{table-Hrbb/fig}, {\sl r}/\pageref{table-Huvprop/fig},
    {\sl s}/\pageref{table-Hdefn/fig}]}
\vskip-9pt
\centerline{\fig{}{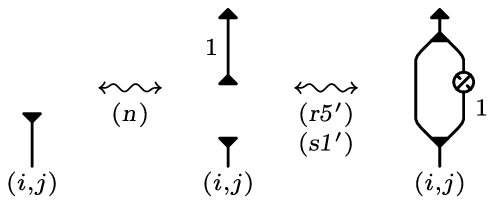}}
\end{Figure}

\begin{theorem}\label{3-cob/thm}
For any $n \geq 4$, we have the following commutative diagrams of equivalence functors:
\end{theorem}

\centerline{\epsfbox{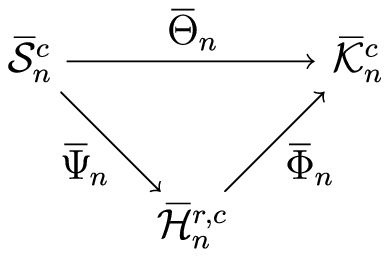}\hskip15mm\epsfbox{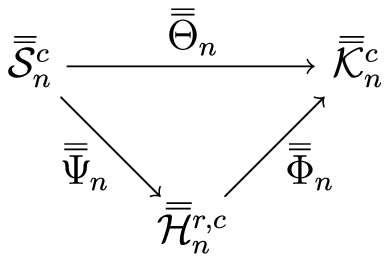}}
\vskip6pt\vskip0pt

\begin{proof}
The existence of the functors has been already established in Propositions
\ref{theta-bd/thm}, \ref{phi-bd/thm} and \ref{psi-bd/thm}. The commutativity of the
diagrams follows from that of the diagram in Theorem \ref{equivalence/thm} by taking the
respective quotients. According to Theorem \ref{ribbon-kirby-bd/thm}, for $n \geq 4$ the
functors $\bTheta_n$ and $\bbTheta_n$ are category equivalences, which implies that in
this case $\bPsi_n$ and $\bbPsi_n$ are faithful. Now, since $\Psi_n$ is a category
equivalence, it is full and any object in $\H_n$ is isomorphic to one in its image. Then,
the same holds for its quotients $\bPsi_n$ and $\bbPsi_n$, which implies that these are
category equivalences as well, and by the commutativity of the diagram so are $\bPhi_n$
and $\bbPhi_n$.
\end{proof}

At this point, we are ready to give the announced algebraic description of the category of
(2-framed) 3-dimensional relative codordisms in terms of the categories $\Hb^r = \Hb^r_1$
(in the 2-framed case) and $\Hbb^r = \Hbb^r_1$. The elementary morphisms and defining
axioms of such algebraic categories are listed in Tables \ref{table-Hr1diags/fig} and
\ref{table-Hr1axioms/fig} and Table \ref{table-Hrbb1/fig}. Observe that axiom \(d) in the
last table expresses the integral $L$ in terms of the cointegral $\lambda$ and the
copairing $\sigma$, and by substituting such expression in axioms \(i2-3) in
\ref{table-Hr1axioms/fig}, we can cancel both, $L$ from the list of the elementary
diagrams in Tables \ref{table-Hr1diags/fig}, and \(d) from the list of the axioms.

\begin{Table}[htb]{table-Hrbb1/fig}
{}{}
\centerline{\fig{}{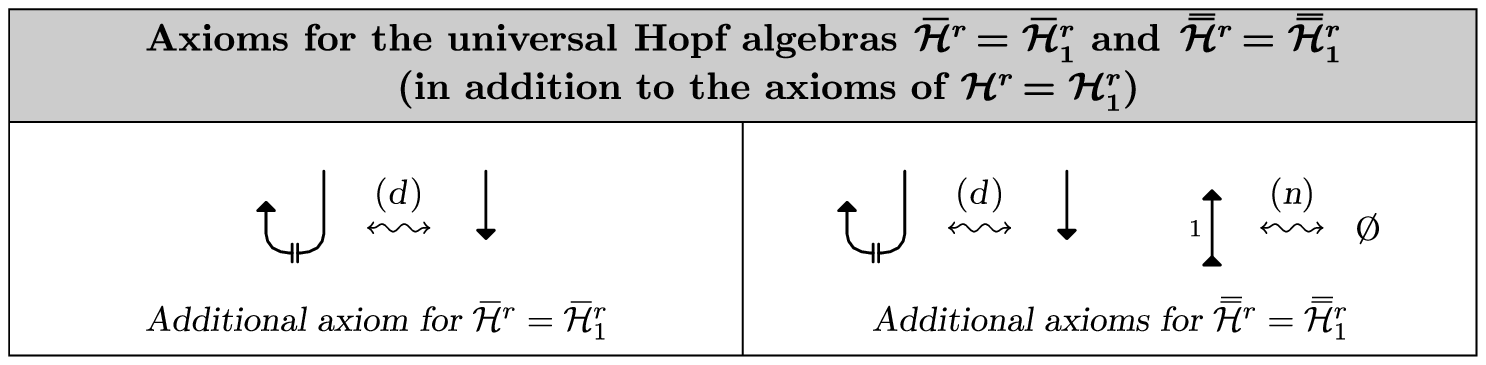}}
\vskip-3pt
\end{Table}

\begin{theorem}\label{alg-kirby-eq-bd/thm}
The functors $\bPhi_n: \Hb_n^{r,c} \to \Kb_n^c$ and $\bbPhi_n:\Hbb_n^{r,c} \to \Kbb_n^c$
are category equivalences for any $n \geq 1$. In particular, the universal self-dual
ribbon Hopf algebra $\Hb^r = \Hb^r_1$ is equivalent to the category of 2-framed
3-dimensional relative cobordisms $\Cobt^{2+1} \!= \Cobt^{2+1}_1$, while universal
boundary ribbon Hopf algebra $\Hbb^r= \Hbb^r_1$ is equivalent to the category of
3-dimensional relative cobordisms $\Cob^{2+1} \!= \Cob^{2+1}_1$.
\end{theorem}

\begin{proof}
According to Theorem \ref{3-cob/thm}, $\Phi_n$ is a category equivalence for any $n \geq
4$. For any $1 \leq n \leq 3$, the commutative diagrams in Proposition \ref{phi-bd/thm}
imply that $\bPhi_n \circ {\down_1^4} = {\down_1^4} \circ \bPhi_4$ (resp. $\bbPhi_n \circ
{\down_1^4} = {\down_1^4} \circ \bbPhi_4$). By Propositions \ref{Kb-reduction/thm} and
\ref{Hb-stabilization/thm}, the reduction functors involved in this identity are category
equivalences and therefore, $\bPhi_n$ (resp. $\bbPhi_n$) is a category equivalence as
well. In particular, for $n = 1$ we obtain that $\bPhi_1: \Hb^r_1 \to \Kb_1$ (resp.
$\bbPhi_1: \Hbb^r_1 \to \Kbb_1$) is a category equivalence. Then, the second part of the
statement follows from Proposition \ref{Kb-category/thm} with $n = 1$.
\end{proof}

\newpage

\section{Branched coverings of $B^4$ and $S^3$%
\label{covering-moves/sec}}

In this chapter we apply the previous results to the covering moves problem for
branched coverings of $B^4$ and $S^3$.

In particular, Section \ref{B4/sec} concerns the representation of 4-dimensional
2-handle\-bodies as simple coverings of $B^4$ branched over ribbon surfaces. Here, we will
derive from the results of Chapters \ref{cobordisms/sec} and \ref{surfaces/sec} an
effective way to convert any Kirby diagram into a 3-labeled ribbon surface providing such
a representation, and a 2-equivalence criterion in terms of moves for labeled ribbon
surface.

Then, in Section \ref{S3/sec}, by adding the further moves introduced in Section
\ref{quotients/sec} and restricting all the moves to the boundary, we obtain a complete
solution of the Fox-Montesinos covering moves problem for simple coverings of $S^3$
branched over links. Finally, we extend such result to arbitrary coverings of $S^3$
branched over graphs.

\subsection{Covers of $B^4$ simply branched over ribbon surfaces%
\label{B4/sec}}

Recalling the definitions in Sections \ref{Chbn/sec}, we have that the 2-equivalence
classes of connected 4-dimensional 2-handlebodies bijectively correspond to the morphisms
$W: M^1_\emptyset \to M^1_\emptyset$ in $\Chb_1^{3+1}$, or equivalently, in terms of Kirby
diagrams, to the morphisms $K: I_\emptyset \to I_\emptyset$ in $\K_1$ (cf. Proposition
\ref{K-category/thm}). This follows from Propositions \ref{0-handles/thm} and
\ref{0-handles/thm}, by taking into account that such a morphism $W$ is a relative
4-dimensional 2-handlebody build on the unique 0-handle $H^0 = Y(M^1_\emptyset,
M^1_\emptyset) \cong B^4$, considered up to 2-deformations that fix $H^0$ and do not
introduce any extra 0-handle.

On the other hand, an $n$-fold covering of $B^4$ simply branched over a ribbon surface, is
represented by an $n$-labeled ribbon surface in $B^4$, which is nothing else than a
morphism $S: J_\emptyset \to J_\emptyset$ in $\S_n$. Moreover, the functor $\Theta_n: \S_n
\to \K_n$ introduced in Section \ref{Theta/sec} sends $S$ to a $n$-labeled Kirby diagram
$K_S: I_\emptyset \to I_\emptyset$ in $\K_n$. This in turn represents a relative
4-dimensional 2-handlebody $W_S: M^n_\emptyset \to M^n_\emptyset$ build on the $n$
0-handles $H^0_1 \sqcup \dots \sqcup H^0_n = Y(M^n_\emptyset, M^n_\emptyset) \cong B^4
\sqcup \dots \sqcup B^4$, considered up to 2-equivalence modulo the 0-handles.

Then, the restriction of $\Theta_n$ from $n$-labeled ribbon surfaces (i.e. ribbon
surface tangles with empty source and target) to $n$-labeled Kirby diagrams (i.e. Kirby
tangles with empty source and target), exactly encodes the realization of 4-dimensional
2-handlebodies up to 2-equivalence as simple branched coverings of $B^4$ given by
Montesinos in \cite{Mo78}. In fact, according to the discussion at the beginning of
Section \ref{Theta/sec}, if $S \subset B^4$ is an $n$-labeled ribbon surface representing
a simple branched covering $p: W \to B^4$, any adapted 1-handlebody decomposition of $S$
induces a 2-handlebody decomposition of $W$ with $n$ 0-handles, whose 1-handles (resp. 
2-handles) correspond to the 0-handles (resp. 1-handles) of $S$. Moreover, Lemma 
\ref{ribbon-to-kirby/thm} says that 1-deformations in $S$ induce 2-deformations in $W$, 
hence the 2-handlebody structure of $W$ turns out to be well-defined up to 2-equivalence
by Propositions \ref{1-handles/thm}.

The main result of Montesinos \cite{Mo78} is that any connected oriented 4-dimensional
2-handlebody $W$ has a 3-fold branched covering representation as above. Actually, the
ribbon surface $S$ can always be made orientable (cf. Remark \ref{orient/rem} or
\cite{LP01,PZ03} for other constructions giving directly orientable branching surfaces).
In that paper the labeled ribbon surface $S$ is obtained from a Kirby diagram of $W$,
after it has been suitably symmetrized with respect to a standard 3-fold simple covering
representation of the 1-handlebody $W^1$.

A simpler and more effective construction of a labeled ribbon surface $S$ representing
$W$, similar to that of labeled links given in \cite{Mo80} for 3-manifolds (cf. Remark
\ref{crossing/rem}), can be derived from Proposition \ref{full-theta3/thm}. This is the
content of Proposition \ref{hat-xi/thm} below. 

But first we need the following definition. Let us denote by $\K_n^\o$ the set of all
$n$-labeled Kirby diagrams $K: I_\emptyset \to I_\emptyset$ in $\K_n$ and by $\S_n^\o$ the
set of all $n$-labeled ribbon surfaces $S: J_\emptyset \to J_\emptyset$ in $\S_n$.

\begin{definition}\label{closure/def}
Given any Kirby tangle $K: I_{\pi_0} \to I_{\pi_1}$ in $\K_n$ with $\pi_0 = ((i^0_1,
j^0_1), \dots, (i^0_{m_0}, j^0_{m_0}))$ and $\pi_1 = ((i^1_1, j^1_1), \dots, (i^1_{m_1},
j^1_{m_1}))$, we define the {\sl closure} of $K$ to be the Kirby diagram $\hat K =
(\epsilon_{(i^1_1, j^1_1)} \diam \dots \diam \epsilon_{(i^1_{m_1}, j^1_{m_1})}) \circ K
\circ (L_{(i^0_1, j^0_1)} \diam \dots \diam L_{(i^0_{m_0}, j^0_{m_0})})$ in $\K_n^\o$ (see
Figure \ref{closure01/fig}). Similarly, given a ribbon surface tangle $S: J_{\sigma_0} \to
J_{\sigma_1}$ in $\S_n$, we define the {\sl closure} of $S$ to be the ribbon surface $\hat
S$ in $\S_n^\o$ shown in Figure \ref{closure02/fig}.
\end{definition}

\begin{Figure}[htb]{closure01/fig}
{}{The closure $\hat K$ of a Kirby tangle $K$ in $\K_n$}
\vskip-6pt
\centerline{\fig{}{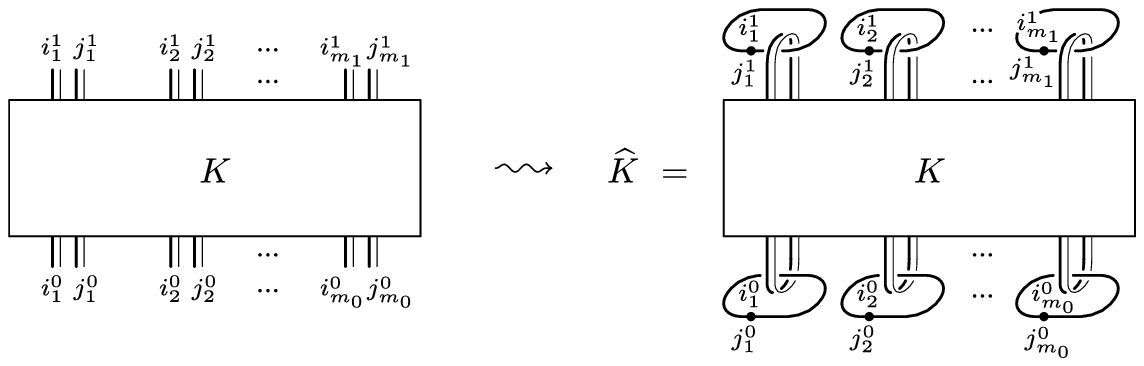}}
\vskip-6pt
\end{Figure}

\begin{Figure}[htb]{closure02/fig}
{}{The closure $\hat S$ of a ribbon surface tangle $S$ in $\S_n$}
\vskip-6pt
\centerline{\fig{}{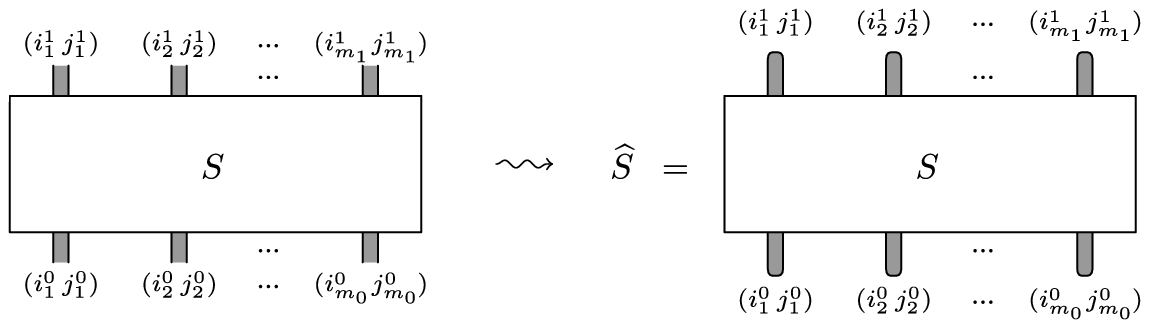}}
\vskip-6pt
\end{Figure}

It is clear from the definitions (cf. Figure \ref{theta01/fig}) that the functor 
$\Theta_n: S_n \to K_n$ preserves closures, i.e. if $\Theta_n(S) = K$ then $\Theta_n(\hat 
S) = \hat K$.

We also observe that the stabilization $\up_k^n K$ of a $k$-labeled Kirby diagram $K$ is
an $n$-labeled Kirby tangle with non-empty source and target $I_{\pi_{n \red k}}$. But the
closure of $\up_k^n K$ is an $n$-labeled Kirby diagram and the corresponding 4-dimensional
2-handlebody is 2-equivalent to the one given by $K$ (see Figure \ref{closure03/fig} for
$k = n-1$). Such 2-equivalence consists in the deletion of canceling pairs of 1/2-handles
and then 0/1-handles (in the rightmost diagram of Figure \ref{closure03/fig}, the $n$-th
0-handle can be canceled against the 1-handle connecting it to the $(n-1)$-th 0-handle,
since $K$ lives in the other 0-handles).

\begin{Figure}[htb]{closure03/fig}
{}{$n$-stabilization of an $(n-1)$-labeled Kirby diagram}
\centerline{\fig{}{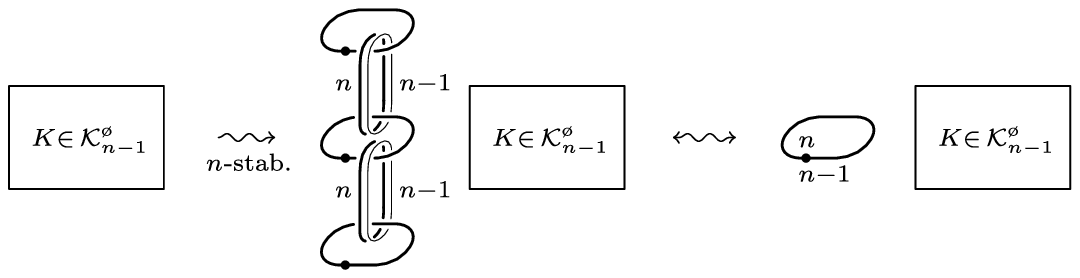}}
\vskip-3pt
\end{Figure}

Analogously, the closure of the stabilization $\up_k^n S$ of a $k$-labeled ribbon surface
$S$ corresponds to an $n$-fold stabilization of the $k$-fold branched covering represented
by $S$, as defined in Section \ref{coverings/sec} (see Figure \ref{closure04/fig} for $k =
n-1$).

\begin{Figure}[htb]{closure04/fig}
{}{$n$-stabilization of an $(n-1)$-labeled ribbon surface}
\vskip3pt
\centerline{\fig{}{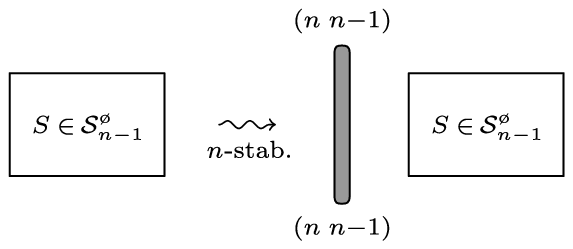}}
\vskip-9pt
\end{Figure}

\begin{proposition}\label{hat-xi/thm}
Let $W$ be a connected 4-dimensional 2-handlebody and $K$ be any uni-labeled Kirby diagram
of it. Then, the closure $\hat S_K$ of the ribbon tangle $S_K$ defined in Section
\ref{fullness/sec} is a 3-labeled ribbon surface such that the 4-dimensional 2-handlebody
described by the 3-labeled Kirby diagram $\Theta_3(\hat S_K)$ is 2-equivalent to $W$. In
other words, $\hat S_K$ represents $W$ up to 2-equivalence as a 3-fold branched covering
of $B^4$. The global structure of $\hat S_K$ is depicted in Figure \ref{closure05/fig}
(see Section \ref{fullness/sec} for the definition of $T_K$).
\end{proposition}

\begin{Figure}[htb]{closure05/fig}
{}{The global structure of the ribbon surface $\hat S_K$}
\vskip-6pt
\centerline{\fig{}{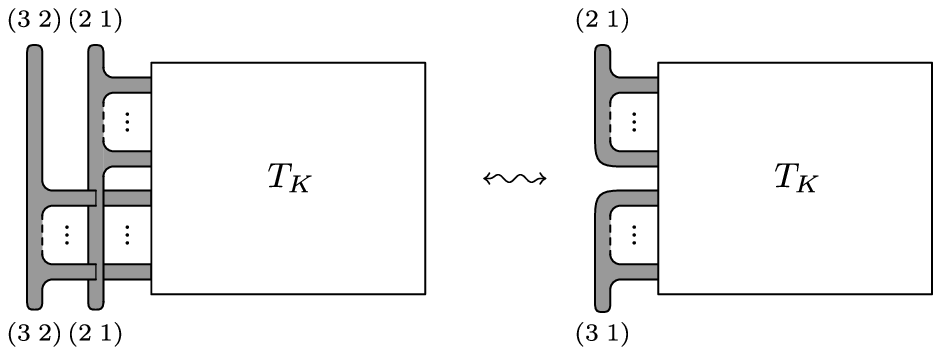}}
\vskip-3pt
\end{Figure}

\begin{proof}
The same argument of the proof of Proposition \ref{full-theta3/thm} still works here,
except for the absence of all the components relative to the ribbon surface tangles $\bar
Q_{m_0}$ and $Q_{m_1}$, which are empty in the present context, and for the use of handle
cancelation in place of the reduction functor $\down_1^3$ in the end. Namely, we start 
with the adapted 1-handlebody decomposition of $\hat S_K$ having the same handles as the 
one of $S_K$ in the proof of Proposition \ref{full-theta3/thm}, apart from the fact that 
the collars of the source and the target are now considered as 0-handles. Then, we 
consider the corresponding 3-labeled Kirby diagram $\Theta_3(\hat S_K)$ and perform on it 
the slidings and the crossing changes described in the Figures \ref{kirby-ribbon08/fig},
\ref{kirby-ribbon09/fig} and \ref{kirby-ribbon10/fig}. After that, we can isotope the 
resulting diagram to the separate union of two chains of handles on the left, like those 
in the first diagram of Figure \ref{closure03/fig} for $n = 3$ and $n = 2$ respectively, 
and a copy of the original $K$ labeled by 1. This is just a 3-stabilization of $K$, hence 
it can be reduced to $K$ by canceling all the handles carrying the labels 3 and 2 in the 
order as shown in Figure \ref{closure03/fig}.
\end{proof}

Our first covering moves theorem concerns the 2-equivalence of 4-dimensional
2-handlebodies represented by labeled ribbon surfaces in $B^4$, as described above in
terms of the $\Theta_n$'s. Its proof will make use of the next two lemmas.

\begin{lemma}\label{hat-reduction/thm}
For any $k$-reducible Kirby tangle $K \in \K_{n \red k}$, the 4-dimensional 2-handlebodies
represented by the Kirby diagrams $\hat{\down_k^n K}$ and $\hat K$ are 2-equivalent.
\end{lemma}

\begin{proof}
The case of $k = n-1$ is shown in Figure \ref{closure06/fig}, while the general case
follows by induction on the difference $n-k$. Starting from $\hat K$ with $K =\break
(\id_{(n,n-1)} \diam L) \circ (\Delta_{(n,n-1)} \diam \id_{\pi_0}) \in \K_{n \red k}$, the
first step in the figure is just a 1/2-handle cancelation, the second one is based on
Lemma \ref{K-push/thm} like the first step in Figure \ref{kirby-stab14/fig}, the third one
consists in a 2-handle sliding. Finally, the last diagram is the closure of can be reduced 
to the closure of $\up_{n-1}^n \down_{n-1}^n K$, hence it can be reduced to the closure of 
$\down_{n-1}^n K$ as said above (cf. Figure \ref{closure03/fig}).
\end{proof}

\begin{Figure}[htb]{closure06/fig}
{}{Proof of Lemma \ref{hat-reduction/thm}}
\centerline{\fig{}{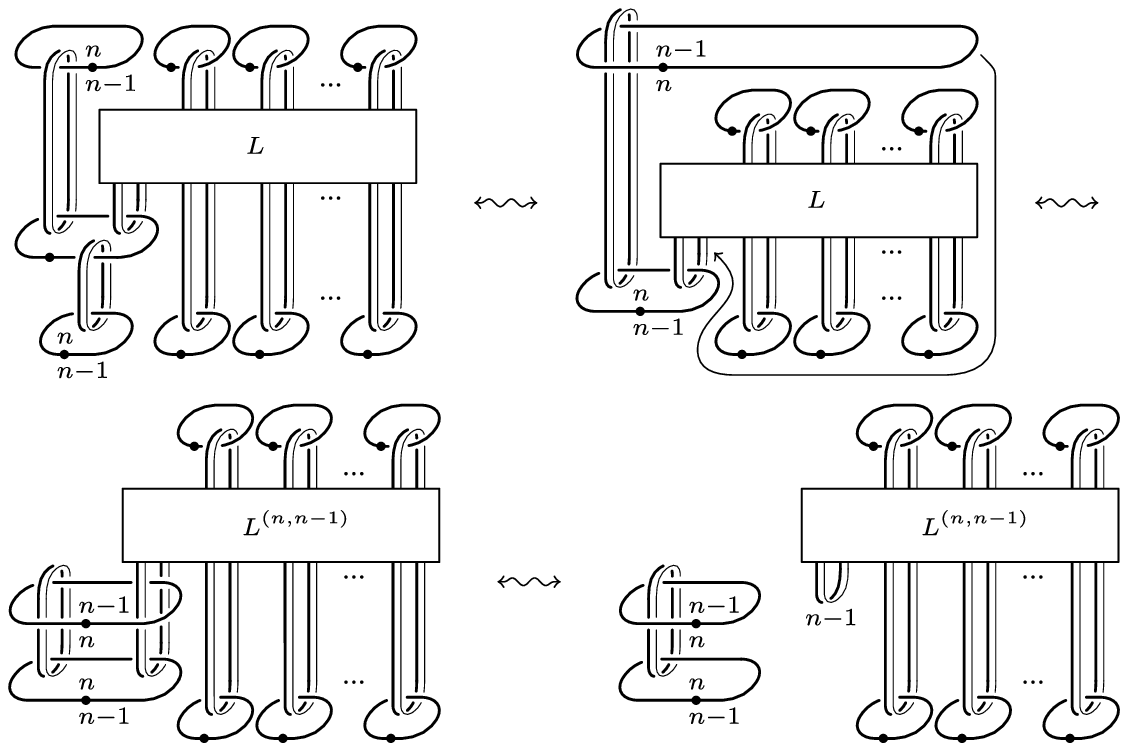}}
\vskip-3pt
\end{Figure}

\begin{lemma}\label{connected-red/thm}
Up to labeled 1-isotopy, any $n$-labeled ribbon surface $S$ representing a connected
branched covering of $B^4$ is the closure $\hat R$ of a 1-reducible ribbon surface tangle
$R = (\id_{\sigma_{n \red 1}} \!\diam T) \circ \Delta_{\sigma_{n \red 1}} : I_{\sigma_{n
\red 1}} \to I_{\sigma_{n \red 1}}$ in $\S^c_n$ (cf. Figure \ref{closure07/fig}).
\end{lemma}

\begin{Figure}[htb]{closure07/fig}
{}{Connected coverings are given by 1-reducible ribbon surfaces}
\centerline{\fig{}{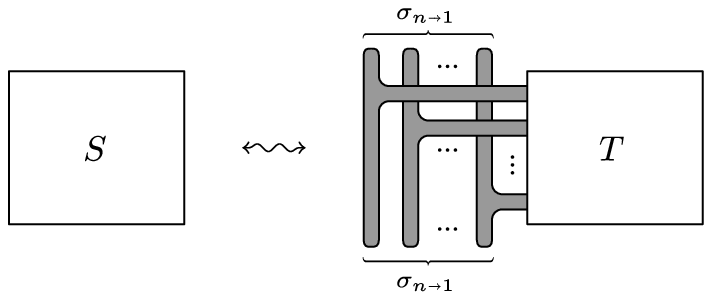}}
\vskip-6pt
\end{Figure}

\begin{proof}
The connectedness of the covering implies that the transpositions which appear as labels
of $S$ generate a transitive subgroup of the symmetric group $\Sigma_n$. This is trivially
equivalent to say that they generate all $\Sigma_n$. Then, we can use the labeled
1-isotopy move \(S24) in Figure \ref{ribbon-surf13/fig} to expand from $S$ a tongue which,
after a suitable sequence of ribbon intersections, is labeled with any given
transposition $\tau \in \Sigma_n$ on its tip. In particular, in this way we can expand
from $S$ the reduction bands making it the closure of a 1-reducible $n$-labeled ribbon
surface tangle, as in Figure \ref{closure07/fig}.
\end{proof}

\begin{theorem}\label{equiv4/thm}
Two connected simple coverings of $B^4$ branched over ribbon surfaces represent
2-equivalent 4-dimensional 2-handlebodies if and only if after sta\-bilization to the same
degree $\geq 4$ their labeled branching surfaces can be related by labeled 1-isotopy, i.e.
labeled 3-dimensional diagram isotopy and the labeled versions of the 1-isotopy moves in
figure Figure \ref{ribbon-surf13/fig} (cf. Proposition \ref{1-isotopy/thm}), and the
ribbon moves \(R1) and \(R2) in Figure \ref{coverings05/fig}. For the reader convenience
those moves are reproduced in Figures \ref{covering-moves01/fig} and
\ref{covering-moves02/fig} below.
\end{theorem}

\begin{Figure}[htb]{covering-moves01/fig}
{}{Labeled 1-isotopy moves (with any labeling)}
\centerline{\fig{}{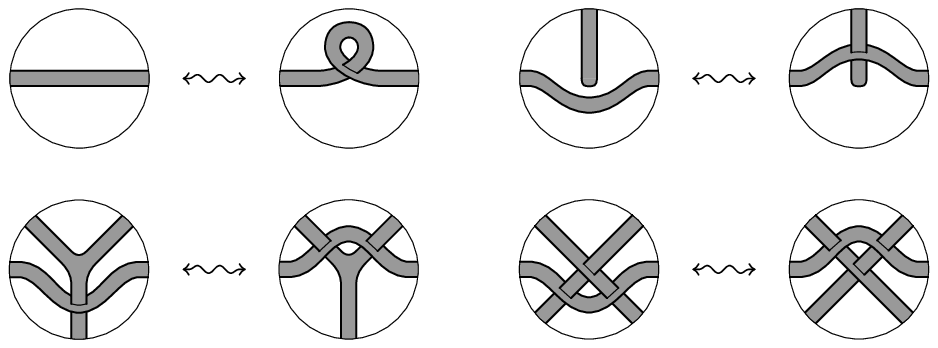}}
\vskip-3pt
\end{Figure}

\begin{Figure}[htb]{covering-moves02/fig}
{}{Ribbon moves ($i,j,k$ and $l$ all different)}
\centerline{\fig{}{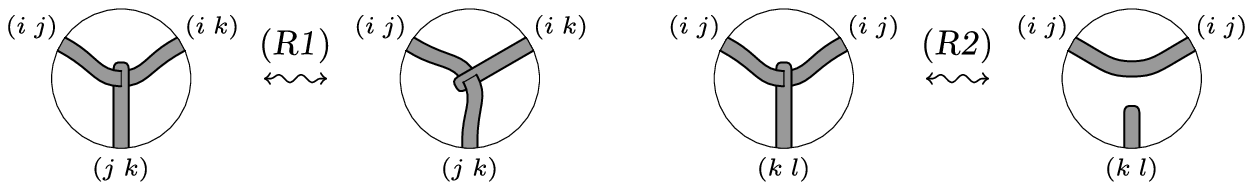}}
\vskip-3pt
\end{Figure}

\begin{proof}
Let $\K_n^{c,\o}$ be the set of morphisms $K: I_{\pi_{n \red 1}} \to I_{\pi_{n \red 1}}$
in $\K_n^c$, and $\S_n^{c,\o}$ be\break the set of morphisms $S: J_{\sigma_{n \red 1}} \to
J_{\sigma_{n \red 1}}$ in $\S_n^c$. Moreover, consider the subsets $\hat\K_n^{c,\o}
\subset \hat\K_n^\o$ and $\hat\S_n^{c,\o} \subset \hat\S_n^\o$ consisting of the closures
of the morphisms in $\K_n^{c,\o}$ and $\S_n^{c,\o}$ respectively. Then, by the above
discussion of the closure and Lemma \ref{hat-reduction/thm}, we have the following
commutative diagram.

\vskip9pt
\centerline{\epsfbox{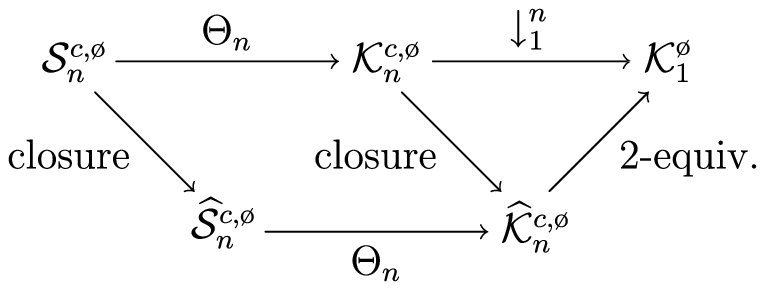}}
\vskip6pt

We observe the $1$-reduction $\down_1^n K$ of a Kirby tangle $K \in \K^{c,\o}_n$ is
already a Kirby diagram, hence we have $\hat{\down_1^n K} = \down_1^n K$. Then, the
2-equivalence arrow is given by Lemma \ref{hat-reduction/thm} and it consists in the
reduction to only one 0-handle as described in the proof of that lemma.

By Lemma \ref{connected-red/thm}, any $n$-labeled ribbon surface that represents a
connected branched covering of $B^4$ can be seen as a representative of an element of
$\hat\S_n^{c,\o}$. Then, according to the discussion at the beginning of the section, the
branched covering representation of 4-dimensional 2-handlebodies in the statement
coincides with the composition of $\Theta_n: \hat\S_n^{c,\o} \to \hat\K_n^{c,\o}$ and the
2-equivalence map $\hat\K_n^{c,\o} \to \K_1^\o$.

Now, the ``if'' part of the statement just says that such composition is well-defined,
which is known from Propositions \ref{theta/thm} and \ref{stab-theta/thm} for any $n \geq
2$. While the ``only if'' part says that it is injective for $n \geq 4$. This immediately
follows from the surjectivity of the closure map $\S_n^{c,\o} \to \hat\S_n^{c,\o}$ and the
bijectivity of ${\down_1^n} \circ \Theta_n: \S_n^{c,\o} \to \K_1^\o$, guaranteed by
Theorem \ref{ribbon-kirby/thm} for any $n \geq 4$.
\end{proof}

Before of going on, let us make a pair of remarks on the above results about the 
representation of 4-dimensional 2-handlebodies as branched coverings of $B^4$.

\begin{remark}\label{disconnected-4hb/rem}
Both Proposition \ref{hat-xi/thm} and Theorem \ref{equiv4/thm} concern connected
handlebodies. However, their generalization to more connected components is
straightforward, being different components independent from each other. Of course, to
represent 4-dimensional 2-handlebodies with $c$ connected components in general are needed
coverings of degree $3c$. While coverings of degree $3c + 1$ are involved in relating two
covering representations of 2-equivalent handlebodies (degree $4c$ is needed if the
stabilizations are required to be performed once for all at the beginning). Moreover, in
contrast with the connected case, also labeling conjugation has to be allowed, in order to
get the same set of labels for the sheets of the two coverings contained in corresponding
components.
\end{remark}

\begin{remark}\label{1-isotopy/rem}
We recall that 1-isotopy of ribbon surfaces in $B^4$ was derived from embedded
1-deformation of embedded 2-dimensional 1-handlebodies in $B^4$, by forgetting the
handlebody structure. On the other hand, isotopy is related in a similar way to a suitable
notion of embedded 2-deformation. In this perspective, isotopy differs from 1-isotopy just
for allowing also addition/deletion of embedded canceling pairs of 1/2-handles and
2-handle isotopy (which may involve non-ribbon intersections in the diagram, such as
double loops and triple points).

Analogously, diffeomorphic 4-dimensional 2-handlebodies are 3-equivalent (cf. Section
\ref{handles/sec}). Hence, the notion of diffeomorphism between 4-dimensional
2-handle\-bodies differs from that of 2-equivalence for allowing also addition/deletion of
canceling pairs of 2/3-handles and 3-handle isotopy.

Now, the connection between labeled 1-isotopy of ribbon surfaces in $B^4$ and
2-equivalence of 4-dimensional 2-handlebodies, established in the proofs of Lemma
\ref{ribbon-to-kirby/thm} and Proposition \ref{theta/thm} (cf. also Lemma
\ref{SK-invariance/thm} for the opposite direction) through branched coverings and
covering moves, can be at least partially extended. More precisely, attaching a labeled
2-handle to the branching surface $S \subset B^4$ of a simple branched covering $p: W \to
B^4$ corresponds to attaching a 3-handle to the covering 4-dimensional handlebody $W$, in
such a way that a canceling pair of labeled 1/2-handles in $S$ corresponds to a canceling
pair of 2/3-handles in $W$.

This suggests a possible approach towards the study of the difference between
2-deformation and diffeomorphism of 4-dimensional 2-handlebodies, by relating it to the
difference between 1-isotopy and isotopy of ribbon surfaces. Good test examples could be
the Akbulut-Kirby 4-balls $\Delta_n$ (see figure \ref{gompf/fig} for the case of $n = 3$),
which were proved to be diffeomorphic to $B^4$ in \cite{Go91}, but are not known to be
2-equivalent to $B^4$. In fact, the proof of the diffeomorphism $\Delta_n \cong B^4$ is
based on the cleaver addition of a canceling pair of 2/3-handles, followed by an isotopy
of the attaching map of the 3-handle and eventually by the cancelation of it against
another 2-handle. It would be interesting to see whether this process corresponds to
changing the branching ribbon surface by labeled isotopy.
\end{remark}

\begin{Figure}[htb]{gompf/fig}
{}{The Akbulut-Kirby 4-ball $\Delta_3$}
\vskip-6pt
\centerline{\kern10pt\fig{}{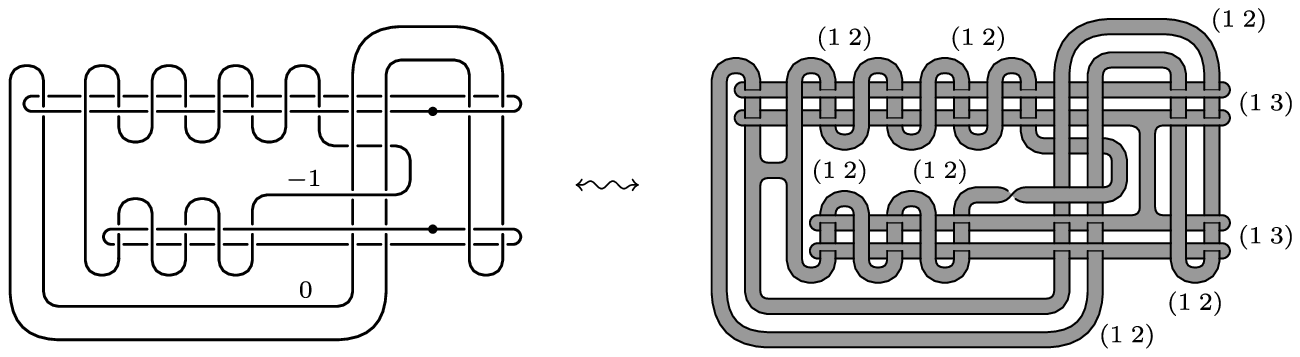}}
\end{Figure}

We conclude the section with the following theorem concerning 4-dimensional 2-handlebodies
having diffeomorphic boundaries. This will be applied in the next section to obtain
covering moves theorems for 3-manifolds.

\begin{theorem}\label{equiv4b/thm}
Two connected simple coverings of $B^4$ branched over ribbon surfaces represent
4-dimensional 2-handlebodies with diffeomorphic oriented boundaries if and only if after
stabilization to the same degree $\geq 4$ their labeled branching surfaces can be related
by labeled 1-isotopy, the ribbon moves \(R1) and \(R2), and the moves \(P) and \(T) in
Figure \ref{covering-moves03/fig} (cf. Figure \ref{boundary04/fig}).
\end{theorem}

\begin{Figure}[htb]{covering-moves03/fig}
{}{Ribbon moves ($i,j,k$ and $l$ all different)}
\centerline{\fig{}{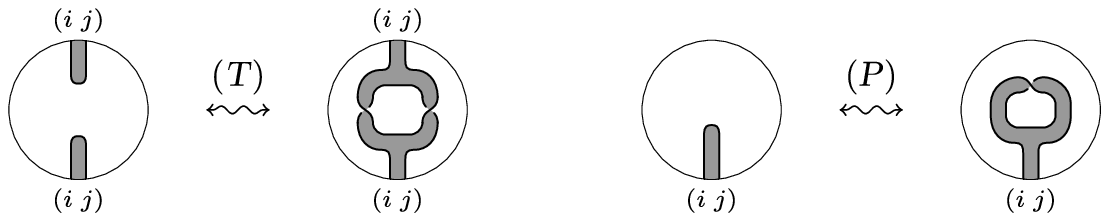}}
\vskip-3pt
\end{Figure}

\begin{proof}
In the light of Propositions \ref{Kb-reduction/thm} and \ref{theta-bd/thm}, the
commutative diagram in the proof of Theorem \ref{equiv4/thm} induces the following one
under the quotient functors $\S_n \to \Sbb_n$ and $\K_n \to \Kbb_n$. Here, the double bar 
denotes the image in the quotient of the corresponding set in the original diagram.

\vskip12pt
\centerline{\epsfbox{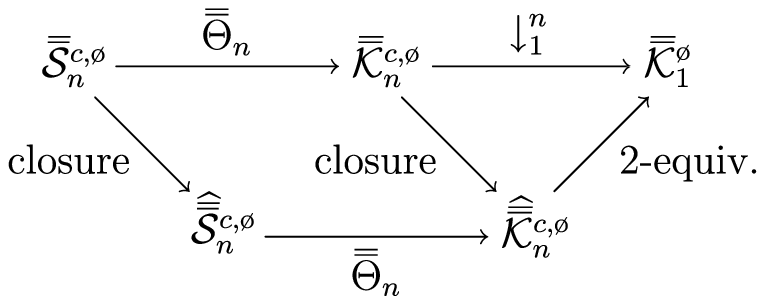}}
\vskip6pt

Up to planar isotopy the moves \(T) and \(P) in Figure \ref{covering-moves03/fig} are 
the same as the homonymous additional relations defining the quotient category $\Sbb_n$.
On the other hand, we know that two uni-labeled Kirby diagrams represent 4-dimensional 
2-handlebodies with diffeomorphic boundaries if and only if they are equivalent in 
$\Kbb_1^\o$ (this is just the classical Kirby's theorem).

Then, the theorem can be rephrased by saying that the composition of $\bbTheta_n:
\Sbbh_n^{c,\o} \to \Kbbh_n^{c,\o}$ with the 2-equivalence arrow is a well-defined
injective map. This can be proved by arguing on the commutative diagram as in the proof of
Theorem \ref{equiv4/thm}, thanks to Proposition \ref{theta-bd/thm} and Theorem
\ref{ribbon-kirby-bd/thm}.
\end{proof}

\subsection{Equivalence of branched covers $S^3$%
\label{S3/sec}}

Our last goal is to derive from Theorem \ref{equiv4b/thm} the announced general solution 
of the covering moves problem for branched coverings of $S^3$.

As a preliminary step, we show that any simply labeled link in $S^3$ can be transformed
through the Montesinos moves in Figure \ref{coverings04/fig} into the boundary of a
labeled ribbon surface in $B^4$ (see Proposition \ref{link/ribbon/thm}). This follows
quite directly from Theorem B of \cite{MP01} about liftable braids, which we state here as
Lemma \ref{braids/thm} after having recalled a couple of definitions.

A simply labeled braid is called a {\sl liftable braid} when the two labelings at its
ends coincide. By an {\sl interval} we mean any braid that is conjugate to a standard
generator in the braid group. Actually, to make both the terms ``liftable'' and
``interval'' meaningful, one should think of braids as self-homeomorphisms of the disk in
the usual way (see \cite{BW85} or \cite{MP01}), but this is not relevant in the present
context.

Of course, a labeled interval, as well as a standard generator, may or may not be
liftable depending on the labeling. We say that a labeled interval $x$ is of {\sl type
$i$} if $x^i$ is the first positive power of $x$ which is liftable. It is not difficult to
realize that conjugation preserves interval types and that each interval is of type 1, 2
or 3 (cf. Lemma 2.4 of \cite{BW85} or Lemma 2.3 of \cite{MP01}).

The labeled intervals $x$, $y$ and $z$, whose first liftable positive powers are depicted
in Figure \ref{covering-moves04/fig}, are the standard models for the three types above.
Namely, any labeled interval of type 1, 2 or 3 is respectively a conjugate of $x^{\pm1}$,
$y^{\pm1}$ or $z^{\pm1}$. Evidently, in the figure only the two non-trivial strings of
each labeled braid are drawn, the other ones being just horizontal arcs with arbitrary
labels. Moreover, in the labeling of each single braid, we assume that $i$, $j$, $k$ and
$l$ are all different.

\begin{Figure}[htb]{covering-moves04/fig}{}{}
\centerline{\fig{}{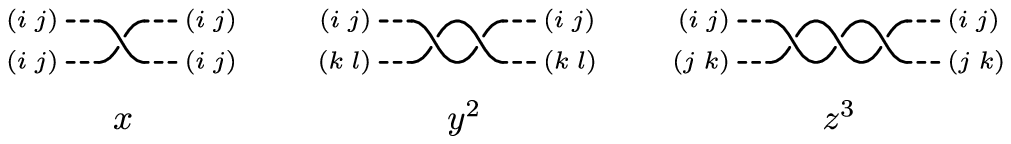}}
\vskip-3pt
\end{Figure}

The main result of \cite{MP01} is the following.

\begin{lemma} \label{braids/thm}
Any liftable braid is a product liftable powers of intervals.
\end{lemma}

We emphasize that the lemma holds without restrictions on the degree $n$ of the labeling.
However, it is worth observing that the case of $n = 2$ is trivial (every braid is
liftable in this case), while the case of $n = 3$ differs from the general one for the
absence of intervals of type 2. This special case was previously proved in \cite{BW85}
(cf. also \cite{BW94}), but the proof of Lemma \ref{braids/thm} given in \cite{MP01} does
not depend on \cite{BW85}.

The relevant consequence of Lemma \ref{braids/thm} in the present context is the following
branched covering counterpart of the vanishing of the oriented cobordism group $\Omega_3$.

\begin{proposition} \label{link/ribbon/thm}
Any labeled link $L \subset S^3$ representing a (possibly disconnected) $n$-fold simple
branched covering of $S^3$ is equivalent, up to labeled isotopy and moves \(M1) and
\(M2) in Figure \ref{coverings04/fig}, to the boundary of labeled ribbon surface $S
\subset B^4$ representing an $n$-fold simple branched covering of $B^4$.
\end{proposition}

\begin{proof}
Up to labeled isotopy, we can assume that the link $L$ is the closure $\widehat B$ of
simply labeled braid $B$ (for example, we can use the labeled version of the well-known
Alexander's braiding procedure). Of course $B$ is a liftable braid. Then, Lemma
\ref{braids/thm} tells us that, up to labeled isotopy, we can think of $B$ a product of
conjugates of braids like $x^{\pm1}$, $y^{\pm2}$ or $z^{\pm3}$ (see Figure
\ref{covering-moves04/fig}). Since braids $y^{\pm2}$ and $z^{\pm3}$ can be obviously
trivialized respectively by moves \(M2) and \(M1), we can reduce ourselves to the
case when $B$ is a product of liftable intervals.

In this case, a simply labeled ribbon surface $S \subset B^4$ bounded by $L$ can be easily
constructed from the band presentation of $B$ (see \cite{Ru83,Ru85}) determined by its
factorization into liftable intervals. Namely, we start with a disjoint union of labeled
trivial disks in $B^4$, spanned by the labeled trivial braid obtained from $B$ by
trivializing all the terms $x^{\pm1}$ appearing in the factorization. Then, we attach to
these disks a labeled half-twisted band for each such term. The 3-dimensional diagram of
$S$ may or may not form ribbon intersection, depending on the conjugating braids of the
liftable intervals in the factorization of $B$ (cf. \cite{Ru83,Ru85}). The labeling
consistency when attaching the bands is always ensured by the liftability of the
intervals.
\end{proof}

Now we can prove our equivalence theorem for simply labeled links in $S^3$.

\begin{theorem}\label{equiv3/thm}
Two connected simple coverings of $S^3$ branched over links represent diffeomorphic
oriented 3-manifolds if and only if after stabilization to the same degree $\geq 4$ their
labeled branching links can be related by labeled isotopy and the Montesinos moves \(M1)
and \(M2) in Figure \ref{covering-moves05/fig} (cf. Figure \ref{coverings04/fig}).
\end{theorem}

\begin{Figure}[htb]{covering-moves05/fig}
{}{Montesinos moves ($i,j,k$ and $l$ all different)}
\centerline{\fig{}{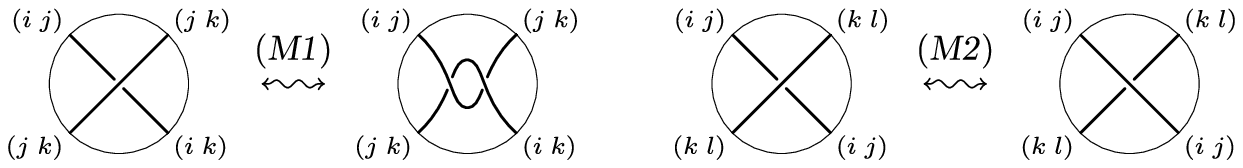}}
\end{Figure}

\begin{proof}
As mentioned in Section \ref{coverings/sec}, it has been known for a long time, since the
early work of Montesinos, that moves \(M1) and \(M2) are covering moves. That is they, as
well as labeled isotopy and stabilization, do not change the covering manifold up to
diffeomorphism (see Figure \ref{coverings06/fig} for a proof of this fact). Therefore,
nothing more has to be added about the ``if'' part of the theorem.

\begin{Figure}[htb]{covering-moves06/fig}{}{}
\centerline{\fig{}{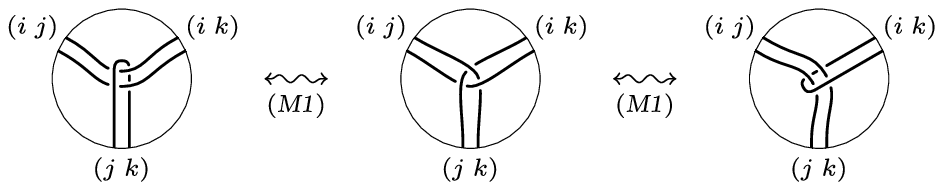}}
\end{Figure}

The ``only if'' part follows from Proposition \ref{link/ribbon/thm} and Theorem
\ref{equiv4b/thm}, taking into account that moves \(T) and \(P) preserve the boundary up
to labeled isotopy, while the restriction of moves \(R1) and \(R2) to the boundary can be
realized by moves \(M1) and \(M2) respectively. The last fact is trivial for move \(R2)
and it is shown in Figure \ref{covering-moves06/fig} for move \(R1).
\end{proof}

Finally, we want to extend Theorem \ref{equiv3/thm} to arbitrary connected branched
coverings of $S^3$, by adding the extra moves \(G1) and \(G2) in Figure
\ref{covering-moves07/fig} (cf. Figure \ref{coverings03/fig}), where the local
orientations of the branching set is only needed to specify the monodromy. The further
extension to disconnected branched coverings is left to the reader (cf. Remark
\ref{disconnected-4hb/rem}).

\begin{theorem}\label{equiv3g/thm}
Two connected coverings of $S^3$ branched over a graph represent diffeomorphic oriented
3-manifolds if and only if after stabilization to the same degree $\geq 4$ their branching
graphs can be related by labeled isotopy and the moves \(M1), \(M2),
\(G1) and \(G2) in Figures \ref{covering-moves05/fig} and \ref{covering-moves07/fig}.
\end{theorem}

\begin{Figure}[htb]{covering-moves07/fig}
{}{Covering moves for labeled graphs ($\sigma = \sigma_1 \!\cdot \sigma_2$)}
\centerline{\fig{}{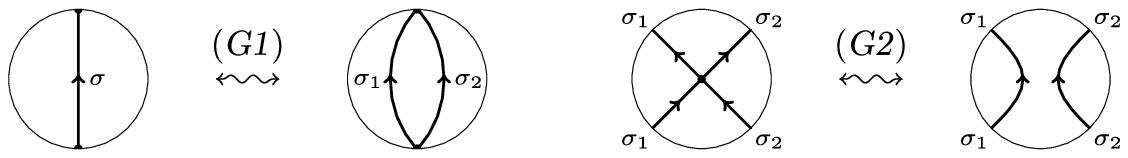}}
\end{Figure}

\begin{proof}
We have already observed in Section \ref{coverings/sec} that moves \(G1) and \(G2) are
covering moves, as they are applications of the coherent monodromies merging principle. In
the light of Theorem \ref{equiv3/thm}, we have only to show that they allow us to
transform any labeled graph into a simply labeled link. We proceed in two steps: first we
make the labeling simple, by performing moves \(G1) on the edges; then we make the graph
into a link, by performing moves \(G2) on the vertices.

Let $G \subset R^3$ be a labeled embedded graph, endowed with a given graph structure
without loops (that is every edge has distinct endpoints). We make the labeling simple,
by operating on the edges of $G$ one by one. Each time, we assume, up to labeled isotopy,
that the edge $e$ under consideration is not involved in any crossing. Denoting by $\sigma
\in \Sigma_n$ the label of $e$, we consider a coherent factorizations $\sigma = \tau_1
\dots \tau_k$ into transpositions (any minimal factorization of $\sigma$ is coherent).
Then, we split $e$ into $k$ edges $e_1, \dots, e_k$ with the same endpoints, such that
$e_i$ is labeled by $\tau_i$, for each $i = 1, \dots, k$. To do that, we perform $k-1$
moves \(G1), which progressively isolate the transpositions $\tau_i$ as labels of new
edges. Once all edges of $G$ have been managed in this way, we are left with a simply
labeled graph which we still denote by $G$.

Now, we operate on the vertices of $G$ one by one, in order to make $G$ into a link. Let
$v$ be a vertex of $G$ and $e_1, \dots, e_h$ be the edges of $G$ having $v$ as an
endpoint, numbered according to the counterclockwise order in which they appear around $v$
in the planar diagram of $G$. Since the total monodromy $\tau_1 \dots \tau_h$ around $v$
must be trivial, $h$ must be even and the edges around $v$, can be reordered, up to
labeled isotopy, in such a way that $\tau_i = \tau_{h-i+1}$, for every $i = 1, \dots,
h/2$. This immediately follows from the well-known classification of the branched
coverings of $S^2$, if one looks at a small 2-sphere around $v$ transversal to $G$ (cf.
\cite{BE79} or \cite{MP01}). Then, by $h/2 - 1$ applications of move \(G2), we replace the
vertex $v$ by $h/2$ non-singular vertices $v_1, \dots, v_{h/2}$, such that $v_i$ is a
common endpoint of $e_i$ and $e_{h-i+1}$, for each $i = 1, \dots, h/2$.\break We leave to
the reader to verify that the sequence $\tau_1, \dots, \tau_{h/2}$ is coherent and that
this suffices for the needed moves \(G2) to be performable. Obviously, after all the
singular vertices of $G$ have been replaced by non-singular ones, we are done.
\end{proof}


\newpage

\thebibliography{---%
\label{references}}

\bibitem{Ak91} S. Akbulut, {\sl An exotic 4-manifold}, J. Diff. Geom. {\bf 33} (1991),
357--361.

\bibitem{AK80} S. Akbulut and R. Kirby, {\sl Branched covers of surfaces in 4-manifolds},
Math. Ann. {\bf 252} (1980), 111--131.

\bibitem{Al20} J.W. Alexander, {\sl Note on Riemann spaces}, Bull. Amer. Math. Soc. {\bf
26} (1920), 370--373.

\bibitem{Ap03} N. Apostolakis, {\sl On 4-fold covering moves}, Algebraic \& Geometric
Topology {\bf 3} (2003), 117--145.

\bibitem{APZ11} N. Apostolakis, R. Piergallini and D. Zuddas, {\sl Lefschetz fibration over the disk}, preprint arXiv:1104.4536 (2011).

\bibitem{At90} S. Atiyah, {\sl On framings of 3-manifolds}, Topology {\bf 29} (1990), 
1--7.

\bibitem{BE78} I. Bernstein and A.L. Edmonds, {\sl The degree and branch set of a brenched 
covering}, Invent. math. {\bf 45} (1978), 213--220.

\bibitem{BE79} I. Bernstein and A.L. Edmonds, {\sl On the construction of branched
coverings of low-dimensional manifolds}, Trans. Amer. Math. Soc. {\bf 247} (1979),
87--124.

\bibitem{BW85} J.S. Birman and B. Wajnryb, {\sl 3-fold branched coverings and the mappings
class group of a surface}, in ``Geometry and Topology'', Lecture Notes in Math. {\bf
1167}, Springer-Verlag 1985, 24--46.

\bibitem{BW94} J. S. Birman and B. Wajnryb, {\sl Presentations of the mapping class group.
Errata: ``3-fold branched coverings and the mapping class group of a surface'' and ``A
simple presentation of the mapping class group of an orientable surface''}, Israel J.
Math. {\bf 88} (1994), 425--427.

\bibitem{BM03} I. Bobtcheva and M.G. Messia, {\sl HKR-type invariants of 4-thickenings of
2-di\-mensional CW-complexes}, Algebraic and Geometric Topology {\bf 3} (2003),
33--87.\label{references/sec}

\bibitem{BP04} I. Bobtcheva and R. Piergallini, {\sl Covering Moves and Kirby Calculus},
preprint arXiv:math.GT/0407032 (2004).

\bibitem{BP05} I. Bobtcheva and R. Piergallini, {\sl A universal invariant of
four-dimensional 2-handlebodies and three-manifolds}, preprint arXiv:math.GT/0612806 
(2006).

\bibitem{BQ} I. Bobtcheva and F. Quinn, {\sl The reduction of quantum invariants of
4-thicken\-ings}, Fundamenta Mathematicae {\bf 188} (2005), 21--43.

\bibitem{BKLT97} Y. Bespalov, T. Kerler, V. Lyubashenko and V. Turaev, {\sl Integrals for
braided Hopf algebras}, J. Pure Appl. Algebra {\bf 148} (2000), 113--164.

\bibitem{Ce70} J. Cerf, {\sl La stratification naturelle des espaces fonction
diff\'erentiables r\'eelles et la th\'eor\`eme de la pseudo-isotopie}, Publ. Math.
I.H.E.S. {\bf 39} (1970).

\bibitem{DMA10} D. Denicola , M. Marcolli and A.Z. al-Yasry, {\sl Spin foams and 
noncommutative geometry}, Class. Quantum Grav. {\bf 27} (2010), 205025. 

\bibitem{FR79} R. Fenn and C. Rourke, {\sl On Kirby's calculus of links}, Topology {\bf
18} (1979), 1--15.

\bibitem{Fo57} R.H. Fox, {\sl Covering spaces with singularities}, in ``Algebraic Geometry
and Topology, A symposium in honour of S. Lefschetz'', Princeton University Press 1957,
243--257.

\bibitem{Gi82} C.A. Giller, {\sl Towards a classical knot theory for surfaces in $R^4$},
Illinois J. Math. {\bf 26} (1982), 591--631.

\bibitem{Go91} R.E. Gompf, {\sl Killing the Akbulut-Kirby 4-sphere, with relevance to the
An\-drews-Curtis and Schoenflies problems}, Topology {\bf 30} (1991), 97--115.

\bibitem{GS99} R.E. Gompf and A.I. Stipsicz, {\sl 4-manifolds and Kirby calculus}, Grad.
Studies in Math. {\bf 20}, Amer. Math. Soc. 1999.

\bibitem{IP02} M. Iori and R. Piergallini, {\sl 4-manifolds as covers of $S^4$ branched
over non-singular surfaces}, Geometry \& Topology {\bf 6} (2002), 393--401.

\bibitem{Ha00} K. Habiro, {\sl Claspers and finite type invariants of links}, Geometry \&
Topology, {\bf 4} (2000), 1--83.

\bibitem{Hr01} F. Harou, {\sl Description chirugicale des rev\^etements triples simples de
$S^3$ ramifi\'es le long d'un entrelacs}, Ann. Inst. Fourier {\bf 51} (2001), 1229--1242.

\bibitem{Hr03} F. Harou, {\sl Description en terme de rev\^etements simples de
rev\^etements ramifi\'es de la sph\`ere}, preprint 2002.

\bibitem{Ht10} E. Hatakenaka, {\sl Invariants of 3-manifolds derived from covering 
presentations}, Math. Proc. Camb. Philos. Soc. {\bf 149} (2010), 263--295.

\bibitem{He96} M. Hennings, {\sl Invariants from links and 3-manifolds obtained from Hopf
algebras}, J. London Math. Soc. (2) {\bf 54} (1996), 594--624.

\bibitem{Hi74} H.M. Hilden, {\sl Every closed orientable 3-manifold is a 3-fold branched
covering space of $S^3$}, Bull. Amer. Math. Soc. {\bf 80} (1974), 1243--1244.

\bibitem{Hi76} H.M. Hilden, {\sl Three-fold branched coverings of $S^3$}, Bull. Amer. J.
Math. {\bf 98} (1976), 989--997.

\bibitem{Hs74} U. Hirsch, {\sl \"Uber offene Abbildungen auf die 3-Sph\"are}, Math. Z.
{\bf 140} (1974), 203--230.

\bibitem{Ke97} T. Kerler, {\sl Genealogy of nonpertrubative quantum invariants of
3-manifolds -- The surgical family}, in ``Geometry and Physics'', Lecture Notes in Pure
and Applied Physics {\bf 184}, Marchel Dekker 1997, 503--547.

\bibitem{Ke98} T. Kerler, {\sl Equivalence of bidged links calculus and Kirby's calculus
on links on nonsimply connected 3-manifolds}, Topology and its Appl. {\bf 87}, (1998),
155--162.

\bibitem{Ke99} T. Kerler, {\sl Bidged links and tangle presentations of cobordism
categories}, Adv. Math {\bf 141} (1999), 207--281.

\bibitem{Ke02} T. Kerler, {\sl Towards an algebraic characterization of 3-dimensional
cobordisms}, Contemporary Mathematics {\bf 318} (2003), 141--173.

\bibitem{KL01} T. Kerler and V.V. Lyubashenko, {\sl Non-semisimple topological quantum
field theories for 3-manifolds with corners}, Lecture Notes in Mathematics {\bf 1765},
Springer-Verlag 2001.

\bibitem{Ki78} R. Kirby, {\sl A calculus for framed links in $S^3$}, Invent. math. {\bf
45} (1978), 36--56.

\bibitem{Ki89} R. Kirby, {\sl The topology of 4-manifolds}, Lecture Notes in Mathematics
{\bf 1374}, Springer-Verlag 1989.

\bibitem{Ku94} G. Kuperberg, {\sl Non-involutory Hopf algebras and 3-manifold invariants},
Duke Math. J. {\bf 84} (1996), 83--129.

\bibitem{LP72} F. Laudenbach and V. Poenaru, {\sl A note on 4-dimensional handlebodies},
Bull. Soc. Math. France {\bf 100} (1972), 337--344.

\bibitem{Li62} W.R.B. Lickorish, {\sl A representation of orientable, combinatorial 
3-manifolds}, Ann. Math. {\bf 76} (1962), 531--540.

\bibitem{LP01} A. Loi and R. Piergallini, {\sl Compact Stein surfaces with boundary as
branched covers of $S^4$}, Invent. math. {\bf 143} (2001), 325--348.

\bibitem{Lu93} G. Lusztig, {\sl Introduction to quantum groups}, Progress in Mathematics
{\bf 110}, Birkh\"auser 1993.

\bibitem{Ly95}  V.V. Lyubashenko, {\sl Modular transformations for tensor categories}, J. 
Pure Appl. Algebra {\bf 98},(1998), 279--327.

\bibitem{McL} S. MacLane, {\sl Natural associativities and commutativities}, Rice Univ.
Studies {\bf 49} (1963), 28--46.

\bibitem{McL71} S.MacLane, {\sl Categories for the working Mathematicians}, Graduate Texts
in Mathematics {\bf 5}, Springer-Verlag 1971.

\bibitem{Ma03} S. Matveev, {\sl Algorithmic topology and classification of 3-manifolds},
Algorithms and Computation in Mathematics {\bf 9}, Springer 2003.

\bibitem{MaP92} S. Matveev and M. Polyak, {\sl A geometrical presentation of the surface
mapping class group and surgery}, Comm. Math. Phys. {\bf 160} (1994), 537--550.

\bibitem{Mo74} J.M. Montesinos, {\sl A representation of closed, orientable 3-manifolds as
3-fold branched coverings of $S^3$}, Bull. Amer. Math. Soc. {\bf 80} (1974), 845--846.

\bibitem{Mo75} J.M. Montesinos, {\sl Sobre la representacion de variedades 
tridimensionales}, unpublished preprint (1975).

\bibitem{Mo76} J.M. Montesinos, {\sl Three-manifolds as 3-fold branched covers of $S^3$},
Quart. J. Math. Oxford (2) {\bf 27} (1976), 85--94.

\bibitem{Mo78} J.M. Montesinos, {\sl 4-manifolds, 3-fold covering spaces and ribbons},
Trans. Amer. Math. Soc. {\bf 245} (1978), 453--467.

\bibitem{Mo79} J.M. Montesinos, {\sl Heegaard diagrams for closed 4-manifolds}, in
``Geometric Topology'', J.C. Cantrell ed., Academic Press 1979, 219--237.

\bibitem{Mo80} J.M. Montesinos, {\sl A note on 3-fold branched coverings of $S^3$}, Math.
Proc. Camb. Phil. Soc. {\bf 88} (1980), 321--325.

\bibitem{Mo83} J.M. Montesinos, {\sl Representing 3-manifolds by a universal branching
set}, Proc. Camb. Phil. Soc. {\bf 94} (1983), 109--123.

\bibitem{Mo85} J.M. Montesinos, {\sl A note on moves and irregular coverings of $S^4$},
Contemp. Math. {\bf 44} (1985), 345--349.

\bibitem{Mon92} S. Montgomery, {\sl Hopf algebras and their action on rings}, CBMS
Regional Conference Series in Mathematics {\bf 82}, Amer. Math. Soc. 1993.

\bibitem{MP01} M. Mulazzani and R. Piergallini, {\sl Lifting braids}, Rend. Ist. Mat.
Univ. Trieste {\bf XXXII} (2001), Suppl. 1, 193--219.

\bibitem{No11} T. Nosaka, {\sl 4-fold symmetric quandle invariants of 3-manifolds},
Algebraic \& Geometric Topology {\bf 11} (2011), 1601--1648.

\bibitem{Oh02} T. Ohtsuki, {\sl Problems on invariants of knots and 3-manifolds}, Geom.
Topol. Monogr. {\bf 4}, in ``Invariants of knots and 3-manifolds (Kyoto, 2001)'', Geom.
Topol. Publ. 2002, 377--572.

\bibitem{Pi91} R. Piergallini, {\sl Covering Moves}, Trans Amer. Math. Soc. {\bf 325}
(1991), 903--920.

\bibitem{Pi95} R. Piergallini, {\sl Four-manifolds as 4-fold branched covers of $S^4$},
Topology {\bf 34} (1995), 497--508.

\bibitem{PZ03} R. Piergallini and D. Zuddas, {\sl A universal ribbon surface in $B^4$},
Proc. London Math. Soc. {\bf 90} (2005), 763--782.

\bibitem{Qu95} F. Quinn, {\sl Lectures on Axiomatic Topological Quantum Field Theory},
IAS/Park City Mathematical Series {\bf 1}, Amer. Math. Soc. 1995.

\bibitem{RT91} N.Yu. Reshetikhin and V.G. Turaev, {\sl Invariants of 3-manifold via link
polynomials and quantum groups}, Invent. Math. {\bf 103} (1991), 547--597.

\bibitem{Ro51} V.A. Rohlin , {\sl A three dimensional manifold is the boundary of a four 
dimensional one}, Dokl. Akad. Nauk. SSSR {\bf 114} (1951), 355--357.

\bibitem{Ru83} L. Rudolph, {\sl Braided surfaces and Seifert ribbons for closed braids},
Comment. Math. Helvetici {\bf 58} (1983), 1--37.

\bibitem{Ru85} L. Rudolph, {\sl Special position for surfaces bounded by closed braids},
Rev. Mat. Ibero-Americana {\bf 1} (1985), 93--133; revised version: preprint 2000.

\bibitem{ST89} M. Scharlemann and A. Thompson, {\sl Link genus and Conway moves}, Comment.
Math. Helvetici {\bf 64} (1989), 527--535.

\bibitem{Sh94} M.C. Shum, {\sl Tortile tensor categories}, Journal of Pure and Applied
Algebra {\bf 93} (1994), 57--110.

\bibitem{St96} T. Standford, {\sl Finite type invariants of knots, links and graphs},
Topology {\bf 35} (1996), 1027--1050.

\bibitem{Tu00} V. Turaev, {\sl Homotopy field theory in dimension 3 and crossed
group-categories}, eprint arXiv:math/0005291 (2000).

\bibitem{Vi01} A. Virelizier, {\sl Alg\'ebres de Hopf gradu\'ees et fibr\'es plats sur les
3-vari\'et\'es}, PhD thesis, Institut de recherche math\'ematique avanc\'ee, Universit\'e
Louis Pasteur et CNRS 2001.

\bibitem{Wa91} K.Walker, {\sl On Witten's 3-manifold invariants}, preprint 1991.

\bibitem{Wa60} A.H. Wallace, {\sl Modifications and cobounding manifolds}, Canad. J. Math. 
{\bf 12} (1960), 503--528.
\end{document}